\documentclass[reqno,english]{smfbook}
\usepackage{smfthm,amssymb,mathrsfs,url,color,graphicx} 
\usepackage[british,francais]{babel}
\usepackage[utf8]{inputenc}
\usepackage[T1]{fontenc}
\usepackage[all]{xy}
\usepackage[refpage]{nomencl}

\usepackage{smfhrefb}

\def\ov#1{{\overline{#1}}}
\def\un#1{{\underline{#1}}}
\def\wh#1{{\widehat{#1}}}
\def\wt#1{{\widetilde{#1}}}
\def\?{\ ???\ \immediate\write16{}%
\immediate\write16{Warning: There was still a question mark . . . }%
\immediate\write16{}}

\newcommand{\Conv}{\operatorname{conv}}
\newcommand{\Cone}{\operatorname{cone}}
\newcommand{\rec}{\operatorname{rec}}
\newcommand{\Vol}{\operatorname{vol}}

\newcommand{\coeff}{\operatorname{coeff}}

\newcommand{\chern}{\operatorname{c}}

\newcommand{\Sym}{\operatorname{Sym}}

\newcommand{\cha}{\operatorname{\widehat{CH}}}
\newcommand{\afc}{\operatorname{\widehat{c}_{1}}}

\newcommand{\cc}{\operatorname{c}}
\newcommand{\surj}{\operatorname{surj}}
\newcommand{\inj}{\operatorname{inj}}

\newcommand{\sat}{\operatorname{sat}}
\newcommand{\h}{\operatorname{h}}
\newcommand{\tor}{\operatorname{tor}}
\newcommand{\htor}{\h^{\tor}}

\newcommand{\Nor}{\operatorname{Nor}}
\newcommand{\MV}{\operatorname{MV}}
\newcommand{\Haar}{\operatorname{Haar}}
\newcommand{\hypo}{\operatorname{hypo}}
\newcommand{\epi}{\operatorname{epi}}

\newcommand{\dist}{\operatorname{dist}}
\newcommand{\diag}{\operatorname{diag}}
\newcommand{\diam}{\operatorname{diam}}
\newcommand{\MI}{\operatorname{MI}}

\newcommand{\Hess}{\operatorname{Hess}}
\newcommand{\Ker}{\operatorname{ker}}
\newcommand{\Gal}{\operatorname{Gal}}

\newcommand{\gp}{\text{\rm gp}}

\newcommand{\dd}{\,\text{\rm d}}
\newcommand{\sep}{\text{\rm sep}}

\newcommand{\df}{\operatorname{def}}
\newcommand{\aff}{\operatorname{aff}}

\newcommand{\dc}{\operatorname{d^{c}}}
\renewcommand{\div}{\operatorname{div}}
\newcommand{\Div}{\operatorname{Div}}

\newcommand{\Pic}{\operatorname{Pic}}
\newcommand{\bfSpec}{\operatorname{\bf Spec}}
\newcommand{\Spec}{\operatorname{Spec}}
\newcommand{\mult}{\operatorname{mult}}

\newcommand{\pr}{\operatorname{pr}}
\newcommand{\inc}{\operatorname{\jmath}}
\newcommand{\psiabs}{\psi}

\newcommand{\phiK}{\phi}

\newcommand{\Stab}{\operatorname{stab}}

\newcommand{\Norm}{\operatorname{N}}

\newcommand{\ess}{\operatorname{ess}}

\newcommand{\ri}{\operatorname{ri}}
\renewcommand{\Im}{\operatorname{im}}
\newcommand{\e}{\operatorname{e}}

\newcommand{\Hom}{\operatorname{Hom}}

\newcommand{\sg}{\operatorname{sg}}

\newcommand{\FS}{{\operatorname{FS}}}
\newcommand{\can}{{\operatorname{can}}}
\newcommand{\ord}{{\operatorname{ord}}}

\newcommand{\red}{{\operatorname{red}}}
\newcommand{\val}{{\operatorname{val}}}

\newcommand{\an}{{\text{\rm an}}}
\newcommand{\alg}{{\text{\rm alg}}}

\newcommand{\Dom}{{\operatorname{dom}}}
\newcommand{\cl}{{\operatorname{cl}}}
\newcommand{\ee}{{\operatorname{\mathbf e }}}
\newcommand{\hooklongrightarrow}{\lhook\joinrel\longrightarrow}
\newcommand{\twoheadlongrightarrow}{\relbar\joinrel\twoheadrightarrow}

\def \A{\mathbb{A}}
\def \C{\mathbb{C}}
\def \D{\mathbb{D}}
\def \F{\mathbb{F}}
\def \G{\mathbb{G}}
\def \K{\mathbb{K}}
\def \L{\mathbb{L}}
\def \N{\mathbb{N}}
\def \P{\mathbb{P}}
\def \Q{\mathbb{Q}}
\def \R{\mathbb{R}}
\def \SS{\mathbb{S}}
\def \SSinv{\mathbb{S}}
\def \T{\mathbb{T}}
\def \Z{\mathbb{Z}}

\def\cA {{\mathcal A}}

\def\cC {{\mathcal C}}

\def\cE {{\mathcal E}}

\def\cL {{\mathcal L}}
\def\cM {{\mathcal M}}

\def\cO {{\mathcal O}}

\def\cU {{\mathcal U}}
\def\cV {{\mathcal V}}
\def\cX {{\mathcal X}}
\def\cY {{\mathcal Y}}

\newcommand{\bft}{{\boldsymbol{t}}}

\newcommand{\bfx}{{\boldsymbol{x}}}

\newcommand{\bfzero}{\boldsymbol{0}}

\makeindex
\makenomenclature

\title[Arithmetic geometry of toric
  varieties]{Arithmetic geometry of toric
  varieties. \\ Metrics, measures and heights}

\alttitle{G\'eom\'etrie 
arithm\'etique des vari\'et\'es toriques. M\'etriques,
    mesures et hauteurs}

\author[Burgos Gil]{Jos\'e Ignacio Burgos Gil}
\address{Instituto de Ciencias Matem\'aticas (CSIC-UAM-UCM-UCM3).
  Calle Nicol\'as Ca\-bre\-ra~15, Campus UAM, Cantoblanco, 28049 Madrid,
  Spain} 
\email{burgos@icmat.es}
\urladdr{\url{http://www.icmat.es/miembros/burgos/}}

\author[Philippon]{Patrice Philippon} \address{Institut de
  Math{\'e}matiques de Jussieu -- U.M.R. 7586 du CNRS, \'Equipe de
  Th\'eorie des Nombres.  Case courrier 247, 4 Place Jussieu, 
  75252 Paris Cedex 05,
  France} 
\email{pph@math.jussieu.fr}
\urladdr{\url{http://www.math.jussieu.fr/~pph}}

\author[Sombra]{Mart\'in~Sombra} \address{ICREA \&
  Universitat de Barcelona, Departament d'{\`A}lgebra i Geometria.
  Gran Via~585, 08007 Bar\-ce\-lo\-na, Spain} 
\email{sombra@ub.edu}
\urladdr{\url{http://atlas.mat.ub.es/personals/sombra}}

\begin{document}

\numberwithin{equation}{section}
\numberwithin{smfthm}{section}

\theoremstyle{definition}
\newtheorem{defn}[smfthm]{Definition}
\newtheorem{notn}[smfthm]{Notation}
\newtheorem{rem}[smfthm]{Remark}
\newtheorem{exmpl}[smfthm]{Example}
\newtheorem{assumption}[smfthm]{Assumption}
\newtheorem{convention}[smfthm]{Convention}
\newtheorem{appl}[smfthm]{Application}
\newtheorem{ques}[smfthm]{Question}
\newtheorem{prop-def}[smfthm]{Proposition-Definition}

\theoremstyle{plain}
\newtheorem{lem}[smfthm]{Lemma}
\newtheorem{mainlem}[smfthm]{Main Lemma}
\newtheorem{thm}[smfthm]{Theorem}
\newtheorem{cor}[smfthm]{Corollary}
\newtheorem*{thm*}{Theorem}

\frontmatter


\begin{abstract}
  We show that the height of a toric variety with respect to a toric
  metrized line bundle can be
  expressed as the integral over a polytope of a certain adelic family
  of concave functions. To state and prove this result, we study the
  Arakelov geometry of toric varieties. In particular, we consider
  models over a discrete valuation ring, metrized line bundles, and
  their associated measures and heights.  We show that these notions
  can be translated in terms of convex analysis, and are closely
  related to objects like polyhedral complexes, concave functions,
  real Monge-Amp\`ere measures, and Legendre-Fenchel duality.

  We also present a closed formula for the integral over a polytope of
  a function of one variable composed with a linear form. This formula
  allows
  us to compute the height of toric varieties with respect to some
  interesting metrics arising from polytopes. We also compute the
  height of toric projective curves with respect to the Fubini-Study
  metric and the height of some toric bundles.
\end{abstract}

\begin{altabstract} 
  Nous montrons que la hauteur d'une vari\'et\'e torique relative \`a un
  fibr\'e en droites m\'etris\'e torique s'\'ecrit comme l'int\'egrale sur un
  polytope d'une certaine famille ad\'elique de fonctions concaves. Afin
  d'\'enoncer et d\'emontrer ce r\'esultat, nous \'etudions la g\'eom\'etrie
  d'Arakelov des vari\'et\'es toriques. En particulier, nous consid\'erons
  des mod\`eles de ces vari\'et\'es sur des anneaux de valuation discr\`ete,
  ainsi que les fibr\'es en droites m\'etris\'es et leurs mesures et
  hauteurs associ\'ees. Nous montrons que ces notions se traduisent en
  termes d'analyse convexe et sont intimement li\'ees \`a des objets tels
  que les complexes polyh\'edraux, les mesures de Monge-Amp\`ere et la
  dualit\'e de Legendre-Fenchel.

  Nous pr\'esentons \'egalement une formule close pour l'int\'egration sur
  un polytope d'une fonction d'une variable compos\'ee avec une forme
  lin\'eaire. Cette formule nous permet de calculer la hauteur de
  vari\'et\'es toriques relativement \`a plusieurs m\'etriques int\'eressantes,
  provenant de polytopes. Nous calculons aussi la hauteur des courbes
  toriques projectives relativement \`a la m\'etrique de Fubini-Study et
  la hauteur des fibr\'es toriques.
\end{altabstract}

\altkeywords{Vari\'et\'e torique, espace de Berkovich, mod\`ele entier,
  fibr\'e en droites m\'etris\'e, hauteur d'une vari\'et\'e, fonction concave,
  dual de Legendre-Fenchel dual, mesure de Monge-Amp\`ere r\'eelle.}  
\thanks{Burgos Gil and Sombra were
  partially supported by the MICINN research projects
  MTM2006-14234-C02-01 and MTM2009-14163-C02-01. Burgos Gil was also
  partially supported by the CSIC research project 2009501001 and the
  MICINN research project MTM2010-17389. 
  Sombra was also partially supported by the MINECO
  research project MTM2012-38122-C03-02.
  Philippon was partially supported by the CNRS research projects PICS
  \og Properties of the heights of arithmetic varieties\fg{} and \og
  Diophantine geometry and computer algebra\fg{}, and by the ANR
  research project \og Hauteurs, modularit\'e,
  transcendance\fg{}.}

\maketitle


\tableofcontents

\mainmatter

{\csname c@secnumdepth\endcsname=-2
\chapter{Introduction}\label{Introduction and statement of results}
}
Systems of polynomial equations appear in a wide variety of contexts
in both pure and applied mathematics. Systems arising from
applications are not random but come with a certain structure. When studying those
systems, it is important to be able to exploit that structure.

A relevant result in this direction is the
Bern\v{s}tein-Ku\v{s}nirenko-Khovanskii theorem
\cite{kushnirenko:pnnm,Bernstein:nrsp}. Let $K$ be a field with
algebraic closure $\ov K$. Let  $\Delta \subset \R^{n}$ be a lattice
polytope and  $f_{1},\dots, f_{n}\in K[t_{1}^{\pm1},\dots, t_{n}^{\pm
  1}]$ a family of Laurent polynomials whose Newton polytope is
contained in $\Delta$. The BKK theorem says that the number
(counting multiplicities) of isolated common zeros of $f_{1},\dots,
f_{n}$ in~$(\ov K^{\times})^{n}$ is bounded above by $n!$ times the
volume of $\Delta$, with equality when $f_{1},\dots, f_{n}$ is generic
among the families of Laurent polynomials with Newton polytope
contained in $\Delta$.  This shows how a geometric problem (the
counting of the number of solutions of a system of equations)
can be translated into a combinatorial, simpler one. It is commonly used to
predict when a given system of polynomial equations has a small number
of solutions. As such, it is a cornerstone of polynomial equation solving
and has motivated a large amount of work and results over the past 25
years, see for instance \cite{GKZ:drmd,sturmfels:sspe,MR2419926} and the
references therein.

A natural way to study polynomials with prescribed Newton polytope is
to associate to the polytope $\Delta$ a toric variety $X$ over $K$
equipped with an ample line bundle $L$. The polytope conveys all
the information about the pair $(X,L)$. For instance, the
degree of $X$ with respect to $L$ is given by the
formula
\begin{displaymath}
  \deg_{L}(X)= n!\Vol(\Delta),
\end{displaymath}
where $\Vol$ denotes the Lebesgue measure
of $\R^{n}$.  The Laurent polynomials $f_{i}$ can be identified with
global sections of $L$, and the BKK theorem can be deduced from this
formula. Indeed, there is a dictionary
which allows to translate algebro-geometric properties of toric
varieties in terms of combinatorial properties of polytopes and fans, and
the degree formula above is one entry in this ``toric dictionary''.

The central motivation for this text is an arithmetic analogue for
heights of this 
formula, which is the theorem stated below. 
The height is a basic arithmetic invariant of a 
proper variety over the field of rational numbers. Together with its degree, it measures the
amount of information needed to represent this variety, for instance,
via its Chow form. Hence, this invariant is also relevant in computational
algebraic geometry, see for instance \cite{Giusti_et_al:lbda,
  AvendanoKrickSombra:fbslp,SchostDahanKadri:betspd}.  The notion of
height of varieties generalizes the height of points already considered by
Siegel, Northcott, Weil and others, it is an essential tool in
Diophantine approximation and geometry.

For simplicity of the exposition, in this introduction we assume that
the pair $(X,L)$ is defined over the field of rational numbers $\Q$,
although in the rest of the book we will work with more general adelic
fields (Definition \ref{def:6}). Let
$\mathfrak{M}_{\Q}$ denote the set of places of $\Q$
and let $(\vartheta_{v})_{v\in \mathfrak{M}_{\Q}}$ be a family of concave
functions on $\Delta$ such that $\vartheta_{v}\equiv0$
for all but a finite number of $v$. We will show that, to this data,
one can associate an adelic family of metrics
$(\Vert\cdot\Vert_{v})_{v}$ on $L$. Write $\ov
L=(L,(\Vert\cdot\Vert_{v})_{v})$ for the resulting metrized line
bundle.

\begin{thm*}
  The height of $X$ with respect to $\ov L$ is given~by
  \begin{displaymath}
\h_{\ov L}(X)= (n+1)!\sum_{v\in
  \mathfrak{M}_{\Q}}\int_{\Delta}\vartheta_{v}\dd \Vol.
  \end{displaymath}
\end{thm*}

This theorem was announced in \cite{BPS09} and we prove it in
the present text. To establish it in a wide generality, we
have been led to study the Arakelov geometry of toric varieties.  In
the course of our research, we have found that a large part of the
arithmetic geometry of toric varieties can be translated in terms of
convex analysis.  In particular, we have added a number of new entries
to the arithmetic geometry chapter of the toric dictionary, including
models of toric varieties over a discrete valuation ring, metrized line
bundles, and their associated measures and heights.  These objects are
closely related to objects of convex analysis like polyhedral
complexes, concave functions, Monge-Amp\`ere measures and
Legendre-Fenchel duality.

These additions to the toric dictionary are very concrete and
well-suited for computations. In particular, they provide a new wealth of
examples in Arakelov geometry where constructions can be made explicit
and properties tested. In relation with explicit computations in
these examples, we present a closed
formula for the integral over a polytope of a function of one variable
composed with a linear form.  This formula allows us to compute the
height of toric varieties with respect to some interesting metrics
arising from polytopes. Some of these heights are related to the average
entropy of a simple random process on the polytope. We also
compute the height of toric projective curves with respect to the Fubini-Study
metric and of some toric bundles.

There are many other arithmetic invariants of toric varieties that may
be studied in terms of convex analysis. 
For instance, in the subsequent paper \cite{BMPS12} we give 
criteria for the positivity properties of a  toric
metrized line bundle and give a formula for its
arithmetic volume. 
In fact, we expect that the
results of this text are just the starting point of a program relating
the arithmetic geometry of toric varieties and convex analysis.
In this direction, we plan to obtain an arithmetic analogue of the BKK theorem
bounding the height of the solutions of a system of Laurent polynomial
equations, refining previous results in
\cite{Maillot:GAdvt,Sombra:msvtp}. 

\medskip
In the rest of this introduction, we will present the context and the
contents of our results. We will refer to the body of the
text for the precise definitions and statements.



Arakelov geometry provides a framework to define and study heights. We
leave for a moment the realm of toric varieties, and we consider a
smooth projective variety~$X$ over $\Q$ of dimension $n$ equipped with 
a regular proper integral model $\cX$. Let $X(\C)$ the analytic
space over the complex numbers associated to $X$. The main idea behind
Arakelov geometry is that the pair $(\cX,X(\C))$ should behave like a
compact variety of dimension $n+1$ \cite{Arakelov:itdas}. Following this philosophy,
Gillet and Soul\'e have developed an arithmetic intersection theory
\cite{GilletSoule:ait}.  As an application of this theory, one can
introduce a very general and precise definition, with a geometric
flavor, of the height of a variety \cite{BostGilletSoule:HpvpGf}. To
the model $\cX$, one
associates the arithmetic intersection ring $\cha^{*}(\cX)_{\Q}$. This
ring is equipped with a trace map $\int\colon \cha^{n+1}(\cX)_{\Q} \to
\R$. Given a line bundle $L$ on $X$, an arithmetic line bundle $\ov L$
is a pair $(\cL,\|\cdot\|)$, where $\cL$ is a line bundle on $\cX$
which is an integral model of $L$, and $\|\cdot\|$ is a smooth metric
on the analytification of $L$, invariant under complex conjugation. In
this setting, the 
analogue of the first Chern class of $L$ is the arithmetic first Chern
class $\afc(\ov L)\in \cha^{1}(\cX)_{\Q}$. The \emph{height} of $X$ with
  respect to $\ov L$ is then defined as
\begin{displaymath}
  \h_{\ov L}(X)=\int \afc(\ov L)^{n+1}\in \R.
\end{displaymath}
This is the arithmetic analogue of the degree of $X$ with respect to
$L$. This formalism has allowed to obtain arithmetic analogues of
important results in algebraic geometry like B\'ezout's theorem,
Riemann-Roch theorem, Lefschetz fixed point formula,
Hilbert-Samuel formula, etc.

This approach has two technical issues. In the first place, it only
works for smooth varieties and smooth metrics. In the second place, it
depends on the existence of an integral model, which puts the
Archimedean and non-Archimedean places in different footing.  For the
definition of heights, both issues were addressed by Zhang
by taking an adelic point of view and considering
uniform limits of semipositive metrics \cite{Zhang:_small}.

Many natural metrics that arise when studying line bundles on toric
varieties are not smooth, but are particular cases of the metrics
considered by Zhang. This is the case for the canonical metric of a
toric line bundle introduced by Batyrev and Tschinkel
\cite{BatyrevTschinkel:tori}, see Proposition-Definition \ref{def:57}.  
The associated canonical height of subvarieties plays an important
role in Diophantine approximation in tori, in particular in the
generalized Bogomolov and Lehmer problems, see for instance
\cite{DavidPhilippon:mhn,AmorosoViada:spst} and the references
therein.  Maillot has extended the arithmetic intersection theory of
Gillet and Soul\'e to this kind of metrics at the Archimedean place,
while maintaining the use of an integral model to handle the
non-Archimedean places \cite{Maillot:GAdvt}.

The adelic point of view of Zhang was developed by Gubler
\cite{GublerHab,Gu03} and by Chambert-Loir \cite{Cha06}.  From this
point of view, the height is defined as a sum of local
contributions. In what follows we outline this procedure, that will be recalled with
more detail in Chapter \ref{sec:metr-line-bundl}.

For the local case, let $K$ be either $\R$, $\C$, or a field complete
with respect to a nontrivial non-Archimedean absolute value. Let $X$
be a proper variety over $K$ and~$L$ a line bundle on $X$, and
consider their analytifications, respectively denoted by $X^{\an}$
and~$L^{\an}$.  In the Archimedean case, $X^\an$ is the complex space
$X(\C)$ (equipped with an anti-linear involution, if $K=\R$), whereas in the
non-Archimedean case it is the Berkovich space associated to $X$. The
basic metrics that can be put on $L^\an$ are the smooth metrics in the
Archimedean case, and the {algebraic metrics} in the non-Archimedean
case, that is, the metrics induced by an integral model of a pair
$(X,L^{\otimes e})$ with $e\ge1$.  There is a notion of semipositivity
for smooth and for algebraic metrics, and the uniform limit of such
metrics leads to the notion of \emph{semipositive{}}
  metric on $L^{\an}$. More generally, a metric on $L^{\an}$ is called 
\emph{DSP} (for \og difference of semipositive\fg{}) if it is the quotient of
two semipositive{} metrics.

Let $\ov L$ be a DSP metrized line bundle on $X$ and $Y$ a
$d$-dimensional cycle of~$X$. These data induce a (signed) measure on
$X^{\an}$, denoted $\chern_{1}(\ov L)^{\wedge d} \wedge\delta_{Y}$ by analogy
with the Archimedean smooth case, where it corresponds with the
current of integration along~$Y^{\an}$ of the $d$-th power of the first Chern
form.
This measure plays an important role in the distribution of points of
small height in the direction of the Bogomolov conjecture and its
generalizations, see for instance \cite{SUZ:epp,Bilu:ldspat,Yuan:blbav}.
Furthermore, if we have sections $s_{i}$, $i=0,\dots, d$, that meet $Y$
properly, one can define a notion of \emph{local height} $\h_{\ov
  L}(Y;s_0,\dots,s_d)$. The metrics and their associated
measures and local heights are related by the B\'ezout-type formula:
\begin{equation*}
  \h_{\ov L}(Y\cdot \div(s_{d});s_0,\dots,s_{d-1})= \h_{\ov L}(Y;s_0,\dots,s_d)+
  \int_{X^{\an}}\log||s_{d}|| \, \chern_1(\ov L)^{\wedge d} \wedge\delta_{Y}.
\end{equation*}

For the global case, consider a proper variety $X$ over $\Q$ and a line
bundle $L$ on $X$. For simplicity, assume that $X$ is projective, although
this hypothesis is not really necessary. A \emph{DSP
  quasi-algebraic} metric on $L$ is a family of DSP metrics
$\Vert\cdot\Vert_v$ on the analytic line bundles $L_v^{\an}$, $v\in
\mathfrak{M}_\Q$, such that there is an integral model of
$(X,L^{\otimes e})$, $e\ge1$, which induces $\|\cdot\|_{v}$ for all
but a finite number of $v$.  Write~$\ov L= (L, (\|\cdot
\|_{v})_{v})$, and $\ov L_{v}=(L_{v}, \|\cdot \|_{v})$ for each $v\in
\mathfrak{M}_{\Q}$. Given a $d$-dimensional cycle $Y$ of $X$, its
\emph{global height} is defined as
\begin{displaymath}
  \h_{\ov L}(Y)=\sum_{v\in \mathfrak{M}_{\Q}}\h_{\ov L_{v}}(Y;s_{0},\dots,s_{d}),
\end{displaymath}
for any family of sections $s_{i}$, $i=0,\dots, d$, meeting $Y$
properly. The fact that the metric is quasi-algebraic implies that the
right-hand side has only a finite number of nonzero terms, and the
product formula implies that this definition does not depend on the
choice of sections.  This notion can be extended to number fields,
function fields and, more generally, to
$M$-fields~\cite{Zhang:_small,Gu03}.

\medskip Now we review briefly the elements of the construction of
toric varieties from combinatorial data, see
Chapter \ref{sec:toric-varieties} for details.  Let $K$ be a field and
$\T\simeq \G_{m}^{n}$ a split torus over $K$.  Let $N
=\Hom(\G_{m},\T)\simeq \Z^n$ be the lattice of one-parameter subgroups
of $\T$ and $M=N^\vee$ the dual lattice of characters of $\T$.  Set
$N_{\R}=N\otimes_{\Z}{\R}$ and $M_{\R}=M\otimes_\Z \R$. To a fan
$\Sigma$ on $N_\R$ one can associate a {toric variety} $X_{\Sigma}$
of dimension $n$.  It is a normal variety that contains $\T$ as a
dense open subset, denoted $X_{\Sigma,0}$, and there is an action of
$\T$ on $X_{\Sigma}$ which extends the natural action of the torus on
itself. In particular, every toric variety has a distinguished point
$x_{0}$ that corresponds to the identity element of~$\T$. The variety
$X_{\Sigma }$ is proper whenever the underlying fan is {complete}.
For sake of simplicity, in this introduction we will restrict to the
proper case.

A Cartier divisor invariant under the torus action is called a
$\T$-Cartier divisor. In combinatorial terms, a $\T$-Cartier divisor
is determined by a {virtual support function on $\Sigma$}, that is, a
continuous function $\Psi\colon N_{\R}\rightarrow{\R}$ whose restriction to
each cone of~$\Sigma$ is an element of $M$.  Let $D_\Psi$ denote the
$\T$-Cartier divisor of $X_{\Sigma}$ determined by~$\Psi$.  A \emph{toric
line bundle} on $X_{\Sigma}$ is a line bundle $L$ on this toric variety, together with the
choice of a nonzero element $z\in L_{x_{0}}$. The total space of a toric
line bundle has a natural structure of toric variety whose
distinguished point agrees with $z$. A rational section of a
toric line bundle is called \emph{toric} if it is regular and nowhere
zero on the principal open subset $X_{\Sigma ,0}$, and $s(x_{0})=z$.  
Given a virtual support function $\Psi $, the line bundle
$L_{\Psi}=\cO(D_{\Psi})$ has a natural structure of toric line bundle
and a canonical toric section $s_{\Psi}$ such that
$\div(s_{\Psi })=D_{\Psi }$.  
Indeed, any line bundle on $X_{\Sigma}$ is
isomorphic to a toric line bundle of the form $L_{\Psi}$ for some $\Psi$.
The line bundle $L_{\Psi}$ is generated by global sections
(respectively, is ample) if and only if $\Psi$ is concave
(respectively, $\Psi$ is strictly concave on $\Sigma $).

Consider the lattice polytope
$$\Delta_\Psi=\{x\in M_{\R}: \langle x,u\rangle\geq\Psi(u) \mbox{ for all }
u\in N_{\R}\}\subset M_{\R}.$$ 
This polytope encodes
a lot of information about the pair $(X_\Sigma ,L_\Psi )$.
In case the virtual support function $\Psi$ is concave, it is determined by this
polytope, and the degree formula can be written more
precisely as  
\begin{equation*}
 \deg_{L_{\Psi }}(X_\Sigma)= n!\Vol_{M}(\Delta_\Psi ),  
\end{equation*} 
where the volume is computed with respect to 
the
Haar measure $\Vol_{M}$ on $M_\R$ normalized so that $M$ has covolume 1.

\medskip
In this text we extend the toric dictionary to
metrics, measures and heights as considered above.
For the local case, let $K$ be either $\R$, $\C$, or a field complete with
respect to a nontrivial non-Archimedean absolute value associated to
a discrete valuation. In this latter case, let
$K^{\circ}$ be the valuation ring, $K^{\circ\circ}$ its maximal ideal
and $\varpi$ a generator of $K^{\circ\circ}$. Let $\T$ be an
$n$-dimensional split torus over $K$, $X$ a toric variety over $K$ with torus $\T$,
and $L$ a toric line bundle on $X$.  The compact torus $\SS$ is
a closed analytic subgroup of the analytic torus $\T^{\an}$ (see
Example \ref{exm:4}) and it acts on
$X^{\an}$.  A metric $\|\cdot\|$ on $L^{\an}$ is \emph{toric} if, for
every toric section $s$, the function $\|s\|$ is invariant under the
action of $\SS$.

The correspondence that to a virtual support function assigns a toric
line bundle with a toric section can be extended to semipositive{} and
DSP metrics. Assume that $\Psi $ is concave, and let
$X_{\Sigma}$, $L_{\Psi}$ and $s_{\Psi}$ be as before. For short,
write $X=X_{\Sigma}$, $L=L_{\Psi }$ and $s=s_{\Psi }$.  There is a
fibration $$\val\colon X_{0}^{\an}\to N_{\R}$$ whose fibers are the
orbits of the action of $\SS$ on $X_{0}^{\an}$.  Now let
$\psiabs\colon N_{\R}\to \R$ be a continuous function.  We define a
metric on the restriction $L^{\an}|_{X_{0}^{\an}}$ by setting
\begin{displaymath}
\|s(p)\|_{\psiabs}=\e^{\psiabs(\val(p))}.   
\end{displaymath}

Our first addition to the toric dictionary is the following
classification result.  Assume that the function $\psiabs$ is concave and
that $|\psiabs-\Psi|$ is bounded. Then~$\|\cdot\|_{\psiabs}$
extends to a semipositive{} toric metric on $L^{\an}$ and,
moreover, every semipositive{} toric metric on $L^{\an}$ arises in this
way (Theorem~\ref{thm:13}\eqref{item:37}).  There is a similar
characterization of DSP toric metrics in terms of differences
of concave functions (Theorem \ref{thm:24}) and a characterization
of toric metrics that involves the topology of the
variety with corners associated to~$X_{\Sigma }$ (Proposition \ref{prop:14}).  As a
consequence of these classification results, we obtain a new
interpretation of the canonical metric of~$L^{\an}$ as the metric
associated to the concave function $\Psi $ under this correspondence.

We can also classify semipositive{} metrics in terms of concave
functions on polytopes: there is a bijective correspondence between
the space of continuous concave functions on~$\Delta_{\Psi}$ and the
space of semipositive{} toric metrics on $L^{\an}$
(Theorem~\ref{thm:13}\eqref{item:38}). This correspondence is induced
by the previous one and the Legendre-Fenchel duality of concave
functions. Namely, let $\|\cdot\|$ be a semipositive{} toric metric on
$L^{\an}$, write $\ov L=(L,\|\cdot\|)$ and $\psiabs $ the
corresponding concave function.  The associated \emph{roof
  function}~$\vartheta _{\ov L,s}\colon \Delta_{\Psi}\to\R$ is the
concave function defined as the Legendre-Fenchel
dual~$\psiabs^{\vee}$. One of the main outcomes of this text is that
the pair $(\Delta_{\Psi},\vartheta_{\ov L,s})$ plays, in the
arithmetic geometry of toric varieties, a role analogous to that of
the polytope in its algebraic geometry.

Our second addition to the dictionary is the following
characterization of the measure associated to a semipositive{} toric
metric. Let $X$, $\ov L$ and $\psiabs$ be as before, and write
\begin{math}
\mu_{\psiabs}=\chern_{1}^{\wedge
  n}(\ov L)
\wedge \delta_{X_{\Sigma}}  
\end{math}
for the induced measure on $X^{\an}$. Then  (Theorem \ref{thm:16})
$$(\val)_{*}(\mu_{\psiabs}|_{X^{\an}_{0}})=
n!\, \cM_{M}(\psiabs),$$ where $\cM_{M}(\psiabs)$ is the (real) Monge-Amp\`ere
measure of $\psiabs$ with respect to the lattice~$M$ (Definition
\ref{def:5}).  The measure $\mu_{\psiabs}$ is determined by this formula,
and the conditions of being invariant under the action of $\SS$ and
that the set $X^{\an}\setminus X_{0}^{\an}$ has measure zero. This
gives a direct and fairly explicit expression for the measure
associated to a semipositive{} toric metric.

The fact that each toric line bundle has a canonical metric allows us
to introduce a notion of local toric height that is independent of a
choice of sections. Let $X$ be an $n$-dimensional projective toric
variety and $\ov L$ a semipositive{} toric line bundle as before, and
let $\ov L^{\can}$ be the same toric line bundle $L$ equipped with the
canonical metric.  The \emph{toric local height} of $X$ with respect
to $\ov L$ is defined as
\begin{displaymath}
  \htor_{\ov L}(X)=\h_{\ov L}(X;s_{0},\dots,s_{n})-\h_{\ov
    L^{\can}}(X;s_{0},\dots,s_{n}), 
\end{displaymath}
for any family of sections  $s_{i}$, $i=0,\dots, n$, that meet
properly on $X$ (Definition \ref{def:35}).
Our third addition to the toric dictionary is the following formula for this toric
local height in terms of the roof function introduced above (Theorem \ref{thm:19}):
\begin{displaymath}
  \htor_{\ov L}(X)=(n+1)!\int_{\Delta _{\Psi}}\vartheta_{\ov L,s}\dd \Vol_{M}.
\end{displaymath}
More generally, the
toric local height can be defined for a family of $n+1$ DSP
toric line bundles on $X$. The formula above can be extended by
multilinearity to compute this toric local height in terms of the
mixed integral of the associated roof functions (Remark \ref{rem:7}).

\medskip For the global case, let $\Sigma$ and $\Psi$ be as before, and
consider the associated toric variety $X$ over $\Q$ equipped with a
toric line bundle $L$ and toric section $s$. Given a family of concave
functions $(\psiabs_{v})_{v\in \mathfrak{M}_{\Q}}$ such that
$|\psiabs_{v}-\Psi|$ is bounded for all $v$ and such that $\psiabs_{v}=\Psi$
for all but a finite number of $v$, the metrized toric line bundle
$\ov L=(L,(\Vert \cdot\Vert_{\psiabs_{v}})_{v})$ is
quasi-algebraic. Moreover, every semipositive{} quasi-algebraic toric
metric on $L$ arises in this way (Proposition \ref{prop:102} and
Theorem \ref{thm:23}).  The associated local roof functions
$\vartheta_{v,\ov L,s}\colon \Delta_{\Psi}\to \R$ are identically
zero except for a finite number of places.  Then, the global height of
$X$ with respect to $\ov L$ can be computed as (Theorem~\ref{thm:22})
\begin{displaymath}
  \h_{\ov L}(X)=\sum_{v\in \mathfrak{M}_{\Q}}\htor_{v,\ov L}(X_{v})=
  (n+1)!\sum _{v\in \mathfrak{M}_{\Q}}\int_{\Delta
    _{\Psi}}\vartheta_{v,\ov L,s} \dd \Vol_{M},
\end{displaymath}
which precises the theorem stated at the beginning of
this introduction. Here, $\htor_{v,\ov L}(X_{v})$ is the toric local
height for the place $v$.

\medskip
A remarkable feature of these results is that they read exactly the
same in the Archimedean and in the non-Archimedean cases.  For general
metrized line bundles, these two cases are analogous but not
identical. By contrast, the classification of toric metrics and the
formulae for the associated measures and local heights are the same in
both cases. We also point out that these results holds in greater
generality than explained in this introduction: in particular, they
hold for proper toric varieties which are not necessarily projective and, in
the global case, for general adelic fields (Definition \ref{def:6}).
We content ourselves with the case when the torus is split. For the
computation of heights, one can always reduce to the split case by
considering a suitable field extension.  Still, it would be
interesting to extend our results to the non-split case by considering
the corresponding Galois actions as, for instance,
in~\cite{ELFST:atv}.

\medskip

The toric dictionary in arithmetic geometry is very concrete and
well-suited for computations.  For instance, let $K$ be a local field,
$X$ a toric variety and $\varphi\colon X\to \P^{r}$ an equivariant
map.  Let $\ov L$ be the toric semipositive{} metrized line bundle on
$X$ induced by the canonical metric on the universal line bundle of
$\P^{r}$, and $s$ a toric section of~$L$. The concave function
$\psiabs\colon N_{\R}\to \R$ corresponding to this metric is piecewise
affine. Hence, it defines a polyhedral complex in $N_{\R}$, and it
turns out that~$(\val)_{*}(\mu_{\psiabs}|_{X_{0}^{\an}})$, the direct
image under $\val$ of the measure induced by $\ov L$, is a discrete
measure on $N_{\R}$ supported on the vertices of this polyhedral
complex (Proposition~\ref{prop:32}). The roof function $\vartheta_{\ov
  L,s}$ is the function parameterizing the upper envelope of a
polytope in $M_{\R}\times \R$ associated to $\varphi$ and the section
$s$ (Example \ref{exm:36}). The toric local height of $X$ with respect
to $\ov L$ can be computed as the integral of this piecewise affine
concave function.

Another nice example is given by toric bundles on a projective
space. For a finite sequence of integers $a_r\ge \dots\ge a_0\ge 1$,
we consider the vector bundle on $\P^{n}_{\Q}$
$$
E=\mathcal{O}(a_0)\oplus\dots \oplus\mathcal{O}(a_r).
$$
The toric bundle $\P(E)\rightarrow\P^n_{\Q}$ is defined as the bundle of
hyperplanes of the total space of $E$.  This is an $(n+r)$-dimensional
toric variety over $\Q$ which can be equipped with an ample universal line
bundle $\mathcal{O}_{\P(E)}(1)$, see \S\ref{Height of toric bundles}
for details.

We equip 
${\mathcal{O}_{\P(E)}(1)}$ with a semipositive{} adelic toric metric as
follows: the Fubini-Study metrics on each line bundle
$\mathcal{O}(a_j)$ induces a semipositive smooth toric metric on
$\mathcal{O}_{\P(E)}(1)$ for the Archimedean place of $\Q$, whereas
for the finite places we consider the corresponding canonical
metric. We show that both the corresponding
concave functions $\psiabs_{v}$ and roof functions $\vartheta_{v}$ can be
described in explicit terms (Lemma~\ref{lemm:16} and Proposition
\ref{prop:67}).  We can then compute the height of~$\P(E)$ with
respect to this metrized line bundle as (Theorem~\ref{fibrestoriques})
\begin{equation*}
  \h_{{\overline{\cO_{\P(E)}(1)}}}(\P(E)) =
\h_{\ov{\cO(1)}}(\P^n)  \hspace{-1.5mm} \sum_{{\boldsymbol{i}\in\N^{r+1}}\atop{|\boldsymbol{i}|=n+1}}
  \boldsymbol{a}^{\boldsymbol{i}} +
  \sum_{{\boldsymbol{i}\in\N^{r+1}}\atop{|\boldsymbol{i}|=n}}
A_{n,r}(\boldsymbol{i})  \, \boldsymbol{a}^{\boldsymbol{i}},
\end{equation*}
where for $\boldsymbol i=(i_{0},\dots, i_{r})\in \N^{r+1}$, we set
$|\boldsymbol{i}|=i_0+\dots+i_r$, $\boldsymbol{a}^{\boldsymbol
  i}=a_0^{i_0}\dots a_r^{i_r}$ and
$A_{n,r}(\boldsymbol{i})=\sum_{m=0}^r(i_m+1)\sum_{j=i_m+2}^{n+r+1}\frac{1}{2j}$,
while
$\h_{\ov{\cO(1)}}(\P^n_{\Q})=\sum_{i=1}^n\sum_{j=1}^i\frac{1}{2j}$
denotes the height of the projective space with respect to the
Fubini-Study metric.  In
particular, the height of $\P(E)$ is a positive rational number.

\medskip The Fubini-Study height of the projective space was computed
by Gillet and Soul\'e \cite[\S 5.4]{GilletSoule:cc}.  Other early
computations for the 
Fubini-Study height of some toric hypersurfaces where obtained
in~\cite{Dan97,CM00}.  Mourougane has determined the height of
Hirzebruch surfaces, as a consequence of his computations of Bott-Chern secondary
classes~\cite{Mou06}. A Hirzebruch surface is a toric bundle over
$\P^{1}_{\Q}$, and the result of Mourougane is a particular case of our
computations for the height of toric bundles (Remark \ref{rem:24}).

The fact that the canonical height of a toric variety is zero is
well-known. It results from its original construction by a limit process 
on the direct images of the toric variety under the so-called ``powers maps''. 
Maillot has studied the Arakelov geometry of toric varieties and line
bundles with respect to the canonical metric, including the
computation of the associated 
Chern currents and their product~\cite{Maillot:GAdvt}.

In \cite{PS08a}, Philippon and Sombra gave a formula for the canonical
height of a ``translated'' toric projective variety, a projective variety which is
the closure of a translate of
a subtorus, defined over a number field. In \cite{MR2419926}, they also
obtain a similar formula for the function field
case. Both results are particular cases of our general formula 
(Remark \ref{rem:9}). Indeed, part of our motivation for the present
text was to understand and generalize this formula in the framework of Arakelov
geometry.

For the Archimedean smooth case, our constructions are related to the 
Guillemin-Abreu classification
of K\"ahler structures on {symplectic toric varieties}~\cite{MR1969265}.
The roof function corresponding to a smooth
metrized line bundle on a smooth toric variety coincides, up to a
sign, with the so-called ``symplectic potential'' of a K\"ahler toric
variety (Remark \ref{rem:12}). 
In the Archimedean continuous case, 
Boucksom and Chen have recently considered
a similar construction
in their study of arithmetic 
Okounkov bodies~\cite{BoucksomChen:Obfs}. 
It would be interesting to further explore the connection with these
results.

\medskip
We now discuss the contents of each chapter, including some 
other results of interest. 

Section~\ref{sec:metr-line-bundl} is devoted to the first half of the
dictionary. Namely, we review DSP metrized line bundles both in the
Archimedean and in the non-Archimedean cases. For the latter case, we
recall the basic properties of Berkovich spaces of schemes. We then explain the
associated measures and heights following
\cite{Zhang:_small,Cha06,Gu03}. For simplicity, the theory presented is not as 
general as the one in~\cite{Gu03}: in the non-Archimedean case we
restrict ourselves to discrete valuation rings and in the global
case to adelic fields, while in \emph{loc. cit.} the theory is
developed for arbitrary valuations and for $M$-fields, respectively. 

Section \ref{sec:legendre-fench-dual} deals with the second half of
the dictionary, that is, convex analysis with emphasis on polyhedral
sets. Most of the material in this section is classical. We have
gathered all the required results, adapting them to our needs and adding
some new ones. We work with {concave} functions, which are the
functions which naturally arise in the theory of toric varieties. For
later reference, we have translated many of the notions and results
of convex analysis, usually stated for convex functions, in terms of
concave functions.

We first recall the basic definitions about convex sets and convex
decompositions, and then we study concave functions and the
Legendre-Fenchel duality. We introduce a notion of Legendre-Fenchel
correspondence for general closed concave functions, as a duality
between convex decompositions (Definition \ref{def:26} and Theorem
\ref{thm:8}). This is the right generalization of both the classical
Legendre transform of strictly concave differentiable functions, and the
duality between polyhedral complexes induced by a piecewise affine
concave function. We also consider the interplay between
Legendre-Fenchel duality and operations on concave functions like, for
instance, the direct and inverse images by affine maps. This latter
study will be important when considering the functoriality with
respect to equivariant morphisms between toric varieties. We next
particularize to two extreme cases: differentiable concave functions whose
stability set is a polytope that will be related to semipositive
smooth toric metrics in the Archimedean case, and to piecewise affine
concave functions that will correspond to semipositive algebraic toric
metrics in the non-Archimedean case. 
Next, we treat differences of concave functions, 
that will be related to DSP metrics.  We end this section by
studying the Monge-Amp\`ere measure associated to a concave
function. There is an interesting interplay between Monge-Amp\`ere
measures and Legendre-Fenchel duality. In this direction, we prove a
combinatorial analogue of the arithmetic B\'ezout's theorem (Theorem
\ref{thm:20}), which is a key ingredient in the proof of our formulae
for the height of a toric variety.

In Chapter \ref{sec:toric-varieties}, we study the algebraic geometry of
toric varieties over a field and of toric schemes over a discrete
valuation ring (DVR).  We start by recalling the basic constructions
and results on toric varieties, including Cartier and Weil divisors,
toric line bundles and sections, orbits and equivariant morphisms, and
positivity properties. Toric schemes over a DVR where first considered
by Mumford in \cite{Kempfals:te}, who studied and classified them in
terms of fans in $N_{\R}\times \R_{\ge0}$. In the proper case, these
schemes can be alternatively classified in terms of complete
polyhedral complexes in~$N_{\R}$~\cite{BurgosSombra:rc}. Given a
complete fan $\Sigma$ in $N_{\R}$, the models over a DVR of the proper
toric variety $X_{\Sigma}$ are classified by complete polyhedral
complexes on $N_{\R}$ whose recession fan (Definition \ref{def:69})
coincides with $\Sigma$ (Theorem \ref{thm:7b}). Let $\Pi$ be such a
polyhedral complex, and denote by $\cX_{\Pi}$ the corresponding model
of $X_{\Sigma}$.  Let~$\Psi$ be a virtual support function on $\Sigma
$ and $(L,s)$ the associated toric line bundle on~$X_{\Sigma}$ and
toric section. We show that the models of $(L,s)$ over $\cX_{\Pi}$ are
classified by functions that are rational piecewise affine on $\Pi$
and whose recession function is $\Psi$ (Theorem~\ref{thm:11b}).  We
also prove a toric version of the Nakai-Moishezon criterion for toric
schemes over a DVR, which implies that semipositive models of $(L,s)$ translate into concave functions under the above
correspondence (Theorem \ref{thm:5}).

In Chapter \ref{sec:appr-integr-metr}, we study toric metrics and their
associated measures.  For the present discussion, consider a local
field $K$, a complete fan $\Sigma$ on $N_{\R}$ and a virtual support
function $\Psi$ on~$\Sigma$, and let $(X,L)$ denote the corresponding
proper toric variety over $K$ and toric line bundle. We first
introduce a variety with corners $N_{\Sigma}$ which is a
compactification of $N_{\R}$, together with a proper map $\val\colon
X_{\Sigma}^{\an}\to N_{\Sigma}$ whose fibers are the orbits of the
action of $\SS$ on $ X_{\Sigma}^{\an}$. We first treat the
problem of obtaining a toric metric from a non-toric one
(Proposition~\ref{prop:60}) and prove the classification theorem for
toric metrics on~$L^{\an}$ (Proposition \ref{prop:14}).  We next treat
smooth metrics in the Archimedean case.  A toric smooth metric is
semipositive if and only if the associated function $\psiabs$ is concave
(Proposition \ref{prop:15}). We make explicit the associated measure
in terms of the Hessian of this function, hence in terms of the
Monge-Amp\`ere measure of~$\psiabs$ (Theorem \ref{thm:18}).  We also
observe that an arbitrary smooth metric on $L$ can be turned into a
toric smooth metric by averaging it by the action of $\SS$. If
the given metric is semipositive, so is the obtained toric smooth
metric.

Next, in the same section, we consider algebraic metrics in the
non-Archimedean case. We first show how to describe the reduction map
for toric schemes over a DVR in terms of the corresponding polyhedral
complex and the map $\val$ (Lemma~\ref{lemm:3}). We then study the
triangle formed by toric metrics, rational piecewise affine functions
and toric models (Proposition \ref{prop:59} and Theorem \ref{thm:15})
and the effect of taking a field extension (Proposition
\ref{prop:23}). Next, we treat in detail the one-dimensional case,
were one can write in explicit terms the metrics, associated functions
and measures. Back to the general case, we use these results to
complete the characterization of toric semipositive algebraic metrics
in terms of piecewise affine concave functions (Proposition
\ref{prop:61}).  We also describe the measure associated to a
semipositive toric algebraic metric in terms of the Monge-Amp\`ere
measure of its associated concave function (Theorem \ref{thm:17}).

Once we have studied smooth metrics in the Archimedean case and
algebraic metrics in the non-Archimedean case, we can study
semipositive{} toric metrics. We show that the same classification
theorem is valid in the Archimedean and non-Archimedean cases (Theorem
\ref{thm:13}). Moreover, the associated
measure is described in exactly the same way in both cases (Theorem
\ref{thm:16}). We end this section by introducing and classifying
adelic toric metrics (Definition \ref{def:59}, Proposition
\ref{prop:102} and Theorem \ref{thm:23}).

In Chapter \ref{sec:heigh-toric-vari}, we prove the formulae for the 
toric local height and for the global height of toric varieties
(theorems \ref{thm:19} and \ref{thm:22}). By using the functorial
properties of the height, we recover, from our general formula, the
formulae for the canonical height of a translated toric projective
variety in \cite[Th\'eor\`eme~0.3]{PS08a} for number fields and
in \cite[Proposition 4.1]{MR2419926} for function fields.
 
In Chapter \ref{Metricfrompolytopes}, we consider the problem of
integrating functions on polytopes. We first present a closed formula
for the integral over a polytope of a function of one variable
composed with a linear form, extending in this direction Brion's
formula for the case of a simplex \cite{Brion:PointsEntiers}
(Proposition \ref {propdefcoeffdintegration} and Corollary
\ref{cor:16}).  This allows us to compute the height of toric
varieties with respect to some interesting metrics arising from
polytopes (Proposition \ref{prop:69}).  We can interpret
some of these heights as the average entropy of a simple random
process defined by the polytope (Proposition~\ref{prop:71}).

In Chapter \ref{courbes}, we study some further examples.  We first
consider translated toric curves in~$\P^{n}_{\Q}$. For these curves,
we consider the line bundle obtained from the restriction of~$\cO(1)$
to the curve, equipped with the metric induced by the Fubiny-Study
metric at the place at infinity and by the canonical metric for the
finite places. We compute the corresponding concave function $\psiabs$
and toric local height in terms of the roots of a univariate
polynomial (Theorem \ref{thm:26}).  We finally consider toric bundles
as explained before, and compute the relevant concave functions,
measures and heights.

\vspace{3mm}
\noindent{\bf Acknowledgements.} 
Jacques-Arthur Weil assisted us, at the beginning of this project, in
the computation of the height of some toric curves. Richard Thomas
called our attention to the relationship between our formulae and the
Legendre-Fenchel duality.  We thank both of them. We also thank
Antoine Chambert-Loir, Teresa Cortadellas, Carlos D'Andrea, Antoine
Ducros, Walter Gubler, Qing Liu, Vincent Maillot, Juan Carlos Naranjo
and Angelo Vistoli for useful discussions and pointers to the
literature. Last but not least, we thank the referee for the thorough
reading of our manuscript and the numerous remarks and corrections
that have significantly improved this text.

Part of this work was done while the authors met at the Universitat de
Barcelona, the Centre de Recerca Matem\`atica (Barcelona), the
Institut de Math\'ematiques de Jussieu (Paris) and the Universit\'e
de Bordeaux.  Short courses on parts of this text were delivered at
the Morningside Center of Mathematics (Beijing), the Centro de
Investigaci\'on en Matem\'aticas (Guanajuato) and the Universidad de
Buenos Aires. We thank all of these institutions for their
hospitality.

\vfill

\pagebreak
\ 
\vskip 1cm
\begin{center}
\bfseries
\large
CONVENTIONS AND NOTATIONS  
\end{center}

\vskip 3cm

For the most part, we follow  generally accepted conventions and
notations. We also use the following:

\begin{itemize}

\item $\N$ and $\N^{\times}$ denote the set of natural numbers with 0
  and without 0, respectively;

\item a multi-index is an element $i\in \N^{S}$, where $S$ is a finite
  set. For a multi-index $i$, we write $|i|=\sum_{s\in S}i_{s}$;

\item A semigroup is a set with an associative binary operation
and an identity element. In particular, a morphism of semigroups will
send the identity element to the identity element. All considered
semigroups will be commutative; 

\item all considered rings are commutative and with a unit;

\item a scheme is a separated scheme of finite type over a
  Noetherian ring;

\item a variety is a reduced and irreducible separated scheme of
  finite type over a field;

\item $\P^{n}$ is a projective space of dimension $n$ over a
  field, with a fixed choice of homogeneous coordinates;

\item by a line bundle we mean a locally free sheaf of rank one;

\item compact spaces are Hausdorff; 

\item measures are non-negative and a signed measure is a difference
  of two measures.
\end{itemize}

For the notations and terminology introduced in this text, the reader
can locate them using the list of symbols and the index at the end of
the book.


\chapter{Metrized line bundles and their associated heights}
\label{sec:metr-line-bundl}

In this chapter,  we recall the adelic theory of heights as
introduced by Zhang \cite{Zhang:_small} and developed by Gubler
\cite{GublerHab,Gu03} and Chambert-Loir~\cite{Cha06}. These heights generalize
the ones that can be obtained from the arithmetic
intersection theory of Gillet and
Soul\'e~\cite{GilletSoule:ait,BostGilletSoule:HpvpGf}.

To explain the difference between both points of view, consider a
smooth variety~$X$ over $\Q$. In Gillet-Soul\'e's theory, we choose a
regular proper model $\cX$ over $\Z$ of $X$, and we also consider the
real analytic space $X^{\an}$ given by the set of complex points
$X(\C)$ and the anti-linear involution induced by the complex
conjugation. By contrast, in the adelic point of view we consider the
whole family of analytic spaces~$X^{\an}_{v}$, $v\in {\mathfrak
  M}_{\Q}$. For the Archimedean place, $X^{\an}_{v}$ is the real analytic space
considered before, while for the non-Archimedean places, it
is the associated Berkovich space \cite{Berkovich:stag}. Both points
of view have advantages and disadvantages. In the former point of
view, there exists a complete formalism of intersection theory and
characteristic classes, with powerful theorems like the arithmetic
Riemann-Roch theorem
and the Lefschetz fixed point theorem,
but one is restricted to smooth
varieties and needs an explicit integral model of $X$. In the
latter point of view, one can define heights, but does not dispose yet
of a complete formalism of intersection theory. Its main advantages
are that it can be easily extended to non-smooth varieties and that
there is no need of an integral model of $X$. Moreover, all places,
Archimedean and non-Archimedean, are set on a similar footing.

\section{Smooth metrics in the Archimedean case}
\label{sec:arqu-smooth-metr}

Let $X$ be a variety over $\C$ and $X^{\an}$
\nomenclature[aX01]{$X^{\an}$}{analytification of a variety over $\C$}%
\index{analytification!of a variety over $\C$}%
its associated complex analytic space. We recall the
definition of differential forms
\index{differential forms on a complex space}%
on
$X^{\an}$ introduced by Bloom and Herrera \cite{BloomHerrera:drcas}.
The space~$X^{\an}$ can be covered by a family of open subsets
$\{U_{i}\}_{i}$ such that each $U_{i}$ can be identified with a closed
analytic subset of an open ball in $\C^r$ for some $r$. On
each~$U_{i}$, the differential forms are defined as the restriction to
this subset of smooth complex-valued differential forms defined on an
open neighbourhood of $U_{i}$ in $\C^r$. Two differential forms on
$U_{i}$ are identified if they coincide on the non-singular locus
of~$U_{i}$. We denote by $\mathscr{A^{\ast}}(U_{i})$ the complex of
differential forms of $U_{i}$, which is independent of the chosen
embedding. In particular, if $U_{i}$ is
non-singular, we recover the usual complex of differential forms.
These complexes glue together to define a
sheaf~$\mathscr{A}^{\ast}_{X^{\an}}$.
\nomenclature[aAan]{$\mathscr{A}^{\ast}_{X^{\an}}$}{sheaf of
  differential forms of a complex space}%
This sheaf is equipped with
differential operators $\dd$, $\dc$, $\partial$, $\bar \partial$, an
external product and inverse images with respect to analytic
morphisms: these operations are defined locally on each
$\mathscr{A^{\ast}}(U_{i} )$ by extending the differential forms to a
neighbourhood of~$U_{i}$ in~$\C^{r}$ and applying the
corresponding operations for $\C^{r}$.  We write $\mathcal
O_{X^{\an}}$\nomenclature[aOX2]{$\mathcal O_{X^{\an}}$}{sheaf of analytic
  functions of a complex space} and $C^{\infty}_{X^{\an}}=
\mathscr{A}^{0}_{X^{\an}}$
\nomenclature[aCinf]{$C^{\infty}_{X^{\an}}$}{sheaf of smooth functions of a complex space}%
for the sheaves of analytic
functions and of smooth functions of~$X^{\an}$, respectively.

Let $L$ be an algebraic line bundle on $X$ and
$L^{\an}$
\nomenclature[aL1]{$L$}{line bundle}%
\nomenclature[aLan1]{$L^{\an}$}{analytification of a line bundle over $\C$}%
its analytification.
\index{analytification!of a line bundle}%

\begin{defn}\label{def:58} A \emph{metric}
\index{metric!on a line bundle!over $\C$}%
\nomenclature[svert11]{$\Vert\cdot\Vert$}{metric on a line bundle (Archimedean case)}%
  on $L^{\an}$ is
  an assignment that, to each open subset $U\subset X^{\an}$ and local
  section $s$ of $L^{\an}$ on $U$,
  associates a continuous function
  \begin{displaymath} \|s(\cdot)\|\colon U \longrightarrow
\mathbb{R}_{\ge 0}
  \end{displaymath} such that
  \begin{enumerate}
  \item it is compatible with the restrictions to smaller open subsets;
  \item for all $p\in U$, $\|s(p)\|=0$ if and only if $s(p)=0$;
  \item for any $p\in U$ and $\lambda \in
\mathcal{O}_{X^{\an}}(U)$,  it holds
    \begin{math} \|(\lambda s)(p)\|=|\lambda (p)|\, \|s(p)\|.
    \end{math}
  \end{enumerate}
The pair $\ov
  L:=(L,\|\cdot\|)$
  \nomenclature[aLb20]{$\ov L$}{metrized line bundle (Archimedean case)}%
  is called a \emph{metrized line bundle}.
  The metric $\|\cdot\|$ is \emph{smooth}
\index{smooth metric}%
\index{metric!smooth|see{smooth metric}}%
if for every local section $s$ of $L^{\an}$,
  the function $\|s(\cdot)\|^{2}$ is smooth.
\end{defn}

We remark that what we call ``metric'' in this text is called ``continuous
metric'' in other contexts. 

Let $\ov L=(L,\|\cdot\|)$ be a smooth metrized line bundle.  
Given a non-vanishing local section $s$ of $L^{\an}$ on an open
subset $U$, the
\emph{first Chern form}
\index{Chern form}%
of $\ov L$ is the $(1,1)$-form defined on $U$ as
$$
\chern_1(\ov{L})
= \partial\bar \partial\log\|s\|^2\in \mathscr{A}^{1,1}(U).
$$
\nomenclature[ac1L]{$\chern_1(\ov{L})$}{Chern form of a smooth metrized line bundle}%
It does not depend on the choice of local section and can be extended
to a global closed $(1,1)$-form. Observe that we are using the
algebro-geometric convention, and so $\chern_1(\ov{L})$ determines a
class in $H^{2}(X^{\an},2\pi i\, \Z )$.

\begin{exmpl}\label{exm:3}
  Let $X=\P^n_\C$ and $L=\cO(1)$, the universal line bundle of
  $\P^{n}_{\C}$.  A rational section $s$ of $\cO(1)$ can be identified
  with a homogeneous rational function $\rho_{s}\in \C(x_{0},\dots,
  x_{n})$ of degree 1. The poles of this section coincide which those
  of $\rho_{s}$. For a point $p=(p_0:\dots:p_n)\in \P^n(\C)$ and a
  rational section $s$ as above which is regular at $p$, the \emph{Fubini-Study
    metric} of $\cO(1)^{\an}$
  \index{metric!Fubini-Study}%
  is defined as
  \begin{displaymath}
    \|s(p)\|_\FS =\frac{|\rho_{s}(p_0,\dots,p_n)|}{(\sum_{i=0}^{n}|p_i|^2)^{1/2}}.
  \end{displaymath}
\nomenclature[svert31]{$\Vert\cdot\Vert_\FS$}{Fubini-Study metric of $\cO(1)^{\an}$}%
  Clearly, this definition does not depend on the choice of a
  representative of $p$. The pair $(\cO(1), \|\cdot\|_{\FS})$ is a
  metrized line bundle. 
\end{exmpl}

Many smooth metrics can be obtained as the inverse image of the
Fubini-Study metric. Let $X$ be a variety over $\C$ and $L$ a line
bundle on $X$, and assume that there is an integer $e\ge 1$ such that
$L^{\otimes e}$ is generated by global sections. Choose a basis of the
space of global sections $\Gamma (X,L^{\otimes e})$ and let $\varphi
\colon X\to \P^{M}_{\C}$ be the induced morphism. Given a local
section $s$ of $L$, let $s'$ be a local section of $\cO(1)$ such that
$s^{\otimes e}=\varphi^{\ast} s'$. Then, the smooth metric on $L^{\an}$ obtained
from the Fubini-Study metric by
inverse image is given by
\begin{equation*}
  \|s(p)\|=\|s'(\varphi(p))\|^{1/e}_{\FS}
\end{equation*}
for any $p\in X^{\an}$ which is not a pole of $s$. 

\begin{defn} \label{def:1} Let $\ov L$ be a smooth metrized line
  bundle on $X$ and
  $\D=\{z\in \C| \, |z|\le 1\}$, the unit disk of $\C$.
  \nomenclature[aD9]{$\D$}{unit disk of $\C$}%
 We say
  that $\ov L$
  is \emph{semipositive}
\index{smooth metric!semipositive}%
if, for every holomorphic map
  \begin{math}
    \varphi\colon \mathbb{D}\longrightarrow X^{\an},
  \end{math}
  \begin{equation*}
   \frac{1}{2\pi i} \int_{\D}\varphi^{\ast}\chern_{1}(\ov L)\ge 0.
  \end{equation*} 
  We say that $\ov L$ is \emph{positive}
  \index{smooth metric!positive}%
  if this integral is strictly
  positive for all non-constant holomorphic maps as before.
\end{defn}

\begin{exmpl}
  The Fubini-Study metric (Example \ref{exm:3}) is positive because
  its first Chern form defines a smooth metric on the holomorphic
  tangent bundle of $\P^{n}(\C)$ \cite[Chapter 0, \S2]{GriffithsHarris:pag}. All metrics obtained as inverse image of
  the Fubini-Study metric are semipositive.
\end{exmpl}

A family of smooth metrized line bundles $\ov L_0,\dots, \ov L_{d-1}$ on $X$
and a $d$-dimensional cycle~$Y$ of $X$ define a
signed measure on $X^{\an}$ as follows.  First suppose that $Y$ is a
subvariety of $X$ and let $\delta _{Y}$ denote the current of
integration along the analytic subvariety $Y^{\an}$,
defined as $\delta _{Y}(\omega )=\frac{1}{(2\pi
  i)^{d}}\int_{Y^{\an}}\omega$ for $\omega \in
\mathscr{A}^{2d}_{X^{\an}} $.  Then the current
$$
\chern_1(\ov L_0)\wedge \cdots\wedge \chern_1(\ov L_{d-1})\wedge \delta _{Y}
$$
is a {signed measure} on $X^{\an}$.
\index{smooth metric!signed measure associated to}%
\index{measure!associated to a smooth metric}%
\nomenclature[ac1L0]{$\chern_1(\ov L_0)\wedge \cdots\wedge
  \chern_1(\ov L_{d-1})\wedge \delta _{Y}$}{signed measure (smooth case)}%
This notion extends by linearity to $Y\in
Z_{d}(X)$.  If~$\ov L_i$, $i=0,\dots, d-1$, are semipositive and $Y$ is effective,
this signed measure is a measure.

\begin{rem} 
\label{rem:17} 
We can reduce the study of algebraic varieties and
line bundles over the field of real numbers
\index{algebraic variety over $\R$}%
to the complex case by
using the following standard technique. 
A variety $X$ over $\R$ induces a variety $X_{\C}$ over $\C$ together
with an anti-linear involution $\varsigma\colon X_{\C} \to X_{\C}$ such
that the diagram 
\begin{displaymath}
  \xymatrix{X_{\C} \ar[r]^{\varsigma} \ar[d] & X_{\C} \ar[d] \\
\Spec(\C) \ar[r]^{\varsigma} & \Spec(\C)}
\end{displaymath}
commutes, where the arrow below denotes the map induced by
complex conjugation.  Following this philosophy, we define the
analytification of $X$ as $X^{\an}_{\R}=(X^{\an}_{\C},\varsigma )$.

A line bundle $L$ on $X$ determines a line
bundle $L_{\C}$ on $X_{\C}$ and an isomorphism $\alpha\colon
\varsigma^{*}L_{\C}\to L_{\C}$ such that a section $s$ of $L_{\C}$ is
real if and only if $\alpha(\varsigma^{*} s) = s$. Thus we define
$L^{\an}_{\R}=(L^{\an}_{\C},\alpha )$. 
By a \emph{metric}
\index{algebraic variety over $\R$!metrized line bundle on}%
\index{metrized line bundle!on an algebraic variety over $\R$}%
on $L^{\an}_{\R}$ we will mean a metric $\|\cdot\|$ on $L_\C^{\an}$ such that
the map $\alpha \colon \varsigma^{*}(L_{\C}, \|\cdot\|)\to (L_{\C}, \|\cdot\|)$
is an isometry.
\nomenclature[g1800]{$\varsigma $}{anti-linear involution defined by a
  variety over $\R$}%

In this way, the above definitions can be extended to metrized line
bundles on varieties over $\R$. For instance, a real
smooth metrized line bundle is semipositive if and only if its associated
complex smooth metrized line bundle is semipositive.  The corresponding
signed measure is a measure over $X_{\C}^{\an}$ which is invariant
under~$\varsigma$. 

In the sequel, every time we have a variety over $\R$, we will work
instead with the
associated complex variety and quietly ignore the anti-linear involution
$\varsigma $, because it will not play an important role in our
results. In particular, 
if $X$ is a real variety, we will denote $X^{\an}=X_{\C}^{\an}$ for the
underlying complex space of 
$X^{\an}_{\R}$. Similarly, we will denote $L^{\an}=L^{\an}_{\C}$.
\end{rem}

\section{Berkovich spaces of schemes}
\label{sec:berkovich-spaces}
\index{Berkovich space|(}%
In this section we recall Berkovich's theory of analytic spaces. We will not
present the most general theory developed in \cite{Berkovich:stag} but we will
content ourselves with the analytic spaces associated to algebraic
varieties, that are simpler to define and enough for our purposes.

Let $K$ be a field which is complete with respect to
a nontrivial non-Archimedean absolute value $|\cdot|$. 
Such fields will be called \emph{non-Archimedean fields}.
\index{non-Archimedean field}%
Let $
K^{\circ}=\{\alpha \in K\mid |\alpha |\le 1\} $
\nomenclature[aKna2]{$K^{\circ}$}{valuation ring of a non-Archimedean field}%
be the
{valuation ring},\index{valuation ring} $ K^{\circ\circ}=\{\alpha \in
K\mid |\alpha |< 1\} $
\nomenclature[aKna3]{$K^{\circ\circ}$}{maximal ideal of a valuation ring}%
the {maximal ideal}
\index{valuation ring!maximal ideal of}%
and $k=K^{\circ}/K^{\circ\circ}$
\nomenclature[aK]{$k$}{residue field of a valuation ring}%
the {residue field}.
\index{valuation ring!residue field of}%

Let $X$ be a scheme of finite type over $K$. Following
\cite{Berkovich:stag}, we can associate an \emph{analytic 
space} $X^{\an}$ 
\nomenclature[aX02]{$X^{\an}$}{Berkovich space of a scheme}%
to the scheme $X$ as follows.  First assume that
$X=\Spec(A)$, where $A$ is a finitely generated $K$-algebra. Then, the
points of~$X^{\an}$
\index{Berkovich space!points of}%
are the multiplicative seminorms of $A$ that
extend the absolute value of $K$, see
\cite[Remark~3.4.2]{Berkovich:stag}. Every element $a$ of $A$ defines a
function $|a(\cdot)|\colon X^{\an}\to \R_{\ge 0}$ given by evaluation
of the seminorm. The topology of $X^{\an}$
\index{Berkovich space!topology of}%
is the coarsest topology
that makes the functions $|a(\cdot)|$ continuous for all $a\in A$.

A  point $p\in X^{\an}$ defines a prime ideal
\begin{math}
 \{a\in A\mid |a(p)|=0\}\subset A. 
\end{math}
This induces a map 
\begin{displaymath}
\pi \colon X^{\an}\to X=\Spec(A).  
\end{displaymath}
\nomenclature[g16000]{$\pi$}{map from $X^{\an}$ to $X$}%
Let $K(\pi(p))$ denote the function field of $\pi(p)$, that is, the
field of fractions of the quotient ring $A/\pi(p)$.  The point $p$ is a
multiplicative seminorm on $A$ and so it induces a non-Archimedean
absolute value on $K(\pi(p))$.  We denote by $\mathscr{H}(p)$
\index{Berkovich space!field associated to a point of}%
\nomenclature[aHp]{$\mathscr{H}(p)$}{field associated to a point of a
  Berkovich space}%
the completion of this field with respect to that absolute value.

Let $U$ be an open subset of $X^{\an}$. An \emph{analytic function} on
$U$ is a function
\index{Berkovich space!analytic function on}%
$$
f\colon U\longrightarrow \coprod_{p\in {U}}\mathscr{H}(p)
$$
such that, for each $p\in U$, $f(p)\in \mathscr{H}(p)$ and there is an
open neighbourhood $U'\subset U$ of~$p$ with the property that, for all
$\varepsilon >0$ and $q\in U'$, there are elements $a,b\in A$
with~$b\not \in \pi({q})$ and $|f(q)-a(q)/b(q)| <\varepsilon $.  The
analytic functions form a sheaf, denoted~$\mathcal{O}_{X^{\an}}$,
\nomenclature[aOX3]{$\mathcal{O}_{X^{\an}}$}{sheaf of analytic
  functions of a Berkovich space}%
and $(X^{\an},\mathcal{O}_{X^{\an}})$ is a locally ringed
space~\cite[\S 1.5 and Remark~3.4.2]{Berkovich:stag}. In particular,
every element $a\in A$ determines an analytic function on $X^{\an}$,
also denoted $a$. The function $|a(\cdot)|$ can then be obtained by
composing $a$ with the absolute value map
\begin{displaymath}
  |\cdot|\colon \coprod_{p\in X^{\an}}\mathscr{H}(p)\longrightarrow \R_{\ge 0},
\end{displaymath}
which justifies its notation.

Now, if $X$ is a scheme of finite type over $K$, the analytic space
$X^{\an}$ 
\index{analytification!of a variety over a non-Archimedean field}%
is defined by gluing together the affine analytic spaces
obtained from an affine open cover of~$X$. If we want to stress the
base field, we will denote $X^{\an}$ by $X_{K}^{\an}$. 

Let $K'$ be a complete extension of $K$ and $X^\an_{K'}$ the analytic
space associated to the scheme $X_{K'}$.  There is a natural map
$X_{K'}^{\an}\to X_{K}^{\an}$ defined locally by restricting seminorms.

\begin{defn} \label{def:75} A \emph{rational point}
\index{Berkovich space!rational point of}%
of $X_{K}^\an$ is a point $p\in X_K^{\an}$
  satisfying $\mathscr{H}(p)=K$.  We denote  by $X^\an(K)$
\nomenclature[aX04]{$X^\an(K)$}{rational points of a Berkovich space}%
the set of rational points of $X_K^{\an}$.
More generally, for a complete extension
  $K'$ of $K$, the set of \emph{$K'$-rational points} of $X_K^{\an}$ is
  defined as $X^\an(K')=X^\an_{K'}(K')$.  There is a map
  $X^{\an}(K')\to X_K^{\an}$, defined by composing the
  inclusion $ X^{\an}(K')\hookrightarrow X_{K'}^{\an}$ with the map
  $X_{K'}^{\an}\to X_{K}^{\an}$ as above. The set of 
\emph{algebraic points}
\index{Berkovich space!algebraic point of}%
of $X^{\an}$ is
  the union of $X^{\an}(K')$ for all finite extensions $K'$ of
  $K$. Its image in $X^{\an}$ is denoted
  $X^{\an}_{\alg}$.
\nomenclature[aX05]{$X^{\an}_{\alg}$}{algebraic points of a Berkovich space}%
We have that
  $X^{\an}_{\alg}=\{p\in X^{\an} | \, [\mathscr{H}(p):K]< \infty\}$.
\end{defn}

The basic properties of $X^{\an}$ are summarized in the following
theorem.

\begin{thm}\label{thm:6}
  Let $X$ be a scheme of finite type over $K$ and $X^{\an}$ the associated
  analytic space.
  \begin{enumerate}
  \item \label{item:1} $X^{\an}$ is a locally compact and locally
    arc-connected topological space.
  \item \label{item:2} $X^{\an}$ is Hausdorff
    (respectively  compact, arc-connected) if and only if $X$ is separated
    (respectively proper, connected).
  \item \label{item:3} The map $\pi \colon X^{\an}\to X$ is
    continuous. A locally constructible subset $T\subset X$ is open
    (respectively closed, dense) if and only if $\pi ^{-1}(T)$ is open
    (respectively closed, dense).
  \item \label{item:4} Let $\psi \colon X\longrightarrow Y$ be a
    morphism of schemes of finite type over $K$ and $\psi ^{\an}\colon
    X^{\an}\longrightarrow Y^{\an}$ its
    analytification.
    Then $\psi $
    is flat (respectively unramified, \'etale, smooth, separated,
    injective, surjective, open immersion, isomorphism) if and only if
    $\psi ^{\an}$ has the same property.
  \item \label{item:5} Let $K'$ be a complete extension of $K$. Then
    the map $\pi_{K'}:X^{\an}_{K'}\to X_{K'} $ induces a bijection
    between $X^{\an}(K')$ and $ X(K')$.
  \item \label{item:98} Set $X_{\alg}=\{p\in X | \, [K(p):K]<
    \infty\}$\nomenclature[aX03]{$X_{\alg}$}{algebraic points of a
      variety}, 
    where $K(p)$ denotes the function field of~$p$.  Then $\pi$
    induces a bijection between $X^{\an}_{\alg}$ and $X_{\alg}$. The
    subset $X^{\an}_{\alg}\subset X^{\an}$ is dense.
  \item \label{item:119} Let $\widehat K$ be the completion of the
    algebraic closure 
    $\ov K$ of $K$. Then the map $X^{\an}_{\widehat K}\to X_{K}^{\an}$
    induces an isomorphism $X^{\an}_{\widehat
      K}/\Gal(\ov K^{\sep}/K)\simeq X^{\an}_{K}$.
  \end{enumerate}
\end{thm}

\begin{proof}
  The proofs can be found in \cite{Berkovich:stag} and the next
  pointers are with respect to the numeration in this
  reference: (\ref{item:1}) follows from
  Theorem 1.2.1, Corollary~2.2.8 and Theorem 3.2.1, 
  (\ref{item:2}) is Theorem~3.4.8, (\ref{item:3}) is Corollary~3.4.5,
  (\ref{item:4}) is Proposition~3.4.6, (\ref{item:5}) is
  Theorem~3.4.1(i), while \eqref{item:98} follows from Theorem~3.4.1(i)
  and Proposition 2.1.15 and \eqref{item:119} follows from Corollary~1.3.6.
\end{proof}

\begin{rem}\label{rem:2} Not every analytic space in the sense of
  Berkovich can be obtained as the analytification of an algebraic variety. The general theory
  is based on spectra of affinoid $K$-algebras, that provide compact
  analytic spaces that are the building blocks of the more general
  analytic spaces.
\end{rem}
\index{Berkovich space|)}%

\begin{exmpl}\label{exm:4}
  Let $M\simeq \Z^{n}$ be a lattice of rank $n$ and consider the
  associated group algebra $K[M]$ and algebraic torus
  $\T_{M}=\Spec(K[M])$.\index{torus!algebraic}
  \nomenclature[aTM]{$\T_{M}$}{algebraic torus associated to a
    lattice}%
  The corresponding analytic space
  $\T_{M}^{\an}$\index{torus!analytic}
  is the set of multiplicative seminorms of $K[M]$ that
  extend the absolute value of $K$. This is an analytic group.
  We warn the reader that
  the set of points of an analytic group is 
  not an abstract group, hence some care has to be taken when speaking
  of actions and orbits. The precise definitions and basic properties
  can be found in  \cite[\S 5.1]{Berkovich:stag}. 

  The analytification $\T_{M}^{\an}$ is an analytic torus as in
  \cite[\S 6.3]{Berkovich:stag}.  
  The subset
  \begin{displaymath}
    \SS= \{ p\in \T_{M}^{\an}|\, |\chi^{m}(p)|=1 \text{ for all } m\in
  M\}.
  \end{displaymath}
  is a compact analytic subgroup, called
  the \emph{compact torus} of
  $\T_M^{\an}$.\index{compact torus}
  \nomenclature[aSMan]{$\SS $}{compact torus (non-Archimedean case)}%
\end{exmpl}

\section{Algebraic metrics in the non-Archimedean case}
\label{sec:algebr-metr-discr}
Let $K$ be a field which is complete with respect to a nontrivial
non-Archimedean absolute value. Let $K^{\circ}$ and $K^{\circ\circ}$
be as in the previous section. For 
simplicity, we will assume from now on that~$K^{\circ}$ is a {discrete
  valuation ring} (DVR), \index{valuation ring!discrete}%
\index{DVR (discrete valuation ring)}%
and we will fix a generator $\varpi$ of its maximal ideal
$K^{\circ\circ}$.  This is the only case we will need in the sequel
and it allows us to use a more elementary definition of measures and
local heights. The reader can consult \cite{Gu03,Gub07}
for the general case.  \nomenclature[g16]{$\varpi $}{generator of the
  maximal ideal of a DVR}%

Let $X$ be a variety over $K$ and $L$ a line bundle
on $X$. Let $X^{\an}$ and $L^{\an}$ be their respective
analytifications.%
\nomenclature[aLan2]{$L^{\an}$}{Berkovich analytification of a line bundle}%
\index{analytification!of a line bundle}%

\begin{defn}\label{def:2} A \emph{metric}
  \index{metric!on a line bundle!over a non-Archimedean field}%
  \nomenclature[svert12]{$\Vert\cdot\Vert$}{metric on a line bundle (non-Archimedean case)}%
  on $L^{\an}$ is
  an assignment that, to each open subset $U\subset X^{\an}$ and local
  section $s$ of $L^{\an}$ on $U$, associates a continuous function
  \begin{displaymath} \|s(\cdot)\|\colon U \longrightarrow
\mathbb{R}_{\ge 0},
  \end{displaymath} such that  
  \begin{enumerate}
  \item it is compatible with the restriction to smaller open subsets;
  \item for all $p\in U$, $\|s(p)\|=0$ if and only if $s(p)=0$;
  \item for any $\lambda \in
\mathcal{O}_{X^{\an}}(U)$,  it holds
\begin{math} \|(\lambda s)(p)\|=|\lambda(p)|\, \|s(p)\|.
\end{math}
  \end{enumerate}
The pair $\ov
  L:=(L,\|\cdot\|)$
  \nomenclature[aLb20]{$\ov L$}{metrized line bundle (non-Archimedean case)}%
  is called a \emph{metrized line bundle}.%
\end{defn}

Models of varieties and line bundles give rise to an important class
of metrics. To introduce and study these metrics, we first
consider the notion of models of varieties.  Write
$S=\Spec(K^{\circ})$.
\nomenclature[aSaz]{$S$}{scheme associated to a DRV}%
The scheme $S$ has two points: the
special point~$o$
\index{special point of $S$}%
\nomenclature[ao]{$o$}{special point of $S$}%
and the generic point $\eta$.
\index{generic point of $S$}%
\nomenclature[g07]{$\eta$}{generic point of $S$}%
Given a scheme
$\cX$ over $S$, we set $\mathcal{X}_{o}=\mathcal{X}\times \Spec (k)$
and $\mathcal{X}_{\eta}=\mathcal{X}\times \Spec (K)$ for its
{special fibre}
\index{special fiber of a scheme over $S$}%
\nomenclature[aX22]{$\mathcal{X}_{o}$}{special fiber of a scheme over $S$}%
and its {generic fibre},
\index{generic fiber of a scheme over $S$}%
\nomenclature[aX23]{$\mathcal{X}_{\eta}$}{generic fiber of a  scheme over $S$}%
respectively.

\begin{defn}\label{def:48}
  A \emph{model}\index{model!of a variety}
  \nomenclature[aX21]{$\cX$}{model of a variety}%
  over $S$ of $X$ is a flat scheme $\cX$ of finite type over $S$ together
  with a fixed isomorphism $X\simeq \mathcal{X}_{\eta}$. This
  isomorphism is part of the model, and so we can 
  identify $\cX_{\eta}$ with $X$. When $X$ is proper, we say that
  the model is \emph{proper}
  \index{model!of a variety!proper}%
  whenever the scheme $\cX$ is proper over $S$.
\end{defn}

Given a model $\cX$ of $X$, there is a reduction map
\index{reduction map} defined on a 
subset of~$X^{\an}$ with values in 
$\mathcal{X}_{o}$ \cite[\S 2.4]{Berkovich:stag}. 
This map can be described as follows. Let~$\{\mathcal{U}_{i}\}_{i\in I}$ be a finite open
cover of $\mathcal{X}$ by affine schemes over $S$ of finite type
and, for each $i$, let $\cA_{i}$ be a $K^{\circ}$-algebra such that
$\cU_{i}=\Spec(\cA_{i})$.
Set $U_{i}=\mathcal{U}_{i}\cap X$ and let $C_{i}$ be the closed
subset of $U_{i}^{\an}$ defined as
\begin{equation}\label{eq:43}
  C_{i}=\{p\in U_{i}^{\an}\mid |a(p)|\le 1, \forall a\in \cA_{i}\}
\end{equation}
For each $p\in C_{i}$, the prime ideal $\mathfrak{q}_p:=\{a\in
\cA_{i}\mid |a(p)|<1\}\subset \cA_{i}$ contains $K^{\circ\circ }\cA_{i}$ and
so it determines a point $\red(p):=\mathfrak{q}_p/K^{\circ\circ
}\cA_{i}\in \mathcal{U}_{i,o}\subset \mathcal{X}_{o}$. Consider the
subset $C=\bigcup_{i} C_{i}\subset X^{\an}$. The above maps glue
together to define a map
\begin{equation}
  \label{eq:87}
  \red\colon
C\longrightarrow \mathcal{X}_{o}.
\end{equation}
\nomenclature[ared]{$\red$}{reduction map}%
This map is surjective and anti-continuous, in the sense that the
preimage of an open subset of $\cX_{o}$ is closed in $C$ \cite[\S
2.4]{Berkovich:stag}. If both $X$ and $\cX$ are proper then, using
the valuative criterion of properness, one can see that 
$C=X^{\an}$ and the reduction map is defined on
the whole of $X^{\an}$.

\begin{prop}\label{prop:92}
  Assume that $X$ and $\cX$ are normal.
For
each irreducible component $V$ of $\cX_{o}$, there is a unique point
$\xi_{V}\in C$
\nomenclature[g14V]{$\xi_{V}$}{point associated to an irreducible component of the special fiber}%
such that
\begin{displaymath}
  \red(\xi_{V})=\eta_{V},
\end{displaymath}
where $\eta_{V}$ denotes the generic point of $V$. If we choose an
affine open subset $\mathcal{U}=\Spec(\cA)\subset \cX$ containing
 $\eta_{V}$ and we write
$A=\cA\otimes_{K^{\circ}}K$ and $U= \cU\cap X$, then $\xi_{V}$ lies in
$U^{\an}$ and it is the multiplicative seminorm on $A$ given, for
$a\in A$, by
\begin{equation}\label{eq:7}
  |a(\xi_{V})|=|\varpi|^{\ord_{V}(a)/\ord_{V}(\varpi)}, 
\end{equation}
where $\ord_{V}(a)$ denotes the order of $a$ at $\eta_{V}$. 
\end{prop}

\begin{proof}
  We first assume that $X$ and $\cX$ are affine.  Let $\cA$ be a
  $K^{\circ}$-algebra such that $\cX=\Spec(\cA)$ and set
  $A=\cA\otimes K$. Since $\cX$ is normal, $\cA_{\eta_{V}}$ is a
  discrete valuation ring. Let $\ord_{V}$ denote the valuation in this
  ring.

  We first show the existence of $\xi_{V}$. Since $\ord_{V}(\varpi
  )\ge1$, the right hand side of the equation (\ref{eq:7}) determines
  a multiplicative seminorm of $A$ that extends the absolute value of
  $K$ and hence a point $\xi_{V}\in U^{\an}\subset X^{\an}$. From the definition, it
  is clear that $|a(\xi_{V})|\le 1$ for all $a\in \cA$ and
  $|a(\xi_{V})|<1$ if and only if $a\in \eta_{V}$. Hence
  $\xi_{V}\in C$ and $\red(\xi_{V})=\eta_{V}$.

  We next prove the unicity of $\xi_{V}$. Let $p\in C$ such that
  $\red(p)=\eta_{V}$. This implies that $p$ is a multiplicative
  seminorm of $A$ that extends the absolute value of $K$ such that
  $|a(p)|\le 1$ for all $a\in \cA$ and $|a(p)|<1$ if and only if $a\in
  \eta_{V}$. In particular, this multiplicative seminorm can be
  extended to $\cA_{\eta_{V}}$. Let $\tau $ be a uniformizer of the
  maximal ideal $\eta_{V}\cA_{\eta_{V}}$ and $a\in \cA$. Write
  $a=u\tau ^{\ord_{V}(a)}$ with $u\in \cA_{\eta_{V}}^{\times}$.  Since
  $|u(p)|=1$, we deduce $|a(p)|=|\tau (p)|^{\ord_{V}(a)}$. Applying
  the same to $\varpi $, we deduce that
\begin{displaymath}
  |a(p)|=|\varpi|^{\ord_{V}(a)/\ord_{V}(\varpi)}. 
\end{displaymath}
Hence $p=\xi_{V}$.

To prove the statement in general, it is enough to observe that, if
$\cU_{1}\subset \cU_{2}$ are two affine open subsets of $\cX$
containing $\eta_{V}$, then the corresponding closed subsets verify
$C_{1}\subset C_{2}$. The result follows 
by the unicity in $C_{2}$ of the point with reduction $\eta_{V}$.
\end{proof}

Next we recall the definition of models of line bundles. 

\begin{defn} 
  \label{def:56} A \emph{model}
  \index{model!of a line bundle}%
  \nomenclature[aLz]{$\cL$}{model of a line bundle}%
  over $S$ of
  $(X,L)$ is a triple
  $(\mathcal{X},\mathcal{L},e)$,
  \nomenclature[aX41]{$(\mathcal{X},\mathcal{L},e)$}{model of a
    variety and a line bundle}%
  where $\mathcal{X}$ is a model over $S$ of $X$,
  $\mathcal{L}$ is a line bundle on $\mathcal{X}$ and $e\ge 1$ is an
  integer, together with an isomorphism $\mathcal{L}|_{X}\simeq
  L^{\otimes e}$.  When $e=1$, the model $(\mathcal{X},\mathcal{L},
  1)$ will be denoted $(\mathcal{X},\mathcal{L})$ for short.  A model
   $(\mathcal{X},\mathcal{L},e)$ is called
  \emph{proper}
  \index{model!of a line bundle!proper}%
  whenever $\cX$ is proper.
\end{defn}

We assume that the variety $X$ is proper for the rest of this section.
To a proper model of a line bundle we can associate a 
metric. 

\begin{defn}\label{def:7} Let $(\cX, \cL, e)$ be a proper model
  of $(X,L)$. Let $s$ be a local section of $L^{\an}$ defined at a
  point $p\in X^{\an}$. Let $\mathcal{U} \subset \cX$ be a
  trivializing open neighbourhood of~$\red(p)$ and $\sigma$ a
  generator of $\cL|_{\cU}$. Let $U=\cU\cap X$ and $\lambda \in
  \cO_{U^{\an}}$ such that $s^{\otimes e}= \lambda \sigma$ on
  $U^{\an}$.  Then, the \emph{metric induced by the proper model}
  $(\cX, \cL, e)$ on $L^{\an}$\index{metric!induced by a model},
  denoted $\|\cdot \|_{\cX, \cL, e}$,
  \nomenclature[svert41]{$\Vert\cdot \Vert_{\cX,\cL, e}$}{metric induced by a
    model}%
 is given by 
  \begin{displaymath}
    \|s(p)\|_{\cX, \cL, e}=|\lambda(p)|^{1/e}.
  \end{displaymath}
  This definition does neither depend on the choice of the open set $\cU$
  nor of the section~$\sigma $, and it
  gives a metric on $L^{\an}$.  The metrics on
  $L^{\an}$ obtained in this way are called
  \emph{algebraic}, and the pair $\ov L:=(L,\Vert\cdot\Vert_{\cX, \cL, e})$ is called an
  \emph{algebraic metrized line bundle}.
\index{metric!algebraic|see{algebraic metric}}%
\index{algebraic metric!}%
\index{metrized line bundle!algebraic}%
\end{defn}

Different models may give rise to the same metric.

\begin{prop}\label{prop:21}
  Let~$(\mathcal{X}, \mathcal{L},e)$ and
  $(\mathcal{X}',\cL',e')$ be proper models of $(X,L)$, and $f\colon
  \mathcal{X}'\to \mathcal{X}$ a morphism of models such that
  $(\cL')^{\otimes e}\simeq f^{\ast}\cL^{\otimes e'}$. Then the
  metrics on $L^{\an}$ induced by both models agree.
\end{prop}

\begin{proof}
  Let $s$ be a local section of $L^{\an}$ defined on a point $p\in
  X^{\an}$. Let $\cU\subset \cX$ be a trivializing open neighbourhood
  of $\red_{\cX}(p)$, the reduction of $p$ with respect to the model
  $\cX$ and $\sigma$ a generator of $\cL|_{\cU}$. Let $\lambda$ be an
  analytic function on $(\cU\cap X)^{\an}$ such that $s^{\otimes e}=
  \lambda \sigma$.

  We have that $f(\red_{\cX'}(p))= \red_{\cX}(p)$ and
  $\cU':=f^{-1}(\cU)$ is a trivializing open set of $\cL'^{\otimes e}$
  with generator $f^{*} \sigma^{\otimes e'}$.  Then $s^{\otimes ee'} =
  \lambda ^{e'} f^{*} \sigma^{\otimes e'}$ on $(\cU'\cap X)^{\an}=
  (\cU\cap X)^{\an}$.  Now the proposition follows directly from
  Definition \ref{def:7}.
\end{proof}

The inverse image of an algebraic metric is algebraic. 

\begin{prop}
  \label{prop:91} Let $\varphi\colon X_{1}\to X_{2}$ be a morphism of
  proper algebraic varieties over $K$ and $\ov L_{2}$ an algebraic
  metrized line bundle on $X_{2}$. Then $\varphi^{*}\ov L_{2}$, the
  inverse image under $\varphi$ of $\ov L_{2}$, is an algebraic
  metrized line bundle on $X_{1}$.
\end{prop}

\begin{proof}
  Let $(\cX_{2},\cL_{2}, e)$ be a proper model of $(X_{2},L_{2})$
  which induces the metric in~$\ov L_{2}$. From Nagata's
  compactification theorem (see for instance \cite{Conrad:Nc}) we can
  find a proper
  model $\cX'_{1}$ of $X_{1}$.  Let $\cX_{1}$ be the Zariski closure of the graph
  of~$\varphi$ in $\cX_{1}'\times_{S}\cX_{2}$.  This is a proper model
  of $X_{1}$ equipped with a morphism $\varphi_{S}\colon \cX_{1}\to
  \cX_{2}$.  Then $(\cX_{1},\varphi_{S}^{*}\cL_{2}, e)$ is a proper
  model of $(X_{1},\varphi^{*}L_{2})$ which induces the metric
  of~$\varphi^{*}\ov L_{2}$.
\end{proof}

Next we give a second description of an algebraic metric.  As before, let $X$ be
a proper variety over $K$ and $L$ a line bundle on $X$, and $\|\cdot\|_{\mathcal{X},
  \mathcal{L},e}$ an algebraic metric on $L^{\an}$. Let $p\in X^{\an}$
and put~$F=\mathscr{H}(p)$, which is a complete extension of $K$. Let
$F^{\circ}$ denote its valuation ring, and $o$ and $\eta$ the 
special and the generic point of $\Spec(F^{\circ})$, respectively. The
point~$p$ induces a morphism of schemes $\Spec(F)\to X$. By the
valuative criterion of properness, there is a unique extension
\begin{equation}\label{eq:26}
  \widetilde p\colon \Spec(F^{\circ})\longrightarrow \mathcal{X}.
\end{equation}
It satisfies $\widetilde p(\eta)=\pi (p)$, where $\pi\colon
X^{\an}\rightarrow X$ is the natural map introduced at the beginning
of \S \ref{sec:berkovich-spaces}, and $\widetilde p(o)=\red(p)$.

\begin{prop}
\label{prop:descmetric} With notation as above, let $s$ be
  a local section of $L$ in a neighbourhood of $\pi (p)$. Then
  \begin{equation}
    \label{eq:9}
    \|s(p)\|_{\mathcal{X},
      \mathcal{L},e}= \inf \big\{|a|^{1/e}\big|
    a\in F^{\times}, a^{-1}\wt p^{\ast} s^{\otimes e}\in
    \widetilde p^{\ast} \mathcal{L}\big\}.
  \end{equation}
\end{prop}
\begin{proof} 
  Write $\|\cdot\|=\|\cdot\|_{\mathcal{X}, \mathcal{L},e}$ for
  short. Let $\mathcal{U}=\Spec(\cA)\ni \red(p)$ be a trivializing
  open affine set of $\mathcal{L}$ and $\sigma $ a generator of
  $\mathcal{L}|_{\cU} $.  Then $s^{\otimes e}=\lambda \sigma $ with
  $\lambda$ in the fraction field of $\cA$. We have that $\lambda
  (p)\in F$ and, by definition, $\|s(p)\|=|\lambda (p)|^{1/e}$. If
  $\lambda (p)=0$, the equation is clearly satisfied. Denote
  temporarily by $\gamma$ the right-hand side of \eqref{eq:9}. If $\lambda
  (p)\not = 0$, then
  \begin{displaymath}
    \lambda (p)^{-1}\widetilde p^{\ast}s^{\otimes e}=\widetilde
    p^{\ast}\sigma \in \wt p^*\mathcal{L}.
  \end{displaymath}
  Hence $\|s(p)\|\ge \gamma$. Moreover, if $a\in F^{\times}$ is such
  that $a^{-1}\widetilde p^{\ast}s^{\otimes e}\in \wt p^*\mathcal{L}$,
  then there is an element $\alpha \in F^{\circ}\setminus \{0\}$ with
  $a^{-1}\widetilde p^{\ast}s^{\otimes e}=\alpha \widetilde
  p^{\ast}\sigma$. Therefore, $a=\alpha^{-1}\lambda (p) $ and
  $|a|^{1/e}=|\alpha |^{-1/e}|\lambda (p)|^{1/e}\ge |\lambda
  (p)|^{1/e}$. Thus, $\|s(p)\|\le \gamma$, completing the proof.
\end{proof}

We give a third description of an algebraic metric in terms of
intersection theory that makes evident the relationship with higher
dimensional Arakelov theory. Let $(\cX,\cL,e)$ be a proper model of
$(X,L)$ and $\iota\colon \cY\to \cX$ a closed algebraic curve.  Let
$\wt \cY$ be the normalization of $\cY$ and $\wt \iota\colon \wt
\cY\to \cX$ and $\rho \colon \wt \cY\to \Spec (K^{\circ})$ the induced
morphisms. Let $s$ be a rational section of $\cL$ such that the
Cartier divisor $\div (s)$
intersects properly $\cY$. Then the intersection number $(\iota\cdot
\div(s))$
\nomenclature[sdot]{$(\iota\cdot\div(s))$}{intersection number of a curve
  with a divisor}%
is defined
as
\begin{displaymath}
  (\iota \cdot \div(s))=\deg(\rho _{\ast}(\div(\wt \iota^{\ast}s))).
\end{displaymath}

\begin{prop}\label{prop:1}
  With the above notation, let $p\in X^{\an}_{\alg}$ and denote by $\widetilde
  p$ the image of the map in \eqref{eq:26}. This is a closed algebraic curve. Let $s$
  be a local section of $L$ defined at $p$ and such that $s(p)\not = 0$. Then
  \begin{equation*}
    {\log\|s(p)\|_{\mathcal{X},
      \mathcal{L},e}}=\frac{(\widetilde p\cdot
      \div (s^{\otimes 
        e}))}{e[\mathscr{H}(p):K]} {\log|\varpi|}. 
  \end{equation*}
\end{prop}

\begin{proof}
  We keep the notation in the proof of
  Proposition~\ref{prop:descmetric}. In particular, $s^{\otimes
    e}=\lambda \sigma$ with $\lambda$ in the fraction field of $\cA$,
  and $\mathscr{H}(p)=F$. Then
  \begin{equation*}
    \frac{\log\Vert s(p)\Vert_{\mathcal{X},
      \mathcal{L},e}}{\log|\varpi|}
    =\frac{\log| \lambda(p)|}{e\log|\varpi|}=\frac{\log|
      \Norm_{F/K}(\lambda(p))|}{e[F:K]\log|\varpi|}=\frac{\ord_{\varpi }
      (\Norm_{F/K}(\lambda(p)))}{e[F:K]},
  \end{equation*}
where $\Norm_{F/K}$ is the norm function of the finite extension
$F/K$. We also verify that
   \begin{multline*}
    (\wt p\cdot \div(s^{\otimes e}))=\deg(\rho _{\ast}(\div (\wt
    p^{\ast}s^{\otimes e})))=\deg(\rho _{\ast}(\div (\lambda (p))))\\=
    \deg(\div(\Norm_{F/K}(\lambda(p))))=\ord_{\varpi }
    (\Norm_{F/K}(\lambda(p))),
  \end{multline*}
  which proves the statement.
\end{proof}

\begin{exmpl}
  Let $X=\P^0_{K}=\Spec(K)$. A line bundle $L$ on $X$ is necessarily
  trivial, that is, $L\simeq K$. Consider the model $(\cX,\cL, e)$ of
  $(X,L)$ given
  by $\cX=\Spec(K^{\circ})$, $e\ge 1$, and $\cL$ a
  free $K^{\circ}$-submodule of $L^{\otimes e}$ of rank one. Let $v\in
  L^{\otimes e}$ be a basis of $\cL$. For a section $s$ of $L$ we can
  write $s^{\otimes e}=\alpha v $ with $\alpha\in K$. Hence,
  $$
  \|s\|_{\mathcal{X},
      \mathcal{L},e}=|\alpha|^{1/e}.
  $$
  All algebraic metrics on $L^{\an}$ can be obtained in this way.
\end{exmpl}

\begin{exmpl} \label{exm:2} Let $X=\P_{K}^n$ and $L=\mathcal{O}(1)$,
  the universal line bundle of ${\P_{K}^n}$.  As a model for $(X,L)$ we
  consider $\mathcal{X}=\P_{K^{\circ}}^n$, the projective space over
  $\Spec(K^{\circ})$, $\mathcal{L}=\mathcal{O}_{\P_{K^{\circ}}^n}(1)$,
  and $e=1$. A rational section $s$ of $L$ can be identified with a
  homogeneous rational function $\rho_{s}\in K(x_{0},\dots, x_{n})$
  of degree 1.

  Let $p=(p_{0}:\dots:p_{n})\in (\P_{K}^n)^\an\setminus \div(s)$ and
  set $F=\mathscr{H}(p)$.  Let $0\le i_{0}\le n$ be such that 
  $|p_{i_{0}}|=\max_{i}\{|p_{i}|\}$. Take $U \simeq
  \A_{K}^{n}$ (respectively $\mathcal{U}\simeq \A_{K^{\circ}}^{n}$) as
  the affine set $x_{i_{0}}\not=0$ over $F$ (respectively
  $F^{\circ}$). The point $p$ corresponds to the morphism
  \begin{displaymath}
    p^{\ast}\colon K[X_{0},\dots,X_{i_{0}-1},X_{i_{0}+1},\dots ,X_{n}]\longrightarrow F
  \end{displaymath}
  that sends $X_{i}$ to $p_{i}/p_{i_{0}}$. The extension $\wt p$ factors
  through the  morphism
  \begin{displaymath}
    \wt p^{\ast}\colon  K^{\circ}[X_{1},\dots,X_{i_{0}-1},X_{i_{0}+1}
    ,\dots,X_{n}]\longrightarrow F^{\circ}
  \end{displaymath}
  with the same definition. Then
  \begin{align*}
    ||s(p)|| &=\inf \big\{ |z| \ \big| z\in F^\times , z^{-1} \wt p^{\ast}s\in
    \wt
    p^*\cL\big\}\\[1mm]
    &=\inf \big\{ |z| \ \big| z\in F^\times , z^{-1}
    \rho_{s}(p_{0}/p_{i_{0}},\dots,1,\dots,p_{n}/p_{i_{0}}) \in F^{\circ}\big\}\\[1mm]
    &=\left| \frac{\rho_{s}(p_{0},\dots,p_{n})}{p_{i_{0}}}\right|\\[1mm]
    &=\frac{|\rho_{s}(p_{0},\dots,p_{n})|}{\max_{i}\{|p_{i}|\}}.
  \end{align*}
  We call this the \emph{canonical metric} of
    $\cO(1)^{\an}$
  \index{canonical metric!of $\cO(1)^{\an}$}%
  and we denote it by
  $\|\cdot\|_{\can}$.
  \nomenclature[svert21]{$\Vert\cdot\Vert_{\can}$}{canonical metric of
    $\cO(1)^{\an}$ (non-Archimedean case)}%
\end{exmpl}

Many other algebraic metrics can be obtained from Example \ref{exm:2},
by considering maps of varieties to projective spaces.  Let $X$ be a
proper variety over $K$ equipped with a line bundle $L$ such that
$L^{\otimes e}$ is generated by global sections for an integer $e\ge
1$. A set of global sections in $\Gamma (X,L^{\otimes e})$ that
generates $L^{\otimes e}$ induces a morphism $\varphi\colon X \to
\P_{K}^{n}$ and, by inverse image, a  metric
$\varphi^{*}\|\cdot\|_{\can}$ on $L$. Then
Proposition \ref{prop:91} shows that this metric is algebraic.

Now we recall the notion of semipositivity for algebraic
metrics. A curve $C$ in~$\cX$ is \emph{vertical}\index{vertical curve} if it is
contained in $\cX_{o}$. 

\begin{defn} 
\label{def:51} Let $X$ be a proper algebraic variety over $K$, $L$ a line
bundle on $X$ and $(\mathcal{X},\mathcal{L}, e)$ a proper model of
$(X,L)$. We say that $(\mathcal{X},\mathcal{L}, e)$ is a 
\emph{semipositive model}\index{model!semipositive} if, for every
vertical curve $C$ in $\cX$,  
  \begin{displaymath}
    \deg_{\mathcal{L}}(C)\ge 0.
  \end{displaymath}
Let $\Vert\cdot\Vert$ be a metric on $L$ and set $\ov
L=(L,\Vert\cdot\Vert)$. We say that $\ov L$ \emph{has a semipositive model}
\index{metric!with a semipositive model}%
if there is a semipositive model
  $(\mathcal{X},\mathcal{L}, e)$ of $(X,L)$ that
  induces the metric. 
\end{defn}

\begin{prop} \label{prop:26} Let $\varphi\colon X_{1}\to X_{2}$ be a
  morphism of proper algebraic varieties over $K$ and $\ov L_{2}$ a
  metrized line bundle on $X_{2}$ with a semipositive model. Then
  $\varphi^{*}\ov L_{2}$ is a metrized line bundle on $X_{1}$ with a
  semipositive model.
\end{prop}

\begin{proof}
  Let $(\cX_{2},\cL_{2},e)$ be a semipositive model inducing the metric
  of $\ov L_{2}$. With notations as in the proof of Proposition
  \ref{prop:91}, $(\cX_{1},\varphi_{S}^{*}\cL_{2},e)$ is a model
  inducing the metric of $\varphi^{*}\ov L_{2}$.  Let $C$ be a
  vertical curve in $\cX_{1}$. By the projection formula,
\begin{displaymath}
  \deg_{\varphi^{*}\cL_{2}}(C)=   \deg_{\cL_{2}}(\varphi_{*}C)\ge0. 
\end{displaymath}
Hence, $(\cX_{1},\varphi_{S}^{*}\cL_{2},e)$ is semipositive.
\end{proof}

\begin{exmpl}
  The canonical metric in Example~\ref{exm:2} has a semipositive
  model: for a {vertical} curve $C$, its degree with respect to
  $\cO_{\P^n_{K^{\circ}}}(1)$ equals its degree with respect to the
  restriction of this model to the special fibre.  This restriction
  identifies with $\cO_{\P^n_k}(1)$, the universal line bundle of
  $\P^{n}_{k}$, which is ample.  Hence, all the metrics obtained by
  inverse image of the canonical metric of $\cO(1)^{\an}$ have
  semipositive models.
\end{exmpl}

Finally, we recall the definition of the signed measures associated
with a family of algebraic metrics.

\begin{defn}
  \label{def:76}
  Let $\ov L_i$, $i=0,\dots, d-1$, be line bundles on $X$ equipped
  with algebraic metrics. For each $i$, choose a model
  $(\mathcal{X}_{i},\mathcal{L}_{i}, e_{i})$ that induces the metric
  of~$\ov L_{i}$. We can assume without loss of generality that the
  models $\mathcal{X}_{i}$ agree with a common model $\mathcal{X}$.
  Let $Y$ be a $d$-dimensional subvariety of $X$ and
  $\mathcal{Y}\subset \mathcal{X}$ be the closure of $Y$. Let
  $\wt{\mathcal{Y}}$ be its normalization, $\wt{\cY}_{o}$ the special
  fibre, $\wt{\cY}_{o}^{(0)}$ the set of irreducible
  components of $\wt{\cY}_{o}$, $\widetilde Y= \wt \cY_{\eta}$ the
  generic fibre, and $\widetilde Y^{\an}$ the analytification of
  $\widetilde Y$. For each $V\in \wt{\cY}_{o}^{(0)}$, consider the
  point $\xi_{V}\in \wt Y^{\an}$ given by Proposition
  \ref{prop:92}. Let $\delta _{\xi_{V}}$ be the Dirac delta measure on
  $ X^{\an}$ supported on the image of $\xi_{V}$.  We define a
  discrete {signed measure} on $X^{\an}$ 
\index{algebraic metric!signed measure associated to}%
  \index{measure!associated to an algebraic metric}%
   by
\begin{equation}
  \label{eq:11}
  \chern_1(\ov L_{0})\land\dots\land \chern_{1}(\ov L_{d-1})\land \delta _{Y}=
\sum_{V\in
    \wt{\cY}_{o}^{(0)}}\ord_{V}(\varpi)
  \frac{\deg_{\mathcal{L}_{0},\dots,\mathcal{L}_{d-1}} (V) }{e_{0}\dots e_{d-1}}
  \delta _{\xi_{V}}.
\end{equation}
\nomenclature[ac1L1]{$\chern_1(\ov L_{0})\land\dots\land 
\chern_{1}(\ov L_{d-1})\land \delta _{Y}$}{signed measure (algebraic case)}%
This notion extends by linearity to the group of $d$-dimensional cycles of $X$.
\end{defn}

This signed measure only depends on the metrics and not
on the particular choice of models~\cite[Proposition~2.7]{Cha06}. Observe that
$\ord_{V}(\varpi)$ is the multiplicity of the component $V$ in
$\wt{\cY}_{o}$ and that the total mass of this measure equals
$\deg_{L_0,\dots, L_{d-1}}(Y)$.  If $\ov L_{i}$ has a semipositive model for
all $i$ and $Y$ is effective, this signed measure is a measure.

\begin{rem}\label{rem:1}
  The above measure was introduced by Chambert-Loir in~\cite{Cha06}.  For
  the subvarieties of a projective space equipped with the canonical
  metric, it is also possible to define similar measures through the
  theory of Chow forms~\cite{Philippon:HAII}.
\end{rem}

\begin{rem}\label{rem:15} Let $\wh K$ be the completion of the
  algebraic closure $\ov K $ of $K$. 
In analogy with  Remark \ref{rem:17}, 
we could have defined a 
  continuous metric on $L^{\an}$ as a continuous metric on
  the line bundle $L^{\an}_{\wh K}$ over $X_{\wh K}^{\an}$ that is invariant under the
  action of the Galois group $\Gal(\ov K^{\sep}/K)$. The obtained theory is
  equivalent to the one outlined here and the reader should
  have no difficulties in translating results from one to the
  other. This point of view is closer to Zhang's approach in
\cite{Zhang:_small}, \end{rem}

\section{Semipositive{} and  DSP metrics, measures and local
  heights} 
\label{sec:admiss-integr-metr}

Let $K$ be either $\R$ or $\C$ (the \emph{Archimedean case})
\index{Archimedean case} as in \S \ref{sec:arqu-smooth-metr}, or a
field which is complete with respect to a nontrivial non-Archimedean
discrete absolute value (the \emph{non-Archimedean
  case})\index{non-Archimedean case} as in \S
\ref{sec:algebr-metr-discr}.  Let $X$ be a proper variety over
$K$. Its analytification $X^{\an}$ will be a complex analytic space in
the Archimedean case (equipped with an anti-linear involution when
$K=\R$), or an analytic space in the sense of Berkovich in the
non-Archimedean case.  A \emph{metrized line bundle} 
\index{metrized line bundle!} on $X$ is a pair $\ov L=(L,\|\cdot\|)$, where $L$ is a
line bundle on $X$ and $\|\cdot\|$ is a metric on $L^{\an}$.  Recall
that the operations on line bundles of tensor product, dual and
inverse image under a morphism extend to metrized line bundles.


Given two metrics $\|\cdot\|$ and $\|\cdot\|'$ on
$L^{\an}$, their quotient defines a continuous function $X^{\an}\to
\R_{>0}$ given by $\|s(p)\|/\|s(p)\|'$ for any local section $s$
of~$L$ not vanishing at $p$. The \emph{distance} 
\index{distance between metrics}%
between $\|\cdot\|$ and $\|\cdot\|'$ is defined as
the supremum of the absolute value of the logarithm of this
function. In other words,
\begin{equation}
  \label{eq:84}
\dist(\|\cdot\|,\|\cdot\|')=\sup_{p\in X^\an\setminus
  \div(s)}|\log(\|s(p)\|/\|s(p)\|')|
\end{equation}
for any nonzero rational section $s$ of $L$. 

\begin{defn} \label{def:65} Let $\ov L=(L,\Vert\cdot\Vert)$ be a
  metrized line bundle on $X$. The metric \hbox{$\|\cdot\|$}
  is  \emph{semipositive{}}
\index{metric!semipositive{}|see{semipositive{} metric}}%
  \index{semipositive{} metric}%
  if there exists a
  sequence $(\|\cdot\|_l)_{l\ge 0}$ of semipositive smooth metrics (in
  the Archimedean case) or 
  metrics with a semipositive model (in the non-Archimedean case)
   on $L^{\an}$ such that
$$
\lim_{l\to \infty}\dist(\|\cdot\|,\|\cdot\|_l)=0.
$$
If this is the case, we say that $\ov
L$ is \emph{semipositive{}}.
\index{metrized line bundle!semipositive{}}%
The metrized line bundle $\ov L$ is
called \emph{DSP} (for \og difference of semipositive\fg{})
\index{DSP (difference of semipositive) metric}%
\index{metrized line bundle!DSP}%
if there are semipositive{} line bundles $\ov M$, $\ov N$ such that $\ov L=\ov
M\otimes \ov N^{\otimes -1}$.
\end{defn}

\begin{rem}\label{rem:31}
  In the Archimedean case, if $\|\cdot\|$ is a smooth metric, one can
  verify that definitions \ref{def:65} and \ref{def:1} are
  equivalent. Thus there is no ambiguity in the use of the term
  semipositive metric. 

  In the non-Archimedean case, a metric with a
  semipositive model is semipositive. In the general case, we do not
  know whether an algebraic and semipositive metric has a
  semipositive model. 
\end{rem}

\begin{rem}\label{rem:29}
Although we define our notion of semipositivity through a limit
process, we believe that a ``good'' definition should be
intrinsic. For example, for smooth projective varieties, in the
Archimedean case \cite[Th\'eor\`eme
4.6.1]{Maillot:GAdvt} and in the non-Archimedean case of
equi-characteristic zero \cite[Theorem 5.11]{boucksom:ssmnag} our
definition is equivalent to the fact that the logarithm of the norm of
a section is a plurisubharmonic function. We hope our definition will
still agree with such an intrinsic one when the theory of
plurisubharmonic functions on Berkovich spaces matures. 

We adopt the terminology of ``DSP metric'' by analogy with the
notion of DC function, used in convex analysis to designate a function
that is a difference of two convex functions. 
\end{rem}

The tensor product and the inverse image of semipositive{} line bundles
are also semipositive{}. The tensor product, the dual and the inverse
image of DSP line bundles are also DSP.

\begin{exmpl} \label{exm:6} Let $X=\P^{n}$ be the projective space
  over $\C$ and $L=\mathcal{O}(1)$. The \emph{canonical metric}
  of~$\mathcal{O}(1)^{\an}$ 
\index{canonical metric!of $\cO(1)^{\an}$}%
  is the metric given, for $p=(p_{0}:\dots:p_{n})\in
  \P^{n}(\C)$, by
 \begin{displaymath}
    \|s(p)\|_\can =\frac{|\rho_{s}(p_0,\dots,p_n)|}{\max_i\{|p_i|\}},
  \end{displaymath}
  \nomenclature[svert22]{$\Vert\cdot\Vert_{\can}$}{canonical metric of
    $\mathcal{O}(1)^{\an}$ (Archimedean case)}%
  for any rational section $s$ of $L$ defined at $p$ and the
  homogeneous rational function $\rho_{s}\in \C(x_{0},\dots, x_{n})$
  associated to $s$.
  
This is a semipositive{} metric. Indeed, consider the $m$-power map
$[m]:\P^n\to \P^n$ defined as
$[m](p_0:\dots:p_n)=(p^m_0:\dots:p^m_n)$. The $m$-th root of the
inverse image by $[m]$ of the Fubini-Study metric of
$\mathcal{O}(1)^{\an}$ is the semipositive smooth metric on~$L^{\an}$ given by
\begin{displaymath}
     \|s(p)\|_m =\frac{|\rho _{s}(p_0,\dots,p_n)|}{(\sum_i|p_i|^{2m})^{1/2m}}.
  \end{displaymath}
 The family of metrics obtained varying $m$ converges uniformly to the
 canonical metric. 
\end{exmpl}

\begin{prop} 
  \label{prop:2} Let $Y$ be a $d$-dimensional subvariety of $X$ and
  $(L_{i},\|\cdot\|_{i})$, $i=0,\dots,d-1$, a collection of
  semipositive{} metrized line bundles on $X$. For each~$i$, let
  $(\|\cdot\|_{i,l})_{l\ge 0} $ be a sequence of semipositive smooth
  metrics (in the Archimedean case) or metrics with a semipositive
  model (in the non-Archimedean case) on $L_{i}^{\an}$ that converge
  to $\|\cdot\|_{i}$.  Then the measures $
  \chern_{1}(L_{0},\|\cdot\|_{0,l})\land \dots\land
  \chern_{1}(L_{d-1},\|\cdot\|_{d-1,l})\wedge \delta _{Y} $ converge
  weakly to a measure on $X^\an$.
\end{prop}

\begin{proof}
The non-Archimedean case is established in~\cite[Proposition~2.7(b)]{Cha06}
and in~\cite[Proposition~3.12]{Gub07}. The Archimedean case can 
be proved similarly.
\end{proof}

\begin{defn}
  \label{def:60}
  Let $\ov
  L_i=(L_{i},\|\cdot\|_{i})$, $i=0,\dots,d-1$, be a collection of
  semipositive{} metrized line bundles on $X$. For a $d$-dimensional
  subvariety $Y\subset X$, we denote by~$
  \chern_1(\ov L_{0})\land\dots\land \chern_{1}(\ov L_{d-1})\wedge
  \delta _{Y}$ the limit measure in Proposition \ref{prop:2}.
For DSP bundles
$\ov L_i$ and a $d$-dimensional cycle $Y$ of $X$, we can associate a signed measure $ \chern_1(\ov
L_{0})\land\dots\land \chern_{1}(\ov L_{d-1})\wedge \delta _{Y}$ on
$X^{\an}$ by multilinearity.
\index{DSP (difference of semipositive) metric!signed measure associated to}%
\index{measure!associated to a DSP metric}%
\nomenclature[ac1L2]{$\chern_1(\ov L_{0})\land\dots\land
  \chern_{1}(\ov L_{d-1})\land \delta _{Y}$}{signed measure (DSP case)}%
\end{defn}

This signed measure behaves well under field extensions.

\begin{prop}\label{prop:28}
  With the previous notation, let $K'$ be a finite extension
  of~$K$. Set $(X',Y')=(X,Y)\times \Spec(K')$ and let $\varphi\colon
  {X'}^{\an}\to X^{\an}$ be the induced map. Let~$\varphi^{\ast}\ov
  L_{i}$, $i=0,\dots,d-1$, be the line bundles with algebraic metrics
  on $X'$ obtained by base change. Then
  \begin{displaymath}
    \varphi_{\ast} \left(\chern_1(\varphi^{\ast}\ov L_{0})\land\dots\land
      \chern_{1}(\varphi^{\ast}\ov L_{d-1})\land \delta _{{Y'}}\right)=
    \chern_1(\ov L_{0})\land\dots\land \chern_{1}(\ov L_{d-1})\land \delta _{
      Y'}.
  \end{displaymath}
\end{prop}

\begin{proof}
This follows from \cite[Remark 3.10]{Gub07}.
\end{proof}

We also have the following functorial property.

\begin{prop} \label{prop:66} Let $\varphi\colon X'\to X$ be a morphism
  of proper varieties over $K$, $Y'$ a $d$-dimensional cycle of $X'$,
  and $\ov L_i$, $i=0,\dots,d-1$, a 
  collection of DSP metrized line bundles on $X$. Then
  \begin{displaymath}
    \varphi_{\ast} \left(\chern_1(\varphi^{\ast}\ov L_{0})\land\dots\land
      \chern_{1}(\varphi^{\ast}\ov L_{d-1})\land \delta
      _{{Y'}}\right)=
    \chern_1(\ov L_{0})\land\dots\land \chern_{1}(\ov L_{d-1})\land \delta _{
      \varphi_{*}Y'}.
  \end{displaymath}
\end{prop}
\begin{proof}
  In the non-Archimedean, this follows from \cite[Corollary
  3.9(2)]{Gub07}. In the Archime\-dean case, this follows from the
  functoriality of Chern classes, the projection formula, and the
  continuity of the direct image of measures.
\end{proof}

\begin{defn}\label{def:4}
   Let $Y$ be a $d$-dimensional cycle of $X$ and
   $(L_{i},s_{i})$, $i=0,\dots,d$, a collection of line bundles on $X$
   with a rational section.  We say that \emph{$s_{0},\dots, s_{d}$
    meet~$Y$ properly}
\index{proper intersection}%
if, for all $I\subset \{0,\dots,d\}$, each irreducible component of
$Y\cap \bigcap_{i\in I} |\div (s_{i})|$ has dimension $d-\# I$.
\end{defn}

The above signed measures allow to integrate continuous
functions on $X^{\an}$. Indeed, it is also possible to integrate certain
functions with logarithmic singularities that play an important role
in the definition of local heights. Moreover, this integration is
continuous with respect to uniform convergence of metrics. 

\begin{thm} 
  \label{prop:4} Let $Y$ be a $d$-dimensional cycle of $X$, $\ov L_i$,
  $i=0,\dots,d-1$, a collection of semipositive{} metrized line
  bundles and $(\ov L_{d},s_{d})$ a metrized line bundle with a
  rational section meeting $Y$ properly.
  \begin{enumerate}
  \item \label{item:12} The support of $\div(s_{d})$ has measure zero and the function
    $\log\|s_{d}\|_{d}$ is integrable with respect to
    the measure $ \chern_1(\ov L_{0})\land\dots\land \chern_{1}(\ov
    L_{d-1})\wedge \delta _{Y} $.

  \item \label{item:27} Let $(\|\cdot\|_{i,n})_{n\ge 1}$ be a sequence of semipositive{}
    metrics that converge to $\|\cdot\|_{i}$ for each $i$. Then
    \begin{multline*}
      \int_{X^{\an}}\log\|s_{d}\|_{d}
      \chern_{1}(\overline L_{0})\land \dots \wedge \chern_{1}(\overline L_{d-1})
      \land \delta _{Y}\\ = \lim_{n_{0},\dots,n_{d}\to \infty} 
      \int_{X^{\an}}\log\|s_{d}\|_{d,n_{d}}
      \chern_{1}(\overline L_{0,n_{0}})\land \dots \wedge \chern_{1}(\overline L_{d-1,n_{d-1}})
      \land \delta _{Y}.
    \end{multline*}
  
  \end{enumerate}
\end{thm}

\begin{proof} In the Archimedean case, when $X$ is smooth, this is
  proved in \cite[th\'eor\`emes 5.5.2(2) and 5.5.6(6)]{Maillot:GAdvt}. 
  For
  completions of number fields 
  this is proved in 
  \cite[Theorem~4.1]{ChambertThuillier:MMel}, both in the Archimedean and
  non-Archimedean cases. Their argument can be easily extended
  to cover the general case.
\end{proof}

\begin{defn} \label{def:3}
  The \emph{local height}
  \index{height of cycles!local}%
  on $X$ is the function that, to each
  $d$-dimensional cycle $Y$ and each family of DSP
  metrized line bundles with sections
  $(\ov L_{i},s_{i})$, $i=0,\dots,d$, such that the sections meet $Y$
  properly, associates a real 
  number $\h_{\ov 
    L_{0},\dots,\ov L_{d}}(Y;s_{0},\dots,s_{d})$ determined inductively by
  the properties:
  \begin{enumerate}
  \item \label{item:77} $\h(\emptyset)=0$;
  \item \label{item:78} if $Y$ is a cycle of dimension $d\ge0$, then
    \begin{multline*}
      \h_{\ov
      L_{0},\dots,\ov L_{d}}(Y;s_{0},\dots,s_{d})=
    \h_{\ov
      L_{0},\dots,\ov L_{d-1}}(Y\cdot \div s_{d};s_{0},\dots,s_{d-1})\\
    -\int_{X^{\an}}\log\|s_{d}\|_{d}
    \chern_{1}(\overline L_{0})\land \dots \wedge \chern_{1}(\overline L_{d-1})
    \land \delta _{Y}.
    \end{multline*}
    \nomenclature[ahL01]{$\h_{\ov L_{0},\dots,\ov
        L_{d}}(Y;s_{0},\dots,s_{d})$}{local height}%
  \end{enumerate}
\end{defn}
In particular, for $p\in X(K)\setminus |\div (s_{0})|$,
\begin{equation}\label{eq:96}
  \h_{\ov
      L_{0}}(p;s_{0})=-\log\|s_0(p)\|_{0}.
\end{equation}

\begin{rem} Definition \ref{def:3} makes sense thanks to Theorem
  \ref{prop:4}. We have chosen to introduce first the measures and then
  heights for simplicity of the exposition. Nevertheless, the approach
  followed in the literature is the inverse, because the proof of
  Theorem \ref{prop:4} relies on the properties of local heights. The
  interested reader can consult  \cite{Chambert-Loir10:heigh} for
  more details.
\end{rem}

\begin{rem}\label{rem:10} Definition \ref{def:3} works better when the
  variety $X$ is projective. In this case, for every cycle $Y$
there exist sections that meet $Y$
  properly, thanks to the 
  moving lemma. This does not necessarily occur for arbitrary proper
  varieties. Nevertheless, we will be able to define the global height
  (Definition \ref{def:61}) of any cycle of a proper variety by
  using Chow's lemma. Similarly we will be able to define the toric local 
  height (Definition \ref{def:35}) of any cycle of a proper toric variety.
\end{rem}

\begin{rem} When $X$ is regular and the metrics are smooth (in the
  Archimedean case) or algebraic (in the non-Archimedean case), the
  local heights of Definition \ref{def:3} agree with the local heights
  that can be derived using the Gillet-Soul\'e arithmetic intersection
  product. In particular, in the Archimedean case, this local height
  agrees with the Archimedean contribution of the Arakelov global
  height introduced by Bost, Gillet and Soul\'e in
  \cite{BostGilletSoule:HpvpGf}. In the non-Archimedean case, the
  local height with respect to an algebraic metric can be interpreted
  in terms of an intersection product. Assume that $Y$ is prime and
  choose models $(\mathcal{X}_{i},\mathcal{L}_{i},e_{i})$ of
  $(X,L_{i})$ that realize the algebraic metrics of $\ov
  L_{i}$. Without loss of generality, we assume that all the
  models $\mathcal{X}_{i}$ agree with a common model
  $\mathcal{X}$. The sections $s^{\otimes e_{i}}_{i}$ can be seen as
  rational sections of $\mathcal{L}_{i}$ over $\mathcal{X}$. With the
  notations in Definition \ref{def:76}, the equation \eqref{eq:7}
  implies that
  \begin{displaymath}
    \log\|s_{d}(\xi_{V})\|=\frac{\log|\varpi
      |\ord_{V}(s_{d}^{\otimes e_{d}})}{e_{d}\ord_{V}(\varpi )}.
  \end{displaymath}
  Therefore, in this case the recurrence in Definition \ref{def:3}\eqref{item:78} can be
  written as
    \begin{multline*}
      \h_{\ov
      L_{0},\dots,\ov L_{d}}(Y;s_{0},\dots,s_{d})=
    \h_{\ov
      L_{0},\dots,\ov L_{d-1}}(Y\cdot \div (s_{d});s_{0},\dots,s_{d-1})\\
    -\frac{\log |\varpi |}{e_{0}\dots e_{d}}\sum_{V\in
      \wt{\mathcal{Y}}_{0}^{(0)}}\ord_{V}(s_{d}^{\otimes e_{d}})
    \deg_{\mathcal{L}_{0},\dots,\mathcal{L}_{d-1}} (V).
    \end{multline*}
\end{rem} 

\begin{rem}\label{rem:5}
  It is a fundamental observation by Zhang \cite{Zhang:_small} that
  the non-Archimedean contribution of the Arakelov global height of a variety
  can be expressed in terms of a family of metrics. In particular,
  this global height only depends on the metrics and not on a
  particular choice of
  models, exhibiting the analogy between the Archimedean and
  non-Archimedean settings. The local heights were extended by Gubler
  \cite{GublerHab, Gu03} to non-necessarily discrete valuations
  and he also weakened the hypothesis of proper
  intersection. 
\end{rem}

\begin{rem}\label{rem:4}
  The local heights of Definition \ref{def:3} agree with the local
  heights introduced by Gubler, see~\cite[Proposition~3.5]{Gu03} for
  the Archimedean case and \cite[Remark~9.4]{Gu03} for the
  non-Archimedean case.  In the Archimedean case, the local height in
  \cite{Gu03} is defined in terms of a refined star product of Green
  currents based on \cite{Burgos:Gftp}. The hypothesis needed in
  Gubler's definition of local heights are weaker than the ones we
  use.  We have chosen the current definition because it is more
  elementary and suffices for our purposes.
\end{rem}

\begin{thm} \label{thm:1} The local height function satisfies the
  following properties. 
  \begin{enumerate}
  \item \label{item:7} It is symmetric and multilinear
    with respect to $\otimes$ in the
    pairs $(\ov L_{i},s_{i})$, $i=0,\dots,d$, provided that all 
    terms are defined.
  \item \label{item:8} Let $\varphi\colon X'\to X$ be a morphism of
    proper varieties over $K$, $Y$ a $d$-dimen\-sional cycle of $X'$,
    and $(\ov L_{i},s_{i})$, $i=0,\dots,d$, a collection of DSP
    metrized line bundles on $X$ with a section. Then
    \begin{displaymath}
            \h_{\varphi^{\ast} \ov
      L_{0},\dots,\varphi^{\ast} \ov L_{d}}(Y;\varphi^{\ast}
    s_{0},\dots,\varphi^{\ast} s_{d})=
      \h_{\ov
      L_{0},\dots,\ov L_{d}}(\varphi_{\ast}Y;s_{0},\dots,s_{d}),
    \end{displaymath}
    provided that both terms are defined.
  \item \label{item:9} Let $X$ be a 
    proper variety
    over $K$, $Y$ a $d$-dimen\-sional cycle of $X$, and $(\ov
    L_{i},s_{i})$, $i=0,\dots,d$, a collection of DSP metrized
    line bundles 
    on $X$ with sections that meet $Y$
    properly. Let $f $ be a rational function such that
    the sections $s_{0},\dots,s_{d-1},fs_{d}$ also meet $Y$ properly. Let
    $Z$ be the zero-cycle $Y\cdot \div 
    (s_{0})\cdots
    \div (s_{d-1})$. Then
    \begin{displaymath}
      \h_{\ov
      L_{0},\dots,\ov L_{d}}(Y;s_{0},\dots,s_{d-1},s_{d})-
      \h_{\ov
      L_{0},\dots,\ov L_{d}}(Y;s_{0},\dots,s_{d-1},fs_{d})=
    \log|f(Z)|,
    \end{displaymath}
    where, if $Z=\sum_{l} m_{l}p_{l}$, then $f(Z)=\prod_{l}
    f(p_{l})^{m_{l}}$.
  \item \label{item:6} Let $\ov {L}'_{d}=(L_{d},\|\cdot\|')$ be another choice of
    metric. Then
    \begin{multline*}
      \h_{\ov
      L_{0},\dots,\ov L_{d-1},\ov L_{d}}(Y;s_{0},\dots,s_{d})-
      \h_{\ov
      L_{0},\dots,\ov L_{d-1},\ov L'_{d}}(Y;s_{0},\dots,s_{d})=\\
    -\int_{X^{\an}}\log (\|s_{d}(p)\|/\|s_{d}(p)\|') \chern_{1}(\overline
    L_{0})\land \dots \wedge \chern_{1}(\overline L_{d-1})\land \delta _{Y}
    \end{multline*}
    is independent of the choice of sections.
  \end{enumerate}
\end{thm}
\begin{proof}
In the Archimedean case,  statement \eqref{item:7} is
\cite[Proposition~3.4]{Gu03}, statement \eqref{item:8}  is
\cite[Proposition 3.6]{Gu03}. 
In the non-Archimedean case, statement \eqref{item:7} and \eqref{item:8} 
are \cite[Remark 9.3]{Gu03}.
The other two statements follow easily from the definition.
\end{proof}

\section{Metrics and global heights over adelic fields}
 \label{sec:adel-metr-glob}

To define global heights of cycles, we first introduce the notion
of adelic field, which is a generalization of the notion of global
field.  In \cite{Gu03} one can find a more general theory of global heights
  based on the concept of $M$-fields.

  \begin{defn} \label{def:6} Let $\K$ be a field and
    $\mathfrak{M}$ 
\nomenclature[aMK]{$\mathfrak{M}$}{absolute values and weights of
  an adelic field}%
\nomenclature[aKad1]{$\K$}{adelic field}%
    a family of absolute values on $\K$
    with positive real weights. The elements of $\mathfrak{M}$ are
    called \emph{places}\index{places}. For each place $v\in \mathfrak{M}$ we denote by
    $|\cdot|_v$
\nomenclature[svert01]{$\vert\cdot\vert_v$}{absolute value of an adelic field}%
the corresponding absolute value, by $n_v\in
    \R_{>0}$
\nomenclature[an]{$n_v$}{weight of an absolute value}%
the weight, and by $\K_{v}$
\nomenclature[aKad2]{$\K_v$}{completion of an adelic field}%
the completion of $\K$ with respect to $|\cdot|_v$.  We
    say that $(\K,\mathfrak{M})$ is an \emph{adelic field}
\index{adelic field}%
if
  \begin{enumerate}
  \item for each $v\in \mathfrak{M}$, the absolute value $|\cdot|_v$ is
    either Archimedean or associated to a nontrivial discrete valuation;
\item for each $\alpha\in
  \K^{\times}$, 
  $|\alpha|_v=1$ except a for a finite number of $v$.
  \end{enumerate}
  For an adelic field $(\K,\mathfrak{M})$ and $\alpha\in
  \K^{\times}$, the \emph{defect}
  \index{defect of an adelic field}%
  of $\alpha$ is
    $$
    \df(\alpha)=\sum_{v\in \mathfrak{M}}n_v\log|\alpha|_v.
    $$
    \nomenclature[ad0ef]{$\df(\alpha )$}{defect of the product formula}%
    Since $\df\colon \K^{\times}\to \R$ is a group homomorphism, we
    have that $\df(\K^{\times})$ is a subgroup of $\R$.  If
    $\df(\K^{\times})=0$, then $\K$ is said to satisfy the 
    \emph{product formula}.\index{product formula} The \emph{group of
      global heights}\index{adelic field!group of global heights of}
    of $\K$ is $\R/\!\df(\K^{\times})$.
\end{defn}

Observe that the complete fields $\K_{v}$ are either $\R$, $\C$ or of the
kind of fields considered 
in \S \ref{sec:algebr-metr-discr}.

\begin{defn}\label{def:80}
Let $(\K,\mathfrak{M})$ be an adelic field and $\F$ a finite
extension of $\K$. For each $v\in \mathfrak{M}$, put
$\mathfrak{N}_v$ for the set of
absolute values $|\cdot|_w$ of $\F$ that extend $|\cdot|_v$, with 
weight
$$ n_w=\frac{[\F_w:\K_v]}{[\F:\K]}n_v.$$ 
Set $\mathfrak{N}=\coprod_{v}
\mathfrak{N}_v$. Then $(\F,\mathfrak{N})$ is an adelic field. In this
case, we say that $(\F,\mathfrak{N})$ is an 
\emph{adelic field extension}\index{adelic field!extension of} of
$(\K,\mathfrak{M})$. 
\end{defn}

The classical examples of adelic fields are 
number fields and function fields of curves.

\begin{exmpl}\label{exempleQ}
  Let $\mathfrak{M}_\Q$ be the set of the Archimedean and $p$-adic
  absolute values of $\Q$, normalized in the standard way, with all
  weights equal to $1$.  Then $(\Q,\mathfrak{M}_{\Q})$ is an adelic
  field that satisfies the product formula. We identify
  $\mathfrak{M}_{\Q}$ with the set $\{\infty\}\cup \{\text{primes of }
  \Z\}$.  For a number field $\K$, the construction in Definition
  \ref{def:80} gives an
  adelic field $(\K,\mathfrak{M}_{\K})$ which satisfies the product
  formula too.
\end{exmpl}

\begin{exmpl}\label{exempleF}
  Consider the function field $K(C)$ of a smooth projective curve $C$
  over a field $k$. For each closed point $v\in C$ and $\alpha \in
  K(C)^{\times}$, we denote by $\ord_{v}(\alpha )$ the order of
  $\alpha $ in the discrete valuation ring $\mathcal{O}_{C,v}$.  We
  associate to each $v$ the absolute value and weight given by
  $$
  |\alpha|_v=c_{k}^{-\ord_v(\alpha)}, \quad n_{v}= [k(v):k]
  $$ 
  with
\begin{displaymath}
    c_{k}=
    \begin{cases}
      \e&\text{if } \#k = \infty,\\
      \#k\ & \text{if } \#k < \infty.
    \end{cases}
  \end{displaymath}
  Let  $\mathfrak{M}_{K(C)}$ denote this set of absolute values and
  weights. The pair $(K(C),\mathfrak{M}_{K(C)})$ is an
  adelic field which satisfies the product formula, since the degree
  of a principal divisor is zero. 

  More generally, let $\K$ be a finite
  extension of $K(C)$. Following Definition \ref{def:80} we obtain an
  adelic field extension $(\K,\mathfrak{M}_{\K/K(C)})$
  In this geometric setting, this construction can be explicited as
  follows.  Let $\pi \colon B \to C$ be a dominant morphism of smooth
  projective curves over $k$ such that the finite extension
  $K(C)\hookrightarrow \K$ can be identified with $\pi ^{\ast}\colon
  K(C)\hookrightarrow K(B)$.  For a closed point $v\in C$, the
  absolute values of~$\K$ that extend $|\cdot|_{v}$ are in bijection
  with the closed points of the fiber of $v$. For each closed point
  $w\in \pi^{-1}(v)$, the corresponding absolute value and weight are
  given, for $\beta\in K(B)^{\times}$, by
  \begin{displaymath} 
    |\beta|_w=c_{k}^{-\frac{\ord_w(\beta)}{e_{w}}}, 
    \quad n_{w}= \frac{e_{w}[k(w):k]}{[K(B):K(C)]},
  \end{displaymath}
  where $e_{w}$ is the ramification index of $w$ over $v$.
  Observe that the structure of adelic field on $\K$ depends on the
  extension and not just on the field $K(B)$. 
  For instance,  $(K(C), \mathfrak{M}_{K(C)})$ corresponds to the
  identity map $C\to C$ in the above construction, but other finite morphism
  $\pi\colon C\to C$ may give a different structure of adelic field on
  $K(C)$. The projection formula for the map $\pi $ implies that, for each
  $v\in \mathfrak{M}_{K(C)}$, the equation
  \begin{displaymath}
    [\K:K(C)]=\sum_{w\mid v}[\K_{w}:K(C)_{v}]
  \end{displaymath}
  is satisfied. From this, it is easy to deduce that $(\K,\mathfrak{M}_{\K/K(C)})$
  satisfies the product formula.
\end{exmpl}

  A simple example of an adelic field that does not satisfy the
  product formula is constructed below. This kind of adelic fields
  can be useful when studying arithmetic intersection on moduli spaces
  (see for instance \cite[\S 6]{BruinierBurgosKuehn:bpaihs}).
 
\begin{exmpl}\label{exm:43}
  Let $N\ge 2$ be an
  integer write $\mathfrak{M}_{N}=\{p\in \mathfrak{M}_{\Q}\mid p
  \nmid N\}$
  and $\K=(\Q,\mathfrak{M}_{N})$. Then $\K$ is an adelic field and 
  \begin{displaymath}
    \df(\K^{\times})=\sum_{p\mid N} \Z\log(p).
  \end{displaymath}
  
\end{exmpl}

\begin{defn} \label{def:46} Let $(\K,\mathfrak{M})$ be an adelic field. Let
  $X$ be a proper variety over~$\K$ and $L$ a line bundle on $X$. For
  each $v\in \mathfrak{M}$ set $X_{v}=X\times \Spec(\K_{v})$ and
  $L_{v}=L\times \Spec(\K_{v})$. 
  A \emph{metric}
    \index{metric!on a line bundle!over an adelic field}%
    on $L$ is a family of metrics
    $\|\cdot\|_v$, $v\in \mathfrak{M}$, where $\|\cdot\|_v$ is a
    metric on $L_v^{\an}$. We will denote by $\ov
    L=(L,(\|\cdot\|_v)_{v})$ the corresponding metrized line
    bundle. This metric is said to be
    \emph{semipositive{}}
\index{semipositive{} metric!over an adelic field}%
(respectively \emph{DSP})
\index{DSP (difference of semipositive) metric!over an adelic field}%
if $\|\cdot\|_v$ is semipositive{}
    (respectively {DSP}) for all $v\in \mathfrak{M}$.

Let $Y$ be a $d$-dimensional 
cycle of $X$ and $(\ov L_i,s_{i})$, $i=0,\dots, d$, DSP metrized line bundles
on $X$ with rational sections meeting $Y$ properly. For  $v\in\mathfrak M$, we note
  \begin{equation*}
    \h_{v,\ov L_0,\dots, \ov L_d}(Y;s_0,\dots,s_d)=  
\h_{\ov L_{0,v},\dots, \ov L_{d,v}}(Y_{v};s_{0,v},\dots,s_{d,v})
  \end{equation*}
where $s_{i,v}$ is the rational section of $L_{i,v}$ induced by
$s_{i }$. 

\end{defn}

For cycles defined over an arbitrary adelic field, the global heights
with respect to DSP metrized line bundles may not
be always defined. An obvious class of cycles where the global height
is well-defined is the following.

\begin{defn} \label{def:47} Let $(\K,\mathfrak{M})$ be an adelic
  field, $X$ a proper variety over $\K$ and~$\ov L_i$, $i=0,\dots, d$,
  a family of DSP metrized line bundles on $X$. Let $Y$ be a
  $d$-dimensional irreducible subvariety of $X$. We say that $Y$ is
  \emph{integrable}\index{integrable cycle} with respect to $\ov
  L_0,\dots,\ov L_d$ if there is a birational proper map
  $\varphi\colon Y'\to Y$ with $Y'$ projective, and rational sections
  $s_i$ of $\varphi ^{\ast} L_i$, $i=0,\dots,d$, meeting $Y'$
  properly, such that for all but a finite number of $v\in
  \mathfrak{M}$,
  \begin{displaymath}
    \h_{v,\varphi ^{\ast} \ov L_0,\dots,\varphi ^{\ast}
      \ov L_d}(Y';s_0,\dots,s_d)=0.
  \end{displaymath}
    \nomenclature[ahL011]{$\h_{v,\ov L_{0},\dots,\ov
        L_{d}}(Y;s_{0},\dots,s_{d})$}{local height in the adelic case}%
  A $d$-dimensional cycle is \emph{integrable} if all its components
  are integrable.
\end{defn}

It is clear from the definition that the notion of integrability of
cycles is closed under tensor products of DSP metrized line
bundles. In addition, it satisfies the following 
properties.

\begin{prop}\label{prop:94} Let $(\K,\mathfrak{M})$ be an adelic
  field, $X$ a proper
  variety over $\K$ and~$\ov L_i$, $i=0,\dots, d$, a family of
  DSP metrized line bundles on $X$. 
  \begin{enumerate}
  \item \label{item:10} Let $Y$ be a
  $d$-dimensional irreducible subvariety of
  $X$ which is integrable with respect to $\ov L_i$, $i=0,\dots,
    d$. Let $\varphi\colon
    Y'\to Y$ be a proper birational map, with $Y'$ projective, and
    $s_i$, $i=0,\dots,d$, rational sections of 
    $\varphi ^{\ast} L_i$ meeting $Y'$ properly. Then
    \begin{displaymath}
      \h_{v,\varphi ^{\ast} \ov L_0,\dots,\varphi ^{\ast}
        \ov L_d}(Y';s_0,\dots,s_d)=0
    \end{displaymath} 
    for all but a finite number of $v\in \mathfrak{M}$.
  \item \label{item:11} Let $\psi\colon X'\to X$ be a morphism of
    proper varieties over $\K$ and $Y$ a $d$-dimensional cycle of
    $X'$. Then $Y$ is integrable with respect to $\psi^{\ast}\ov
    L_{0},\dots,\psi^{\ast}\ov L_{d}$ if and only if $\psi_{\ast}Y$ is
    integrable with respect to $\ov L_{0},\dots,\ov L_{d}$.
  \end{enumerate}
\end{prop}

\begin{proof}
  To prove \eqref{item:10}, we start by assuming that there are rational
  sections $s'_i$, $i=0,\dots,d$,  of 
    $\varphi ^{\ast} L_i$ meeting $Y'$ properly such that
    \begin{displaymath}
      \h_{v,\varphi ^{\ast} \ov L_0,\dots,\varphi ^{\ast}
        \ov L_d}(Y';s'_0,\dots,s'_d)=0
    \end{displaymath} 
    for all but a finite number of $v\in \mathfrak{M}$. Since $Y'$ is
    projective, there are rational sections $s''_{i}$ of $
    \varphi^{\ast} L_{i}$, $i=0,\dots,d$, such
    that, for any partition $I\sqcup J=\{0,\dots,d\}$, both families of
    sections $\{s_{i},i\in I,s''_{j},j\in J\}$ and $\{s'_{i},i\in
    I,s''_{j},j\in J\}$ meet $Y'$ properly. Using the
    definition of adelic field and 
    Theorem~\ref{thm:1}\eqref{item:9} we can deduce that     
    \begin{displaymath}
      \h_{v,\varphi ^{\ast} \ov L_0,\dots,\varphi ^{\ast}
        \ov L_d}(Y';s_0,\dots,s_d)=0
    \end{displaymath} 
    for all but a finite number of $v\in \mathfrak{M}$, which proves
    the statement in this case.

    We prove now the general case. Then there is a birational proper map
    \begin{displaymath}
      Y''\overset{\varphi'}{\longrightarrow}Y,
    \end{displaymath}
    with $Y''$ projective, and sections $s'_i$ of
    $\varphi'^{\ast}L_i$, $i=0,\dots,d$, meeting $Y''$ properly such
    that
    \begin{displaymath}
      \h_{v,\varphi ^{\ast} \ov L_0,\dots,\varphi ^{\ast}
        \ov L_d}(Y'';s'_0,\dots,s'_d)=0
    \end{displaymath} 
    for all but a finite number of $v\in \mathfrak{M}$. There is a
    commutative diagram of proper birational morphisms
    \begin{displaymath}
      \xymatrix{ Y''\ar[r]^{\varphi'}& Y\\
        Y'''\ar[u]^{\varphi'''}\ar[r]^{\varphi''}& Y' \ar[u]_{\varphi}.
      }
    \end{displaymath}
    Since $Y'$ and $Y''$ are projective, we can
    find rational sections $s'_i$, $i=0,\dots,d$,  of 
    $\varphi ^{\ast} L_i$ meeting $Y'$ properly and such that the
    family of sections $\varphi''{} ^{\ast} s'_i$, $i=0,\dots,d$,
    meet $Y'''$ properly, and rational sections $s''_i$, $i=0,\dots,d$,  of 
    $\varphi'{} ^{\ast} L_i$ meeting $Y''$ properly and such that the
    family of sections $\varphi'''{} ^{\ast} s''_i$, $i=0,\dots,d$,
    meet $Y'''$ properly. Then we deduce the result in this case from the
    previous case and Theorem~\ref{thm:1}\eqref{item:8}.

    The proof of \eqref{item:11} can be done in a  similar way.
\end{proof}

\begin{defn} \label{def:61} Let $X$ be a proper variety over $\K$,
  $\ov L_0,\dots,\ov L_d$ DSP metrized line bundles on $X$, and
  $Y$ an integrable $d$-dimensional irreducible subvariety of 
  $X$. Let $Y'$
  and $s_0,\dots, s_d$ be as in Definition \ref{def:47}. The
  {global height of $Y$ with respect to $s_0,\dots,s_d$}
  is defined as
  $$ 
  \h_{\ov L_0,\dots,\ov L_d}(Y;s_0,\dots,s_d)=\sum_{v\in \mathfrak{M}}n_v
  \h_{v,\varphi^{\ast}\ov L_0,\dots,\varphi^{\ast}\ov
    L_d}(Y';s_0,\dots,s_d)\in\R. 
  $$
  The \emph{global height} of $Y$,
\index{height of cycles!global}%
  denoted $\h_{\ov L_0,\dots,\ov L_d}(Y)$,
\nomenclature[ahL03]{$\h_{\ov L_0,\dots,\ov L_d}(Y)$}{global height}%
  is the class of $\h_{\ov
    L_0,\dots,\ov L_d}(Y;s_0,\dots,s_d)$ in the quotient group
  $\R/\!\df(\K^\times)$. The global height of integrable cycles is
  defined by linearity.
\end{defn}

Observe that the global height is well-defined as an element of
$\R/\!\df(\K^\times)$ because of Theorem~\ref{thm:1}(\ref{item:9}).  In
particular, if $\K$ satisfies the product formula, the global height is
a well-defined real number.

\begin{prop}
    Let $(\F,\mathfrak{N})$ be a finite adelic
    field extension of  $(\K,\mathfrak{M})$.  Let $X$ be a $\K$-variety, $\ov
  L_{i}$, $i=0,\dots,d$, DSP metrized line bundles on $X$ and
  $Y$ a $d$-dimensional integrable cycle on $X$. Let $\pi \colon
  X_{\F}\to X$ be the morphism obtained by base change. Denote by
  $Y_{\F}$  and $\pi^{\ast}\ov
  L_{i}$, $i=0,\dots,d-1$, the cycle and DSP metrized line bundles
  obtained by base change. Then
  \begin{displaymath}
    \h_{\pi ^{\ast}\ov L_{0},\dots,\pi ^{\ast}\ov L_{d}}(Y_{\F})=
    \h_{\ov L_{0},\dots,\ov L_{d}}(Y)\text{ in }\R/\df(\F).
  \end{displaymath}
\end{prop}
\begin{proof}
    This is proved by induction using
  Proposition \ref{prop:28} in the algebraic case and the
  formula for the change
  of variables of an integral in the smooth case. Then the
  semipositive{} case follows by continuity and the DSP case by
  multilinearity. 
\end{proof}

\begin{thm} \label{thm:14} The global height of integrable cycles
  satisfies the following properties.
  \begin{enumerate} 
  \item \label{item:89} It is symmetric and multilinear
    with respect to tensor products of DSP metrized line bundles. 
  \item \label{item:93} Let $\varphi\colon X'\to X$ be a morphism of
    proper varieties over $\K$, $\ov L_{i}$, $i=0,\dots, d$, DSP
    metrized line bundles on $X$. Let $Y$ be a $d$-dimensional cycle
    of $X'$, integrable with respect to the metrized line
    bundles $\varphi^{\ast}\ov L_{0},\dots,\varphi^{\ast}\ov
    L_{d}$. Then
    \begin{equation*}
      \h_{\varphi^{\ast} \ov
        L_{0},\dots,\varphi^{\ast} \ov L_{d}}(Y)=
      \h_{\ov
        L_{0},\dots,\ov L_{d}}(\varphi_{\ast}Y).
    \end{equation*}
  \end{enumerate}
\end{thm}

\begin{proof}
  The first statement follows from Theorem \ref{thm:1}(\ref{item:7}),
  while the second follows readily from Proposition
  \ref{prop:94}\eqref{item:11} and Theorem
  \ref{thm:1}(\ref{item:8}).
\end{proof}

We turn now our attention to number fields and function fields.

\begin{defn} \label{def:66} A \emph{global field} is a finite
  extension $\K/\Q$ or $\K/K(C)$ for a smooth projective curve $C$
  over a field $k$, with the structure of adelic field given in
  examples \ref{exempleQ} or \ref{exempleF}, respectively. To lighten
  the notation, we will usually denote those global field by $\K$ and
  the set of places by $\mathfrak{M}_{\K}$, although, in the function
  field case, the structure of adelic field depends on the particular
  extension.
\end{defn}


Our use of the terminology ``global field'' is slightly more
general than the usual one where, in
the function field case, the base field is finite and the
extension is separable.

\begin{defn} \label{def:64} Let $\K$ be a global field. Let
  $X$ be a proper variety over~$\K$ and $L$ a line bundle on $X$. For
  each $v\in \mathfrak{M}_{\K}$ set $X_{v}=X\times \Spec(\K_{v})$ and
  $L_{v}=L\times \Spec(\K_{v})$. 
  A metric on $L$ is called
    \emph{quasi-algebraic}
    \index{metric!quasi-algebraic}%
    if there exists a finite subset
    $S\subset \mathfrak{M}_{\K}$ containing the Archimedean places, an
    integer $e\ge 1$ and a proper model $(\cX,\cL,e)$ over
    $\K^{\circ}_{S}$ of $(X,L)$ such that, for each $v\notin S$, the
    metric $\|\cdot\|_v$ is induced by the localization of this model
    at~$v$.
\end{defn}

For global fields and quasi-algebraic metrics, all cycles are
integrable.

\begin{prop}\label{prop:29}
  Let $\K$ be a global field and $X$ a proper variety
  over $\K$ of dimension $n$. Let $d\le n$ and  $\ov L_{i}$,
  $i=0,\dots,d$,  a family of line bundles with quasi-algebraic DSP
  metrics.  Then every $d$-dimensional cycle of $X$ is
  integrable with respect to $\ov L_0,\dots,\ov L_d$.
\end{prop}

\begin{proof} It is enough to prove that every prime cycle is
  integrable. Applying Chow's lemma to the support of the cycle
  and using that the inverse image of a quasi-algebraic metric is
  quasi-algebraic, we are reduced to the case when $X$ is projective.

  We proceed by induction on the dimension of $Y$. For $Y=\emptyset$,
  the statement is clear, and so we consider the case when
  $d=\dim(Y)\ge 0$.  Let $Y$ be a $d$-dimensional cycle of $X$ and
  $s_i$, $i=0,\dots,d$, rational sections of $L_i$ that intersect $Y$
  properly. Let $\wt Y$ be the normalization of $Y$. By the
hypothesis of quasi-algebricity, there is a finite subset $S\subset \mathfrak{M}_{\K}$ containing the
  Archimedean places such that there exists a normal proper model
  $\cY$ over $\K^{\circ}_{S}$ of $\wt Y$ and models $\cL_{i}$ of $L_{i}^{\otimes
    e_{i}}|_{Y}$,  $i=0,\dots, d$, for some integers $e_{i}\ge1$, all
  of them being line bundles over $\cY$. Then $s_{d}^{\otimes
    e_{d}}|_{\cY}$ is a  nonzero rational section of
  $\cL_{d}$ and so it defines a finite number of vertical
  components. Let $v\notin S$ be a place that is not below
  any of these vertical components. Let $\cY_{v}$ be the fibre of
  $\cY$ over $v$.
Let $V_{j}$, $j=1,\dots, l$, be the components of this fibre, and
$\xi_{j}$ the corresponding points of $\wt Y^{\an}$ given by
Proposition \ref{prop:92}. On the one hand, the measure 
$ \chern_{1}(\overline L_{0})\land \dots \wedge \chern_{1}(\overline L_{d-1})
    \land \delta _{Y}$ is concentrated on these points. On the other
    hand, by Proposition \ref{prop:descmetric},
    $\|s_{d}(\xi_{j})\|_{d,v}=1$ for all $j$. Hence, 
\begin{displaymath}
   \h_{v,\ov L_0,\dots,\ov L_d}(Y;s_0,\dots,s_d)=    \h_{v,\ov
     L_0,\dots,\ov L_{d-1}}(Y\cdot \div(s_{d});s_0,\dots,s_{d-1}),
\end{displaymath}
because of the definition of local heights. 
The statement follows then from the
inductive hypothesis. 
\end{proof}





\chapter{The Legendre-Fenchel duality}
\label{sec:legendre-fench-dual}

In this chapter, we explain the notions of convex analysis that we will
use in our study of the arithmetic of toric varieties. The central
theme is the Legendre-Fenchel duality of concave functions.  
A basic reference in this subject is the classical
book by Rockafellar \cite{Roc70} and we will refer to it for many of
the proofs.

Although the usual references in the literature deal with convex
functions, we will work instead with \emph{concave} functions.  These
are the functions which arise in the theory of toric
varieties. In this respect, we remark that the functions which are
called ``convex'' in the classical books on toric varieties
\cite{Kempfals:te,Ful93} are concave
in the sense of convex analysis. 

\section{Convex sets and convex decompositions}
\label{sec:convex-sets-convex}

Let $N_{\R}\simeq\R^{n}$ be a real vector space of dimension $n$ and 
$M_{\R}=N_{\R}^{\vee}$ its dual space. 
\nomenclature[aNl02]{$N_{\R}$}{real vector space}%
\nomenclature[aMl02]{$M_{\R}$}{vector space dual to $N_{\R}$}%
The pairing between $x\in
M_{\R}$ and $u\in N_{\R}$ 
 will be alternatively denoted by  
$\langle x,u\rangle$, $x(u)$ or $u(x)$. 

A non-empty subset $C$ of $N_{\R}$ is \emph{convex}\index{convex set}
if, for each pair of points $u_1,u_2\in C$, the line segment
\begin{displaymath}
  \ov{u_1u_2}=\{t u_1+(1-t)u_2\mid 0\le t\le 1 \}
\end{displaymath}
is contained in $C$.  Throughout this text, convex sets are assumed to
be non-empty.  A non-empty subset $\sigma\subset N_{\R}$ is a
\emph{cone}\index{cone} if $\lambda \sigma = \sigma$ for all
$\lambda\in \R_{>0}$.

The \emph{affine hull}
\index{convex set!affine hull of}%
of a convex set $C$, denoted $\aff(C)$, is the minimal affine space which contains
it.
\nomenclature[a]{$\aff(C)$}{affine hull of a convex set}%
The \emph{dimension} of $C$ is defined as the dimension of its affine hull.
\index{convex set!dimension of}%
The \emph{relative interior} of $C$,
\index{convex set!relative interior of}%
\nomenclature[ari]{$\ri(C)$}{relative interior of a convex set}%
denoted $\ri(C)$, is defined as the
interior of~$C$ relative to its affine hull. The
\emph{recession cone} 
\index{convex set!recession cone of}%
\nomenclature[arec]{$\rec(C)$}{recession cone of a convex set}%
of $C$, denoted by $\rec(C)$, is the set
\begin{displaymath}
  \rec(C)=\{u\in N_{\R}\mid C+u\subset C\}.
\end{displaymath}
It is a cone of $N_{\R}$. The \emph{cone}
\index{cone!of a convex set}%
\nomenclature[ac0]{$\cc(C)$}{cone of a convex set}%
 of $C$ is defined as 
\begin{displaymath}
  \cc(C) =\ov{\R_{>0}(C\times \{1\})}\subset N_{\R}\times
\R_{\ge 0}.
\end{displaymath}
It is a closed cone. If $C$ is closed, then $\rec(C)\times\{0\}=\cc(C)\cap
(N_{\R}\times \{0\})$.

\begin{defn}
  Let $C$ be a convex set. A convex subset $F\subset C$ is called a
  \emph{face} of $C$
\index{face of a convex set}%
  if, for every closed line segment $\ov{u_1u_2}\subset C$ such that
  $\ri(\ov{u_1u_2}) \cap F\not= \emptyset$, the inclusion~$\ov{u_1u_2}\subset F$ holds.
  A face of $C$ of codimension 1 is called a \emph{facet}.
  \index{facet}%
  A non-empty subset $F\subset C$ is called an
  \emph{exposed face} of $C$
\index{face of a convex set!exposed}%
  if there exists $x\in M_{\R}$ such that
  $$F=\{u\in C\mid \langle x,u\rangle \le
  \langle x, v\rangle ,\, \forall
  v\in C\}.$$ 
\end{defn}
Any exposed face of a convex set is a face, and the facets of a convex
set are always exposed.  However, a convex set may have faces which
are not exposed. For instance, think about the four points of junction
of the straight lines and bends of the boundary of the inner area of a
racing track in a stadium.

\begin{defn}
  \label{def:27}
  Let $\Pi$ be a non-empty collection of convex subsets of
  $N_{\R}$. The collection $\Pi$ is called
  a \emph{convex subdivision}\index{convex subdivision} if it
  satisfies the conditions:
\begin{enumerate}
\item \label{item:58} every face of an element of $\Pi $
is also in $\Pi$;
\item \label{item:59} every two elements of $\Pi$ are either
disjoint or they intersect in a
common face.
\end{enumerate}
If $\Pi $ satisfies only \eqref{item:59}, then it is called a
\emph{convex decomposition}.\index{convex decomposition} The
\emph{support} of $\Pi$
\index{convex decomposition!support of}%
\nomenclature[svert00]{$\vert\Pi\vert$}{support of a convex decomposition}%
is defined as 
the set $|\Pi|=\bigcup_{C \in \Pi } C$.  We say that $\Pi $ is
\emph{complete}\index{convex decomposition!complete} if its support is
the whole of $N_{\R}$. For a given set
$E\subset N_{\R}$, we say that $\Pi $ is a convex subdivision (or
decomposition) \emph{in $E$}
whenever $|\Pi|\subset
E$. A convex subdivision in $E$ is called \emph{complete}
\index{convex subdivision!complete}%
if $|\Pi|= E$.
\end{defn}

For instance, the collection of all faces of a convex set defines a
convex subdivision of this set. The collection of all
exposed faces of a convex set is a convex decomposition, but it is not
necessarily a convex subdivision. 

In this text, we will be mainly concerned with the polyhedral
case. Since we will only deal with polyhedra which are convex,
we call them polyhedra for short.

\begin{defn} \label{def:30} A \emph{polyhedron} of
  $N_{\R}$\index{polyhedron} is a convex set defined as the
  intersection of a finite number of closed halfspaces.  It is called
  \emph{strongly convex}
  \index{polyhedron!strongly convex}%
  if it does
  not contain any line. A \emph{polyhedral
    cone}\index{polyhedral cone} is a polyhedron $\sigma
  $ such that $\lambda \sigma =\sigma $ for all $\lambda > 0$.  
\nomenclature[g1801]{$\sigma $}{cone}%
A \emph{polytope}\index{polytope} is a bounded polyhedron. 
\end{defn}

For a polyhedron,  
there is no difference between faces and exposed faces.

By the Minkowski-Weyl theorem, polyhedra can be explicitly described
in two dual ways, either by 
the
\emph{H-representation},
\index{H-representation!of a polyhedron}%
as an intersection of half-spaces, or by the \emph{V-representation},
\index{V-representation!of a polyhedron}%
as the Minkowski sum of a cone and a 
polytope \cite[Theorem
19.1]{Roc70}. An
H-representation of a polyhedron $\Lambda$ in $N_{\R}$ is a finite set of affine 
equations~$\{(a_{j},\alpha _{j})\}_{1\le j\le k}\subset M_{\R}\times
\R$ so that
\begin{equation}
  \label{eq:10}
  \Lambda=\bigcap_{1\le j\le k} \{u\in N_{\R} \mid \langle a_j,u\rangle
  +\alpha_{j}\ge 0 \}.
\end{equation}
With this representation, the recession cone can be written as
\begin{displaymath}
  \rec(\Lambda )=\bigcap_{1\le j\le k} \{u\in N_{\R} \mid \langle a_j,u\rangle
  \ge 0 \}.
\end{displaymath}
A V-representation of a polyhedron $\Lambda'$ in $N_{\R}$ consists in a set of vectors $\{b_{j}\}_{1\le
  j\le k}$ in the tangent space $T_{0}N_{\R}(\simeq N_{\R})$ and a
non-empty set of points $\{b_{j}\}_{k+1\le j\le l}\subset N_{\R}$ such
that
\begin{equation}
  \label{eq:35}
  \Lambda'=\Cone(b_{1},\dots,b_{k})+\Conv(b_{k+1},\dots,b_{l})
\end{equation}
where
\begin{displaymath}
  \Cone(b_{1},\dots,b_{k}):= \bigg\{\sum_{j=1}^{k}\lambda_{j}b_{j}\bigg| \  
  \lambda _{j}\ge 0\bigg\}
\end{displaymath}
\nomenclature[ac0one]{$\Cone(b_{1},\dots,b_{l})$}{cone generated by a
  set of vectors}%
is the cone generated by the given vectors (with the convention that
$\Cone(\emptyset)=\{0\}$) and
\begin{displaymath}
  \Conv(b_{k+1},\dots,b_{l}):=\bigg\{\sum_{j=k+1}^{l}\lambda
  _{j}b_{j}\bigg| \  
  \lambda _{j}\ge 0,\ \sum_{j=k+1}^{l}\lambda _{j}=1\bigg\}
\end{displaymath}
\nomenclature[ac0f]{$\Conv(b_{1},\dots,b_{l})$}{convex hull of a set of points}%
is the convex hull of the given set of points.  With this second
representation, the recession cone can be obtained as 
\begin{displaymath}
\rec(\Lambda')=\Cone(b_{1},\dots,b_{k}).  
\end{displaymath}

\begin{defn} \label{def:36}
  A \emph{polyhedral complex}\index{polyhedral complex} 
in $N_{\R}$ is a finite convex
  subdivision whose elements are polyhedra. A polyhedral
  complex is called 
  \emph{strongly convex}
\index{polyhedral complex!strongly convex}%
if all of its polyhedra are strongly
  convex. It is called \emph{conic}
\index{polyhedral complex!conic}%
if all of its
  elements are cones. A strongly convex conic polyhedral complex  is
  called a \emph{fan}\index{fan}. If $\Pi $ is a polyhedral complex,
  we will denote by $\Pi ^{i}$ the subset of $i$-dimensional polyhedra
  of~$\Psi $. In particular, if $\Sigma $ is a fan, $\Sigma ^{i}$ is
  its subset of $i$-dimensional cones.
  \nomenclature[g161]{$\Pi$}{polyhedral complex}%
  \nomenclature[g1811]{$\Sigma$}{fan}%
  \nomenclature[g162]{$\Pi^{i}$}{set of $i$-dimensional polyhedra
    of a complex}%
  \nomenclature[g1813]{$\Sigma ^{i}$}{set of $i$-dimensional cones
    of a fan}%
\end{defn}

There are two natural processes for linearizing a polyhedral
complex.   
\begin{defn}\label{def:69}
The \emph{recession} of $\Pi$
\index{polyhedral complex!recession of} is defined as the collection of
polyhedral cones of $N_{\R}$ given by
\begin{displaymath}
  \rec(\Pi )=\{\rec(\Lambda )\mid \Lambda\in \Pi \}. 
\end{displaymath}
\nomenclature[arec]{$\rec(\Pi )$}{recession of a polyhedral complex}%
The \emph{cone}
of $\Pi$\index{cone!of a polyhedral complex} is defined as the
collection of cones
in  $N_{\R}\times\R$ given by  
\begin{displaymath}
 \cc (\Pi)= \big\{\cc(\Lambda)\mid \Lambda \in
\Pi\big\}\cup \big\{\sigma\times\{0\}\mid \sigma\in \rec(\Pi)\big\}.
\end{displaymath}
\nomenclature[ac00cp]{$\cc(\Pi )$}{cone of a polyhedral complex}%
\end{defn}

It is natural to ask  whether the recession or the cone of
a given polyhedral complex is a complex too. The 
following example shows that this is not always the case.

\begin{exmpl}
  \label{exm:17}
  Let $\Pi$ be the polyhedral complex in $\R^{3}$
  containing the faces of the polyhedra 
  \begin{displaymath}
    \Lambda_{1}=\{(x_{1},x_{2},0) |
    \,  x_{1},x_{2}\ge0\} , \quad \Lambda_{2}=\{( x_{1},x_{2},1) | \, x_{1}+x_{2},
    x_{1}-x_{2}\ge 0\}. 
  \end{displaymath}
  Then $\rec(\Lambda_{1})$ and $\rec(\Lambda_{2})$
  are two  cones in $\R^{2}\times \{0\}$
  whose intersection is the cone 
  $\{(x_{1},x_{2},0) |
  x_{2},x_{1}-x_{2}\ge0\} $. This cone is neither a face of  $\rec(\Lambda_{1})$
  nor of $\rec(\Lambda_{2})$. Hence $\rec(\Pi)$ is not a complex and,
  consequently, neither is $\cc(\Pi)$.
  In Figure~\ref{fig:example} we see the polyhedron $\Lambda _{1}$
  in light grey, the polyhedron $\Lambda _{2}$ in darker grey and
  $\rec(\Lambda _{2})$ as dashed lines. 
\end{exmpl}

\begin{figure}[!h]
  \centering
  \input{ejemplo.pspdftex}  
  \caption{\ }  
  \label{fig:example}
\end{figure}

Therefore, to assure that $\rec(\Pi )$ or $\cc(\Pi )$ are complexes,
we need to impose some condition on $\Pi$. This question has been
addressed in \cite{BurgosSombra:rc}.  Because of our applications, we
are mostly interested in the case when $\Pi$ is complete. It turns out
that this assumption is enough to avoid the problem raised in Example
\ref{exm:17}.
 
\begin{prop}\label{prop:7}
Let $\Pi $ be a complete polyhedral complex in $N_{\R}$. Then  
$\rec(\Pi )$ and~$\cc( \Pi)$ are complete conic polyhedral complexes
in $N_{\R}$ and $N_{\R}\times \R_{\ge0}$, respectively.
If, in addition, $\Pi$ is 
  strongly convex, then both $\rec(\Pi )$ and $\cc(\Pi)$ are fans.
\end{prop}
\begin{proof}
  This is a particular case of \cite[Theorem~3.4]{BurgosSombra:rc}. 
\end{proof}

\begin{defn}\label{def:72}
  Let $\Pi _{1}$ and $\Pi _{2}$ be two polyhedral complexes in $N_{\R}$. The
  \emph{complex of intersections} of $\Pi _{1}$ and $\Pi _{2}$
  \index{polyhedral complex!of intersections}%
  is defined as the collection of polyhedra
  \begin{displaymath}
    \Pi _{1}\cdot \Pi _{2}=\{\Lambda_{1}\cap \Lambda_{2}|
    \Lambda_{1}\in \Pi _{1}, \Lambda_{2}\in \Pi _{2}\}.
  \end{displaymath} 
\nomenclature[sdot]{$\Pi _{1}\cdot \Pi _{2}$}{complex of intersections}%
\end{defn}

\begin{lem} \label{lemm:15}
  The collection $\Pi _{1}\cdot \Pi _{2}$ is a polyhedral
  complex. If $\Pi _{1}$ and $\Pi_{2}$ are complete, then
  \begin{displaymath}
    \rec(\Pi _{1}\cdot \Pi _{2})=\rec(\Pi_{1})\cdot \rec(\Pi_{2}).
  \end{displaymath}
\end{lem}
\begin{proof}
  Using the H-representation of polyhedra, one verifies that,
  if $\Lambda _{1}$ and $\Lambda _{2}$ are polyhedra with non-empty
  intersection, then any face of $\Lambda_{1} \cap \Lambda _{2}$ is
  the intersection of a face of $\Lambda _{1}$ with a face of $\Lambda
  _{2}$. This implies that $\Pi_{1}\cdot \Pi _{2}$ is a
  polyhedral complex.

  Now suppose that $\Pi_{1}$ and $\Pi _{2}$ are complete. Let $\sigma
  \in \rec(\Pi_{1}\cdot \Pi_{2})$. This means that $\sigma
  =\rec(\Lambda )$ and $\Lambda =\Lambda _{1}\cap \Lambda _{2}$ with 
  $\Lambda _{i}\in \Pi _{i}$. It is easy to verify that $\Lambda \not
  = \emptyset$ implies  
  $\rec(\Lambda )=\rec(\Lambda _{1})\cap \rec(\Lambda
  _{2})$. Therefore $\sigma \in \rec(\Pi_{1})\cdot
  \rec(\Pi_{2})$. This shows 
  \begin{displaymath}
    \rec(\Pi _{1}\cdot \Pi _{2})\subset \rec(\Pi_{1})\cdot \rec(\Pi_{2}).
  \end{displaymath}
  Since both complexes are complete, they agree. 
\end{proof}

We consider now an integral structure in $N_{\R}$. 
\nomenclature[aNl01]{$N$}{lattice}%
\nomenclature[aMl01]{$M$}{lattice dual to $N$}%
Let $N \simeq \Z^n$ be a lattice of rank~$n$ such that
$N_\R=N\otimes \R$. Set 
$M=N^\vee=\Hom(N,\Z)$ for its dual lattice so
$M_\R=M\otimes\R$. We also set $N_{\Q}=N\otimes \Q$ and
$M_{\Q}=M\otimes \Q$.

\begin{defn} \label{def:37}
  Let $\Lambda $  be a polyhedron in $N_{\R}$. We say that
  $\Lambda $ is a \emph{lattice polyhedron}
  \index{polyhedron!lattice}%
  if it admits a V-representation as (\ref{eq:35}) with 
  integral vectors and points, that is, with  $b_{j}\in N$ for
  $j=1\dots,l$. We say that it is \emph{rational} 
  \index{polyhedron!rational}%
  if it
  admits a V-representation with $b_{j}\in N_{\Q}$ for $j=1,\dots, l$.
\end{defn}
Observe that any rational polyhedron admits an H-representation as
(\ref{eq:10}) with $a_{j}\in M$ and $\alpha _{j}\in \Z$, for
$j=1,\dots.k$. 
\begin{defn} \label{def:42}
  Let $\Pi $ be a strongly convex polyhedral complex in
  $N_{\R}$. We say that $\Pi$ is \emph{lattice}
\index{polyhedral complex!lattice}%
(respectively \emph{rational})
\index{polyhedral complex!rational}%
  if all of its elements are lattice (respectively rational)
  polyhedra. For short, a {strongly convex rational polyhedral
  complex}
\index{polyhedral complex!SCR (strongly convex rational)}%
is called an \emph{SCR polyhedral complex}.
A conic SCR polyhedral complex is called a 
\emph{rational fan}.
\index{fan!rational}%
\end{defn}

\begin{rem}\label{rem:8}
  The statement of Proposition \ref{prop:7} is compatible with
  rational structures. Namely, if $\Pi $ is rational, the same is true
  for $\rec(\Pi )$ and $\cc(\Pi )$.
\end{rem}

\begin{cor} \label{cor:10} The correspondence $\Pi \mapsto \cc(\Pi )$
  is a bijection between the set of complete polyhedral complexes
  in $N_{\R}$ and the set of 
  complete conical polyhedral complexes in 
  $N_{\R}\times 
  \R_{\ge 0}$. Its
  inverse is the correspondence that, to each  conic polyhedral
  complex  $\Sigma$ in
$N_{\R}\times\R_{\ge 0}$ corresponds 
the complex 
in $N_{\R}$ obtained by intersecting ~$\Sigma$ with the hyperplane $N_{\R}\times\{1\}$.  
These bijections preserve rationality and strong convexity.
\end{cor}

\begin{proof}
  This is \cite[Corollary~3.12]{BurgosSombra:rc}.
\end{proof}

\section{The Legendre-Fenchel dual of a concave function}
\label{sec:basic-prop-legendre}

Let $N_{\R}$ and $M_{\R}$ be as in the previous section.
Set $\un \R=\R\cup \{-\infty\}$ 
\nomenclature[aRz]{$\un \R$}{real line with $-\infty$ added}%
with the natural order and
arithmetic operations. 
A {function} $f\colon N_{\R}\to \un\R$
is \emph{concave}
\index{concave function}%
if 
$$
f(t u_1+(1-t)u_2) \ge tf(u_1)+(1-t) f(u_2)$$ 
for all
$u_1,u_2\in N_{\R}$, $0<t<1$ and  $f$ is not identically
$-\infty$. Observe that a function $f$ is concave in our sense if and only if $-f$ is a
proper convex function in the sense of \cite{Roc70}.
The \emph{effective domain} $\Dom(f)$
\index{effective domain!of a concave function}%
\nomenclature[ad0om1]{$\Dom(f)$}{effective domain of a concave function}%
of such a function is the
subset of points of $N_{\R}$ where $f$ takes finite values. 
It is a convex set. A concave function $f\colon
N_{\R}\to \un\R$ defines a concave function with finite values
$f\colon \Dom(f)\to 
\R$. Conversely, if
$f\colon C\to \R$ is a
concave function  defined on some convex set $C$, we can extend it to
the whole of $N_{\R}$ by declaring
that its value at any point of $N_{\R}\setminus C$ is $-\infty$. 
We will move freely from the point of view of
concave functions on the whole of $N_{\R}$ with possibly infinite values  
to the point of view of real-valued concave functions on
arbitrary convex sets.

A concave function is
\emph{closed}
\index{concave function!closed}%
if it is upper semicontinuous.  
This includes the case of 
continuous concave functions defined on closed convex sets. Given an arbitrary
concave function, there exists a unique minimal closed concave
function above $f$. This function is called the \emph{closure}
\index{concave function!closure of}%
of $f$ and is denoted by $\cl(f)$. 
\nomenclature[ac1]{$\cl(f)$}{closure of a concave function}%

Let $f$ be a concave function on $N_{\R}$. 
The \emph{Legendre-Fenchel dual} of $f$
\index{Legendre-Fenchel dual of a concave function}%
is the function
$$
f^{\vee}\colon M_{\R}\longrightarrow \un\R , \quad x\longmapsto
\inf_{u\in N_{\R}} (\langle x,u\rangle-f(u)).
$$
\nomenclature[sdual]{$f^{\vee}$}{Legendre-Fenchel dual of a concave function}%
It is a closed concave function. The Legendre-Fenchel duality is an
involution between such functions: if $f$ is closed, 
then $f^{\vee\vee}  =f$~\cite[Corollary~12.2.1]{Roc70}. 
In fact, for any concave function $f$ we have $f^{\vee\vee}  =\cl(f)$.

The effective domain of $f^{\vee}$ is called the \emph{stability set}
of $f$.
\index{stability set of a concave function}%
It can be described as
\begin{displaymath}
  \Stab(f)=\Dom(f^{\vee})=\{x\in M_{\R} \mid \langle x,u\rangle -f(u) \text{ is
    bounded below}\}. 
\end{displaymath}
\nomenclature[asatab]{$\Stab(f)$}{stability set of a concave function}%

\begin{exmpl}\label{exm:7}
The \emph{indicator function} of a convex set $C\subset
N_\R$
\index{convex set!indicator function of}%
is the
concave function $\iota _{C}$ 
\nomenclature[g09]{$\iota _{C}$}{indicator function of a convex set}%
defined as $\iota  _{C}(u)=0$ for $u\in C$ and $\iota _{C}(u)=-\infty$ for $u\not
  \in C$. Observe that $\iota  _{C}$ is the logarithm of the
  {characteristic function} of $C$.
This function is closed if and only if $C$ is a closed set. 

The \emph{support function} of a convex set $C$
\index{support function!of a convex set}%
is the function  
\begin{displaymath}
\Psi_C\colon M_\R\longrightarrow \R , \quad x\longmapsto \inf_{u\in
  C} \langle x,u\rangle.
\end{displaymath}
\nomenclature[g2301]{$\Psi_C$}{support function of a convex set}%
It is a closed concave function. 
A function $f\colon M_{\R}\to \R$ is called
\emph{conical}
\index{conical function}%
if $f(\lambda x)=\lambda f(x)$ for 
all $\lambda \ge 0$. The support function $\Psi _{C}$ is conical. The
converse is also true: all conical closed concave functions are
of the form $\Psi_C$ for a closed convex set $C$. 

We have $\iota_C^\vee= \Psi_C$ and $\Psi_C^\vee= \cl(\iota_C)=\iota_{\ov C}$.
Thus, the Legendre-Fenchel duality defines a bijective correspondence between
indicator  functions of closed convex subsets of $N_{\R}$ and
closed concave conical functions on $M_{\R}$.
\end{exmpl}

Next result shows that the Legendre-Fenchel duality is monotonous.

\begin{prop} \label{prop:93}
  Let $f$ and $g$ be concave functions 
    such that $g(u)\le f(u)$ for all $u\in N_\R$. Then $\Dom(g)\subset
    \Dom(f)$, 
    $\Stab(g)\supset \Stab(f)$ and $g^{\vee}(x)\ge f^{\vee}(x)$ for
    all~$x\in M_{\R}$. 
\end{prop}
\begin{proof}
  It follows directly from the definitions.
\end{proof}

The Legendre-Fenchel duality is continuous with
  respect to uniform 
convergence.

\begin{prop} \label{prop:3}
Let $(f_{i})_{i\ge1}$ be a sequence of concave functions  which
    converges uniformly to a  function $f$. Then $f$ is a concave
    function and the sequence $(f_{i}^{\vee})_{i\ge 1}$ converges  
    uniformly to $f^{\vee}$. In particular, there is some $i_0\ge1$ such that 
    $\Dom(f_{i})=\Dom(f)$ and $\Stab(f_i)=\Stab(f)$ for all $i\ge i_0$.
\end{prop}

\begin{proof}
  Clearly $f$ is concave.
  Let $\varepsilon >0$. Then there is an $i_{0}$ such that, for all
  $i\ge i_{0}$, $f-\varepsilon \le f_{i}\le f+\varepsilon $. By
  Proposition \ref{prop:93} this implies $\Dom(f_{i})=\Dom(f\pm\varepsilon )=\Dom(f)$ and
  $\Stab(f_i)=\Stab(f\pm\varepsilon )=\Stab(f)$ and that
  \begin{displaymath}
    f^{\vee}-\varepsilon =(f+\varepsilon )^{\vee}\le
    f_{i}^{\vee}\le(f-\varepsilon )^{\vee}=f^{\vee}+\varepsilon, 
  \end{displaymath}
  which implies the uniform convergence of $f_{i}^{\vee}$ to $f^{\vee}$.
\end{proof}

The classical Legendre duality of strictly concave differentiable
functions can be described in terms of the gradient map $\nabla f$, 
\nomenclature[snabla]{$\nabla f$}{gradient of a differentiable function}%
called in this setting the {``Legendre transform''}.
\index{Legendre transform}%
We will next show that the Legendre transform can be extended to the general concave case as
a correspondence between convex decompositions.

Let $f$ be a concave function on $N_{\R}$. The
\emph{sup-differential} of $f$ at a point
$u\in N_{\R}$
\index{sup-differential of a concave function}%
is defined as the set
\begin{displaymath}
  \partial f(u)=
    \{ x\in M_{\R} \mid \langle x,v-u\rangle \ge f(v)
    -f(u) \text{ for all } v\in N_{\R}\}
\end{displaymath}
if $u\in \Dom(f)$, and the empty set if $u\not\in \Dom(f)$.
\nomenclature[spartial]{$\partial f$}{sup-differential of a concave function}%
For an arbitrary concave function, 
the sup-differential is a generalization of the gradient.
In general, $\partial f(u)$ may contain more than one point, so the
sup-differential has to be 
regarded as a  multi-valued function. 

We say that $f$ is \emph{sup-differentiable} at a point 
$u\in N_{\R}$
\index{concave function!sup-differentiable}%
if $\partial f(u) \ne \emptyset$. 
The \emph{effective domain} of $\partial f$,
\index{effective domain!of a sup-differential}%
denoted $\Dom(\partial f)$,
is the set of points where $f$ is sup-differentiable.
\nomenclature[ad0om2]{$\Dom(\partial f)$}{effective domain of the sup-differential}%
For a subset $E \subset N_{\R}$ we define
\begin{displaymath}
  \partial f(E)= \bigcup_{u\in E} \partial f(u).
\end{displaymath}
In particular, the \emph{image} of $\partial f$
\index{sup-differential of a concave function!image of}%
is defined as $\Im(\partial f)=\partial f(N_{\R})$.

The sup-differential $\partial f(u)$ is a closed
convex set for all $u\in \Dom(\partial f)$. 
It is bounded if and only if  $u\in
\ri(\Dom(f))$. Hence, in the particular case when $\Dom(f)=N_{\R}$, we have
that $\partial f(u)$ is a bounded closed convex subset of $M_{\R}$ for
all $u\in N_{\R}$.
The effective domain of the sup-differential is not necessarily convex but it
differs very little from being convex since, by \cite[Theorem
23.4]{Roc70}, it satisfies
\begin{displaymath}
  \ri(\Dom(f)) \subset \Dom(\partial f) \subset \Dom(f).
\end{displaymath}

Let $f$ be a closed concave function and consider the pairing 
\nomenclature[aPf]{$P_{f}$}{pairing associated to a concave function}%
\begin{equation}
  \label{eq:19}
  P_{f}\colon M_{\R}\times N_{\R} \longrightarrow \un \R 
  , \quad (u,x)\longmapsto  f(u)+f^{\vee}(x)-\langle x,u\rangle.
\end{equation}
This pairing satisfies $P_{f}(u,x)\le 0$ for all $u,x$.

\begin{prop}
  \label{prop:36}
Let $f$ be a closed concave function on $N_{\R}$. 
For $u\in N_{\R}$
and $x\in M_{\R}$, the following conditions are  equivalent:
\begin{enumerate}
\item \label{item:48} $x\in \partial f(u)$;
\item \label{item:49} $u\in \partial f^{\vee}(x)$;
\item \label{item:50} $P_{f}(u,x)=0$.
\end{enumerate}
\end{prop}

\begin{proof}
  This is proved in \cite[Theorem~23.5]{Roc70}.
\end{proof}

If $f$ is closed, then $\Im(\partial f)= \Dom(\partial
f^{\vee})$ and so the 
image of the sup-differential is close to be a convex set, in the
sense that 
\begin{equation}\label{eq:634}
\ri(\Stab(f))\subset   \Im(\partial f)\subset\Stab(f).
\end{equation}

\begin{defn}\label{def:29} Let $f$ be a closed concave function on $N_{\R}$. 
We denote by  $\Pi (f)$  the collection of all sets of the form
\nomenclature[g163]{$\Pi (f)$}{convex decomposition associated to a concave function}%
\begin{displaymath}
  C _{x}:= \partial f^{\vee}(x)
\end{displaymath}
\nomenclature[aC9]{$C _{x}$}{piece of a convex decomposition}%
for some $x\in \Stab(f)$.
\end{defn}

\begin{lem} \label{lemm:19} Let $f$ be a closed concave function on
  $N_{\R}$. 
  Let $x\in \Stab (f)$. Then
  \begin{math}
    C_{x}=\{u\in N_{\R} \mid
  P_{f}(u,x)=0\}.
  \end{math}
In other words, the set $C_{x}$ is characterized by the condition
\begin{equation}\label{eq:32}
  f(u)=\langle x,u\rangle -f^{\vee}(x) \text{ for } u\in
  C_{x}
\quad \text{and} \quad f(u)< \langle x,u\rangle -f^{\vee}(x)
\text{ for } u\not\in C_{x}. 
\end{equation}
Thus the restriction of $f$ to $C_{x}$ is an affine
function with linear part given by $x$, and $C_{x}$ is the
maximal subset where this property holds.
\end{lem}
\begin{proof}
  The first statement follows from the equivalence of \eqref{item:49}
  and \eqref{item:50} in Proposition~\ref{prop:36}. The second
  statement follows from the definition of $P_{f}$ and its non-positivity.  
\end{proof}

The \emph{hypograph}
\index{hypograph of a concave function}%
of a concave function $f$ is defined as the set 
\begin{displaymath}
\hypo(f)=\{ (u,\lambda)\mid u\in N_{\R}, \lambda \le f(u)\}
\subset N_{\R}\times \R.   
\end{displaymath}
\nomenclature[ah]{$\hypo(f)$}{hypograph of a concave function}%
A face of the hypograph is called \emph{non-vertical}
\index{hypograph of a concave function!non-vertical face of}%
if it projects injectively in $N_{\R}$. 

\begin{prop}
  \label{prop:33}
Let $f$ be a closed concave function on $N_{\R}$. For a subset $C \subset
N_{\R}$, the following conditions are equivalent:
\begin{enumerate}
\item \label{item:41} $C\in \Pi(f)$;
\item \label{item:61} $C= \{ u\in N_{\R} \mid x\in \partial f(u) \}$
  for a $x\in 
  M_{\R}$; 
\item \label{item:47} there exist $x_{C}\in
M_{\R}$ and $\lambda_{C}\in \R$ such that 
the set $\{ (u, \langle x_{C},u\rangle -\lambda_{C}) \mid u\in C\}$ is
an exposed face
of the hypograph of $f$. 
\end{enumerate}
In particular, the
correspondence 
\begin{displaymath}
C_{x}\mapsto \{ (u, \langle x,u\rangle -f^{\vee}(x))
\mid u\in C_{x}\}   
\end{displaymath}
is a bijection between $\Pi(f)$ and the set of
non-vertical exposed faces of $\hypo(f)$. 
\end{prop}

\begin{proof}
The equivalence between the conditions~\eqref{item:41} and
\eqref{item:61} comes directly from
Proposition~\ref{prop:36}. The equivalence with the
condition~\eqref{item:47} 
follows from~\eqref{eq:32}.
\end{proof}

\begin{prop}
  \label{prop:35}
Let $f$ be a closed concave function. Then $\Pi(f)$ is a convex decomposition
of $\Dom(\partial f)$.
\end{prop}

\begin{proof} 
The collection of non-vertical exposed faces of $\hypo(f)$ forms a
convex decomposition of a subset of $N_{\R}\times \R$. Using Proposition
\ref{prop:33} the projection to $N_{\R}$ of this decomposition agrees
with $\Pi(f)$ and so, it is a convex decomposition of $|\Pi(f)|=
\Dom(\partial f)$.
\end{proof}

We need the following result in order to properly define the
Legendre-Fenchel correspondence for an arbitrary concave function as a
bijective correspondence between convex decompositions.

\begin{lem}
  \label{lemm:14}
Let $f$ be a closed concave function and $C\in \Pi(f)$. Then for any $u_{0}\in \ri(C)$,
\begin{displaymath}
\bigcap_{u\in C}\partial f(u)= \partial f(u_{0}).  
\end{displaymath}
\end{lem}

\begin{proof}
Fix  $x_{0}\in \Dom(\partial f^{\vee})$ such that $C=C_{x_{0}}$ and  $u_{0}\in
\ri(C)$. 
Let $x\in  \partial f(u_{0})$. Then 
\begin{equation}
  \label{eq:132}
\langle x,v-u_{0}\rangle \ge f(v)-f(u_{0}) \quad \text{for all } v\in N_{\R}.  
\end{equation}
Let $u\in C$. By~\eqref{eq:32}, we have $f(u)-f(u_{0})=\langle
x_{0},u-u_{0}\rangle$ and so the above inequality implies 
$
\langle x,u-u_{0}\rangle \ge \langle x_{0},u-u_{0}\rangle .
$ 
The fact $u_{0}\in \ri(C)$ implies $u_{0}+\lambda (u_{0}-u) \in C$ 
for some small $\lambda >0$. Applying the same argument to this
element we obtain the reverse inequality $
\langle x,u-u_{0}\rangle \le \langle x_{0},u-u_{0}\rangle$ and so
\begin{equation}
  \label{eq:34}
\langle x-x_{0},u-u_{0}\rangle =0.
\end{equation}
In particular, $f(u)-f(u_{0})= \langle
x_{0},u-u_{0}\rangle=\langle x,u-u_{0}\rangle$
and from~\eqref{eq:132} we obtain 
$$
\langle x,v-u\rangle =
\langle x,v-u_{0}\rangle +f(u_{0})-f(u) \ge 
f(v)-f(u) \quad \text{for all } v\in N_{\R}.  
$$
Hence $x\in \bigcap_{u\in C}\partial f(u)$ and so 
$\partial f(u_{0}) \subset \bigcap_{u\in C}\partial f(u)$, which
implies the stated equality. 
\end{proof}

\begin{defn}
  \label{def:26}
Let $f$ be a closed concave function.
The \emph{Legendre-Fenchel correspondence} of $f$
\index{Legendre-Fenchel correspondence of a con\-cave function}%
is defined as 
$$
\cL f \colon \Pi(f) \longrightarrow \Pi(f^{\vee}), \quad  
C\longmapsto 
\bigcap_{u\in C}\partial f(u).
$$
\nomenclature[aLzz]{$\cL f$}{Legendre-Fenchel correspondence of a concave function}%
\end{defn}

By Lemma~\ref{lemm:14}, $\cL f(C)=  \partial f(u_{0})$ for
any $u_{0}\in \ri(C)$. Hence, 
$$
\cL f(C)\in \Pi (f ^{\vee}).
$$

\begin{defn}
  \label{def:9}
Let $E, E'$ be subsets of $N_{\R}$ and $M_{\R}$
respectively, and $\Pi , \Pi '$ 
convex decompositions of $E$ and $E'$, respectively. 
We say that $\Pi $ and $\Pi '$ are \emph{dual convex
  decompositions}
\index{dual!convex decompositions}%
if there 
exists a bijective map $\Pi  \to \Pi ', C \mapsto C ^{\ast} $
\nomenclature[sast]{$C^{\ast}$}{corresponding convex set in a dual decomposition}%
such that 
\begin{enumerate}
\item for all $C, D \in \Pi $ we have $C\subset D $ if and
  only if $C^{\ast}\supset D^{\ast}$;
\item \label{item:17} for all $C\in \Pi $ the sets $C$ and $C^{\ast}$
  are contained in orthogonal affine spaces of $N_{\R}$ and $M_{\R}$,
  respectively. 
\end{enumerate}
\end{defn}

\begin{thm}
  \label{thm:8}
Let $f$ be a closed concave function, then $\cL f$ is a duality between
$\Pi(f)$ and $\Pi(f^{\vee})$ with inverse $(\cL f)^{-1}=\cL f^{\vee}$.
\end{thm}

\begin{proof}
We will prove first that  $\cL f^{\vee}= (\cL f)^{-1}$. 
Fix $C\in \Pi(f)$ and set $C'=\cL f(C)$. 
Let $y_{0}\in M_{\R}$ such that $C=C_{y_{0}}$
and let  $u_{0}\in \ri(C)$.
Hence $u_{0}\in C_{y_{0}}=\partial f^{\vee}(y_{0})$ and so 
$y_{0}\in \partial f(u_{0})=C'$ by
Proposition~\ref{prop:36} and
Lemma~\ref{lemm:14}. Hence
$$
\cL f^{\vee}(\cL f(C)) = \cL f^{\vee}(C') =\bigcap_{x\in C'} \partial f^{\vee}(x) \subset \partial f^{\vee}(y_0) = C.
$$
On the other hand, let $x_{0}\in \ri(C')$. 
In particular, $x_{0}\in \partial f(u_0)$ and so $u_{0}\in \partial
f^{\vee}(x_{0})= \cL f^{\vee}(C')$ for all $u_{0}\in
C$. It implies 
\begin{displaymath}
C \subset \cL f^{\vee}(C') = \cL f^{\vee}(\cL f(C)).  
\end{displaymath}
Thus $\cL f^{\vee}(\cL f(C))=C$ and applying the same argument to
$f^{\vee}$ we conclude that $\cL f^{\vee}= (\cL f)^{-1}$ and that $\cL
f$ is bijective.  

Now we have to prove that $\cL$  is a duality between
$\Pi(f)$ and $\Pi(f^{\vee})$. 
Let $C,D\in \Pi(f)$ such that $C\subset D$. Clearly, 
$\cL f(C) \supset \cL f(D)$.
The reciprocal follows by applying the same argument to $f^{\vee}$. 
The fact that $C$ and $\cL f(C)$ lie in orthogonal affine spaces
has already been shown during the proof of Lemma~\ref{lemm:14} above,
see~\eqref{eq:34}.  
\end{proof}

\begin{defn}\label{def:19}
Let $f$ be a closed concave function. The pair of convex decompositions
$(\Pi (f),\Pi(f^{\vee}))$ 
will be called the \emph{dual pair of convex decompositions} induced by $f$.
\index{dual!convex decompositions!induced by a concave function}%
\end{defn}

In particular, for $C\in \Pi(f)$ put $C^{*}:=\cL f(C)$. 
For any $u_{0}\in \ri(C)$ and $x_{0}\in \ri(C^{*})$, we have 
\begin{displaymath}
  C=\{u\in N_{\R} \mid
P_{f}(u,x_{0})=0\} \quad \text{and} \quad C^{*}=\{x\in M_{\R} \mid
P_{f}(u_{0},x)=0\}.
\end{displaymath}

Following~\eqref{eq:32}, the restrictions $f|_{C}$ and  $f^{\vee}|_{C^{*}}$ are 
affine functions. Observe that we can recover the Legendre-Fenchel dual
from the Legendre-Fenchel correspondence by writing, for $x\in C^{\ast}$
and any $u\in C$,
\begin{displaymath}
f^{\vee}(x)=\langle x, u\rangle -f(u).
\end{displaymath}

\begin{exmpl}
  \label{exm:15}
Let $\Vert \cdot\Vert_{2}$ denote the Euclidean norm on $\R^{2}$ and 
$B_{1}$ the unit ball. Consider the
concave function $f\colon B_{1}\to \R$ defined as $f(u)=-\Vert u\Vert_{2}$.  
Then $\Stab(f)= \R^{2}$ and the Legendre-Fenchel dual is the function defined by 
$f^{\vee}(x)=0$ if $ \Vert x\Vert_{2}\le 1$ and $f^{\vee}(x)=1-\Vert
x\Vert_{2}$ otherwise. 
The decompositions $\Pi(f)$ and $\Pi(f^{\vee})$ consist of a collection
of pieces of three different types and the Legendre-Fenchel
correspondence 
$\cL f\colon \Pi(f)\to \Pi(f^{\vee})$ is given, for $z\in S^{1}$, by
$$
\cL f(\{0\})=  B_{1} , \quad  \cL f([0,1]\cdot z)=  \{z\}
, \quad  
\cL f(\{z\})=\R_{\ge 1}\cdot z.
$$
\end{exmpl}
In the above example both decompositions are in fact subdivisions. But
this is not always the case, as shown by the next example.

\begin{exmpl} \label{exm:18}
  Let $f\colon [0,1]\to \R$ the function defined by
  \begin{displaymath}
    f(u)=
    \begin{cases}
      -u \log(u) & \text{ if }0\le u\le \e^{-1},\\
      \e^{-1} & \text{ if }\e^{-1}\le u\le 1-\e^{-1},\\
      -(1-u)\log(1-u) & \text{ if }1-\e^{-1}\le u\le 1.
    \end{cases}
  \end{displaymath}
  Then $\Stab(f)=\R$ and
  the Legendre-Fenchel dual is the function $f^{\vee}(x)=x-\e^{x-1}$ for $x\le
  0$ and  $f^{\vee}(x)=-\e^{-x-1}$ for $x\ge 0$. Then $\Dom(\partial
  f)=(0,1)$ and $\Dom(\partial f^{\vee})=\R$. Moreover,
  \begin{displaymath}
    \Pi (f)=(0,\e^{-1})\cup \{[\e^{-1},1-\e^{-1}]\}\cup
    (1-\e^{-1},1),\quad
    \Pi (f^{\vee})=\R.
  \end{displaymath}
  The Legendre-Fenchel correspondence sends bijectively
  $(0,\e^{-1})$ to $\R_{>0}$ and $(1-\e^{-1},1)$ to $\R_{<0}$, and sends the
  element $[\e^{-1},1-\e^{-1}]$ to the point $\{0\}$. In this example,
  $\Pi (f)$ is not a subdivision while $\Pi (f^{\vee})$ is.
\end{exmpl}

\section{Operations on concave functions and duality}
\label{sec:dual-operations}

In this section we consider the basic operations on concave functions
and their interplay with the Legendre-Fenchel duality. 

Let $f_{1}$ and $f_{2}$ be two concave functions on $N_{\R}$ such that their
stability sets are not disjoint. 
Their \emph{sup-convolution}
\index{sup-convolution of concave functions}%
is the function  
$$
f_{1}\boxplus f_{2}\colon N_{\R}\longrightarrow \un\R , \quad v\longmapsto
\sup_{u_{1}+u_{2}=v} (f_{1}(u_{1})+f_{2}(u_{2})).
$$
\nomenclature[ssupconv]{$f_{1}\boxplus f_{2}$}{sup-convolution of concave functions}%
This is a concave function whose effective domain is the  Minkowski
sum $\Dom(f_{1})+\Dom(f_{2})$. 
This operation is associative and
commutative whenever the terms are defined. 

The operations of pointwise addition and sup-convolution
are dual to each other.
When working with general concave functions, there are some technical issues 
in this duality that will disappear when considering uniform limits of
piecewise affine concave functions.

\begin{prop}\label{prop:10} Let  $f_{1},\dots,f_l$ be concave 
  functions on $N_{\R}$.
  \begin{enumerate}
  \item \label{item:31} If $\Stab(f_{1})\cap\dots \cap\Stab(f_{l})\not
    =\emptyset $,
    then
    \begin{equation*}
      (f_{1}\boxplus\dots \boxplus f_{l})^{\vee}=f_{1}^{\vee}+\dots
      +f_{l}^{\vee}.
    \end{equation*}
    In particular, $\Stab(f_{1}\boxplus\dots \boxplus
    f_{l})=\Stab(f_{1})\cap\dots\cap \Stab(f_{l}).$
  \item \label{item:32}If $\Dom(f_{1})\cap\dots \cap\Dom(f_{l})\not
    =\emptyset $, then
    \begin{equation*}
      (\cl(f_{1})+\dots +\cl(f_{l}))^{\vee}=\cl(f_{1}^{\vee}\boxplus\dots \boxplus
      f_{l}^{\vee}).
    \end{equation*}
  \item \label{item:33} If $\ri (\Dom(f_{1}))\cap\dots \cap\ri (\Dom(f_{l}))\not =\emptyset $, then
    \begin{equation*}
      (f_{1}+\dots +f_{l})^{\vee}=f_{1}^{\vee}\boxplus\dots \boxplus
      f_{l}^{\vee}.
    \end{equation*}
        In particular, $\Stab(f_{1}+\dots +
    f_{l})=\Stab(f_{1})+\dots +\Stab(f_{l}).$
  \end{enumerate}
\end{prop}
\begin{proof}
  This is proved in \cite[Theorem 16.4]{Roc70}.
\end{proof}

\begin{rem}\label{rem:6}
  When some of the $f_{i}$, say
  $f_{1},\dots ,f_{k}$, are piecewise affine, the statement
  \eqref{item:33} of the previous proposition
  holds under the weaker hypothesis~\cite[Theorem~20.1]{Roc70}
  \begin{displaymath}
    \Dom(f_{1})\cap\dots\cap\Dom(f_{k})\cap\ri(\Dom(f_{k+1}))
    \cap\dots\cap\ri(\Dom(f_{l}))  
    \not = \emptyset.
  \end{displaymath}
\end{rem}

Let $f\colon N_{\R}\to \un \R$ be a function. For $\lambda >0$, the \emph{left}
and \emph{right scalar multiplication} of $f$ by $\lambda $  
\index{left scalar multiplication}%
\index{right scalar multiplication}%
are the functions defined, for
$u\in N_\R$,   
by
$(\lambda
f)(u)=\lambda f(u)$ and  $(f\lambda )(u)=\lambda f(u/\lambda
)$ respectively. 
\nomenclature[smult]{$\lambda f$}{left scalar multiplication}%
\nomenclature[smult]{$f \lambda$}{right scalar multiplication}%
For a point $u_{0}\in N_{\R}$, 
the \emph{translate} of $f$ by 
$u_{0}$ \index{translate of a function} is the 
function defined as 
$(\tau_{u_{0}}f)(u)=f(u-u_{0})$ for $u\in N_\R$.
\nomenclature[g19]{$\tau_{u_{0}}f$}{translate of a concave function}%
If $f$ is concave, then its left and right multiplication by a scalar and
its translation by a point are also concave functions. 

\begin{prop} \label{prop:8}
  Let $f$ be a concave function on $N_\R$, $c,\lambda \in \R$ with $\lambda >0$, $u_{0}\in
  N_{\R}$ and $x_{0}\in M_{\R}$. Then
  \begin{enumerate}
  \item \label{item:118} $\Dom(f+c)=\Dom(f)$,
$\Stab(f +c)=\Stab(f)$ and  $(f+c)^{\vee}=f^{\vee}-c$;
\item \label{item:57} $\Dom(\lambda f)=\Dom(f)$,
$\Stab(\lambda f)=\lambda \Stab(f)$ and  $(\lambda f)^{\vee}=f^{\vee}\lambda$;
\item \label{item:42} $\Dom(f\lambda)=\lambda\Dom(f)$, $\Stab
  (f\lambda )=\Stab (f)$ and $(f\lambda )^{\vee}=\lambda f^{\vee}$;
  \item \label{item:13}    
        $\Dom(\tau _{u_{0}}f)=\Dom(f)+u_{0}$,
    $\Stab(\tau_{u_{0}}f)=\Stab(f)$ and
    $(\tau_{u_{0}}f)^{\vee}=f^{\vee} +u_{0}$;
  \item \label{item:14}
    $\Dom(f+x_{0})=\Dom(f)$,  $\Stab(f+x_{0})=\Stab(f)+x_{0}$ and  
    $(f+x_{0})^{\vee}=\tau_{x_{0}}f^{\vee}$.
\end{enumerate}
\end{prop}

\begin{proof}
  This follows easily from the
  definitions.
\end{proof}
\nomenclature[aHn]{$H$}{linear map of vector spaces}%
We next consider direct and inverse images  of concave functions 
by affine maps.
Let $Q_\R$ be a another finite dimensional real vector space  and 
set $P_{\R}=Q_{\R}^{\vee}$ for its dual space. For a linear map
$H\colon Q_{\R}\to N_{\R}$ we denote by 
$H^{\vee}\colon M_{\R}\to P_\R$ the dual map. 
\nomenclature[sdual]{$H^{\vee}$}{dual of a linear map}%
We need the following lemma in order to properly define direct images.

\begin{lem}\label{lemm:2}
  Let $H\colon Q_{\R}\to N_{\R}$ be a linear map and $g$ a concave
  function on $Q_{\R}$.  If $\Stab(g)\cap \Im(H^{\vee})\not =
  \emptyset$ then, for all $u\in N_{\R}$,
  \begin{displaymath}
 \sup_{v\in H^{-1}(u)} g(v) <\infty.    
  \end{displaymath}
\end{lem}
\begin{proof}
  Let $x\in M_{\R}$ such that 
  $H^{\vee}(x)\in \Stab(g)$. By the definition of the stability set, $
    \sup_{v\in Q_{\R}}(g(v)-\langle v,H^{\vee}(x) \rangle) <
    \infty$.
Thus, for any $u\in N_{\R}$, 
  \begin{align*}
    \sup_{v\in Q_{\R}}(g(v)-\langle v, H^{\vee}(x) \rangle) &=
    \sup_{v\in Q_{\R}} (g(v)-\langle x,H (v) \rangle) \\
    &\ge \sup_{v\in H^{-1}(u)}(g(v)-\langle x,H (v) \rangle)
    =  \sup_{v\in H^{-1}(u)} g(v)-\langle x,u \rangle
  \end{align*}
and so $\sup_{v\in H^{-1}(u)} g(v)$ is bounded above, as stated. 
\end{proof}

\begin{defn}\label{def:11} Let $A\colon Q_\R\to N_\R$ be an affine map
  defined as $A=H+u_0$ for a linear map $H$
and a point $u_0\in N_{\R}$. 
\nomenclature[aA]{$A$}{affine map of vector spaces}%
Let  $f$ be a concave function on $N_{\R}$
such that $\Dom(f)\cap
  \Im(A)\not = \emptyset$ and $g$ a concave function on $Q_{\R}$  
such that $\Stab(g)\cap \Im(H^{\vee})\not = \emptyset$. Then  
the \emph{inverse image} of $f$ by $A$
\index{inverse image of a concave function by an affine map}%
is defined as 
\begin{displaymath}
  A^{\ast}f\colon Q_\R \longrightarrow \underline{\R} , \quad 
  v\longmapsto f\circ A(v) ,
\end{displaymath}
\nomenclature[sast]{$A^{\ast}f$}{inverse image of a concave function
  by an affine map}%
and the \emph{direct image} of $g$ by $A$
\index{direct image of a concave function by an affine map}%
is defined as 
  \begin{displaymath}
    A_{\ast}g\colon N_\R \longrightarrow \underline{\R} , \quad 
u\longmapsto \sup_{v\in A^{-1}(u)}g(v). 
  \end{displaymath}
\nomenclature[sast]{$A_{\ast}g$}{direct image of a concave function
  by an affine map}%
\end{defn}

It is easy to see that the inverse image  $A^{\ast }f$ is concave
with effective domain $\Dom(A^{\ast}f)=A ^{-1}(\Dom(f))$. 
Similarly, the direct image 
$A_{\ast}g $ is concave with effective domain 
$\Dom(A_{\ast}g)=A(\Dom(g))$, thanks to 
Lemma \ref{lemm:2}. 

The inverse image of a closed function is also closed. 
In contrast, the direct image
of a closed function is not necessarily closed: 
consider for instance the indicator function $\iota_C$ of the set $C=\{(x,y)\in \R^{2}\mid
xy\ge 1, x>0\}$, which is a closed concave function. Let $A\colon
\R^{2}\to \R$ be the first projection. Then $A_{\ast}\iota_C$ is the
indicator function of the subset $\R_{>0}$, which is not a closed
concave function. 

We now turn to the behaviour of the sup-differential with respect to the basic operations. 
A first important property is the additivity.

\begin{prop}
  \label{prop:37}
For each $i=1,\dots,l$, let $f_i$ be a concave function and
$\lambda_i>0$ a real number.
Then, for all $u\in N_{\R}$, 
\begin{enumerate}
\item \label{item:34} $\partial \left( \sum_{i}\lambda_{i}f_{i}\right)(u) 
\supset \sum_{i} \lambda_{i}\partial (f_{i})(u)$;
\item \label{item:116} if $\ri(\Dom(f_{1}))\cap \dots\cap \ri(\Dom(f_{l})) \ne \emptyset$, then 
  \begin{equation}
    \label{eq:232}
\partial \bigg( \sum_{i}\lambda_{i}f_{i}\bigg)(u) 
= \sum_{i} \lambda_{i}\partial (f_{i})(u).    
  \end{equation}
\end{enumerate}
\end{prop}

\begin{proof}
  This is~\cite[Theorem~23.8]{Roc70}.
\end{proof}

As in Remark \ref{rem:6}, if  $f_{1},\dots,f_{k}$ are
piecewise affine, then~\eqref{eq:232} holds under the weaker hypothesis 
$$
\Dom(f_{1})
  \cap\dots\cap\Dom(f_{k}) \cap\ri (\Dom(f_{k+1})) \cap \dots \cap \ri
  (\Dom(f_{l}))\ne \emptyset.
$$

The following result gives the behaviour of the sup-differential with respect to linear
maps

\begin{prop}
  \label{prop:38}
Let $H\colon Q_{\R}\to N_{\R}$ be a linear map, $u_{0}\in N_{\R}$ and
$A=H+u_{0}$ the associated affine map. Let $f$ be a concave function on
$N_{\R}$, then 
\begin{enumerate}
\item \label{item:30} $  \partial (A^{*}f)(v)\supset H^{\vee} \partial
  f(Av)$ for all 
  $v\in Q_{\R}$;
\item  \label{item:43} if either $ \ri (\Dom(f))\cap \Im(A)\ne
  \emptyset$ or  $f$ is
  piecewise affine  and 
$ \Dom(f)\cap \Im(A)\ne \emptyset$, then  for all  $v\in
Q_{\R}$ we have
$$  \partial (A^{*}f)(v)= H^{\vee} \partial f(Av).$$
\end{enumerate}
\end{prop}

\begin{proof}
  The linear case $u_{0}=0$ is \cite[Theorem~23.9]{Roc70}. The general
  case follows from the linear case and the commutativity of
  the sup-differential and the translation. 
\end{proof}

We summarize the behaviour of direct and inverse images of affine maps 
with respect to the Legendre-Fenchel duality.

\begin{prop} \label{prop:11}
   Let $A\colon Q_\R\to N_\R$ be an affine map defined as $A=H+u_0$ for a linear map $H$ 
and a point $u_0\in N_{\R}$. Let  $f$ be a concave function on $N_{\R}$
such that $\Dom(f)\cap
  \Im(A)\not = \emptyset$ and $g$ a concave function on $Q_{\R}$  
such that $\Stab(g)\cap \Im(H^{\vee})\not = \emptyset$. Then  
  \begin{enumerate}
  \item \label{item:21} $\Stab(A_{\ast}g)=(H^\vee)^{-1}(\Stab(g))$ and
 $$
(A _{\ast}g)^{\vee}=(H ^{\vee}) ^{\ast}(g^{\vee})+u_{0};
$$
  \item \label{item:22} 
$H^{\vee}(\Stab(f))\subset \Stab(A^{\ast}f)
      \subset \ov {H^{\vee}(\Stab(f))}
      $    and 
$$
(A ^{\ast} \cl(f))^{\vee}
=\cl((H ^{\vee})_{\ast}(f^{\vee}-u_{0}));
$$
  \item \label{item:23} if $\ri(\Dom(f))\cap
    \Im(A)\not = \emptyset$ then $\Stab(A^{\ast}f)=H^{\vee}(\Stab(f))$
    and, for all $y$ in this set,
    \begin{equation*}
      (A^{\ast} f)^{\vee}(y)=(H^{\vee})_{\ast}(f^{\vee}-u_{0})(y)=
      \max_{x\in(H^{\vee})^{-1}(y)}(f^{\vee}(x)-\langle x,u_{0}\rangle).
    \end{equation*}
    Moreover, if $f$ is closed and $y\in \ri (\Stab(A^{\ast}f))$, 
    then a point $x\in (H^{\vee})^{-1}(y)$ attains this maximum if and
    only if there exists $v\in Q_{\R}$ such that
    $x\in \partial f(Av)$. The element $v$ verifies
    $y\in \partial(A^{*}f)(v)$.  
  \end{enumerate}
\end{prop}

\begin{proof}
By Proposition~\ref{prop:8}(\ref{item:13},\ref{item:14}),
\begin{equation}
  \label{eq:6}
  A^{\ast}(f)= (H+{u_{0}})^{\ast}(f)= H^{\ast}(\tau _{-u_{0}}f),\quad
  A_{\ast}g= (H+{u_{0}})_{\ast}g= \tau_{u_{0}}(H_{\ast}g).
\end{equation}
Then, except for the second half of \eqref{item:23}, the result
follows by combining
(\ref{eq:6}) with the case when $A$ is a linear map, treated in~\cite[Theorem
16.3]{Roc70}.

To prove the second half of \eqref{item:23}, we first note that 
the concave function
$$
(f^{\vee}-u_{0}) |_{(H^{\vee})^{-1}(y)}
$$
attains its maximum at a point $x$ if and only if its sup-differential
at $x$ contains $0$. 
We consider the linear inclusion
$$
\iota\colon \Ker(H^{\vee}) \hookrightarrow M_{\R} 
$$
and denote by $\iota ^{\vee}\colon N_{\R}\to N_{\R}/\Im(H)$ its dual.  
We fix a point $$x_{0}\in(H^{\vee})^{-1}(y)\cap \ri(\Stab(f)),$$
that exists because 
$y\in \ri(\Stab(A^{\ast}f))$ and $\Stab(A^{\ast}f)=H^{\vee}(\Stab(f))$.

Set $F=(\iota+x_{0})^{*} (f^{\vee}-u_{0})$. Up to a translation, $F$ is
the restriction of
$f^{\vee}-u_{0}$ to $(H^{\vee})^{-1}(y)$. Since, by choice, 
$x_{0}\in \ri(\Dom (f^{\vee}-u_{0}))\cap \Im(\iota+x_{0})$, we can
apply Proposition \ref{prop:38}\eqref{item:43} to the concave
function $f^{\vee}-u_{0}$ and the affine map $\iota+x_{0}$ to deduce
that, for any 
$z\in \Ker(H^{\vee}) $, we have
$$
\partial F(z) = \iota^{\vee}(\partial f^{\vee}(z+x_{0}) - u_{0}).
$$ 
Therefore $0\in \partial F(z)$ if and only if  
$\partial  f^{\vee}(z+x_{0}) \cap (\Ker(\iota^{\vee})+u_{0}) \not =
\emptyset$.
Since $\Ker(\iota^{\vee})+u_{0}=\Im(A)$, a point 
$x=z+x_{0}\in (H^{\vee})^{-1}(y)$ attains the maximum if and only if
there is a $v\in Q_{\R} $ such that $Av\in \partial
f^{\vee}(x)$. Being $f$ closed, by Proposition \ref{prop:36} this is
equivalent to 
$x\in \partial f(Av)$ for some $v\in Q_{\R}$. By Proposition
\ref{prop:38}, $v$ satisfies $y\in \partial
(A^{*}f)(v)$.  
\end{proof}

In particular, the operations of direct and inverse image of
\emph{linear} maps are dual to each other. 
In the notation of Proposition~\ref{prop:11} and assuming for simplicity
$\ri(\Dom(f))\cap\Im(H)\not = \emptyset$, we have 
$$
(H_{\ast}g)^{\vee}=(H ^{\vee}) ^{\ast}(g^{\vee}) , \quad
 (H^{\ast} f)^{\vee}=(H^{\vee})_{\ast}(f^{\vee}),
$$
while the stability sets relate by 
$\Stab(H_{\ast}g)=(H^\vee)^{-1}(\Stab(g))$ and $\Stab(H^{\ast}f)=H^{\vee}(\Stab(f))$.

The last concept we recall in this section is the notion of recession
function of
a concave function.

\begin{defn}\label{def:12}  The \emph{recession function} of a
concave function $f\colon N_{\R}\to \un \R$,
\index{recession function!of a concave function}%
denoted $\rec(f)$, is the function
\begin{displaymath}
  \rec(f)\colon N_{\R} \longrightarrow \un\R , \quad u\longmapsto \inf_{v\in \Dom(f)} (f(u+v)-f(v)).
\end{displaymath}
\nomenclature[arec]{$\rec(f)$}{recession function of a concave function}%
This is a concave  conical function. If $f$ is
closed, its recession function can be defined as the limit
\begin{equation}\label{eq:49}
  \rec(f)(u)=\lim_{\lambda\to\infty} \lambda ^{-1} f(v_{0}+\lambda u)
\end{equation}
for any $v_{0}\in \Dom(f)$ \cite[Theorem 8.5]{Roc70}.
\end{defn}

It is clear from the definition that $\Dom(\rec(f))\subset
\rec(\Dom(f))$. The equality does not hold in general, as can be seen
by considering the concave  function $\R\to\R$, $u\mapsto -\exp(u)$.

If $f$
is closed then  the function $\rec(f)$ is closed~\cite[Theorem
8.5]{Roc70}. 
Hence it is natural to regard recession functions as support functions.

\begin{prop} \label{prop:17}
  Let $f$ be a concave function. Then $\rec(f^{\vee})$ is the support function of 
  $\Dom(f)$. 
  If $f$ is closed, then
  $\rec(f)$ is the support function of $\Stab(f)$.
\end{prop}
\begin{proof}
  This is \cite[Theorem 13.3]{Roc70}.
\end{proof}

\section{The differentiable case}
\label{sec:differentiable-case}

In this section we make explicit 
the Legendre-Fenchel duality for smooth concave functions,
following~\cite[Chapter~26]{Roc70}.

In the differentiable and strictly concave case, the decompositions
$\Pi(f)$ and~$\Pi(f^{\vee})$ consist of the collection of
all points of $\Dom(\partial f)$ and of $\Dom(\partial f^{\vee})$
respectively. The Legendre-Fenchel correspondence agrees with the
gradient map, and it is called the {Legendre transform} in this context.
\index{Legendre transform}%

Recall that a function $f\colon N_{\R}\to \un\R$ is
{differentiable} at a point $u\in N_{\R}$ 
\index{differentiable function} 
with $f(u)> -\infty$, if there exists some linear form $\nabla f(u)\in
M_{\R}$ such that
$$
f(v)=f(u)+\langle \nabla f (u),v-u\rangle +o(||v-u||),
$$
where $||\cdot||$ denotes any fixed norm on $N_{\R}$. This
linear form $\nabla(f)(u)$ is the gradient of $f$ in the
classical sense.  
It can be shown that a concave function $f$ is differentiable at a
point $u\in \Dom(f)$ if and only if $\partial f(u) $ consists of a
single element. If this is the case, then $\partial f(u)= \{\nabla
f(u)\}$~\cite[Theorem~25.1]{Roc70}. 
Hence, the gradient and the sup-differential agree in the differentiable case.

Let $C\subset N_{\R}$ be a  convex set. 
A function $f\colon C\to\R$ is \emph{strictly concave}
\index{strictly concave function}%
if 
$f(t u_1+(1-t)u_2) > tf(u_1)+(1-t) f(u_2)$ for all
different $u_1,u_2\in C$ and $0< t< 1$. 

\begin{defn}
  \label{def:28}
Let $C\subset N_{\R}$ be an open convex set
and $||\cdot||$ any fixed norm on $M_{\R}$.
A differentiable concave function  $f\colon
C\to\R$ is of \emph{Legendre type}
\index{concave function!of Legendre type}%
if it is strictly concave and
$\lim_{i\to \infty} \Vert\nabla f (u_i)\Vert\to \infty$ for every sequence
  $(u_i)_{i\ge1}$ converging to a point in the boundary of $C$. 
In particular, any differentiable and strictly concave function on
$N_{\R}$ is of Legendre type. 
\end{defn}

The stability
set of a function of Legendre type has maximal dimension \cite[Theorem
26.5]{Roc70}. Therefore its
relative interior agrees with its interior and, in this case, we will use the
classical notation  $\Stab(f)^{\circ}$ for the interior of
$\Stab(f)$.

The following result summarizes the basics properties of the
Legendre-Fenchel duality acting on functions of Legendre type. 

\begin{thm}   \label{thm:2}
Let
$f\colon C\to\R$ be a
concave function of Legendre type
defined on an open set $C\subset N_{\R}$ and let 
$D=\nabla f (C)\subset M_{\R}$ be the image of the gradient map. Then 
  \begin{enumerate}
  \item \label{item:19} $D=\Stab(f)^{\circ}$;
  \item \label{item:15} $f^{\vee}|_{D}$ is a concave function of
    Legendre type;
  \item \label{item:16}  $\nabla f \colon C\to D$ is a homeomorphism and
    $(\nabla f)^{-1}=\nabla f^{\vee}$;
  \item \label{item:51} for all $x\in D$ we have $f^{\vee}(x)=
    \langle x, (\nabla f)^{-1}(x)\rangle -f( (\nabla f)^{-1}(x))$.
  \end{enumerate}
\end{thm}

\begin{proof}
This follows from~\cite[Theorem~26.5]{Roc70}. 
\end{proof}

\begin{exmpl}
  \label{exm:5}
Consider the function
$$
f_\FS\colon\R^n \longrightarrow  \R , \quad u\longmapsto
-\frac{1}{2} \log\Big(1+\sum_{i=1}^n \e^{-2u_i}\Big).
$$
Let $\Delta ^{n}=\{(x_{1},\dots,x_{n})\subset \R^{n}\mid x_{i}\ge
0,\sum x_{i}\le 1\}$ be the 
\emph{standard simplex} of $\R^{n}$.
\index{standard simplex}%
\nomenclature[g04n]{$\Delta ^{n}$}{standard simplex}%
For
$(x_{1},\dots,x_{n})\in \Delta ^{n}$,
write $x_0=1-\sum_{i=1}^n x_i$ and set 
\begin{equation}
  \label{eq:103}
\varepsilon_{n}\colon  \Delta^n\longrightarrow \R ,\quad  x\longmapsto -\sum_{i=0}^n
x_i\log(x_i).   
\end{equation}
We have 
$\displaystyle \nabla f_{\FS} (u)= \frac{1}{1+\sum_{i=1}^n \e^{-2u_i}}
\left(\e^{-2u_1},\dots,\e^{-2u_n}\right) $ and so
\begin{align*}
\frac{1}{2}\varepsilon_{n}(\nabla f_{\FS} (u))&= \frac{\sum_{i=1}^n
  \e^{-2u_i} u_i}{1+\sum_{i=1}^n \e^{-2u_i}}  +\frac{1}{2}
\log\Big(1+\sum_{i=1}^n \e^{-2u_i}\Big)= \langle \nabla f_\FS (u), u \rangle - f_\FS(u),
\end{align*}
which shows that $\Stab(f_{\FS})=\Delta^{n}$ and that $f_\FS^\vee=\frac{1}{2} \varepsilon_{n}$.
\end{exmpl}


The fact that the sup-differential agrees with the gradient and is single-valued can
simplify some statements. 
It is interesting to make explicit the computation of the Legendre-Fenchel
dual of the inverse image by an affine map of a concave function of
Legendre type.

\begin{prop}
\label{prop:27}
Let $A\colon Q_\R\to N_\R$ be an affine map defined as $A=H+u_0$ for
an injective linear map $H$ 
and a point $u_0\in N_{\R}$. 
Let $f\colon C\to \R$ be a concave function of Legendre type 
defined on an open convex set $C\subset
N_{\R}$ such that $C\cap \Im(A) \not = \emptyset$. Then 
$A^{*}f$ is a concave function of Legendre type on $A^{-1}(C)$, 
$$
\Stab(A^{\ast}f)^{\circ}=\Im(\nabla (A^{*}f)) =
H^{\vee}(\Im(\nabla f))=H^{\vee}(\Stab(f)^{\circ}) ,
$$
    and, for all $v\in A^{-1}C$,
    \begin{equation*}
      (A^{\ast} f)^{\vee}( \nabla (A^{\ast }f)(v))= f^{\vee}(\nabla
      f(Av)) - \langle \nabla f(Av),u_{0}\rangle .
    \end{equation*}
Moreover, there is a section $\imath _{A,f}$
of $H^{\vee}|_{\Stab(f)^{\circ}}$ such that the
diagram
\begin{displaymath}
  \begin{split}
  \xymatrix{
    & A^{-1}C
    \ar[dd]^{A}\ar[dl]_{A^{\ast}f}\ar[rr]^{\nabla(A^{\ast}f)}& & 
    \Stab(A^{\ast}f)^{\circ}\ar[dd]_{\imath _{A,f}} \ar[rd]^{(A^{\ast}f)^{\vee}}&\\
    \R & & & & \R\\
    & C \ar[ul]^{f}\ar[rr]_{\nabla f}& &\Stab(f)^{\circ}\ar[ru]_{f^{\vee}-u_{0}}&
  }    
  \end{split}
\end{displaymath}
commutes.
\end{prop}
\begin{proof}
This follows readily from Proposition~\ref{prop:11}.  
\end{proof}

The section $\imath _{A,f}$  embeds $\Stab(A^{\ast}f)^{\circ}$ as a real
submanifold  of $\Stab(f)^{\circ}$. 
Varying~$u_{0}$ in a suitable space of parameters, we obtain a foliation of
$\Stab(f)^{\circ}$ by ``parallel'' submanifolds. We illustrate this
phenomenon with an example in dimension~2. 

\begin{exmpl} \label{exm:16} 
Consider the function $f\colon \R^{2}\to \R$ given by
$$f(u_1,u_2) =
-\frac{1}{2}\log\left(1+\e^{-2u_1}+\e^{-4u_1-2u_2}+\e^{-2u_1-4u_2}\right).$$
It is a concave function of Legendre type whose stability set is the
polytope $\Delta=\Conv((0,0),(1,0),(2,1),(1,2))$.
The restriction of its Legendre-Fenchel dual to $\Delta ^{\circ}$ is
also a concave function of Legendre type.

For $c\in \R$, consider the affine  map 
\begin{displaymath}
A_c\colon \R\to \R^{2}, \quad 
u\longmapsto (-u, u+c).
\end{displaymath}
We write $A_{c}=H+(0,c)$ for a linear function $H$. 
The dual of $H$ is the function $H^\vee\colon
\R^2\to\R$, $(x_1,x_2)\mapsto x_2-x_1$.
Then $\Stab(A_c^*f)^\circ=H^\vee(\Delta^\circ)$ is the open
interval~$(-1,1)$. By Proposition \ref{prop:27}, there is a map 
$\imath_{A_c,f}$
embedding $(-1,1)$ into $\Delta^\circ$ in such a way that
$\imath_{A_c,f}\circ\nabla(A_c^*f) = (\nabla f)\circ A_c$. For
$u\in\R$,
\begin{align*}
\nabla(A_c^*f) (u)&=
\frac{\e^{-2u-4c}-\e^{2u}-\e^{2u-2c}}{1+\e^{2u}+\e^{2u-2c}+\e^{-2u-4c}}\in
(-1,1),\\
(\nabla f)\circ A_c(u) &=
\frac{\left(\e^{2u}+2\e^{2u-2c}+\e^{-2u-4c},\e^{2u-2c}+2\e^{-2u-4c}\right)}
{1+\e^{2u}+\e^{2u-2c}+\e^{-2u-4c}}\in\Delta^\circ.
\end{align*}
From this, we compute $\imath_{A_c,f}(x)=\left(x_1,x_2\right)$ with
$$\left\{\begin{aligned}
    x_1 &= \frac{-\e^{-2c}}{2(1+\e^{-2c})}x +
    \frac{2+3\e^{-2c}}{2(1+\e^{-2c})}
    \left(\frac{x^2}{\sqrt{\rho_c^2+(1-\rho_c^2)x^2}+\rho_c} +
      \frac{\rho_c}{1+\rho_c}\right),\\
    x_2 &= \frac{2+\e^{-2c}}{2(1+\e^{-2c})}x +
    \frac{2+3\e^{-2c}}{2(1+\e^{-2c})}
    \left(\frac{x^2}{\sqrt{\rho_c^2+(1-\rho_c^2)x^2}+\rho_c} +
      \frac{\rho_c}{1+\rho_c}\right),
\end{aligned}\right.$$
where we have set $\rho_c=2\e^{-2c}\sqrt{1+\e^{-2c}}$ for short. In
particular, the image of the map $\imath_{A_c,f}$ is an arc of conic:
namely the
intersection of 
$\Delta^\circ$ with the conic of equation
$$(x_2-x_1)^2 = (1-\rho_c^2)L_c(x_1,x_2)^2 + 2\rho_cL_c(x_1,x_2),$$
with $L_c(x_1,x_2)=\frac{2+\e^{-2c}}{2+3\e^{-2c}}x_1 +
\frac{\e^{-2c}}{2+3\e^{-2c}}x_2 - \frac{\rho_c}{1+\rho_c}$. 
Varying $c\in \R$, these arcs of conics form a
foliation of $\Delta^\circ$, they all pass through the vertex $(1,2)$
as $x\to 1$, and their other end as $x\to-1$ parameterizes the
relative interior of the edge
$\Conv((1,0),(2,1))$, see Figure~\ref{fig:10}.

\begin{figure}[htpb]
\centerline{
\includegraphics[height=4.5cm]{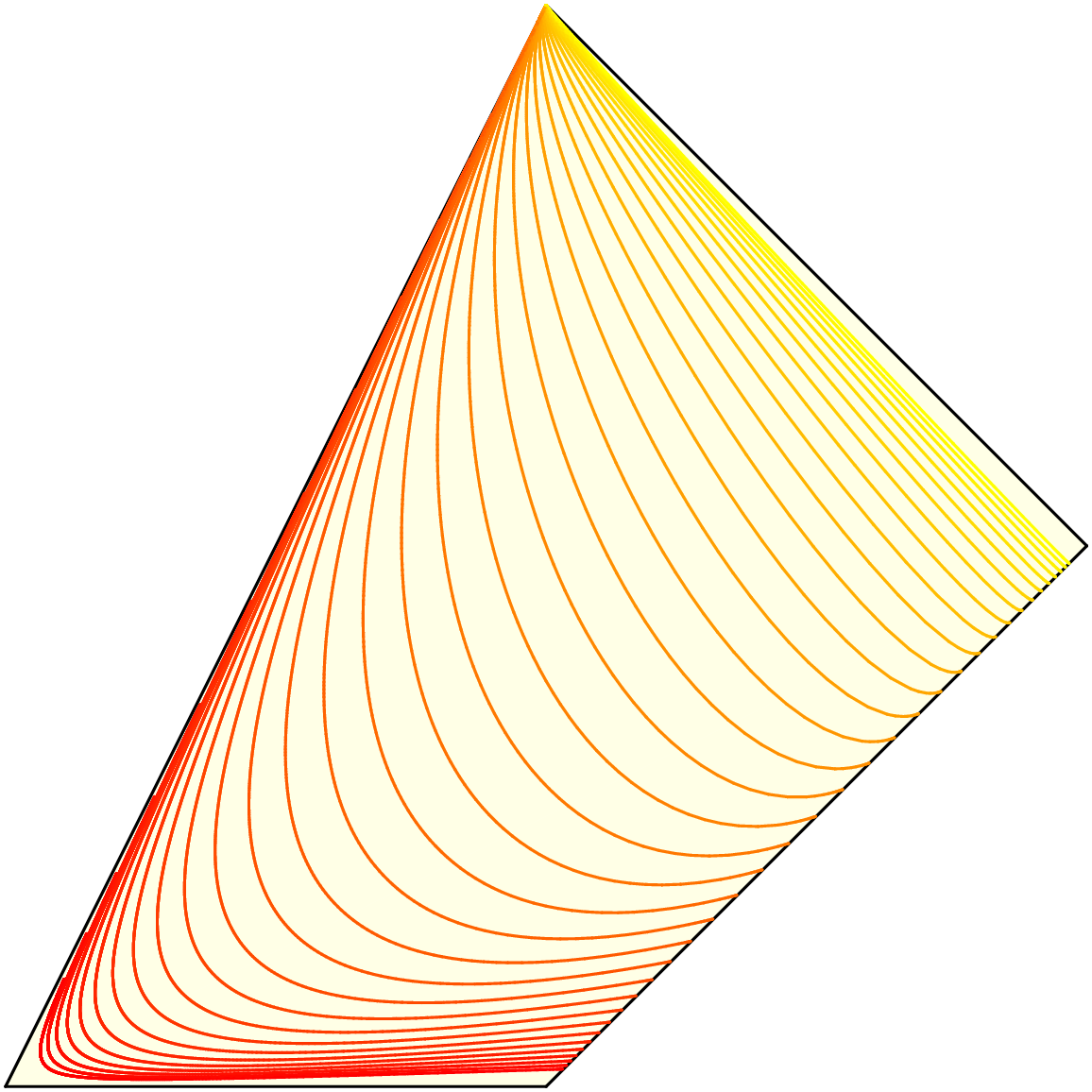}}
\caption{A foliation of $\Delta^\circ$ by curves}\label{fig:10}
\end{figure}
\end{exmpl}

\section{The piecewise affine case}
\label{sec:functions-polyhedra}

The Legendre-Fenchel duality for piecewise affine concave
functions can be described
in combinatorial terms. Moreover, some technical issues of the general
theory disappear when 
dealing with piecewise affine concave functions on convex polyhedra and
uniform limits of such functions.    

\begin{defn} \label{def:20}
  Let $C\subset N_{\R}$ be a polyhedron. A function
  $f\colon C\to \R$ is 
   \emph{piecewise affine}\index{piecewise affine function} if there
   is a
   finite cover of $C$ by closed subsets such
   that the restriction of $f$ to each of these subsets is an affine
   function. Such affine functions are called \emph{defining
     functions}\index{piecewise affine function!defining functions of a} of $f$.
   A concave function $f\colon N_{\R}\to \un \R$ is said to be
   \emph{piecewise affine} if $\Dom(f)$ is a polyhedron and the
   restriction $f|_{\Dom(f)}$ is piecewise affine.
\end{defn}

\begin{rem}\label{rem:30}
  Considering a polyhedron $C$ inside its affine envelope, we may
  think of it as a
  closed domain, that is, the closure of an open set. By an argument
  of general topology, if a closed domain $C$ is a finite union of closed
  subsets $C=\bigcup_{i=1}^{m}D_{i}$, then
  $C=\bigcup_{i=1}^{m}\overline {D^{\circ}_{i}}$. Therefore, when the
  domain is a polyhedron, 
  our definition of piecewise affine function agrees with the notion
  of piecewise linear function of~\cite{Ovchinnikov:maxmin}.
\end{rem}

\begin{lem}\label{lemm:4}
  Let $f$ be a piecewise affine function defined on a
  polyhedron~$C\subset N_{\R}$. Then there exists a polyhedral
   complex $\Pi$ in $C$ such that the restriction of $f$ to each
   polyhedron of $\Pi$ is an affine function.
\end{lem}
\begin{proof}
  This is an easy consequence of the max-min representation of
  piecewise affine functions in~\cite{Ovchinnikov:maxmin}, that can be
  applied thanks to Remark \ref{rem:30}. 
\end{proof}

\begin{defn} \label{def:54}
Let $C$ be a polyhedron, $\Pi $ 
a polyhedral complex in $C$ and $f\colon C\to \R$ 
a piecewise affine function. We say that $\Pi $ and $f$ are
\emph{compatible}
\index{piecewise affine function!on a polyhedral complex}%
\index{polyhedral complex!compatible with a piecewise affine function}%
if $f$ is affine on each polyhedron of $\Pi$.
Alternatively, we say that $f$ is a 
piecewise affine function \emph{on  $\Pi $}. 
If the function $f$ is concave, it is said to be \emph{strictly
  concave on $\Pi$}
\index{strictly concave function!on a polyhedral complex}%
if $\Pi=\Pi(f)$. The polyhedral complex $\Pi $ is said to be
\emph{regular}
\index{polyhedral complex!regular}%
if there exists a concave piecewise affine function $f$
such that $\Pi =\Pi (f)$.
\end{defn}


As was the case for polyhedra, piecewise affine concave
functions 
\index{piecewise affine concave function}%
can be described in two dual ways, which we refer as the
H-representation and the V-representation. 

 For the
\emph{H-representation},
\index{H-representation!of a piecewise affine concave function}%
we consider a polyhedron 
\begin{displaymath}
  \Lambda=\bigcap_{1\le j\le k} \{u\in N_{\R} \mid \langle a_j,u\rangle
  +\alpha_{j}\ge 0 \}
\end{displaymath}
as in
\eqref{eq:10} and a set of affine equations $\{(a_{j},\alpha
_{j})\}_{k+1\le j\le l}\subset M_{\R}\times \R$. We then define a
concave function on $N_{\R}$ as
\begin{equation}
  \label{eq:38}
  f(u)=
  \begin{cases}
    \min_{k+1\le j \le l}(\langle a_{j},u\rangle +\alpha_{j}) & 
    \text{ for }  u\in \Lambda,\\
    -\infty & \text{ for } u\notin \Lambda.
  \end{cases}
\end{equation}
The equation (\ref{eq:38}) is an H-representation of the function $f$. 
With this representation, the recession function of $f$ is given by 
\begin{displaymath}
  \rec(f)(u)=    \min_{k+1\le j \le l}\langle a_{j},u\rangle,  \quad
  \text{ for } u\in \rec(\Lambda)
\end{displaymath}
and $\rec(f)(u)=-\infty$ for $u\notin \rec(\Lambda)$.
  In particular,
\begin{equation}
  \label{eq:18}
  \Dom(\rec(f))=\rec(\Dom(f)),\quad
  \Stab(\rec(f))=\Stab(f).
\end{equation}
For the \emph{V-representation},
\index{V-representation!of a piecewise affine concave function}%
we consider a polyhedron 
\begin{displaymath}
  \Lambda'=\Cone(b_{1},\dots,b_{k})+\Conv(b_{k+1},\dots,b_{l})
\end{displaymath}
as
in~\eqref{eq:35}, a set of slopes $\{\beta _{j}\}_{1\le j\le k}\subset
\R$ and a set of values $\{\beta _{j}\}_{k+1\le j\le l} \subset \R$. 
We then define a concave function on $N_{\R}$  as
\begin{equation}
  \label{eq:39}
g(u) = \sup\bigg\{\sum_{j=1}^{l}\lambda_{j}\beta _{j}\bigg|\
\lambda_{j}\ge 0, \ \sum_{j=k+1}^{l}\lambda_{j}=1, \
\sum_{j=1}^{l}\lambda_{j}b_{j}=u
\bigg \}.   
\end{equation}
This equation is the V-representation of the function $g$.
With this second representation, we obtain the recession function as
\begin{displaymath}
\rec(g)(u) = \sup\bigg\{\sum_{j=1}^{k}\lambda_{j}\beta _{j} \bigg|\
\lambda_{j}\ge 0, \ 
\sum_{j=1}^{k}\lambda_{j}b_{j}=u
\bigg \}.     
\end{displaymath}
We will typically use the H-representation for functions on
$N_{\R}$ while we will use the V-representation for functions in
$M_{\R}$. 

As we have already mentioned, the Legendre-Fenchel duality of 
piecewise affine concave functions can be described in combinatorial terms.

\begin{prop}\label{prop:19} Let $\Lambda $ be a polyhedron in $N_{\R}$ and 
$f$ a piecewise affine concave function with $\Dom(f)=\Lambda $
given as
\begin{align*}
  \Lambda&=\bigcap_{1\le j\le k} \{u\in N_{\R} \mid \langle a_j,u\rangle
  +\alpha_{j}\ge 0 \},\\
  f(u)&=    \min_{k+1\le j \le l}(\langle a_{j},u\rangle +\alpha_{j})  \quad
    \text{for }  u\in \Lambda
\end{align*}
with $a_{j}\in M_{\R}$ and $\alpha_{j}\in \R$. 
Then
\begin{align*}
\Stab(f)&= \Cone(a_{1},\dots,a_{k})
+\Conv(a_{k+1},\dots,a_{l}),\\ 
f^{\vee}(x) &= \sup\bigg\{\sum_{j=1}^{l}-\lambda_{j}\alpha_{j}\bigg|\ 
\lambda_{j}\ge 0, \sum_{j=k+1}^{l}\lambda_{j}=1, \  
\sum_{j=1}^{l}\lambda_{j}a_{j}=x \bigg \}\    \text{ for }  x\in \Stab(f). 
\end{align*}
\end{prop}
\begin{proof}
  This is proved in \cite[pages~172--174]{Roc70}.
\end{proof}

\begin{exmpl} \label{exm:1} Let $\Lambda $ be a polyhedron in
  $N_{\R}$. Then both the indicator function~$\iota_{\Lambda}$
and the support function
  $\Psi_{\Lambda}$
are concave and piecewise
  affine. We have~$\Psi_{\Lambda}^{\vee}=\iota_{\Lambda}$.  In
  particular, if we fix an isomorphism $N_{\R}\simeq \R^{n}$, the
  function
$$
\Psi _{\Delta ^{n}} \colon N_{\R}\longrightarrow \R , \quad 
(u_{1},\dots,u_{n})\longmapsto \min\{ 0, u_{1},\dots, u_{n}\} 
$$
is the support function of the {standard simplex}
\index{standard simplex}%
\index{support function!of the standard simplex}%
$\Delta^n=\Conv(\bfzero,e^{\vee}_1,\dots,e^{\vee}_{n})\subset
M_{\R}$, 
where $\{e_{1},\dots,
e_{n}\}$ is the standard basis of $\R^{n}$ and $\{e_{1}^{\vee},\dots,
e_{n}^{\vee}\}$ is the dual basis. Hence, $\Stab(\Psi _{\Delta ^{n}})=\Delta^n$ and
$\Psi _{\Delta ^{n}}^{\vee}=\iota  _{\Delta^{n}}$.
\end{exmpl}

Let $\Lambda $ be a polyhedron in $N_{\R}$ and 
$f$ a piecewise affine concave function with $\Dom(f)=\Lambda
$. Then $\Dom(\partial f)=\Lambda$ and $\Pi(f)$ and $\Pi(f^{\vee})$
are convex decompositions of
$\Lambda$ and of $\Lambda':=\Stab(f)$ respectively.
By Theorem~\ref{thm:8}, the Legendre-Fenchel correspondence 
$$
\cL f \colon \Pi (f)\longrightarrow \Pi (f^{\vee}) 
$$
is a duality in the sense of Definition~\ref{def:9}.
However in the polyhedral case, these decompositions are dual in a stronger
sense.
We need to introduce some more definitions before we can properly 
state this duality. 

\begin{defn} \label{def:78}
Let $\Lambda $ be a polyhedron and $K$ a face of $\Lambda $. The 
\emph{angle} of $\Lambda $ at $K$
\index{polyhedron!angle at a face}%
is defined as
$$
\angle(K,\Lambda )= \{ t(u-v) \mid u\in \Lambda , v\in K, t\ge 0\}.
$$
\nomenclature[sdtangle]{$\angle(K,\Lambda )$}{angle of a polyhedron at a face}%
It is a polyhedral cone.   
\end{defn}

\begin{defn} \label{def:77}
The \emph{dual} of a convex cone $\sigma\subset N_{\R}$
\index{dual!of a convex cone}%
is defined as
$$
\sigma ^{\vee} = \{x\in M_{\R} \mid \langle x,u\rangle \ge 0 
\text{ for all } 
  u\in \sigma \}.
$$
\nomenclature[sdual]{$\sigma ^{\vee}$}{dual of a convex cone}%
This is a convex closed cone.     
\end{defn}

If $\sigma $ is a convex closed cone, then $\sigma ^{\vee\vee}=\sigma
$. The following result is a direct consequence of Proposition
\ref{prop:19}.

\begin{cor}\label{cor:2}
Let  $f$ be a piecewise affine concave function on $N_{\R}$. Then
\begin{displaymath}
  \rec(\Dom(f))^{\vee}=\rec(\Stab(f)).
\end{displaymath}  
In particular, if $\Dom(f)=N_{\R}$, then $\Stab(f)$ is a polytope. 
\end{cor}

\begin{defn}\label{def:10}
Let $C, C'$ be polyhedra in $N_{\R}$ and $M_{\R}$,
respectively, and~$\Pi , \Pi '$ 
polyhedral complexes in $C$ and $C'$, respectively. 
We say that $\Pi $ and $\Pi '$ are \emph{dual} polyhedral
  complexes
\index{dual!polyhedral complexes}%
if there is a bijective map $\Pi  \to \Pi ', \Lambda \mapsto\Lambda ^{*} $
such that 
\begin{enumerate}
\item for all $\Lambda ,K \in \Pi $, the inclusion $K\subset \Lambda $
  holds if and
  only if $K^{*}\supset \Lambda ^{*}$;
\item \label{item:18} for all $\Lambda ,K \in \Pi $, if 
  $K\subset \Lambda $, then $\angle(\Lambda^*,K^*
  )=\angle(K,\Lambda )^{\vee}$.
\end{enumerate}
\end{defn}

For $\Lambda\in \Pi$, the angle $\angle(\Lambda ,\Lambda )$ is the
linear subspace generated 
by differences of points in $\Lambda $. 
Condition~(\ref{item:18}) above implies that $\angle(\Lambda ,\Lambda )$
and $\angle(\Lambda^{*} ,\Lambda ^{*})$ are orthogonal. In 
particular, $\dim(\Lambda )+\dim(\Lambda ^{*})=n$.

\begin{prop}\label{prop:63} Let $f$ be a piecewise affine concave
  function with $\Lambda =\Dom(f)$ and $\Lambda '=\Stab(f)$. Then 
  $\Pi(f)$ and $\Pi(f^{\vee})$ are polyhedral complexes in $\Lambda $
  and $\Lambda '$ respectively. Moreover, they are dual of each other.
  In particular, the vertices of~$\Pi(f)$ are in bijection with the
  polyhedra of $\Pi(f^{\vee})$ of maximal dimension.
\end{prop}
\begin{proof}
  This is proved in \cite[Proposition~1]{PassareRullgard2004}.
\end{proof}

\begin{exmpl} \label{exm:8}
Consider the  standard simplex $\Delta^{n}$ of Example \ref{exm:1}. 
Its indicator function induces the standard polyhedral complex in
$\Delta^{n}$ 
consisting of the collection of its faces. 
The dual of $\iota_{\Delta^{n}}$, the support function $\Psi _{\Delta ^{n}}$,
induces a fan $\Sigma_{\Delta ^{n}}:=\Pi(\Psi _{\Delta ^{n}})$ of
$N_{\R}$.  
The duality between these polyhedral complexes can be made explicit as
\begin{displaymath}
\Pi(\iota_{\Delta^{n}})\longrightarrow \Sigma_{\Delta ^{n}},\quad
F\longmapsto \angle(F,\Delta^{n})^{\vee}.
\end{displaymath}
\end{exmpl}

 \begin{exmpl}
   \label{exm:19}
   The previous example can be generalized to an arbitrary polytope
   $\Delta\subset M_{\R}$ . The indicator function $\iota _{\Delta }$
   induces the standard decomposition of $\Delta$ into its faces and,
   by duality, the support function $\Psi _{\Delta }$ induces a
   polyhedral complex $\Sigma _{\Delta }:=\Pi (\Psi _{\Delta })$
   \nomenclature[g1814]{$\Sigma _{\Delta}$}{fan associated to a
     polytope}%
   made of cones. If $\Delta $ is of maximal dimension, then $\Sigma
   _{\Delta }$ is a fan.

   The faces of $\Delta$ are in one-to-one correspondence with the
   cones of $\Sigma_{\Delta}$ through the Legendre-Fenchel correspondence.
   For a face $F$ of $\Delta$, its corresponding cone is
   $$
   \sigma_F:= F^{*}=\{ u\in N_\R\mid \langle u,y-x\rangle \ge 0 \mbox{ for all }
   y\in \Delta, x\in F\}.
   $$
   \nomenclature[g1802]{$\sigma_F$}{cone dual to a face}%
   Reciprocally, to each cone $\sigma $ corresponds a face
   of $\Delta$ of complementary dimension 
   $$
   F_\sigma:=\sigma^{*}=\{ x\in \Delta\mid \langle x,u\rangle =
   \Psi_{\Delta }(u) \mbox{ 
     for all } u\in \sigma\}. 
   $$
   \nomenclature[aF]{$F_\sigma$}{face dual to a cone}%
   On a cone $\sigma\in \Sigma$, the function $\Psi _{\Delta }$ is
   defined by any 
   vector $m_{\sigma}$ in the affine space~$\aff(F_{\sigma})$. The
   cone $\sigma $ is normal to $F_{\sigma}$. 

   We will use the notation $F_{\sigma }$ in a more general
   context. If $\Sigma $ is a refinement of $\Sigma _\Delta $ and $\sigma \in
   \Sigma $, we will denote by $F_{\sigma }$ the face of $\Delta $
   given by the condition
   \begin{displaymath}
     F_{\sigma }=\{ x\in \Delta \mid \langle u,y-x\rangle \ge 0 \mbox{ for all }
   y\in \Delta, u\in \sigma \}.
   \end{displaymath}
\end{exmpl}

For piecewise affine concave functions, the operations of taking the
recession function and the associated 
polyhedral complex commute with each other.

\begin{prop} \label{prop:80}
  Let $f$ be a piecewise affine concave function on $N_{\R}$. Then
    \begin{displaymath}
      \Pi (\rec(f))=\rec(\Pi(f)).
    \end{displaymath}
\end{prop}
\begin{proof}
  Let $P_{f}(u,x)=f(u)+f^{\vee}(x)-\left< u,x\right>$ be the function
  introduced in \eqref{eq:19}. For each $x\in \Stab(f)$ write
  $P_{f,x}(u)=P_{f}(u,x)$ which is a piecewise affine concave function. Let
  $C_x$ be as in Definition \ref{def:29}. By 
  Lemma \ref{lemm:19},
  \begin{displaymath}
    C_{x}=\{u\in \Dom(f)\mid P_{f,x}(u)=0\}. 
  \end{displaymath}
  Write $P'(v)=\rec(f)(v)-\left<u,x\right>$. Then
  $P'=\rec(P_{f,x})$. 

  We claim that, for each $x\in \Stab(f)$,
  \begin{displaymath}
    \rec(C_{x})=\{v\in \Dom(\rec(f))\mid P'(v)=0\}.
  \end{displaymath}
Let $v\in \rec(C_{x})$. Clearly $v\in \Dom(\rec(f))$ and,
  since $x\in \Stab(f)$, the set $C_{x}$ is non-empty. Let $u_{0}\in
  C_{x}$. Then, for each $\lambda >0$, $u_{0}+\lambda v\in
  C_{x}$. Therefore,
  \begin{displaymath}
    P'(v)=\lim_{\lambda \to \infty}\frac{P_{f,x}(u_{0}+\lambda
      v)-P_{f,x}(u_{0})}{\lambda }=0. 
  \end{displaymath}
  Conversely, let $v\in \Dom(\rec(f))$ satisfying $P'(v)=0$ and
  $u\in C_{x}$. On the one hand, by
  the properties of the function $P_{f}$, we have $P_{f,x}(u+v)\le 0$. On
  the other hand, since $P'=\rec(P_{f,x})$ and $P_{f,x}$ is a
  piecewise affine concave function,
  \begin{displaymath}
    P_{f,x}(u+v)-P_{f,x}(u)\ge P'(v)=0.
  \end{displaymath}
  Thus $P_{f,x}(u+v)\ge P_{f,x}(u) = 0$ and finally $P_{f,x}(u+v) = 0$. This
  implies that, if $u\in C_{x}$ then $u+v\in C_{x}$, showing 
  $v\in \rec(C_{x})$. Hence the claim is proved.

  By definition $\Pi (f)=\{C_{x}\}_{x\in \Stab(f)}$. Hence $\rec(\Pi
  (f))=\{\rec(C_{x})\}_{x\in \Stab(f)}$. For each $x\in \Stab(\rec(f))$,
  write
  \begin{displaymath}
    C'_{x}=\{v\in \Dom(\rec(f))\mid P'(v)=0\}.
  \end{displaymath}
  Then $\Pi (\rec(f))=\{C'_{x}\}_{x\in \Stab(\rec(f))}$. The result
  follows from the previous claim and the fact that
  $\Stab(f)=\Stab(\rec(f))$ by \eqref{eq:18}.
\end{proof}

Now we want to study the compatibility of Legendre-Fenchel duality and
integral and rational structures.

\begin{defn} \label{def:13} Let $L$ be a lattice in a finite
  dimensional real vector space $L_{\R}=L\otimes \R$ and $L^{\vee}$
  the dual lattice. A piecewise
  affine concave function $f$ on
  $L_{\R}$ is an \emph{H-lattice}
  \index{piecewise affine concave function!H-lattice}%
  concave function 
  if it has an H-representation as (\ref{eq:38}) with $a_{j}\in
  L^{\vee}$ and $\alpha _{j}\in \Z$ for $j=1,\dots,l$.
  It is a \emph{V-lattice} concave function
  \index{piecewise affine concave function!V-lattice}%
  if it has a V-representation as (\ref{eq:39}) with $b_{j}\in L$ and
  $\beta _{j}\in \Z$, for $j=1,\dots,l$. 
  We say that $f$ is a
  \emph{rational} piecewise affine concave function
  \index{piecewise affine concave function!rational}%
  if it has an H-representation as before with  $a_{j}\in
  L^{\vee}\otimes \Q$ and $\alpha _{j}\in \Q$ for $j=1,\dots,l$, 
  or equivalently, a V-representation with $b_{j}\in L\otimes \Q$ and
  $\beta _{j}\in \Q$.
\end{defn}

Observe that the domain of a V-lattice concave function is a lattice
polyhedron, whereas the domain of an H-lattice concave function is a
rational polyhedron.

Let $N \simeq \Z^n$ be a lattice of rank $n$ such that
$N_\R=N\otimes \R$. Set 
$M=N^\vee=\Hom(N,\Z)$ for its dual lattice, so
$M_\R=M\otimes\R$. We also set $N_{\Q}=N\otimes \Q$ and
$M_{\Q}=M\otimes \Q$.

\begin{rem}
The notion of H-lattice concave functions defined on the whole
$N_{\R}$ coincides with the notion of 
{tropical Laurent polynomials}
\index{tropical Laurent polynomial}%
over the integers, that is, the elements of the group 
semi-algebra $\Z_{\text{trop}}[N]$,
where the arithmetic operations of the base semi-ring
$\Z_{\text{trop}}=(\Z,\oplus,
\odot)$ are defined as
$ x\oplus y=\min(x,y)$ and $x\odot y=x+y$.  
\end{rem}

\begin{prop}\label{prop:18} Let $f$ be a 
 piecewise affine concave 
  function on $N_{\R}$.
  \begin{enumerate}
\item \label{item:26} $f$ is an H-lattice concave function 
  (respectively, a 
rational piecewise affine concave function) if and only if $f^{\vee}$
is a V-lattice concave function (respectively, a  rational piecewise
    affine concave  function) on $M_{\R}$.
\item  \label{item:20} $\rec(f)$ is an H-lattice concave function if and only if
    $\Stab(f)$ is a lattice polyhedron. In this case $\rec(f)$ is the
    support function of $\Stab(f)$.
    \end{enumerate}
\end{prop}
\begin{proof}
  This follows easily from Proposition \ref{prop:19}.
\end{proof}

\begin{exmpl}
  \label{exm:13} If $\Delta $ is a lattice polytope, its
  indicator function is a V-lattice function, its support
  function $\Psi _{\Delta }$ is an H-lattice function and, when
  $\Delta $ has maximal dimension, the fan
  $\Sigma _{\Delta }$ is a rational fan. In particular, if the
  isomorphism $N\simeq \R^{n}$ of Example \ref{exm:1} is given by the
  choice of an integral basis
  $e_{1},\dots,e_{n}$  of 
  $N$, then $\Delta ^{n}$ is a lattice polytope, the function $\Psi
  _{\Delta^{n} }$ 
  is an H-lattice concave function and $\Sigma
  _{\Delta ^{n}}$ is a rational fan. If we write
  $e_{0}=-\sum_{i=1}^{n}e_{i}$, this is the fan generated by the
  vectors $e_{0},e_{1}, 
  \dots,e_{n}$ in the sense that each cone of
  $\Sigma_{\Delta ^{n}}$ is the cone generated by a strict subset of the
  above set of vectors. Figure \ref{fig:simplex} illustrates the case $n=2$. 
 \end{exmpl}
 \begin{figure}[h]
   \centering
   \input{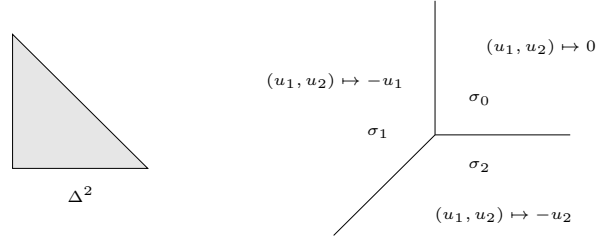}
   \caption{The standard simplex $\Delta ^{2}$, its associated fan and
   support function}
   \label{fig:simplex}
 \end{figure}

\begin{defn} \label{def:52}
 Let $\Lambda$ and $\Lambda'$ be polyhedra in $N_{\R}$ and in
 $M_{\R}$, respectively.  We set $\mathscr{P}(\Lambda,\Lambda')$ for
 the {space of piecewise affine concave functions with effective
   domain $\Lambda$ and stability set
   $\Lambda'$}.
\nomenclature[aPiece2]{$\mathscr{P}(\Lambda,\Lambda')$}{spaces of piecewise affine
   functions with given effective domain and stability set}%
We also set $\ov{\mathscr{P}}(\Lambda,\Lambda')$
\nomenclature[aPiece6]{$\ov{\mathscr{P}}(\Lambda,\Lambda')$}{closure
  of $\mathscr{P}(\Lambda,\Lambda')$}%
 for the closure of this space
 with respect to uniform convergence.  We set
$$
\mathscr{P}(\Lambda)= \bigcup_{\Lambda'}
\mathscr{P}(\Lambda,\Lambda'), \quad 
\ov{\mathscr{P}}(\Lambda)= \bigcup_{\Lambda'} \ov{\mathscr{P}}(\Lambda,\Lambda')
$$
\nomenclature[aPiece3]{$\mathscr{P}(\Lambda)$, $\mathscr{P}(\Lambda)_{\Z}$}{spaces of piecewise affine
   functions with given effective domain}%
\nomenclature[aPziece5]{$\ov{\mathscr{P}}(\Lambda)$,
  $\ov{\mathscr{P}}(\Lambda)_{\Z}$}{closures of $\mathscr{P}(\Lambda)$
and  $\mathscr{P}(\Lambda)_{\Z}$}%
for the {space of piecewise affine concave functions with
  effective domain $\Lambda$} and for its closure
with respect to uniform convergence, respectively. If we want to single
out the elements of the previous spaces whose stability set is a
lattice polyhedron  we will write
$$
\mathscr{P}(\Lambda)_{\Z}= \bigcup_{\Lambda' \text{ lat. pol.}}
\mathscr{P}(\Lambda,\Lambda'), \quad 
\ov{\mathscr{P}}(\Lambda)_{\Z}= \bigcup_{\Lambda'\text{ lat. pol.}} \ov{\mathscr{P}}(\Lambda,\Lambda'),
$$
where the union runs over all lattice polyhedrons.
We also set
$$
\mathscr{P}= \bigcup_{\Lambda,\Lambda'}
\mathscr{P}(\Lambda,\Lambda'), \quad 
\ov{\mathscr{P}}= \bigcup_{\Lambda,\Lambda'} \ov{\mathscr{P}}(\Lambda,\Lambda').
$$
\nomenclature[aPiece1]{$\mathscr{P}$}{space of piecewise affine functions}%
\nomenclature[aPiece4]{$\ov{\mathscr{P}}$}{closure of $\mathscr{P}$}%
When we need to specify the vector space $N_{\R}$ we will denote it as
a subindex as in~$\mathscr{P}_{N_{\R}}$ or $\ov{\mathscr{P}}_{N_{\R}}$.
\end{defn}

The following propositions contain the basic properties of the Legendre-Fenchel
duality acting on  $\ov{\mathscr{P}}$.
The elements in $\ov{\mathscr{P}}$ are continuous
functions on polyhedra. In particular, they are closed concave functions. 
Observe that when working with uniform limits of piecewise affine
concave functions,
the technical issues in \S \ref{sec:basic-prop-legendre} disappear. 

\begin{prop}\label{prop:13} The piecewise affine concave functions and
  their uniform limits satisfy the following properties. 
  \begin{enumerate}
  \item \label{item:74} Let $f\in \ov{\mathscr{P}}_{N_{\R}}$. Then
    $f^{\vee\vee}=f$. 
  \item \label{item:73}  If $f\in \mathscr{P}(\Lambda,\Lambda')$
    (respectively $f\in 
    \ov{\mathscr{P}}(\Lambda,\Lambda')$) then $f^{\vee}\in
    \mathscr{P}(\Lambda',\Lambda)$ (respectively $f^{\vee}\in
    \ov{\mathscr{P}}(\Lambda',\Lambda)$).
  \item \label{item:54} If $f \in \ov{\mathscr{P}}(\Lambda) $ then
    $\Dom(\rec(f))=\rec(\Lambda)$.
  \item \label{item:117} Let $f_{i}\in \mathscr{P}(\Lambda_{i}, \Lambda_{i}')$
    (respectively $f_{i}\in \ov{\mathscr{P}}(\Lambda_{i},
    \Lambda_{i}')$), $i=1,2$, with $\Lambda_{1}\cap\Lambda_{2}\not
    =\emptyset$. Then $f_{1}+f_{2}\in
    \mathscr{P}(\Lambda_{1}\cap\Lambda_{2},
    \Lambda_{1}'+\Lambda_{2}')$ (respectively $f_{1}+f_{2}\in \ov{
      \mathscr{P}}(\Lambda_{1}\cap\Lambda_{2},
    \Lambda_{1}'+\Lambda_{2}')$) and
    $(f_{1}+f_{2})^{\vee}=f_{1}^{\vee}\boxplus f_{2}^{\vee}$.
  \item \label{item:120} Let $f_{i}\in \mathscr{P}(\Lambda_{i}, \Lambda_{i}')$
    (respectively $f_{i}\in \ov{\mathscr{P}}(\Lambda_{i},
    \Lambda_{i}')$), $i=1,2$, with $\Lambda_{1}'\cap\Lambda_{2}'\not
    =\emptyset$. Then $f_{1}\boxplus f_{2}\in
    \mathscr{P}(\Lambda_{1}+\Lambda_{2},
    \Lambda_{1}'\cap\Lambda_{2}')$ (respectively $f_{1}+f_{2}\in \ov{
      \mathscr{P}}(\Lambda_{1}+\Lambda_{2},
      \Lambda_{1}'\cap\Lambda_{2}')$) and $(f_{1}\boxplus
    f_{2})^{\vee}=f_{1}^{\vee}+ f_{2}^{\vee}$.
  \item \label{item:123} Let $(f_{i})_{i\ge 1} \subset \ov{\mathscr {P}}$ be a sequence converging 
    uniformly to a function $f$. Then $f\in \ov{\mathscr {P}}$.
  \end{enumerate}
\end{prop}
\begin{proof}
  All the statements follow, either directly from the definition,
  or the propositions~\ref{prop:19} and~\ref{prop:3}.  
\end{proof}

\begin{prop} \label{prop:101} Let $A\colon Q_\R\to N_\R$ be an affine
  map defined as $A=H+u_0$ for a linear map $H$ and a point
  $u_0\in N_{\R}$. Let $f\in \mathscr{P}_{N_{\R}}$ (respectively $f\in
  \ov{\mathscr{P}}_{N_{\R}}$) with $\Dom(f)\cap\Im(A)\not
  =\emptyset$ and $g\in \mathscr{P}_{Q_{\R}}$ (respectively $g\in
  \ov{\mathscr{P}}_{Q_{\R}}$) such that $\Stab(g)\cap
  \Im(H^{\vee})\not = \emptyset$.  Then $A^{\ast}f\in
  {\mathscr{P}}_{Q_{\R}}$ (respectively $A^{*}f\in
  \ov{\mathscr{P}}_{Q_{\R}}$) and $A_{*}g \in \mathscr{P}_{N_{\R}}$
  (respectively $A_{*}g\in \ov{\mathscr{P}}_{N_{\R}}$). Moreover,
  \begin{enumerate}
  \item \label{item:121} $\Stab(A^{\ast}f)=H^{\vee}(\Stab(f))$, $
    (A^{\ast} f)^{\vee}=(H^{\vee})_{\ast}(f^{\vee}-u_{0})$ and, for all
    $y\in \Stab(A^{\ast}f)$,
$$      (A^{\ast} f)^{\vee}(y)=
      \max_{x\in(H^{\vee})^{-1}(y)}(f^{\vee}(x)-\langle x,u_{0}\rangle);$$
    \item \label{item:122} $\Stab(A_{\ast}g)=(H^\vee)^{-1}(\Stab(g))$,
      $ (A _{\ast}g)^{\vee}=(H ^{\vee}) ^{\ast}(g^{\vee})+u_{0}$
      and, for all $u\in \Dom(A_{*}g)$,
  $$A_{*}g(u)= \max_{v\in A^{-1}(u)} g(v).
$$ 
  \end{enumerate}
\end{prop}

\begin{proof}
  These statements follow either from Proposition~\ref{prop:11} or from
  \cite[Corollary~19.3.1]{Roc70}.
\end{proof}

We will be concerned mainly with functions in $\ov{\mathscr{P}}$ whose
effective domain is either a polytope or the whole space
$N_{\R}$. These are the kind of functions that arise when considering
proper toric varieties. The functions in $\mathscr{P}(N_{\R})$ can be
realized as the inverse image of the support function of the standard
simplex, while the functions of $\mathscr{P}(\Delta )$ can be realized
as direct images of the indicator function of the standard simplex.

\begin{lem}\label{lemm:17}
  Let $f\in \mathscr{P}(N_{\R})$ and let $f(u)=\min_{0\le i\le r}(a_{i}(u)+\alpha
  _{i})$ be an H-represen\-ta\-tion of $f$. Write $\boldsymbol{ \alpha}= 
(\alpha _{i}-\alpha _{0})_{i=1,\dots,r}$, and  consider the
  linear map $H\colon N_{\R}\to \R^{r}$ given by
  \begin{math}
    H(u)=(a_{i}(u)-a_{0}(u))_{i=1,\dots,r}
  \end{math}
  and the affine map $A=H+\boldsymbol{ \alpha}.$ Then
  \begin{enumerate}
  \item \label{item:75} $f=A^{\ast}\Psi _{\Delta ^{r}}+a_{0}+\alpha _{0};$
  \item \label{item:76} $f^{\vee}=\tau _{a_{0}}(H^{\vee})_{\ast}(\iota _{\Delta ^{r}}
    -\boldsymbol{ \alpha})-\alpha _{0}.$
  \end{enumerate}
  This second function can be alternatively described as 
 the function which param\-eterizes the upper envelope of
the extended polytope
$$
\Conv((a_{1},-\alpha_{1}),\dots,(a_{l},-\alpha_{l}))\subset M_{\R}\times
\R.
$$
\end{lem}
\begin{proof}
  Statement \eqref{item:75} follows from the explicit description of
  $\Psi _{\Delta ^{r}}$ in Example~\ref{exm:1}. Statement
  \eqref{item:76} follows from Proposition \ref{prop:8}, Proposition
  \ref{prop:101}\eqref{item:121} and Example~\ref{exm:1}. The last
  statement is a consequence of Proposition \ref{prop:19}. 
\end{proof}

The next proposition characterizes the elements of
$\ov{\mathscr{P}}(N_{\R})$ and $\ov{\mathscr{P}}(\Delta )$ for a
polytope $\Delta $.

\begin{prop}\label{prop:16} Let $\Delta $ be a convex polytope of
    $M_{\R}$.
  \begin{enumerate}
  \item \label{item:24}  The space 
    $\ov{\mathscr{P}}(\Delta, N_{\R})$ agrees
    with the space of all continuous 
    concave functions on $\Delta $.
  \item \label{item:25} A concave function $f$ belongs to  $\ov{
    \mathscr{P}}(N_{\R},\Delta )$ if and only if
    $\Dom(f)=N_{\R}$ and $|f-\Psi _{\Delta }|$ is bounded.
  \end{enumerate}
\end{prop}
\begin{proof}
  We start by proving~\eqref{item:24}. By the properties of uniform
  convergence, it is clear that any element of
  $\ov{\mathscr{P}}(\Delta, N_{\R}
    )$ is concave and continuous.  Conversely, a continuous function
  $f$ on $\Delta $ is uniformly continuous because $\Delta $ is
  compact. Therefore, given $\varepsilon >0$ there is a $\delta >0$
  such that $|f(u)-f(v)|<\varepsilon $ for all $u,v\in \Delta$ such
  that $\|u-v\|<\delta $.  By compactness, we can find a triangulation
  $\Delta =\bigcup_i \Delta _{i}$ with $\diam(\Delta _{i})<\delta
  $. Let $\{b_{j}\}_{j}$ be the vertices of this triangulation and
  consider the function $g\in \mathscr{P}(\Delta ,N_{\R})$ defined as
  \begin{displaymath}
    g(u) = \sup\bigg\{\sum_{j=1}^{l}\lambda_{j}f(b _{j})\bigg|\ 
    \lambda_{j}\ge 0, \sum_{j}\lambda_{j}=1, 
    \sum_{j}\lambda_{j}a_{j}=x
    \bigg\}. 
  \end{displaymath}
  For $u\in \Delta $, let $b_{j_{0}},\dots,b_{j_{n}}$ denote the vertices
  of an element of the triangulation containing $u$. We
  write $u=\lambda _{j_{0}}u_{j_{0}}+\dots +\lambda
  _{j_{n}}u_{j_{n}}$ for some $\lambda_{j_i}\ge 0 $ and $\lambda
  _{j_{0}}+\dots +\lambda 
  _{j_{n}}=1$. By concavity, we have
  \begin{displaymath}
    f(u)\ge g(u)\ge \sum _{k=0}^{n}\lambda _{j_{k}}f(u_{j_{k}})\ge f(u)-\varepsilon, 
  \end{displaymath}
  which shows that any continuous function on $\Delta $ can
  be arbitrarily approximated by elements of $\mathscr{P}(\Delta,N_{\R} )$. 

  We now prove~\eqref{item:25}. Let $f\in
  \ov{\mathscr{P}}(N_{\R},\Delta )$. By definition, 
  for each $\varepsilon >0$ we can find a function $g\in
  \mathscr{P}(N_{\R},\Delta )$ with $\sup|f-g|\le \varepsilon $. In
  particular, $|f-g|$ is bounded. Furthermore, $\rec(g)=\Psi _{\Delta
  }$ and $|g-\rec(g)|$ is bounded because $g\in 
  \mathscr{P}(N_{\R})$. Hence $\Dom(f)=
  \Dom(g)=N_{\R}$ and $|f-\Psi _{\Delta }|$ is bounded. 

  Conversely, let $f$ be a concave function such that
  $\Dom(f)=N_{\R}$ and $|f-\Psi _{\Delta }|$ is
  bounded. Then $\Stab(f)=\Stab(\Psi _{\Delta })=
  \Delta$. By \cite[Theorem 12.2]{Roc70} $f^{\vee}$ is a closed
  concave function on $\Delta$. Since $\Delta $ is a polytope, by
  \cite[Theorem 10.2]{Roc70}, $f^{\vee}$ is continuous on $\Delta $.  
  Hence we can apply~\eqref{item:24} to $f^{\vee}$ to obtain functions
  $g_i\in\mathscr{P}(\Delta,N_{\R})$ approaching $f^{\vee}$
  uniformly. By Proposition \ref{prop:3}, we
  conclude that the functions $g_i^{\vee} \in \mathscr{P}(N_{\R},\Delta )$
  approach $f$ uniformly and so $f\in \ov{\mathscr{P}}(N_{\R},\Delta )$.
\end{proof}

\begin{prop}\label{prop:22}
  Let $\Delta $ be a lattice polytope of $M_{\R}$. Then the subset of rational
  piecewise affine concave functions in  $\ov{\mathscr{P}}(\Delta,N_{\R}
  )$ (respectively, in  $\ov{\mathscr{P}}(N_{\R},\Delta )$)
    is dense with respect to uniform convergence. 
\end{prop}
\begin{proof}
  This follows from Proposition \ref{prop:16} and the density of
  rational numbers.
\end{proof}

\section{Differences of concave functions}
\label{sec:diff-conc-funct}

Let $C\subset N_{\R }$ be a convex set.
A function $f\colon C\to \R$ is called a \emph{difference of concave
functions} or a \emph{DC function}
\index{difference of concave functions}%
\index{DC function|see{difference of concave functions}}%
if it can be written as $f=g-h$ for  concave functions $g,h\colon
C\to \R$. DC functions play an important role in non-convex
optimization and have been widely studied, see for
instance~\cite{DCprogramming} and the references therein.
We will be interested in a subclass of DC functions, namely those which
are a difference of uniform limits of piecewise affine concave
  functions.
\index{piecewise affine concave function!difference of uniform limits of}%

\begin{defn}
  \label{def:23}
For a polyhedron $\Lambda$ in $N_{\R}$ we
set
\begin{displaymath}
{\mathscr{D}}(\Lambda)= \{ g-h \mid g,h\in
\mathscr{P}(\Lambda)\}, \quad 
\ov{\mathscr{D}}(\Lambda)= \{ g-h \mid g,h\in \ov{\mathscr{P}}(\Lambda)\}.  
 \nomenclature[aDz91]{${\mathscr{D}(\Lambda)}, {\mathscr{D}(\Lambda)_{\Z}}$}{spaces of differences of piecewise affine concave functions}%
 \nomenclature[aDz91\ov]{$\ov{\mathscr{D}}(\Lambda), \ov{\mathscr{D}}(\Lambda)_{\Z}$}{spaces of differences of uniform limits of piecewise affine concave functions}%
\end{displaymath}
and 
\begin{displaymath}
{\mathscr{D}}(\Lambda)_{\Z}= \{ g-h \mid g,h\in
\mathscr{P}(\Lambda)_{\Z}\}, \quad 
\ov{\mathscr{D}}(\Lambda)_{\Z}= \{ g-h \mid g,h\in \ov{\mathscr{P}}(\Lambda)_{\Z}\}.  
\end{displaymath}
\end{defn}

These spaces are closed under the operations of taking finite linear
combinations, upper envelope  and lower
envelope.

\begin{prop}
  \label{prop:41}
Let $\Lambda$ be a polyhedron in $N_{\R}$ and
$f_{1},\dots, f_{l}$ functions in ${\mathscr{D}}(\Lambda)$
(respectively in $\ov{\mathscr{D}}(\Lambda)$,  ${\mathscr{D}}(\Lambda)_{\Z}$
or $\ov{\mathscr{D}}(\Lambda)_{\Z}$). Then the functions
\begin{enumerate}
\item \label{item:39} $\sum_{i}\lambda_{i}f_{i}$ for any
  $\lambda_{i}\in \R$ (respectively $\lambda_{i}\in \Z$ for ${\mathscr{D}}(\Lambda)_{\Z}$
or $\ov{\mathscr{D}}(\Lambda)_{\Z}$), 
\item \label{item:40} $\max_{i} f_{i}$, $\min_{i} f_{i}$
\end{enumerate}
are also in ${\mathscr{D}}(\Lambda)$ (respectively in $\ov{\mathscr{D}}(\Lambda)$,
 ${\mathscr{D}}(\Lambda)_{\Z}$ or $\ov{\mathscr{D}}(\Lambda)_{\Z}$).
\end{prop}

\begin{proof}
Statement  \eqref{item:39} is obvious. For the statement~\eqref{item:40}, write
  $f_{i}=g_{i}-h_{i}$ with $g_{i},h_{i}$ in $\mathscr{P}(\Lambda)$
  (respectively, in $\ov{\mathscr{P}}(\Lambda)$).  Then the upper
  envelope admits the DC decomposition $\max_{i} f_{i} =g-h$ with
  $$
  g:= \sum_j g_{j}, \quad 
  h:=\min_i \bigg( h_{i}+ \sum_{j\ne i} g_{j}\bigg),  
  $$
  which are both concave functions in $\mathscr{P}(\Lambda)$
  (respectively, in $\ov{\mathscr{P}}(\Lambda)$).  This shows that
  $\max_{i} f_{i}$ is in $\mathscr{D}(\Lambda)$ (respectively, in
  $\ov{\mathscr{D}}(\Lambda)$).  The statement for the lower
  envelope follows similarly.
\end{proof}

In particular, if $f$ lies in $\mathscr{D}(\Lambda)$
or in $\ov{\mathscr{D}}(\Lambda)$, the same holds for the functions
$|f|$, $\max (f,0)$ and $\min(f,0)$.

\begin{cor}
  \label{cor:6}
The space $ \mathscr{D}(\Lambda)$ coincides with the space of
piecewise affine functions on $\Lambda$.
\end{cor}

\begin{proof}
  This follows from the  max-min representation of
  piecewise affine functions in~\cite{Ovchinnikov:maxmin} and
  Proposition~\ref{prop:41}\eqref{item:40}.
\end{proof}

Some constructions for concave functions can be extended to this kind
of functions. In
particular, we can define the recession of a function in~$
\ov{\mathscr{D}}(\Lambda)$.

\begin{defn}
  \label{def:24}
Let $\Lambda$ be a polyhedron in $N_{\R}$ and $f\in\ov{\mathscr{D}}(\Lambda)$. 
The \emph{recession function} of $f$
\index{recession function!of a difference of concave functions}%
\nomenclature[arec]{$\rec(f)$}{recession of a difference of concave functions}%
is defined as 
\begin{equation}\label{eq:50}
  \rec(f)\colon \rec(\Lambda) \longrightarrow \R, \quad 
u\longmapsto 
\lim_{\lambda\to\infty} \frac{f(v_{0}+\lambda u)-f(v_{0})}{\lambda}
\end{equation}
for any $v_{0}\in \Lambda$.
\end{defn}

Write $f=g-h $ for any $g,h\in \ov{\mathscr{P}}(\Lambda )$.  By
Proposition \ref{prop:13}\eqref{item:54}, the effective domain of both
$\rec(g)$ and $\rec(h)$ is $\rec(\Lambda)$. Therefore,
by~\eqref{eq:49}, for all $u\in \rec(\Lambda)$, the
limit~\eqref{eq:50} exists and
$$
\rec(f)(u)=\rec(g)(u)-\rec(h)(u).
$$
Observe that the recession function of a function in $
\mathscr{D}(\Lambda)$
is a piecewise linear function on a subdivision of the 
cone $\rec(\Lambda)$ into polyhedral cones. 
Observe also that 
$$
|f-\rec(f)|\le |g-\rec(g)|+|h-\rec(h)| = O(1).
$$  
We will be mostly interested in the case when $\Lambda=N_{\R}$. 

\begin{prop}
  \label{prop:42}
Let $\Vert\cdot\Vert $ be any metric on $N_{\R}$ and $f\in
\ov{\mathscr{D}}(N_{\R})$. Then there exists a constant $\kappa >0$ such that, for all $u,v\in N_{\R}$,
\begin{displaymath}
  |f(u)-f(v)| \le \kappa \Vert u-v\Vert.
\end{displaymath}
\end{prop}

A function which verifies the conclusion of this proposition is called
\emph{Lipchitzian}.
\index{Lipchitzian function}%

\begin{proof}
  Let $f=g-h $ with $g,h\in \ov{\mathscr{P}}(N_{\R})$. 
The effective domain of the recessions of $g$ and of $h$ is the whole
of $N_{\R}$. 
By~\cite[Theorem~10.5]{Roc70}, both $g$ and $h$ are 
Lipschitzian, hence so is $f$. 
\end{proof}

Observe that   $\ov {\mathscr{D}}(N_{\R})$ is \emph{not} the
completion of  $ {\mathscr{D}}(N_{\R})$ with respect to uniform
convergence and neither $\ov {\mathscr{D}}(N_{\R})_{\Z}$ is the
completion of  $ {\mathscr{D}}(N_{\R})_{\Z}$. It is easy to construct functions  which are 
uniform limits of piecewise affine ones but do not verify the
Lipschitz condition. 

We will consider the integral and rational structures on the space of
piecewise affine functions. We will use the notation previous to
Definition \ref{def:13}.

\begin{defn}\label{def:31}
  Let $\Lambda $ be a polyhedron and $f\in
  {\mathscr{D}}(\Lambda )$. We say that $f$ is an \emph{H-lattice}
\index{piecewise affine function!H-lattice}%
  (respectively \emph{V-lattice}) function
\index{piecewise affine function!V-lattice}%
if it can be written as the difference of two H-lattice (respectively
  V-lattice) concave functions with effective domain $\Lambda $. 
We say that $f$ is a \emph{rational
    piecewise affine} function
  \index{piecewise affine function!rational}%
  if it is the difference of two rational piecewise affine concave
  functions with effective domain $\Lambda $.
\end{defn}

\begin{prop}\label{prop:9}
  If $f$ is an H-lattice function (respectively a
  rational piecewise
  affine function) on a polyhedron $\Lambda \subset N_{\R}$, then
  there is a
  polyhedral complex
  $\Pi $ in $N_{\R}$, with $|\Pi |=\Lambda $, such that, for every
  $\Lambda' \in \Pi $, 
  \begin{displaymath}
    f|_{\Lambda' }(u)=\langle m_{\Lambda' },u\rangle + l_{\Lambda' },
  \end{displaymath}
  with $(m_{\Lambda' },l_{\Lambda' })\in M\times \Z$
  (respectively $(m_{\Lambda' },l_{\Lambda' })\in M_{\Q}\times
  \Q$). Conversely, every piecewise affine 
  function on $\Lambda $ such that its defining functions have
  integral (respectively rational) coefficients, is an H-lattice function
  (respectively a rational piecewise affine function).
\end{prop}
\begin{proof}
  We will prove the statement for H-lattice functions. The statement for
  rational piecewise affine functions is proved with the same
  argument. If $f$ is an H-lattice function, we can write
  $f=g-h$, where $g$ and $h$ are H-lattice concave functions. We
  obtain $\Pi $ as any common refinement of $\Pi (g)$ and $\Pi 
  (h)$ to a polyhedral complex. Then the statement follows from the
  definition of 
  H-lattice concave functions.

  We also prove the converse only for H-lattice functions. Let $g_{i}$,
  $i=1,\dots,n$, be the set of $H$-lattice distinct defining functions
  of $f$. By \cite[Theorem 2.1]{Ovchinnikov:maxmin} there is a family
  $\{S_{j}\}_{j\in J}$ of subsets of $\{1,\dots,n\}$ such that, for all 
  $x\in \Lambda$,  
  \begin{displaymath}
    f(x)=\max_{j\in J} \min_{i\in S_{j}}g_{i}(x).
  \end{displaymath}
  For $j\in J$, write $f_{j}=\min_{i\in S_{j}}g_{i}$. It is an
  H-lattice concave function. Then we can write
  \begin{displaymath}
    f(x)=\sum_{j\in J}f_{j}(x)-\left(\sum_{j\in J}f_{j}(x)-\max_{j\in J}
    f_{j}(x)\right).  
  \end{displaymath}
  Since both $\sum_{j\in J}f_{j}$ and $\sum_{j\in J}f_{j}-\max_{j\in J}
    f_{j}=\min_{j\in J}\sum_{i\in J\setminus \{j\}}f_{i}$ are
    H-lattice concave functions on $\Lambda $, we conclude that $f$
    is an H-lattice function. 
\end{proof}

\begin{defn}\label{def:33}
  Let $f$ be a rational piecewise
  affine function on $N_{\R}$, and let $\Pi $ and $\{(m_{\Lambda
  },l_{\Lambda })\}_{\Lambda \in \Pi }$ be as in Proposition
  \ref{prop:9}. 
  The family $\{(m_{\Lambda },l_{\Lambda })\}_{\Lambda \in \Pi }$ is
  called a set of \emph{defining vectors} of $f$.
  \index{defining vectors!of a piecewise affine function}%
\end{defn}

\begin{prop}\label{prop:43}
  Let $\Pi $ be a complete SCR polyhedral
  complex in $N_{\R}$ and  $f$  an H-lattice function
  on $\Pi $. Then $\rec(f)$ 
  is a conic H-lattice function on the fan $\rec(\Pi )$.
\end{prop}
\begin{proof}

  Let $\Lambda \in \Pi $ and $(m,l)\in M\times \Z$ such that $f
  (u)=\langle m,u\rangle+l $ for $u\in \Lambda$. Then, by the
  definition of $\rec(f)$, it is clear that $ \rec(f)|_{\rec(\Lambda
    )}(u)=\langle m,u\rangle$. Hence, $\rec(f)$ is a conic H-lattice
  function on $\rec(\Pi )$.
\end{proof}

\section{Monge-Amp\`ere measures}
\label{sec:monge-ampere}

Let  $f\colon C\to \R$ be a concave function 
of class~$\cC^{2}$ on an open convex set $C\subset \R^{n}$.
Its {Hessian matrix}\index{Hessian matrix} 
\begin{displaymath}
  \Hess(f)(u):=\left(\frac{\partial^{2} f}{\partial u_{i}\partial
      u_{j} }(u)\right)_{1\le i,j\le n} 
\end{displaymath}
is a negative semi-definite matrix which quantifies the curvature of
$f$ at the point $u$. The {real Monge-Amp\`ere operator}
is defined as
$(-1)^{n}$ times the determinant of this matrix. This notion can be
extended as a measure to the case of an arbitrary concave function.  A
good reference for Monge-Amp\`ere measures
is~\cite{MR0454331}. 

Let $\mu $ be a Haar measure
\nomenclature[g12]{$\mu $}{Haar measure on $M_{\R}$}%
 of $M_{\R}$. Assume that we choose linear
coordinates $(x_{1},\dots,x_{n})$ of $M_{\R}$ such that $\mu $ is
the measure associated to the differential form $\omega =\dd
x_{1}\land \dots \land \dd x_{n}$ and the orientation of $M_{\R}$
defined by this system of coordinates. Let $(u_{1},\dots,u_{n})$ be the dual
coordinates of $N_{\R}$. We will use the induced orientation to
identify a top differential form with a signed measure.

\begin{defn} \label{def:5}
Let $f$ be a closed concave function on $N_{\R}$. The \emph{real Monge-Amp\`ere
  measure} of $f$ with respect to $\mu $
\index{Monge-Amp\`ere measure}%
is defined, for a Borel subset $E$ of $N_{\R}$, as
\begin{displaymath}
  \cM_{\mu }(f)(E)=\mu (\partial f (E)).
\end{displaymath}
\nomenclature[aMzA1]{$\cM_{\mu }(f)$}{Monge-Amp\`ere measure
  associated to a concave function and a measure}%
It is a measure with support contained in $\Dom(\partial f)$. 
The correspondence $f\mapsto   \cM_{\mu }(f)$ is called the \emph{Monge-Amp\`ere operator}.
\index{Monge-Amp\`ere operator}%

When
the measure $\mu $ is clear from the context, we will drop it from the
notation. 
Moreover, since we are not going to consider complex Monge-Amp\`ere
measures, we will simply call $\cM_{\mu}(f)$ the Monge-Amp\`ere
measure of $f$. 
\end{defn}

By \eqref{eq:634}, the total mass of the Monge-Amp\`ere measure is
given by
\begin{equation}
  \label{eq:13}
  \cM_{\mu}(f)(N_{\R})=\mu(\Stab(f)).
\end{equation}
In particular, when $\Stab(f)$ is bounded,
$\cM_{\mu}(f)$ is a finite measure.
 
\begin{prop} \label{prop:87}
The {Monge-Amp\`ere measure} is a continuous map from the space
of concave functions with the topology defined by uniform
convergence on compact sets to the space of $\sigma$-finite measures
on $N_{\R}$ with
the weak topology. 
\end{prop}

\begin{proof}
This is proved in~\cite[\S3]{MR0454331}. 
\end{proof}

The two basic examples of Monge-Amp\`ere measures that we are
interested in, are the ones associated to smooth functions and the ones
associated to piecewise linear functions. 

\begin{prop} \label{prop:5}
Let $C$ be an open convex set in $N_{\R}$ and $f \in \cC^{2}(C)$ a
concave function. Then
$$
\cM_{\mu}(f)=(-1)^{n}\det(\Hess(f))\dd
u_{1}\land \dots \land \dd u_{n},
$$
where the Hessian matrix is calculated with respect to the coordinates
$(u_{1},\dots,u_{n})$.   
\end{prop}
\begin{proof}
  This is \cite[Proposition~3.4]{MR0454331}
\end{proof}

By contrast, the Monge-Amp\`ere measure of a piecewise affine
concave function is a discrete measure supported on the 
vertices of a polyhedral complex.
 
\begin{prop} \label{prop:32}
Let $f$ be a piecewise affine concave function with $\Dom(f)=N_{\R}$ and $(\Pi
(f),\Pi (f^{\vee}))$ the dual pair of polyhedral complexes associated to
$f$. Denote 
by $\Lambda \mapsto \Lambda ^{\ast}$ the correspondence
$\mathcal{L}f$. Then  
\begin{displaymath}
  \cM_{\mu }(f)=\sum_{v\in \Pi(f)^{0}} \mu(\partial f(v)) \delta_{v}= 
  \sum_{v\in \Pi(f)^{0}} \mu(v^{\ast}) \delta_{v}=\sum_{\Lambda \in
    \Pi(f^{\vee})^{n}} \mu(\Lambda ) \delta_{\Lambda ^{\ast}}, 
\end{displaymath}
where $\delta _{v}$ is the Dirac measure supported on $v$.
\end{prop}

\begin{proof}
  This follows easily from the definition of $\cM(f)$ and the
  properties of the Legendre correspondence of piecewise affine functions. 
\end{proof}

\begin{exmpl}
  \label{exm:30}
Let $\Delta\subset M_{\R}$ be a polytope and $\Psi_{\Delta}$
its support function. Since $\Pi (\Psi _{\Delta })$ is a fan, it has
the origin as its single vertex. Moreover, $0^{\ast}=\Delta
$. Therefore,  
\begin{displaymath}
  \cM_{\mu }(\Psi_{\Delta})=\mu(\Delta) \delta_{0}.
\end{displaymath}
\end{exmpl}

The following relation between Monge-Amp\`ere measure and
Legendre-Fenchel duality is one of the key ingredients in the computation of
the height of a toric variety. We will consider the
$(n-1)$-differential form on $N_{\R}$
\begin{displaymath}
  \lambda =\sum_{i=1}^{n}(-1)^{i-1}x_{i}\dd x_{1}\land\dots\land
  \widehat {\dd x_{i}}\land \dots \land \dd x_{n}.
\end{displaymath}
It satisfies $\dd \lambda =n\omega $.

Let $D\subset N_{\R}$ be a compact convex set and set $\partial
D=D\setminus \ri(D)$ for its relative boundary. If the interior of $D$
is non-empty, by using \cite[theorems 25.5 and 10.4]{Roc70} one can
show that $D$ has piecewise smooth boundary in the sense of
\cite[Definition 7.2.17]{Mardsen:mtaa}. Therefore, for any continuous
function $g$ on $D$ the integral
\begin{displaymath}
    \int_{\partial D } g \,\lambda
\end{displaymath}
is well-defined. If the interior of $D$ is empty, then we define this
integral as zero.  

\begin{thm}\label{thm:20}
  Let $f\colon N_{\R}\to \R$ be a concave function such that
  $D=\Stab(f)$ is compact. Then
  \begin{equation}
    \label{eq:63}
    -\int_{N_{\R}}f\dd\cM_{\mu }(f)=
    (n+1)\int_{D }f^{\vee} \dd \mu -
    \int_{\partial D } f^{\vee} \lambda.
  \end{equation}
\end{thm}

To prove this result,  we will use the following lemma  to reduce to
the case of strictly concave smooth functions.

  \begin{lem}\label{lemm:20} Let $f\colon N_{\R}\to \R$ be a concave
    function such that $\Stab(f)$ is bounded and has non-empty
    interior. Then there is a sequence of strictly concave smooth
    functions $(f_{l})_{l\ge 1}$ that converges to $f$ uniformly in
    $N_{\R}$. 
  \end{lem}

  \begin{proof}
    Let $\|\cdot\|$ be the Euclidean norm on $N_{\R}$ induced by the
    choice of linear coordinates. By \cite[Corollary
    13.3.3]{Roc70}, the hypothesis that
    $\Stab(f)$ is bounded implies that there is a constant $\kappa>0$
    such that, for all $x,y\in N_{\R}$,
    \begin{equation}\label{eq:8}
    |f(x)-f(y)|\le \kappa \|x-y\|.
    \end{equation}
    For $l\ge 1$, consider the Gaussian function
    \begin{displaymath}
      \rho _{l}(x)=\frac{l^{n}}{(2\pi)^{n/2}}\e^{\frac{-l^{2}\|x\|^{2}}{2}}.
    \end{displaymath}
    We define
    \begin{displaymath}
      f_{l}(x)=\int_{N_{\R}}\rho _{l}(x-y)f(y)\dd \mu (y).
    \end{displaymath}

    The fact that $\rho_{l}$ is smooth implies that $f_{l}$ is smooth
    too, and the facts that $\Stab(f)$ has non-empty interior and that
    $\rho_{l}$ is strictly positive on the whole of $N_{\R}$ imply
    that $f_{l}$ is strictly concave. The equation (\ref{eq:8}) implies
    that the sequence $(f_{k})_{k\in \N}$ converges uniformly
    to~$f$.
  \end{proof}

\begin{proof}[{Proof of Theorem \ref{thm:20}}]
  If the interior of $D$ is empty, then both sides of the equation
  \eqref{eq:63} are zero. Therefore, the theorem is trivially true in
  this case. Thus, we may assume that~$D $ has non-empty
  interior. 

  Since $\Stab(f)$ is compact, the right-hand side of \eqref{eq:63} is
  continuous with respect to uniform 
  convergence of functions, thanks to Proposition~\ref{prop:3}.  
  Moreover, Proposition \ref{prop:87} and the fact that 
  $\cM_{\mu }(f)$ is finite imply that the left-hand side is also
  continuous with respect to uniform
  convergence. By Lemma \ref{lemm:20}, we can find  a sequence of strictly
  concave smooth functions $(f_{l})_{l\ge 1}$ that converges uniformly to
  $f$. Hence, we may 
  assume that $f$ is smooth and strictly concave. In this case,
  the Legendre transform $\nabla f\colon N_{\R}\to D ^{\circ}$
  is a diffeomorphism (Theorem \ref{thm:2}). 

  By the definition of the Monge-Amp\`ere measure,
  \begin{equation}\label{eq:67}
    \int_{N_{\R}}f\dd \cM_{\mu }(f)=\int_{D}f((\nabla
    f)^{-1}x)\dd \mu (x),
  \end{equation}
  which, in particular, shows that the integral on the left is
  convergent for smooth strictly concave functions with compact
  stability set. Therefore, it is convergent for any concave function within
  the hypothesis of the theorem.

  By Theorem \ref{thm:2}\eqref{item:51}, 
  \begin{equation}\label{eq:64}
    -f((\nabla f)^{-1}(x))=f^{\vee}(x)-\langle (\nabla f)^{-1}(x),x\rangle.  
  \end{equation}
  Moreover, 
  \begin{align}
    \dd(f^{\vee} \lambda )(x)&=\dd f^{\vee}\land \lambda (x)+
    f^{\vee}\dd \lambda (x)\notag\\ 
    &=\langle \nabla f^{\vee}(x),x\rangle\omega +nf^{\vee}\omega \notag\\
    &=\langle (\nabla f)^{-1}(x),x\rangle\omega +nf^{\vee}\omega, \label{eq:66} 
  \end{align}
  where the last equality follows from Theorem \ref{thm:2}\eqref{item:16}.
  The result is obtained by combining the equations \eqref{eq:67},
  \eqref{eq:64} and \eqref{eq:66} with the piecewise smooth Stokes'
  theorem \cite[Theorem 8.2.20]{Mardsen:mtaa}.
\end{proof}

We now particularize Theorem \ref{thm:20} to the case when the Haar
measure comes from a lattice and the convex set is a lattice
polytope of maximal dimension. 

\begin{defn} \label{def:38}
Let $L$ be a lattice and set $L_{\R}=L\otimes \R$.   
We denote by $\Vol_{L}$ the {Haar measure} on $L_{\R}$
\index{normalized Haar measure}%
\nomenclature[avaolL]{$\Vol_{L}$}{normalized Haar measure}%
normalized so that $L$ has covolume $1$. 
\end{defn}

Let 
$N$ be a lattice of $ N_{\R}$ and set 
$M=N^\vee$ for its dual lattice. For a concave function $f$, we denote by
$\mathcal{M}_{M}(f)$ the 
Monge-Amp\`ere measure with respect to the normalized Haar measure $\Vol_{M}$.
\nomenclature[aMzA2]{$\cM_{M}(f)$}{Monge-Amp\`ere measure associated to a concave function and a lattice}%

\begin{notn} \label{def:79}
Let $\Lambda $ be a rational polyhedron in $M_{\R}$ and $\aff(\Lambda )$
its affine hull. We denote by
$L_{\Lambda }$ the linear subspace of $M_{\R}$ associated to
$\aff(\Lambda )$ and by  $M(\Lambda )$
the induced lattice $M\cap
L_{\Lambda }$. 
\nomenclature[aMl08]{$M(\Lambda )$}{sublattice associated to a
  polyhedron in $M_{\R}$}%
By definition, $\Vol_{M(\Lambda )}$ is a measure on
$L_{\Lambda }$, and we will denote also by $\Vol_{M(\Lambda )}$ the
measure induced on $\aff(\Lambda )$. 
\nomenclature[aLl]{$L_{\Lambda }$}{linear space associated to a polyhedron}%
If $v\in N_{\R}$ is orthogonal to~$L_{\Lambda }$, we define
$\langle \Lambda, v \rangle=\langle x,v\rangle$ for any $x\in
\Lambda $.
Furthermore, when $\dim(\Lambda )=n$ and~$F$ is a facet of $\Lambda $,
we will denote by $v_{F}\in N$ 
\nomenclature[avaalp2]{$v_{F}$}{integral inner orthogonal vector 
  of a facet}%
the vector of minimal length that is orthogonal to $L_{F}$ and
satisfies $\langle F,v_{F}\rangle\le\langle x,v_{F}\rangle$ for
each $x\in \Lambda$. In other words,~$v_{F}$ is the minimal
inner integral orthogonal vector of~$F$ as a facet of $\Lambda $.      
\end{notn}

\begin{cor}\label{cor:7}
  Let $f\colon N_{\R}\to \R$ be a concave function such that $\Delta
  =\Stab(f)$ is a lattice polytope of dimension $n$. Then
  \begin{displaymath}
        -\int_{N_{\R}}f\dd\cM_{M}(f)=
        (n+1)\int_{\Delta }f^{\vee}\dd \Vol_{M}+ \sum_{F}\langle
        F,v_{F}\rangle \int_{F}f^{\vee}\dd\Vol_{M(F)},
  \end{displaymath}
where the sum is over the facets $F$ of $\Delta$.
\end{cor}
\begin{proof}
  We choose $(m_1,\dots,m_n)$ a basis of $M$ such that
  $(m_{2},\dots,m_{n})$ is a basis of $M(F)$ and $m_{1}$ points to the
  exterior direction. Expressing $\lambda $ in this basis we obtain
  \begin{displaymath}
    \lambda |_{F}=-\langle F,v_{F}\rangle \dd \Vol_{M(F)}.
  \end{displaymath}
The result then follows from Theorem \ref{thm:20}.
\end{proof}

In \S \ref{sec:heigh-toric-vari}, we will see that we can express the
height of a toric variety in terms of integrals of the form
$\int_{\Delta}f^{\vee}\dd\Vol_{M}$ 
as in the above result.
In some situations, it will be useful to translate those integrals
to integrals on $N_{\R}$. 

Let $f\colon N_{\R}\to \R$ be a concave function and $g\colon
\Stab(f)\to \R$  an integrable function. We consider the signed
measure on $N_{\R}$ defined, for a Borel subset $E$ of $N_{\R}$,~as
\begin{displaymath}
   \cM_{M,g}(f)(E)= \int_{\partial f(E)} g\dd\Vol_{M}.
\end{displaymath}
Clearly,  $\cM_{M,g}(f)$
is uniformly continuous with respect to $\cM_{M}(f)$. By the
Radon-Nicodym theorem, there is an $\cM_{M}(f)$-measurable
function, that we
denote $g\circ \partial f$, such that \nomenclature[spartiala]{$g\circ \partial f$}{an $\cM(f)$-measurable
  function}%
\begin{equation}\label{eq:121}
  \int_{E} g\circ \partial f\dd\cM_{M}(f) = \int_{E}\dd\cM_{M,g}(f) =
  \int_{\partial f(E)} g\dd\Vol_{M}.
\end{equation}

\begin{exmpl} \label{exm:29} When the function $f$ is differentiable or
  piecewise affine, the measurable function $f^{\vee}\circ \partial f$
  can be made explicit.
  \begin{enumerate}
  \item \label{item:96} Let $f\in \cC^{2}(N_{\R})$. Proposition
    \ref{prop:5} and the change of variables formula imply $
    g\circ \partial f= g\circ \nabla f$. For the particular case when
    $g=f^{\vee}$, Theorem~\ref{thm:2}\eqref{item:51} implies, for
    $u\in N_{\R}$,
    \begin{displaymath}
  f^{\vee}\circ \partial f(u)= \langle
    \nabla f(u),u\rangle -f(u).    
    \end{displaymath}

\item \label{item:97} Let $f$ a piecewise affine concave function on
  $N_{\R}$.
By Proposition \ref{prop:32}, $\cM_{M}(f)$ is supported in the finite set $\Pi(f)^{0}$ and
so is  $\cM_{M,g}(f)$. For $v\in  \Pi(f)^{0}$ write $v^{*}\in
\Pi(f^{\vee})^{n}$ for the dual polyhedron. 
Then $ g\circ \partial f(v)= \frac{1}{\Vol_{M}(v^{*})}\int_{v^{*}}g\dd
\Vol_{M}$, which implies
\begin{displaymath}
  f^{\vee}\circ \partial f(v)=
\frac{1}{\Vol_{M}(v^{*})}\int_{v^{*}} \langle x,v\rangle \dd \Vol_{M}
- f(v).
\end{displaymath}
The function $f^{\vee}\circ \partial f$ is defined as a
$\cM_{M}(f)$-measurable function. Therefore, only its values at the
points $v\in  \Pi(f)^{0}$ are well defined. Nevertheless, we can
extend the function $f^{\vee}\circ \partial f$ to the whole $N_{\R}$
by writing
\begin{displaymath}
    f^{\vee}\circ \partial f(u)=
\frac{1}{\Vol_{\mu }(\partial f(u))}\int_{\partial f(u)} \langle x,u\rangle \dd \mu 
- f(u)
\end{displaymath}
for any Haar measure $\mu $ on the affine space determined by
$\partial f(u)$.
  \end{enumerate}
\end{exmpl}

The Monge-Amp\`ere operator is homogeneous of
degree $n$. There is an associated
multilinear operator, introduced by Passare and
Rullg{\aa}rd~\cite{PassareRullgard2004},  which takes $n$ concave functions
as arguments.

\begin{defn}\label{def:18}
Let $f_{1},\dots,f_{n}$ be closed concave functions on $N_{\R}$. The \emph{mixed
  Monge-Amp\`ere measure}
\index{Monge-Amp\`ere measure!mixed}%
is defined by the formula
\begin{displaymath}
  \cM_{M}(f_{1},\dots, f_{n})=\frac{1}{n!} \sum_{j=1}^n (-1)^{n - j}
  \sum_{1 \le i_1 < \cdots < i_j \le n} 
  \cM_{M}(f_{i_{1}}+\dots+f_{i_{j}}).
\end{displaymath}
\nomenclature[aMzA3]{$\cM_{M}(f_{1},\dots, f_{n})$}{mixed
  Monge-Amp\`ere measure}%
\end{defn}

In principle, the mixed Monge-Amp\`ere measure is a signed
measure. Nevertheless it can be shown that it is a measure (see \cite[\S
5]{PassareRullgard2004}). Moreover, it is symmetric and multilinear
in the concave functions $f_{i}$ with respect to the pointwise addition.

\begin{prop}
  \label{prop:39}
  The {mixed Monge-Amp\`ere measure} is a continuous map from the space
  of $n$-tuples of concave functions with the topology defined by uniform
  convergence on compact sets to the space of $\sigma$-finite measures on $N_{\R}$ with 
  the weak topology.
\end{prop}

\begin{proof}
The general mixed case reduces to the unmixed case
$f_1=\dots=f_{n}$, which is  Proposition~\ref{prop:87}. 
\end{proof}

\begin{defn}
  \label{def:41}
The \emph{mixed volume}
\index{mixed volume of a family of compact convex sets}%
of a family of compact convex sets $Q_1, \dots, Q_n$ of $M_{\R}$ 
is defined as
\begin{displaymath} 
\MV_M(Q_1, \dots, Q_n) = \sum_{j=1}^n (-1)^{n - j} \sum_{1 \le i_1 < \cdots < i_j \le n} \Vol_M(Q_{i_1} + \cdots + Q_{i_j}) 
\end{displaymath} 
\end{defn}
\nomenclature[aMzixed1]{$\MV_M(Q_1, \dots, Q_n)$}{mixed volume of a
  family of convex sets}%
Since $\MV_M(Q, \dots,
Q) = n! \, \Vol_M(Q)$, the mixed volume is a generalization of the
volume of a convex body.  The mixed volume is symmetric and linear
in 
each variable $Q_i$ with respect to the Min\-kows\-ki sum, and
monotone with respect to inclusion~\cite[Chapter~IV]{Ewa96}.

The total mass of the mixed Monge-Amp\`ere measure is given by a mixed
volume.

\begin{prop}\label{prop:98}
Let $f_{1},\dots,f_{n}$ be concave functions such that
$\ri(\Dom(f_{1}))\cap \dots \cap \ri(\Dom(f_{n}))\ne \emptyset$, then
\begin{displaymath}
  \cM_{M}(f_{1},\dots, f_{n})(N_{\R})=\frac{1}{n!}\MV_{M}(\Stab(f_{1}),
  \dots, \Stab(f_{n})).
\end{displaymath}
\end{prop}

\begin{proof}
  If $\Dom(f_{i})=N_{\R}$ for all $i$, this is proved in
  \cite[Proposition 3(iv)]{PassareRullgard2004}. In the general case,
  this follows from the definitions of mixed Monge-Amp\`ere measure
  and mixed volume, the equation (\ref{eq:13}) and Proposition
  \ref{prop:10}\eqref{item:33}. 
\end{proof}

Following \cite{PS08a}, we introduce an extension of the notion of integral of a
concave function. 

\begin{defn}
\label{def:49}
Let $Q_{i}$, $i=0,\dots, n$, be a family of compact convex subsets of
$M_{\R}$ and $ g_{i}\colon
Q_{i}\to \R$ a  concave function on $Q_{i}$.    
The \emph{mixed integral of $g_{0},\dots, g_{n}$}
\index{mixed integral of a family of concave functions}%
is defined as
\begin{equation*}
\MI_M(g_{0},\dots, g_{n}) = \sum_{j=0}^n (-1)^{n - j} \sum_{0 \le i_0
  < \cdots < i_j \le n} \int_{Q_{i_0} + \cdots + Q_{i_j}}
g_{i_0} \boxplus \cdots \boxplus g_{i_j} \dd \Vol_{M}. 
\end{equation*} 
\nomenclature[aMzixed]{$\MI_M(g_{0},\dots, g_{n})$}{mixed integral of
  a family of concave functions}%
\end{defn}

For a compact convex subset $Q\subset M_{\R}$ and a concave function
$g$ on $Q$, we have $\MI_M(g, \dots, g)=
(n+1)! \int_{Q}g  \dd \Vol_{M}$. The mixed integral is symmetric and
additive in each variable $g_{i}$ with respect to the sup-convolution. 
For a scalar $\lambda\in \R_{\ge0}$, we have  $\MI_M(\lambda
g_{0},\dots, \lambda g_{n})= \lambda\MI_M(g_{0},\dots, g_{n})$. 
We refer to \cite{PS08a, MR2419926} for the proofs and more information about this notion.


\chapter{Toric varieties}
\label{sec:toric-varieties}

In this chapter we recall some basic facts about the algebraic
geometry of toric varieties and schemes. In the first place, we consider toric
varieties over a field and then toric schemes over a DVR. We refer
to \cite{Kempfals:te,Oda88,Ful93,Ewa96,CoxLittleSchenck:tv} for more
details.

We will use the notations of the previous section concerning concave
functions and polyhedra, with the proviso that the vector space
$N_{\R}$ will always be equipped with a 
lattice $N$ and most of the objects we consider will be compatible with
this integral structure, even if not said explicitly. In particular,
from now on, by a \emph{fan} (Definition~\ref{def:42})
\index{fan}%
\index{fan!rational}%
we will  
mean a rational fan and by a 
\emph{polytope}
\index{polytope!lattice}%
we will mean a lattice polytope.

\section{Fans and toric varieties}
\label{Toric varieties}

Let $K$ be a field and $\T\simeq \G_{m}^{n}$ a split torus over~$K$. 
We alternatively  denote it by $\T_{K}$ if we want to refer to 
its  field of definition.
\nomenclature[aT]{$\T$}{split algebraic torus}%
\begin{defn}\label{def:16}
  A \emph{toric variety}
\index{toric variety}%
is a normal variety $X$ over $K$ equipped with a dense open embedding
  $\T\hookrightarrow X$ and an action $\mu \colon \T\times X\to X$ 
  that extends the action of $\T$ on itself by
  translations. When we want to stress the torus, we will call~$X$ a
  toric variety \emph{with torus} $\T$.
  \nomenclature[g12]{$\mu $}{action of a torus on a toric variety}%
\end{defn}

Toric varieties can be
described in combinatorial terms as we recall in the sequel.
Let $N =\Hom(\G_{m},\T)\simeq \Z^n$ be the lattice of one-parameter
subgroups of $\T$ and 
$M=N^\vee=\Hom(N,\Z)$ its dual lattice. For a ring $R$ we set $N_R=N\otimes R$ and 
$M_R=M\otimes R$. 
We will use the additive notation for the group operations in $N$ and $M$. There is a
canonical isomorphism $M\simeq \Hom(\T,\G_{m})$ with the group of
characters of $\T$. For $m\in
M$ we will denote by $\chi^{m}$ the corresponding character.  

To a fan $\Sigma$ we associate a
\nomenclature[g1813]{$\Sigma$}{rational fan}%
toric variety $X_\Sigma$ over $K$ by gluing together the affine toric
varieties corresponding to the cones of the fan. 
\nomenclature[aX12]{$X_{\Sigma }$}{toric variety associated to a fan}%
For $\sigma\in
\Sigma$, let $\sigma ^{\vee}$ be the dual cone (Definition
\ref{def:77}) and set 
\nomenclature[aMl04]{$M_\sigma$}{semigroup of $M$ associated to a cone}%
$$
M_\sigma=\sigma^\vee\cap M=\{m\in M\mid \langle m,u\rangle\ge 0,
\ \forall u\in \sigma \} 
$$ 
for the 
saturated semigroup of its lattice points. We consider the 
{semigroup algebra}
\index{semigroup algebra}%
\nomenclature[aKsg1]{$K[M_\sigma]$}{semigroup $K$-algebra of a cone}%
\nomenclature[g22a]{$\chi^m$}{character of $\T$}%
$$
K[M_\sigma]=\Big\{\sum_{m\in M_\sigma} \alpha_m\chi^m\Big| \alpha_m\in
K,  \alpha_m=0  \text{ for almost all }m\Big\}
$$
of formal finite sums of elements of $M_\sigma$, with the natural ring
structure. It is an integrally closed domain of Krull dimension
$n$. We set $X_\sigma=\Spec(K[M_\sigma])$ for the associated 
{affine toric variety}.
\index{toric variety!affine}%
\nomenclature[aX11]{$X_\sigma$}{affine toric variety}%
If $\tau$ is a face of $\sigma$, then
$K[M_\tau]$ is a localization of $K[M_\sigma]$. Hence there is an
inclusion of open sets 
$$
X_\tau=\Spec(K[M_\tau]) \hooklongrightarrow X_\sigma=\Spec(K[M_\sigma]).
$$
For $\sigma, \sigma'\in\Sigma$, the affine toric varieties $X_\sigma$,
$X_{\sigma'}$ glue together through the open subset  $X_{\sigma \cap
  \sigma'}$ 
corresponding to their common face.
Thus these affine varieties  glue together to form the 
{toric variety}\index{toric variety!associated to a fan}
$$
X_\Sigma=\bigcup_{\sigma\in\Sigma}X_\sigma.
$$
 This is a normal variety
over $K$ of dimension $n$. When we need to specify the field of
definition we will denote it as $X_{\Sigma ,K}$.  
We denote by $\mathcal{O}_{X_{\Sigma }}$ 
\nomenclature[aOX1]{$\mathcal{O}_{X}$}{sheaf of algebraic functions of
  a scheme}%
its structural sheaf and by
$\mathcal{K}_{X_{\Sigma }}$ 
\nomenclature[aKX]{$\mathcal{K}_{X}$}{sheaf of rational functions of a
scheme}%
its sheaf of rational functions. The open
subsets $X_{\sigma }\subset X_{\Sigma }$ may be denoted by $X_{\Sigma
  ,\sigma }$ when we want to include the ambient toric variety in the
notation.

The cone $\{0\}$, that we denote simply by $0$, 
is a face of every cone
and its associated affine scheme 
$$
X_0=\Spec(K[M])
$$
is an open subset of all the schemes $X_\sigma$.
This variety is an algebraic group over~$K$ canonically isomorphic to
$\T$. We identify this variety with $\T$ and
call it the \emph{principal open subset} of $X_{\Sigma }$.
\index{toric variety!principal open subset of}%
\nomenclature[aX13]{$X_{\Sigma ,0}$}{principal open subset}%

For each $\sigma\in\Sigma$, the homomorphism 
\begin{displaymath}
  K[M_\sigma]\to K[M]\otimes K[M_\sigma],\quad  \chi^m\mapsto\chi^m\otimes\chi^m 
\end{displaymath}
induces an action of 
$\T$ on $X_\sigma$. This action is compatible with the inclusion
of open sets and so it extends to an action on the whole of $X_\Sigma$ 
$$
\mu\colon \T\times X_\Sigma \longrightarrow X_\Sigma.
$$
Thus we have obtained a toric variety in
the sense of Definition \ref{def:16}. In fact, all toric varieties are
obtained in this way.

\begin{thm} 
\label{thm:10}
The correspondence $\Sigma\mapsto X_{\Sigma}$ is a bijection 
between the set of fans in $N_{\R}$ and the set of
isomorphism classes of toric varieties with torus $\T$.
\end{thm}
\begin{proof} This result is \cite[\S I.2, Theorem 6(i)]{Kempfals:te}.
\end{proof}

For each $\sigma \in \Sigma $, the set of $K$-rational points in
$X_{\sigma }$ can be
identified with 
the set of semigroup homomorphisms from $(M_\sigma,+)$ to the semigroup
$(K,\times):=K^{\times}\cup\{0\}$. That is,
$$
X_\sigma(K)= \Hom_{\sg}(M_\sigma,(K,\times)).
$$
In particular,  the set of $K$-rational points of the algebraic torus
can be written intrinsically as
$$
\T(K)= \Hom_{\sg}(M_0,(K,\times))=\Hom_{\gp}(M,K^\times)\simeq(K^\times)^n.
$$
\nomenclature[aHomsg]{$\Hom_{\sg}$}{semigroup homomorphisms}%
\nomenclature[aHomgp]{$\Hom_{\gp}$}{group homomorphisms}%
Every affine toric variety has a {distinguished rational point}:
\index{toric variety!distinguished point of}%
 we will
denote by
\nomenclature[ax00]{$x_{\sigma }$}{distinguished point of an affine toric variety}%
$x_{\sigma }\in X_{\sigma }(K)=\Hom_{\sg}(M_\sigma,(K,\times))$ the
point given
by the semigroup homomorphism
\begin{displaymath}
 M_{\sigma }\ni m\longmapsto
 \begin{cases}
   1&\text{ if }-m\in M_{\sigma },\\
   0 &\text{ otherwise}.
 \end{cases}
\end{displaymath}
For instance, the point $x_{0}\in X_{0}=\T$ is the unit of $\T$. 

Most algebro-geometric properties of the toric scheme
translate into combinatorial properties of the fan. In particular,
$X_\Sigma$ is proper if and only if the fan is 
\emph{complete}
\index{fan!complete}%
in the sense that $|\Sigma|=N_\R$.
 The variety $X_\Sigma$ is smooth if and
only if every cone $\sigma\in\Sigma$ can be written as
$\sigma=\R_{\ge0}v_1+\cdots+\R_{\ge0}v_k$ with $v_1,\dots,v_k$
which are part of an integral basis of $N$. The variety $X_{\Sigma }$
is projective if and only if the fan $\Sigma $ is complete and regular
(Definition \ref{def:54}).

\begin{exmpl} \label{exm:10}
  Let
  $\Sigma_{\Delta ^{n}} $ be the fan in Example
  \ref{exm:8}. The toric variety
  $X_{\Sigma _{\Delta ^{n}}}$ is the projective space
  $\P^{n}_{K}$.
  \index{projective space!as a toric variety}%
More generally, to a polytope $\Delta \subset
  M_{\R}$ of maximal dimension we can
  associate a complete toric variety $X_{\Sigma _{\Delta }}$, where
  $\Sigma _{\Delta }$ is the fan of Example \ref{exm:19}.
\end{exmpl}

\section{Orbits and equivariant morphisms}
\label{Orbits}
\typeout{Orbits}

The action of the torus induces a decomposition of a toric variety
into disjoint orbits. These orbits are in one to one correspondence
with the cones of the fan.
Let $\sigma \in
\Sigma $ and set
\nomenclature[aNl05]{$N(\sigma)$}{quotient lattice associated to a cone}%
\nomenclature[aMl05]{$M(\sigma)$}{dual sublattice associated to a cone}%
\nomenclature[sdtbot]{$\sigma ^{\bot}$}{orthogonal space of a cone}%
\begin{equation}
  \label{eq:53}
N(\sigma)=N/(N\cap\R\sigma),\quad M(\sigma)= N(\sigma)^\vee=
M\cap \sigma^\bot,
\end{equation}
where $\R\sigma $ is the linear space spanned by $\sigma $ and $\sigma
^{\bot}$ is the orthogonal space to $\sigma $. 
We will denote by $\pi_\sigma\colon N\to N(\sigma)$ the projection of
lattices. By abuse of notation, we will also denote by
$\pi_\sigma\colon N_{\R}\to N(\sigma)_{\R}$ the induced projection of
vector spaces.
\nomenclature[g16001]{$\pi_\sigma$}{projection associated to a cone}%

The orthogonal space $\sigma^\bot$ is the maximal linear space inside
$\sigma^\vee$ and $M(\sigma)$ is the maximal subgroup sitting inside
the semigroup $M_\sigma$.
\index{orbit!in a toric variety}%
Set\nomenclature[aOrbit1]{$O(\sigma)$}{orbit in a toric variety}
$$O(\sigma)=\Spec(K[M(\sigma)]),$$
which is a torus over $K$ of dimension $n-\dim(\sigma)$. 
The surjection of rings
$$
K[M_\sigma]\longrightarrow K[M(\sigma)], \quad \chi^a\longmapsto 
\begin{cases}
    \chi^a & \text{
      if } a\in
    \sigma^\bot , \\ 
    0 & \text{ if } a\notin \sigma^\bot, 
\end{cases}
$$
induces a closed immersion
$O(\sigma)
\hookrightarrow X_\sigma$. In terms of
rational points, the inclusion $O(\sigma)(K)
\hookrightarrow X_\sigma(K)$ sends a group homomorphism $\gamma \colon
M(\sigma )\to K^{\times}$ to the semigroup homomorphism $\wt \gamma
\colon M_{\sigma
}\to (K,\times)$ obtained by extending $\gamma $ by zero. In
particular, the distinguished point $x_{\sigma }\in X_{\sigma }(K)$
belongs to the image of $O(\sigma )(K)$ by the above
inclusion. 
Composing with the open immersion $X_{\sigma }\hookrightarrow
X_{\Sigma }$, we identify $O(\sigma )$  with a locally closed
subvariety of $X_{\Sigma }$. For
instance, the orbit associated to the 
cone $0$ agrees with the principal open subset $X_0$.
In fact, if 
we consider $x_{\sigma }$ as a rational point of $X_{\Sigma
}$, then $O(\sigma )$ agrees with the orbit of $x_{\sigma
}$ by~$\T$. 

\nomenclature[aVo1]{$V(\sigma)$}{closure of an orbit of a toric variety}%
We denote by $V(\sigma)$ the Zariski closure of $O(\sigma)$ with its
induced structure of  closed subvariety of $X_\Sigma$.
The subvariety $V(\sigma )$ has a natural structure of toric
variety. To see it, we consider the fan
on $N(\sigma )_{\R}$
\begin{equation}
  \label{eq:62}
  \Sigma(\sigma):=\{ \pi_\sigma(\tau)| \tau \supset\sigma \}.
\end{equation}
This fan is  called the \emph{star} of $\sigma$ in $\Sigma$.
\index{star!of a cone in a fan}%
\nomenclature[g1815]{$\Sigma(\sigma)$}{star of a cone in a fan}%
  For each $\tau \in \Sigma $ with $\sigma \subset \tau $, set
  $\ov \tau =\pi _{\sigma }(\tau )\in \Sigma (\sigma )$. Then,
  \begin{math}
    M(\sigma )_{\ov \tau }=M(\sigma )\cap M_{\tau }.
  \end{math}
  There is a surjection of rings
  $$
  K[M_\tau ]\longrightarrow K[M(\sigma)_{\ov \tau }], \quad \chi^m\longmapsto 
  \begin{cases}
    \chi^m & \text{
      if } m\in
    \sigma^\bot , \\ 
    0 & \text{ if } m\notin \sigma^\bot, 
  \end{cases}
  $$
  that defines a closed immersion $X_{\ov \tau }\hookrightarrow
  X_{\tau }$. These maps glue together to give a closed immersion
  $\iota_{\sigma }\colon X_{\Sigma (\sigma )}\hookrightarrow X_{\Sigma
  }$.
\nomenclature[g0918]{$\iota_{\sigma }$}{closed immersion of the closure of an orbit into a toric variety}%

\begin{prop}\label{prop:73}
  The closed immersion $\iota _{\sigma }$ induces an isomorphism 
  \begin{math}
    X_{\Sigma (\sigma )}\simeq V(\sigma ).
  \end{math}
\end{prop}
\begin{proof}
 Since the image
  of each $X_{\ov \tau }$ contains $O(\sigma )$ as a dense orbit, we
  deduce the result from the construction of $\iota _{\sigma }$.
\end{proof}

In view of this proposition, we will identify $V(\sigma
)$ with  $X_{\Sigma (\sigma )}$ and consider it as a toric variety.

We now discuss more general equivariant morphisms of toric varieties.
\begin{defn}
  \label{def:25}
Let $\T_{i}\simeq \G_{m}^{n_{i}}$, $i=1,2$, be split tori over
$K$, and $\varrho \colon \T_{1}\to \T_{2}$ a group morphism. Let
$X_{i}$, $i=1,2$, be toric varieties with torus $\T_{i}$. A
morphism $\varphi\colon X_{1}\to X_{2}$ is 
\emph{$\varrho $-equivariant}
\index{equivariant morphism!of toric varieties}%
if the diagram
\begin{displaymath}
  \xymatrix{
    \T_{1}\times X_{1} \ar[r]^{\mu_{1} } \ar[d]_{\varrho \times \varphi}&
    X_{1}\ar[d]^{\varphi}\\
    \T_{2}\times X_{2} \ar[r]^{\mu_{2} } &
    X_{2}
  }
\end{displaymath}
is commutative.
\nomenclature[g1701]{$\varrho $}{morphism of tori}%
A morphism $\varphi\colon X_{1}\to X_{2}$ is 
\emph{$\varrho$-toric}
\index{toric morphism!of toric varieties}%
if its restriction to $\T_{1}$ agrees with $\varrho $.  We say that
$\varphi$ is \emph{equivariant} or \emph{toric} if it is $\varrho $-equivariant
or $\varrho $-toric, respectively, for some $\varrho$.
\end{defn}

Toric morphisms are equivariant. Indeed, a morphism is
toric if and only if it is equivariant and sends the distinguished
point $x_{1,0}\in X_{1}(K)$ to the distinguished point
$x_{2,0}\in X_{2}(K)$.

The inclusion $V(\sigma )\to X_{\Sigma }$ 
is an example of equivariant morphism that is not toric. Moreover,
the underlying morphism of tori
depends on the choice of a section
of the projection $\pi _{\sigma }\colon N\to N(\sigma )$.

A general equivariant morphism is obtained by composing an equivariant
morphism whose image intersects the principal open subset, with the
inclusion of this image as the Zariski closure of an orbit.

Equivariant morphisms whose image intersects the
principal open subset can be characterized in combinatorial terms. Let
$\T_{i}$, $i=1,2$, be split tori over $K$. Put 
  $N_{i}=\Hom(\G_{m},\T_{i})$ and let $\Sigma _{i}$ be fans in
  $N_{i,\R}$.
\nomenclature[aHm]{$H$}{linear map of lattices}%
Let $H\colon N_{1}\to N_{2}$ be a linear map
  such that, for every cone $\sigma_{1}\in \Sigma_{1}$, there exists
  a cone $\sigma_{2}\in \Sigma_{2}$ with $H(\sigma_{1}) \subset
  \sigma_{2}$, and let $p\in X_{\Sigma_{2},0}(K)$ be a rational point.
The linear map induces a group homomorphism
\begin{displaymath}
 \varrho _{H}\colon \T_{1}\to
\T_{2}. 
\end{displaymath}  
\nomenclature[g1702]{$\varrho_{H} $}{morphism of tori induced by $H$}%
Let $\sigma_{i}\in
\Sigma_{i}$, $i=1,2,$ be cones such that
$H(\sigma_{1}) \subset \sigma_{2}$. Let 
$H^{\vee}\colon M_{2}\to M_{1}$ be the map dual to 
$H$. Then there is a homomorphism of semigroups $M_{2,\sigma_{2}}\to
M_{1,\sigma_{1}}$ which we also denote by $H^{\vee}$.
For a monomial 
$\chi ^{m}\in K[M_{2,\sigma _{2}}]$ we denote by $\chi ^{H^{\vee}m}$ its image in
$K[{M_{1,\sigma_{1}}}]$. The assignment $\chi ^{m}\mapsto \chi
^{m}(p)\chi ^{H^{\vee}m}$ 
induces morphisms of algebras
$
  K[M_{2,\sigma _{2}}]\to K[M_{1,\sigma _{1}}]
$
that, in turn, induce morphisms
$$
X_{\sigma_{1}}=\Spec(K[M_{1,\sigma _{1}}]) \longrightarrow 
X_{\sigma_{2}}=\Spec(K[M_{2,\sigma _{2}}]). 
$$
These morphisms are compatible with the restriction to open
subsets, and they glue together into a $\varrho_{H}$-equivariant morphism
\begin{equation}\label{eq:16}
  \varphi_{p,H}\colon X_{\Sigma_1} \longrightarrow X_{\Sigma_{2}}.  
\end{equation}
\nomenclature[g21pH]{$\varphi_{p,H}$}{equivariant morphism of toric varieties}%
\index{equivariant morphism!of toric varieties}%
In case $p=x_{2,0}$, the distinguished
  point on the principal open subset of $X_{\Sigma _{2}}$, this morphism
  is a toric morphism
\index{equivariant morphism!of toric varieties}%
and will be 
  denoted as $\varphi_{H}$ for short.
\nomenclature[g215]{$\varphi_{H}$}{toric morphism of toric varieties}%

\begin{rem} \label{rem:26} The restriction of $\varphi_{p,H}$ to
  the principal open subset
  can be written in coordinates by choosing
  bases of $N_{1}$ and $N_{2}$. Let $n_{i}$  be the rank of
  $N_{i}$. The chosen bases determine  
  isomorphisms $X_{\Sigma_{i} ,0}\simeq \G_{m}^{n_{i}}$, which give
  coordinates $\bfx =(x_{1},\dots,x_{n_{1}})$ 
  and $\bft =(t_{1},\dots,t_{n_{2}})$ for $X_{\Sigma_{1} ,0}$ and $X_{\Sigma_{2}
    ,0}$, respectively. We write the linear map $H$ with respect
  to these basis as a matrix, and we denote its rows by 
  $a_{i}$, $i=1,\dots,n_{2}$. Write  $p=(p_{1},\dots
  ,p_{n_{2}})$.
  In these
  coordinates, the morphism~$\varphi_{p,H}$ is given by
  \begin{displaymath}
    \varphi_{p,H}(\bfx)=
    (p_{1}\bfx^{a_{1}},\dots, p_{n_{2}}\bfx^{a_{n_{2}}}).
  \end{displaymath}  
\end{rem}

\begin{thm}\label{thm:25}
  Let $\T_{i}$, $N_{i}$ and $\Sigma _{i}$,  $i=1,2$, be as
  above. Then the correspondence $(p,H)\mapsto \varphi_{p,H}$ is a
  bijection between
  \begin{enumerate}
  \item the set 
    of pairs $(p,H)$, where $H\colon N_{1}\to N_{2}$ is a linear map
    such that for every cone $\sigma_{1}\in \Sigma_{1}$ there exists
    a cone $\sigma_{2}\in \Sigma_{2}$ with $H(\sigma_{1}) \subset
    \sigma_{2}$, and $p$ is a rational point of $ X_{\Sigma_{2},0}(K)$, 
  \item the set of equivariant
  morphisms $\varphi\colon X_{\Sigma _{1}}\to X_{\Sigma _{2}}$ whose image
  intersects the principal open subset of $X_{\Sigma _{2}}$. 
  \end{enumerate}
\end{thm}
\begin{proof}
  For a point $p\in X_{\Sigma_{2},0}(K)=\T_{2}(K)$, let $t_{p}\colon
  X_{\Sigma _{2}}\to X_{\Sigma _2}$ be the morphism induced by the toric action. Denote by
  $x_{1,0}\in X_{\Sigma _{1}}(K)$ the distinguished point of
  the principal open subset of $X_{\Sigma _{1}}$. The
  correspondence $\varphi\mapsto (t_{\varphi(x_{1,0})}^{-1}\circ
  \varphi,\varphi(x_{1,0}))$ establishes a bijection between the set
  of equivariant morphisms $\varphi\colon X_{\Sigma _{1}}\to X_{\Sigma
    _{2}}$ whose image intersects the principal open subset of $X_{\Sigma
    _{2}}$ and the set of pairs $(\varphi_{0} ,p)$, where $\varphi_{0} \colon
  X_{\Sigma _{1}}\to X_{\Sigma _{2}}$ is a toric morphism and $p\in
  X_{\Sigma_{2},0}(K)$ is a rational point in the principal open
  subset. Then the result follows from \cite[Theorem~1.13]{Oda88}.
\end{proof}

Following \cite[Proposition~1.14]{Oda88}, we now show how to refine the
Stein factorization for an equivariant  
morphism whose image intersects the principal open
subset\index{equivariant morphism!Stein factorization of},  in
terms of combinatorial data. Let $N_{i}$, $\Sigma_{i}$, $H$ and $p$ be as in Theorem \ref{thm:25}.
The  linear map $H$ factorizes as
\begin{displaymath}
  N_{1}\overset{H_{\surj}}{\twoheadlongrightarrow} N_{3}:=H(N_{1})
\overset{H_{\sat}}{\hooklongrightarrow} N_{4}:=\sat(N_{3}) 
\overset{H_{\inj}}{\hooklongrightarrow} N_{2},
\end{displaymath}
where $N_{3}$ is the image of $H$ and $N_{4}$ is the saturation of
$N_{3}$ with respect to $N_{2}$. Clearly $N_{3,\R}=N_{4,\R}$. 
By restriction, the fan $\Sigma_{2}$ induces a fan in
this linear space. We will call this fan either $\Sigma _{3}$ or
$\Sigma _{4}$, depending on the lattice we are considering.
Applying the combinatorial construction of equivariant morphisms, we
obtain the following factorization of $\varphi_{p,H}$:
\begin{equation}\label{eq:139}
  X_{\Sigma _{1}}\overset {\varphi_{H_{\surj}}}{\longrightarrow}
  X_{\Sigma _{3}}\overset {\varphi_{H_{\sat}}}{\longrightarrow}
  X_{\Sigma _{4}}\overset {\varphi_{p,H_{\inj}}}{\longrightarrow}
  X_{\Sigma _{2}}.
\end{equation}
The first morphism has connected fibres, the second morphism is finite
and surjective, and the third morphism is also finite.  Therefore,
$\varphi_{H_{\surj}}$ and $ {\varphi_{p,H_{\inj}}} \circ
\varphi_{H_{\sat}}$ give a Stein factorization of
$\varphi_{p,H}$. Furthermore, by \cite[Corollary~1.16]{Oda88},
\begin{equation}
  \label{eq:54}
  \deg(\varphi_{H_{\sat}})=[N_{4}:N_{3}].
\end{equation}
The morphism $\varphi_{p,H_{\inj}}$ can
be further factorized as a normalization followed by a closed
immersion.  In  what follows, we describe this latter factorization
with independent notations. 

Consider a
saturated sublattice~$Q$ of~$N$,~$\Sigma$ a fan in $N_{\R}$ and $p\in
X_{\Sigma,0}(K)$.  
\nomenclature[aQ]{$Q$}{saturated sublattice of $N$}%
Let $\Sigma _{Q}$ be the induced fan in $Q_{\R}$ and
$\iota\colon Q\hookrightarrow N$ the inclusion of $Q$ into $N$.  
\nomenclature[g09]{$\iota $}{inclusion of a saturated sublattice}%
Then
we have a finite equivariant morphism
\begin{displaymath}
  \varphi_{p,\iota}\colon X_{\Sigma_{Q}}\longrightarrow X_{\Sigma}.
\end{displaymath}
Set $P=Q^\vee=M/Q^\bot$ 
\nomenclature[aP]{$P$}{lattice dual to $Q$}%
and let $\iota^{\vee}\colon M\to P$ be the dual of
$\iota$. 
Let $\sigma\in \Sigma$ and $\sigma'=\sigma\cap Q_{\R} \in\Sigma_Q$.
The natural semigroup homomorphisms $M_{\sigma }\to P_{\sigma '}$ factors as
$$
M_\sigma \twoheadlongrightarrow {M_{Q,\sigma}}:=(M_\sigma+Q^\bot)/Q^\bot
\hooklongrightarrow P_{\sigma'}:=P\cap (\sigma')^\vee. 
$$
The first arrow is the projection and will be denoted as $m\mapsto [m]$,
while the second one is the inclusion of ${M_{Q,\sigma}}$ into its
saturation with respect to $P$. We have a  diagram of $K$-algebra
morphisms
$$ 
K[M_{\sigma }]\twoheadlongrightarrow K[{M_{Q,\sigma}}]\hooklongrightarrow
K[P_{\sigma'}],$$
where the left map is given by $\chi^{m}\mapsto
\chi^{m}(p)\chi^{[m]}$, and the right map is given by
$\chi^{[m]}\mapsto \chi^{\iota^{\vee}m}$. 
Let $Y_{\sigma,Q,p} \simeq \Spec(K[{M_{Q,\sigma}}])$ be the closed subvariety of
$X_\sigma$ given by the left surjection. Then we have
induced maps
$$
X_{\sigma'}\twoheadlongrightarrow Y_{\sigma,Q,p} \hooklongrightarrow X_\sigma. 
$$
These maps are compatible with the restriction to open subsets and so
they glue together into a factorization of $\varphi_{p,\iota}$:
\begin{equation}
  \label{eq:31}
  X_{\Sigma_Q} \twoheadlongrightarrow Y_{\Sigma, Q,p} \hooklongrightarrow X_\Sigma.  
\end{equation}
Denote by $Y_{\Sigma, Q,p,0}$  the orbit of $p$ under the
action of the subtorus of $\T$ determined by $Q$
Then $Y_{\Sigma, Q,p}$ is the closure of $Y_{\Sigma, Q,p,0}$, while the toric
variety $X_{\Sigma_Q}$ is the normalization of $Y_{\Sigma, Q,p}$. 
When $p=x_{0}$, the subvariety $Y_{\Sigma ,Q,p}$ will be denoted by
$Y_{\Sigma ,Q}$ for short.  \nomenclature[aY18Qp]{$Y_{\Sigma
    ,Q,p}$}{translated toric subvariety}%
\nomenclature[aY18Q]{$Y_{\Sigma ,Q}$}{toric subvariety}%

Observe in the previous construction that, when $\sigma =0$, hence $\sigma
'=0$, then $M_{Q,0}=P_{0}$. Therefore the difference between
$X_{\Sigma_Q}$ and $Y_{\Sigma, Q,p}$ is
concentrated in the complement of the principal open subset:

\begin{prop}\label{prop:110} The normalization map $X_{\Sigma_Q}\to
  Y_{\Sigma, Q,p}$ induces an isomorphism $X_{\Sigma _{Q},0}\to Y_{\Sigma, Q,p,0}$.
\end{prop}

\begin{defn}\label{def:74}
  A subvariety $Y$  of $X_{\Sigma}$ will be
  called a \emph{toric subvariety} 
\index{toric subvariety}%
  (respectively, a \emph{translated toric subvariety})
\index{toric subvariety!translated}%
if it is of the form $Y_{\Sigma,Q}$ (respectively,
  $Y_{\Sigma ,Q,p}$) for a saturated sublattice $Q\subset N$ and
  $p\in X_{\Sigma,0}(K)$.
\end{defn}

A translated toric subvariety is not
necessarily a toric variety in the sense of Definition \ref{def:16},
since it may be non-normal.

\begin{exmpl} Let $N=\Z^{2}$, $(a,b)\in N$
  with $\gcd(a,b)=1$ and $\iota \colon Q\hookrightarrow N$ the
  saturated 
  sublattice generated by 
  $(a,b)$. Let $\Sigma $ be the fan in 
  $N_{\R}$ of Example \ref{exm:8}. Then $X_{\Sigma }=\P^{2}$ with
  projective coordinates $(x_{0}:x_{1}:x_{2})$. The fan
  induced in $Q_{\R}$ has three cones: $\Sigma
  _{Q}=\{\R_{\le 0},\{0\},\R_{\ge 0}\}$. Thus $X_{\Sigma
    _{Q}}=\P^{1}$. Let $p=(1:p_{1}:p_{2})$ be a point of $X_{\Sigma,0
  }(K)$. Then
  $\varphi_{p,\iota}((1:t))=(1:p_{1}t^{a}:p_{2}t^{b})$. Therefore, 
  $Y_{\Sigma, Q,p}$ is the curve of equation
  \begin{displaymath}
    p_{2}^{a}x_{0}^{a}x_{1}^{b}- p_{1}^{b}x_{0}^{b}x_{2}^{a}=0.
  \end{displaymath}
In general, this curve is not normal. Hence it is not a toric variety.
\end{exmpl}

We end this section by stating the compatibility between equivariant
morphisms and orbits. 

\begin{prop}\label{prop:112} With the notations of Theorem
  \ref{thm:25}. Let $\sigma _{1}\in \Sigma _{1}$ and let $\sigma
  _{2}\in \Sigma _{2}$ be the unique cone such that $H(\sigma
  _{1})\subset \sigma _{2}$ and $H(\sigma _{1})\cap \ri(\sigma
  _{2})\not = \emptyset$. Let $H'\colon N_{1}(\sigma _{1})\to
  N_{2}(\sigma _{2})$ be the linear map induced by $H$ and let $p'\in
  O(\sigma _{2})=\Spec(K[M_{2}(\sigma _{2})])$ be the point determined
  by the map $K[M_{2}(\sigma _{2})]\to K$, $\chi^{m}\mapsto
  \chi^{m}(p)$, $m\in M_{2}(\sigma _{2})$. Then there is a commutative
  diagram
  \begin{displaymath}
    \xymatrix{
      X_{\Sigma _{1}(\sigma
        _{1})}\ar[r]^{\varphi_{p',H'}}\ar[d]_{\iota _{\sigma _{1}}}
      & X_{\Sigma _{2}(\sigma
        _{2})}\ar[d]^{\iota _{\sigma _{2}}}\\
      X_{\Sigma_{1}}\ar[r]_{\varphi_{p,H}}
      & X_{\Sigma _{2}}.
    }
  \end{displaymath}  
\end{prop}

\section{$\T$-Cartier divisors and toric line bundles}
\label{Equivariant Cartier divisors}
\typeout{Equivariant Cartier divisors}

When studying toric varieties, the objects that admit a combinatorial
description are those that are compatible with the torus action. These
objects are enough for many purposes.  For instance, the divisor class
group of a toric variety is generated by invariant divisors.

 Let
$\pi _{2}\colon \T\times X\to X$ denote the projection to the second factor and
$\mu \colon \T\times X\to X$ the torus action. A Cartier divisor $D$ is invariant 
if and only if  
\nomenclature[aDcart1]{$D$}{Cartier divisor}%
$$\pi _{2}^{\ast}D=\mu^{\ast}D.$$

\begin{defn} \label{def:14}
Let  $X$ be a toric
variety with torus $\T$.
A Cartier divisor on $X$ is called a \emph{$\T$-Cartier
  divisor}
\index{T-Cartier divisor@$\T$-Cartier divisor!on a toric variety}%
if it is 
invariant under the action of $\T$ on $X$. 
\end{defn}

The combinatorial description of $\T$-Cartier divisors is done in
terms of virtual support functions.

\begin{defn}\label{def:50}
Let  $\Sigma$ be a fan in $N_\R$.
A 
function  $\Psi\colon |\Sigma |\to \R$ is called a 
\emph{virtual support function} on  $\Sigma$
\index{virtual support function}%
\index{support function!virtual|see{virtual support function}}%
if it is a conic $H$-lattice function (Definition \ref{def:31}). 
Alternatively, a virtual support function is a 
function  $\Psi\colon |\Sigma |\to \R$ such that, for every cone
$\sigma\in \Sigma$, there exists $m_{\sigma}\in M$ with $\Psi(u) =
\langle m_\sigma, 
u\rangle$ for  all $u\in \sigma$.
A set of functionals $\{m_{\sigma }\}_{\sigma\in\Sigma}$ as above
is called a  
set of \emph{defining vectors} of $\Psi $.
A concave virtual
support function on a complete fan will be called a \emph{support function}.
\index{support function!on a fan}%
\index{defining vectors!of a virtual support function}%
\end{defn}
\nomenclature[g2300]{$\Psi $}{virtual support function}%
\nomenclature[am18]{$m_{\sigma }$}{defining vector of a virtual support function}%

A support function on a complete fan in the sense of the previous definition, is the
support function of a polytope as in Example \ref{exm:7}:
it is the support function of the polytope
\begin{displaymath}
  \Conv(\{m_{\sigma}\}_{\sigma \in \Sigma ^{n}})\subset M_{\R},
\end{displaymath}
where $\Sigma ^{n}$ is the subset of $n$-dimensional cones of $\Sigma $. 

Two vectors $m,m'\in M$
define the same functional on a cone $\sigma$ if and only if $m-m'\in \sigma^\bot$.
Hence, for a given virtual support function 
$\Psi$ on a fan $\Sigma$, each defining vector 
$m_\sigma$ is unique up to the orthogonal space $\sigma^\bot$. In particular, $m_\sigma
\in M$ is uniquely defined for $\sigma\in\Sigma^n$ and, in the other
extreme, $m_0$ can be any point of~$M$.

Let $\{m_\sigma\}_{\sigma\in \Sigma }$ be a set of defining vectors of $\Psi$.  
These vectors have to satisfy the
compatibility condition
\begin{displaymath}
m_\sigma|_{\sigma\cap\sigma'}=m_{\sigma'}|_{\sigma\cap\sigma'} 
\mbox{ for all } \sigma,\sigma'\in\Sigma.
\end{displaymath}
On each open set $X_\sigma $, the vector $m_{\sigma }$ determines a
rational function $\chi^{-m_\sigma }$.  For $\sigma,
\sigma'\in\Sigma$, the above compatibility condition implies that
$\chi^{-m_\sigma}/\chi^{-m_{\sigma'}}$ is a regular function on the overlap
$X_\sigma\cap X_{\sigma'}=X_{\sigma\cap\sigma'}$ and so $\Psi $
determines a Cartier divisor on $X_\Sigma$:
\nomenclature[aDcart2]{$D_{\Psi }$}{$\T$-Cartier divisor on a toric variety}%
\begin{displaymath}
D_\Psi:=\left\{ (X_{\sigma },\chi^{-m_\sigma})\right \}_{\sigma \in \Sigma }.  
\end{displaymath}
This Cartier divisor does not depend on the choice of defining vectors
and it is a $\T$-Cartier divisor. All $\T$-Cartier divisors are
obtained in this way.

\begin{thm}\label{thm:3} Let $\Sigma $ be a fan in $N_{\R}$ and
  $X_{\Sigma }$ the corresponding toric variety. The
  correspondence  $\Psi\mapsto D_{\Psi}$ is a bijection
  between the set of virtual support functions on
  $\Sigma $ and the set of $\T$-Cartier divisors on $X_{\Sigma }$. Two
  Cartier divisors $D_{\Psi _{1}}$ and $D_{\Psi _{2}}$ are rationally equivalent
  if and only if the function $\Psi _{1}-\Psi _{2}$ is linear.
\end{thm}
\begin{proof} This is proved in \cite[\S I.2, Theorem 9]{Kempfals:te}.
\end{proof}

We next recall the relationship between Cartier divisors and line
bundles in the toric case. 

\begin{defn} \label{def:71} Let $X$ be a toric variety and $L$ a line
  bundle on $X$. A \emph{toric structure}
  \index{toric structure on a line bundle}%
on $L$ is the choice of a
  nonzero vector $z$ on the fibre
  $L_{x_{0}}=x_{0}^{\ast}L$ over the distinguished point. A
  \emph{toric line bundle}
\index{toric line bundle!on a toric variety}%
is a pair $(L,z)$, where $L$ is a line bundle
  on $X$ and~$z$ is a toric structure on $L$.
  A rational section $s$ of a toric line bundle
  is a \emph{toric section}
\index{toric line bundle!on a toric variety!toric section of}%
  if it is regular and nowhere 
  vanishing on the principal open subset $X_{0}$ and $s(x_{0})=z$.
  In order not to burden the notation, a toric line bundle will
  generally be denoted by $L$, the vector $z$ being implicit.
\end{defn}

\begin{rem} \label{rem:3} Let $L$ be a toric line bundle and denote by
  $0$ its zero section. Let $V(L)=\bfSpec_{X}(\Sym(L^{\vee}))$ be the
  total space of $L$. Then $\T':=V(L|_{\T})\setminus 0(\T)$ admits a
  unique structure of split torus of dimension $n+1$ characterized by
  the properties
  \begin{enumerate}
  \item $z$ is the unit of $\T'$;
  \item the projection $\T'\to \T$ is a morphism of algebraic groups;
  \item every toric section $s$ induces a morphism of algebraic groups
    $\T\to \T'$.
  \end{enumerate}
  The terminology ``toric structure'', ``toric line bundle'' and ``toric section''
    comes from the fact that $V(L)$ 
     admits a unique structure of
    toric variety with torus $\T'$ satisfying the conditions:
    \begin{enumerate}
    \item $z$ is the distinguished point of the
    principal open subset;
  \item the structural morphism $V(L)\to X$ is a toric
    morphism;
  \item for each point $x\in X$ and vector $w\in L_{x}$, the morphism $\G_{m}\to V(L)$, given
    by scalar multiplication $\lambda \mapsto \lambda w$, is equivariant;
  \item every toric section $s$ 
    determines a toric morphism $U\to V(L)$, where $U$ is the 
    invariant open subset of regular points of $s$.
    \end{enumerate}
This can be shown using the construction of $V(L)$ as a toric
variety in \cite[Proposition~2.1]{Oda88}.
  \end{rem}

\begin{rem}
Every toric line bundle equipped with a toric section admits a
  unique structure of $\T$-equivariant line bundle such that the toric section
  becomes an invariant section. Conversely, every $\T$-equivariant
  toric line bundle 
  admits a unique invariant toric section. Thus, there is a natural
  bijection between the space of $\T$-equivariant toric line bundles
  and the space of toric line bundles with a toric section. In
  particular, every line bundle admits a structure of $\T$-equivariant line
  bundle. This is not the case for higher rank vector
  bundles on toric varieties, nor for line bundles on other spaces with group
  actions like, for instance, elliptic curves. 
\end{rem}

To a  Cartier divisor $D$, one associates an invertible
sheaf of fractional ideals of~$\mathcal{K}_{X}$, denoted 
$\mathcal{O}(D)$. When $D$ is a $\T$-Cartier divisor given by a set
of defining vectors $\{m_\sigma\}_{\sigma\in \Sigma }$, the sheaf 
$\mathcal{O}(D)$ 
\nomenclature[aOX4]{$\mathcal{O}(D)$}{line bundle associated to a Cartier divisor}%
can be realized as the subsheaf of
$\mathcal{O}_{X}$-modules generated, in each open subset
$X_{\sigma }$, by the rational function $\chi^{m_{\sigma }}$. The
section $1\in \mathcal{K}_{X}$ provides us with a distinguished
rational 
section $s_{D}$ such that $\div(s_{D})=D$. Since $D$ is supported on
the complement of the principal open subset, $s_{D}$ is regular and
nowhere vanishing on $X_{0}$. We set
$z=s_{D}(x_{0})$. This is a toric structure on $\mathcal{O}(D)$. From
now on, we will assume that $\mathcal{O}(D)$ is equipped with this toric
structure. Then $((\mathcal{O}(D),z),s_{D})$ is a 
toric line bundle with a toric section.

\begin{thm}\label{thm:27}
  Let $X$ be a toric variety with torus $\T$. Then
  the
  correspondence $D\mapsto ((\mathcal{O}(D),s_{D}(x_{0})),s_{D})$
  determines a bijection between
  the sets of
  \begin{enumerate}
  \item $\T$-Cartier divisors on $X$,
  \item isomorphism classes of pairs $(L,s)$ where $L$ is a toric line
    bundle and $s$ is a toric section.
  \end{enumerate}
\end{thm}
\begin{proof}
  We have already shown that a $\T$-Cartier divisor produces a toric line
  bundle with a toric section. Let now $((L,z),s)$ be a toric line bundle
  equipped with a toric
  section and $\Sigma $ the fan that defines $X$. Since every line
  bundle on an affine toric variety 
  is trivial, for each $\sigma \in \Sigma $ we can find a section $s_\sigma $
  that generates $L$ on $X_{\sigma }$ and
  such that $s_{\sigma }(x_{0})=z$. Since $s$ is regular and
  nowhere vanishing on $X_{0}$ and $s(x_{0})=z$, we
  can find elements $m_{\sigma }\in M$ such 
  that $s=\chi^{-m_{\sigma }}s_{\sigma } $, because any
  regular nowhere vanishing function on a torus is a constant times a
  monomial. The elements
  $m_{\sigma }$ glue together to define a virtual support function
  $\Psi$ on $\Sigma $ that does not depend on the chosen
  trivialization. It is easy to see that the correspondence 
  $(L,s)\mapsto D_{\Psi }$ is the inverse of the previous one, which
  proves the theorem.     
\end{proof}

Thanks to this result and Theorem \ref{thm:3}, we can freely move
  between the languages of virtual support functions, $\T$-Cartier
  divisors, and toric line 
  bundles with a toric section.

\begin{notn}\label{def:73}
  Let $\Psi $ be a virtual support function, we write
  $((L_{\Psi },z_{\Psi }),s_{\Psi })$ for the toric line bundle with
  toric section 
  associated to the $\T$-Cartier divisor $D_{\Psi }$ by Theorem
  \ref{thm:27}. When we do not need to make explicit the vector $z_{\Psi }$,
  we will simply write $(L_{\Psi },s_{\Psi })$.
  \nomenclature[aL23]{$L_{\Psi }$}{toric line bundle}%
  \nomenclature[asa23]{$s_{\Psi }$}{toric section}%
  \nomenclature[az23]{$z_{\Psi }$}{toric structure}%
  Conversely, given a toric line bundle $L$ with toric section $s$ we will
  denote $\Psi_{L,s}$ the corresponding virtual support function.
\end{notn}

We next recall the relationship between Cartier divisors and Weil
divisors in the toric case.
\begin{defn}
  \label{def:21}
A \emph{$\T$-Weil divisor} on a toric variety  $X$
\index{T-Weil divisor@$\T$-Weil divisor}%
is a finite formal linear combination of hypersurfaces of $X$ which are
invariant under the torus action.
\end{defn}

The invariant hypersurfaces of a toric variety are
particular cases of the toric subvarieties 
considered in the previous section: they are the varieties of the form
$V(\tau)$ for $\tau\in \Sigma^{1}$ a ray.
Hence, a $\T$-Weil divisor is a finite formal linear combination
of subvarieties of the form $V(\tau)$ for $\tau\in \Sigma^{1}$.

There is a correspondence that to each Cartier divisor on $X$ associates a
Weil divisor. To the 
$\T$-Cartier divisor $D_{\Psi}$, it corresponds the $\T$-Weil divisor
\begin{equation}
  \label{eq:45}
   [D_\Psi]=\sum_{\tau\in \Sigma^1} -\Psi(v_\tau) V(\tau),
\end{equation}
where $v_\tau\in N$ is the smallest nonzero lattice point in $\tau$.
\nomenclature[sbracket]{$[D]$}{Weil divisor associated to a Cartier
  divisor}%
\nomenclature[avaaalp1]{$v_{\tau }$}{smallest nonzero lattice point in a
  ray}%

\begin{exmpl}\label{exm:9} We continue with the notation of examples
  \ref{exm:13} and \ref{exm:10}. The fan~$\Sigma _{\Delta ^{n}}$ has
  $n+1$ rays.  For each $i=0,\dots,n$, the closure of the orbit
  corresponding to the ray generated by the vector $e_{i}$ is the
standard  hyperplane of $\P^{n}$ 
  \nomenclature[aHo]{$H_{i}$}{standard hyperplane of $\P^{n}$}%
  $$
  H_{i}:=V(\langle e_{i}\rangle)=\{(p_{0}:\dots:p_{n})\in \P^{n}\mid p_{i}=0\}.
  $$  
  The function $\Psi _{\Delta ^{n}}$ is a support function on 
  $\Sigma_{\Delta ^{n}}$ and the $\T$-Weil divisor associated to 
  $D_{\Psi _{\Delta ^{n}}}$ is $[D_{\Psi _{\Delta ^{n}}}]=H_{0}$.
\end{exmpl}

For a toric variety $X_{\Sigma }$ of dimension $n$, we denote by
$\Div_{\T}(X_{\Sigma })$ its group of $\T$-Cartier divisors, and by
$Z_{n-1}^{\T}(X_{\Sigma })$ its group of $\T$-Weil divisors.  
\nomenclature[aDivT]{$\Div_{\T}(X_{\Sigma })$}{group of $\T$-Cartier divisors}%
\nomenclature[aZT]{$Z_{n-1}^{\T}(X_{\Sigma })$}{group of $\T$-Weil divisors}%
Recall that $\Pic(X_{\Sigma })$, the {Picard group} of $X_{\Sigma }$,
\index{Picard group}%
\nomenclature[apic]{$\Pic(X)$}{Picard group} is the group of
isomorphism classes of line bundles. Let $A_{n-1}(X_{\Sigma })$
denote the {Chow group} of cycles of dimension $n-1$.  
\index{Chow group}%
\nomenclature[aa]{$A_{d}(X)$}{Chow group of $d$-dimensional
  cycles}%
The following result shows that these
groups can be computed in terms of invariant divisors.

\begin{thm}\label{thm:9}
  Let $\Sigma $ be a fan in $N_{\R}$ that is not contained in
  any hyperplane. Then there is a commutative diagram with exact rows
  \begin{displaymath}
    \xymatrix{
      0 \ar[r] & M \ar[r] \ar@{=}[d] & \Div_{\T}(X_{\Sigma }) \ar[r]
      \ar@{^{(}->}[d] & \Pic(X_{\Sigma })  \ar[r]
      \ar@{^{(}->}[d] & 0 \\
      0 \ar[r] & M \ar[r] & Z_{n-1}^{\T}(X_{\Sigma }) \ar[r]
      & A_{n-1}(X_{\Sigma })  \ar[r]
      & 0
    }.
  \end{displaymath}
\end{thm}
\begin{proof}
This is the first proposition in \cite[\S 3.4]{Ful93}.
\end{proof}
\begin{rem}
  In the previous theorem, the hypothesis that $\Sigma $ is not
  contained in any hyperplane is only needed for the injectivity of
  the second arrow in each row of the diagram.
\end{rem}

In view of Theorem~\ref{thm:27}, the upper exact sequence of the
diagram in Theorem~\ref{thm:9} can be interpreted as follows.

\begin{cor} \label{cor:18} Let $X$ be a toric variety with torus $\T$.
  \begin{enumerate}
  \item \label{item:101} Every toric line bundle $L$ on $X$ admits a toric
    section. Moreover, if $s$ and~$s'$ are two toric
    sections, then there exists $m\in M$ such that $s'=\chi^{m} s$.
  \item \label{item:102}   If the fan $\Sigma $ that defines $X$ is not contained in any
  hyperplane, and $L$ and $L'$ are toric line bundles on
  $X$, then there is at most one isomorphism between them.
  \end{enumerate}
\end{cor}
\begin{proof}
  This follows from theorems \ref{thm:9} and \ref{thm:27}.
\end{proof}

We next study the intersection of a $\T$-Cartier divisor with the
closure of an orbit.  Let $\Sigma $ be a fan in $N_{\R}$ and  $\Psi $
the virtual support function on
$\Sigma$ given by the set of defining vectors $\{m_{\tau}\}_{\tau\in
  \Sigma}$.  Let $\sigma$ be a cone of $\Sigma $ and $\iota _{\sigma }\colon
V(\sigma )\hookrightarrow X_{\Sigma }$ the associated closed
immersion. We consider first the case when $\Psi|_{\sigma}=0$. 
Let $\tau
\supset \sigma $ be another cone of $\Sigma $. For vectors $u\in \tau$
and $v\in \R\sigma$ such that $u+v\in 
\tau$, the condition $\Psi|_{\sigma}=0$ implies 
$$
\Psi (u+v)= \langle m_\tau, u+v\rangle =
\langle m_\tau, u\rangle = \Psi(u)
$$
because $m_\tau\big|_{\R\sigma}=0$.
Hence, we can define a function
  \nomenclature[g2303]{$\Psi(\sigma)$}{virtual support function
    induced on a quotient}%
\begin{equation}
  \label{eq:4}
\Psi(\sigma)\colon N(\sigma)_{\R}\longrightarrow \R , \quad
u+\R\sigma\longmapsto  \Psi 
(u+v)   
\end{equation}
for any $v\in \R\sigma$ such that $u+v\in \bigcup_{\tau\supset \sigma}
\tau$. 

It is easy to produce
a set of defining vectors of $\Psi (\sigma )$. For
each cone $\tau \supset \sigma $ we denote by
$\ov \tau =\pi _{\sigma}(\tau )$ the corresponding cone in $\Sigma
(\sigma)$. Since $m_\tau\big|_{\R\sigma}=0$, then $m_{\tau }\in
M(\sigma )=M\cap \sigma ^{\perp}$. We set $m_{\ov \tau
}=m_{\tau }\in M(\sigma )$. 

\begin{prop}\label{prop:72}
Let notation be as above. If $\Psi|_{\sigma }=0$, then $D_{\Psi }$
  intersects $V(\sigma )$ properly and $\iota_\sigma^{\ast}D_{\Psi
  }=D_{\Psi (\sigma )}$.  Moreover, $\{m_{\ov \tau }\}_{\ov \tau \in
    \Sigma (\sigma )}$ is a set of defining vectors of~$\Psi (\sigma
  )$.
\end{prop}
\begin{proof}
   The $\T$-Cartier divisor $D_{\Psi }$ is given by $\{(X_{\tau },\chi^{-m_{\tau
     }})\}_{\tau\in \Sigma}$. If $m_{\sigma }=0$, the local equation of $D_{\Psi }$ in
   $X_{\sigma }$ is $\chi^{0}=1$. Therefore, the orbit $O(\sigma )$
   does not meet the support of $D_{\Psi }$. Hence $V(\sigma )$ and
   $D_{\Psi }$ intersect properly.

   To see that $\{m_{\ov \tau
   }\}_{\ov \tau \in \Sigma (\sigma )}$ is a set of defining vectors, we
   pick a point $\ov u\in \ov \tau $ and we choose $u\in \tau $ such
   that $\pi _{\sigma }(u)=\ov u$. Then
   \begin{displaymath}
     \Psi (\sigma )(\ov u)=\Psi (u)=m_{\tau }(u)=m_{\ov \tau }(\ov u),
   \end{displaymath}
   which proves the claim. Now, using
   the characterization of $\Psi (\sigma )$ in terms of defining
   vectors, we have
   \begin{displaymath}
     \iota_\sigma^{\ast}D_{\Psi}=\{(X_{\tau }\cap V(\sigma ),\chi ^{-m_{\tau
       }}\mid_{X_{\tau }\cap V(\sigma )})\}_{\ov \tau}=
     \{(X_{\ov \tau },\chi ^{-m_{\ov \tau
       }})\}_{\ov \tau}=D_{\Psi (\sigma )}.
   \end{displaymath}
\end{proof}

When $\Psi|_{\sigma}\not=0$, the cycles $D_{\Psi}$ and $V(\sigma)$ do not
intersect properly, and we can only intersect $D_{\Psi}$ with $V(\sigma)$
up to rational equivalence. To this end, we choose any $m_{\sigma}'$
such that $\Psi(u) =\langle m'_{\sigma},u\rangle$ for every $u\in
\sigma$.  Then the divisor $D_{\Psi -m_{\sigma }'}$ is rationally
equivalent to $D_{\Psi }$ and $\Psi -m_{\sigma }'|_{\sigma}=0$. 
By the above result, this divisor intersects  $V(\sigma )$
properly, and its restriction to $V(\sigma )$ is given by the
virtual support function $(\Psi -m_{\sigma }')(\sigma)$.

\begin{exmpl}
  \label{exm:20}
We can use the above description of the restriction of
a line bundle to an orbit to compute the degree of an orbit of
dimension one. Let $\Sigma $ be a complete fan and
$\tau \in \Sigma^{n-1} $. Hence $V(\tau )$ is a toric curve. Let $\sigma _{1}$ and $\sigma _{2}$
be the two $n$-dimensional cones that have $\tau $ as a common
face. Let $\Psi $ be a virtual support function.  Choose
$v\in \sigma _{1}$ such that $\pi _{\tau }(v)$ is a generator of the
lattice $N(\tau )$. Then, by \eqref{eq:45} and \eqref{eq:4},
\begin{displaymath}
  \deg_{D_{\Psi }}(V(\tau ))=
  \deg(\iota_{\tau}^{*}D_{\Psi})=m_{\sigma
    _{2}}(v)-m_{\sigma _{1}}(v). 
\end{displaymath}
\end{exmpl}

Let now $(L,z)$ be a toric line bundle on $X_{\Sigma }$ and
$\sigma\in \Sigma$. The line bundle $\iota ^{\ast}_{\sigma }L$ on
$V(\sigma )$ has an 
induced toric structure.
Let $s$ be a
toric section of~$L$ that is regular and nowhere vanishing on
$X_{\sigma }$, and set $z_{\sigma }=s(x_{\sigma })\in L_{x_{\sigma
  }}\setminus \{0\}$. If~$s'$ is another such section, then
$s'=\chi^{m}s $ for an $m\in M$ such that $m|_{\sigma }=0$, by
Corollary~\ref{cor:18}. Therefore $s'(x_{\sigma })=s(x_{\sigma
})$. Hence, $z_{\sigma }$ does not depend on the choice of section and
$(\iota ^{\ast}_{\sigma }L, z_{\sigma })$ is the induced toric line
bundle. The following result follows easily from the constructions.

\begin{prop} \label{prop:81} Let $(L,z)$ be a toric line bundle on
  $X_{\Sigma }$ and $\sigma \in \Sigma $. Let $\Psi $ be a virtual
  support function such that $\Psi
  |_{\sigma }=0$ and $(L,z)\simeq (L_{\Psi},z_{\Psi })$ as toric line
  bundles. Then $\iota ^{\ast}_{\sigma }(L,z)\simeq
  (L_{\Psi(\sigma)},z_{\Psi(\sigma )})$.
\end{prop}

We next study the inverse image of a $\T$-Cartier divisor with respect
to equivariant morphisms as those in Theorem \ref{thm:25}. Let
$N_{i}$, $\Sigma _{i}$, $i=1,2$, and let $H\colon N_{1}\to N_{2}$ and $p\in
X_{\Sigma _{2},0}(K)$ be as in Theorem \ref{thm:25}. Let
$\varphi_{p,H}$ be the associated equivariant morphism, $\Psi $ a
virtual support function on $\Sigma_{2} $ and 
$\{m'_{\tau'}\}_{\tau'\in \Sigma _{2}}$ a set of defining vectors
of $\Psi $.  For each cone $\tau \in\Sigma _{1}$ we choose a cone
$\tau' \in \Sigma _{2}$ such that $H(\tau )\subset \tau '$ and we
write $m_{\tau }=H^{\vee}(m'_{ \tau '})$.  The following result
follows easily from the definitions

\begin{prop}\label{prop:70}
  The divisor $D_{\Psi }$ intersects properly the image of
  $\varphi_{p,H}$. The function 
  $\Psi\circ H$ is a virtual support function on $\Sigma _{1}$ and
\begin{displaymath}
  \varphi^{\ast}_{p,H}D_{\Psi}=D_{\Psi \circ H}.
\end{displaymath}  
Moreover, $\{m_{\tau }\}_{\tau \in \Sigma _{1}}$ is a set of defining
vectors of $\Psi\circ H$.
\end{prop}

\begin{rem}\label{rem:19}
  If $L$ is a toric line bundle on $X_{\Sigma _{2}}$ and $\varphi$ is
  a toric morphism, then $\varphi^{\ast}L$ has an induced toric
  structure.
\index{toric structure on a line bundle}%
Namely, $\varphi^{\ast}(L,z) = 
  (\varphi^{\ast}L,\varphi^{\ast}z)$.  By contrast, if
  $\varphi\colon X_{\Sigma _{1}}\to X_{\Sigma _{2}}$ is a general
  equivariant morphism that meets the principal open subset, there is no
  natural toric structure on $\varphi^{\ast}L$, because
  the image of the distinguished point $x_{1,0}$ does not need to
  agree with $x_{2,0}$.  If $(L,s)$ is a toric line bundle equipped with
  a toric section, then we set $\varphi^{\ast}(L,s)=
  ((\varphi^{\ast}L,(\varphi^{\ast}s)(x_{1,0})),\varphi^{\ast}s)$. However,
  the underlying toric bundle of $\varphi^{\ast}(L,s)$ depends on
  the choice of the toric section.
\end{rem}

\section{Positivity properties of $\T$-Cartier divisors}
\label{pos_divisors}

Let $\Sigma$ be a fan in $N_{\R}$ and 
$\Psi$ a virtual support function on $\Sigma$.
In this section, we will assume that $\Sigma $ is complete or,
equivalently, that the variety $X_\Sigma $ is proper. 

Many geometric properties of the pair $(X_\Sigma, D_\Psi)$
can be read directly from $\Psi$. For instance, the following result
relates the concavity of the virtual support function $\Psi $ with the
positivity of $D_{\Psi }$. 

\begin{prop}\label{prop:99} Let $\mathcal{O}(D_\Psi)$ be the line
  bundle associated to $D_{\Psi }$.
  \begin{enumerate}
  \item \label{item:56} $\mathcal{O}(D_\Psi)$ is generated by global 
    sections if and only if $\Psi$ is concave.
  \item \label{item:62}$\mathcal{O}(D_\Psi)$ is ample if and only if
    $\Psi$ is strictly concave on $\Sigma $.   
  \end{enumerate}
\end{prop}
\begin{proof}
  This is classical, see for instance \cite[\S 3.4]{Ful93}.
\end{proof}
 
In the latter case, the fan $\Sigma $ agrees with the polyhedral
complex $\Pi (\Psi )$ (Definition~\ref{def:29}) and the pair
$(X_\Sigma, D_\Psi)$ is completely determined by $\Psi $. Thus, the
variety $X_{\Sigma }$ is projective if and only if the fan $\Sigma $
is complete and regular (Definition \ref{def:54}). 


We associate to $\Psi$ the subset of $M_\R$ 
\nomenclature[g0423]{$\Delta_\Psi$}{polytope associated to a virtual
  support function}%
$$
\Delta_\Psi=\{x\in M_\R\mid  \langle x,u\rangle \ge \Psi(u) \mbox{ for
  all } u\in N_\R\}.
$$
This set is either empty or a lattice polytope.
\index{polytope!associated to a virtual support function}%
\index{toric variety!associated to a polytope}%
When $\cO(D_{\Psi })$ is generated by
global sections, the polytope $\Delta
_{\Psi }$ agrees with $\Stab(\Psi )$, and $\Psi$ is the support
function of $\Delta
_{\Psi }$.

The polytope $\Delta _{\Psi }$ encodes a lot of information about
the pair $(X_\Sigma ,D_\Psi )$. 
For instance,
we can read from it the space of global sections of
$\mathcal{O}(D_{\Psi })$.
\begin{prop}
A monomial rational section
$\chi^{m}\in \mathcal{K}_{X_{\Sigma }}$, $m\in M$, is a regular global
section of $\mathcal{O}(D_{\Psi })$ if and only if $m\in \Delta _{\Psi
}$. Moreover, the set $\{\chi^{m}\}_{m\in M\cap \Delta _{\Psi }}$ is a
$K$-basis of the space of global sections $\Gamma
(X_{\Sigma},\mathcal{O}(D_{\Psi }))$.     
\end{prop}
\begin{proof}
  See for instance \cite[\S 3.4]{Ful93}.
\end{proof}

Also the intersection number between toric divisors can be read off
from the corresponding polytopes. 

\begin{prop}
  \label{prop:46}
Let $D_{\Psi_{i}}$, $i=1,\dots,n$, be $\T$-Cartier divisors on
$X_{\Sigma}$ generated by their global sections. Then 
\begin{equation}
  \label{eq:65}
  (D_{\Psi_{1}}\cdot \dots\cdot D_{\Psi_{n}})=
  \MV_{M}(\Delta_{\Psi_{1}},\dots,\Delta_{\Psi_{n}}).
\end{equation}
where $ \MV_{M}$ denotes the mixed volume function associated to the
Haar measure $\Vol_{M}$ on $M_\R$ (Definition \ref{def:41}).  In
particular, for a $\T$-Cartier divisor  $D_{\Psi }$ generated by its
global sections,
\begin{equation}
  \label{eq:3}
 \deg_{D_{\Psi }}(X_\Sigma)=(D_{\Psi }^n)=n!\Vol_{M}(\Delta_\Psi ).  
\end{equation} 
\end{prop}

\begin{proof}
  This follows from \cite[Proposition~2.10]{Oda88}.
\end{proof}

\begin{rem}\label{rem:21}
  The intersection multiplicity and the degree in the above
  proposition only depend on the isomorphism class of the line bundles
  $\cO(D_{\Psi _{i}})$ and not on the $\T$-Cartier divisors
  themselves. It is easy to check directly that the right-hand sides of
  \eqref{eq:65} and 
  \eqref{eq:3} only depend on the isomorphism classes of the line
  bundles. In  
  fact, let $L$ be a toric line bundle generated by global sections
  and $s_{1}$, $s_{2}$ two toric 
  sections.  For $i=1,2$, set $D_{i}=\div(s_{i})$ and let $\Psi _{i}$
  be the corresponding support function and $\Delta _{i}$ the
  associated polytope. Then $s_{2}=\chi^{m}s_{1}$
  for some $m\in M$. Thus $\Psi _{2}=\Psi _{1}-m$ and $\Delta _{2}=\Delta
  _{1}-m$. Since the volume and the mixed volume are invariant under
  translation, we see that these formulae do not depend on the choice
  of sections. 
\end{rem}

\begin{defn}
  A \emph{polarized toric variety}
\index{toric variety!polarized}%
  is a pair $(X_{\Sigma},D_{\Psi
  })$, where $X_{\Sigma }$ is a toric variety and $D_{\Psi }$ is an
  ample $\T$-Cartier divisor.
\end{defn}

Polarized toric varieties can be classified in terms of their polytopes. 

\begin{thm} \label{thm:21}\ 
The correspondence $(X_{\Sigma},D_{\Psi})\mapsto
  \Delta_{\Psi}$ is a bijection between the set of 
  polarized 
  toric varieties
  and the set of lattice
  polytopes of 
  dimension $n$ of $M$. 
  Two ample $\T$-Cartier divisors $D_{\Psi}$
  and $D_{\Psi'}$ on a toric variety $X_{\Sigma}$ are rationally
  equivalent if and only if $\Delta _{\Psi'}$ is the
  translate of $\Delta _{\Psi}$ by an element of $M$. 
\end{thm}

\begin{proof}
If $\Psi$ is a
strictly concave function on $\Sigma$, then
$\Delta_{\Psi}$ is an 
$n$-dimensional lattice polytope.
Conversely, if $\Delta$ is a lattice polytope in $M_{\R}$, then 
 $ \Psi_{\Delta}$, the support function of $\Delta$, is a strictly
 concave function on  
 the complete fan $\Sigma_{\Delta}=\Pi(\Psi_{\Delta})$ (see
 examples~\ref{exm:19} and~~\ref{exm:13}). Therefore, the result
 follows from Theorem 
 \ref{thm:3} and the construction in Remark \ref{rem:21}. 
\end{proof}

\begin{rem} \label{rem:11} When $D_{\Psi}$ is only generated by its
  global sections, the polytope~$\Delta_{\Psi}$ may not determine the
  variety $X_{\Sigma }$, but it does determine a polarized toric
  variety that is the image of $X_{\Sigma }$ by a toric morphism.
  Write $\Delta=\Delta_{\Psi}$ for
  short. Let $M(\Delta)$ be as in Notation
  \ref{def:79} and choose $m\in \aff(\Delta )\cap M$. Set
  $N(\Delta)=M(\Delta)^{\vee}$. The translated polytope
  $\Delta-m$ has the same dimension as its ambient space
  $L_{\Delta }=M(\Delta)_{\R}$. By the theorem above, it defines a complete fan
  $\Sigma_{\Delta}$ in $N(\Delta)_{\R}$ together with a support
  function $\Psi_{\Delta}\colon N(\Delta)\to \R $. The projection
  $N\to N(\Delta)$ induces a toric morphism
  \begin{displaymath}
    \varphi\colon   X_{\Sigma}\longrightarrow X_{\Sigma_{\Delta}},
  \end{displaymath}
  the divisor $D_{\Psi_{\Delta}}$ is ample, and 
  $D_{\Psi}=\varphi^{*}D_{\Psi_{\Delta}}+\div(\chi^{-m})$.
\end{rem}

\begin{exmpl} \label{exm:26}  The projective morphisms associated to
  $\T$-Cartier divisors generated by global sections can also be made
  explicit in 
  terms of the lattice points of the associated polytopes. 
  Consider a proper toric variety $X_{\Sigma }$ of dimension
  $n$ equipped with a
  $\T$-Cartier divisor $D_{\Psi }$ generated by global sections.
  Let $m_0,\dots, m_r\in \Delta_\Psi\cap M$ be such that
  $\Conv(m_0,\dots, m_r)=\Delta _{\Psi }$. These vectors determine an
  H-representation 
  $\Psi =\min_{i=0,\dots,r}m_{i}$. Let $H\colon N_{\R}\to \R^{r}$ be
  the linear map defined by $H(u)= (m_{i}(u)-m_{0}(u))_{i=1,\dots,r}$.
  By Lemma
  \ref{lemm:17}, $\Psi =H ^{\ast} \Psi _{\Delta ^{r}}+m_{0}$.

  In $\R^{r}$ we
  consider the fan $\Sigma _{\Delta ^{r}}$, whose associated  toric
  variety is $\P^{r}$.  One easily verifies
  that,
  for each $\sigma \in \Sigma $, there is $\sigma'\in \Sigma
  _{\Delta ^{r}}$ with $H(\sigma )\subset \sigma '$.
  Let $p=(p_{0}:\dots:p_{r})$ be an arbitrary rational point of the
  principal open subset of  $\P^{r}$. 
  The equivariant morphism $\varphi_{p,H} \colon X \to \P^{r}_K$
  can be written explicitly as $(p_{0}\chi^{m_0}: 
  \dots: p_{r}\chi^{m_r})$. Moreover, $D_{\Psi }=\varphi_{p,H}
  ^{\ast}D_{\Psi _{\Delta ^{r}}} +\div(\chi ^{-m_{0}})$.
\end{exmpl}

The orbits of a polarized toric variety $(X_{\Sigma},D_{\Psi})$ are in
one-to-one correspondence with the
faces of  $\Delta _{\Psi }$.

\begin{prop}\label{prop:34} Let $\Sigma$ be a complete fan in
  $N_{\R}$ and $\Psi$ a strictly concave function on $\Sigma$. The
  correspondence $F\mapsto O(\sigma_{F})$ is a bijection between the
  set of faces of $\Delta_{\Psi}$ and the set of the orbits under the
  action of $\T$ on $X_{\Sigma}$.
\end{prop}

\begin{proof}
This follows from Example~\ref{exm:19}.
\end{proof}

The equation \eqref{eq:45} gives a formula for the Weil divisor $[D_{\Psi }]$
in terms of the virtual support function $\Psi $. When the line bundle
$\mathcal{O}(D_{\Psi })$ is ample,  we can
interpret this formula in terms of the facets of the polytope $\Delta
_{\Psi }$. 

Let $D_{\Psi}$ be an
ample divisor on $X_{\Sigma}$. The polytope $\Delta _{\Psi }$ has
maximal dimension $n$. For each facet $F$ of $\Delta _{\Psi }$,
let $v_{F}$ be as in Notation \ref{def:79}. The ray $\tau
_{F}=\R_{\ge 0}v_{F}$ is a cone of $\Sigma $.

\begin{prop}  \label{prop:65} With the previous hypothesis,
  \begin{displaymath}
    \div(s_{\Psi })=[D_{\Psi }]=\sum_{F}-\langle F,v_{F}\rangle V(\tau  _{F}),
  \end{displaymath}
  where the sum is over the facets $F$ of $\Delta$.
\end{prop}
\begin{proof}
  Since $\Psi $ is strictly concave on $\Sigma $, the Legendre-Fenchel
  correspondence shows that the set of rays of the form $\tau _{F}$
  agrees with the set $\Sigma ^{1}$. Moreover, $\Psi (v_{F})=\langle
  F,v_{F}\rangle$, because $\Psi $ is the support function of
  $\Delta $. The proposition then follows from \eqref{eq:45}.
\end{proof}

For a $\T$-Cartier divisor generated by global sections, we can
interpret its intersection with the closure of an orbit, and its inverse
image with respect to an equivariant morphism, in terms of direct and
inverse images of concave functions.

\begin{prop}  \label{prop:47} Let $\Sigma$ be a complete fan in
  $N_{\R}$ and $\Psi\colon N_{\R}\to \R $ a
  support function on $\Sigma$.
  \begin{enumerate}
  \item \label{item:52} Let $\sigma\in \Sigma$, $F_{\sigma }$ the
    associated face of $\Delta_{\Psi}$, and
    $m_{\sigma}'\in F_{\sigma }\cap M$.  Let $\pi _{\sigma }\colon N_{\R}\to
    N(\sigma )_{\R}$ be the natural projection. Then
\begin{equation}\label{eq:76}
  (\Psi-m'_{\sigma })(\sigma)=(\pi _{\sigma
    })_{\ast}(\Psi-m_{\sigma}').
\end{equation}
In particular, the  restriction of $D_{\Psi-m'_{\sigma }}$ to $V(\sigma)$ 
is given by the concave function $(\pi _{\sigma
    })_{\ast}(\Psi-m_{\sigma}')$.
Moreover,  the 
associated polytope is 
\begin{equation}
  \label{eq:77}
 \Delta_{(\Psi-m_{\sigma }')(\sigma)}
=F_\sigma-m_\sigma'\subset M(\sigma)_\R=\sigma^\bot. 
\end{equation}
\item \label{item:53} Let $H\colon N'\to N$ be a linear map and
  $H^{\vee}\colon M\to M'$ its dual map, where $M'=(N')^{\vee}$. Let
  $\Sigma '$ be a 
  fan in $N'_{\R}$ such that, for each $\sigma '\in \Sigma '$ there is
  $\sigma \in \Sigma $ with $H(\sigma ')\subset \sigma $, and let $p\in
  X_{\Sigma ,0}(K)$. Then
\begin{equation}
  \label{eq:80}
 \varphi_{p,H}^{\ast}D_{\Psi }=D_{H
  ^{\ast}\Psi }, 
\end{equation}
and the associated polytope is
\begin{equation}
  \label{eq:81}
 \Delta_{H
  ^{\ast}\Psi }=H^{\vee}(\Delta_\Psi) \subset M'_{\R}. 
\end{equation}
  \end{enumerate}
\end{prop}
\begin{proof}
  The equation \eqref{eq:76} follows from \eqref{eq:4}, while the equation
  \eqref{eq:80} follows from Prop\-osi\-tion \ref{prop:70}. Then
  \eqref{eq:77} and \eqref{eq:81} follow from Proposition
  \ref{prop:101}.
\end{proof}

As a consequence of the above construction, we can compute easily the
degree of any orbit.

\begin{cor} \label{cor:9} Let $\Sigma$ be a complete fan in $N_{\R}$,
  $\Psi\colon N_{\R}\to \R $ a support function on
  $\Sigma$, and $\sigma \in \Sigma $ a cone of dimension
  $n-k$. Then
  \begin{displaymath}
    \deg_{D_{\Psi }}(V(\sigma ))=k! \Vol_{M(F_{\sigma })}(F_{\sigma }).
  \end{displaymath}
\end{cor}
\begin{proof}
  In view of \eqref{eq:77} and \eqref{eq:3}, it is enough to
  prove that $M(\sigma )=M(F_{\sigma })$. But this follows from the
  fact that $L_{F_{\sigma }}=\sigma ^{\perp}$ (see Notation \ref{def:79}).
\end{proof}

\begin{exmpl} \label{exm:21}
Let $\tau\in \Sigma^{n-1}$. The degree of the curve $V(\tau)$ 
agrees with the lattice length of $F_{\tau }$.
\end{exmpl}

We will also need the toric version of the {Nakai-Moishezon
  criterion}.
\index{Nakai-Moishezon criterion!for a toric variety}%

\begin{thm}\label{thm:4} 
  Let $X_{\Sigma}$ be a proper toric variety and
  $D_{\Psi }$ a $\T$-Cartier divisor on~$X_{\Sigma}$.
  \begin{enumerate}
  \item \label{item:29} The following properties are equivalent:
    \begin{enumerate}
    \item \label{item:63} $D_{\Psi }$ is ample;
    \item \label{item:81} $(D_{\Psi }\cdot C) > 0$ for every curve $C$
      in $X_{\Sigma }$; 
    \item \label{item:103} $(D_{\Psi }\cdot V(\tau ))>0$ for every $\tau \in
      \Sigma^{n-1} $;
    \item \label{item:80} the function $\Psi $ is strictly concave on $\Sigma $.
    \end{enumerate}
  \item \label{item:44} The following properties are equivalent:
    \begin{enumerate}
    \item \label{item:104} $D_{\Psi }$ is generated by its global sections;
    \item \label{item:105} $(D_{\Psi }\cdot C) \ge 0$ for every curve
      $C$ in $X_{\Sigma 
      }$;
    \item \label{item:106} $(D_{\Psi }\cdot V(\tau ))\ge 0$ for every
      $\tau \in \Sigma^{n-1} $; 
    \item \label{item:113} the function $\Psi $ is concave.
    \end{enumerate}
  \end{enumerate}
\end{thm}

\begin{proof} The equivalence between \eqref{item:63} and
  \eqref{item:80} and between \eqref{item:104} and \eqref{item:113} is
  Proposition \ref{prop:99}.  The rest of \eqref{item:29} and
  \eqref{item:44} is proved in \cite{Mavlyutov:shtv}, see also
  \cite[Theorem~2.18]{Oda88} for \eqref{item:29} in the case of smooth
  toric varieties.
\end{proof}

A direct consequence of theorems \ref{thm:4} and \ref{thm:9} is
that, in a toric variety, a divisor is nef if and only if it is
generated by global sections, and every ample divisor is generated by
global sections.

\section{Toric schemes over a discrete valuation ring}
\label{Toric schemes}

In this section we recall some basic facts about the algebraic
geometry of toric schemes over a DVR. These toric  
schemes were introduced in
\cite[Chapter~IV, \S 3]{Kempfals:te}, and we refer to this reference 
for more details or to \cite{Gub12} for a study of toric schemes
over general valuation rings and their relation with tropical
geometry. They are described and classified in terms of
 fans in $N_{\R}\times
\R_{\ge 0}$. In this section
we will mostly consider
proper toric schemes over a DVR. As a
consequence of Corollary~\ref{cor:10}, proper toric schemes over a~DVR can be
described 
and classified in terms of complete SCR polyhedral
complexes in $N_{\R}$ as, for instance, in
\cite{NishinouSiebert:tdtvtc}. 

Let $K$ be a field equipped with a nontrivial discrete valuation
$\val_{K}\colon K^{\times}\twoheadrightarrow
\R$\index{valuation map!of a non-Archimedean
  field}\nomenclature[aval1]{$\val_{K}$}{valuation map of a field} whose group of values is $\Z$.  In
this section we do not assume $K$ to be complete. As usual, we denote
by $K^{\circ}$
the valuation ring, by
$K^{\circ\circ}$
its maximal ideal, by  $\varpi $ a generator of $K^{\circ\circ}$ and by $k$
the residue
field. Since the group of values of  $\val_K$ is $\Z$, then
$\val_{K}(\varpi )=1$. 
We denote by $S$ the base scheme $S=\Spec(K^{\circ})$, by $\eta$ and
$o$ the generic and the special points of $S$ and, 
for a scheme $\mathcal{X}$ over $S$, we set
$\mathcal{X}_\eta=\mathcal{X}\times_S \Spec(K)$ and  
$\mathcal{X}_o=\mathcal{X}\times_S \Spec(k)$ for its generic and special fibre 
respectively.
We will denote by $\T_{S}=\T_{K^{0}}\simeq \G_{m,S}^{n}$ a
split torus over $S$. Let $\T=\T_{K}$, $N$ and $M$ be as in
\S \ref{Toric varieties}. We will write $\wt N=N\oplus \Z$ and  $\wt
M=M\oplus \Z$.
\nomenclature[aNl03]{$\wt N$}{$N\oplus \Z$}%
\nomenclature[aMl03]{$\wt M$}{$M\oplus \Z$}%

\begin{defn}\label{def:15}
A \emph{toric scheme over $S$} of relative dimension
  $n$\index{toric scheme} is a normal integral separated $S$-scheme
of finite type, $\cX$, equipped with a dense open embedding
$\T_{K}\hookrightarrow \cX_{\eta}$ and an $S$-action of 
$\T_{S}$ over $\cX$ that extends the action of $\T_{K}$ on itself by
translations. If we want to stress the torus acting on $\cX$ we will
call them toric schemes \emph{with torus $\T_{S}$}.
\index{torus!over $S$, acting on a toric scheme}%
\end{defn}

If $\cX$ is a toric scheme over $S$, then $\cX_{\eta }$
is a toric variety over $K$ with torus $\T$.

\begin{defn}\label{def:40} Let $X$ be a toric variety over $K$ with
  torus $\T_{K}$ and let $\cX$ be a toric scheme over $S$ with
  torus $\T_{S}$. We say that $\cX$ is a \emph{toric model} of
    $X$ over $S$
\index{toric model!of a toric variety}%
if the identity of $\T_{K}$ can be extended to an isomorphism
  from $X$ to $\cX_{\eta}$. 

  If $\cX$ and $\cX'$ are toric models of $X$ and $\alpha \colon \cX\to
  \cX'$ is an $S$-morphism, we say that $\alpha $ is a \emph{morphism
    of toric models}
\index{toric model!of a toric variety!morphism of}%
if its restriction to $\T_{K}$ is the identity.
\end{defn}
Since, by definition, a toric scheme is integral and contains $\T$ as a
dense open subset, it is flat over $S$. Thus a toric model
is a particular case of a model as in 
Definition \ref{def:48}.


Let $\wt \Sigma $ be a fan in $N_{\R}\times \R_{\ge
  0}$. 
\nomenclature[g1816]{$\wt \Sigma $}{fan in $N_{\R}\times \R_{\ge 0}$}%
To the fan $\widetilde \Sigma $ we associate a toric scheme
$\cX_{\widetilde \Sigma }$ over $S$.
\index{toric scheme!associated to a fan}%
Let $\sigma \in \wt \Sigma $ be a cone and $\sigma ^{\vee}\subset \wt
M_{\R}$ its dual cone. 
Set $\wt M_{\sigma }=\wt M\cap \sigma ^{\vee}$. 
Let
$K^{\circ}[\wt M_{\sigma }]$ be the semigroup
$K^{\circ}$-algebra of $\wt M_{\sigma }$. 
\nomenclature[aKsg2]{$K^{\circ}[\wt M_{\sigma }]$}{semigroup
  $K^{\circ}$-algebra of a cone}%
By definition, $(0,1)\in \wt
M_{\sigma }$. Thus $(\chi ^{(0,1)}-\varpi )$ is an ideal of
$K^{\circ}[\wt M_{\sigma }]$.   
\nomenclature[aMl06]{$\wt M_{\sigma }$}{semigroup of $\wt M$
  associated to a cone}%
There is a natural isomorphism
\begin{multline}\label{eq:72}
K^{\circ}[\wt M_{\sigma  }]/(\chi ^{(0,1)}-\varpi )
\simeq\\\Big\{\sum_{(m,l)\in \wt M_{\sigma }}
\alpha_{m,l}\varpi^{l}\chi^m\,\Big | \,\alpha_{m,l}\in K^{\circ}
\text{ and }\alpha_{m,l}=0 \text{ for almost all } (m,l) 
\Big\}
\end{multline}
that we use to identify both rings. The ring $K^{\circ}[\wt M_{\sigma }]/(\chi
^{(0,1)}-\varpi )$ is an integrally closed domain.
We set
$$\cX_{\sigma } =\Spec(K^{\circ}[\wt M_{\sigma }]/(\chi
^{(0,1)}-\varpi ))$$  
for the associated {affine toric scheme} over $S$.
\index{toric scheme!affine}%
\nomenclature[aX31]{$\cX_{\sigma }$}{affine toric scheme associated to
a cone}%
For short we will use the notation
\begin{equation}\label{eq:25}
  K^{\circ}[\cX_{\sigma }]=K^{\circ}[\wt M_{\sigma }]/(\chi
^{(0,1)}-\varpi ).
\end{equation}
\nomenclature[aKsg3]{$K^{\circ}[\cX_{\sigma }]$}{ring of functions of
  an affine toric scheme}%
For cones $\sigma  ,\sigma '\in \wt \Sigma  $, with $\sigma  \subset
\sigma  '$ we have a natural open immersion of affine schemes
$\cX_{\sigma }\hookrightarrow \cX_{\sigma '}$. Using these open
immersions as gluing data, we define the scheme
$$
\cX_{\widetilde \Sigma }=\bigcup_{\sigma  \in\wt \Sigma  }\cX_{\sigma }.
$$
\nomenclature[aX33]{$\cX_{\wt \Sigma }$}{toric scheme associated to
a fan}%
This is a reduced and irreducible normal scheme of finite type over $S$
of relative dimension $n$.

There are two types of cones in $\wt \Sigma $. The ones that are
contained in the hyperplane $N_{\R}\times \{0\}$, and the ones that
are not.  If $\sigma $ is contained in $N_{\R}\times \{0\}$, then
$(0,-1)\in \wt M_{\sigma }$, and $\varpi $ is invertible in
$K^{\circ}[\cX_{\sigma }]$. Therefore $K^{\circ}[\cX_{\sigma}]\simeq
K[M_{\sigma }]$; hence $\cX_{\sigma }$ is contained in the
generic fibre and it agrees with the affine toric variety $X_{\sigma
}$. If $\sigma $ is not contained in $N_{\R}\times \{0\}$, then
$\cX_{\sigma }$ is not contained in the generic fibre. 

To stress the difference between both types of affine schemes we will
use the following notations.
Let $\Pi $ be the SCR polyhedral complex in $N_{\R}$ obtained by
intersecting $\wt \Sigma $ by the hyperplane $N_{\R}\times \{1\}$ 
as in Corollary  \ref{cor:10}, and 
$\Sigma $ the fan in $N_{\R}$ obtained by intersecting $\wt \Sigma $
with $N_{\R}\times \{0\}$.
For $\Lambda \in \Pi $, the cone $\cc(\Lambda)\in \widetilde \Sigma $
is not contained in $N\times \{0\}$. We will write $\wt M_{\Lambda
}=\wt 
M_{\cc(\Lambda) }$, $K^{\circ}[\wt M_{\Lambda }]=K^{\circ}[\wt
M_{\cc(\Lambda) }]$, $\cX_{\Lambda} =\cX_{\cc(\Lambda)}$ and 
$K^{\circ}[\cX_{\Lambda}]=K^{\circ}[\cX_{\cc(\Lambda)}]$.  
\nomenclature[aX32]{$\cX_{\Lambda}$}{affine toric scheme associated to
a polyhedron}%
\nomenclature[aMl07]{$\wt M_{\Lambda }$}{semigroup associated to a
  polyhedron}%
\nomenclature[aKsg4]{$K^{\circ}[\wt M_{\Lambda }]$}{semigroup
  $K^{\circ}$-algebra of a polyhedron}%
\nomenclature[aKsg5]{$K^{\circ}[\cX_{\Lambda }]$}{ring of functions of
  an affine toric scheme}%

Given polyhedrons $\Lambda ,\Lambda '\in \Pi $,
with $\Lambda \subset 
\Lambda '$, we have a natural open immersion of affine toric schemes
$\cX_{\Lambda }\hookrightarrow \cX_{\Lambda '}$. 
Moreover, if a cone
$\sigma \in \Sigma $ is a face of a cone $\cc(\Lambda)$ for some
$\Lambda\in\Pi$,   
then  the affine toric variety $X_{\sigma }$ is also
an open subscheme of $\cX_{\Lambda }$. The open cover \eqref{eq:25}
can be written as 
$$
\cX_{\widetilde \Sigma }=\bigcup_{\Lambda \in\Pi }\cX_{\Lambda  }\cup
\bigcup _{\sigma 
  \in \Sigma } X_{\sigma }.
$$
We will
reserve the notation $\cX_{\Lambda }$, $\Lambda \in \Pi $, for the affine
toric schemes that are not contained in the generic fibre and denote
by $X_{\sigma }$, $\sigma \in \Sigma $, the affine toric schemes
contained in the generic fibre, because they are toric varieties
over $K$. 

The scheme $\cX_{0}$ corresponding to the polyhedron $0:=\{0\}$ 
is a 
group $S$-scheme which is canonically isomorphic to $\T_{S}$.
The $S$-action of $\T_{S}$ over 
$\cX_{\widetilde \Sigma}$ is constructed as in the case of varieties over a
field. Moreover there are open immersions $\T_K\hookrightarrow 
\cX_{\eta}\hookrightarrow
\cX_{\widetilde \Sigma}$  
of schemes over $S$ and the action of $\T_S$ on $\cX_{\widetilde
  \Sigma}$ extends the
action of $\T_K$ on itself. 
Thus $\cX_{\widetilde \Sigma}$ is a toric scheme over $S$.
Moreover, the fan $\Sigma$ defines a toric variety
over $K$ which coincides with the generic fibre $\cX_{\widetilde \Sigma,\eta}$.
Thus, $\cX_{\widetilde \Sigma}$ is a toric model of $X_{\Sigma}$. 
The special fibre
\begin{math}
  \cX_{\widetilde \Sigma ,o}=\cX_{\widetilde
    \Sigma}\underset{S}{\times} \Spec(k)  
\end{math}
has an induced action by $\T_{k}$, but, in general, it
is not a toric variety over $k$, because it
is not irreducible nor reduced. The reduced schemes
associated to 
its irreducible components are toric varieties 
over $k$ with this action.

Every toric scheme over $S$ can be obtained by the above
construction. Indeed, this 
construction  gives a classification of  
toric schemes by fans in $N_{\R}\times \R_{\ge 0}$
\cite[\S IV.3(e)]{Kempfals:te}.

If the fan $\wt \Sigma $ is complete,
then the scheme $\cX_{\wt \Sigma }$ is proper over $S$. In this case the set
$\{\cX_{\Lambda }\}_{\Lambda \in \Pi }$ is an open cover of
$\cX_{\wt \Sigma }$. Proper toric schemes over $S$ can also be classified by complete
SCR polyhedral complexes in $N_{\R}$. This
is not the case for general toric schemes over $S$ as is shown in \cite{BurgosSombra:rc}.

\begin{thm}\label{thm:7}
  The correspondence $\Pi\mapsto
  \cX_{\cc(\Pi)}$, where $\cc(\Pi )$ is the fan introduced
  in Definition \ref{def:69},
  is a bijection between the set of complete SCR polyhedral
  complexes in $N_\R$ and the set of isomorphism classes of
  proper toric schemes over $S$ of relative dimension~$n$.
\end{thm}
 \begin{proof}
   Follows from \cite[\S IV.3(e)]{Kempfals:te}  
and Corollary~\ref{cor:10}.
 \end{proof}

If we are interested in toric schemes as toric models of a toric
variety, we can restate the previous result as follows. 

\begin{thm}\label{thm:7b}
  Let $\Sigma $ be a complete fan in $N_{\R}$. Then
  there is a bijective correspondence between equivariant 
  isomorphism classes of proper toric models over 
  $S$ of $X_{\Sigma }$ and complete SCR polyhedral complexes $\Pi$ in $N_\R$
  such that $\rec(\Pi )= \Sigma $.
\end{thm}
\begin{proof}
  Follows easily from Theorem \ref{thm:7}.
\end{proof}

For the rest of the section we will restrict ourselves to the proper
case and we will denote by $\Pi $ a complete SCR polyhedral
complex. 
\index{toric scheme!proper, associated to a SCR polyhedral complex}%
To it we associate a complete fan $\cc(\Pi )$ in $N_{\R}\times \R_{\ge
0}$ and a complete fan $\rec(\Pi )$ in $N_{\R}$. 
For short, we will use the notation
 \begin{displaymath}
   \cX_{\Pi }=\cX_{\cc(\Pi )},
 \end{displaymath}
and we will identify the generic fibre $\cX_{\Pi,\eta}$ with the
toric variety $X_{\rec(\Pi )}$.
\nomenclature[aX34]{$\cX_{\Pi }$}{toric scheme associated to a
  polyhedral complex}%
\begin{exmpl} \label{exm:14}
  We continue with Example \ref{exm:10}. The fan $\Sigma _{\Delta
    ^{n}}$ is in particular an SCR polyhedral complex and
  the associated toric scheme over $S$ is $\P^{n}_{S}$, the 
{projective space} over $S$.\index{projective space!as a toric scheme}
\end{exmpl}

This example can be generalized to any complete fan $\Sigma$ in
$N_{\R}$.

\begin{defn}\label{def:43} Let $\Sigma $ be a complete fan in
  $N_{\R}$. Then $\Sigma $ is also a complete SCR polyhedral complex.
  Clearly $\rec({\Sigma })=\Sigma $. 
  The toric scheme $\cX_{{\Sigma }}$ is a model over $S$ of
  $X_\Sigma$ which is called the \emph{canonical model}.
\index{canonical model!of a toric variety}%
  Its special fibre
  \begin{displaymath}
    \cX_{\Sigma ,o}=X_{\Sigma ,k}
  \end{displaymath}
  is the toric variety over $k$ defined by the fan $\Sigma $.  
\end{defn}

The description of {toric orbits}
\index{orbit!in a toric scheme}%
in the case of a toric scheme over a
DVR is more involved than the case of toric
varieties over a field, because we have to consider two kind of orbits. 

In the first place, there is a bijection between $\rec(\Pi )$
and the set of orbits under the action of $\T_K$ on $\cX_{\Pi ,\eta}$,
that sends a cone $\sigma\in \rec(\Pi)$ to the orbit $O(\sigma
)\subset \cX_{\Pi ,\eta}=X_{\rec(\Pi)}$ as in the case of toric varieties
over a field. We will denote by $\cV(\sigma)$ the Zariski closure in
$\cX_{\Pi}$ of 
the orbit 
$O(\sigma )$ with its structure of reduced closed subscheme. Then $\cV(\sigma)$
is a {horizontal $S$-scheme},
\index{horizontal scheme}%
\index{horizontal orbit in a toric scheme}%
in the sense that the structure morphism $\cV(\sigma
)\to S$ is dominant, of relative dimension $n-\dim(\sigma)$.
\nomenclature[aVo2]{$\cV(\sigma )$}{horizontal closure of an orbit of a toric
  scheme}%

Next we describe $\cV(\sigma)$ as a toric scheme over $S$. 
As before, we write
$N(\sigma )=N/(N\cap \R \sigma )$ and let $\pi _{\sigma }\colon
N_{\R}\to N(\sigma )_{\R}$ be the linear projection. Each polyhedron
$\Lambda \in \Pi$ such that $\sigma 
\subset \rec(\Lambda )$ defines a polyhedron $\pi _{\sigma }(\Lambda
)$ in $N(\sigma )_{\R}$. One verifies that these polyhedra form a
complete SCR 
polyhedral complex in $N(\sigma )_{\R}$, that we denote $\Pi
(\sigma )$. This polyhedral complex is called the \emph{star} of $\sigma$ 
in $\Pi$.
\index{star!of a cone in a polyhedral complex}%
\nomenclature[g164]{$\Pi (\sigma )$}{star of a cone in a polyhedral complex}%
\begin{prop}\label{prop:74}
  There is a canonical isomorphism of toric schemes 
  $$ \cX_{\Pi (\sigma )}\longrightarrow \cV(\sigma ).$$
\end{prop}
\begin{proof}
  The proof is analogous to the proof of Proposition \ref{prop:73}.
\end{proof}

In the second place, there is a bijection between $\Pi$ and the set of
orbits under the action of $\T_k$ on $\cX_{o}$ over the closed point
$o$. Given a polyhedron $\Lambda \in \Pi $, we set
\nomenclature[aNl10]{$\wt N(\Lambda)$}{quotient lattice of $\wt N$ associated to
  a polyhedron}%
\nomenclature[aMl10]{$\wt M(\Lambda)$}{sublattice of $\wt M$ associated to
  a polyhedron}%
\begin{displaymath}
  \wt N(\Lambda)=\wt N / (\wt N \cap \R\negthinspace \cc(\Lambda)) ,
  \quad
  \wt M(\Lambda)= \wt N(\Lambda)^{\vee}= \wt M\cap \cc(\Lambda)^{\bot}.
\end{displaymath}
We denote $\T(\Lambda)=\Spec(k[\wt M(\Lambda)])$.
\nomenclature[aOrbit2]{$O(\Lambda)$}{vertical orbit in a toric scheme}%
This is a torus
over the residue field $k$ of dimension $n-\dim(\Lambda)$. There is a
surjection of rings
\begin{displaymath}
  K^{\circ}[\wt M_{\Lambda}]\longrightarrow k[\wt M(\Lambda)] , \quad
\chi^{(m,l)}\longmapsto
\begin{cases}
  \chi^{(m,l)} & \text{ if } (m,l)\in  \wt M(\Lambda), \\
  0 & \text{ if } (m,l)\notin  \wt M(\Lambda).
\end{cases}
\end{displaymath}
Since the element $(0,1)$ does not belong to 
$\wt M(\Lambda )$, this surjection sends the ideal $(\chi ^{(0,1)}-\varpi
)$ to zero. Therefore, it factorizes through a surjection 
$K^{\circ}[\cX_{\Lambda }]\to k[\wt
M(\Lambda)]$, that defines a closed immersion $\T(\Lambda
)\hookrightarrow \cX_{\Lambda }$. Let $O(\Lambda )$ be the image of
this map and $V(\Lambda )$ 
\nomenclature[aVo3]{$V(\Lambda )$}{vertical closure of an orbit of a
  toric scheme}%
the Zariski closure of this orbit in
$\cX_{\Pi}$.
The subscheme
$O(\Lambda )$ is contained in the special
fibre $\cX_{\Pi ,o}$, because the surjection sends
$\varpi $ to zero. By this reason, the orbits of this type will be called
\emph{vertical}.
\index{vertical orbit in a toric scheme}%
Therefore, $V(\Lambda )$ is a vertical cycle
in the sense 
that its image by the structure morphism is the closed point $o$.

The variety $V(\Lambda)$ has a structure of toric
variety with torus $\T(\Lambda )$. This structure is not canonical
because the closed immersion  $\T(\Lambda
)\hookrightarrow \cX_{\Lambda }$ depends on the choice of $\varpi$.
We can describe this structure as follows.
For each polyhedron $\Lambda '$ such that $\Lambda $ is a face of
$\Lambda '$, the image of
$\cc(\Lambda ')$ under the projection $\pi _{\Lambda }\colon \wt
N_{\R}\to \wt N(\Lambda 
)_{\R}$ is a strongly convex rational cone that we denote 
$\sigma _{\Lambda '}$. The cones $\sigma _{\Lambda '}$ form a fan
of $\wt N(\Lambda )_{\R}$ that we denote $\Pi (\Lambda )$. Observe
that the fan $\Pi (\Lambda )$ is the analogue of the star of a cone
defined in \eqref{eq:62}. 
\index{star!of a polyhedron in a polyhedral complex}%
For each cone $\sigma \in \Pi (\Lambda )$
there is a unique polyhedron $\Lambda _{\sigma }\in \Pi $ such that
$\Lambda $ is a face of $\Lambda _{\sigma }$ and $\sigma =\pi
_{\Lambda }(\cc(\Lambda _{\sigma }))$. 

\begin{prop}\label{prop:75} There is an isomorphism of toric varieties over
  $k$ 
  \begin{displaymath}
    X_{\Pi(\Lambda ),k}\longrightarrow V(\Lambda ).
  \end{displaymath}
\end{prop}
\begin{proof}
  Again, the proof is analogous to the proof of Proposition
  \ref{prop:73}. 
\end{proof}

The description of the adjacency relations between orbits is similar
to the one for toric varieties over a field. 
The 
orbit $V(\Lambda)$ is contained in
$V(\Lambda ')$ if and only if the polyhedron $\Lambda '$ is
a face of 
the polyhedron $\Lambda$. Similarly, $\cV(\sigma )$ is contained in
$\cV(\sigma ')$ if and only if $\sigma '$ is a face of $\sigma$. 
Finally, $V(\Lambda )$ is contained in $\cV(\sigma )$ if and only if
$\sigma $ is a face of 
the cone $\rec(\Lambda )$.

\begin{rem}\label{rem:13}
As a consequence of the above construction, we see that there is a
one-to-one correspondence between the 
vertices of $\Pi$ and the components of the special fibre. For each
$v\in \Pi^{0}$, the component $V(v)$ is a toric variety over $k$
defined by the fan $\Pi(v)$ in $\wt N_{\R}/\R(v,1)$.
The orbits contained in $V(v)$ correspond to the polyhedra $\Lambda\in \Pi$
containing $v$. In particular, the components given by two vertices
$v,v'\in\Pi^{0}$ share an 
orbit of dimension $l$ if and only if there exists a polyhedron
of dimension $n-l$ containing both $v$ and $v'$.
\end{rem}

To each polyhedron $\Lambda \in \Pi $, hence to each vertical orbit, we can
associate a combinatorial invariant, 
which we call its multiplicity. For a vertex $v\in \Pi ^{0}$, this
invariant agrees with the order of vanishing of
$\varpi $ along the component $V(v)$ (see \eqref{eq:44}).

Denote by $\inc\colon N\to \wt N$ the inclusion $\inc(u)=(u,0)$ and by
$\pr\colon \wt M \to M$ the projection $\pr(m,l)=m$. We identify
$N$ with its image. We set
\nomenclature[aNl09]{$N(\Lambda)$}{quotient lattice of $N$ associated to
  a polyhedron}%
\nomenclature[aMl09]{$M(\Lambda)$}{sublattice associated to
  a polyhedron in $N_{\R}$}%
\begin{displaymath}
  N(\Lambda )= N/(N\cap \R\negthinspace \cc(\Lambda )), \quad
  M(\Lambda )= M\cap \pr(\cc(\Lambda )^{\bot}).
\end{displaymath}
\begin{rem}
  The lattice $M(\Lambda )$ can also be described as
  $M(\Lambda )=M\cap L_{\Lambda }^{\bot}$. Therefore, for a cone $\sigma
  \subset N_{\R}$,
  the notation just introduced agrees with the one
  in~\eqref{eq:53}. Here, the polyhedron $\Lambda $ is contained in
  $N_{\R}$.  By contrast, for a polyhedron $\Gamma  \subset M_{\R}$,
  we follow 
  Notation \ref{def:79}, so
  $M(\Gamma )=M\cap L_{\Gamma }$. 
\end{rem}

Then  $\inc$  and $\pr$ induce inclusions of lattices of finite index 
 $N(\Lambda )\to \wt
N(\Lambda )$ and $\wt M(\Lambda )\to M(\Lambda )$, that we denote also
by $\inc$ and $\pr$, respectively. 
These inclusions are dual of each other and in particular,
their indexes agree.

\begin{defn}\label{def:22}
  The \emph{multiplicity} of a polyhedron $\Lambda \in \Pi $
  \index{multiplicity of a polyhedron}%
  \nomenclature[am1b]{$\mult(\Lambda)$}{multiplicity of a polyhedron}%
  is defined as 
  \begin{displaymath}
    \mult(\Lambda)=[M(\Lambda ):\pr(\wt M(\Lambda ))]=[\wt
  N(\Lambda ):\inc(N(\Lambda ))]. 
  \end{displaymath}
\end{defn}

\begin{lem}\label{lemm:13}
If $\Lambda\in \Pi $, 
then 
$
\mult(\Lambda)= \min \{n\ge 1
\mid \exists p\in \aff(\Lambda ),\ np\in N\}
$.
\end{lem}

\begin{proof}
We consider
  the inclusion $\Z\to \wt N(\Lambda )$ that sends $n\in \Z$ to the
  class of $(0,n)$. There is a commutative diagram with exact rows and columns
  \begin{displaymath}
    \xymatrix{
      &0\ar[d]&0\ar[d]&&\\
      0\ar[r]&N(\Lambda )\cap \Z \ar[d]\ar[r]& 
      \Z \ar[d]\ar[r] & \Z/(N(\Lambda )\cap \Z)\ar[d]\ar[r] & 0\\
      0\ar[r]&N(\Lambda ) \ar[d]\ar[r]& \wt N(\Lambda ) \ar[d]\ar[r]& 
      \wt N(\Lambda )/N(\Lambda )\ar[r] & 0\\
      & N(\Lambda ) /(N(\Lambda )\cap \Z )\ar[d] \ar[r] & 
      \wt N(\Lambda ) /\Z \ar[d] & &\\
      & 0 & 0 & &
    }
  \end{displaymath}
It is easy to see that the bottom arrow in the diagram 
is an isomorphism. By the Snake lemma the right vertical arrow is an
isomorphism. Therefore
\begin{displaymath}
  \mult(\Lambda )=[\Z: \Z\cap N(\Lambda )].
\end{displaymath}
We verify that 
\begin{math}
  \Z \cap N(\Lambda )=\{n\in \Z
\mid \exists p\in \aff(\Lambda),\ np\in N\},
\end{math}
from which the lemma follows.
\end{proof}

We now discuss equivariant morphisms of toric schemes. 

\begin{defn} \label{def:39}
Let $\T_{i}$, $i=1,2$, be split tori over $S$ and $\varrho\colon
\T_{1}\to  \T_{2}$ a morphism of algebraic group schemes. 
Let $\cX_{i}$ be toric schemes over $S$ with torus $\T_{i}$ and let
$\mu_{i}$ denote the corresponding action.
A morphism $\varphi\colon \cX_{1}\to \cX_{2}$ is \emph{$\varrho$-equivariant}
\index{equivariant morphism!of toric schemes}%
if
the diagram 
\begin{displaymath}
  \xymatrix{ \T_{1}\times \cX_{1}
    \ar[r]^{\mu_{1}}\ar[d]_{\varrho \times \varphi}
& \cX_{1} \ar[d]^{\varphi} \\ 
\T_{2}\times \cX_{2} \ar[r]^{\mu_{2}}
& \cX_{2}
  }
\end{displaymath}
commutes. A morphism $\varphi\colon \cX_{1}\to \cX_{2}$ is
\emph{$\varrho$-toric}
\index{toric morphism!of toric schemes}%
if its restriction to $\T_{1,\eta}$, the torus over $K$,
coincides with that of $\varrho$.
\end{defn}

It can be verified that a toric morphism of schemes over $S$ is also 
equivariant.
In the sequel, we extend the construction of equivariant
morphisms in \S \ref{Orbits} to proper toric schemes. Before that, we
need to relate rational points on 
the open orbit of the toric variety with lattice points in $N$.
\begin{defn}\label{def:32} The valuation map of the field,
$\val_{K}\colon K^{\times}\to \Z$, induces a 
{valuation map} on $\T(K)$,
\index{valuation map!of a torus}%
also denoted $\val_{K}\colon \T(K)\to N$, by the
  identifications
$\T(K)=\Hom(M,K^{\times})$ and $N=\Hom(M,\Z)$.
\end{defn}
Let $\T_{S,i}$, $i=1,2$, be split tori over $S$. For each $i$, let  
$N_{i}$ be the corresponding lattice and $\Pi_{i}$ a complete
SCR polyhedral complex in $N_{i,\R}$. 
Let $A\colon N_{1}\to N_{2}$ be an affine map 
such that, for every $\Lambda_{1}\in \Pi_{1}$, there exists
$\Lambda_{2}\in\Pi_{2}$ with $A(\Lambda_{1})\subset \Lambda_{2}$.
Let $p\in
\cX_{\Pi_{2},0}(K)=\T_{2}(K)$ such that $\val_{K}(p)=A(0)$.
Write $A=H+ \val_{K}(p)$, where $H\colon N_{1}\to
N_{2}$ is a linear map. $H$ induces a morphism of algebraic groups
\begin{displaymath}
  \varrho_{H}\colon \T_{S,1}\longrightarrow \T_{S,2}.
\end{displaymath}
Let $\Sigma _{i}=\rec(\Pi _{i})$. For each cone $\sigma _{1}\in
\Sigma _{1}$, there exists a cone $\sigma _{2}\in \Sigma _{2}$ with
$H(\sigma _{1})\subset \sigma _{2}$. Therefore $H$ and $p$ define an
equivariant morphism $\varphi_{p,H}\colon X_{\Sigma _{1}}\to X_{\Sigma
    _{2}}$ of toric varieties over $K$ as in Theorem \ref{thm:25}. 

\begin{prop}\label{prop:45} With the above hypothesis,
  the morphism $\varphi_{p,H}$ can be extended to a
  $\varrho _{H}$-equivariant morphism 
  \begin{displaymath}
    \Phi _{p,A}\colon \cX_{\Pi _{1}}\longrightarrow \cX_{\Pi _{2}}.
  \end{displaymath}
\end{prop}
\nomenclature[g21z]{$\Phi _{p,A}$}{equivariant morphism of toric schemes}%
\begin{proof}
Let $\Lambda_{i}\in \Pi_{i}$ such that $A(\Lambda_{1})\subset
\Lambda_{2}$. Then the map $\wt M_{2}\to \wt M_{1}$ given by
$(m,l)\mapsto (H^{\vee}m, \langle m,\val_{K}(p)\rangle+ 
l)$ for $m\in M$ and $l\in \Z$ (which is just the dual of the
linearization of $A$) induces a 
morphism of semigroups $\wt M_{2,\Lambda_{2}}\to \wt
M_{1,\Lambda_{1}}$. Since $\chi^{m}(p)\varpi^{-\langle m, \val_{K}(p)\rangle}$
belongs to 
$K^{\circ}$, the assignment 
\begin{displaymath}
\chi^{(m,l)}\longmapsto (\chi^{m}(p)\varpi^{-\langle m,\val_{K}(p)\rangle})
\chi^{(H^{\vee}m,\langle m, \val_{K}(p)\rangle+l)} 
\end{displaymath}
defines a ring morphism $K^{\circ}[\wt M_{2,\Lambda_{2}}]\to
K^{\circ}[\wt M_{1,\Lambda_{1}}]$. This morphism sends
$\chi^{(0,1)}-\varpi $ to $\chi^{(0,1)}-\varpi $, hence induces a
morphism  $K^{\circ}[\cX_{\Lambda_{2}}]\to
K^{\circ}[\cX_{\Lambda_{1}}]$ and a map $\cX_{\Lambda_{1}}\to
\cX_{\Lambda_{2}}$. Varying $\Lambda _{1}$ and $\Lambda _{2}$ we
obtain maps that glue together into a map 
\begin{displaymath}
  \Phi_{p,A}\colon \cX_{\Pi_{1}}\longrightarrow \cX_{\Pi_{2}}.
\end{displaymath}
By construction, this map extends $\varphi_{p,H}$ and is
equivariant with respect to the morphism $\varrho_{H}$. 
\end{proof}

As an example of the above construction, we consider the toric
subschemes associated to orbits under the 
action of subtori. Let $N$ be a lattice, $\Pi $ a complete SCR
polyhedral complex in $N_{\R}$ and set $\Sigma =\rec(\Pi )$.
Let $Q\subset N$ be a saturated sublattice and let $p\in X_{\Sigma ,0}(K)$.
We set $u_{0}=\val_{K}(p)$.
We
consider the affine map $A\colon Q_{\R}\to N_{\R}$ given by
$A (v)=v+u_{0}$. Recall that the sublattice $Q$ and the point $p$ induce
maps of toric varieties \eqref{eq:31}
$$
X_{\Sigma_Q} \longrightarrow Y_{\Sigma_Q,p}\hooklongrightarrow X_\Sigma.
$$
We want to identify the toric model of $X_{\Sigma_Q}$ induced by
the toric model $\cX_{\Pi }$  of $X_\Sigma$.
We define the complete SCR polyhedral complex
$\Pi _{Q,u_{0}}=A ^{-1}\Pi $ of $Q_{\R}$. Then, $\rec(\Pi
_{Q,u_{0}})= \Sigma_{Q}$. 
Applying the construction of Proposition \ref{prop:45},
 we obtain an equivariant morphism of schemes over $S$
\begin{displaymath}
  \cX_{\Pi _{Q,u_{0}}} 
  \longrightarrow \cX_{\Pi }.
\end{displaymath}
The image of this map is
the Zariski closure of $Y_{\Sigma_Q,p }$ and $\cX_{\Pi _{Q,u_{0}}}$ is a
toric model of $X_{\Sigma _Q}$. This map will be denoted either as
$\Phi _{p,A}$ or $\Phi _{p,Q}$.
Observe that the abstract toric scheme $ \cX_{\Pi _{Q,u_{0}}}$ only
depends on $Q$ and on $\val_{K}(p)$.


\section{$\T$-Cartier divisors on toric schemes}
\label{T Cartier divisors schemes}

The theory of $\T$-Cartier divisors carries over to the case of
toric schemes over a DVR. We keep the notations of the previous
section. In particular, $K$ is a field equipped with a nontrivial
discrete valuation $\val_{K}$. Let $\cX$ be a toric
scheme over $S=\Spec(K^{\circ})$ with torus $\T_{S}$. 
There are two morphisms from $\T_{S}\times \cX$ to
$\cX$: the toric action, that we denote by $\mu $,
and the second projection, that we denote by $\pi_{2}$.
A Cartier divisor $D$ on $\cX$ is called a 
\emph{$\T$-Cartier divisor}
\index{T-Cartier divisor@$\T$-Cartier divisor!on a toric scheme}%
if 
$\mu ^{\ast}D=\pi _{2}^{\ast}D.$

$\T$-Cartier divisors over a toric scheme can be described
combinatorially.
For simplicity, we will discuss only the case of
proper schemes. So, let  $\Pi $ be a
complete SCR polyhedral complex in $N_\R$, and $\cX_{\Pi }$ the
corresponding toric scheme.  
Let $\phiK $ be an H-lattice function on 
$\Pi$ (Definitions \ref{def:31} and \ref{def:54}).
Then $\phiK $ defines a $\T$-Cartier
divisor in a way similar to the one for toric varieties 
over a field. 
We recall that the schemes $\{\cX_{\Lambda}\}_{\Lambda\in \Pi}$ form
an open cover of $\cX_\Pi$.
\nomenclature[g2111]{$\phiK $}{H-lattice function}%
\nomenclature[aDcart3]{$D_\phiK $}{$\T$-Cartier divisor on a toric scheme}%
Choose a set of defining vectors 
$\{(m_\Lambda ,l_{\Lambda})\}_{\Lambda \in \Pi }$ of $\phiK  $. Then we set
\begin{equation} \label{eq:59}
  D_{\phiK }=\{(\cX_{\Lambda },\varpi ^{-l_{\Lambda }}\chi^{-m_{\Lambda
    }})\}_{\Lambda  \in \Pi }, 
\end{equation}
where we are using the identification \eqref{eq:72}.
The divisor
$D_{\phiK }$ only depends on $\phiK $ and not on a particular choice of defining vectors. 

We consider now toric schemes and $\T$-Cartier divisors over
$S$ as models of toric varieties and $\T$-Cartier divisors over $K$. 

\begin{defn}\label{def:17}
  Let $\Sigma $ be a complete fan in $N_{\R}$ and $\Psi $ a virtual
  support function on $\Sigma $. Let $(X_{\Sigma },D_{\Psi})$ be the
  associated toric variety and $\T$-Cartier divisor defined over $K$.  
  A \emph{toric model} of
    $(X_{\Sigma },D_{\Psi  })$
\index{toric model!of a $\T$-Cartier divisor}%
\index{model!toric|see{toric model}}%
is a triple
  $(\cX,D,e)$, where $\cX$ is a toric model over
  $S$ of $X$,  
  $D$ is a $\T$-Cartier divisor on $\cX$ and $e>0$ is an integer such that the isomorphism
  $\iota \colon X_{\Sigma 
  }\to \cX_{\eta}$ that extends the identity of $\T_{K}$ satisfies
  $\iota^{\ast}(D)=eD_{\Psi }$. 
  When $e=1$, the toric model $(\cX,D,1)$ will be denoted simply by
  $(\cX,D)$. A toric model will be
  called \emph{proper}
\index{toric model!of a $\T$-Cartier divisor!proper}%
whenever the scheme $\cX$ is proper over $S$. 
\end{defn}

\begin{exmpl}\label{exm:23} We continue with Example
  \ref{exm:14}. The function $\Psi _{\Delta ^{n}}$ is an H-lattice
  concave function on $\Sigma _{\Delta ^{n}}$ and
  $(\P^{n}_{S},D_{\Psi 
    _{\Delta ^{n}}})$ is a proper toric model of
  $(\P^{n}_{K},D_{\Psi _{\Delta ^{n}}})$. 
\end{exmpl}

This example can be generalized as follows.

\begin{defn}\label{def:55}
  Let $\Sigma $ be a complete fan in
  $N_{\R}$ and $\Psi $  a virtual support function on $\Sigma $. Then 
  $\Sigma $ is a complete SCR polyhedral complex in $N_{\R}$  and
  $\Psi $ is
  a rational piecewise affine function on $\Sigma $. Then 
  $(\cX_{{\Sigma }},D_{\Psi })$ is a model over
  $S$ of  $(X_\Sigma,D_{\Psi })$, which is called the
  \emph{canonical model}.
\index{canonical model!of a $\T$-Cartier divisor}%
\index{model!canonical|see{canonical model}}%
\end{defn}

\begin{defn}\label{def:70} Let $\cX$ be a toric scheme and $\cL$ a
  line bundle on $\cX$. A \emph{toric structure}
\index{toric structure on a line bundle}%
on $\cL$ is the choice of an
  element $z$ of the fibre  $\cL_{x_{0}}$, where $x_{0}\in
  \cX_{\eta}$ is the distinguished point.  
  A \emph{toric line bundle}
\index{toric line bundle!on a toric scheme}%
  on $\cX$ is a
  pair $(\cL,z)$, where $\cL$ is a line bundle over $\cX$ and
  $z$ is a toric structure on $\cL$. Frequently, when the toric
  structure is clear from the context, the element
  $z$ will be omitted from the notation and a toric line bundle will
  be denoted by the underlying line bundle. 
  A \emph{toric section}
\index{toric line bundle!on a toric scheme!toric section of}%
is a rational 
  section that is regular and non vanishing over the principal open
  subset $X_{0}\subset \cX_{\eta}$ and such that $s(x_{0})=z$. Exactly as in the case
  of toric varieties over a field, each $\T$-Cartier divisor defines a
  toric line bundle $\cO(D)$ together with a toric section. When the
  $\T$-Cartier divisor comes from an H-lattice function $\phiK $, the
  toric line bundle and toric section will be denoted $\cL_{\phiK }$
  and $s_{\phiK }$ respectively.
\nomenclature[aLza]{$\cL_{\phiK }$}{model of a toric line bundle}%
\nomenclature[asabs]{$s_{\phiK }$}{toric section on a toric scheme}%
\end{defn}

The following result follows directly form the definitions.

\begin{prop}\label{prop:89} Let $(X_{\Sigma },D_{\Psi })$ be a toric
  variety with a $\T$-Cartier divisor.
  Every toric model $(\cX,D,e)$ of $(X_{\Sigma },D_{\Psi })$
  induces a model $(\cX,\mathcal{O}(D),e)$ of $(X_{\Sigma
  },L_{\Psi })$, in the sense of Definition
  \ref{def:56}, where the identification of $\mathcal{O}(D)|_{X_\Sigma
  }$ with $L_{\Psi }^{\otimes e}$ matches the toric
  sections determined by the Cartier divisors (Theorem
  \ref{thm:27}). Such models will be called  
  \emph{toric models}.
\end{prop}

\begin{prop-def} \label{prop:50}
 We say that two toric
models $(\cX_{i},D_{i},e_{i})$, $i=1,2$, are 
\emph{equivalent},
\index{toric model!of a $\T$-Cartier divisor!equivalence of}%
if there  
exists a toric model
$(\cX',D',e')$ of $(X 
_{\Sigma },D_{\Psi })$ and morphisms of 
toric models $\alpha _{i}\colon \cX'\to \cX_{i}$, $i=1,2$, such that 
$e'\alpha _{i}^{\ast}D_{i}=e_{i}D'$. This is an
equivalence relation.  
\end{prop-def}
\begin{proof}
  Symmetry and reflexivity are straightforward. For transitivity assume
  that we have toric models $(\cX_{i},D_{i},e_{i})$,
  $i=1,2,3$, that the first and second model are equivalent
  through $(\cX',D',e')$ and that the second and the third are
  equivalent through $(\cX'',D'',e'')$. Then, by Theorem
  \ref{thm:7b},  $\cX'$ and
  $\cX''$ are defined by SCR polyhedral complexes $\Pi '$ and $\Pi''$
  respectively, with $\rec(\Pi ')=\rec(\Pi'')=\Sigma $. Let
  $\Pi'''=\Pi '\cdot \Pi''$. By Lemma \ref{lemm:15},
  $\rec(\Pi''')=\Sigma $. Thus $\Pi'''$ determines a model $\cX'''$ of
  $X_{\Sigma }$.  This model has morphisms $\beta '$ and $\beta ''$
  to 
  $\cX'$ and $\cX''$ respectively. We put $e'''=e'e''$ and
  $D'''=e''\beta '{}^{\ast} D'=e'\beta ''{}^{\ast} D''$. Now it is
  easy to verify that $(\cX''',D''',e''')$ provides the
  transitivity property.
\end{proof}

We are interested in proper toric models and equivalence classes
because, by Definition \ref{def:7}, a proper toric model of $(X_{\Sigma
},D_{\Psi })$ induces an algebraic metric on $L_{\Psi }^{\an}$.   
By Proposition \ref{prop:21}, equivalent toric models define the same
algebraic metric.

We can classify proper models of $\T$-Cartier divisors (and therefore of
toric line bundles) in terms of H-lattice functions. We first recall
the classification of $\T$-Cartier divisors.

\begin{thm} \label{thm:11}
  Let $\Pi $ be a complete SCR polyhedral complex
  in $N_{\R}$ and $\cX_{\Pi }$  the associated toric scheme
  over $S$. The correspondence  $\phiK\mapsto D_{\phiK}$ is an isomorphism
  between the group of H-lattice functions
  on $\Pi$ and the group of $\T$-Cartier divisors on $\cX_{\Pi
  }$. Moreover, if $\phiK _{1}$ and $\phiK _{2}$ are two H-lattice
  functions on $\Pi $, then the divisors $D_{\phiK _{1}}$ and $D_{\phiK
    _{2}}$ are rationally equivalent if and only if $\phiK _{1}-\phiK
  _{2}$ is affine.
\end{thm}
\begin{proof} The result 
follows from \cite[\S IV.3(h)]{Kempfals:te}.  
\end{proof}

We next derive the classification theorem for models of $\T$-Cartier divisors.

\begin{thm} \label{thm:11b}
  Let $\Sigma $ be a complete fan in $N_{\R}$
    and $\Psi $ a virtual 
  support function on $\Sigma $.
  Then the correspondence $(\Pi,\phiK)\mapsto (\cX_{\Pi},D_{\phiK})$ is a
  bijection between:
  \begin{itemize}
  \item[$\bullet$] the set of pairs $(\Pi
    ,\phiK )$, where $\Pi$ is a complete SCR polyhedral complex in
    $N_\R$ with $\rec(\Pi)$= $\Sigma $ and $\phiK $ is an H-lattice function
    on $\Pi$ such that $\rec(\phiK )=\Psi $;
  \item[$\bullet$]  the set of isomorphism
  classes of toric models $(\cX,D)$ of $(X_{\Sigma },D_{\Psi })$.
  \end{itemize}
\end{thm}
\begin{proof}
Denote by $\iota\colon X_{\Sigma }=X_{\rec(\Pi) } \to \cX_{\Pi }$
the open immersion of the generic fibre.
The recession function (Definition \ref{def:24}) determines the
restriction of the $\T$-Cartier 
divisor to the fibre over the generic point. Therefore, 
 when $\phiK $ is an H-lattice
function on $\Pi $ with $\rec(\phiK )=\Psi $, we have that
\begin{displaymath}
  \iota^{\ast}D_{\phiK }=D_{\rec(\phiK)}=D_{\Psi }.
\end{displaymath}
Thus $(\cX_{\Pi },D_{\phiK })$ is a toric model of $(X_{\Sigma
},D_{\Psi })$. The statement follows from Theorem~\ref{thm:7b} and
Theorem \ref{thm:11}.
\end{proof}

\begin{rem}
  Let $\Sigma $ be a complete fan in $N_{\R}$
    and $\Psi $ a virtual 
  support function on $\Sigma $. Let $(\cX,D,e)$ be a toric model of
  $(X_{\Sigma },D_{\Psi })$. Then, by Theorem \ref{thm:11b}, there exists a
  complete SCR polyhedral complex $\Pi $ in $N_{\R}$ with
  $\rec(\Pi )=\Sigma $ and a rational piecewise affine function $\phiK
  $ on $\Pi $ such that $e\phiK $ is an H-lattice function, $\rec(\phiK
  )=\Psi $ and $(\cX,D,e )=(\cX_{\Pi },D_{e\phiK },e)$. Moreover, if
  $(\cX',D',e' )$ is another toric model that gives the function $\phiK
  '$, then 
  both models are equivalent if and only if $\phiK =\phiK '$. 
  Thus, to every toric model we have associated a rational piecewise
  affine function $\phiK 
  $ on $\Pi $ such that $\rec(\phiK
  )=\Psi $. Two equivalent models give rise to the same function. 

  The converse is not true.
  Given a rational piecewise affine function $\phiK $, with $\rec(\phiK
  )=\Psi $, we can find a complete SCR polyhedral complex $\Pi $ such
  that $\phiK $ is piecewise affine on $\Pi $. But, in general,
  $\rec(\Pi )$ does not agree with $\Sigma $.  What we can expect is
  that $\Sigma ':=\rec(\Pi )$
  is a refinement of $\Sigma $. Therefore the function $\phiK $
  gives us an equivalence
  class of toric models of $(X_{\Sigma '},D_{\Psi })$. But $\phiK $
  may not determine an equivalence class of
  toric models of $(X_{\Sigma },D_{\Psi })$. In Corollary \ref{cor:11}
  in next section we
  will give a necessary condition for a function $\phiK $ to define an
  equivalence class of toric models of $(X_{\Sigma },D_{\Psi })$ and
  in Example \ref{exm:27} we will exhibit a function that does not
  satisfy this necessary condition. By
  contrast, as we will see in Theorem \ref{thm:12}, the concave case
  is much more transparent.
\end{rem}

The correspondence between $\T$-Cartier divisors and $\T$-Weil
\index{T-Weil divisor@$\T$-Weil divisor}%
divisors has to take into account that we have two types of
orbits. Each vertex $v  \in\Pi ^0$ defines a vertical 
invariant 
prime Weil divisor $V(v )$ and every ray $\tau \in \rec(\Pi)^{1}$ defines 
a horizontal prime Weil divisor $\cV(\tau )$. If $v \in\Pi ^0$ is a
vertex, by Lemma \ref{lemm:13}, its multiplicity $\mult(v)$  is
the smallest positive integer $\nu\ge1$ such
that $\nu v\in N$. If $\tau $ is a ray, we denote
by $v_{\tau }$ the smallest lattice point of $\tau\setminus\{0\} $. 

\begin{prop} \label{prop:51}
Let $\phiK $ be an H-lattice function on $\Pi $. Let $D_{\phiK }$ be the
associated $\T$-Cartier divisor. Then the corresponding $\T$-Weil
divisor is given by
\begin{displaymath}
  [D_{\phiK }]=
  \sum_{v \in \Pi ^0} -\mult(v)\phiK (v ) V(v )+
  \sum_{\tau \in \rec(\Pi )^{1}} -\rec(\phiK)(v_{\tau })\cV(\tau ).  
\end{displaymath}  
\end{prop}
\begin{proof}
  By Lemma \ref{lemm:13}, for $v\in \Pi^{0}$, the vector $\mult(v)v$
  is the minimal lattice vector in 
  the ray $\cc(v)$. Now it is easy to adapt 
  the proof of \cite[\S 3.3, Lemma]{Ful93} to prove this proposition. 
\end{proof}

\begin{exmpl} \label{exm:24}
Consider the constant H-lattice function
$\phiK (u)=-1$. This function corresponds to the principal divisor
$\div (\varpi )$. Then
\begin{equation}\label{eq:44}
  \div(\varpi )=\sum_{v \in \Pi ^{0}}\mult(v)V(v).
\end{equation}
Thus, for a vertex $v$, the multiplicity of $v$ agrees with the
multiplicity of the divisor $V(v)$ in the special fibre $\div(\varpi
)$. In particular, the special fibre $\cX_{\Pi ,o}$ is reduced if and only
if all vertices of $\Pi ^{0}$ belong to $N$.
\end{exmpl}

We next study the restriction of $\T$-Cartier divisors to orbits and
their inverse image by equivariant morphisms. Let $\Pi $ be a complete
SCR polyhedral 
complex  in $N_{\R}$, and $\phiK $ an H-lattice
function on $\Pi $. Set $\Sigma =\rec(\Pi )$, and
$\Psi =\rec(\phiK )$. Choose sets of defining vectors
 $\{(m_{\Lambda },l_{\Lambda
})\}_{\Lambda \in \Pi }$ and 
$\{m_{\sigma }\}_{\sigma \in 
  \Sigma }$ for $\phiK$ and $\Psi $, respectively.

Let $\sigma \in \Sigma $. We describe the restriction of
$D_{\phiK }$ to $\cV(\sigma)$, the closure of a horizontal orbit. As in the
case of toric varieties over a field, we first consider the case when
$\Psi |_{\sigma}=0$. 
Recall that $\cV(\sigma )$ agrees with the toric scheme associated to the
polyhedral complex $\Pi (\sigma )$ and that each element of $\Pi
(\sigma )$ is the image by $\pi _{\sigma }\colon N_{\R}\to N(\sigma
)_{\R}$ of a 
polyhedron $\Lambda \in \Pi $ with 
$\sigma \subset \rec(\Lambda )$. The condition $\Psi |_{\sigma}=0$ implies
that we can define
\begin{equation}\label{eq:73}
\phiK(\sigma)\colon N(\sigma)_{\R}\longrightarrow \R , \quad
u+\R\sigma\longmapsto  \phiK  
(u+v)   
\end{equation}
for any $v\in \R\sigma$ such that $u+v\in \bigcup_{\rec(\Lambda )\supset \sigma}
\Lambda $. The function $\phiK (\sigma )$ can also be described in terms of defining
vectors. For each $\Lambda \in \Pi $ with $\sigma \subset \rec(\Lambda
)$, we will denote $\ov \Lambda \in \Pi (\sigma )$ for its image by
$\pi _{\sigma }$. For each $\Lambda $ as before, the condition
$\Psi |_{\sigma}=0$ implies that $m_{\Lambda }\in M(\sigma )$. Hence we
define $(m_{\ov \Lambda },l_{\ov \Lambda })=(m_{\Lambda },l_{\Lambda
})$ for $\Lambda \in \Pi $ with $\rec(\Lambda )\supset \sigma $.  

\begin{prop}\label{prop:76} If $\Psi |_{\sigma}=0$ then the divisor
  $D_{\phiK }$ and the horizontal orbit $\cV(\sigma )$ intersect
  properly. Moreover, the set $\{(m_{\ov \Lambda },l_{\ov \Lambda
  })\}_{\ov \Lambda \in \Pi (\sigma )}$ is a set of defining vectors
  of $\phiK (\sigma )$ and the restriction of $D_{\phiK }$ to
  $\cV(\sigma )$ is $D_{\phiK (\sigma )}$. 
\end{prop}
\begin{proof}
  The proof is analogous to the proof of Proposition \ref{prop:72}. 
\end{proof}

If $\Psi |_{\sigma}\not=0$, then $\cV(\sigma )$ and
$D_{\phiK }$ do not intersect properly and we can only restrict $D_{\phiK}$ with $\cV(\sigma)$
up to rational equivalence. To this end, we consider the divisor
$D_{\phiK -m_{\sigma }}$, that is rationally
equivalent to $D_{\phiK }$ and intersects
properly with $\cV(\sigma )$. The restriction of  
this divisor to $\cV(\sigma )$ corresponds to the H-lattice function $(\phiK -m_{\sigma
})(\sigma)$ 
as defined above.  

Let now $\Lambda \in \Pi $ be a polyhedron. We will denote by 
$\wt \pi _{\Lambda }\colon \wt N\to \wt N(\Lambda )$ and $\pi
_{\Lambda }\colon N\to N(\Lambda )$ the projections and by 
$\wt \pi _{\Lambda }^{\vee}\colon \wt M(\Lambda )\to \wt M$ and $\pi
_{\Lambda }^{\vee}\colon M(\Lambda )\to M$ the dual maps. We will use
the same notation for the linear maps obtained by tensoring with $\R$.

We first assume that $\phiK
|_{\Lambda }=0$. 
If $u\in \wt N(\Lambda
)_{\R}$, then there exists a polyhedron $\Lambda '$ with $\Lambda $ a
face of $\Lambda '$ and a point $(v,r)\in \cc(\Lambda ')$ that
is sent to $u$ under the projection $\wt \pi _{\Lambda }$. Then we set  
\begin{equation}\label{eq:74}
\phiK(\Lambda )\colon \wt N(\Lambda )_{\R}\longrightarrow \R , \quad
u \longmapsto  r\phiK (v/r)=m_{\Lambda' }(v)+rl_{\Lambda' }.  
\end{equation}
The condition $\phiK |_{\Lambda }=0$ implies that the above equation
does not depend on the choice of $(v,r)$.

We can describe also $\phiK(\Lambda )$ in terms of defining vectors.
For each cone $\sigma \in \Pi (\Lambda )$ let $\Lambda _{\sigma }\in
\Pi $ be the polyhedron that has $\Lambda $ as a face and such that
$\cc(\Lambda )$ is mapped to $\sigma $ by $\wt \pi _{\Lambda }$.  The
condition $\phiK |_{\Lambda }=0$ implies that $(m_{\Lambda _{\sigma
  }},l_{\Lambda _{\sigma }})\in \wt M(\Lambda )$. We set $m_{\sigma
}=(m_{\Lambda _{\sigma }},l_{\Lambda _{\sigma }})$.

\begin{prop}\label{prop:77}
  If $\phiK |_{\Lambda }=0$ then the divisor $D_{\phiK }$ intersects
  properly the orbit $V(\Lambda )$. Moreover, the set $\{m_{\sigma
  }\}_{\sigma \in \Pi (\Lambda )}$ is a set of defining vectors of
  $\phiK (\Lambda )$ and the restriction of $D_{\phiK }$ to $V(\Lambda
  )$ is the divisor $D_{\phiK (\Lambda )}$.
\end{prop}
\begin{proof}
  The proof is analogous to that of Proposition \ref{prop:72}.
\end{proof}

As before, when $\phiK |_{\Lambda }\not=0$, we can only restrict $D_{\phiK }$ to
$V(\Lambda )$ up to rational equivalence. In this case we just apply
the previous proposition to the function $\phiK -m_{\Lambda
}-l_{\Lambda }$. 

\begin{exmpl} \label{exm:22}
We particularize \eqref{eq:74} to the case of one-dimensional vertical
orbits. 
Let $\Lambda $ be a $(n-1)$-dimensional
polyhedron. Hence $V(\Lambda )$ is a  vertical curve. Let $\Lambda
_{1}$ and $\Lambda _{2}$ 
be the two $n$-dimensional polyhedron that have $\Lambda $ as a common
face.  Let
$v\in N_{\Q}$ such that the class $[(v,0)]$ is a generator of the
lattice $\wt N(\Lambda )$ and the affine space $(v,0)+\R \cc(\Lambda
)$ meets $\cc(\Lambda _{1})$. This second condition fixes one of the
two generators of $\wt N(\Lambda )$. Then, by the equation \eqref{eq:45}
\begin{equation}
\label{eq:47}  
  \deg_{D_{\phiK }}(V(\Lambda ))=
  \deg([D_{\phiK}|_{V(\Lambda )}])=m_{\Lambda 
    _{2}}(v)-m_{\Lambda _{1}}(v). 
\end{equation}
\end{exmpl}

\medskip
We end this section discussing the inverse image of a $\T$-Cartier
divisor by an equivariant morphism. With the notation of Proposition
\ref{prop:45}, let $\phiK $ be an H-lattice function on $\Pi _{2}$, and
$\{(m_{\Lambda },l_{\Lambda })\}_{\Lambda \in \Pi _{2}}$  a set
of defining vectors of $\phiK $. For each $\Gamma \in \Pi _{1}$ we
choose a polyhedron $\Gamma '\in \Pi _{2}$ such that $A(\Gamma
)\subset \Gamma '$. We set $m_{\Gamma }=H^{\vee}(m_{\Gamma '})$ and
$l_{\Gamma }=m_{\Gamma '}(\val_{K}(p))+l_{\Gamma '}$. The following
proposition follows easily.

\begin{prop} \label{prop:78}
  The divisor $D_{\phiK }$ intersects properly the image of
  $\Phi_{p,A}$. The function 
  $\phiK\circ A$ is an H-lattice function on $\Pi  _{1}$ and
\begin{displaymath}
  \Phi^{\ast}_{p,A}D_{\phiK}=D_{\phiK \circ A}.
\end{displaymath}  
Moreover, $\{(m_{\Gamma },l_{\Gamma })\}_{\Gamma \in \Pi  _{1}}$ is a set of defining
vectors of $\phiK\circ A$.
\end{prop}

\section{Positivity on toric schemes}
\label{sec:posit-prop-}

The relationship between the positivity of the line bundle and the
concavity of the virtual support function can be extended to the case
of toric schemes over a DVR. In particular, we have the following
version of the {Nakai-Moishezon criterion}.
\index{Nakai-Moishezon criterion!for a toric scheme}%

\begin{thm} \label{thm:5}
  Let $\Pi$ be a complete SCR complex in $N_{\R}$ and
  $\cX_{\Pi }$ its associate toric scheme over $S$. Let $\phiK $ be an
   H-lattice function on $\Pi $ and
  $D_{\phiK }$ the corresponding $\T$-Cartier divisor on $\cX_{\Pi}$.
  \begin{enumerate}
  \item \label{item:45} The following properties are equivalent:
    \begin{enumerate}
    \item \label{item:60} $D_{\phiK }$ is ample;
    \item \label{item:64} $D_{\phiK }\cdot C > 0$ for every vertical
      curve $C$ contained in 
      $X_{\Pi ,o}$; 
    \item \label{item:65} $D_{\phiK  }\cdot V(\Lambda  )>0$ for every
      $(n-1)$-dimensional polyhedron
      $\Lambda \in \Pi $;
    \item \label{item:66} The function $\phiK $ is strictly concave on $\Pi $.
    \end{enumerate}
  \item \label{item:67} The following properties are equivalent:
    \begin{enumerate}
    \item \label{item:68} $D_{\phiK }$ is generated by global sections;
    \item \label{item:69} $D_{\phiK }\cdot C \ge 0$ for every vertical
      curve $C$ contained 
      in $X_{\Sigma,o }$;
    \item \label{item:70} $D_{\phiK }\cdot V(\Lambda )\ge 0$ for every
      $(n-1)$-dimensional polyhedron 
      $\Lambda \in \Pi $;
    \item \label{item:71} The function $\phiK $ is concave.
    \end{enumerate}
  \end{enumerate}
 \end{thm}
 \begin{proof}
   In both cases, the fact that (a) implies
   (b) and that (b) implies (c) is
   clear. The fact that (c) implies (d)
   follows from the equation \eqref{eq:47}. The fact that
   \eqref{item:66} implies \eqref{item:60} is
   \cite[\S IV.3(k)]{Kempfals:te}. 

   Finally, we prove that
   \eqref{item:71} implies \eqref{item:68}. Let $\phiK $ be an
   H-lattice concave function. Each pair $(m,l)\in \wt M$
   defines a rational section $\varpi ^{l}\chi^{m}s_{\phiK }$ of
   $D_{\phiK }$. The section is regular if and only if the function
   $m(u)+l$ lies above $\phiK $. Moreover, for a polyhedron $\Lambda \in
   \Pi $,  this section does not vanish
   on $\cX_{\Lambda }$ if and only if $\phiK (u)=m(u)+l$ for all $u\in
   \Lambda $. Therefore, the affine pieces of the graph of $\phiK $
   define a set of global sections that generate $\mathcal{O}(D_{\phiK
   })$.        
 \end{proof}

 Let $\Sigma $ be a complete fan in $N_{\R}$ and $\Psi $ a virtual
 support function on $\Sigma $.  Let $X_{\Sigma }$ and $D_{\Psi}$ be
 the associated proper toric variety over $K$ and $\T$-Cartier
 divisor.

 \begin{defn} \label{def:8} Let $(\cX,D,e)$ be a toric model of
   $(X_{\Sigma },D_{\Psi})$. Then $(\cX,D,e)$ is
   \emph{semipositive}
  \index{toric model!of a $\T$-Cartier divisor!semipositive}
if the $\T$-Cartier divisor $D$ satisfies any of the equivalent
conditions of Theorem \ref{thm:5}\eqref{item:67}. 
\end{defn}

Observe that, if a toric model $(\cX,D,e)$ of
  $(X_{\Sigma },D_{\Psi})$ is semipositive, then $(\cX,\cO(D),e)$ is a
  semipositive model of $(X_{\Sigma },\cO(D_{\Psi}))$ in the sense of Definition
  \ref{def:51}.     
Equivalence classes of semipositive toric models are classified by
rational concave functions. 

\begin{thm} \label{thm:12} Let $\Sigma $ be a complete fan in $N_{\R}$
  and $\Psi $ a virtual support function on $\Sigma $.  Then the
  correspondence of Theorem \ref{thm:11b} induces a bijective
  correspondence between the space of equivalence classes of
  semipositive toric models of $(X_{\Sigma },D_{\Psi })$ over~$S$ and
  the space of rational piecewise affine concave functions $\phiK $ on
  $N_{\R}$ with 
  $\rec(\phiK )=\Psi $.
\end{thm}
\begin{proof}
  Let $(\cX,D,e)$ be a semipositive toric model. By Theorem
  \ref{thm:11b}, 
  to the pair $(\cX,D)$ corresponds a pair $(\Pi ,\phiK ')$, where $\phiK
  '$ is an H-lattice function on $\Pi $, $\rec(\Pi )=\Sigma $
  and $\rec(\phiK ')=e\Psi $. By Theorem \ref{thm:5}, the function
  $\phiK '$ is concave. We put $\phiK =\frac{1}{e}\phiK '$. It is clear
  that equivalent models produce the same function.

  Conversely, let $\phiK $ be a rational piecewise affine concave
  function. Let $\Pi '=\Pi (\phiK )$. This is a rational polyhedral
  complex. Let $\Sigma '=\rec (\Pi ')$. This is a conic rational
  polyhedral complex. By Proposition \ref{prop:80}, $\Sigma '=\Pi
  (\Psi )$. Since $\Psi=\rec(\phiK) $ is concave, hence a support
  function on $\Sigma $, we deduce that $\Sigma $ is a refinement of
  $\Sigma '$. Put $\Pi =\Pi '\cdot \Sigma $ (Definition
  \ref{def:72}). Since $\Pi '$ is a rational polyhedral complex and
  $\Sigma $ is a fan, then $\Pi $ is an SCR polyhedral
  complex. Moreover, by Lemma \ref{lemm:15},
  \begin{displaymath}
    \rec(\Pi )=\rec(\Pi'\cdot \Sigma )=\rec(\Pi')\cdot \rec(\Sigma
    )=\Sigma '\cdot \Sigma =\Sigma.
  \end{displaymath}

  Let $e>0$ be an
  integer such that 
  $e\phiK $ is 
  an H-lattice function. Then $(\cX_{\Pi },D_{e\phiK },e)$ is a
  toric model of $(X_{\Sigma },D_{\Psi })$. Both procedures are
  inverse of each other.   
\end{proof}

A direct consequence of Theorem \ref{thm:12} is that the $\T$-Cartier
divisor $D_{\Psi}$ admits a semipositive model if and only if $\Psi$
is concave, hence a support function. By Proposition
\ref{prop:99}\eqref{item:56}, this is equivalent
to the fact that $D_{\Psi}$ is generated by global sections.

Recall that, for a toric variety over a field, a $\T$-Cartier divisor
generated by global sections can be determined either by the
support function $\Psi $ or by its stability set $\Delta _{\Psi }$. In
the case of toric schemes over a DVR, if $\phiK$ is a
concave rational piecewise affine function on $\Pi $ and $\Psi
=\rec(\phiK )$, then  the stability set of $\phiK $ agrees with the
stability set of 
$\Psi $. Then the equivalence class of toric models determined by
$\phiK $ is also determined by the Legendre-Fenchel dual function $\phiK
^{\vee}$. 

\begin{cor} \label{cor:3} Let $\Sigma $ be a complete fan in $N_{\R}$
  and $\Psi $ a support function on $\Sigma $. The correspondence of
  Theorem \ref{thm:12} and Legendre-Fenchel duality induce a bijection
  between the space of equivalence classes of semipositive toric models of
  $(X_{\Sigma },D_{\Psi })$ and that of rational piecewise affine concave
  functions on $M_{\R}$ with effective domain $\Delta _{\Psi }$.
\end{cor}
\begin{proof}
  From Theorem \ref{thm:12}, the space of equivalence classes of
  semipositive toric models of 
  $(X_{\Sigma },D_{\Psi })$ is in bijection with the space of
  rational piecewise affine concave functions $\phiK $ on $N_{\R}$ with 
  $\rec(\phiK )=\Psi $. 

  Let $\phiK$ be a function in this latter space. Then $\Dom(\phiK)=N_{\R}$
  and $\Stab(\phiK)=\Delta 
  _{\Psi }$. By
  propositions \ref{prop:18}\eqref{item:26}
  and \ref{prop:13}\eqref{item:73}, the function $\phiK^{\vee}$ is a
  rational piecewise affine concave
  function on $M_{\R}$ with effective domain~$\Delta _{\Psi
  }$. Conversely, if $\vartheta $ is a rational piecewise affine concave
  function on $M_{\R}$ with effective domain~$\Delta _{\Psi
  }$, then, by the same propositions, $\vartheta ^{\vee}$ is a
  rational piecewise affine concave function with effective domain $N_{\R}$
  and stability set $\Delta _{\Psi }$. By Proposition
  \ref{prop:18}\eqref{item:20}, the function $\rec(\vartheta ^{\vee})$
  agrees with $\Psi $. By Proposition \ref{prop:13}\eqref{item:72}
  the above correspondences are inverse of each other, thus
  stablishing the bijection.
\end{proof}

Let $\Pi$ be a complete SCR complex in $N_{\R}$ and $\phiK $ an
H-lattice concave function on~$\Pi $.  Then the $\T$-Cartier divisor $D_{\phiK}$ is generated by
global sections and we can interpret its
restriction to toric orbits in terms of direct and inverse images of
concave functions.

\begin{prop} \label{prop:52}
  Let $\Pi $ be a complete SCR polyhedral complex in $N_{\R}$ and
  $\phiK $ an H-lattice concave function on $\Pi $. Set $\Sigma
  =\rec(\Pi )$ and $\Psi =\rec(\phiK )$. Let $\sigma \in \Sigma $ and
  $m_{\sigma }\in M$ such that $\Psi |_{\sigma }=m_{\sigma
  }|_{\sigma }$. Let $\pi
  _{\sigma }\colon N_{\R}\to N(\sigma )_{\R}$ be the projection and
  $\pi^{\vee}_{\sigma }\colon M(\sigma )_{\R}\to M_{\R}$ the dual
  inclusion. Then
  \begin{equation}
    \label{eq:69}
    (\phiK -m_{\sigma })(\sigma )=(\pi _{\sigma })_{\ast}(\phiK -m_{\sigma }),
  \end{equation}
  Hence the restriction of the divisor $D_{\phiK -m_{\sigma }}$ to
  $\cV(\sigma )$ corresponds to the H-lattice concave function $(\pi
  _{\sigma })_{\ast}(\phiK -m_{\sigma })$. Dually,
  \begin{equation}
    \label{eq:70}
    (\phiK -m_{\sigma })(\sigma )^{\vee}=(\pi^{\vee} _{\sigma }+m_{\sigma
    })^{\ast}\phiK ^{\vee}.  
  \end{equation}
  In other words, the
  Legendre-Fenchel dual of $(\phiK -m_{\sigma })(\sigma )$ is the
  restriction of $\phiK ^{\vee}$ to the face $F_{\sigma }$ translated
  by $-m_{\sigma }$.
\end{prop}
\begin{proof} For the equation \eqref{eq:69},
we suppose without loss of generality that $m_{\sigma}=0$, and hence
$\Psi |_{\sigma }=0$. Let $u\in N(\sigma )_{\R}$. Then, the function
$\phiK |_{\pi ^{-1}_{\sigma }(u)}$ is concave. Let $\Lambda \in \Pi $
such that $\rec(\Lambda )=\sigma $ and $\pi ^{-1}_{\sigma }(u)\cap
\Lambda \not = \emptyset$. Then, $\pi
^{-1}_{\sigma }(u)\cap \Lambda $ is a polyhedron of maximal dimension
in $\pi
^{-1}_{\sigma }(u)$. The restriction of  $\phiK $ to this polyhedron
is constant and, by \eqref{eq:73}, agrees with
$\phiK(\sigma )(u)$. Therefore, by concavity,
\begin{displaymath}
  (\pi_{\sigma })_{\ast} \phiK (u)=\max_{v\in \pi ^{-1}_{\sigma }(u)}\phiK (v),
\end{displaymath}
agrees with $\phiK(\sigma )(u)$. Thus we obtain 
\eqref{eq:69}. The equation \eqref{eq:70} follows from the previous
equation and Proposition \ref{prop:101}\eqref{item:122}. To prove
\eqref{eq:70} when $m_{\sigma }\not = 0$ we use  Proposition
\ref{prop:8}\eqref{item:14}.  
\end{proof}

We now consider the case of a vertical orbit. For a function $\phiK $
as before, with $\Psi =\rec(\phiK )$,  we denote by $\cc(\phiK )\colon
\wt N_{\R}\to \underline \R$ 
the concave function given by
\begin{displaymath}
  \cc(\phiK )(u,r)=
  \begin{cases}
    r\phiK (u/r)&\text{ if }r>0, \\
    \Psi (u)&\text{ if }r=0, \\
    -\infty&\text{ if }r<0.
  \end{cases}
\end{displaymath}
The function $\cc(\phiK )$ is a support function on $\cc(\Pi
)$.

\begin{lem}\label{lemm:12}
  The stability set of $\cc(\phiK)$ is the epigraph
  $\epi(-\phiK^{\vee})\subset \wt M_{\R}$.
\end{lem}
\begin{proof}
  The H-representation of $\cc(\phiK)$ is
  \begin{align*}
    \Dom(\cc(\phiK))&=\{(u,r)\in \wt N_{\R}\mid r\ge 0\},\\ 
    \cc(\phiK )(u,r)&=\min_{\Lambda }(m_{\Lambda }(u)+l_{\Lambda }r).
  \end{align*}
By Proposition \ref{prop:19}
\begin{displaymath}
  \Stab(\cc(\phiK ))=\R_{\ge 0}(0,1)+\Conv(\{(m_{\Lambda },l_{\Lambda
  })\}_{\Lambda \in \Pi }).
\end{displaymath}
Furthermore, by the same proposition, for $x\in \Stab(\phiK )$, 
\begin{displaymath}
  \phiK ^{\vee}(x)=\sup\left\{
    \sum_{\Lambda }-\lambda _{\Lambda }l_{\Lambda }\bigg|
    \lambda _{\Lambda }\ge 0, \sum_{\Lambda }\lambda _{\Lambda
    }=1, \sum_{\Lambda }\lambda _{\Lambda }m_{\Lambda }=x
  \right\}.
\end{displaymath}
Hence $\epi(-\phiK ^{\vee})=\R_{\ge 0}(0,1)+\Conv(\{(m_{\Lambda },l_{\Lambda
  })\}_{\Lambda \in \Pi }),$ which proves the statement.
\end{proof}

\begin{prop}\label{prop:53}
  Let $\Pi$ and $\phiK $ be as before and let $\Lambda \in \Pi $. Let
  $m_{\Lambda }\in M$ and $l_{\Lambda }\in \Z$ be such that $\phiK
  |_{\Lambda }=(m_{\Lambda }+l_{\Lambda })|_{\Lambda }$. Let $\wt \pi
  _{\Lambda }\colon \wt N_{\R}\to \wt N(\Lambda )_{\R}$ be the
  projection, and
  $\wt \pi^{\vee} _{\Lambda }\colon \wt M(\Lambda )_{\R}\to \wt M_{\R}$ the
  dual map. 
  Then
  \begin{equation}
    \label{eq:71}
    (\phiK -m_{\Lambda }-l_{\Lambda })(\Lambda )=
    (\wt \pi _{\Lambda })_{\ast}(\cc(\phiK -m_{\Lambda }-l_{\Lambda })).
  \end{equation}
  Moreover, this is a support function on the fan
  $\Pi (\Lambda )$. Its stability set is 
  the polytope $\Delta _{\phiK ,\Lambda }:=(\wt \pi^{\vee}
  _{\Lambda }+(m_{\Lambda 
  },l_{\Lambda }))^{-1}\epi(-\phiK ^{\vee})$. Hence, the restriction of the divisor $D_{\phiK
    -m_{\Lambda }-l_{\Lambda }}$ to the variety $V(\Lambda )$
  is the divisor associated to the support function of $\Delta _{\phiK ,\Lambda }$.
\end{prop}
\begin{proof}
  To prove the equation \eqref{eq:71} we may assume that $m_{\Lambda }=0$
  and $l_{\Lambda }=0$. Let $u\in \wt N(\Lambda )_{\R}$. Then, the
  function $\cc(\phiK) |_{\wt \pi ^{-1}_{\Lambda }(u)}$ is concave. Let
  $\Lambda' \in \Pi $ such that $\Lambda $ is a face of $\Lambda '$
  and $\wt \pi ^{-1}_{\Lambda }(u)\cap \cc(\Lambda') \not =
  \emptyset$. Then, $\wt \pi ^{-1}_{\Lambda }(u)\cap \cc(\Lambda ')$ is a
  polyhedron of maximal dimension of $\wt \pi ^{-1}_{\Lambda }(u)$ and the
  restriction of $\cc(\phiK) $ to this polyhedron is constant and, by
 \eqref{eq:74}, agrees with $\phiK(\Lambda )(u)$. Therefore,
  by concavity,
  \begin{displaymath}
    (\wt \pi_{\Lambda})_{\ast} \cc(\phiK) (u)=\max_{v\in \pi ^{-1}_{\sigma
      }(u)}\cc(\phiK) (v), 
  \end{displaymath}
  agrees with $\phiK(\Lambda  )(u)$. This proves the equation
  \eqref{eq:71}.  

  Back in the general case when $m_{\Lambda }$ and
  $l_{\Lambda }$ may be different from zero, by Proposition
  \ref{prop:101}, Proposition \ref{prop:8}\eqref{item:14} and Lemma
  \ref{lemm:12} we have
  \begin{align*}
    \Stab((\wt \pi _{\Lambda })_{\ast}(\cc(\phiK -m_{\Lambda
    }-l_{\Lambda })))&= (\wt \pi^{\vee}_{\Lambda })^{-1}\Stab(\cc(\phiK
    -m_{\Lambda }-l_{\Lambda
    }))\\
    &= (\wt \pi^{\vee}_{\Lambda })^{-1}(\Stab(\cc(\phiK)) -(m_{\Lambda
    },l_{\Lambda
    }))\\
    &=(\wt \pi^{\vee} _{\Lambda }+(m_{\Lambda },l_{\Lambda
    }))^{-1}\Stab(\cc(\phiK )) \\
    &=(\wt \pi^{\vee}_{\Lambda }+(m_{\Lambda },l_{\Lambda }))^{-1}
    \epi(-\phiK^{\vee} ).
  \end{align*}
  The remaining statements are clear.
\end{proof}
We next interpret the above result in terms of dual polyhedral
complexes. Let $\Pi (\phiK )$ and $\Pi (\phiK ^{\vee})$ be the pair of
dual polyhedral complexes associated to $\phiK $. Since $\phiK $ is
piecewise affine on $\Pi $, then $\Pi $ is a refinement of $\Pi
(\phiK )$. For each $\Lambda \in \Pi $ we will denote by $\overline
{\Lambda }\in \Pi (\phiK )$ the smallest element of $\Pi (\phiK )$ that
contains $\Lambda $. It is characterized by the fact that $\ri(\Lambda
)\cap \ri(\overline \Lambda )\not =\emptyset.$ Let $\Lambda ^{\ast}\in
\Pi (\phiK ^{\vee})$ be the polyhedron $\Lambda ^{\ast}=\cL \phiK
(\overline \Lambda )$. This polyhedron agrees with $\partial \phiK
(u_{0})$ for any $u_{0}\in \ri(\Lambda )$.  Then the function $\phiK
^{\vee}|_{\Lambda ^{\ast}}$ is affine. The polyhedron $\Lambda
^{\ast}-m_{\Lambda }$ is contained in $M(\Lambda )_{\R}$. 
The polyhedron
\begin{displaymath}
  \wt {\Lambda ^{\ast}}=\{(x,-\phiK ^{\vee}(x))|x\in \Lambda ^{\ast}\}
\end{displaymath}
is a face of $\epi(-\phiK ^{\vee})$ and it agrees with the intersection
of the image of $\pi^{\vee}_{\Lambda }+(m_{\Lambda },l_{\Lambda })$ with
this epigraph.  We consider the commutative diagram of lattices
\begin{displaymath}
  \xymatrix{
  \wt M(\Lambda )\ar[rr]^{\wt \pi ^{\vee}_{\Lambda } +(m_{\Lambda} ,l_{\Lambda })}
  \ar[d]_{\pr} &&
  \wt M \ar[d]^{\pr}\\
  M(\Lambda )\ar[rr]^{\pi ^{\vee}_{\Lambda } +m_{\Lambda }} &&M,}
\end{displaymath}
where $\pi ^{\vee}_{\Lambda }$ is the inclusion $M(\Lambda )\subset M$, 
and the corresponding commutative diagram of real vector spaces
obtained by tensoring with $\R$. This diagram induces a commutative
diagram of polytopes
\begin{displaymath}
  \xymatrix{
  \Delta _{\phiK ,\Lambda }\ar[rr]^{ \wt \pi ^{\vee}_{\Lambda
    }+(m_{\Lambda} ,l_{\Lambda })} 
  \ar[d]_{\pr} &&
  \wt {\Lambda ^{\ast}} \ar[d]^{\pr}\\
  \Lambda ^{\ast}-m_{\Lambda }\ar[rr]^{\pi ^{\vee}_{\Lambda }
    +m_{\Lambda }} && \Lambda ^{\ast},
}
\end{displaymath} 
where all the arrows are isomorphisms. 

In other words, the polytope $\Delta _{\phiK ,\Lambda }$ associated
to the restriction of $D_{\phiK -m_{\Lambda }-l_{\Lambda }}$ to
$V(\Lambda )$ is obtained as follows. We include $\wt M(\Lambda
)_{\R}$ in $\wt M_{\R}$ throughout the affine map $\wt \pi ^{\vee}_{\Lambda
} +(m_{\Lambda} ,l_{\Lambda })$. The image of this map intersects
the polyhedron $\epi(-\phiK ^{\vee})$ in the face of it that lies above
$\Lambda ^{\ast}$. 
The inverse image of this face agrees with $\Delta _{\phiK ,\Lambda }$.

Since we have an explicit description of the polytope $\Delta _{\phiK ,\Lambda }$,
we can easily calculate the degree with respect to $D_{\phiK }$ of an
orbit $V(\Lambda )$.

\begin{prop}\label{prop:31}
  Let $\Pi $ be a complete SCR polyhedral complex in $N_{\R}$ and $\phiK
  $ an H-lattice concave function on $\Pi
  $. Let $\Lambda
  \in \Pi $ be a polyhedron of dimension $n-k$, $u_{0}\in
  \ri(\Lambda )$  and $\Lambda 
  ^{\ast}=\partial \phiK (u_{0})$.
  Then
  \begin{equation}
    \label{eq:27}
    \mult(\Lambda) \deg_{D_{\phiK }}(V(\Lambda)
    )=k! \Vol_{M(\Lambda )}(\Lambda ^{\ast}),       
  \end{equation}
  where $\mult(\Lambda )$ is the multiplicity of $\Lambda $ (see
  Definition \ref{def:22}).
\end{prop}
\begin{proof}
  From the description of $D_{\phiK }|_{V(\Lambda )}$ and Proposition
  \ref{prop:46}, we know that 
  \begin{displaymath}
    \deg_{D_{\phiK }}(V(\Lambda))=k! \Vol_{\wt M(\Lambda )}(\Delta _{\phiK
      ,\Lambda} ). 
  \end{displaymath}
  Since
  \begin{displaymath}
    \Vol_{\wt M(\Lambda )}(\Delta _{\phiK
      ,\Lambda} )=\frac{1}{[M(\Lambda ):\wt M(\Lambda )]}
    \Vol_{M(\Lambda )}(\Lambda ^{\ast}),
  \end{displaymath}
  the result follows from the definition of the multiplicity.
\end{proof}

\begin{rem}
  If $\dim(\Lambda ^{\ast})< k$, then both sides of \eqref{eq:27} are
  zero. If $\dim(\Lambda ^{\ast})=k$, then $M(\Lambda )=M(\Lambda
  ^{\ast})$ and $\Vol_{M(\Lambda )}(\Lambda ^{\ast})$ agrees with the
  lattice volume of $\Lambda ^{\ast}$.
\end{rem}

We now interpret the inverse image by an equivariant morphism, of a $\T$-Cartier
divisor generated by global sections, in terms of direct and inverse
images of concave functions. 

\begin{prop}\label{prop:48}
  With the hypothesis of Proposition \ref{prop:45}, let $\phiK_{2}$
  be an H-lattice concave function on $\Pi _{2}$ and $D_{\phiK _{2}}$
  the corresponding $\T$-Cartier divisor. Then $\Phi
  _{p,A}^{\ast}D_{\phiK _{2}}$ is the $\T$-Cartier divisor associated
  to the H-lattice concave function $\phiK _{1}=A^{\ast}\phiK
  _{2}$. Moreover the Legendre-Fenchel dual is given by
  \begin{displaymath}
    \phiK _{1}^{\vee}=(H^{\vee})_{\ast}(\phiK_{2}^{\vee}-\val_{K}(p)).
  \end{displaymath}
\end{prop}
\begin{proof}
  The first statement is Proposition \ref{prop:78}. The second statement
  follows from Proposition \ref{prop:101}\eqref{item:121}. 
\end{proof}




\begin{exmpl}\label{exm:12}
  Let $\Sigma $ be a complete fan in $N_{\R}$ and $\Psi $ a support
  function on~$\Sigma $. By Theorem \ref{thm:12}, any equivalence
  class of semipositive models of $(X_{\Sigma },D_{\Psi })$ is
  determined by a rational piecewise affine concave function $\phiK $
  with $\rec(\phiK )=\Psi $. By Lemma \ref{lemm:17}, any such function
  can be realized as the inverse image by an affine map of the support
  function of a standard simplex. Using the previous proposition, any
  equivalence class of semipositive toric models can be induced by an
  equivariant projective morphism.

  More explicitly, let $e>0$ be an integer such that $e\phiK $ is an
  H-lattice concave
  function.  Let $\Pi$ be a complete SCR complex in $N_{\R}$
  compatible by $e\phiK $ and such that $\rec(\Pi )=\Sigma $ (see the
  proof of Theorem \ref{thm:12}). Then, $(\cX_{\Pi },D_{e\phiK },e)$ is
  a toric model of $(X_{\Sigma },D_{\Psi })$ in the class determined
  by $\phiK $.

  Choose an H-representation
  \begin{math}
    e\phiK (u)=\min_{0\le i\le r}(m_{i}(u)+l_{i})
  \end{math}
  with $(m_{i},l_{i})\in \wt M$ for $i=0,\dots,r$.
  Put
  $\boldsymbol{\alpha }=(l_{1}-l_{0},\dots,l_{r}-l_{0})$. Let $H$ and $A$
  be as in
  Lemma \ref{lemm:17}. In our case, $H$ is a morphism of lattices and
  \begin{displaymath}
    e\phiK =A^{\ast}\Psi _{\Delta ^{r}}+m_{0}+l_{0}.
  \end{displaymath}
  We follow examples \ref{exm:10}, \ref{exm:9}, \ref{exm:26} and
  \ref{exm:23}, and consider $\P^{r}_{S}$ as a toric scheme over $S$.
  Let $p=(p_{0}:\dots:p_{r})$ be a rational point in the principal open
  subset of $\P^{r}_{K}$ such that $\val_{K}(p)=\boldsymbol{\alpha
  }$. Observe that, in this example, the map $\val_{K}$ from the set of
  rational points of the principal open subset of
  $\P^{r}_{K}$ to $N$ (Definition \ref{def:32}) is given explicitly by
  the formula
  \begin{displaymath}
    \val_{K}(p_{0}:\dots:p_{r})=(\val_{K}(p_{1}/p_{0}),\dots,\val_{K}(p_{r}/p_{0})).
  \end{displaymath}
  One can verify that the hypothesis of Proposition \ref{prop:45} are
  satisfied. 
  Let
  $\Phi _{p,A}\colon  \cX_{\Pi }\to \P^{r}_{S}$ be the associated morphism. 
  Then 
  \begin{displaymath}
    D_{e\phiK }=\Phi_{p,A}^{\ast}
    D_{\Psi _{\Delta ^{r}}}+\div(\varpi ^{-l_{0}}\chi ^{-m_{0}}).
  \end{displaymath}
\end{exmpl}


\chapter{Metrics and measures on toric varieties} 
\label{sec:appr-integr-metr}

In this chapter, we study the metrics on a toric line bundle that are
invariant under the action of the compact torus.  Our aim is to obtain
a characterization, in terms of convex analysis, of semipositive toric
metrics and of their associated measures.

We set the notation for  most of this chapter. Let $K$ be
either $\R$, $\C$ or a field which is complete with respect to a
non-Archimedean absolute value. In the non-Archimedean case, we will
use the notations of \S \ref{sec:algebr-metr-discr}, although, for the
moment, we do not assume that the absolute value is associated to a
discrete valuation. 

Let $\T$ be an $n$-dimensional split torus over $K$. Set $N$ and
$M=N^{\vee}$ for the corresponding lattices and let $\Sigma $ be a fan
in $N_{\R}$ as in  \S\ref{Toric varieties}. For each cone $\sigma \in
\Sigma$, we will denote by $X_{\sigma}$ the corresponding affine toric
variety and by $X_{\sigma }^{\an}$ its analytification. These spaces
glue together into a toric variety $X_{\Sigma}$ and an analytic space
$X_{\Sigma}^{\an}$, respectively.  When $K=\C$, the latter  agrees with the
complex analytic space $X_{\Sigma }(\C)$ whereas in the
non-Archimedean case, it is a Berkovich space.  When $K=\R$, we will
use the technique of Remark~\ref{rem:17} to reduce the study of
$X_{\Sigma}^{\an}$ to that of the associated complex analytic space.

\section{The variety with corners $X_{\Sigma }(\R_{\ge 0})$}
\label{sec:semi-analyt-vari}

The variety with corners associated to the fan $\Sigma$ is a partial
compactification of $N_{\R}$, and can be seen as a real analogue of
the toric variety $X_{\Sigma}$. It was introduced by Mumford in
\cite{AMRT:sclsv}. More recently, it has also appeared in the context
of tropical
geometry as the \og extended tropicalization\fg{} of
\cite{Kashiwara:ttg} and \cite{Payne:alt}.

With notations as above, for a cone $\sigma \in \Sigma $ we set
\begin{displaymath}
X_{\sigma }(\R_{\ge 0})=\Hom_{\sg}(M_{\sigma },(\R_{\ge 0},\times)).
\end{displaymath}
On $X_{\sigma }(\R_{\ge 0})$, we put the coarsest topology such that,
for each $m\in M_{\sigma }$, the map $X_{\sigma }(\R_{\ge 0})\to
\R_{\ge 0}$ given by $\gamma \mapsto \gamma (m)$ is continuous. 
Using \cite[\S 1.2, Proposition~2]{Ful93}, we can see that, if $\tau $
is a face of $\sigma $, then there is a dense
open immersion
$X_{\tau }(\R_{\ge 0})\hookrightarrow X_{\sigma }(\R_{\ge 0})$. Hence
the topological spaces $X_{\sigma }(\R_{\ge 0})$ glue together to
define a topological space $X_{\Sigma }(\R_{\ge 0})$.
\nomenclature[aX14]{$X_{\Sigma }(\R_{\ge 0})$}{variety with corners associated to a toric variety}%
This is the \emph{variety with corners} associated to $X_\Sigma$.
\index{variety with corners associated to a toric variety}%
Analogously to
the algebraic case, one can prove that 
this topological space is Hausdorff and that the spaces $X_{\sigma
}(\R_{\ge 0})$ can be identified with open subspaces of $X_{\Sigma
}(\R_{\ge 0})$ satisfying, for~$\sigma,\sigma'\in \Sigma$, 
\begin{displaymath}
  X_{\sigma }(\R_{\ge 0})\cap X_{\sigma' }(\R_{\ge 0})=X_{\sigma\cap
    \sigma ' }(\R_{\ge 0}).  
\end{displaymath}
We have that
\begin{displaymath}
  \T(\R_{\ge 0}):=\Hom_{\sg}(M, \R_{\ge0})=\Hom_{\gp}(M, \R_{>0})\simeq (\R_{> 0})^{n}
\end{displaymath}
is a topological Abelian group that acts on
$X_{\Sigma }(\R_{\ge0})$.

 For each $\sigma \in \Sigma $ there is a continuous
map $\rho _{\sigma } \colon  X^{\an}_{\sigma }\to X_{\sigma }(\R_{\ge
  0})$. This map is given, in the 
Archimedean case, by
\begin{displaymath}
  X^{\an}_{\sigma }=\Hom_{\sg}(M_{\sigma },(\C,\times))
  \overset{|\cdot|}{\longrightarrow} \Hom_{\sg}(M_{\sigma
  },(\R_{\ge 0},\times))= X_{\sigma }(\R_{\ge 0}).
\end{displaymath}
In the non-Archimedean case, since a point $p\in X_{\sigma }^{\an}$
corresponds to a multiplicative seminorm on $K[M_{\sigma }]$ and a
point in $X_{\sigma }(\R_{\ge 0})$ corresponds to a semigroup
homomorphism from $M_{\sigma }$ to $(\R_{\ge 0},\times)$, we define 
$\rho _{\sigma } (p)$  
as the semigroup homomorphism  that to an element $m\in
M_{\sigma }$ corresponds $|\chi ^{m}(p)|$. 

In both cases, these maps glue together to
define a continuous map
\nomenclature[g1703]{$\rho_{\Sigma} $}{projection of an analytic toric
  variety onto a variety with corners}%
\begin{equation}
  \label{eq:111}
 \rho  _{\Sigma }:X_{\Sigma }^{\an} \to
X_{\Sigma }(\R_{\ge  0}). 
\end{equation}
\begin{lem}\label{lemm:1}
For $\sigma\in \Sigma$, the map $\rho _{\Sigma }$ satisfies $\rho _{\Sigma }^{-1}(X_{\sigma
  }(\R_{\ge 0}))=X_{\sigma }^{\an}$.
\end{lem}
\begin{proof}
  By construction, $X_{\sigma }^{\an}\subset \rho  _{\Sigma }
  ^{-1}(X_{\sigma }(\R_{\ge 0}))$. For the reverse inclusion we will
  write 
  only the non-Archimedean case. Assume that $p\in \rho  _{\Sigma }
  ^{-1}(X_{\sigma }(\R_{\ge 0}))$. There is a cone $\sigma '$ with $p\in
  X_{\sigma '}^{\an}$. Let $\tau =\sigma\cap \sigma'$ be the common
  face. Then $p$ is 
  a multiplicative 
  seminorm of $K[M_{\sigma' }]$ and we show next that it can be
  extended to a multiplicative seminorm of $K[M_{\tau }]$. By
  \cite[\S 1.2 Proposition 2]{Ful93}  there is an element $u\in M_{\sigma'
  }$ such that $M_{\tau }=M_{\sigma' }+\Z_{\ge 0}(-u)$. Hence
  $K[M_{\tau }]=K[M_{\sigma' }+\Z_{\ge 0}(-u)]$. Since $\rho  _{\Sigma
  } (p)\in X_{\tau 
  }(\R_{\ge 0})$ we have that $|\chi ^{u}(p)|\not = 0$. Therefore $p$
  extends to a multiplicative seminorm of $K[M_{\tau }]$. Hence $p\in
  X^{\an }_{\tau} \subset X^{\an}_{\sigma }$.
\end{proof}

\begin{cor}\label{cor:29}
  The map $\rho _{\Sigma }\colon X_{\Sigma }^{\an}\to X_{\Sigma
  }(\R_{\ge 0})$ is proper.
\end{cor}
\begin{proof}
When $\Sigma $ is complete, the analytic space $X_{\Sigma }^{\an}$ is
compact. 
Since $X_{\sigma
  }(\R_{\ge 0})$ is Hausdorff, 
the map $\rho  _{\Sigma }$ is proper. 
Assume now that $\Sigma $ is not necessarily complete. Let $\sigma \in
\Sigma $ be a cone. Let $\Sigma '$ be a complete fan that contains
$\sigma $.  By Lemma~\ref{lemm:1}, the fact that $\rho _{\Sigma '}$ is proper
implies that $\rho _{\sigma }$
is proper. Since the condition of being proper is local on $X_{\Sigma
  }(\R_{\ge 0})$, the fact that $\rho _{\sigma }$ is proper for all
  cones $\sigma \in \Sigma $ implies that $\rho _{\Sigma }$ is proper.
\end{proof}

We denote by $\ee\colon \R \to \R_{>0}$ the map $u\mapsto
\exp(-u)$. This map induces a homeomorphism
\begin{displaymath}
  N_{\R}=\Hom(M,\R)\longrightarrow \Hom_{\sg}(M,(\R_{\ge
    0},\times))=X_{0}(\R_{\ge 0})
\end{displaymath}
that we also denote by $\ee$.
\nomenclature[ae]{$\ee$}{parameterization of a variety with corners}%

We define a valuation map in both Archimedean and non-Archimedean cases, by setting, for $\alpha\in
K^{\times}$, \index{valuation map!of an Archimedean field}\nomenclature[aval1]{$\val$}{valuation map}%
\begin{displaymath}
  \val(\alpha)= -\log
  |\alpha|.
\end{displaymath}

Next, we define a map $\val\colon X_{0}^{\an}\to N_{\R}$.
For $p\in X^{\an}_{0}$, we denote by
$\val(p)\in \Hom_{\sg}(M,\R)=N_{\R}$ the morphism given by 
\begin{equation}
    \label{eq:42}
    m\longmapsto \langle m, \val(p)\rangle = -\log
  |\chi^{m}(p)|.
\end{equation}
Then, there is a commutative diagram
\begin{equation}\label{eq:82}
  \xymatrix{ & X_{0}^{\an}\ar[dl]_{\val} \ar[d] ^{\rho_{0}}\\
    N_{\R}\ar[r]_{\ee}&X_{0}(\R_{\ge 0}).
  }
\end{equation}

When $K$ is non-Archimedean and the associated valuation
is discrete, we set
\begin{equation}\label{eq:95}
  \lambda _{K}=
    -\log|\varpi |, \quad  \ee_{K}(u)=\ee(\lambda _{K}u), \quad 
\val_{K}(p)=\frac{\val(p)}{\lambda _{K} }
\end{equation}
for $u\in N_{\R}$ and $p\in X_{0}^{\an}$. 
\nomenclature[g11k]{$\lambda_K$}{scalar associated to a non-Archimedean local field}%
\nomenclature[aea]{$\ee_K$}{normalized parameterization of a variety with corners}%
\nomenclature[aval1]{$\val_K$}{normalized valuation map}%
This latter map extends
the map $\val_{K}\colon \T(K)\to N$ of Definition \ref{def:32}. In
order to make some statements more compact, if $K$ is Archimedean, we
will write $\lambda _{K}=1$,  $\ee_{K}=\ee$ and $\val_{K}=\val$.

\begin{rem} \label{rem:34}
  The map $\val$ only depends on the
  absolute value and is invariant 
  under valued field extensions. It can be defined for arbitrary
  valued fields. The map $\val_{K}$ is the valuation normalized with
  respect to the field $K$. It is only
  defined for discrete valuations. The advantage of $\val_K$ is that
  the image of a rational point belongs to the lattice, that is
  \begin{displaymath}
    \val_{K}(X_{0}(K))\subset N.
  \end{displaymath}
\end{rem}

The map $\ee$ allows us to see $X_{\Sigma }(\R_{\ge 0})$ as a
partial compactification on $N_{\R}$.
Following \cite[Chapter I, \S 1]{AMRT:sclsv} we can give another
description of the topology of $X_{\Sigma }(\R_{\ge 0})$.
For $\sigma \in \Sigma $, we consider the set
\nomenclature[aNl14]{$N_{\sigma}$}{compactification of $N_{\R}$ with respect to a cone}%
\begin{equation}\label{eq:48}
  N_{\sigma }=\coprod_{\tau \text{ face of }\sigma  }N(\tau)_\R,
\end{equation}
where $N(\tau )$ denotes the lattice introduced in (\ref{eq:53}). 

We first extend the map $\ee$ to a map $N_{\sigma }\to
X_{\sigma }(\R_{\ge 0})$. For $\tau $ a face of $\sigma $, we consider
the semigroup $M(\tau )\cup \{-\infty\}$. Each group homomorphisms
$M(\tau )\to (\R_{>0},\times)$ can be extended to a semigroup morphism 
$M(\tau )\cup\{-\infty\}\to (\R_{\ge 0},\times)$ by sending $-\infty$ to $0$.
There is a morphism of semigroups $M_{\sigma
}\to M(\tau )\cup \{-\infty\}$ given by
\begin{displaymath}
  m\longmapsto
  \begin{cases}
    m&\text{ if }m\in \tau ^{\perp},\\
    -\infty&\text{ otherwise,}
  \end{cases}
\end{displaymath}
that induces an injective map 
$$N(\tau )_{\R}\overset{\ee}{\longrightarrow}\Hom_{\gp}(M(\tau
),(\R_{> 0},\times))\longrightarrow 
\Hom_{\sg}(M_{\sigma },(\R_{\ge 0},\times))=X_{\sigma }(\R_{\ge 0}).
$$
Glueing together these maps for all faces of $\sigma $ we obtain the map
$\ee\colon N_{\sigma }\to
X_{\sigma }(\R_{\ge 0})$. One may verify that this map is a
bijection. 

We
next define a topology on $N_{\sigma }$. To this end,
we choose a positive definite bilinear pairing
in $N_{\R}$. Hence we can identify the quotient spaces $N(\tau
)_\R$ with subspaces of $N_{\R}$, that, for simplicity, we will
denote also by $N(\tau )_{\R}$.
For a point $u\in N(\tau)_\R$, let $U\subset N(\tau)_\R$
be a neighbourhood of $u$. For each $\tau '$ face of $\tau $, $\tau $
induces a cone $\pi _{\tau '}(\tau )$ contained in $N(\tau
')_{\R}$. If $p\in \tau$, then 
its image $\pi _{\tau '}(p)$ in $N(\tau ')_{\R}$ is contained in $\pi
_{\tau '}(\tau )$. We write  
\begin{equation}\label{eq:20}
  W(\tau ,U,p)=\coprod _{\tau ' \text { face of }\tau }\pi _{\tau ' }(U+p+\tau).
\end{equation}
The topology of $N_{\sigma }$ is defined by the fact that $\{W(\tau
,U,p)\}_{U,p}$ is a basis of neighbourhoods of 
$u$ in $N_{\sigma }$. With this topology, the map $\ee\colon N_{\sigma }\to
X_{\sigma }(\R_{\ge 0})$ is a  
homeomorphism. 
\nomenclature[aNzl14]{$N_{\Sigma}$}{compactification of $N_{\R}$ with respect to a fan}%

We write
\begin{displaymath}
  N_{\Sigma }=\coprod_{\sigma \in \Sigma  }N(\sigma )_\R,
\end{displaymath}
and put on $ N_{\Sigma }$ the topology that makes $\{N_{\sigma
}\}_{\sigma \in \Sigma }$ an open cover. Then the map
$\ee$ 
extends to a homeomorphism between $N_{\Sigma }$ and $X_{\Sigma
}(\R_{\ge 0})$ and the map $\val$ in (\ref{eq:82}) extends to a proper continuous
map $\val\colon X^{\an}_{\Sigma }\to N_{\Sigma }$ such that the diagram
\begin{equation}\label{eq:89}
  \xymatrix{ & X_{\Sigma }^{\an}\ar[dl]_{\val} \ar[d] ^{\rho_{\Sigma }}\\
    N_{\Sigma }\ar[r]_{\ee}&X_{\Sigma }(\R_{\ge 0})
  }
\end{equation}
is commutative.

\begin{rem}
In case we are given a strictly concave support function $\Psi$ on  a
complete fan $\Sigma$, then $N_\Sigma$ is homeomorphic to the polytope
$\Delta_\Psi$ introduced in \S \ref{pos_divisors}. An
homeomorphism is obtained as the composition of $\ee$ with the
\nomenclature[g12]{$\mu $}{moment map}%
moment map\index{moment map} $\mu\colon X_{\Sigma}(\R_{\ge
  0})\rightarrow \Delta_\Psi$ induced by $\Psi$: 
$$\begin{matrix}
N_{\Sigma} &\buildrel{\ee}\over{\longrightarrow} &X_{\Sigma}(\R_{\ge 0}) &\buildrel{\mu}\over\longrightarrow &\Delta_\Psi\\[2mm]
u &\longmapsto &\ee(u) &\longmapsto &\frac{\sum\exp(-\langle m,u\rangle)m}{\sum\exp(-\langle m,u\rangle)}
\end{matrix}$$
where the sums in the last expression are over the elements $m\in
M\cap \Delta _{\Psi }$.
\end{rem}

We end this section stating the functorial properties of the space
$X_{\Sigma }(\R_{\ge 0})$. We start by studying field extensions.
Assume that $K$ is non-Archimedean and let $K'$ be a complete valued
field extension of $K$. As explained before Definition \ref{def:75},
there is a map $\nu\colon X^{\an}_{\Sigma,K' }\to X^{\an}_{\Sigma,K }$.
\begin{prop}\label{prop:113}
  The diagram
  \begin{displaymath}
    \xymatrix{X_{\Sigma,K' }^{\an}\ar[d]_{\val} \ar[dr] ^(.7){\rho_{\Sigma,K'
        }}|\hole  \ar[r]^{\nu} & X_{\Sigma,K }^{\an}\ar[dl]_(.7){\val} \ar[d]
      ^{\rho_{\Sigma,K }}\\
    N_{\Sigma }\ar[r]_{\ee}&X_{\Sigma }(\R_{\ge 0})
    }
  \end{displaymath}
  is commutative. 
\end{prop}
\begin{proof}
  The commutativity of the diagram follows from the fact that the map
  $X^{\an}_{\Sigma,K' }\to X^{\an}_{\Sigma,K }$ is given by
  restricting seminorms. 
\end{proof}

We next study the inclusion of closed orbits.
Let $N$ and $\Sigma $ be as before and
  $\sigma \in \Sigma $. Recall that there is an induced fan $\Sigma
  (\sigma )$ in $N(\sigma )_{\R}$ defined in (\ref{eq:62}) and a
  closed immersion $\iota_{\sigma }\colon X_{\Sigma (\sigma )}\to
  X_{\Sigma }$ defined before Proposition \ref{prop:73}. We will also
  denote by $\iota _{\sigma }$ the induced 
  morphism of analytic spaces. The map $\ee$ gives us an
  homeomorphism $N(\sigma )_{\R}\to X_{\Sigma (\sigma ),0}(\R_{\ge 0})$. Hence
  the natural map $N(\sigma
  )_{\R}\hookrightarrow N_{\sigma }$ induces an inclusion  $X_{\Sigma
    (\sigma ),0}(\R_{\ge 0})\to X_{\sigma }(\R_{\ge 0})$.

\begin{prop}\label{prop:55}   The inclusion  $X_{\Sigma
    (\sigma ),0}(\R_{\ge 0})\to X_{\sigma }(\R_{\ge 0})$ extends to a
  continuous map $\iota_{\sigma }\colon X_{\Sigma (\sigma
    )}(\R_{\ge 0})\to X_{\Sigma }(\R_{\ge 0})$. Moreover, there is a
  commutative diagram
  \begin{displaymath}
    \xymatrix{
      X_{\Sigma (\sigma )}^{\an} \ar[r]^{\iota_{\sigma }}\ar[d]_{\rho _{\Sigma (\sigma
          )}}& 
      X_{\Sigma }^{\an}\ar[d]^{\rho _{\Sigma}}\\
      X_{\Sigma (\sigma )}(\R_{\ge 0}) \ar[r]^{\iota_{\sigma }}& X_{\Sigma }(\R_{\ge 0}).
    }
  \end{displaymath}
\end{prop}
\begin{proof} To construct the map $\iota _{\sigma }$ at the level of
  varieties with corners one can imitate the construction of the
  morphism $\iota _{\sigma }$ given before Proposition
  \ref{prop:73}. It is possible to verify that the obtained map is
  continuous. To prove the commutativity of the diagram, it is enough
  to restrict oneself to the principal open affine subset $X_{\Sigma
    (\sigma ),0}^{\an}$ or $X_{\Sigma (\sigma ),0}(\R_{\ge 0})$, where
  it follows from the concrete description of points either as
  multiplicative seminorms or as semigroup homomorphisms. We leave to
  the reader the verification of the details.
\end{proof}

\begin{notn} \label{def:90} Let $N_{i}$ and $\Sigma _{i}$ be a
  complete fan in $N_{i,\R}$, $i=1,2$. Let $H\colon N_{1}\to N_{2}$ be
  a linear map such that, for each cone $\sigma_{1} \in \Sigma _{1}$,
  there is a cone $\sigma _{2}\in \Sigma _{2}$ with $H(\sigma
  _{1})\subset \sigma _{2}$. Let $p\in X_{\Sigma_{2},0 }(K)$ and
  $A\colon N_{1,\R}\to N_{2,\R}$ the affine map $A=H+\val(p)$. By
  Theorem \ref{thm:25} there is an equivariant morphism
  $\varphi_{p,H}\colon X_{\Sigma_{1}}\to X_{\Sigma_{2}}$. We denote
  also by $\varphi_{p,H}$ both, the corresponding morphism between
  analytic spaces, and the map
\begin{equation}\label{eq:123}
 \varphi_{p,H}=\ee\circ A\circ \ee^{-1}\colon X_{\Sigma_{1},0}(\R_{\ge 0})\to
  X_{\Sigma_{2},0}(\R_{\ge 0}),
\end{equation}
even though  the latter depends only on $\val(p)$ and $H$ and not on $p$ itself.
\end{notn}

The proof of the following
proposition is left to the reader.

\begin{prop}\label{prop:56} The map (\ref{eq:123})
  extends to a continuous map
  $X_{\Sigma_{1}}(\R_{\ge 0})\to
  X_{\Sigma_{2}}(\R_{\ge 0})$ that we also denote by
  $\varphi_{p,H}$. Moreover, there is a commutative diagram 
  \begin{displaymath}
    \xymatrix{
      X_{\Sigma _{1}}^{\an} \ar[r]^{\varphi_{p,H}}\ar[d]_{\rho _{\Sigma_{1}}}& 
      X_{\Sigma_{2} }^{\an}\ar[d]^{\rho _{\Sigma_{2}}}\\
      X_{\Sigma _{1}}(\R_{\ge 0}) \ar[r]^{\varphi_{p,H}}& X_{\Sigma _{2}}(\R_{\ge 0}).
    }
  \end{displaymath} 
\end{prop}

\section{Analytic torus actions}
\label{sec:non-arch-analyt}

When $K$ is Archimedean, the analytic torus
$\T^{\an}\simeq(\C^{\times})^{n}$ is a group which acts on the
analytic toric variety $X_{\Sigma}^{\an}=X_{\Sigma}(\C)$. 
The \emph{compact torus} \index{compact torus}
\nomenclature[aSMan]{$\SS $}{compact torus (Archimedean case)}%
 of $\T^{\an}$ is defined as the subset 
\begin{displaymath}
  \SS=\{p\in \T^{\an}\mid |\chi^{m}(p)|=1 \text{ for all } m\in M\}.
\end{displaymath}
It is a compact topological subgroup of $\T^{\an}$, homeomorphic to
$(S^{1})^{n}$ and which
has a Haar measure of total volume $1$.
  The map
$\rho_{\Sigma}$ defined in (\ref{eq:111})
is equivariant, in the sense that, for all $t\in\T^{\an}$ and $p\in X_{\Sigma}^{\an}$, 
\begin{displaymath}
  \rho_{\Sigma}(t\cdot p)= \rho_{0}(t)\cdot\rho_{\Sigma}(p).
\end{displaymath}
The orbits of the action of $\SS$ on $X_{\Sigma}^{\an}$ agree with the
fibers of the map $\rho_{\Sigma}$ defined in~(\ref{eq:111}): for a
point $p\in X_{\Sigma}^{\an}$,
\begin{equation}\label{eq:138}
  \SS\cdot p= \rho_{\Sigma}^{-1}(\rho_{\Sigma}(p)). 
\end{equation}
Therefore the variety with corners $X_{\Sigma }(\R_{\ge 0})$ can be
understood as the quotient of $X_{\Sigma }^{\an}$ by the action of the
closed subgroup $\SS $. Since the map $\rho _{\Sigma }$ is
proper, the topology of $X_{\Sigma }(\R_{\ge 0})$ is the final
topology with respect to this map.

In the non-Archimedean case, the analogues of these properties are more
subtle. For the remainder of the section we will assume that $K$ is
non-Archimedean. 
Following \cite[Chapter 5]{Berkovich:stag}, an \emph{analytic group} \index{analytic group}
$G$ over $K$ is an analytic space over $K$ endowed with three
morphisms $G\times G\to G$ (multiplication), $\Spec(K)^{\an}\to G$
(identity) and $G\to G$ (inversion), satisfying the natural
conditions. 

An \emph{action} \index{analytic action} of $G$ on an analytic space $X$ over $K$ is a morphism
\begin{displaymath}
  \mu\colon G\times X\longrightarrow X,
\end{displaymath}
also satisfying the natural conditions in this context. 
\nomenclature[g12mu]{$\mu $}{action of an analytic group on an analytic space}%

The rational points $G(K)$ form an abstract group but, in general, the
set of points of the topological space underlying the analytic space $G$ has no natural group
structure induced by the analytic group structure. Instead, we can
define a correspondence that associates, 
to $g\in G$ and $p\in X$, the subset of points
\begin{displaymath}
  g\cdot p=\mu (\pr^{-1}(g,p)),
\end{displaymath}
where $\pr\colon G\times X\to G\times_{\mbox{\tiny topo}}X$ is the projection induced
from the functorial properties of the direct product of sets, the first product being in the category of analytic spaces whereas the second is in the category of topological spaces. The set
$g\cdot p$ may contain more than one point as shown in Corollary
\ref{cor:31}, for example. 
This \og  multiplication\fg{} of points satisfies the properties that,
for all $g,h\in G$ and $p,q\in X$, 
\begin{align}
g\cdot(h\cdot p)&=
(g\cdot h)\cdot p,\label{eq:128}\\ 
p\in g\cdot q &\Longleftrightarrow q\in g^{-1}\cdot p,\label{eq:135}  
\end{align}
where $g^{-1}$ denotes the image of $g$ by the inversion map
\cite[Proposition~5.1.1(i)]{Berkovich:stag}. 
If either $g\in G(K)$ or $p\in X(K)$, then $g\cdot p$ consists of a
single point.  

A non-empty subset $H\subset G$ is a subgroup if it satisfies that, for
all $g,h\in H$, $g^{-1}\in H$ and $g\cdot h \subset H$.
For a subgroup $H$ and a point $x\in X$, the orbit of $p$ with respect
to $H$ is defined as the subset 
$$H\cdot p=\bigcup_{h\in H}h\cdot p .$$
By (\ref{eq:135}), different  orbits are either disjoint or coincide. 

Although we have defined $g\cdot p$ as a set, for some special
elements of $G$ or $X$  
we can single out a distinguished point of this subset with good
properties. 

Let $K'$ and $K''$ be two complete valued field extensions of $K$, recall that
the tensor product $K'\otimes_K K''$ has a tensor product norm defined as
\begin{displaymath}
  \|\gamma \|= \inf_{\gamma =\sum_{i}\alpha _{i}\otimes \beta
    _{i}}\max_{i}|\alpha _{i}||\beta _{i}|. 
\end{displaymath}
Then $K'\wh\otimes K''$ is defined as the completion of $K'\otimes_K
K''$ with respect to this norm.

\begin{defn}\label{def:81} Let $Z$ be an analytic space over $K$. A
  point $p\in Z$ is called \emph{peaked} if, for any complete valued
  field \index{peaked point}
  extension $K'$ of $K$, the tensor product norm of $\mathscr{H}(p)\wh
  \otimes K'$ is multiplicative.
\end{defn}

Let $g\in  G$ and $p\in X$. The set $\pr^{-1}(g, p)$ can be identified with
the set of multiplicative seminorms of $\mathscr{H}(g)\wh
\otimes \mathscr{H}(p) $ that are bounded by the tensor product
norm. 

\begin{defn}\label{def:85}
  Let $g\in G$ and $p\in X$. It follows from the definition that if
  one of these points is peaked, then the tensor product norm of
  $\mathscr{H}(g)\wh \otimes \mathscr{H}(p) $ is multiplicative, and
  so it defines a point of $\pr^{-1}(g, p)\in G\times X$. We denote by
  $g\ast p\in X$ the image by $\mu$ of this point.
\end{defn}
\nomenclature[sast]{$\ast$}{star product with a peaked point}%

\begin{rem}\label{rem:33}
  Assume that $G$ and $X$ are the analytification of an affine algebraic
  group $\Spec(A)$ and of an affine algebraic variety $\Spec(B)$ over $K$
  respectively. Assume also that the action is induced by a morphism
  $B\to A\otimes B$. Let 
  $g\in G$ and $p\in X$. If either $g$ or $p$ is peaked,
  then the point $g\ast p\in X$ is given by the multiplicative seminorm of $B$
  induced by the tensor product norm of $\mathscr{H}(g)\wh\otimes
  \mathscr{H}(p)$ through the composition  $B\to A\otimes B\to
  \mathscr{H}(g)\wh\otimes 
  \mathscr{H}(p)$.  
\end{rem}

\begin{prop}\label{prop:109} Let $G$ be an analytic group and $X$ an
  analytic space with an action of $G$.
  \begin{enumerate}
  \item \label{item:129} The points of $G(K)$ and of $X(K)$ are
    peaked. If either $g\in 
    G$ or $p\in X$ is rational, then $g\cdot p=\{g\ast p\}$.
  \item \label{item:130}If $g\in
    G$ and $p\in X$ are peaked, then $g\ast p$ is peaked.
  \item \label{item:131}If two of the three points $g_{1}\in
    G$, $g_{2}\in
    G$ and $p\in X$ are peaked, then $(g_{1}\ast g_{2})\ast
    p=g_{1}\ast (g_{2}\ast p)$.
  \item \label{item:132}If $g\in G$ is peaked, then the map $X\to X$ given by
    $p\mapsto g\ast p$ is continuous. If $p\in X$ is peaked then the map $G\to X$ given by
    $g\mapsto g\ast p$ is continuous.
  \end{enumerate}
\end{prop}
\begin{proof}
  The first statement follows directly from the definition. The
  remaining statements are proved in \cite[Proposition 5.2.8]{Berkovich:stag}.
\end{proof}

\begin{prop}\label{prop:111}
  Let $\varphi\colon Y\to Z$ be a closed immersion of algebraic varieties over
  $K$. Then $p\in Y^{\an}$ is peaked if and only if $\varphi^{\an}(p)$ is peaked.
\end{prop}
\begin{proof}
  Since $\varphi$ is a closed immersion, we have that
  $\mathscr{H}(\varphi^{\an}(p))=\mathscr{H}(p)$, which implies the result. 
\end{proof}

The example of interest for us is when $G$ and $X$ are the
analytification of a
split algebraic torus and an algebraic toric variety
over $K$, respectively.   
Let notations be as at the beginning of this chapter and assume that
$K$ is non-Archimedean.
Then the analytic torus $\T^{\an}$ is an analytic group as above.

The map $\rho _{\sigma }$ is equivariant in the following sense.

\begin{prop}\label{prop:105} Let $t\in \T^{\an}$ and $p\in X_{\Sigma
  }^{\an}$. Then
  \begin{displaymath}
    \rho _{\Sigma }(t\cdot p)=\rho _{0}(t)\cdot \rho _{\Sigma }(p).
  \end{displaymath}
\end{prop}
\begin{proof}
  We can assume that $p\in X^{\an}_{\sigma }$ for a cone $\sigma \in \Sigma
  $. The set $\pr^{-1}(t,p)$ is the set of multiplicative seminorms of
  $K[M]\otimes K[M_{\sigma }]$ that extend the absolute value of $K$
  and that satisfy, for $f\in K[M]$ and $g\in K[M_{\sigma }]$,
  \begin{equation}
    \label{eq:127}
   |f\otimes 1|=|f(t)|, \quad |1\otimes g|=|g(p)|. 
  \end{equation}
  Therefore, if $m\in
  M_{\sigma }$, and $q\in t\cdot p$ is the image by $\mu $ of a multiplicative
  seminorm $|\cdot|_{q}$ of $K[M]\otimes K[M_{\sigma }]$ satisfying
  (\ref{eq:127}), 
  then
  \begin{displaymath}
    \rho _{\sigma }(q)(m)=|\chi^{m}(q)|=|\chi^{m}\otimes
    \chi^{m}|_{q}=|(\chi^{m}\otimes 1) (1\otimes
    \chi^{m})|_{q}
  \end{displaymath}
  By the multiplicativity of $|\cdot|_{q}$
  \begin{displaymath}
    \rho _{\sigma }(q)(m)=|(\chi^{m}\otimes 1)|_{q}|(1\otimes
    \chi^{m})|_{q}.
  \end{displaymath}
  By (\ref{eq:127}),
  \begin{displaymath}
    \rho _{\sigma }(q)(m)=|\chi^{m}(t)||\chi^{m}(p)|=\rho _{0}(t)(m)\rho
    _{\sigma }(p)(m)= (\rho _{0}(t)\cdot \rho _{\sigma }(p))(m),
  \end{displaymath}
  proving the result.
\end{proof}

The compact torus $\SS\subset \T^{\an}$
is a subgroup in the analytic sense and its 
underlying topological space is compact. However,
as discussed previously it is not an abstract group. Thus we cannot apply
the theory of locally compact topological groups to obtain a Haar
measure on $\SS$. The role of the Haar measure of $\SS$ will be played by a
Dirac delta measure centred at a special point of $\SS$. 

\begin{defn}\label{def:86} The \emph{Gauss norm} of $K[M]$ is the norm 
 given, for $f=\sum \alpha _{m}\chi ^{m}\in K[M]$, by
  \begin{math}
    \max_{m}|\alpha _{m}|.
  \end{math}
\index{Gauss norm}
\end{defn}

The following result is classical.

\begin{prop}\label{prop:104}
    The Gauss norm is multiplicative. 
\end{prop}
\begin{proof}
Let
  $f=\sum_{m} \alpha _{m}\chi^{m}$ and $g=\sum_{l} \beta
  _{l}\chi^{l}$ and write $fg=\sum_{k} \epsilon  _{k}\chi^{k}$ with
  $\epsilon _{k}=\sum_{m+l=k}\alpha _{m}\beta _{l}$. Then, since
  the absolute value of $K$ is ultrametric, 
  \begin{displaymath}
    \max_{k\in M_{\sigma }} (|\epsilon _{k}|)\le
    \max_{m\in M_{\sigma }} (|\alpha_{m} |)
    \max_{l\in M_{\sigma }} (|\beta _{l}|). 
  \end{displaymath}
  Let $C_{f}=\{m\in M_{\sigma }|\max_{m'}(|\alpha _{m'}|)=|\alpha
  _{m}|\}$ and define $C_{g}$ analogously. Let
  $r$ be a vertex of the Minkowski sum
  $\Conv(C_{f})+\Conv(C_{g})$. Then there is a unique decomposition
  $r=m_{r}+l_{r}$ with $m_{r}\in C_{f}$ and $l_{r}\in C_{g}$. Hence
  $\epsilon _{r}= \alpha _{m_{r}}\beta _{l_{r}}$. Thus
  \begin{displaymath}
    \max_{k\in M_{\sigma }} (|\epsilon _{k}|)\ge
    |\epsilon _{r}|=
    \max_{m\in M_{\sigma }} (|\alpha_{m} |)
    \max_{l\in M_{\sigma }} (|\beta _{l}|),
  \end{displaymath}
  which concludes the proof.
\end{proof}

\begin{defn}\label{def:87}
    The \emph{Gauss point} of $\T^{\an}$ is the point $\zeta $ corresponding to
  the Gauss norm of $K[M]$. Thus, if $f=\sum \alpha _{m}\chi ^{m}\in
  K[M]$, then \nomenclature[g06Gauss]{$\zeta$}{Gauss point}%
  \begin{displaymath}
    |f(\zeta)|=\max_{m}|\alpha _{m}|.
  \end{displaymath}
  It is clear that $\zeta\in \SS\subset \T^{\an}$. \index{Gauss point}
\end{defn}

The Gauss point satisfies the following invariance
property, that indicates that it is reasonable to consider the Dirac
delta measure $\delta _{\zeta}$ as the non-Archimedean analogue of the
Haar measure on 
$\SS$. 

\begin{prop} \label{prop:97} The Gauss point $\zeta$ is peaked. Moreover, 
for any $t\in \SS$ one has $t\ast\zeta=\zeta$. 
\end{prop}
\begin{proof} 
Let $K'$ be a complete valued field  over $K$. We denote by $\wh{K'[M]}$
the completion of $K'[M]$ with respect to the Gauss norm. Since there
is an isometry
\begin{displaymath}
  \wh{K[M]}\wh{\otimes}K'=\wh{K'[M]},
\end{displaymath}
and the Gauss norm is multiplicative,
then \cite[Lemma 5.2.2]{Berkovich:stag} implies that the Gauss point
is peaked. 

 Let $f=\sum\alpha_{m}\chi^{m}\in
K[M]$. The action of $\T$ on itself is given by the morphism of
algebras $K[M]\to K[M]\otimes K[M]$ that sends $f$ to $\sum \alpha
_{m}\chi^{m}\otimes \chi^{m}$.

For $t\in  \T^{\an}$, by Remark \ref{rem:33}, the value $|f(t\ast
\zeta)|$ is the norm
of the image of $f$ in $\mathscr{H}(\zeta)\wh{\otimes}
\mathscr{H}(t)$. 
Since the map
\begin{displaymath}
  \wh{K[M]}\wh{\otimes} \mathscr{H}(t)\to
  \mathscr{H}(\zeta)\wh{\otimes} \mathscr{H}(t)
\end{displaymath}
is an isometric embedding, it is enough to compute the norm of the
image of $f$ in $\wh{K[M]}\wh{\otimes}
\mathscr{H}(t)=\wh{\mathscr{H}(t)[M]}$. Therefore
\begin{displaymath}
|f(t\ast \zeta)|= \left|\sum\alpha_{m}\chi^{m}(t)\chi^{m}\right|=
\max_{m} (|\alpha_{m}| |\chi^{m}(t)|).  
\end{displaymath}
Assume now that $t\in \SS$. Then
$|\chi^{m}(t)|=1$ for all $m\in M$. Thus $|f(t\ast \zeta)|=\max_{m}
|\alpha_{m}|$ and so $t\ast \zeta=\zeta $.
  \end{proof}

\begin{cor}\label{cor:31} The Gauss point satisfies
  \begin{math}
    \zeta\cdot\zeta=\SS.
  \end{math}
\end{cor}
\begin{proof}
  Since $\SS$ is a subgroup, $\zeta\cdot\zeta\subset \SS$. Let now
  $t\in \SS$. By Proposition \ref{prop:97}, $\zeta=\zeta\ast t\in
  \zeta\cdot t$. By
  (\ref{eq:135}), $t\in \zeta\cdot \zeta^{-1}$. Since
  $\zeta=\zeta^{-1}$, we deduce that $t\in \zeta\cdot \zeta$, proving the
  result. 
\end{proof}

On each fibre of the map $\rho _{\Sigma }$ there is a point
with similar properties to those of the Gauss point, giving a
continuous section of  $\rho _{\Sigma }$.

\begin{prop-def}\label{prop:54} Let $\sigma \in \Sigma $.
  For each $\gamma \in \Hom_{\sg}(M_{\sigma },\R_{\ge 0})$, the
  seminorm that, to a function $\sum \alpha _{m}\chi^{m}\in K[M_{\sigma
  }]$ assigns the value $\max_{m}(|\alpha _{m}|\gamma (m))$, is a
  multiplicative seminorm on $K[M_{\sigma
  }]$ that extends the absolute value of $K$. Therefore it determines a point of
  $X^{\an}_{\sigma }$ that we denote $\theta_{\sigma }(\gamma
  )$. The maps $\theta_{\sigma }$ are injective, continuous and
  proper. Moreover, they glue together to define a map
\nomenclature[g0818]{$\theta _{\Sigma }$}{injection of the variety with corners in the corresponding analytic toric variety}%
  \begin{displaymath}
    \theta _{\Sigma }\colon X_{\Sigma }(\R_{\ge 0})\longrightarrow X_{\Sigma}^{\an}
  \end{displaymath}
  that is a continuous and proper section of $\rho _{\Sigma }$. 
\end{prop-def}
\begin{proof} Let $\gamma \in \Hom_{\sg}(M_{\sigma },\R_{\ge 0})$. The
  fact that the seminorm $\theta_{\sigma }(\gamma
  )$ extends the absolute value of $K$ is clear. That it is multiplicative is
  proved with an argument similar to the one in the proof of
  Proposition \ref{prop:104}. Thus we obtain a point $\theta _{\sigma }(\gamma
  )\in X^{\an}_{\sigma }$. 

  We show next that the map $\theta _{\sigma }$ is continuous. The
  topology of $X^{\an}_{\sigma }$ is the coarsest topology that makes
  the functions $p\mapsto |f(p)|$ continuous for all $f\in K[M_{\sigma
  }]$. Thus to show that $\theta _{\sigma }$ is continuous, it is enough
  to show that the map $\gamma \mapsto |f(\theta _{\sigma }(\gamma ))|$ is
  continuous on $X_{\sigma }(\R_{\ge 0})=\Hom_{\sg}(M_{\sigma
  },\R_{\ge 0})$.
  The topology of $X_{\sigma }(\R_{\ge 0})$ is the
  coarsest topology such that, for each $m\in M_{\sigma }$, the map $\gamma
  \mapsto \gamma (m)$ is continuous. Since, for $f=\sum_{m\in M_{\sigma
    }}\alpha _{m}\chi ^{m}$, 
 $$
  |f(\theta _{\sigma }(\gamma ))|=\max(|\alpha _{m}|\gamma (m)),
 $$
  we deduce that $\theta _{\sigma }$ is continuous. 

  The facts that the maps $\theta _{\sigma }$ glue together to give a
  continuous map $\theta _{\Sigma }$ and that $\theta _{\Sigma }$ is a
  section of $\rho _{\Sigma }$ follow easily from the definitions. The
  analogue of Lemma \ref{lemm:1} is true for $\theta_{\Sigma }$, hence
  $\theta _{\sigma }$ is proper by the argument of the proof of
  Corollary \ref{cor:29}.
\end{proof}

Observe that $\zeta=\theta _{0}(1)$. 
The following result extends Proposition \ref{prop:97} to the points in
the image of $\theta_\Sigma $.
\begin{prop} \label{prop:96}
Let $t\in \T^{\an}$, $p\in X_{\Sigma }^{\an}$,
$\gamma\in X_{\Sigma}(\R_{\ge0})$ and $\tau  \in \T(\R_{\ge
0})$.  Then the points $\theta_{\Sigma}(\gamma)$ and $\theta_{0}(\tau
)$ are peaked and
\begin{displaymath}
t\ast \theta_{\Sigma}(\gamma)=
\theta_{\Sigma}(\rho_{0}(t)\cdot \gamma), \qquad 
\theta _{0}(\tau )\ast p=\theta _{\Sigma }(\tau \cdot \rho _{\Sigma }(p)).
\end{displaymath}
\end{prop}
  
\begin{proof}
Let  $\sigma\in \Sigma$ such that $\gamma\in X_{\sigma}(\R_{\ge0})$. 
By similar arguments as those in the proof of Proposition
\ref{prop:97}, we see that $\theta _{\sigma }(\gamma )$ is peaked and
that, for $f=\sum \alpha _{m}\chi ^{m}\in K[M_{\sigma }]$,
\begin{displaymath}
  |f(t\ast \theta _{\sigma }(\gamma ))|=\max_{m}|\alpha
  _{m}||\chi^{m}(t)|\gamma (m) =
  \max_{m} |\alpha
  _{m}|(\rho _{0}(t)\cdot \gamma) (m)=|f(\theta _{\sigma }(\rho
  _{0}(t)\cdot \gamma))|,
\end{displaymath}
which proves the first formula. The rest of the proposition can be
proved along the same lines. 
  \end{proof}

  A direct consequence of Proposition \ref{prop:96} is the following
  equivariance result for~$\theta _{\Sigma }$. It implies that
  $(\Im(\theta_{0}), *)$ is a topological group acting by $*$ on the
  topological space $\Im(\theta_{\Sigma})$, with an action isomorphic
  to the action of $\T(\R_{\ge0})$ on $X_{\Sigma}(\R_{\ge 0})$.

\begin{cor}\label{cor:32} Let $\tau \in \T(\R_{\ge 0})$ and $\gamma
  \in X_{\Sigma }(\R_{\ge 0})$. Then
  \begin{displaymath}
 \theta _{\Sigma }(\tau \cdot
  \gamma )=\theta _{0}(\tau) \ast \theta _{\Sigma }(\gamma ).   
  \end{displaymath}
\end{cor}

The orbits of the action of
$\SS$ on $X_{\Sigma}^{\an}$ agree with the fibers of the map
$\rho_{\Sigma}$.
\begin{prop}\label{prop:106}
 Let $p\in X_{\Sigma}^{\an}$. Then
\begin{displaymath}
  \SS\cdot p= \rho_{\Sigma}^{-1}(\rho_{\Sigma}(p)). 
\end{displaymath}  
\end{prop}
\begin{proof}
  By Proposition \ref{prop:105}, $\rho _{\Sigma }(\SS\cdot p)=\rho
  _{\Sigma }(p)$ and so $\SS\cdot p\subset
  \rho_{\Sigma}^{-1}(\rho_{\Sigma}(p))$.

  Conversely, let $q\in \rho_{\Sigma}^{-1}(\rho_{\Sigma}(p))$. By Proposition
  \ref{prop:96},  $\zeta \ast p=\zeta \ast q$. Therefore $\SS\cdot
  p\cap \SS\cdot q\not = \emptyset$. Thus, both orbits agree and $q\in
  \SS\cdot p$, concluding the proof.
\end{proof}

The previous proposition shows that, also in the non-Archimedean case, the
space $X_{\Sigma }(\R_{\ge 0})$  can be understood as the quotient of $X_{\Sigma
}^{\an}$ by the action of the closed subgroup $\SS $. Note that,
since the map $\rho _{\Sigma }$ is proper, the topology of $X_{\Sigma
}(\R_{\ge 0})$ is the final topology with respect to this map. 

We next discuss the functorial properties for the map $\theta_{\Sigma} $.
Let $K'$ be a complete valued field extension of $K$, and consider the map
$\nu\colon X^{\an}_{\Sigma,K' }\to X^{\an}_{\Sigma,K }$. 
\begin{prop}\label{prop:114}
  The diagram
  \begin{displaymath}
    \xymatrix{X_{\Sigma,K' }^{\an} \ar[r]^{\nu} & X_{\Sigma,K }^{\an}\\
    &X_{\Sigma }(\R_{\ge 0})\ar[lu]^{\theta _{\Sigma ,K'}}\ar[u]_{\theta _{\Sigma,K }}
    }
  \end{displaymath}
  is commutative. 
\end{prop}
\begin{proof}
  The commutativity of the diagram follows from the fact that the map
  $X^{\an}_{\Sigma,K' }\to X^{\an}_{\Sigma,K }$ is given by
  restricting seminorms. 
\end{proof}

The map $\theta _{\Sigma }$ is compatible with the inclusion of
closure of orbits. The proof of the following proposition is left to
the reader.

\begin{prop} \label{prop:107} With the notations of Proposition
  \ref{prop:55}, there is
  a commutative diagram
  \begin{displaymath}
    \xymatrix{
      X_{\Sigma (\sigma )}^{\an} \ar[r]^{\iota_{\sigma }}& X_{\Sigma }^{\an}\\
      X_{\Sigma (\sigma )}(\R_{\ge 0})\ar[r]^{\iota_{\sigma
        }}\ar[u]^{\theta _{\Sigma(\sigma )}} 
      & X_{\Sigma }(\R_{\ge 0}). \ar[u]_{\theta _{\Sigma}}
    }
  \end{displaymath}  
\end{prop}

In some cases, the map $\theta _{\Sigma }$ is compatible with
equivariant maps.    

\begin{prop}\label{prop:108} With Notation \ref{def:90}, assume that
  the dual linear map $H^{\vee}$ is injective. Then the diagram
  \begin{displaymath}
        \xymatrix{
      X_{\Sigma _{1}}^{\an} \ar[r]^{\varphi_{p,H}}& 
      X_{\Sigma_{2} }^{\an}\\
      X_{\Sigma _{1}}(\R_{\ge 0}) \ar[r]^{\varphi_{p,H}}\ar[u]^{\theta 
        _{\Sigma_{1}}}
      & X_{\Sigma _{2}}(\R_{\ge 0})\ar[u]_{\theta _{\Sigma_{2}}}
    }
  \end{displaymath}
is commutative.
\end{prop}
\begin{proof}
  It is enough to treat the local case. Write $M_{i}$ for the dual
  lattice of $N_{i}$, $i=1,2$, and $H^{\vee}\colon M_{2}\to M_{1}$ for
  the dual of $H$.  Let $\sigma \in \Sigma _{1}$ be a cone, $q\in
  X_{\sigma }^{\an}$, and $f=\sum \alpha _{m} \chi^{m}\in
  K[M_{2,\sigma }]$. Then
  \begin{displaymath}
    |f(\varphi_{p,H}(q))|=\left|\sum_{m\in M_{2}}\alpha
      _{m}\chi^{m}(p)\chi^{H^{\vee}(m)}(q)\right|=
    \left| \sum_{n\in M_{1}}\left(
        \sum_{\substack{m\in M_{2}\\H^{\vee}(m)=n}}\alpha _{m}\chi^{m}(p)\right)\chi^{n}(q)\right|
  \end{displaymath}
  If $\gamma \in X_{\sigma }(\R_{\ge 0})$ and $q=\theta _{\sigma
  }(\gamma )$ then
  \begin{displaymath}
    |f(\varphi_{p,H}(q))|=\max_{n\in M_{1}}\left|
        \sum_{\substack{m\in M_{2}\\H^{\vee}(m)=n}}\alpha
        _{m}\chi^{m}(p)\right|\gamma (n). 
  \end{displaymath}
  Since $H^{\vee}$ is injective
  \begin{displaymath}
        |f(\varphi_{p,H}(q))|=\max_{m\in M_{2}}|
        \alpha
        _{m}|\,|\chi^{m}(p)|\gamma (H^{\vee}(m)).
      \end{displaymath}
  But, by Proposition \ref{prop:56}, $|\chi^{m}(p)|\gamma (H^{\vee}(m))=\rho _{0}(p)(m)\gamma
  (H^{\vee}(m))=\varphi_{p,H}(\gamma )(m)$. Thus
  \begin{displaymath}
    |f(\varphi_{p,H}(\theta_{\Sigma _{1}} (\gamma )))| = \max_{m\in M_2}|\alpha_m|\varphi_{p,H}(\gamma )(m) = |f(\theta
    _{\Sigma _{2}}(\varphi_{p,H}(\gamma )))|, 
  \end{displaymath}
  concluding the proof.
\end{proof}

\begin{cor}\label{cor:33} With Notation \ref{def:90}, for any point $\gamma \in X_{\Sigma
    _{1}}(\R_{\ge 0})$, the point $\varphi_{p,H}(\theta _{\Sigma_{1}
  }(\gamma ))$ is peaked.
\end{cor}
\begin{proof} We first treat the case when $\gamma \in  X_{\Sigma
    _{1},0}(\R_{\ge 0})$. Following (\ref{eq:139}), we factorize
  $\varphi_{p,H}$ as  
  \begin{displaymath}
    X_{\Sigma _{1}}\overset {\varphi_{H_{\surj}}}{\longrightarrow}
    X_{\Sigma _{3}}\overset {\varphi_{H_{\sat}}}{\longrightarrow}
    X_{\Sigma _{4}}\overset {\varphi_{p,H_{\inj}}}{\longrightarrow}
    X_{\Sigma _{2}}.
  \end{displaymath}
  Since $H_{\surj}^{\vee}$ and $H^{\vee}_{\sat}$ are
  injective, by Proposition \ref{prop:108}, we deduce that
  \begin{displaymath}
    \varphi_{H_{\sat}}(\varphi_{H_{\surj}}(\theta_{\Sigma _{1}}(\gamma
    )))=
    \theta _{\Sigma _{4}}(\varphi_{H_{\sat}}(\varphi_{H_{\surj}}(\gamma
    ))).
  \end{displaymath}
  By Proposition \ref{prop:96}, this latter point is peaked. By Proposition
  \ref{prop:110}, the map $\varphi_{p,H_{\inj}}\colon X_{\Sigma
    _{4},0}\to X_{\Sigma _{2},0}$ is a closed immersion. Therefore, by
  Proposition \ref{prop:111}, we deduce that $\varphi_{p,H}(\theta _{\Sigma_{1}
  }(\gamma ))$ is peaked, proving the result in this case.

  The general case follows from the previous one  together with
  propositions \ref{prop:112} and \ref{prop:111}.
\end{proof}

\section{Toric metrics}
\label{sec:equiv-herm-metr}

With the notations at the beginning of this chapter, assume
furthermore that $\Sigma $ is complete. Let $L$ be a toric line bundle
on $X_{\Sigma }$ and $s$ a toric section of $L$
(Definition~\ref{def:71}). By theorems \ref{thm:27} and
\ref{thm:3}, we can find a virtual support function $\Psi$ on $\Sigma
$ such that there is an isomorphism $L\simeq \mathcal{O}(D_{\Psi })$
that sends $s$ to $s_{\Psi }$.  The algebraic line bundle $L$ defines
an analytic line bundle $L^{\an}$ on $X_{\Sigma }^{\an}$. Let $\ov
L=(L,\|\cdot\|)$, where $\|\cdot\|$ is a metric on $L^{ \an}$.

Every {toric} object has a certain invariance property with
respect to the action of $\T$. This is also the case for
metrics. Since $\T^{\an}$ is non compact, we cannot ask for a metric
to be $\T^{\an}$-invariant, but we can impose
$\SS $-invariance. In view of equation (\ref{eq:138}) and Proposition
\ref{prop:106}, a function $f\colon \T^{\an}\to \R$ will be called
\emph{$\SS$-invariant} if it is constant along the fibres of $\rho _{0}$.

We need a preliminary result.  

\begin{prop} \label{prop:79}
  With notations as above, if the function 
  $p\mapsto \|s(p)\|$ is $\SS$-invariant, then, for every toric
  section $s'$, the function $p\mapsto \|s'(p)\|$ is 
  $\,\SS $-invariant too.  
\end{prop}
\begin{proof}
  If $s'$ is a toric section of $L$, then there is an element $m\in M$
  such that $s'=\chi^{m}s$. Since for any element $t\in \SS $ we
  have $|\chi^{m}(t)|=1$, if the function $\|s(p)\|$ is
  $\SS $-invariant, then the function
  \begin{math}
    \|s'(p)\|=\|\chi^{m}(p)s(p)\|
  \end{math} is also $\SS $-invariant.
\end{proof}

\begin{defn}\label{def:44} 
Let $L$ be a toric line bundle on $X_{\Sigma }$. A metric on
$L^{\an}$ is  \emph{toric}
\index{toric metric}\index{metric!toric|see{toric metric}}%
if the function $p\mapsto \|s(p)\|$ is $\SS $-invariant.
\end{defn}

Given an arbitrary metric on a toric line bundle, we can associate to
it a toric metric by an averaging process.

\begin{defn} \label{def:82} Let $L$ be a toric line bundle on
  $X_{\Sigma}$ and $\|\cdot \|$ a metric on $L^{\an}$. 
For $\sigma\in \Sigma$, let $s_{\sigma}$ be a toric section of $L$
which is regular and non-vanishing in $X_{\sigma}$. 

  If $K$ is Archimedean, we set, for $p\in X^{\an}_{\sigma}$,
\begin{displaymath}
\|s_{\sigma}(p)\|_{\SS}= 
\exp\bigg(\int_{\SS }\log \|s_{\sigma}(t \cdot
  p)\|\dd\mu_{\Haar} (t)\bigg),
\end{displaymath}
where $\mu_{\Haar} $ is the Haar measure of $\SS $ of total
volume 1.

If $K$ is non-Archimedean, we set, for $p\in X^{\an}_{\sigma}$,
\begin{displaymath}
\|s_{\sigma}(p)\|_{\SS}= \|s_{\sigma} (\theta_{\Sigma}(\rho_{\Sigma}(p)))\|
\end{displaymath}
where $\rho_{\Sigma}$ is defined in (\ref{eq:111}) and 
$\theta_{\Sigma}$ in Proposition-Definition \ref{prop:54}.

It is easy to verify that these functions define a toric metric
  $\|\cdot\|_{\SS}$ on $L^{\an}$.
\nomenclature[svert53]{$\Vert\cdot\Vert_{\SSinv}$}{toric metric from a metric}%
\end{defn}
Observe that the previous definition is compatible with the idea
that $\delta _{\zeta}$ is the analogue, in the non-Archimedean case, of
the Haar measure of $\SS$ of total volume 1, because
$\theta_{\Sigma}(\rho_{\Sigma}(p))=\zeta\ast p$. 

\begin{prop} \label{prop:60}
  The averaging process in Definition \ref{def:82} is
  multiplicative with respect to products of metrized line bundles,
  is continuous with respect to uniform convergence of metrics and leaves
  invariant toric metrics.
\end{prop}

\begin{proof}
  This follow easily from the definition of
  $\|\cdot\|_{\SS}$.  
\end{proof}

To the metrized line bundle $\ov L $ and the section $s$ we associate
the function $g_{\ov L,s}\colon X_{0}^{\an}\to \R$ given by $g_{\ov
  L,s}(p)=-\log\|s(p)\|$. In the Archimedean case, the
function $g_{\ov L,s}$ is $1/2$ times the usual Green function associated to
the metrized line bundle $\ov L$ and the section $s$. 
\nomenclature[agLs]{$g_{\ov L,s}$}{Green function on $X_{0}^{\an}$}%
The metric $\|\cdot\| $ is
toric if and only if the function $g_{\ov L,s}$ is
$\SS $-invariant. In this case we can form the commutative
diagram
\begin{equation}\label{eq:55}
  \xymatrix{ X_{0}^{\an}
    \ar[r]^{g_{\ov L,s }}\ar[d]_{\val} & \R\\
    N_{\R}\ar@{-->}[ru]
  }
\end{equation}
The dashed arrow exists as a continuous function because $\rho  _{0
}$, hence $\val$, is a proper surjective map and, by
$\SS $-invariance, $g_{\ov 
L,s} $
is constant along the 
fibres. This justifies the following definition.

\begin{defn} \label{def:68} Let $L$ be a toric line bundle on
  $X_{\Sigma}$ and $s$ a
  toric section of~$L$. Let~$\|\cdot\|$ be a  metric on $L^{\an}$ and set
  $\ov L=(L,\|\cdot\|)$.  
We define the function $\psiabs_{\ov L,s}
  \colon N_{\R}\to \R$ given, for $u\in N_{\R}$, by
\begin{equation} \label{eq:40}
  \psiabs_{\ov L,s} (u)=\log \|s(p)\|_{\SS}
\end{equation}
for any $p\in X_{0}^{\an}$ with $\val(p)=u$. 
When the line bundle and the section are clear
from the context, we will alternatively denote this function as
$\psiabs_{\|\cdot \|}$. \nomenclature[g23]{$\psiabs_{\ov L,s}$,
  $\psiabs_{\Vert\cdot\Vert}$}{function on $N_{\R}$ associated to a
  metrized toric line bundle and section}%
\end{defn}

The facts that $\|\cdot\|_{\SS}$ is $\SS $-invariant and that
$s$ is a nowhere vanishing regular section on $X_{0}^{\an}$
imply that (\ref{eq:40}) gives a well--defined continuous function on
$N_{\R} $.
In the case when $\|\cdot\|$ is toric, we have that
\begin{equation}
  \label{eq:98}
  \psiabs_{\ov L,s} (u)=\log \|s(p)\|
\end{equation}
for $u\in N_{\R}$ and any $p\in X_{0}^{\an}$ with $\val(p)=u$. 

We will also use the following variant of the function $\psiabs_{\ov
  L,s}$. It will be most useful when treating  metrics
induced by integral models.

\begin{defn}\label{def:88} Let notations be as in Definition
  \ref{def:68} and suppose that absolute value of $K$ is either 
Archimedean or associated to a discrete valuation. 
We define the function $\phiK_{\ov L,s}
  \colon N_{\R}\to \R$ given, for $u\in N_{\R}$, by
\begin{equation*}
  \phiK_{\ov L,s} (u)=\frac{\log \|s(p)\|_{\SS}}{\lambda_{K}}
\end{equation*}
for any $p\in X_{0}^{\an}$ with $\val_{K}(p)=u$. 
When the line bundle and the section are clear
from the context, we will alternatively denote this function as
$\phiK_{\|\cdot \|}$.  \nomenclature[g2123]{$\phiK_{\ov L,s}$,
  $\phiK_{\Vert\cdot\Vert}$}{function on $N_{\R}$ associated to a
  metrized toric line bundle and section}%
\end{defn}


\begin{rem} \label{rem:35} The function $\phiK_{\ov L,s}$ agrees with
  the right multiplication $\psiabs_{\ov L,s}\lambda ^{-1}_{K}$, that
  is,
  \begin{equation} \label{eq:140}
    \phiK_{\ov L,s}(u)= \lambda ^{-1}_{K}{\psiabs_{\ov L,s}(\lambda _{K}u)}
  \end{equation}
  for all $u\in N_{\R}$. Hence, the functions $\phiK_{\|\cdot\|}$ and
  $\psiabs_{\|\cdot\|}$ carry the same information and it is easy to
  move from one to the other. The difference between both functions is
  similar to the difference between $\val_K$ and $\val$ discussed in
  Remark~\ref{rem:34}.
\end{rem}

We study the effect of taking a field extension.  Let ${K'}/K$ be a
finite extension of complete valued fields.  We denote by $e_{{K'}/K}$ the
{ramification degree} of $K'$ over $K$. \index{ramification degree of
  a finite field extension}%

\begin{prop}\label{prop:115}
  Let notations be as in Definition \ref{def:68} and consider a finite
  extension of complete valued fields $K'/K$. Let $\ov L{'}$ and $s'$
  be the metrized toric line bundle and toric section on
  $X_{\Sigma}\times\Spec(K')$ obtained after base change to $K'$. Then
  \begin{displaymath}
    \psiabs_{\ov L',s'}=\psiabs_{\ov L,s},\quad        \phiK_{\ov L',s'}=\phiK_{\ov L,s}e_{K'/K}, 
  \end{displaymath}
  that is,
  \begin{math}
    \phiK_{\ov L',s'}(u)=e_{K'/K}\phiK_{\ov L,s}\big(\frac{u}{e_{K'/K}}\big)
  \end{math}
  for all $u\in N_{\R}$.
\end{prop}

\begin{proof}
  The first statement follows from the definition of $\psiabs_{\ov
    L,s}$ and propositions \ref{prop:113} and \ref{prop:114}. The
  second statement follows from the first one, equation (\ref{eq:140})
  and the fact that  $\lambda _{K}=e_{{K'}/K}\lambda _{{K'}}$.
\end{proof}

\begin{exmpl} \label{exm:37} With the notation in examples \ref{exm:1}
  and \ref{exm:9}, consider the standard simplex 
\index{standard simplex} $\Delta ^{n}$ with fan $\Sigma= \Sigma _{\Delta ^{n}}$
  and support function $\Psi=\Psi _{\Delta ^{n}}$. To these data
  correspond the toric variety $X_{\Sigma}=\P^{n}$, the toric line
  bundle $L_{\Psi}=\mathcal{O}(1)$, and the toric section $s_{\Psi}$
  whose associated Weil divisor is the hyperplane at infinity~$H_{0}$.
  \begin{enumerate}
  \item \label{item:107} The canonical metrics
\index{canonical metric!of $\cO(1)^{\an}$}%
$\|\cdot\|_{\can}$ in
    examples~\ref{exm:2} and \ref{exm:6} are toric and both satisfy
    $\psiabs_{\|\cdot\|_{\can}}=\Psi$.
  \item \label{item:108} The Fubini-Study metric
\index{metric!Fubini-Study}%
    $\|\cdot\|_{\FS}$ in
    Example \ref{exm:3} is also toric and satisfies
    $\psiabs_{\|\cdot\|_{\FS}}=f_{\FS}$, where $f_{\FS}$ is the function in Example
    \ref{exm:5}.
  \end{enumerate}
\end{exmpl}

In general, the space of toric metrics on the line bundle $L$ can be put into a
one-to-one correspondence with a certain class of continuous functions
on $N_{\R}$.

\begin{prop}\label{prop:14} Let $\Sigma $ be a complete fan in
  $N_{\R}$ and $\Psi $ a virtual support function
  on $\Sigma $. Let $X_{\Sigma}$ and $L=\mathcal{O}(D_{\Psi })$ be the
  associated proper toric variety over $K$ and toric line bundle.  
  \begin{enumerate}
  \item \label{item:124} Given a metric $\|\cdot\|$ on $L^{\an}$, the
    function $\psiabs_{\|\cdot\|}-\Psi$ extends to a continuous function on $N_{\Sigma }$.
  \item \label{item:125} The
  correspondence $\|\cdot\| \mapsto \psiabs _{\|\cdot\| }$ 
  is a bijection between 
  the set of toric metrics on $L^{\an}$ and the
  set of continuous 
  functions $\psiabs\colon N_{\R}\to \R$ with the property that $\psiabs -\Psi $ can be
  extended to a continuous function on $N_{\Sigma }$. 
  \end{enumerate}
\end{prop}
\begin{proof}
  We first prove \eqref{item:124}. Let
  $\{m_{\sigma }\}$ be a set of defining vectors of $\Psi $. For each
  cone $\sigma \in \Sigma $, the
  section $s_{\sigma}=\chi ^{m_{\sigma }}s$ is a  nowhere vanishing regular
  section on $X_{\sigma }^{\an}$. By \eqref{eq:42},
  for $p\in X^{\an}_{0}$, 
  \begin{displaymath}
  \psiabs _{\|\cdot\| }(\val(p))-\langle m_{\sigma }, \val(p)\rangle
=\log\|s(p)\|_{\SS}+\log|\chi ^{m_{\sigma }}(p)|
=\log\|s_{\sigma}(p)\|_{\SS}.
  \end{displaymath}
  Since $\|s_{\sigma}\|_{\SS}$ is a nowhere vanishing regular function
  on $X^{\an}_{\sigma}$, the function
  $\log\|s_{\sigma}\|_{\SS}$ is a continuous function on 
  $X_{\sigma }^{\an}$ that is $\SS $-invariant. So it defines a
  continuous function on 
  $X_{\sigma}(\R_{\ge 0})$. As a consequence, $\psiabs _{\|\cdot\| }-m_{\sigma }$ extends
  to a continuous function on $N_{\sigma }$. 

  Now, if we see that $\Psi -m_{\sigma }$ extends also to a continuous
  function on $N_{\sigma }$ we will be able to extend $\psiabs
  _{\|\cdot\| }-\Psi $ to a continuous function on $N_{\sigma }$ for
  every $\sigma \in \Sigma $ and therefore to $N_{\Sigma }$.

  Let $\tau $ be a face of $\sigma $ and 
  let $u\in N(\tau )_{\R}$. Let $W(\tau ,U,p)$ be a neighbourhood of
  $u$ as in 
  \eqref{eq:20}. By taking $U$ small enough and $p\in\tau$ far away from the origin, we can
  assume that $ W(\tau ,U,p)\cap N_\R $ is contained in the set of cones
  that have $\tau $ as a face. Since $\Psi $ and $m_{\sigma }$ agree
  when restricted to $\sigma$ (hence when restricted to $\tau$) it
  follows that, if $w+t\in W(\tau ,U,p)\cap N_\R $ with $w\in U$ and $t\in
  p+\tau$, then $(\Psi -m_{\sigma })(w+t)$ only depends on $w$ and not
  on $t$. Hence it can be extended to a continuous function on the
  whole $W(\tau ,U,p)$. By moving $\tau $, $u$, $U$ and $p$ we see
  that it can be 
  extended to a continuous function on $N_{\sigma }$, which completes
  the proof of the first statement. 

  For the second statement, let now $\psiabs $ be a function on
  $N_{\R}$ such that $\psiabs -\Psi $ extends to a continuous function on
  $N_{\Sigma }$. We define a toric metric $\|\cdot\| $ on
  the restriction $L^{\an}|_{X_{0}^{\an}}$ by the formula
  \begin{equation} \label{eq:101}
    \|s(p)\|=\exp(\psiabs (\val(p))).
  \end{equation}
  Then, by the argument before, $\psiabs -m_{\sigma }$ extends to a
  continuous function on $N_{\sigma }$, which proves that $\|\cdot\|$
  extends to a metric over $X_{\sigma }^{\an}$. Varying $\sigma \in
  \Sigma $ we obtain that $\|\cdot\|$ extends to a metric over
  $X_{\Sigma }^{\an}$. We can verify that this assignment is
  the inverse of the correspondence $\|\cdot\|\mapsto\psiabs_{\|\cdot\|}$,
  when the latter  is restricted to the space of toric metrics on
  $L^{\an}$. 
\end{proof}

\begin{rem}\label{rem:36} Assume that $K$ is non-Archimedean and with
  discrete valuation. Let $\psiabs\colon N_{\R}\to \R$ be a function
  and consider the right multiplication $\phiK =\psiabs \lambda
  ^{-1}_{K}$. Since $\Psi $ is conic, $\Psi \lambda ^{-1}_{K}=\Psi
  $. Therefore $\psiabs-\Psi$ extends to a continuous function on
  $N_{\Sigma }$ if and only if $\phiK-\Psi$ does. Thus the statement
  of Proposition \ref{prop:14} remains true if we replace the function
  $\psiabs$ by the function $\phiK$.
\end{rem}

\begin{notn} \label{def:84}
For a function $\psiabs\colon N_{\R}\to \R$ with the property that $\psiabs
-\Psi $ can be extended to a continuous function on $N_{\Sigma }$, we
denote by $\|\cdot\|_{\psiabs } $
\nomenclature[svert531]{$\Vert\cdot\Vert_{\psiabs}$}{toric metric from a
  function}%
the metric given by the correspondence in Proposition
\ref{prop:14}\eqref{item:125}. It is the metric defined in
(\ref{eq:101}) above.
\end{notn}

\begin{cor} \label{cor:8} For any metric
  $\|\cdot\| $ on $L^{\an}$, the function $|\psiabs_{\|\cdot\| } -\Psi |$ is bounded.
\end{cor}
\begin{proof}
  Since we are assuming that $\Sigma $ is complete, the space $N_{\Sigma
  }\simeq X_{\Sigma }(\R_{\ge 0})$ is compact. Thus the corollary
  follows from Proposition \ref{prop:14}\eqref{item:124}.
\end{proof}

\begin{prop} \label{prop:24} The correspondence $(\ov L,s)\mapsto \psiabs
  _{\ov L,s}$ satisfies the following properties.
  \begin{enumerate}
  \item \label{item:72} Let $\ov L_{i}=(L_{i},\|\cdot\|_{i})$,
    $i=1,2$, be toric line bundles equipped with a metric and $s_{i}$
    a toric section of $L_{i}$. Then
\begin{displaymath}
  \psiabs_{\ov L_{1}\otimes \ov L_{2}, s_{1}\otimes s_{2}}= 
  \psiabs_{\ov L_{1}, s_{1}} +\psiabs_{\ov L_{2}, s_{2}} .
\end{displaymath}
 
\item \label{item:79} Let $\ov L=(L,\|\cdot\|)$ be a toric line bundle
  equipped with a metric and $s$ a toric section of $L$.  Then
\begin{displaymath}
    \psiabs_{\ov L^{\otimes -1}, s^{\otimes-1}} =   -\psiabs_{\ov L, s}. 
\end{displaymath}
\item \label{item:126} Let $(L,s)$ be a toric line bundle and section, 
  and $(\|\cdot\|_{l})_{l\ge1}$ a sequence of metrics on $L$
  converging to a metric $\|\cdot\|$ with respect to the distance in
  (\ref{eq:84}). Then $\psiabs_{\|\cdot\|_{l}}$
  converges uniformly to $\psiabs_{\|\cdot\|}$.
\end{enumerate}
\end{prop}

\begin{proof}
This follows easily from the definitions.
\end{proof}

 A consequence of Proposition \ref{prop:14}\eqref{item:125} is that every toric line
 bundle has a distinguished toric metric. 

\begin{prop-def} \label{def:57}
  Let $\Sigma $ be a  complete fan, $X_{\Sigma }$ the corresponding
  toric variety, and $L$ a toric line bundle on
  $X_{\Sigma }$. Let $s$
  be a toric section of $L$ and $\Psi $ the virtual 
  support function on $\Sigma $ associated to $(L,s)$ by theorems
  \ref{thm:27} and \ref{thm:3}. The metric on $L^{\an}$
  associated to the 
  function $\Psi $ by Proposition \ref{prop:14}\eqref{item:125} only depends on the
  structure of toric line bundle of $L$. This toric metric is called the
  \emph{canonical metric} of $L^{\an}$
  \index{canonical metric!of a toric line bundle}%
\index{metric!canonical|see{canonical metric}}%
and is denoted $\|\cdot\|_{\can}$.
  We write $\ov L^{\can}=(L,\|\cdot \|_{\can})$.
\nomenclature[svert23]{$\Vert\cdot\Vert_{\can}$}{canonical metric of a toric line bundle}%
\nomenclature[aLcan]{$\ov L^{\can}$}{toric line bundle with its canonical metric}%
\end{prop-def}
\begin{proof}
  Let $s'$ be another toric section of $L$. Then there is an element
  $m\in M$ such that $s'=\chi^{m}s$. The corresponding virtual
  support function is $\Psi'=\Psi -m $. Denote by $ \|\cdot\|$ and
  $\|\cdot\|'$ the metrics associated to $s,\Psi $ and to $s',\Psi '$
  respectively. Then
  \begin{displaymath}
    \|s(p)\|'=\|\chi^{-m}s'(p)\|'=
    \e^{(m+\Psi ')(\val(p))}=
    \e^{\Psi (\val(p))}=\|s(p)\|.
  \end{displaymath}
  Thus both metrics agree.
\end{proof}

The canonical metrics $\|\cdot\|_{\can}$ in examples~\ref{exm:2} and
\ref{exm:6} are particular cases of the canonical metric of
Proposition-Definition \ref{def:57}.

\begin{prop}\label{prop:90} The canonical metric is compatible with the
  tensor product of line bundles.
  \begin{enumerate}
  \item \label{item:35}
Let $L_{i}$, $i=1,2$, be toric line bundles on $X$. Then $\ov {L_{1}\otimes
  L_{2}}^{\can}=\ov {L_{1}}^{\can}\otimes \ov{L_{2}}^{\can}$. 
\item \label{item:36} 
Let $L$  be a toric  line bundle on $X$. 
Then $\ov{L^{\otimes-1}}^{\can}= (\ov{L}^{\can})^{\otimes-1}$. 
\end{enumerate}
\end{prop}

\begin{proof}
This follows easily from the definitions.
\end{proof}

Next we describe the behaviour of the correspondence of Proposition
\ref{prop:14}\eqref{item:125} with respect to equivariant morphisms. We start
with the case of orbits. Let $\Sigma $ be a complete fan in $N$ and
$\Psi $ a virtual support function on $\Sigma $. Let
$L$ and $s$ be the associated toric line bundle and toric section, and
$\{m_{\sigma 
}\}_{\sigma \in \Sigma }$ a set of defining vectors of $\Psi $. Let
$\sigma \in \Sigma $ and  $V(\sigma )$  the corresponding closed
subvariety. As in Proposition \ref{prop:81}, the restriction of $L$ to
$V(\sigma )$ is a toric line bundle. Since $V(\sigma )$ and $\div(s)$ may not 
intersect properly we can not restrict $s$ directly to $V(\sigma
)$. By contrast, $D_{\Psi -m_{\sigma }}=\div (\chi ^{m_{\sigma }}s)$
intersects properly $V(\sigma )$ and we can restrict the section $\chi
^{m_{\sigma }}s$ to $V(\sigma )$ to obtain a toric section of
$\mathcal{O}(D_{(\Psi-m_{\sigma })(\sigma )})\simeq L\mid _{V(\sigma
  )}$. Denote 
$\iota\colon V(\sigma )\to X_{\Sigma }$ the closed
immersion.
For short, 
we write $s_{\sigma}=\chi
^{m_{\sigma }}s$. Then $\iota ^{\ast}s_{\sigma}$ is a nowhere vanishing
section on $O(\sigma )$. Recall that $V(\sigma )$ has a structure of
toric variety given by the fan $\Sigma (\sigma )$ on $N(\sigma )$
(Proposition~\ref{prop:73}). The principal open subset of $V(\sigma )$
is the orbit $O(\sigma )$.
 
Let $\|\cdot\|$ be a metric on $L^{\an}$ and write
$\ov L=(L,\|\cdot\|)$.  By the proof of Proposition~\ref{prop:14}, the
function $\psiabs _{\ov L,s}-m_{\sigma }=\psiabs_{\ov L,s_{\sigma}}$ can be extended
to a continuous function on $N_{\sigma }$ that we denote $\ov {\psiabs
}_{\ov L,s_{\sigma}}$.

\begin{prop}\label{prop:57}
  The function $\psiabs _{\iota ^{\ast}\ov L,\iota
    ^{\ast}s_{\sigma}} \colon N(\sigma )_{\R}\to
  \R$ agrees with the restriction of $\ov{\psiabs} _{\ov
    L,s_{\sigma}}$ to the subset  $N(\sigma )_{\R}$ of $N_{\sigma }$.
\end{prop}
\begin{proof}
  The section $s$ is  nowhere vanishing on $X_{\Sigma ,\sigma }$. Therefore, the function $g_{\ov L,
    s_{\sigma}}\colon X^{\an}_{0}\to \R$ of diagram
  \eqref{eq:55} can be extended to a continuous function on $X^{\an}_{\Sigma
    ,\sigma }$ that we also denote $g_{\ov L,
    s_{\sigma}}$.  By the definition of the inverse image of a metric, there is a commutative diagram
  \begin{displaymath}
  \xymatrix{ O(\sigma )^{\an}\ar [r]^{\iota }
    \ar[dr] _{g_{\iota ^{\ast}\ov L, \iota
    ^{\ast}s_{\sigma}}}
 & X^{\an}_{\Sigma ,\sigma }\ar[d]^{g_{\ov L, s_{\sigma}}}\\
    & \R.
  }
  \end{displaymath}
  We next prove the result in the Archimedean case. Let $\T_{\sigma }$
  be the torus corresponding to the quotient lattice $N(\sigma )$, and
  $\SS_{\sigma }$ the compact subtorus of $\T^{\an}_{\sigma }$. Denote
  by $\pi _{\sigma }\colon \SS\to \SS_{\sigma }$. Let $\mu
  _{\Haar,\sigma}$ be that Haar measure of $\SS_{\sigma }$ of total
  measure $1$. Then $\mu
  _{\Haar,\sigma}=(\pi _{\sigma })_{\ast}\mu _{\Haar}$.
  The inclusion
  $\iota$ satisfies that, for $t\in \SS$ and $p\in O(\sigma )^{\an}$
  then $\iota(\pi _{\sigma }(t)\cdot p)=t\cdot \iota(p)$. Thus
  \begin{multline*}
    \log \|\iota^{\ast}s_{\sigma }(p)\|_{\SS_{\sigma }}=
    -\int _{\SS_{\sigma }}g_{\iota ^{\ast}\ov L, \iota
    ^{\ast}s_{\sigma}}(t\cdot p)\dd
    \mu _{\Haar,\sigma}(t)\\=
    -\int _{\SS}g_{\ov L, s_{\sigma}}(t\cdot \iota(p))\dd
    \mu _{\Haar}(t)=\log\|s_{\sigma }(\iota (p))\|_{\SS},
  \end{multline*}
  which implies the result.

We next prove the statement in the  non-Archimedean case. By
propositions \ref{prop:55} and 
\ref{prop:107},
\begin{multline*}
  \log \|\iota^{\ast}s_{\sigma }(p)\|_{\SS_{\sigma }}=
  -g_{\iota ^{\ast}\ov L, \iota
    ^{\ast}s_{\sigma}}(\theta _{\Sigma (\sigma )}(\rho _{\Sigma
    (\sigma )}(p)))\\
  =-g_{\ov L, s_{\sigma}}(\theta _{\Sigma }(\rho _{\Sigma }(\iota (p))))
  =\log \|s_{\sigma }(\iota (p))\|_{\SS},
\end{multline*}
which implies the result.
\end{proof}

\begin{cor} \label{cor:19}
  Let $\ov L$ be a toric line bundle on $X_{\Sigma }$ equipped
  with the canonical metric,  $\sigma\in \Sigma $ and $\iota
  \colon V(\sigma )\to X_{\Sigma }$ the closed
  immersion. Then the 
  restriction $\iota^{\ast}\ov L$ is a toric line bundle
  equipped with the canonical metric.
\end{cor}
\begin{proof}
  Choose a toric section $s$ of $L$ whose divisor meets $V(\sigma )$
  properly. Let $\Psi $ be the corresponding 
  virtual support function. The condition of proper intersection is
  equivalent to $\Psi |_{\sigma }=0$. Then $\Psi $ extends to a
  continuous function 
  $\ov \Psi $ on $N_{\sigma }$ and the restriction of $\ov \Psi $ to
  $N(\sigma )$ is equal to $\Psi (\sigma )$ (Proposition
  \ref{prop:72}). Hence the result follows 
  from Proposition~\ref{prop:57}.   
\end{proof}

We next study the case of an equivariant morphism whose image intersects
the principal open subset.  Let $N_{i}$, $\Sigma _{i}$, $i=1,2$, $H$, $p$
and $A$ be as in Proposition \ref{prop:56}. Let $\Psi _{2}$ be a
virtual support function on $\Sigma _{2}$ and let $\Psi _{1}=\Psi
_{2}\circ H$. This is a virtual support function on $\Sigma _{1}$. Let
$(L_{i},s_{i})$ be the corresponding toric line bundles and
sections. By Proposition \ref{prop:70} and Theorem \ref{thm:27}, there
is an isomorphism $\varphi_{p,H}^{\ast}L_{2}\simeq L_{1} $ that sends
$\varphi_{p,H}^{\ast} s_{2}$ to $s_{1}$. We use this isomorphism to
identify them.  The following result
follows from Proposition \ref{prop:56} and is left to the reader.

\begin{prop}\label{prop:58} Let $\|\cdot\|$ be a toric metric on
$L_{2}^{\an}$ and write $\ov L_{2}=(L_{2},\|\cdot\|)$, $\ov
L_{1}=(L_{1},\varphi_{p,H}^{\ast}\|\cdot\|)$.
   The equality
   \begin{math}
     \psiabs _{\ov L_{1},s_{1}}=\psiabs _{\ov L_{2},s_{2}}\circ A
   \end{math}
   holds.
\end{prop}

The canonical metric is stable by
inverse image under toric morphisms. The following result follows
easily from the 
definitions. 

\begin{cor}\label{cor:20}
  Assume furthermore that $p=x_{0}$ and so the equivariant morphism
  $\varphi_{p,H}=\varphi_{H}\colon  X_{\Sigma _{1}}\to X_{\Sigma
    _{2}}$ is a toric morphism. If $\ov L$ is a 
  toric line bundle on $X_{\Sigma _{2}}$ equipped with the canonical
  metric, then $\varphi_{H}^{\ast}\ov L$ is a toric line bundle equipped
  with the canonical metric. 
\end{cor}

The inverse image of the canonical metric by an equivariant map does
not need to be the canonical metric. In fact, the analogue of 
Example \ref{exm:12} in terms of metrics shows that many
different metrics can be obtained as the inverse image of the
canonical metric on the projective space.

\begin{exmpl} \label{exm:35}
Let $\Sigma $ be a complete fan in $N_{\R}$ and  $X_{\Sigma }$ the
corresponding toric variety. Recall the description of the
projective space $\P^{r}$ as a toric variety
\index{projective space} given in Example
\ref{exm:10}.  
Let $H\colon N\to
\Z^{r}$ be a linear map such that, for each $\sigma \in \Sigma $ there
exist $\tau \in \Sigma _{\Delta ^{r}}$ with $H(\sigma )\subset \tau
$. Let $p\in \P^{r}_{0}(K)$. Then we have an equivariant morphism
$\varphi_{p,H}\colon X_{\Sigma }\to \P^{r}$. Consider the support
function $\Psi _{\Delta ^{r}}$ on $\Sigma _{\Delta ^{r}}$. Then
$L_{\Psi _{\Delta ^{r}}}=\mathcal{O}_{\P^{r}}(1)$. Write 
$L=\varphi^{\ast}_{p,H}L_{\Psi _{\Delta ^{r}}}$,
$s=\varphi^{\ast}_{p,H}s_{\Psi _{\Delta ^{r}}}$ and $\Psi
=H^{\ast}\Psi _{\Delta ^{r}}$. Thus $(L,s)=(L_{\Psi },s_{\Psi })$. 

Set $A=H+\val(p)$ for the affine map.
Let $\|\cdot\|$ be
the metric on $L^{\an}$ induced by the canonical metric of
$\mathcal{O}(D_{\Psi _{\Delta ^{r}}})^{\an}$ and let $\psiabs $ be the function
associated to it by Definition~\ref{def:68}. By Proposition \ref{prop:58},
$\psiabs =A^{\ast}\Psi
_{\Delta ^{r}}$. This is a piecewise affine concave function on $N_{\R}$
with  $\rec(\psiabs )=\Psi $ that can be made explicit as follows.

Let $\{e_{1},\dots,e_{r}\}$ be the standard basis of $\Z^{r}$ and 
$\{e_{1}^{\vee},\dots,e_{r}^{\vee}\}$  the dual basis. Write
$m_{i}=e_{i}^{\vee}\circ H\in M$ and
$l_{i}=e_{i}^{\vee}(\val(p))\in \R$. 
Then
\begin{displaymath}
  \Psi =\min\{0,m_{1},\dots,m_{r}\},\quad
  \psiabs =\min\{0,m_{1}+l_{1},\dots,m_{r}+l_{r}\}.
\end{displaymath}

\end{exmpl}

We want to characterize all the functions that can be obtained with a
slight generalization of the previous construction. In this case we
will use the function $\phiK $ instead of $\psiabs$.

\begin{prop}\label{prop:44} Assume that the absolute value of $K$ is
  either Archimedean or associated to a discrete
  valuation. 
 Let $\Sigma $ be a complete fan in $N$ and $\Psi $ a support
 function on~$\Sigma $. Write $L=L_{\Psi }$ and $s=s_{\Psi }$. Let
 $\phiK \colon N_{\R}\to \R$ a
 piecewise affine concave function with $\rec(\phiK )=\Psi $, that has 
 an $H$-representation
 \begin{displaymath}
   \phiK =\min_{i=0,\dots,r}\{m_{i}+l_{i}\},
 \end{displaymath}
 with $m_{i}\in
 M_{\Q}$ and $l_{i}\in \R$ in the Archimedean case and $l_{i}\in \Q$
 in the non-Archimedean case. Then there is an
 equivariant morphism $\varphi\colon X_{\Sigma }\to \P^{r}$, an
 integer $e>0$ and an isomorphism $L^{\otimes e}\simeq \varphi^{\ast} \mathcal{O}(1)$
 such that the metric $\|\cdot\|$ induced on $L^{\an}$ by the canonical metric of
 $\mathcal{O}(1)^{\an}$ satisfies $\phiK_{\|\cdot\|} =\phiK$.
\end{prop}
\begin{proof}
  First observe that the condition $l_{i}\in \R$ in the Archimedean
  case and $l_{i}\in \Q$ in the non-Archimedean case is equivalent to
  the condition $l_{i}\in \Q\,\val_{K}(K^{\times})$. Let $e>0$ be an
  integer such that $em_{i}\in M$ and $el_{i}\in \val_{K}(K^{\times})$
  for $i=0,\dots,r$.

  Consider the linear map $H\colon N_{\R}\to \R^{r}$ given by
  \begin{math}
    H(u)=(em_{i}(u)-em_{0}(u))_{i=1,\dots,r}
  \end{math}
  and the affine map $A=H+ \boldsymbol{l}$ with
  $\boldsymbol{l}=(el_{i}-el_{0})_{i=1,\dots,r}$. By Lemma
  \ref{lemm:17},
  \begin{displaymath}
    e\phiK=A^{\ast}\Psi _{\Delta ^{r}}+em_{0}+el _{0}.
  \end{displaymath}
  For each $\sigma\in \Sigma$, we claim that there exists
  $\sigma_{i_{0}}\in \Sigma_{\Delta^{r}}$ such that $H(\sigma)\subset
  \sigma_{i_{0}}$. Indeed, $\Psi(u)=\min_{i}\{m_{i}(u)\}$. Since
  $\Psi$ is a support function on $\Sigma$, for each $\sigma\in
  \Sigma$, there exists an $i_{0}$ such that $\Psi(u)= m_{i_{0}}(u)$
  for all $u\in \sigma$.  Writing $e_{0}^{\vee}=0$, this condition
  implies
  \begin{displaymath}
    \min_{0\le i\le r}\{e_{i}^{\vee}(H(u))\}= e_{i_{0}}^{\vee}(H(u))
    \quad \text{for all } u\in \sigma.
  \end{displaymath}
  Hence, $H(\sigma) \subset \sigma_{i_{0}}$, where $\sigma_{i_{0}}\in
  \Sigma_{\Delta^{r}}$ is the cone $\{v| \min_{0\le i\le
    r}\{e_{i}^{\vee}(v)\}= e_{i_{0}}^{\vee}(v)\}$ and the claim is proved.
  
  Therefore, we can apply Theorem \ref{thm:25} and given a point $p\in
  \P^{r}(K)$ such that $\val_K(p)= 
  \boldsymbol{l}$, there is an equivariant map $\varphi_{p,H}\colon
  X_{\Sigma}\to \P^{r}$.  By Example \ref{exm:26}, there is an
  isomorphism $ L^{\otimes e} \simeq \varphi_{p,H}^{*}\mathcal{O}(1)$
  and $a\in K^{\times}$ with $\val_K(a)=l_{0}$ such that
  $(a^{-1}\chi ^{-m_{0}}s)^{\otimes e}$ corresponds to
  $\varphi_{p,H}^{*}( s_{\Psi _{\Delta^{r}}})$.

Let $\ov L$ be the line bundle $L$ equipped with the metric induced by
the above isomorphism and the canonical metric of
$\mathcal{O}(1)^{\an}$. Then 
\begin{displaymath}
  \phiK_{\ov L,s}= \phiK_{\ov L,a^{-1}\chi ^{-m_{0}}s}+
  m_{0}+l_{0}
=\frac{1}{e}  A^{\ast}\Psi _{\Delta ^{r}}+m_{0}+l _{0}=\phiK,
\end{displaymath}
as stated.
\end{proof}

\begin{cor} \label{cor:22}
Let $\phiK$ be as in Proposition \ref{prop:44} and $\psiabs=\phiK
\lambda _{K}$. Then
the metric $\|\cdot\|_{\psiabs }$ is semipositive{}.
\end{cor}

\begin{proof}
  This follows readily from the previous result together with Example
  \ref{exm:6} in the Archimedean case and Example \ref{exm:2} in the
  non-Archimedean case and the fact that the inverse image of a
  semipositive{} metric is also semipositive{}. 
\end{proof}
The conclusion of Proposition \ref{prop:58} is not true for non toric
metrics, because the averaging process in Definition \ref{def:82} does
not commute with inverse images by equivariant
morphisms. Nevertheless, with the notations before Proposition
\ref{prop:58}, we can compute $\psiabs _{\ov L_{2},s_{2}}\circ A$ as
an average over all equivariant morphisms associated with the affine
map $A$.  In the non-Archimedean case, this averaging process will be
described by a limit process on algebraic points of
$X^{\an}_{\Sigma_{2}}$. Recall that the algebraic points of a
Berkovich space are dense. Since Berkovich
spaces are not necessary metrizable, in principle, one should
approximate an arbitrary point by a \emph{net} of algebraic
points. Nevertheless, thanks to 
\cite[Th\'eor\`eme 5.3]{Poineau:eba}, Berkovich spaces are of type
Fr\'echet-Uryshon, which implies that every point can be approximated by
a \emph{sequence} of algebraic points.

The next result will be
needed in the proof of Proposition~\ref{prop:61}. 

\begin{prop}\label{prop:95} With the notations  previous to Proposition
  \ref{prop:58}, let $\|\cdot\|$ be a metric on $L_{2}^{\an}$.
  \begin{enumerate}
  \item \label{item:128} Assume that $K$ is Archimedean.  Let $p\in
    X_{\Sigma_2,0}(K)$ and put $u_{0}=\val(p)$.  Then, for $u\in
    N_{1,\R}$,
  \begin{displaymath}
    \psiabs _{\|\cdot\|}(u_{0}+H(u))=\int_{\SS_{2}}\psiabs
    _{\varphi_{t\cdot p,H}^{\ast}\|\cdot\|}(u)\dd\mu_{\Haar_{2}}(t),
  \end{displaymath}
where $\SS_{2}$ is the compact subtorus of the torus associated
to the lattice $N_{2}$, and  $\mu _{\Haar,2}$ is the Haar measure of $\SS_{2}$ of
total volume 1.

\item \label{item:127} Assume that $K$ is non-Archimedean. Let 
  $u_0\in \lambda_{K}N_{2,\Q}$ and $(q_{i})_{i\in \N}$ be a sequence of
  points $q_{i}\in \val^{-1}(u_{0})\cap X_{\Sigma_{2},\alg}^{\an}$ with
  $\lim_{i\to \infty} q_{i}=\theta_{\Sigma_{2}} \circ \ee
  (u_{0})$. For each $i\in \N$, let $K'_{i}$ be a finite extension of
  $K$ and $\widetilde q_{i}\in X_{\Sigma _{2},0}(K'_{i})$ a point over
  $q_{i}$. We denote by $\|\cdot\|_{K'_{i}}$ the metric induced on the
  line bundle $L_{2,K'_{i}}$ by base change. Then, for $u\in N_{1,\R}$, 
  \begin{displaymath}
    \psiabs _{\|\cdot\|}(u_{0}+H(u))= \lim_{i\to \infty}\psiabs
    _{\varphi_{\wt q_{i},H}^{\ast}\|\cdot\|_{K'_{i}}}(u). 
  \end{displaymath}
  \end{enumerate}
\end{prop}
\begin{proof}
  We first prove \eqref{item:128}. Let $q\in X_{\Sigma _{1}}^{\an}$
  with $\val(q)=u$. By definition,
  \begin{displaymath}
    \int_{\SS_{2}}\psiabs
    _{\varphi_{t_{2}\cdot
        p,H}^{\ast}\|\cdot\|}(u)\dd\mu_{\Haar_{2}}(t_{2})=
    -\int_{\SS_{2}}\int_{\SS_{1}}g
    _{\varphi_{t_{2}\cdot p,H}^{\ast}\|\cdot\|}(t_{1}\cdot
    q)\dd\mu_{\Haar_{1}}(t_{1})\dd\mu_{\Haar_{2}}(t_{2}). 
  \end{displaymath}
  where $\SS_{i}$ is the compact subtorus of the torus associated to
  the lattice $N_{i}$, and $\mu _{\Haar,i}$ is the Haar measure of
  $\SS_{i}$ of total volume 1, $i=1,2$.  Let $\varrho _{H}\colon
  \T_{1}\to \T_{2}$ be the morphism of tori induced by the linear map
  $H$. Now we compute
  \begin{displaymath}
    g
    _{\varphi_{t_{2}\cdot p,H}^{\ast}\|\cdot\|}(t_{1}\cdot
    q)=g_{\|\cdot\|}(\varphi_{t_{2}\cdot p,H}(t_{1}\cdot q))
    =g_{\|\cdot\|}(t_{2}\cdot \varrho _{H}(t_{1})\cdot\varphi_{p,H}(q)).
  \end{displaymath}

  Consider the morphism of compact tori $\varrho \colon \SS_{2}\times
  \SS_{1}\to \SS_{2}$ given by $\varrho (t_{2},t_{1})=t_{2}\cdot \varrho
  _{H}(t_{1})$. The measure $\varrho _{\ast}(\mu _{\Haar,1}\times \mu
  _{\Haar,2})$ is an invariant measure on $\SS_{2}$ of total volume
  $1$. Thus agrees with $\mu _{\Haar,2}$. Hence
  \begin{multline*}
    \int_{\SS_{2}}\int_{\SS_{1}}
    g_{\|\cdot\|}(t_{2}\cdot \varrho _{H}(t_{1})\cdot\varphi_{p,H}(q))
    \dd\mu_{\Haar_{1}}(t_{1})\dd\mu_{\Haar_{2}}(t_{2})\\=
    \int_{\SS_{2}}
    g_{\|\cdot\|}(t_{2}\cdot\varphi_{p,H}(q))
    \dd\mu_{\Haar_{2}}(t_{2}).
  \end{multline*}
  Since $\val(\varphi_{p,H}(q))=u_{0}+H(u)$, we obtain,
  \begin{displaymath}
    -\int_{\SS_{2}}
    g_{\|\cdot\|}(t_{2}\cdot\varphi_{p,H}(q))
    \dd\mu_{\Haar_{2}}(t_{2})=\psiabs _{\|\cdot\|}(u_{0}+H(u)),
  \end{displaymath}
  proving the result.

  Next we prove \eqref{item:127}. In view of Proposition
  \ref{prop:115}, it is enough to treat the case when $K$ is
  algebraically closed and hence $K'_{i}=K$ and $\wt q_{i}=q_{i}$. 
  Then, by definitions \ref{def:68} and \ref{def:82},
  \begin{displaymath}
    \lim_{i\to\infty}\psiabs
    _{\varphi_{q_{i},H}^{\ast}\|\cdot\|}(u)=
    -\lim_{i\to\infty}g
    _{\varphi_{q_{i},H}^{\ast}\|\cdot\|}(\theta _{\Sigma _{1}}(\ee(u)))=
    -\lim_{i\to\infty}g
    _{\|\cdot\|}(\varphi_{q_{i},H}(\theta _{\Sigma _{1}} (\ee(u)))).
  \end{displaymath}
  Identifying $X_{\Sigma _{2},0}$ with $\T_{2}$ and denoting by
  $x_{0}$ the distinguished point of $X_{\Sigma _{2},0}(K)$, we
  obtain, using \ref{prop:109}\eqref{item:129}
  \begin{displaymath}
    \varphi_{q_{i},H}(\theta _{\Sigma _{1}} (\ee(u)))=
    q_{i}\cdot\varphi_{x_0,H}(\theta _{\Sigma _{1}} (\ee(u)))=
    q_{i}\ast\varphi_{x_0,H}(\theta _{\Sigma _{1}} (\ee(u))).
  \end{displaymath}
  By Corollary \ref{cor:33}, the point $\varphi_{x_0,H}(\theta
  _{\Sigma _{1}} (\ee(u)))$ is peaked. Thus, by Proposition~\ref{prop:109}\eqref{item:132},
  \begin{displaymath}
    \lim_{i\to\infty} g
    _{\|\cdot\|}(q_{i}\ast\varphi_{x_0,H}(\theta _{\Sigma _{1}}
    (\ee(u))))=g
    _{\|\cdot\|}(\theta_{\Sigma_{2}} (\ee
  (u_{0})))\ast \varphi_{x_0,H}(\theta _{\Sigma _{1}}
    (\ee(u))).
  \end{displaymath}
 By propositions \ref{prop:96} and \ref{prop:56}
 \begin{multline*}
   \theta_{\Sigma_{2}} (\ee
  (u_{0}))\ast \varphi_{x_0,H}(\theta _{\Sigma _{1}}
    (\ee(u)))=\theta_{\Sigma_{2}} (\ee
  (u_{0})\cdot \rho _{\Sigma _{2}}( \varphi_{x_0,H}(\theta _{\Sigma _{1}}
    (\ee(u)))))\\=\theta_{\Sigma_{2}} (\ee
  (u_{0})\cdot  \varphi_{x_0,H}(\ee(u)))= \theta_{\Sigma_{2}} (\ee
  (u_{0}+H(u))).
 \end{multline*}
 Therefore
 \begin{displaymath}
   \lim_{i\to\infty}\psiabs
    _{\varphi_{q_{i},H}^{\ast}\|\cdot\|}(u)=-g
    _{\|\cdot\|}(\theta_{\Sigma_{2}} (\ee
  (u_{0}+H(u))))=\psiabs_{\|\cdot\|}(u_{0}+H(u)),
 \end{displaymath}
proving the result.
\end{proof}

\section{Smooth metrics and their associated measures}   
\label{sec:pos-smooth-metr}

We now discuss the relationship between semipositivity of smooth
metrics and concavity of the associated function in the Archimedean
case. Moreover we will determine the associated measure. 

We keep the notation at the beginning of the chapter but we restrict
to the case when $K$ is either $\R$ or $\C$ and $\Sigma$ is a complete
fan.  Let $\Psi $ be a virtual support function on $\Sigma $, with $L$
and $s$ the corresponding toric line bundle and section. Let $L^{\an}$ be
the analytic line bundle on $X_{\Sigma }^{\an}$ associated to $L$.

\begin{prop}\label{prop:15}  Let $\|\cdot\| $ be a
  smooth toric metric 
  on $L^{\an}$. Then $\|\cdot\| $ is
  semipositive if and only if the function $\psiabs =\psiabs _{\|\cdot\| }$
  is concave.
\index{smooth metric!semipositive}%
\end{prop}
\begin{proof}
  Since the condition of being semipositive is closed, it is enough to
  check it in the open set $X_{0}^{\an}$. We choose an integral basis
  of $M=N^{\vee}$. This determines isomorphisms
  \begin{displaymath}
    X_{0}^{\an}\simeq (\C^{\times})^{n},\quad X_{0}(\R_{\ge
      0})\simeq(\R_{>0})^{n},\quad
    N_{\C}\simeq \C^{n},\quad N_{\R}\simeq \R^{n}.
  \end{displaymath}
  Let $z_{1},\dots ,z_{n}$ be the coordinates of $X_{0}^{\an}$
  and $u_{1},\dots,u_{n}$ the coordinates of $N_{\R}$ determined by
  these isomorphisms. With these coordinates, the map 
  $$\val \colon X_{0}^{\an}\to N_{\R}$$ 
  is given by
  \begin{displaymath}
    \val(z_{1},\dots,z_{n})= \frac{-1}{2}
    (\log(z_{1}\bar z_{1}),\dots,\log(z_{n}\bar z_{n})).
  \end{displaymath}
  As usual, we set $\ov L=(L,\|\cdot\| )$ and $g=g_{\ov L,s}=-\log
  \|s\|$. Then the integral valued first Chern class is given by
  \begin{equation}
    \label{eq:22}
    \frac{1}{2\pi i}\chern_{1}(\overline L)=\frac{-1}{\pi i}\partial\bar \partial g
    =\frac{i}{\pi}\sum_{k,l}\frac{\partial^{2}g}{\partial
      z_{k} \partial \bar z_{l}}\dd z_{k}\land \dd \bar z_{l}.
  \end{equation}
  The standard orientation of the unit disk $\D \subset \C$ is
  given by $\dd x \land \dd y=(i/2)\dd z\land \dd \bar z$. Hence,  
  the metric of $\ov L$ is semipositive if and only if the matrix
  $G=(\frac{\partial^{2}g}{\partial 
      z_{k}\partial \bar z_{l}})_{k,l}$ is semi-positive
    definite. Since
    \begin{equation}\label{eq:24}
      \frac{\partial^{2}g}{\partial
      z_{k}\partial \bar z_{l}}=
    \frac{-1}{4 z_{k}\bar z_{l}}\frac{\partial^{2}\psiabs }{\partial
      u_{k}\partial \bar u_{l}},
    \end{equation}
    if we write $\Hess (\psiabs )=(\frac{\partial^{2}\psiabs }{\partial 
      u_{k}\partial \bar u_{l}})_{k,l}$ and $Z=\diag
    ((2z_{1})^{-1},\dots,(2z_{n})^{-1})$, 
    then $G=-\bar Z^{t}\Hess (\psiabs ) Z$. Therefore $G$ is
    semi-positive definite 
    if and only if $\Hess(\psiabs )$ is semi-negative definite, hence, if
    and only if 
    $\psiabs $ is concave.
\end{proof}

\begin{prop} \label{prop:30} 
Let $\|\cdot\|$ be a smooth metric on $L^{\an}$. 
Then $\|\cdot\|_{\SS}$ is also smooth. Moreover, if
$\|\cdot\|$ is semipositive, then $\|\cdot\|_{\SSinv}$ is semipositive
too.
\end{prop}

\begin{proof} The first statement follows from the definition of
  $\|\cdot\|_{\SS}$ and the preservation of smoothness under
  integration of $\log\|s_{\sigma}\|$ along the compact subsets $
  \SS \cdot p$ for $p\in X^{\an}_{\sigma}$, $\sigma\in \Sigma$.

For the second statement, we have
  \begin{displaymath}
\chern_{1}(L,\|\cdot\|_{\SSinv})=\int_{\SS }t^{\ast}
\chern_{1}(L,\|\cdot\|)\dd\mu_{\Haar}(t)
  \end{displaymath} 
where $t^{*}$ denotes the inverse image under the multiplication map
$t\colon X_{\Sigma}^{\an}\to  X_{\Sigma}^{\an}$. 
Therefore, if $(L,\|\cdot\|)$ is semipositive,
then $(L,\|\cdot\|_{\SSinv})$ is semipositive too.
\end{proof}

As a direct consequence of propositions \ref{prop:30} and
\ref{prop:15}, if the line bundle $L^{\an}$ admits a semipositive smooth metric,
then its virtual support function $\Psi $ is concave. By Proposition
\ref{prop:99}\eqref{item:56}, this latter condition is equivalent to
the fact that $L$ is generated by global sections.

For a semipositive smooth toric metric $\|\cdot \|$ on $L^{\an}$, we can
characterize the associated measure on $X^{\an}$ in terms of the
Monge-Amp\`ere measure of the
concave function~$\psiabs_{\|\cdot\|}$.

\begin{defn} \label{def:83} Let $\psiabs \colon N_{\R}\to \R$ be a
  concave function and $\mathcal{M}_{M}(\psiabs )$ the Monge-Amp\`ere
  measure associated to $\psiabs$ and the lattice $M$.
  \index{Monge-Amp\`ere measure}%
  \nomenclature[aMzA4]{$\ov{\mathcal{M}}_{M}(\psiabs )$}{Monge-Amp\`ere
    measure on $N_{\Sigma}$}%
  We denote by $\ov{\mathcal{M}}_{M}(\psiabs )$ the measure on $N_{\Sigma
  }$ given by
  \begin{equation*} 
    \ov{\mathcal{M}}_{M}(\psiabs )(E)= \mathcal{M}_{M}(\psiabs )(E\cap N_{\R}) 
  \end{equation*}
  for any Borel subset $E$ of $N_{\Sigma }$. We will use the same
  notation for the mixed Monge-Amp\`ere measure.
\end{defn}

By its very definition, the measure $\ov{\mathcal{M}}_{M}(\psiabs )$ is
bounded with 
total mass $$\ov{\mathcal{M}}_{M}(\psiabs )(N_{\Sigma
})=\Vol_{M}(\Delta _{\Psi })$$ and the set $N_{\Sigma }\setminus
N_{\R}$ has measure zero.

\begin{thm}\label{thm:18}
  Let $\|\cdot\| $ be a semipositive
  smooth toric metric 
  on $L^{\an}$.
  Let $\chern_{1}(\ov L)^{\wedge n}\land \delta 
  _{X_{\Sigma }}$ be the measure defined by $\ov L$. Then,
  \begin{equation}\label{eq:56}
    \val_{\ast} (\chern_{1}(\ov L)^{\wedge n}\land \delta
    _{X_{\Sigma }})= 
    n!\,\ov{\mathcal{M}}_{M}(\psiabs ),
  \end{equation}  
  where $\val $ is the map in the diagram \eqref{eq:89}. In
  addition, this measure is uniquely 
  characterized by the equation 
  \eqref{eq:56} and the property of being $\SS $-invariant.
\end{thm}
\begin{proof} 
  Since the measure $\chern_{1}(\ov L)^{\wedge n}\land \delta 
  _{X_{\Sigma }}$ is given by a smooth volume form and $X^{\an}_{\Sigma
  }\setminus X^{\an}_{0}$ is a set of Lebesgue measure 
  zero, the measure $\chern_{1}(\ov L)^{\wedge n}\land \delta 
  _{X_{\Sigma }}$ is determined by its restriction to the dense open subset 
  $X_{0}^{\an}$. Thus, to prove \eqref{eq:56} it is enough to
  show that
  \begin{equation}\label{eq:90}
    \val_{\ast} (\chern_{1}(\ov L)^{\wedge n}\land \delta
    _{X_{\Sigma }}|_{X^{\an}_{0}})= 
    n!\mathcal{M}_{M}(\psiabs ).
  \end{equation}  

  We use the coordinate system of the proof of Proposition
  \ref{prop:15}.
  We denote by $\wt {\ee}\colon  N_{\C}\to X_{0}(\C)$ the map induced
  by the 
  morphism $\C\to \C^{\times}$ given by $z\mapsto \exp(-z)$.
  We write $u_{k}+i v_{k}$ for the complex coordinates of
  $N_{\C}$. Then
  \begin{equation}
    \label{eq:83}
    \wt \ee^{\ast}\left(\frac{\dd z_{k}\land \dd \bar z _{k}}
    {z_{k}\bar z_{k}}\right)= (-2i)\dd
  u_{k}\land \dd v_{k}.
  \end{equation}
  Using now the equations \eqref{eq:22}, \eqref{eq:24} and \eqref{eq:83},
  we obtain that
  \begin{align*}
    \frac{1}{(2\pi i)^{n}}\wt {\ee} ^{\ast} \chern_{1}(\ov L)^{\wedge n}&=\wt
    {\ee} ^{\ast}\left(
    \frac{i^{n}}{\pi^{n}} n!\det (G) \dd z_{1}\land \dd\bar z_{1}\land \dots\land 
    \dd z_{n}\land \dd \bar z_{n}\right)\\
    &=\frac{(-1)^{n}}{(2 \pi)^{n}} n!\det (\Hess(\psiabs  ))
    \dd u_{1}\land \dd v_{1}\land \dots \land \dd u_{n}\land \dd v_{n}.
  \end{align*}
  Since the map $\val$ is the composition of $\wt{\ee}^{-1}$ with the
  projection $N_{\C}\to N_{\R}$,
  integrating with respect
  to the variables $v_{1},\dots ,v_{n}$ in the domain $[0,2\pi]^{n}$,
  taking into account the natural orientation of $\C^{n}$ and the
  orientation of $N_{\R}$ given by the coordinate system,  and the
  fact that  
  the normalization factor $1/(2\pi i)^{n}$ is implicit in
  the current $\delta _{X_{\Sigma }}$, we obtain
  \begin{displaymath}
    \val_{\ast} (\chern_{1}(\ov L)^{\wedge n}\land \delta
    _{X_{\Sigma }}|_{X_{0}^{\an}} )=
    (-1)^{n}n!\det (\Hess (\psiabs )) \dd u_{1}\land\dots\land \dd u_{n}.
  \end{displaymath}
  Thus the equation \eqref{eq:90} follows from Proposition \ref{prop:5}.
  Finally, the last statement follows from the fact
  that, in a compact Abelian group there is a unique Haar measure
  with fixed total volume. 
\end{proof}

We end this section by making explicit the compatibility of the previous
constructions with the conjugation in the case of toric
varieties over $\R$.
Let $\Sigma $ be a fan in $N_{\R}$, and $X_{\Sigma ,\R}$
and $X_{\Sigma ,\C}$ the corresponding toric varieties over $\R$ and
$\C$.  Recall that the
underlying complex analytic
spaces of $X^{\an}_{\Sigma ,\C}$ and $X^{\an}_{\Sigma ,\R}$ agree
(see Remark \ref{rem:17}) and are denoted by $X^{\an}_{\Sigma
}$. 

\begin{prop}\label{prop:40} Let $\Sigma $ be a complete fan and
   $\varsigma \colon X^{\an}_{\Sigma }\to X^{\an}_{\Sigma }$
  the anti-linear involution of  Remark \ref{rem:17}.
  \begin{enumerate}
  \item \label{item:28} There are commutative diagrams
    \begin{displaymath}
      \xymatrix{
        X_{\Sigma }^{\an}\ar[r]^{\varsigma } \ar[rd]_{\rho _{\Sigma}}
        &X^{\an}_{\Sigma}\ar[d]^{\rho _{\Sigma}}\\ 
        & X_{\Sigma }(\R_{\ge 0}),}\qquad
      \xymatrix{
        X_{\Sigma }^{\an}\ar[r]^{\varsigma } \ar[rd]_{\val}
        &X^{\an}_{\Sigma}\ar[d]^{\val}\\ 
        & N_{\Sigma }.}
    \end{displaymath}
  \item \label{item:55} Let $L_{\R}$ be a line bundle on $X_{\Sigma ,\R}$ and $L_{\C}$
the line bundle over $X_{\Sigma ,\C}$ obtained by base change. The
assignment that to each metric $\|\cdot\|_{\R}$ on 
    $L_{\R}$ associates the metric $\|\cdot\|_{\C}$ on $L_{\C}$ given
    by forgetting the anti-linear involution, induces a bijection between the set of toric
    metrics on $L_{\R}$ and the set of toric metrics on
    $L_{\C}$. Moreover $\psiabs _{\|\cdot\|_{\C}}=\psiabs
    _{\|\cdot\|_{\R}}$. 
  \end{enumerate}
\end{prop}
\begin{proof} We first prove \eqref{item:28}. The first commutativity
  follows from the invariance of the absolute value under complex
  conjugation and the second follows from the first and the commutativity
  of diagram \eqref{eq:89}.

  To prove \eqref{item:55} we have to show that, if $\|\cdot\|$ is a
  toric metric on $L_{\C}$, then it is compatible with complex
  conjugation. That is, if $s$ a toric section of $L_{\C}$ defined over
  $\R$ and $p\in X^{\an}_{0}$, then $\|s(p)\|=\|s(\varsigma (p))\|$. Since
  the fibres of $\rho _{\Sigma }$ are orbits under $\SS $, by
  \eqref{item:28}, there is an element $t\in \SS $ such that
  $\varsigma (p)=t(p)$. Since the metric is toric
  \begin{displaymath}
    \|s(\varsigma (p))\|=\|s(t(p))\|=\|s(p)\|.
  \end{displaymath}
  The last assertion is clear because the definitions of $\psiabs
  _{\|\cdot\|_{\C}}$ and $\psiabs
    _{\|\cdot\|_{\R}}$ agree. 
\end{proof}

\section{Algebraic metrics from toric models}
\label{sec:toric-algebr-metr}

Next we study some properties of algebraic metrics with particular
emphasis on the ones that arise from toric models.  We keep the
notation at the beginning of the chapter and we assume that $K$ is a
complete field with respect to an absolute value associated to a
nontrivial discrete valuation. 
We also assume that the fan $\Sigma$ is
complete.

Since we will discuss the relationship between metrics and algebraic
models it is preferable to work with the functions $\phiK $ of
Definition \ref{def:88} instead of the functions of Definition
\ref{def:68}.

We begin by studying the relationship between the maps $\val_K$ and $\red$.

\begin{lem}\label{lemm:3}
  Let $\Pi $ be a complete SCR polyhedral complex of $N_{\R}$ such that
  $\rec (\Pi )=\Sigma $. Let $\cX:=\cX_{\Pi }$ be the model of
  $X_{\Sigma }$ determined by $\Pi $. Let $\Lambda \in \Pi $ and
  $p\in X_{0}^{\an}$. Then
  $\red(p)\in \cX_{\Lambda }$ if and only if $\val_{K}(p)\in \Lambda $.
\end{lem}
\begin{proof} By the definition of the semigroup $\wt M_{\Lambda }$, 
  the condition $\val_K(p)\in \Lambda $ holds if and only if $\langle
  m,\val_K(p)\rangle +l \ge 0$ for all $(m,l)\in \wt M_{\Lambda
  }$. This is equivalent to $\log |\chi ^{-m}(p)|+\log |\varpi
  |^{-l}\ge 0$ for all   $(m,l)\in \wt M_{\Lambda
  }$. In turn, this is equivalent to $|\chi ^{m}(p)\varpi^{l}|\le 1$,
  for all 
  $(m,l)\in \wt M_{\Lambda }$. Hence, $\val_K(p)\in \Lambda $ if and
  only if $|a(p)|\le 1$ for all $a\in  K^{\circ}[\cX_{\Lambda }]$,
  which is exactly the condition  $\red(p)\in
  \cX_{\Lambda }$ (see \eqref{eq:43}). 
\end{proof}

\begin{cor}\label{cor:1}
  With the same hypothesis as in Lemma \ref{lemm:3}, $\red(p)\in O(\Lambda
  )$ if and only if $\val_K(p)\in \ri(\Lambda )$.
\end{cor}
\begin{proof}
  This follows from Lemma \ref{lemm:3} and the fact that the special
  fibre is
  \begin{displaymath}
    \cX_{\Lambda ,o}=\coprod_{\Lambda '
      \text{ face of }\Lambda }O(\Lambda '),
  \end{displaymath}
  and $\ri(\Lambda )=\Lambda \setminus \bigcup_{\Lambda ' \text{
      proper face of }\Lambda }\Lambda '$.
\end{proof}

Let $\Psi $ be a virtual support function on $\Sigma $ and $(L,s)$ the
corresponding toric line bundle and section.  We denote by $L^{\an}$
the analytic line bundle on $X_{\Sigma }^{\an}$ associated to $L$.
Let $\Pi $ be a complete SCR polyhedral complex in $N_{\R}$ such that
$\rec(\Pi )=\Sigma $ and $\phiK$ a rational piecewise affine function
on $\Pi $ with $\rec(\phiK )=\Psi $. Let $e>0$ be an integer such that
$e\phiK $ is an H-lattice function. By Theorem \ref{thm:11b}, the pair
$(\Pi ,e\phiK )$ determines a toric model $(\cX_{\Pi },\cL_{e\phiK },e)$
of $(X_{\Sigma },L)$. We will write $\cL=\cL_{e\phiK }$ for short. Definition
\ref{def:7} gives us an algebraic metric $\|\cdot\|_{\cL}$ on
$L^{\an}$.  The following proposition closes the circle.

\begin{prop} \label{prop:59} The algebraic metric $\|\cdot\|_{\cL}$ is toric and
  the equality $\phiK
  _{\|\cdot\|_{\cL}}=\phiK $ holds. The function $\phiK -\Psi $ extends to a
  continuous function on $N_{\Sigma }$ and the metric $\|\cdot\|_{\phiK\lambda _{K}
  }$ associated to $\phiK\lambda _{K} $ (Notation~\ref{def:84}) agrees with
  $\|\cdot\|_{\cL}$. 
\end{prop}
\begin{proof} The tensor product
  $s^{\otimes e}$ defines a rational section of $\cL$.
  Let $\Lambda \in \Pi $ and choose $m_{\Lambda }\in 
  M$, $l_{\Lambda }\in \Z$ such that $e\phiK |_{\Lambda }=m_{\Lambda
  }+l_{\Lambda }|_{\Lambda }$. 
  Let $u\in \Lambda $ and $p\in X^{\an}_{0 }$ with
  $u=\val_K(p)$. 
  Then $\red(p)\in \cX_{\Lambda 
  }$. But in $\cX_{\Lambda }$ the section $\chi ^{m_{\Lambda }}\varpi
  ^{l_{\Lambda }}s^{\otimes e}$ is regular and non-vanishing. Therefore, by
  Definition \ref{def:7},
  \begin{displaymath}
    \|\chi ^{m_{\Lambda }}(p)\varpi
    ^{l_{\Lambda }}s^{\otimes e}(p)\|_{\cL}=1.    
  \end{displaymath}
  Thus
  \begin{displaymath}
  \frac{1}{\lambda _{K}}\log\|s(p)\|_{\cL}
  =\frac{1}{e\lambda_{K}}\log|\chi ^{-m_{\Lambda }}(p)\varpi
    ^{-l_{\Lambda }}|
    =\frac{1}{e}(\langle m_{\Lambda },u\rangle +l_{\Lambda })
    =\phiK(u), 
  \end{displaymath}
  which shows that the metric is toric. Moreover,
  \begin{displaymath}
    \phiK _{\|\cdot\|_{\cL} }(u)=  \frac{1}{\lambda _{K}}\log\|s(p)\|_{\cL}=\phiK(u),
  \end{displaymath}
  and therefore, $\phiK $ agrees with the function associated to the metric
  $\|\cdot\|_{\cL}$. By Proposition \ref{prop:14}\eqref{item:124}, and
  Remark \ref{rem:36}, $\phiK -\Psi $ extends to a
  continuous function on $N_{\Sigma }$ and the metric $\|\cdot\|_{\phiK\lambda _{K}
  }$ agrees with
  $\|\cdot\|_{\cL}$. 
\end{proof}

\begin{exmpl} \label{exm:28} In the non-Archimedean case, the canonical
  metric of Proposi\-tion-Definition 
  \ref{def:57} is the toric algebraic metric induced by the canonical
  model of  Definition \ref{def:55}.
\index{canonical metric!of a toric line bundle}
\index{canonical model!of a $\T$-Cartier divisor}%
\end{exmpl}

Proposition \ref{prop:59} imposes a necessary condition for a rational
piecewise affine function to determine a model of
$(X_{\Sigma },L_{\Psi })$. 

\begin{cor} \label{cor:11} Let $\Psi$ be a virtual support function on
  $\Sigma $ and $\phiK $  a rational piecewise affine function on
  $N_{\R}$, with $\rec(\phiK )=\Psi $, such that there exists a
  complete SCR polyhedral complex $\Pi $ with $\rec(\Pi )=\Sigma $ and
  $\phiK $ piecewise affine on $\Pi $. Then $\phiK -\Psi $ can be
  extended to a continuous function on $N_{\Sigma }$.
\end{cor}
\begin{proof}
  If there exists such a SCR polyhedral complex $\Pi $, then $\Pi $
  and $\phiK $ determine a model of $\mathcal{O}(D_{\Psi })$ and hence
  an algebraic metric $\|\cdot\|$ arising from a toric model. By
  Proposition \ref{prop:59}, $\phiK =\phiK _{\|\cdot\|}$ and, by
  Proposition~\ref{prop:14}\eqref{item:124}, the function $\phiK
  _{\|\cdot\|}-\Psi $ extends to a continuous function on $N_{\Sigma
  }$.
\end{proof}

\begin{exmpl} \label{exm:27}
  Let $N=\Z^{2}$ and consider the fan $\Sigma $ generated by the vectors
  $e_{0}=(-1,-1)$, $e_{1}=(1,0)$ and $e_{2}=(0,1)$. Then $X_{\Sigma
  }=\P^{2}$. The
  virtual support function $\Psi =0$ corresponds to the trivial line
  bundle $\mathcal{O}_{\P^{2}}$.
  Consider the function
  \begin{displaymath}
    \phiK (x,y)=
    \begin{cases}
      0&\text{ if }x\le 0,\\
      x&\text{ if }0\le x\le 1,\\
      1&\text{ if }1\le x.\\
    \end{cases}
  \end{displaymath}
  Then $\rec(\phiK )=\Psi $, but $\phiK $ does not extend to a
  continuous function on $N_{\Sigma }$ and therefore it does not
 determine a model of
  $(X_{\Sigma },\mathcal{O}_{\P^{2}})$. By contrast, let $\Sigma '$ be the
  fan obtained subdividing $\Sigma $ by adding the edge corresponding
  to $e'=(0,-1)$. Then $\pi\colon X_{\Sigma '}\to X_{\Sigma}$ is isomorphic to a blow-up of
  $\P^{2}$ at one point.  The function $\phiK $ extends to a continuous
  function on $N_{\Sigma '}$ and it corresponds to a toric model of
  $(X_{\Sigma '},\pi^{*}\mathcal{O}_{\P^{2}})$.
\end{exmpl}

\begin{ques} 
Is the condition in Corollary \ref{cor:11} also
  sufficient? In other words,
  let $N$, $\Sigma $ and $\Psi $ be as before and $\phiK $  a
  rational piecewise affine function on $N_{\R}$ such that $\phiK
  -\Psi $ can be extended to a continuous function on $N_{\Sigma
  }$. Does it 
  exist a complete SCR polyhedral complex $\Pi $ such that $\rec(\Pi
  )=\Sigma $ and $\phiK $ is piecewise affine on~$\Pi $? 
\end{ques}

\begin{rem}\label{rem:25}
  By the proof of Theorem \ref{thm:12} and Corollary \ref{cor:11},
  when $\phiK $ is a piecewise affine concave function, the conditions
  \begin{enumerate}
  \item $|\phiK -\Psi |$ is bounded;
  \item $\phiK -\Psi $ can be extended to a continuous function on $N_{\Sigma }$;
  \item there exist  a complete SCR polyhedral complex $\Pi $ with $\rec(\Pi
  )=\Sigma $ and $\phiK $ piecewise affine on $\Pi $;
  \end{enumerate}
  are equivalent. In particular, the answer to the above question is
  positive when $\phiK $ is concave.
\end{rem}

\begin{cor} \label{cor:12} Let $\Sigma $ be a complete fan and
 $\Psi $ a support function on $\Sigma $. Let~$\phiK $ be a
  rational piecewise 
affine concave function on $N_{\R}$ with $\rec(\phiK )=\Psi $. Then
  the metric $\|\cdot\|_{\phiK \lambda _{K}} $ is algebraic and has a semipositive
  toric model.
\end{cor}
\begin{proof} 
  By Theorem \ref{thm:12}, the concave function $\phiK $ determines an equivalence class of
  semipositive toric models of $(X_{\Sigma },L)$. Any toric model in
  this class defines an algebraic metric on $L^{\an}$. By
  Proposition~\ref{prop:21}, this metric only depends on $\phiK $ and
  we denote it by $\|\cdot\|$.  Proposition \ref{prop:59} implies that
  $\|\cdot\|_{\phiK\lambda _{K}}=\|\cdot\|$, hence this metric is given by a semipositive toric
  model. 
\end{proof}

We have seen that rational piecewise affine functions give rise to
toric algebraic metrics. We now study the converse. In fact this
converse is more general, in the sense that any algebraic metric
determines a rational piecewise affine function. 

\begin{thm}\label{thm:15} Let $\Sigma $ be a complete fan,
  $\Psi $ a virtual support function on $\Sigma $ and
  $(L,s)$ the corresponding toric line bundle and section.   
  Let $\|\cdot \|$ be an algebraic metric on
  $L^{\an}$. 
  \begin{enumerate}
  \item \label{item:114} The function $\phiK _{\|\cdot\|}$ is rational
    piecewise affine. 
  \item  \label{item:115} If  $\|\cdot\|$ is toric
    and $\phiK _{\|\cdot\|}$ is concave, then this metric has
    a semipositive toric model.
  \end{enumerate}
\end{thm}
Before proving the theorem, we introduce a variant of the function
$\phiK $ for rational functions.  Let $g$ be a rational function on
$X_{\Sigma }$. Then we consider the function $\phiK _{g}\colon
N_{\R}\to \R$ defined, for $u\in N_{\R}$, as \nomenclature[g2114]{$\phiK _{g}$}{function on $N_{\R}$ associated to a rational function}%
\begin{displaymath}
\phiK _{g}(u)=\frac{\log |g \circ \theta
_{0}\circ \ee_{{K}}(u)|}{\lambda_{K} }
\end{displaymath}
where $\theta_{0}$ is defined in Proposition-Definition \ref{prop:54}
and $\ee_{K}$ in~(\ref{eq:95}).

\begin{lem}\label{lemm:6}
  Let $g$ be a rational function on $X_{\Sigma }$. Then the function
  $\phiK _{g}$ is an H-lattice function (Definition
  \ref{def:31}). In particular, it is piecewise affine.
\end{lem}
\begin{proof}
  The function $g$ can be written as 
  $g=\frac {\sum _{m\in M}\alpha _{m}\chi^{m}}
  {\sum _{m\in M}\beta_{m}\chi^{m}}$. Then
  \begin{align*} 
    \phiK _{g}(u)&=\frac{1}{\lambda_{K} }\log|g\circ \theta _{0}\circ \ee_{{K}
    }(u)|\\
    & = \frac{1}{\lambda_{K} }\log\Bigl|\sum _{m\in M}\alpha
    _{m}\chi^{m}(\theta _{0}\circ \ee_{K} (u))\Bigr|-\frac{1}{\lambda_{K}
    }\log\Bigl|\sum _{m\in M}\beta _{m}\chi^{m}(\theta _{0}\circ \ee_{K}(u))\Bigr|\\
    &= \max _{m\in M}\Bigl(\frac{\log |\alpha _{m}|}{\lambda_{K} }-\left <
      m,u\right >\Bigr)- \max _{m\in M}\Bigl(\frac{\log |\beta
      _{m}|}{\lambda_{K} }-\left < m,u\right >\Bigr)\\
    & = \max _{m\in M}(-\val_K (\alpha _{m})-\left < m,u\right >)- \max
    _{m\in M}(-\val_K (\beta
    _{m})-\left < m,u\right >)\\
    & = \min _{m\in M}(\left < m,u\right >+\val_K (\beta
    _{m}))- \min
    _{m\in M}(\left < m,u\right >+\val_K (\alpha 
    _{m})).
  \end{align*}
Thus, it is the difference of two H-lattice concave functions.
\end{proof}

\begin{proof}[Proof of Theorem \ref{thm:15}] 
  Since the metric is algebraic, there exist a proper $K^{\circ}$-
  scheme $\mathcal{X}$ and a line bundle $\mathcal{L}$ on
  $\mathcal{X}$ such that the base change of $(\mathcal{X},
  \mathcal{L})$  to $K$ is isomorphic to $(X_{\Sigma
  },L^{\otimes e})$. Let $\{\cU_{i},s_{i}\}$ be a
  trivialization of $ \mathcal{L}$. Let $C_{i}=\red^{-1}(\cU_{i}\cap
  \mathcal{X}_{o})$. The subsets $C_{i}$ form a finite closed cover of
  $X_{\Sigma }^{\an}$. On $\cU_{i}$ we can write $s_{\Psi }^{\otimes
    e}=g_{i}s_{i}$ for a certain rational function $g_{i}$. Therefore,
  on $C_{i}$, we have $\log \|s_{\Psi }(p)\|=\frac{\log
    |g_{i}(p)|}{e}$. By Lemma \ref{lemm:6}, it follows that there is a
  finite closed cover of $N_{\R}$ and the restriction of $\phiK
  _{\|\cdot\|}$ to each of these closed subsets is rational piecewise
  affine. Therefore $\phiK
  _{\|\cdot\|}$ is rational piecewise affine. This proves
  \eqref{item:114}.
 
  We now prove the statement \eqref{item:115}. By \eqref{item:114} and
  Proposition~\ref{prop:14}\eqref{item:124}, the concave function
  $\phiK_{\|\cdot\|}$ is rational piecewise affine with recession
  function equal to $\Psi$. Since $\|\cdot\|$ is toric,
  Proposition~\ref{prop:14}\eqref{item:125} implies that it agrees
  with the metric associated to $\phiK_{\|\cdot\|}\lambda _{K}$. The
  statement then follows from Corollary \ref{cor:12}.
\end{proof}




 We now study the effect of taking a field extension. 

\begin{prop}\label{prop:23}
Let $\Sigma $ be a complete fan in $N_{\R}$ and $\Pi $ a
complete SCR
polyhedral complex in $N_{\R}$ with $\Sigma =\rec(\Pi )$.
Let $\Pi' $ be the polyhedral complex in $N_{\R}$
  obtained from $\Pi $ by applying a homothety of ratio $e_{{K'}/K}
  $. Then
  \begin{displaymath}
    \cX_{\Pi',{{K'}}^{\circ}}=\Nor(\cX_{\Pi ,K^{\circ}}\times \Spec({K'}^{\circ})),
  \end{displaymath}
  where $\Nor$ denotes the normalization of a scheme.
\end{prop}
\begin{proof}

  The statement can be checked locally. Let
  $\Lambda $ be a polyhedron of $\Pi $. Let $\Lambda '= e_{{K'}/K}
  \Lambda $. Then it is clear that
  \begin{displaymath}
    K^{\circ}[\cX_{\Lambda }]\underset{K^{\circ}}{\otimes }
    {K'}^{\circ}\subset  {K'}^{\circ}[\cX_{\Lambda' }].
  \end{displaymath}
  Since the right-hand side ring is integrally closed, the integral
  closure of the left side ring is contained in the right side
  ring. Therefore we need to prove that ${K'}[\wt M_{\Lambda'
  }]$ is integral over the left side ring. Let $(a,l)\in \wt
  M_{\Lambda '}$. Thus $(e_{{K'}/K}a,l)\in \wt M_{\Lambda }$. Then the
  monomial $\chi ^{a}\varpi '{}^{l}\in 
  {K'}^{\circ}[\cX_{\Lambda' }]$ satisfies
  \begin{displaymath}
    (\chi ^{a}\varpi '{}^{l})^{e_{{K'}/K}}=
    (\chi ^{e_{{K'}/K}a}\varpi ^{l})\in K^{\circ}[\cX_{\Lambda
    }]\underset{K^{\circ}}{\otimes } 
    {K'}^{\circ}.  
  \end{displaymath}
  Hence $\chi ^{a}\varpi '{}^{l}$ is integral over $K^{\circ}[\cX_{\Lambda
    }]\underset{K^{\circ}}{\otimes } 
    {K'}^{\circ}.$ Since these monomials generate
    ${K'}^{\circ}[\cX_{\Lambda' }]$, we obtain 
  the result.
\end{proof}

\section{The one-dimensional case}
\label{sec:one-dimensional-case}

We now study in detail the non-Archimedean one-dimensional
case. Besides being a concrete example of the relationship between
functions, models, algebraic metrics, and measures, it is also a
crucial step in the proof that a toric metric is semipositive if and
only if the corresponding function is concave. Of this equivalence, up
to now we have proved only one implication and the reverse implication
will be proved in the next section.

The only one-dimensional proper toric variety over a field is the
projective line.  Algebraic metrics over this toric variety come from
integral models of line bundles. We will describe these in detail. The first
part of this section dissects models of the projective line itself,
while the second part turns to models of line bundles and metrics.  A
good reference for curves over local rings or more generally over
Dedekind domains is the book \cite{Liu:agac}, where the reader can
find most of the results that we need.

Let $K$ be a field which is complete with respect to an absolute value
associated to a nontrivial discrete valuation. We will use the
notation in \S\ref{sec:berkovich-spaces}. In particular, $K^{\circ}$
denotes the valuation ring of $K$.  

\begin{defn}\label{def:45}
Let $X$ be an integral
  projective curve over $K$. A \emph{semi-stable}
\index{model!semi-stable}%
model of $X$ is an integral
  projective scheme $\cX$ of finite type over $\Spec
  (K^{\circ})$ with an
  isomorphism $X\to \cX_{\eta}$, such that the special fibre
  $\cX_{o}$, after extension of scalars to the
  algebraic closure of the residue field,  
  is reduced and its singular points are ordinary double points. We
  will say that the model is regular if the scheme $\cX$ is regular.
\end{defn}

This definition is the specialization of
\cite[Definition~2.1]{Liu:StableReduction} to models of curves over a
DVR. Note that a semi-stable model as in Definition \ref{def:45} is a
semi-stable curve over $\Spec(K^{\circ})$ in the sense
of \cite[Definition~10.3.14]{Liu:agac} because, by \cite[Proposition
4.3.9]{Liu:agac} the hypothesis on $\cX$ imply that it is flat over $\Spec
  (K^{\circ})$.

To a semi-stable model whose special fibre has split double points, we
can associate the
dual graph of the special 
fibre. This graph contains one vertex for each irreducible component
of the special fibre and one edge for each double point, see 
\cite[Definition~10.3.17]{Liu:agac}.

We will use the following result to eventually reduce any integral model of
$\P^{1}_{K}$ to a simpler form.

\begin{prop} \label{prop:12}
Let $\cX$ be a
  projective model over $K^{\circ}$ of $\P^{1}_{K}$.
  Then there exists a
  finite extension $K'$ of $K$ with valuation ring ${K'}^{\circ}$,
  a regular semi-stable model $\cX'$ of $\P^{1}_{K'}$, and
  a proper morphism of models $\mathcal{X}'\to 
  \mathcal{X}\times \Spec({K'}^{\circ})$.  
\end{prop}
\begin{proof}
  The existence of the finite extension $K'$ and a semi-stable model
  dominating $\cX\times \Spec({K'}^{\circ})$ is guaranteed by
  \cite[Corollary~2.8]{Liu:StableReduction}. Then
  \cite[Proposition~2.7]{Liu:StableReduction} implies the existence of
  a regular semi-stable model dominating $\cX\times \Spec({K'}^{\circ})$.
\end{proof}

Consider the toric variety
$X_{\Sigma }\simeq \P^{1}$. We can choose an isomorphism
$N\simeq \Z$ and $N_{\R}\simeq \R$. Then $\Sigma
=\{\R_{-},\{0\},\R_{+}\}$. Let $0$ denote the invariant point of $\P^{1}_{K}$
corresponding to the cone $\R_{+}$ and $\infty$ the invariant point
corresponding to the cone $\R_{-}$.  

Let $\mathcal{X}$ be a regular semi-stable model of $\P^{1}_{K}$. We
will assume that all the components and double points of the special
fibre are defined over the residue field $k=K^{\circ}/K^{\circ \circ}$
and that each irreducible component contains a rational point. Since
the generic fibre $\mathcal{X}_{\eta}\simeq \P^{1}_{K}$ is connected
and of genus zero, by \cite[Lemma~10.3.18]{Liu:agac}, we deduce that
all components of the special fibre are rational curves and that its
dual graph is a tree. Let $D_{0}$ and $D_{\infty}$ denote the
horizontal divisors corresponding to the point~$0$ and $\infty$ of
$\P^{1}_{K}$. Then, there is a chain of rational curves that links the
divisor $D_{0}$ with $D_{\infty}$ and that is contained in the special
fibre. We will denote the irreducible components of the special fibre
that form this chain by $E_{0},\dots,E_{l} $, in such a way that the
component $E_{0}$ meets $D_{0}$, the component $E_{l}$ meets
$D_{\infty}$ and, for $0<i<l$, the component $E_{i}$ meets only
$E_{i-1}$ and $E_{i+1}$. The other components of $\mathcal{X}_{o}$
will be grouped in branches, each branch has its root in one of the
components $E_{i}$. We will denote by $F_{i,j}$, $j\in \Theta _{i}$
the components that belong to a branch with root in~$E_{i}$.
\nomenclature[aEF]{$E_{i},F_{i,j}$}{components of the special fibre of
  a semi-stable model of $\P^{1}$}%
\nomenclature[g08]{$\Theta _{i}$}{set of components of the special
  fibre of a semi-stable model of $\P^{1}$}%
We are not giving any particular order to the sets $\Theta _{i}$.

We denote by $E\cdot F$ the intersection product of two $1$-cycles of
$\cX$. 
\nomenclature[scdot]{$E\cdot F$}{intersection product of two
  $1$-cycles on a surface}%
Since the special fibre is reduced, we have
\begin{displaymath}
  \div(\varpi )=\sum_{i=0}^{l}\bigg(E_{i}+\sum_{j\in \Theta _{i}}F_{i,j}\bigg).
\end{displaymath}
Again by the assumption of semi-stability, the intersection product of
two different components of $\mathcal{X}_{o}$ is either $1$, if they
meet, or zero, if they do not meet. Since the intersection product of
$\div(\varpi )$ with any component of $\mathcal{X}_{o}$ is zero, we
deduce that, if $E$ is any component of
$\mathcal{X}_{o}$, the self-intersection product $E\cdot E$ is equal
to minus the number of components that meet $E$. In
particular, all components $F_{i,j}$ that are terminal, are
$(-1)$-curves. By Castelnuovo's criterion
\cite[Theorem~9.3.8]{Liu:agac}, we can successively blow-down 
all the components 
$F_{i,j}$ to obtain a new regular semi-stable model of $\P^{1}_{K}$ whose
special fibre consist of a chain of rational curves. For reasons that
will become apparent later, we denote this model as
$\cX_{\SSinv}$.  
\nomenclature[aX42]{$\cX_{\SSinv}$}{toric model associated to a
  semi-stable model}%

\begin{lem} \label{lemm:8}
  If we view $\chi^1$ as a rational function on $\mathcal{X}$, then there is an integer $a$ such that
\begin{displaymath}
  \div(\chi^1)=D_{0}-D_{\infty}+\sum_{i=0}^{l}(a-i)\bigg(E_{i}+\sum_{j\in \Theta_{i}}F_{i,j}\bigg). 
\end{displaymath}
\end{lem}
\begin{proof}
It is clear that
\begin{displaymath}
  \div(\chi^1)=D_{0}-D_{\infty}+\sum_{i=0}^{l}a_{i}E_{i}+\sum_{j\in \Theta
    _{i}}a_{i,j}F_{i,j} 
\end{displaymath}
for certain coefficients $a_{i}$ and $a_{i,j}$ that we want to
determine as much as possible. 

If a component $F$ of
$\mathcal{X}_{0}$, with coefficient
$a$, does not meet $D_{0}$ nor $D_{\infty}$, but meets $r\ge 1$ other
components, and the coefficients of $r-1$
of these components are equal to $a$, while the coefficient of the
remaining component is $b$, we obtain that
\begin{displaymath}
  0=\div(\chi^1)\cdot F=a F\cdot F+a(r-1)+b=-r a+a(r-1)+b=b-a
\end{displaymath}
Thus $b=a$. Starting with the components $F_{i,j}$ that are terminal,
we deduce 
that, for all $i$ and $j\in \Theta _{i}$, $a_{i}=a_{i,j}$. Therefore,
\begin{displaymath}
  \div(\chi^1)=D_{0}-D_{\infty}+\sum_{i=0}^{l}a_{i}\bigg(E_{i}+\sum_{j\in \Theta
    _{i}}F_{i,j}\bigg). 
\end{displaymath}
In particular, the lemma is proved for $l=0$. Assume now that
$l>0$. 

It only remains to show that $a_{i}=a_{0}-i$, that we prove by
induction. For $i=1$, 
we compute
\begin{displaymath}
  0=\div(\chi^1)\cdot E_{0}=D_{0}\cdot E_{0}+a_{0}E_{0}\cdot
  E_{0}+a_{0}\sum _{j\in \Theta _{0}}F_{0,j}\cdot
  E_{0}+a_{1}E_{1}\cdot E_{0}=1-a_{0}+a_{1}.
\end{displaymath}
Thus $a_{1}=a_{0}-1$. For $1<i\le l$, by induction hypothesis,
$a_{i-1}=a_{i-2}-1$. Then 
\begin{displaymath}
  0=\div(\chi^1)\cdot E_{i-1}=a_{i-2}-2a_{i-1}+a_{i}=1-a_{i-1}+a_{i}.
\end{displaymath}
  Hence $a_{i}=a_{i-1}-1$. Applying again the induction hypothesis, we
  deduce $a_i=a_{0}-i$, proving the lemma. 
\end{proof}

The determination of $\div(\chi^1)$ allows us to give a partial
description of the map $\red\colon X_{\Sigma }^{\an}\to
\mathcal{X}_{o}$. The points of $\mathcal{X}_{o}$ appearing in this
description, are the points $q_{0}:=D_{0}\cap E_{0}$,
$q_{i}:=E_{i-1}\cap E_{i}$, $i=1,\dots,l$, $q_{l+1}:= E_{l}\cap
D_{\infty}$ and the generic points of the components $E_{i}$ that we
denote $\eta_{i}$, $i=0,\dots,l$.

\begin{lem}\label{lemm:11}
  Let $p\in X_{\Sigma }^{\an}$. Then
  \begin{displaymath}
    \red(p)=
    \begin{cases}
      q_{0}& \text{ if } |\chi^1(p)|<|\varpi |^{a},\\
      q_{i}&\text{ if } |\varpi
      |^{a-i+1}<|\chi^1(p)|<|\varpi |^{a-i},\\ 
      q_{l+1}&\text{ if } |\varpi
      |^{a-l}<|\chi^1(p)|,\\
      \eta_{i} &
      \text{ if } |\chi^1(p)|=|\varpi |^{a-i} \text{ and }p\in\Im(\theta
      _{\Sigma }). 
    \end{cases}
  \end{displaymath}
\end{lem}
\begin{proof}
  Let $1\le i \le l$. The rational function $x:=\chi^1\varpi ^{-a+i}$ has a
  zero of order one along the component $E_{i-1}$ and the support of its divisor
  does not contain the component $E_{i}$. On the other hand, the
  rational function $y:=\chi^{-1}\varpi ^{a-i+1}$ has a
  zero of order one along the component $E_{i}$ and the support of its divisor
  does not contain the component $E_{i-1}$. Thus $\{x,y\}$ is a
  system of parameters in a neighbourhood of~$q_{i}$. We denote
  \begin{displaymath}
    A=K^{\circ}[\chi^1\varpi ^{-a+i},\chi^{-1}\varpi
  ^{a-i+1}]\simeq K^{\circ}[x,y]/(xy-\varpi ).
  \end{displaymath}

  The local ring at
  the point $q_{i}$ is $A_{(x,y)}$. Let $p$ be a
  point such that  $|\varpi
      |^{a-i+1}<|\chi^1(p)|<|\varpi |^{a-i}$. Therefore, for $f\in A$ we have
      $|f(p)|\le 1$. Moreover, if 
      $f\in (x,y)$, then
      $|f(p)|<1$. Since the ideal $(x,y)$ is maximal, we deduce that, for $f\in
      A$,  the condition 
      $|f(p)|<1$ is equivalent to the condition $f\in (x,y)$. This
      implies that $\red(p)=q_{i}$. A similar 
      argument works for $q_{0}$ and $q_{l+1}$. 

      Assume now that $p\in\Im(\theta _{\Sigma })$ and that
      $|\chi^1(p)|=|\varpi |^{a-i}$. If $i\not = 0$ we consider again the
      ring $A$, but in this case $|x(p)|=|\chi^1(p)\varpi ^{-a+i}|=1$. Let
      $I=\{f\in A\mid |f(p)|<1\}$. It is clear that $(y,\varpi
      )\subset I$. 
      For $f=\sum_{m\in \Z}\beta _{m}\chi^{m}\in A$, since $p\in
      \Im(\theta _{\Sigma })$, we have
      \begin{displaymath}
        |f(p)|=\sup_{m}(|\beta _{m}||\chi^1(p)|^{m}).
      \end{displaymath}
      This implies that $I\subset (y,\varpi )$. Hence $I$ is the ideal
      that defines the component $E_{i}$ and this is equivalent to
      $\red(p)=\eta_{i}$. The case $i=0$ is analogous. 
\end{proof}

The image by $\red$ of the remaining points of $ X_{\Sigma }^{\an}$
is not characterized only by the value of $|\chi^1(p)|$. Using a proof
similar to that of the lemma, one can show that, if $|\chi^1(p)|=|\varpi
|^{a-i}$ then $\red(p)$ belongs either to $E_{i}$ or to any of the 
components $F_{i,j}$, $j\in \Theta_{i}$. 

We denote by $\xi _{i}$ (respectively $\xi _{i,j}$) the point of $X_{\Sigma
}^{\an}$ associated by Proposition~\ref{prop:92} to the component $E_{i}$
(respectively $F_{i,j}$). These points satisfy $\red(\xi _{i})=\eta _{i}$ and
$\red(\xi_{i,j})=\eta _{i,j}$, where $\eta _{i,j}$ is the generic point of $F_{i,j}$. 

\begin{lem}\label{lemm:9}
  Let $0\le i\le l$. Then, for every $j\in \Theta _{i}$,
  \begin{displaymath}
    \val_K(\xi _{i})=\val_K(\xi_{i,j})=a-i,
  \end{displaymath}
  where $a$ is the integer of Lemma \ref{lemm:8}.
  Furthermore, for $j\in\Theta_i$ we have
  $$\xi_i = \theta_\Sigma(\rho_\Sigma(\xi_i)) = \theta_\Sigma(\rho_\Sigma(\xi_{i,j})).$$
\end{lem}
\begin{proof}
  We consider the rational function $\varpi ^{-a+i}\chi^1$. Since the
  support of $\div(\varpi ^{-a+i}\chi^1)$ does not contain the component
  $E_{i}$ nor any of the components $F_{i,j}$, we have
  \begin{displaymath}
    |\varpi ^{-a+i}\chi^1(\xi _{i})|=
    |\varpi ^{-a+i}\chi^1(\xi _{i,j})|=1.
  \end{displaymath}
  Then, using \eqref{eq:42} we deduce
  \begin{displaymath}
    \val_K(\xi_{i})=\frac{-\log|\chi ^{1}(\xi_{i})|}{\lambda _{K}
    }= \frac{-\log|\varpi ^{a-i}|}{-\log |\varpi |}=a-i
  \end{displaymath}
and similarly with $\xi_{i,j}$, which proves the first part of the statement.

From these equalities and the definitions \eqref{eq:111} and
\ref{prop:54} of the maps $\theta_\Sigma$ and $\rho_\Sigma$, the point
$\theta_\Sigma(\rho_\Sigma(\xi_i))\in X^\an_\Sigma$ is the
multiplicative seminorm
$$\sum_{m\in M}\alpha_m\chi^m \longmapsto \max_{m\in M}|\alpha_m\chi^m(\xi_i)| = |\varpi|^{\min_{m\in M}(\val_K(\alpha_m)+(a-i)m)}.$$
By Proposition \ref{prop:92}, the point $\xi_i\in X^\an_\Sigma$
is the multiplicative semi-norm
$$\sum_{m\in M}\alpha_m\chi^m \mapsto |\varpi|^{\ord_{E_i}(\sum_{m\in M}\alpha_m\chi^m)}.$$
Using the same notations as in the proof of Lemma \ref{lemm:11}, any
element $f = \sum_{m\in M}\alpha_m\chi^m \in A$ can be written $f =
\sum_{m\in\Z}\beta_m \varpi^{\val_K(\alpha_m)+(a-i+1)m}y^{-m}$ where
the $\beta_m$'s are units in $K^\circ$. Now, the ideal of definition
of $E_i$ is $(y,\varpi)$ and therefore $\ord_{E_i}(f) =
\min_{m\in\Z}(\val_K(\alpha_m)+(a-i)m)$. Thus, the point $\xi_i$
coincide with $\theta_\Sigma(\rho_\Sigma(\xi_i))$. This shows the
equality $\xi_i = \theta_\Sigma(\rho_\Sigma(\xi_i))$.  The remaining
equality follows then from the fact that $\rho_\Sigma(\xi_{i,j}) =
\rho_\Sigma(\xi_i)$, $j\in\Theta_i$, the image by $\rho_{\Sigma}$
depending only on the valuation $\val_K(\xi_i) = \val_K(\xi_{i,j})$.
This completes the proof of the statement.
\end{proof}

Let now $\Psi $ be a virtual support function on $\Sigma $ and
$L=\mathcal{O}(D_{\Psi })$ the associated toric line bundle on
$X_{\Sigma}$.  Let $e\ge 1$ and $\mathcal{L}$ a line bundle on
$\mathcal{X}$ which is a model over $\Spec(K^{\circ})$ of $L^{\otimes
  e}$. Let $\|\cdot\|$ be the metric on $L^{\an}$ given by the proper
model $(\cX,\cL,e)$ and $\phiK_{\|\cdot\|}$ the corresponding function
on $N_{\R}\simeq \R$.

Let $m_{0},m_{\infty}\in M\simeq\Z$ such that
\begin{displaymath}
  \Psi (u)=
  \begin{cases}
    m_{\infty} u & \text{ if } u\le 0,\\
    m_{0} u & \text{ if } u\ge 0.
  \end{cases}
\end{displaymath}
Then the divisor of the toric section $s_{\Psi}$ is given by
$\div(s_{\Psi })=-m_{0}[0]+m_{\infty}[\infty]$, and so $L\simeq
\mathcal{O}(m_{\infty}-m_{0})$.  

Let's now consider $s_{\Psi }^{\otimes e}$ as a rational section of
$\mathcal{L}$ and denote by $D$ its associated Cartier divisor, so
that $\cL\simeq\cO(D)$. Then
\begin{equation}\label{eq:52}
D=
  -em_{0}D_{0}+em_{\infty}D_{\infty}+\sum_{i=0}^{l}\bigg(\alpha
    _{i}E_{i} +\sum_{j\in \Theta
    _{i}}\alpha _{i,j}F_{i,j}\bigg)
\end{equation}
for certain coefficients $\alpha _{i}$ and $\alpha _{i,j}$. To have
more compact formulae, we will use the conventions
\begin{equation}   \label{eq:2}
\begin{matrix}
 E_{-1}=D_{0}, &  \alpha _{-1}=-em_{0}, & \Theta
  _{-1}=\emptyset ,\\ 
E_{l+1}=D_{\infty}, & \alpha
_{l+1}=em_{\infty}, & \Theta _{l+1}=\emptyset.
\end{matrix}
\end{equation}
For example, the equation (\ref{eq:52}) can then be written as
\begin{equation*}
D=\sum_{i=-1}^{l+1}\bigg(\alpha
    _{i}E_{i} +\sum_{j\in \Theta
      _{i}}\alpha _{i,j}F_{i,j}\bigg).  
\end{equation*}

\begin{lem} \label{lemm:10}
  The function $\phiK _{\|\cdot\|}$ is given by
  \begin{displaymath}
    \phiK _{\|\cdot\|}(u)=
    \begin{cases}
      m_{0}u-m_{0}a-\frac{\alpha _{0}}{e} &\text{ if } u\ge a,\\
      \frac{(\alpha _{i+1}-\alpha _{i})u-(\alpha _{i+1}-\alpha
        _{i})(a-i)-\alpha _{i}}{e} &\text{ if } a-i\ge u\ge a-i-1,\\
      m_{\infty}u-m_{\infty}(a-l)-\frac{\alpha _{l}}{e} &\text{ if }
      a-l\ge u.
    \end{cases}
  \end{displaymath}
In other words, 
if $\Pi $ is the polyhedral complex in $N_{\R}$ given by the
intervals
\begin{displaymath}
 (-\infty,a-l],\quad [a-i,a-i+1],\  i=1,\dots,l,\quad
[a,\infty),
\end{displaymath}
then
$ \phiK _{\|\cdot\|}$ is the rational piecewise affine function on $\Pi $
characterized by the conditions
\begin{enumerate}
\item $\rec(\phiK _{\|\cdot\|})=\Psi $,
\item the value of $\phiK
_{\|\cdot\|}$ at the point $a-i$ is $-\alpha _{i}/e$. 
\end{enumerate}
\end{lem}

\begin{proof}
  Let $p\in \Im (\theta _{\Sigma })$ be such that $\val_K(p)>a$, hence
  $|\chi^1(p)|<|\varpi |^{a}$. By Lemma~\ref{lemm:11}, this implies
  that $\red(p)=q_{0}$. In a neighbourhood of $q_{0}$, the divisor of
  $\cX$ of the rational section $s_{\Psi }^{\otimes
    e}\chi^{em_{0}}\varpi ^{-\alpha _{0}-em_{0}a}$ is zero, and so
  \begin{displaymath}
    \|s_{\Psi }^{\otimes e}(p)\chi^{em_{0}}(p)\varpi ^{-\alpha
    _{0}-em_{0}a}\|=1.
  \end{displaymath}
Let $u\in N_\R$ and $p\in \Im (\theta _{\Sigma })$ such that $\val_K(p)=u$. Then, by definitions \ref{def:88} and~\ref{def:82},
\begin{align*}
  \phiK _{\|\cdot\|}(u)&=\frac{\log\|s_{\Psi }^{\otimes e}(p)\|}{e \lambda _{K}}\\
  &=
  \frac{-em_{0}\log|\chi^1(p)|+(\alpha _{0}+em_{0}a)\log|\varpi |}
  {-e\log|\varpi |}\\
  &=m_{0}(u-a)-\frac{\alpha _{0}}{e}.
\end{align*}
The other cases are proved in a similar way.
\end{proof}

Since $\rec(\Pi )=\Sigma $, by Theorem \ref{thm:7}, this polyhedral
complex defines a toric model $\mathcal{X}_{\Pi }$ of $X_{\Sigma
}$.

\begin{prop}\label{prop:25}
  The identity map of $X_{\Sigma }$ extends to an isomorphism of models
  $\mathcal{X}_{\SSinv}\to \mathcal{X}_{\Pi }$.
\end{prop}
\begin{proof}
  The special fibre of $\mathcal{X}_{\Pi }$ is a chain of rational
  curves $E_{i}$, $i=0,\dots,l$, corresponding to the points
  $a-i$. The monomial $\chi ^{1}$ is a section of the trivial line
  bundle on $\cX_\Pi$ and corresponds to the function $\phiK (u)=-u$. Using
  Proposition 
  \ref{prop:51} we obtain that
  \begin{displaymath}
    \div(\chi ^{1})=D_{0}-D_{\infty}+\sum_{i=0}^{l}(a-i)E_{i}, 
  \end{displaymath}
  where $D_{0}$ and $D_{\infty}$ are again the horizontal divisors
  determined by the points $0$ and~$\infty$.

  Since the vertices of the polyhedral complex $\Pi $ are
  integral, using  the equation~\eqref{eq:44} we deduce that $\div(\varpi )$
  is reduced.   

  Then the result follows from \cite[Corollary 1.13]{MR0230724} using
  an explicit description of the local rings at the points of the
  special fibre as in the proof of  Lemma \ref{lemm:11}. 
\end{proof}

From Proposition \ref{prop:25}, we obtain a proper morphism
$\pi\colon \mathcal{X}\to \cX_{\SSinv}\simeq\mathcal{X}_{\Pi }$. 
Set $D_{\SSinv}=\pi _{\ast}D$ with $D$ the divisor on $\cX$ in the
equation~\eqref{eq:52}. Then
\begin{equation}\label{eq:85}
  D_{\SSinv}=-em_{0}\pi _{\ast}D_{0}+em_{\infty}\pi _{\ast}D_{\infty}+\sum_{i=0}^{l}\alpha 
    _{i}\pi _{\ast}E_{i}=\sum_{i=-1}^{l+1}\alpha 
    _{i}\pi _{\ast}E_{i},
\end{equation}
where in the last expression we use the conventions  \eqref{eq:2}.

Let $D_{e\phiK _{\|\cdot\|}}$ be the $\T$-Cartier divisor on
$\cX_{\Pi}$ determined by the function $e\phiK _{\|\cdot\|}$ as in
(\ref{eq:59}).  Using equation~\eqref{eq:85},
Proposition~\ref{prop:51} and Lemma \ref{lemm:10}, we see that
$D_{\SSinv}=D_{e\phiK _{\|\cdot\|}}$.  Thus
$(\cX_{\SSinv},\cO(D_{\SSinv}))$ is the toric model
of~$(X_\Sigma,L^{\otimes e})$ induced by $(\Pi,e\phiK_{\|\cdot\|})$
through Theorem~\ref{thm:11b} and Proposition~\ref{prop:89}. Let
$\|\cdot\| _{\SSinv}$ be the toric metric associated to $\|\cdot\|$
(Definition~\ref{def:82}). By Propositions \ref{prop:59} and
\ref{prop:14}, the metric $\|\cdot\|_{\SSinv}$ agrees with the toric
metric defined by the model $(\cX_{\SSinv},\cO(D_{\SSinv}),e)$. Thus,
we have identified a toric model that corresponds to the metric
$\|\cdot\|_{\SSinv}$.  This allows us to compute directly the
associated measure.

\begin{lem} \label{lemm:21} Let $X_{\Sigma }\simeq \P^{1}_{K}$ be a
  one-dimensional toric variety over $K$ and $L\simeq
  \mathcal{O}(D_{\Psi })$ a toric line bundle on $X_{\Sigma}$. Let
  $\|\cdot\|$ be an algebraic metric defined by a regular semi-stable
  model whose components and double points of the special fibre are
  defined over the residue field $K^{\circ}/K^{\circ \circ}$ and such
  that each irreducible component contains a rational point.  Let
  $\|\cdot\|_{\SSinv}$ the associated toric metric. Then
  \begin{displaymath}
    \chern_{1}(L,\|\cdot\|_{\SSinv})\land \delta _{X_{\Sigma
    }}=(\theta _{\Sigma })_{\ast}(\rho _{\Sigma
  })_{\ast}\left(\chern_{1}(L,\|\cdot\|)\land \delta _{X_{\Sigma
    }}\right) =(\theta _{\Sigma })_{\ast}(\ee_{K})_{\ast}(-\phiK''_{\|\cdot\|}),
  \end{displaymath}
where the second derivative $\phiK''_{\|\cdot\|}$ is taken in the
sense of distributions.
\end{lem}
\begin{proof}
  Let $(\cX,\cL,e)$ be a regular semi-stable model defining the metric
  $\|\cdot\|$. Let $D$ be the divisor on $\cX$ defined by the rational
  section $s_\Psi^{\otimes e}$, so that $\cL \simeq \cO(D)$.  Since the
  special fibre is reduced, by the equation \eqref{eq:11} we have
  \begin{displaymath}
     \chern_{1}(L,\|\cdot\|)\land \delta _{X_{\Sigma
    }} =\frac{1}{e}\sum_{i=0}^{l}\bigg(\deg_{\mathcal{L}}(E_{i})\delta _{\xi _{i}}
      +\sum _{j\in \Theta _{i}}\deg_{\mathcal{L}}(F_{i,j})\delta _{\xi
        _{i,j}}\bigg).
  \end{displaymath}
Denote  temporarily this measure by $\mu $. Then, using the conventions
 \eqref{eq:2} and Lemma \ref{lemm:9}, for $l>0$, we obtain,
\begin{align*}
  (\theta _{\Sigma })_{\ast}(\rho _{\Sigma })_{\ast}\mu&=
  \frac{1}{e}\sum_{i=0}^{l}\bigg(\deg_{\mathcal{L}}(E_{i})
      +\sum _{j\in \Theta _{i}}\deg_{\mathcal{L}}(F_{i,j})\bigg)\delta
    _{\xi _{i}}\\ 
    &=
  \frac{1}{e}\sum_{i=0}^{l}\bigg(D\cdot E_{i}
      +\sum _{j\in \Theta _{i}}D\cdot F_{i,j}\bigg)\delta
    _{\xi _{i}}\\
    &=\frac{1}{e}\sum_{i=0}^{l}
    \sum_{r=-1}^{l+1}\bigg(\alpha _{r} E_{r}
      +\sum _{s\in \Theta _{r}}\alpha _{r,s}F_{r,s}\bigg)\cdot
  \bigg(E_{i}
      +\sum _{j\in \Theta _{i}}F_{i,j}\bigg)\delta
    _{\xi _{i}}\\
    &=\frac{1}{e}\sum_{i=0}^{l}
    \bigg(\alpha _{i-1} E_{i-1}+\alpha _{i} E_{i}+\alpha _{i+1} E_{i+1}
      \bigg)\cdot
  \bigg(E_{i}
      +\sum _{j\in \Theta _{i}}F_{i,j}\bigg)\delta
    _{\xi _{i}}\\
    &=\frac{1}{e}\sum_{i=1}^{l-1}(\alpha _{i-1}-2\alpha _{i}+\alpha
    _{i+1})\delta_{\xi _{i}}+\frac{1}{e}\sum_{i=0,l}(\alpha
    _{i-1}-\alpha _{i}+\alpha 
    _{i+1})\delta_{\xi _{i}},
  \end{align*}
  while, for $l=0$, we obtain
  \begin{displaymath}
      (\theta _{\Sigma })_{\ast}(\rho _{\Sigma })_{\ast}\mu=
      (\alpha
    _{-1}+\alpha 
    _{1})\delta_{\xi _{0}}.
  \end{displaymath}

  In the previous computation we have used that, since $E_{r}\cdot \div(\varpi
  )=F_{r,s}\cdot \div(\varpi )=0$, then, for all $i,j,r,s$, 
  \begin{displaymath}
    F_{r,s}\cdot(E_{i}+\sum_{j\in \Theta _{i}}F_{i,j})=0
  \end{displaymath}
and
  \begin{displaymath}
       E_{r}\cdot\bigg(E_{i}+\sum_{j\in \Theta _{i}}F_{i,j}\bigg)=
    \begin{cases}
      0&\text{ if }r\not = i-1,i,i+1,\\
      1&\text{ if }r=i-1,i+1,\\
      -2&\text{ if }0<r=i<l,\\
      -1&\text{ if }0<l, r=i=0,l\\
      0&\text{ if }r=i=l=0.\\
    \end{cases}
  \end{displaymath}

  An analogous computation with the model $(\cX_{\SSinv},\cO(D_{\SSinv}),e)$ 
  shows that
$$\chern_{1}(L,\|\cdot\|_{\SSinv})\land \delta _{X_{\Sigma
  }}=\frac{1}{e}\sum_{i=0}^{l}\deg_{\cO(D_{\SSinv})}(\pi _{\ast}E_i)\delta_{\xi
  _{i}} = \frac{1}{e}\sum_{i=0}^{l}(D_{\SSinv}\cdot
\pi _{\ast}E_i) \delta_{\xi_i}$$ where the intersection product now is on
$\cX_{\SSinv}$.  Using again the conventions \eqref{eq:2} and~\eqref{eq:85}, we get
\begin{displaymath}
  D_{\SSinv}\cdot\pi _{\ast}E_i = \sum_{r=-1}^{l+1}\alpha_r
  (\pi _{\ast}E_r\cdot\pi _{\ast}E_i)  = 
  \begin{cases}
    \alpha_{i-1}-2\alpha_i+\alpha_{i+1}&\text{ if }0<i<l,\\ 
    \alpha_{i-1}-\alpha_i+\alpha_{i+1}&\text{ if }0<l,i=0,l,\\ 
    \alpha_{-1}+\alpha_{1}&\text{ if }0=l,i=0.\\     
  \end{cases}
\end{displaymath}
For $l>0$, we obtain
\begin{equation}\label{eq:86}
    \chern_{1}(L,\|\cdot\|_{\SSinv})\land \delta _{X_{\Sigma
    }}=
    \frac{1}{e}\sum_{i=1}^{l-1}(\alpha _{i-1}-2\alpha _{i}+\alpha
    _{i+1})\delta_{\xi _{i}}+\frac{1}{e}\sum_{i=0,l}(\alpha _{i-1}-\alpha _{i}+\alpha
    _{i+1})\delta_{\xi _{i}}
  \end{equation}
  while, for $l=0$,
  \begin{equation}
    \label{eq:145}
    \chern_{1}(L,\|\cdot\|_{\SSinv})\land \delta _{X_{\Sigma
    }}=\frac{1}{e}(\alpha _{-1}+\alpha
    _{1})\delta_{\xi _{0}}.
  \end{equation}
  The first equality follows then by comparing equations \eqref{eq:86}
  and \eqref{eq:145} with the previous computation of $(\theta
  _{\Sigma })_{\ast}(\rho _{\Sigma })_{\ast}\mu$.

For the second equality, Lemma \ref{lemm:11} and the definition of
$\xi_{i}$ imply that $\xi_{i}=\theta_{\Sigma }(\ee_{K}(a-i))$.  
Hence 
\begin{displaymath}
(\theta_{\Sigma})_{*}(\ee_{K})_{*}(\delta_{a-i})= \delta_{\xi_{i}}.  
\end{displaymath}
The result follows from the explicit description of $\phiK
_{\|\cdot\|}$ in Lemma \ref{lemm:10} and the explicit description of
$\chern_{1}(L,\|\cdot\|_{\SSinv})\land \delta _{X_{\Sigma }}$ given by
\eqref{eq:86} and \eqref{eq:145}.
\end{proof}

Using Proposition \ref{prop:12}, we can extend Lemma \ref{lemm:21} to the
case when the model is not semi-stable.

\begin{thm} \label{thm:28}
  Let $X_{\Sigma }\simeq \P^{1}_{K}$ be a
  one-dimensional toric variety over $K$. Let $L\simeq
  \mathcal{O}(D_{\Psi })$ be a toric line bundle, $\|\cdot\|$ an
  algebraic metric and
  $\|\cdot\|_{\SSinv}$ the associated toric
  metric. Then
  \begin{displaymath}
    \chern_{1}(L,\|\cdot\|_{\SSinv})\land \delta _{X_{\Sigma
    }} =(\theta _{\Sigma })_{\ast}(\rho _{\Sigma
  })_{\ast}\left(\chern_{1}(L,\|\cdot\|)\land \delta _{X_{\Sigma
    }} \right)  =(\theta _{\Sigma })_{\ast}(\ee_{K})_{\ast}
(-\phiK''_{\|\cdot\|}),
  \end{displaymath}
where the second derivative $\phiK''_{\|\cdot\|}$ is taken in the
sense of distributions.
\end{thm}
\begin{proof}
  Let $(\mathcal{X},\mathcal{L},e)$ be a proper model of $(X_{\Sigma
  },L^{\otimes e})$ that realizes the algebraic metric
  $\|\cdot\|$. For short, denote
  \begin{displaymath}
  \mu=\chern_{1}(L,\|\cdot\|)\land
  \delta _{X_{\Sigma }} ,\quad   
\mu_{\SSinv}=\chern_{1}(L,\|\cdot\|_{\SSinv})\land \delta_{X_{\Sigma }} .
  \end{displaymath}
  By Proposition \ref{prop:12}, there is a finite extension $K'$ of
  $K$, a regular semi-stable model $\cX'$ of $X_{\Sigma ,K'}$ and a
  proper morphism of models $\cX'\to\cX\times \Spec (K'^{\circ})$. We
  may further assume that all the components and the double points of
  the special fibre of $\cX'$ are defined over
  ${K'}^{\circ}/{K'}^{\circ\circ}$ and that each irreducible component
  contains a rational point. Let $(L',\|\cdot\|')$ be the metrized
  line bundle obtained by base change to $K'$. Using propositions
  \ref{prop:21} and \ref{prop:115}, it is possible to show that the
  toric metric $(\|\cdot\|')_{\SSinv}$ agrees with the metric obtained
  from $\|\cdot\|_{\SSinv}$ by base change. We denote by $\nu \colon
  X^{\an}_{\Sigma ,K'}\to X^{\an}_{\Sigma,K}$ the map of analytic
  spaces and by $\mu '$, $\mu '_{\SSinv}$, $\theta _{\Sigma }'$ and
  $\rho _{\Sigma }'$ the corresponding objects for $X_{\Sigma
    ,K'}$. Then, by the propositions \ref{prop:28}, \ref{prop:114} and
  \ref{prop:113} and the first equality in Lemma \ref{lemm:21},
  \begin{displaymath}
    \mu _{\SSinv}=\nu _{\ast}\mu '_{\SSinv}=\nu _{\ast}(\theta'
    _{\Sigma })_{\ast}(\rho' _{\Sigma })_{\ast}\mu'=
    (\theta _{\Sigma })_{\ast}(\rho' _{\Sigma })_{\ast}\mu'=
    (\theta _{\Sigma })_{\ast}(\rho _{\Sigma })_{\ast}\mu,
  \end{displaymath}
which proves the first equality of the theorem. 

Using the second equality in Lemma \ref{lemm:21} and Proposition
\ref{prop:114}, we deduce that
  \begin{equation}\label{eq:97}
    \mu _{\SSinv}=\nu _{\ast}\mu '_{\SSinv}=\nu _{\ast}(\theta'
    _{\Sigma })_{\ast}(\ee_{K'})_{\ast} (-\phiK''_{\|\cdot\|'})=(\theta
    _{\Sigma })_{\ast}(\ee_{K'})_{\ast} (-\phiK''_{\|\cdot\|'}).
  \end{equation}
  We have that, for $u\in N_{\R}$,
  \begin{displaymath}
    \ee_{K'}(u)=\ee_{K}\Big(\frac{u}{e_{K'/K}}\Big), \quad 
\phiK_{\|\cdot\|'}(u)=
  e_{K'/K}\, \phiK_{\|\cdot\|}\Big(\frac{u}{e_{K'/K}}\Big),
  \end{displaymath}
  where the second equality follows from Proposition \ref{prop:115}.
  Using these formulae, one can verify that
\begin{displaymath}
  (\ee_{K'})_{\ast} =   (\ee_{K'})_{\ast}
  \Big(\frac{1}{e_{K'/K}}\Big)_{*},\quad
\phiK''_{\|\cdot\|'}=({e_{K'/K}})_{*}
\phiK''_{\|\cdot\|},
\end{displaymath}
where $\frac{1}{e_{K'/K}}$ and $e_{K'/K}$ denote the corresponding
homotheties of $N_{\R}$. This implies  
\begin{displaymath}
  (\ee_{K'})_{\ast} \big(\phiK''_{\|\cdot\|'}\big)=(\ee_{K})_{\ast} 
\big(\phiK''_{\|\cdot\|}\big).
\end{displaymath}
The second equality in the statement then follows from this equation
together with~(\ref{eq:97}).
\end{proof}

We can now relate semipositivity of the metric with concavity of the
associated function in the one-dimensional case.

\begin{cor}\label{cor:5}
  Let $X_{\Sigma }\simeq \P^{1}_{K}$ be a one-dimensional toric
  variety over $K$. Let $(L,s)$ be a toric line bundle with a toric
  section and $\|\cdot\|$ a metric with a semipositive model. Then the
  function $\phiK _{\|\cdot\|}$ is concave and the toric metric
  $\|\cdot\|_{\SS}$ has a semipositive model.
\end{cor}

\begin{proof}
  Since the metric $\|\cdot\|$ has a semipositive model,
  $\chern_{1}(L,\|\cdot\|)\land \delta _{X_{\Sigma }} $ is a measure,
  and not just a signed measure. 
  Theorem \ref{thm:28} implies that the direct image by $
  (\val_{K})_{\ast }(\theta _{\Sigma })_{\ast}(\ee_{K})_{\ast}$ of
  this measure coincides with $-\phiK''_{\|\cdot\|}$. Hence
  $-\phiK''_{\|\cdot\|}$ is also a measure and so $\phiK _{\|\cdot\|}$
  is concave, proving the first statement.

  For the second statement, observe that $\phiK
  _{\|\cdot\|_{\SS}}= \phiK _{\|\cdot\|}$.  By Proposition
  \ref{prop:14}\eqref{item:124}, the recession of this function agrees
  with $\Psi$. Corollary \ref{cor:12}
  then implies that the metric $\|\cdot\|_{\SSinv}$ has a semipositive toric
  model.
\end{proof}

\section{Algebraic metrics and their associated
  measures} 
\label{sec:pos-algebr-metr}

We come back to the setting of \S \ref{sec:toric-algebr-metr}. We
assume that $K$ is a complete field with respect to an absolute value
associated to a nontrivial discrete valuation and that $\Sigma$ is a
complete fan.  Let $\Psi $ be a virtual support function on $\Sigma $ and set
$(L,s)=(L_{\Psi },s_{\Psi })$. 

\begin{prop} \label{prop:61} Let $\|\cdot\|$ be a metric with a
  semipositive model on $L^{\an}$. Then both functions
  $\psiabs_{\|\cdot\|}$ and $\phiK _{\|\cdot\|}$ are concave.
\end{prop}
\begin{proof} Assume that $\|\cdot\|$ has a semipositive model.  Since
  the condition of being concave is closed, if we prove that, for all
  choices of $u_{0}\in \lambda_{K}N_{\Q}$ and $v_{0}\in N$ primitive,
  the restriction of $\psiabs _{\|\cdot\|}$ to the line $u_{0}+\R
  v_{0}$ is concave, we will deduce that the function $\psiabs
  _{\|\cdot\|}$ is concave. Fix $u_{0}\in \lambda_{K}N_{\Q}$ and
  $v_{0}\in N$ primitive and let $e\ge1$ such that $e u_{0}\in
  \lambda_{K} N$. Then $K':=K(\varpi ^{1/e})$ is a finite extension of
  $K$ and there is a unique extension of the absolute value of $K$ to
  $K'$. We will denote with $'$ the objects obtained by base change to
  $K'$.
  
  We consider the affine map $A\colon \Z \to N$ given by $l\mapsto
  eu_{0}+l v_{0}$, and let $H$ be the linear part of $A$. By Theorem
  \ref{thm:6}\eqref{item:98}, the subset of $X^{\an}_{\Sigma ,K'}$ of
  algebraic points is dense. By \cite[Th\'eor\`eme 5.3]{Poineau:eba},
  we can choose a sequence $K'_{i}$, $i\in 
  \N$, of finite extensions of $K'$ and a sequence of points $\wt q_{i}\in
  X_{\Sigma ,0}(K'_{i})$, $i\in \N$, such that, if we denote by $q_{i}$
  the image of $\wt q_{i}$ in $X^{\an}_{\Sigma ,K'}$, then
  $\val(q_{i})=u_{0}$ and
  \begin{displaymath}
    \lim_{i\to\infty}q_{i}
    = \theta_\Sigma(\ee(u_0)).
  \end{displaymath}
  Recall the equivariant morphisms $\varphi_{\wt q_{i},H}\colon
  \P^{1}_{K'_{i}}\to X_{\Sigma,K'_{i} }$ of Theorem~\ref{thm:25}. By
  Definition \ref{def:88} and Proposition
  \ref{prop:95}\eqref{item:127}, we have
  \begin{displaymath}
    \psiabs_{\|\cdot\|}(u_0+uv_0)= \lim_{i\to\infty} 
\psiabs_{\varphi^{\ast}_{\wt q_{i},H}\|\cdot\|_{K'_{i}}}(u). 
  \end{displaymath}
  By Corollary \ref{cor:5}, for each $i\in \N$, the function
  $\psiabs_{\varphi^{\ast}_{\wt q_{i},H}\|\cdot\|_{K'_{i}}}$ is
  concave. Since the limit of concave functions is concave, the
  restriction of $\psiabs_{\|\cdot\|}$ to $u_{0}+\R v_{0}$ is concave. We
  conclude that $\psiabs _{\|\cdot\|}$ is concave. Hence, 
$ \phiK _{\|\cdot\|}$ is concave too. 
\end{proof}

\begin{cor}\label{cor:23}
  Let $\|\cdot\|$ be a metric with a semipositive model on
  $L^{\an}$. Then the toric metric $\|\cdot\|_{\SS}$ has a
  semipositive toric model.
\end{cor}
\begin{proof}
  By Proposition \ref{prop:61}, the function $\phiK=\phiK _{\|\cdot\|}$ is
  concave.  By Theorem~\ref{thm:15}\eqref{item:114}, it is also
  rational piecewise affine. By Proposition
  \ref{prop:14}\eqref{item:124}, its recession agrees with $\Psi$,
  hence this latter is concave.  Corollary \ref{cor:12} then implies
  that the metric $\|\cdot\|_{\SSinv}=\|\cdot\|_{\phiK \lambda _{K}}$ has a
  semipositive toric model.
\end{proof}

  Putting together Proposition \ref{prop:61} and Theorem \ref{thm:15},
  we see that the relationship between semipositivity of the metric
  and concavity of the associated function given in the Archimedean
  case by Proposition \ref{prop:15} carries over to the
  non-Archimedean case.

\begin{cor}\label{cor:44} 
  Let $\|\cdot\|$ be a toric algebraic metric and $\phiK _{\|\cdot\|}$
  the associated function. Then $\|\cdot\|$ has a semipositive model if and
  only if the function $\phiK _{\|\cdot\|}$ is concave.
\end{cor}

We can now characterize the Chambert-Loir measure associated to a toric
metric with a semipositive model.

\begin{thm} \label{thm:17} Let $\|\cdot\|$ be a toric metric on
  $L^{\an}$ with a semipositive model and $\phiK =\phiK _{\|\cdot\|}$
  the associated function on $N_{\R}$. Let $\chern_{1}(\overline
  L)^{\wedge n}\land \delta _{X_{\Sigma }}$ be the associated
  measure. Then
  \begin{equation}\label{eq:57}
    \val_{K,\ast}(\chern_{1}(\overline L)^{\wedge n}\land
  \delta _{X_{\Sigma }})=n!\, \ov{\mathcal{M}}_{M}(\phiK),
  \end{equation}
  where $\ov{\mathcal{M}}_{M}(\phiK)$ is the measure  in Definition
  \ref{def:83}. Moreover,
  \begin{equation}
    \label{eq:58}
    \chern_{1}(\overline L)^{\wedge n} \land
    \delta _{X_{\Sigma }} =(\theta _{\Sigma })_{\ast}(\ee_{K})_{\ast}n!
    \ov{\mathcal{M}}_{M}(\phiK ). 
  \end{equation}
\end{thm}
\begin{proof}
  Since the metric has a semipositive model, by Proposition
  \ref{prop:61} the function $\phiK $ is concave. By Theorem
  \ref{thm:15} it is defined by 
  a toric 
  model $(\cX_{\Pi },D_{\phiK },e)$ of $(X_{\Sigma },D_{\Psi
  })$ in the equivalence class determined by $\phiK $. As in Remark
  \ref{rem:13}, the irreducible
  components of $\cX_{\Pi ,o}$ are in bijection with the vertices of
  $\Pi $. For each vertex $v\in \Pi ^{0}$, let $\xi_{v}$ be the point of
  $X_{\Sigma }^{\an}$ corresponding to the generic point of
  $V(v)$ given by Proposition \ref{prop:92}. Then, by the equations
  \eqref{eq:11} and \eqref{eq:44},
  \begin{displaymath}
    \chern_{1}(\overline L)^{\wedge n} \land
    \delta _{X_{\Sigma }} 
    =\frac{1}{e^{n}}\sum_{v\in \Pi ^{0}}\mult(v)\deg_{D_{\phiK
      }}(V(v))\delta _{\xi_{v}}. 
  \end{displaymath}
  Thus, by Corollary \ref{cor:1},
  \begin{displaymath}
    \val_{K,\ast}(\chern_{1}(\overline L)^{\wedge n}\land
  \delta _{X_{\Sigma }})=
    \frac{1}{e^{n}}\sum_{v\in \Pi ^{0}}\mult(v)\deg_{D_{\phiK
      }}(V(v))\delta _{v}.
  \end{displaymath}
  But, using Proposition \ref{prop:32} and Proposition \ref{prop:31},
  the 
Monge-Amp\`ere measure is given by 
  \begin{align*}
    \mathcal{M}_{M}(\phiK )&=\frac{1}{e^{n}}\mathcal{M}_{M}(e\phiK )\\
    &=\frac{1}{e^{n}}\sum_{v\in \Pi ^{0}}\Vol_{M}(v^{\ast})\delta
    _{v}\\
    &=\frac{1}{n!e^{n}}\sum_{v\in \Pi ^{0}}\mult(v)\deg_{D_{\phiK
      }}(V(v))\delta _{v}.
  \end{align*}
  Since $\mathcal{M}_{M}(\phiK )$ is a finite sum of Dirac delta measures, we
  obtain that
  \begin{displaymath}
    \ov{\mathcal{M}}_{M}(\phiK )=
    \frac{1}{n!e^{n}}\sum_{v\in \Pi ^{0}}\mult(v)\deg_{D_{\phiK}}(V(v))\delta _{v}.
  \end{displaymath}
  Hence we have proved \eqref{eq:57}. 
  To prove 
  \eqref{eq:58}, we just observe that
  $\xi_{v}=\theta
  _{0}\circ \ee_{K}(v)$.
\end{proof}

We can rewrite Theorem \ref{thm:17} in terms of the function
$\psiabs_{\|\cdot\|}$ of Definition \ref{def:68}. 

\begin{cor} \label{cor:34}  Let $\|\cdot\|$ be a toric metric on
  $L^{\an}$ with a semipositive model and $\psiabs =\psiabs _{\|\cdot\|}$
  the associated function on $N_{\R}$. Let $\chern_{1}(\overline
  L)^{\wedge n}\land \delta _{X_{\Sigma }}$ be the associated
  measure. Then
  \begin{displaymath}
    \val_{\ast}(\chern_{1}(\overline L)^{\wedge n}\land
  \delta _{X_{\Sigma }})=n!\, \ov{\mathcal{M}}_{M}(\psiabs).
  \end{displaymath}
  Moreover,
  \begin{displaymath}
    \chern_{1}(\overline L)^{\wedge n} \land
    \delta _{X_{\Sigma }} =(\theta _{\Sigma })_{\ast}(\ee)_{\ast}n!
    \ov{\mathcal{M}}_{M}(\psiabs ). 
  \end{displaymath}
\end{cor}

\section{Semipositive{} and DSP metrics}
\label{sec:appr-integr-invar}

Let $K$ be a valued field and $\T$ an $n$-dimensional split torus over
$K$, as in the beginning of the chapter. In the non-Archimedean
case assume furthermore that the valuation is discrete.
Let $\Sigma $ be a complete fan in $N_{\R}$ and $\Psi $ a virtual
support function on
$\Sigma $, and denote by  $(L, s)$  the corresponding toric line
bundle and section.




We are now in position to characterize toric semipositive{} metrics.

\begin{thm}\label{thm:13} Let notation be as above. 
  \begin{enumerate}
  \item \label{item:37} The assignment $\|\cdot\|\mapsto \psiabs
    _{\|\cdot\|}$ is a bijection
    between the space of semipositive{} toric
    metrics on $L^{\an}$ and the space of 
    concave functions $\psiabs $ on $N_{\R}$ such that $|\psiabs -\Psi |$ is
    bounded. 
  \item \label{item:38} Assume that $\Psi$ is a support function and
    let $\Delta_{\Psi}$ be the corresponding polytope. The
    assignment $\|\cdot\|\mapsto \psiabs _{\|\cdot\|}^{\vee}$ is a
    bijection between the space of semipositive{} toric metrics on
    $L^{\an}$ and the space of continuous concave functions on $\Delta
    _{\Psi }$. 
  \end{enumerate}
\end{thm}
\begin{proof}
  To prove the statement \eqref{item:37}, consider a toric
  semipositive{} metric $\|\cdot\|$.  By Corollary \ref{cor:8} the
  function $|\psiabs _{\|\cdot\|}-\Psi |$ is bounded.  By Definition
  \ref{def:65}, there is a sequence $(\|\cdot\|_{l})_{l\ge1}$ of
  smooth (respectively algebraic) semipositive metrics that converges
  to $\|\cdot\|$. Since $\|\cdot\|$ is toric, $\|\cdot\|_{\SSinv}=\|\cdot\|$. 
  Hence, by Proposition \ref{prop:60}, the sequence of toric metrics
  $(\|\cdot\|_{l,\SSinv})_{l\ge1}$ also converges to $\|\cdot\|$. We
  set $\psiabs _{l}=\psiabs_{\|\cdot\|_{l,\SSinv}}$.  By the
  propositions~\ref{prop:30} and~\ref{prop:15} in the Archimedean case
  and Proposition \ref{prop:61} and Corollary \ref{cor:23} in the
  non-Archimedean case, the functions $\psiabs _{l}$ are
  concave. Since, by Proposition \ref{prop:24}\eqref{item:126},
  the sequence $(\psiabs _{l})_{l\ge 1}$ converges uniformly on $N_{\R}$
  to $\psiabs _{\|\cdot\|}$, the latter is concave. 

  Conversely, let now $\psiabs $ be a concave function on $N_{\R}$
  such that $|\psiabs-\Psi |$ is bounded. Then $\Psi$ is a support
  function and $\Stab(\psiabs )=\Stab(\Psi )$ agrees with the polytope
  $ \Delta _{\Psi }$. Let~$\|\cdot\|$ be the metric on the restriction
  of $L^{\an}$ to $X_{0}^{\an}$ determined by $\psiabs $. Write
  $\phiK=\psiabs \lambda ^{-1}_{K}$. By
  Proposition~\ref{prop:22} there is a sequence of rational piecewise
  affine concave functions $(\phiK _{l})_{l\ge 1}$ with stability
  set $\Delta _{\Psi }$, that converge uniformly to $\phiK$. Since 
  $\Stab(\phiK _{l})=\Delta _{\Psi }$, by Proposition~\ref{prop:17},
  $\rec(\phiK _{l})=\Psi $. Since $\phiK_{l}$ is a piecewise
  affine concave function, by Remark \ref{rem:25}, $\phiK_{l}-\Psi 
  $ can be extended to a continuous function on $N_{\Sigma
  }$. Therefore, $\phiK-\Psi $ and hence $\psiabs-\Psi$, can be
  extended to a continuous 
  function on $N_{\Sigma }$. Consequently the metric $\|\cdot\|$ can
  be extended to $X^{\an}_\Sigma$.  Let $\|\cdot\|_{l}$ be the metric
  associated to $\phiK _{l}\lambda _{K}$. Then the sequence of metrics
  $(\|\cdot\|_{l})_{l\ge1}$ converges to $\|\cdot\|$. By Corollary
  \ref{cor:22}, the metric $\|\cdot\|_{l}$ is semipositive{}, both in
  the Archimedean and in the non-Archimedean cases.  We deduce that
  $\|\cdot\|$ is semipositive{}, which completes the proof
  of~\eqref{item:37}.

  The statement \eqref{item:38} follows from~\eqref{item:37} and
  propositions~\ref{prop:16} and \ref{prop:13}\eqref{item:73}.
\end{proof}

\begin{rem} \label{rem:32} With notations as in Theorem \ref{thm:13},
  if $\Psi$ is a support function, then the space in \eqref{item:37}
  coincides with $\ov{\mathscr{P}}(N_{\R},\Delta_{\Psi})$ (Definition
  \ref{def:52}), otherwise it is empty.  The space in \eqref{item:38}
  coincides with $\ov{\mathscr{P}}(\Delta_{\Psi},N_{\R})$.
\end{rem}

\begin{rem} \label{rem:12}
For the case $K=\C$, 
statement \eqref{item:38} in the above result is related to the Guillemin-Abreu classification
of K\"ahler structures on {symplectic toric varieties}
\index{toric variety!symplectic}%
as explained in~\cite{MR1969265}.
By definition, a symplectic toric variety is a compact symplectic  
manifold of dimension $2n$ together with a Hamiltonian action of the compact
torus $\SS \simeq (S^{1})^{n}$. These spaces are classified by
Delzant polytopes of $M_{\R}$, see for instance \cite[Definition page
8]{Gui94} for the definition of Delzant polytope and \cite[Appendix 1]{Gui94}
for the classification.
For a given {Delzant polytope}
$\Delta\subset M_{\R}$, 
the possible $(S^{1})^{n}$-invariant K\"ahler forms on the 
symplectic toric variety corresponding to $\Delta$ are classified by smooth convex functions on
$\Delta^{\circ}$ satisfying some conditions near the border of
$\Delta$.
Several differential geometric invariants of a K\"ahler toric variety can be
translated and studied in terms of this
convex function, also called the ``{symplectic
  potential}''.\index{symplectic potential}

For a  smooth positive toric metric $\|\cdot \|$ on 
 $L_{\Psi_{\Delta}}(\C)$, the Chern form defines a K\"ahler
structure on the complex toric variety $X_{\Sigma_{\Delta}}(\C)$. 
It turns out that the corresponding symplectic potential coincides
with minus the function  $\psiabs ^{\vee}_{\|\cdot\|}$.
It would be most interesting to explore further this connection.
\end{rem}

\begin{prop}\label{prop:49}
Let $\|\cdot\|$ be a semipositive metric on $L^{\an}$. Then
$\|\cdot\|_{\SS}$ is a semipositive toric metric. In particular,
$\psiabs_{\|\cdot\|}$ is concave.
\end{prop}

\begin{proof}
Let $(\|\cdot\|_{l})_{l\ge1}$ be a
  sequence of smooth (respectively algebraic) semipositive metrics on $L^{\an}$
  that converges to $\|\cdot\|$. By Proposition \ref{prop:60}, the sequence of toric metrics
  $(\|\cdot\|_{l,\SSinv})_{l\ge1}$ converges to $\|\cdot\|_{\SS}$. By
  Proposition \ref{prop:30} in the Archimedean case and Corollary
  \ref{cor:23} in the non-Archimedean case, the metrics
  $\|\cdot\|_{l,\SSinv}$ are smooth (respectively algebraic)
  semipositive. Hence, $\|\cdot\|_{\SS}$ is semipositive.
The last statement follows from Theorem \ref{thm:13}\eqref{item:37}.
\end{proof}

\begin{cor} \label{cor:14}
The line bundle $L^{\an}$ admits a semipositive metric if and only if
$L$ is generated by global sections.   
\end{cor}

\begin{proof}
  Suppose that $L^{\an}$ admits a semipositive metric $\|\cdot\|$. By
  Proposition \ref{prop:49}, $\psiabs_{\|\cdot\|}$ is concave. Hence,
  $\Psi=\rec(\psiabs_{\|\cdot\|})$ is concave too which, by
  Proposition~\ref{prop:99}\eqref{item:56}, is equivalent to the fact
  that $L$ is generated by global sections.
  
  Reciprocally, if $L$ is generated by its global sections, then the
  function $\Psi$ is concave and therefore defines a semipositive
  toric metric on $L^\an$, by Theorem \ref{thm:13}\eqref{item:37}.
\end{proof}

Here is what we can say about toric DSP metrics.  

\begin{thm} \label{thm:24} Let $\Psi $ be a virtual support function
  on~$\Sigma $. Then the map $\|\cdot\|\mapsto \psiabs _{\|\cdot\|}$
  is a bijection between:
  \begin{itemize}
  \item[$\bullet$]  the space of toric metrics on $L_{\Psi
  }^{\an}$ such that there is a refinement $\Sigma'$ of $\Sigma$
  with associated birational toric morphism
  $\varphi\colon X_{\Sigma'}\to X_{\Sigma}$ so that
  $\varphi^{*}\|\cdot\|$ is a DSP toric metric on
  $\varphi^{*}L_{\Psi}^{\an}$;
\item[$\bullet$] the space of functions $\psiabs \in
  \ov{\mathscr{D}}(N_{\R})_{\Z}$ (Definition \ref{def:23}) with
  $\rec(\psiabs )=\Psi $.
  \end{itemize}
\end{thm}

\begin{proof} Let $\|\cdot\|$ be a toric metric on $L_{\Psi}^{\an}$
  and $\Sigma'$ a refinementof $\Sigma$ with associated birational
  toric morphism $\varphi\colon X_{\Sigma'}\to X_{\Sigma}$ so that
  $\varphi^{*}\|\cdot\|$ is a DSP toric metric on
  $\varphi^{*}L_{\Psi}^{\an}$. By definition, there exists
  semipositive metrized line bundles $(L_{1},\|\cdot\|_{1})$ and
  $(L_{2},\|\cdot\|_{2})$ on $X_{\Sigma'}$ such that
  \begin{displaymath}
  (\varphi^{*}L_{\Psi},\varphi^{*}\|\cdot\|)= (L_{1},\|\cdot\|_{1})\otimes (
  L_{2},\|\cdot\|_{2})^{\otimes -1}.     
  \end{displaymath}
  By propositions \ref{prop:60} and \ref{prop:49}, $
  (\varphi^{*}L_{\Psi},\varphi^{*}\|\cdot\|)= (L_{1},\|\cdot\|_{1,\SS})\otimes (
  L_{2},\|\cdot\|_{2,\SS})^{\otimes -1}$ and $\|\cdot\|_{i,\SS}$, $i=1,2$, is
  a semipositive  toric metric.  By Theorem
  \ref{thm:13}\eqref{item:37},
  $\psiabs_{\|\cdot\|_{i,\SS}}\in\ov{\mathscr{P}}(N_{\R},\Delta _{i})$,
  where $\Delta_{i}$ denotes the lattice polytope corresponding to
  $L_{i}$. In particular, $\psiabs_{\|\cdot\|_{i,\SS}}\in
  \ov{\mathscr{P}}(N_{\R})_{\Z}$ (Definition \ref{def:52}), hence using
  Proposition \ref{prop:58},
  \begin{displaymath}
    \psiabs_{\|\cdot\|}=       \psiabs_{\varphi^{*}\|\cdot\|}=     \psiabs_{\|\cdot\|_{1,\SS}} -  \psiabs_{\|\cdot\|_{2,\SS}}\in
\ov{\mathscr{D}}(N_{\R})_{\Z}
  \end{displaymath}
and $\rec(\psiabs_{\|\cdot\|})=\rec(\psiabs_{\|\cdot\|_{1,\SS}}) -
\rec(\psiabs_{\|\cdot\|_{2,\SS}})=\Psi$.

Conversely, let $\psiabs \in \ov{\mathscr{D}}(N_{\R})_{\Z}$ such that
$\rec(\psiabs )=\Psi $. Let $ \psiabs= \psiabs_{1}-\psiabs_{2}$ with
$\psiabs_{i}\in \ov{\mathscr{P}}(N_{\R})_{\Z}$. By Definition
\ref{def:52} and Corollary~\ref{cor:2}, there are lattice polytopes
$\Delta _{i}$, $i=1,2$, such that $\psiabs
_{i}\in\ov{\mathscr{P}}(N_{\R},\Delta _{i})$. We can assume without
loss of generality that these lattice polytopes have dimension
$n$. Let $(X_{\Sigma_{i}},D_{i})$ be the polarized toric variety
determined by $\Delta _{i}$ by the correspondence in Theorem
\ref{thm:21}. Let $\Sigma'$ be a fan on $N_{\R}$ simultaneously
refining $\Sigma$ and $\Sigma_{i}$, $i=1,2$, and let $\varphi\colon
X_{\Sigma'}\to X_{\Sigma}$ and $\varphi_{i}\colon X_{\Sigma'}\to
X_{\Sigma_{i}}$ be the associated birational toric maps. Set
$L_{i}=\varphi_{i}^{*}O(D_{i})$, $i=1,2$.

Since $\Psi=\rec(\psiabs_{1})-\rec(\psiabs_{2}) $, we have
$\varphi^{*}L_{\Psi }=L_{1}\otimes L_{2}^{\otimes -1}$.  By Theorem
\ref{thm:13}\eqref{item:37}, the function $\psiabs _{i}$
determines a semipositive toric metric on $L_{i}^{\an}$
that we denote by $\|\cdot\|_{i}$. Then
\begin{displaymath}
  (\varphi^{*}L_{\Psi},\varphi^{*}\|\cdot\|):= (L_{1},\|\cdot\|_{1})\otimes (
  L_{2},\|\cdot\|_{2})^{\otimes -1}
\end{displaymath}
is a toric DSP metrized line bundle on $X_{\Sigma'}$ and $\psiabs _{\|\cdot\|}=\psiabs $.
\end{proof}

\begin{exmpl}\label{exm:39}
  Let $\Psi $ be a virtual support function on $\Sigma $. By Corollary
  \ref{cor:6}, $\Psi \in \mathscr{D}(N_{\R})_{\Z} \subset
  \ov{\mathscr{D}}(N_{\R})_{\Z}$ and, moreover $\rec(\Psi )=\Psi
  $. Therefore, by Theorem \ref{thm:24}, there is a birational toric
  morphism $\varphi\colon X_{\Sigma'}\to X_{\Sigma}$ so that the
  inverse image $\varphi^{*}\|\cdot\|_{\Psi}$ of the canonical metric
  on $L_{\Psi}^{\an}$ is a DSP toric metric on
  $\varphi^{*}L_{\Psi}^{\an}$. Note that $\varphi^{*}\|\cdot\|_{\Psi}$
  coincides with the canonical metric on $\varphi^{*}L_{\Psi}^{\an}$.

  By Theorem \ref{thm:13}, if the function $\Psi $ is concave or,
  equivalently by Proposition~\ref{prop:99}\eqref{item:56}, if the
  line bundle $\cO(D_{\Psi })$ is generated by global sections, then
  $\|\cdot\|_{\Psi}$ is semipositive.
\end{exmpl}

\begin{rem}
  \label{rem:37}
  The correspondence in Theorem \ref{thm:24} gives also a bijection
  between the space of DSP toric metrics on $L_{\Psi}^{\an}$ and the
  space of functions $\psiabs \in \ov{\mathscr{D}}(N_{\R})_{\Z}$ such
  that $\rec(\psiabs )=\Psi $ that can be written as $ \psiabs=
  \psiabs_{1}-\psiabs_{2}$ with $\psiabs
  _{i}\in\ov{\mathscr{P}}(N_{\R},\Delta _{i})$ for a lattice polytope
  $\Delta_{i}$ whose support function is compatible with the fan
  $\Sigma$ in the sense of Definition \ref{def:54}. This follows
  easily from the proofs of theorems \ref{thm:21} and \ref{thm:24}.

  Whether these spaces coincide with those in Theorem \ref{thm:24} is
  yet to be decided.
\end{rem}

We now study the compatibility of the restriction of
semipositive{} toric
metrics to toric orbits and its inverse image by equivariant maps with
direct and inverse image of concave functions. This is an extension of
propositions \ref{prop:52} and \ref{prop:48}. We start with
the case of orbits, and we state a variant of  Proposition
  \ref{prop:57} for semipositive{} metrics. 

\begin{prop}\label{prop:62} 
  Let $\Psi$ be a support function on $\Sigma$, set
  $L=L_{\Psi}$ and $s=s_{\Psi}$. Let $\|\cdot\|$ be a semipositive{} toric 
  metric on $L^{\an}$, denote $\ov L=(L,\|\cdot\|)$ and $\psiabs 
  =\psiabs _{\ov L,s}$ the associated concave function on $N_{\R}$. 
  Let $\sigma \in\Sigma $ and $m_{\sigma }\in M$ such that $\Psi |_{\sigma
  }=m_{\sigma }|_{\sigma }$. Let $\pi _{\sigma }\colon N_{\R}\to
  N(\sigma )_{\R}$ be the projection, $\pi ^{\vee}_{\sigma }\colon
  M(\sigma )_{\R}\to M_{\R}$ the dual inclusion and $\iota \colon
  V(\sigma )\to X_\Sigma$ the closed immersion. Set 
  $s_{\sigma}=\chi^{m_{\sigma}}s$. Then
  \begin{equation}
    \label{eq:91}
    \psiabs_{\iota ^{\ast}\ov L,\iota ^{\ast}s_{\sigma}}=(\pi _{\sigma })_{\ast}(\psiabs -m_{\sigma }).
  \end{equation}
  Dually, we have
  \begin{equation}
    \label{eq:92}
    \psiabs_{\iota ^{\ast}\ov L,\iota ^{\ast}s_{\sigma}}^{\vee}=(\pi ^{\vee} _{\sigma }+m_{\sigma
    })^{\ast}\psiabs ^{\vee}.  
  \end{equation}
  In other words, the Legendre-Fenchel dual of $\psiabs_{\iota
    ^{\ast}\ov L,\iota ^{\ast}s_{\sigma}}$ is the translate by $-m_{\sigma
  }$ of the restriction of $\psiabs ^{\vee}$ to the face $F_{\sigma
  }$.
\end{prop}
\begin{proof}
  As in the proof of Proposition \ref{prop:52}, it is enough to prove
  the equation \eqref{eq:91}. By replacing $\psiabs $ by $\psiabs -m_{\sigma
  }$, we assume without loss of generality that $m_{\sigma
  }=0$. By the continuity of the metric, the function $\psiabs$ can be
  extended to a continuous function $\ov{\psiabs }_{\sigma }$ on
  $N_{\sigma }$, where $N_{\sigma }$ is the compactification of
  $N_{\R}$ in the directions of the cone $\sigma $ (see
  \eqref{eq:48}). In this way, the function $\psiabs_{\iota ^{\ast}\ov L,\iota ^{\ast}s_{\sigma}}$ is the restriction of $\ov{\psiabs
  }_{\sigma }$ to $N(\sigma)_\R$. Fix $u_{0}\in N(\sigma )_{\R}$ and 
  write $s=\ov{\psiabs }_{\sigma }(u_{0})$. By continuity, for any 
  $\varepsilon>0$ there exists a neighbourhood $W$ of $u_0$ in $N_\sigma$ 
  such that for all $u\in W\cap N_\R$ we have $|f(u)-s|<\varepsilon$. 
  By the definition of the topology of $N_{\sigma }$ (see \eqref{eq:20}), 
  such a set $W\cap N_\R$ is of the form $U+p+\sigma$ with $U$ a 
  neighbourhood of a point $u\in N_\R$ such that $\pi_\sigma(u)=u_0$ and 
  $p\in\R\sigma$. Therefore, we conclude that for any $\varepsilon>0$ there 
  exists $u\in N_\R$ satisfying $\pi_\sigma(u)=u_0$ and $p\in \R\sigma $ 
  such that, for all $r\in p+\sigma$,
  \begin{equation*}
    s-\varepsilon < \psiabs (u+r) < s+\varepsilon .
  \end{equation*}
  Now, by definition
  \begin{displaymath}
    (\pi _{\sigma })_{\ast}(\psiabs)(u_{0})=\sup_{\stackrel{u\in N_\R}{\pi_\sigma(u)=u_0} } \psiabs(u).
  \end{displaymath}
  Thus it is clear that $(\pi _{\sigma })_{\ast}(\psiabs)(u_{0})\ge s$, 
  suppose $(\pi _{\sigma })_{\ast}(\psiabs)(u_{0})> s$. 
  Let $v\in N_\R$ satisfying $\pi_\sigma(v)=u_0$ be such
  that $\psiabs(v)> s$ and set $\varepsilon =\psiabs(v)-s>0$. By the
  previous discussion, there exists $u\in N_\R$ satisfying 
  $\pi_\sigma(u)=u_0$ and $p\in \R\sigma $ such that, for all $r\in p+\sigma$,
  \begin{equation}
    \label{eq:93}
    s-\varepsilon < \psiabs (u+r) < s+\varepsilon = \psiabs (v) .
  \end{equation}
  Write $q=v-u\in \R\sigma$, since $\sigma $ is a cone of maximal dimension 
  in $\R\sigma $, there
  exists a point $r\in (q+\sigma )\cap (p+\sigma )$. By the right
  inequality of \eqref{eq:93} $\psiabs (u+r)< \psiabs (u+q)$ and the 
  function $g(\lambda):=\psiabs (u+r+\lambda(r-q))$ of the variable $\lambda\in\R$, satisfies $g(0)<g(-1)$. 
  Furthermore, since the function $\psiabs$ is concave, so is $g$ which 
  therefore stays for $\lambda>0$ below a line of negative slope $g(0)-g(-1)$. 
  This implies $\lim_{\lambda \to +\infty}g(\lambda)=-\infty$, that is
  \begin{equation}\label{eq:94}
    \lim_{\lambda \to +\infty}\psiabs (u+r+\lambda (r-q))=-\infty.
  \end{equation}
  Since, by construction $r+\R_{\ge 0}(r-q)$ is contained in $p+\sigma$, 
  the equation \eqref{eq:94} contradicts the left inequality of
  \eqref{eq:93}. Hence, for $u_0\in N(\sigma)_\R$,
  $$(\pi_\sigma)_*(f)(u_0) = \sup_{\stackrel{u\in N_\R}{\pi_\sigma(u)=u_0} } 
  \psiabs(u) = s = \ov{\psiabs }_{\sigma }(u_{0}) = 
  \psiabs_{\iota ^{\ast}\ov L,\iota ^{\ast}s_{\sigma}}(u_0),$$
  which proves equation \eqref{eq:91}.
\end{proof}

We now interpret the inverse image of a semipositive{} toric metric by 
an equivariant map whose image intersects the principal open subset in 
terms of direct and inverse images of concave functions. 

\begin{prop}\label{prop:64}
  Let $N_{1}$ and $N_{2}$ be lattices and $\Sigma _{i}$ be a complete fan
  in $N_{i,\R}$, $i=1,2$. Let $H\colon N_{1}\to N_{2}$ be a linear map such
  that, for each $\sigma_{1} \in \Sigma _{1}$, there exists
  $\sigma _{2}\in \Sigma _{2}$ with $H(\sigma _{1})\subset \sigma
  _{2}$. Let $p\in X_{\Sigma_{2},0 }(K)$ and write $A\colon N_{1,\R}\to
  N_{2,\R}$ for the affine map $A=H+\val(p)$. Let $\Psi _{2}$ be a
  support function on $\Sigma _{2}$ and $\|\cdot\|$ a
  semipositive{} toric metric on $L_{\Psi_{2}}^{\an}$. Then
   \begin{displaymath}
     \psiabs _{\varphi_{p,H}^{\ast}\|\cdot\|}=A^{*}\psiabs _{\|\cdot\|}.
   \end{displaymath}
Moreover, the Legendre-Fenchel dual of this function is given by 
\begin{displaymath}
     \psiabs _{\varphi_{p,H}^{\ast}\|\cdot\|}^{\vee}=
     (H^{\vee})_{*}\big(\psiabs^{\vee}_{\|\cdot\|}-\val(p)\big).
\end{displaymath}
\end{prop}
\begin{proof}
The first statement is a direct consequence of Proposition
\ref{prop:58} while the second one follows from Proposition \ref{prop:101}\eqref{item:121}.
\end{proof}

Finally, we characterize the measures associated to semipositive{}
metrics. 

\begin{thm}\label{thm:16} Let  $\Psi $ be a 
  support function on $\Sigma $ and set $L=L_{\Psi}$.  Let $\|\cdot\|$
  be a semipositive{} metric on $L^{\an}$ and $\psiabs =\psiabs
  _{\|\cdot\|}$ the corresponding concave function.  Then
  \begin{equation}\label{eq:60}
    \val_{\ast}(\chern_{1}(\overline L)^{\wedge n}\land
  \delta _{X_{\Sigma }}) =n!\, \ov{\mathcal{M}}_{M}(\psiabs).    
  \end{equation}
  Moreover, the measure $\chern_{1}(\overline L)^{\wedge n}\land
  \delta _{X_{\Sigma }}$ is characterized, in the Archimedean
  case, by the equation \eqref{eq:60} and the fact of being
  toric, while in the non-Archimedean case it is given
  by 
  \begin{displaymath}
    \chern_{1}(\overline L)^{\wedge n} \land
    \delta _{X_{\Sigma }} =(\theta _{\Sigma })_{\ast}(\ee)_{\ast}n!
    \ov{\mathcal{M}}_{M}(\psiabs ). 
  \end{displaymath}
\end{thm}
\begin{proof}
  For short, denote $\mu =\val_{\ast}(\chern_{1}(\overline L)^{\wedge
    n}\land \delta _{X_{\Sigma }})$.  Let $\|\cdot\|_{l}$ be a
  sequence of semipositive smooth metrics (respectively metrics with a
  semipositive model) converging to $\|\cdot\|$. By
  Proposition~\ref{prop:2}, the measures
  $\chern_{1}(L,\|\cdot\|_{l})^{\wedge n}\land \delta _{X_{\Sigma }}$
  converge to $\chern_{1}(\ov L)^{\wedge n}\land \delta _{X_{\Sigma
    }}$. Therefore, the measures
  $\val_{\ast}(\chern_{1}(L,\|\cdot\|_{l})^{\wedge n}\land \delta
  _{X_{\Sigma }})$ converge to the measure $\mu $ on $N_{\Sigma }$.
  Theorem~\ref{prop:4}\eqref{item:12} implies that the measure of
  $X_{\Sigma }^{\an}\setminus X_{0}^{\an}$ with respect to
  $\chern_{1}(\ov L)^{\wedge n}\land \delta _{X_{\Sigma }}$ is
  zero. Therefore $N_{\Sigma }\setminus N_{\R}$ has $\mu $-measure
  zero.  Denote $\psiabs _{l}=\psiabs _{(\|\cdot\|_{l})_{\SSinv}}$. By
  Proposition \ref{prop:87}, the measures $\mathcal{M}_{M}(\psiabs
  _{l})$ converge to the measure $\mathcal{M}_{M}(\psiabs)$. Thus $\mu
  |_{N_{\R}}=n!  \mathcal{M}_{M}(\psiabs)$ by (\ref{eq:56}) and
  (\ref{eq:57}). Then, \eqref{eq:60} follows from this and the fact
  that the measure of $N_{\Sigma }\setminus N_{\R }$ is zero.

  The last statement of the theorem follows from Theorem \ref{thm:18}
  in the Archimedean case and Corollary \ref{cor:34} in the
  non-Archimedean case by a limit argument.
\end{proof}

\begin{cor}\label{cor:30}
  For $i=0,\dots, n-1$, let $\Psi _{i}$ be a support function on
  $\Sigma $ and set $L_{i}=L_{\Psi_{i}}$.  Let $\|\cdot\|_{i}$ be a
  semipositive{} metric on $L^{\an}_{i}$ and $\psiabs_{i} =\psiabs
  _{\|\cdot\|_{i}}$ the corresponding concave function.  Then
  \begin{displaymath}
    \val_{\ast}(\chern_{1}(\overline L_{0})\land\dots\land \chern_{1}(\overline L_{n-1})\land
  \delta _{X_{\Sigma }}) =n! \,\ov{\mathcal{M}}_{M}(\psiabs_{0},\dots,\psiabs_{n-1}).    
  \end{displaymath}
\end{cor}
\begin{proof}
  This follows from Theorem \ref{thm:16} by multilinearity.
\end{proof}

\section{Adelic toric metrics}
\label{sec:appr-integr-adel}

Now let $(\K,\mathfrak{M})$ be an adelic field
(Definition \ref{def:6}). We fix a complete fan $\Sigma$ in $N_{\R}$
and a virtual support function $\Psi$ on $\Sigma$. Let $X$ be the
associated toric variety and  
$(L,s)$ the associated toric line bundle and section. 


\begin{defn} \label{def:59} A \emph{toric metric} on~$L$
\index{toric metric}%
is a family  
  $(\|\cdot\|_{v})_{v\in \mathfrak{M}}$, where $\|\cdot\|_{v}$ is a toric
  metric on $L_{v}^{\an}$. A toric metric is \emph{adelic}
\index{toric metric!adelic}%
if
  $\psiabs_{\|\cdot\|_v} =\Psi$ for all but finitely many $v$.
\end{defn}

The following result is a direct consequence of Theorem \ref{thm:13}.

\begin{prop} \label{prop:102} With the previous notations,
  \begin{enumerate}
  \item there
    is a bijection between the set of semipositive{} adelic toric metric
    on $L$ and the set of families of continuous concave functions 
    $(\psiabs _{v})_{v\in \mathfrak{M}}$ on $N_{\R}$ such that $|\psiabs _{v}-\Psi |$ is
    bounded and $\psiabs _{v}=\Psi $ for all but finitely many $v$;
  \item there
    is a bijection between the set of semipositive{} adelic toric metric
    on $L$
    and the set of families of continuous concave functions 
    $(\psiabs ^{\vee}_{v})_{v\in \mathfrak{M}}$ on $\Delta _{\Psi }$ 
    such that $\psiabs ^{\vee}_{v}=0$ for all but finitely many $v$.
  \end{enumerate}
\end{prop}

For global fields, the notions of quasi-algebraic toric metric and
of adelic toric metric agree. 

\begin{thm} \label{thm:23} Let $\K$ be a global field. A
  toric metric on $L$ is
  quasi-algebraic (Definition \ref{def:64}) if and only if it is an adelic toric metric.
\end{thm}

\begin{proof}
  Let $(\|\cdot\|_{v})_{v\in \mathfrak{M}_{\K}}$ be a metric on $L$ and
  write $\ov L=(L, (\|\cdot\|_{v})_{v\in \mathfrak{M}_{\K}})$.  Suppose first
  that $\ov L$ is toric and quasi-algebraic. Let $S\subset
  \mathfrak{M}_{\K}$ be a finite set containing the Archimedean places,
  $\K^{\circ}_{S}$ as in Definition \ref{def:66},
  $e\ge 1$
  an integer 
  and $(\cX,\cL)$ a proper model over $\K^{\circ}_{S}$ of
  $(X, L^{\otimes e})$ so that $\|\cdot \|_{v}$ is induced by
  the localization $\cL_{v}$ for all $v\notin S$. The generic fibre of
  $(\cX,\cL)$ is isomorphic with that of the canonical model
  $(\cX_{\Sigma},\cO(D_{e\Psi}))$  
(Definitions \ref{def:43} and \ref{def:55}). Since $\K^{\circ}_{S}$ 
  is Noetherian, this isomorphism and its inverse are defined 
  over $\K^{\circ}_{S'}$ for certain finite subset $S'$ containing
  $S$. Thus, enlarging the
  finite set $S$ if necessary, we can suppose that, over $\K^{\circ}_{S}$,
  $(\cX,\cL)$ agrees with the canonical model $(\cX_{\Sigma},
  \cO(D_{e\Psi}))$.
  Hence, $\|\cdot\|_{v}=\|\cdot\|_{v,e\Psi}^{1/e}=\|\cdot \|_{v,\Psi}$ 
  for all places~$v\notin S$. In consequence, it is an adelic toric metric.

  Conversely, suppose that $\ov L$ is a toric adelic
  metrized line bundle. Let $S$ be the union of the set of Archimedean
  places and $\{v\in \mathfrak{M}_{\K}| \psiabs_{v}\ne \Psi\}$. By
  definition, this is a finite set. Let $(\cX_{\Sigma}, \cO(D_{\Psi}))$
  be the canonical model over $\K^{\circ}_{S}$ of $(X_{\Sigma},
  L)$. Then $\|\cdot\|_{v}$ is the metric induced by this model, for
  all $v \notin S$. Hence $\ov L$ is quasi-algebraic.
\end{proof}


\chapter{Height of toric varieties}
\label{sec:heigh-toric-vari}

In this chapter, we will state and prove a formula to compute the
height of a toric variety with respect to a toric line bundle.

\section{Local heights of toric varieties}
\label{sec:local-heights-toric}

Let $K$ be either $\R$, $\C$ or a complete field with respect to an
absolute value associated to a nontrivial discrete valuation. Let
$N\simeq \Z^{n}$ be a lattice and 
$M=N^{\vee}$ the dual lattice. We will use the notations of \S
\ref{sec:toric-varieties} and we recall the definition of $\lambda
_{K}$ in \eqref{eq:95}.  

Let $\Sigma $ be a complete fan on $N_{\R }$ and $X_{\Sigma }$ the
corresponding proper toric variety. In Definition \ref{def:3} we
recalled the definition of local heights. These local heights depend,
not only on cycles and metrized line bundles, but also on the choice
of sections of the involved line bundles.  For toric line bundles,
Proposition-Definition~\ref{def:57}, provides us with a distinguished
choice of a toric metric, the canonical
metric. This metric is DSP and, if the line bundle is generated
by global sections, it is semipositive{} (see Example
\ref{exm:39}). By comparing any DSP 
metric to the canonical
metric, we can define a local height for toric line bundles that is
independent from the choice of sections.
 
\begin{defn}\label{def:35}
  Let $\ov L_{i}=(L_{i},\|\cdot\|_{i})$, $i=0,\dots,d$, be a family of
  toric line bundles, with DSP
  toric metrics. Denote by $\ov
  L_{i}^{\can}$ the same line bundles equipped with the canonical
  metric. Let $Y$ be a $d$-dimensional irreducible subvariety of $X_{\Sigma
  }$ and $\varphi\colon Y'\to Y$ a birational morphism with $Y'$
  projective. Then the \emph{toric local height} of $Y$  with respect
  to $\ov 
  L_{0},\dots,\ov L_{d}$
\index{height of cycles!toric local}%
\nomenclature[ahL04]{$\htor_{\ov L_{0},\dots,\ov L_{d}}(Y)$}{toric local height}%
is 
  \begin{displaymath}
    \htor _{\ov L_{0},\dots,\ov L_{d}}(Y)= 
    \h_{\varphi^{\ast}\ov L_{0},\dots,\varphi^{\ast}\ov
      L_{d}}(Y';s_{0},\dots,s_{d})- 
    \h_{\varphi^{\ast} \ov L_{0}^{\can},\dots,\varphi^{\ast} \ov
      L_{d}^{\can}}(Y';s_{0},\dots,s_{d}), 
  \end{displaymath}
  where 
  $s_{0},\dots,s_{d}$ are sections meeting 
  $Y'$ properly. We extend the definition to $d$-dimensional cycles by
  linearity. When $\ov L_{0}=\dots =\ov L_{d}=\ov L$ we
  will denote
  \begin{displaymath}
    \htor_{\ov L}(Y)=\htor_{\ov L_{0},\dots,\ov L_{d}}(Y).
  \end{displaymath}
\end{defn}

\begin{rem} \label{rem:18} 
  Even if the notion of toric local height in the
  above definition differs from that of local height of Definition
  \ref{def:3}, we will be able to use it to compute global heights
  because, for toric subvarieties and closures of orbits, the sum
  over all places of the local canonical heights is
  zero (see Proposition \ref{prop:88}). This is the case, in particular, for
  the height of the total space $X_{\Sigma }$.

  By Theorem \ref{thm:1} (\ref{item:8}, \ref{item:9}), the toric local
  height $\htor _{\ov L_{0},\dots,\ov L_{d}}(Y)$ does not depend on
  the choice of $Y'$ nor on the choice of sections. However, it does
  depend on the toric structure of the line bundles (see Definition
  \ref{def:71}), because the canonical metric depends on the toric
  structure.
\end{rem}

\begin{prop}\label{prop:85}
The toric local height is symmetric and multilinear
    with respect to tensor product of metrized toric line bundles. 
In particular, let $\Sigma$ be a complete fan, $\ov L_{i}$ a family of
  $d+1$ toric line bundles with DSP
  toric metrics  and $Y$ an algebraic cycle of $X_{\Sigma}$ of
dimension $d$. Then 
    \begin{equation}
      \label{eq:118}
      \htor_{\ov
      L_{0},\dots,\ov L_{d}}(Y)= 
      \frac{1}{(d+1)!}
\sum_{j=0}^d (-1)^{d - j} \sum_{0 \le i_0
  < \cdots < i_j \le d} \htor_{\ov L_{i_{0}}\otimes \cdots \otimes \ov
    L_{i_{j}}} (Y).
    \end{equation}  
\end{prop}

\begin{proof} It is enough to treat the case when $Y$ is a
  $d$-dimensional irreducible subvariety. Let $\varphi\colon Y'\to Y$
  be a birational map with $Y'$ projective. By abuse of language we
  will denote $\varphi^{\ast}\ov L_{i}$ by $\ov L_{i}$. By the Moving
  Lemma, we 
  can choose sections $s_{i}$ 
  of $L_{i}$,   $i=0,\dots, 
  d$, such that $s_{0},\dots,s_{d}$ meet $Y'$ properly.

 The symmetry of the toric local height follows readily from the
analogous property for the local height, see Theorem \ref{thm:1}\eqref{item:7}.
For the multilinearity, let $\ov L_{d}'$ be a further metrized line bundle.
Again by the moving lemma,
  there is a  section 
  $s_{d}'$ of $L_{d}'$ such that $s_{0},\dots, s_{d-1},
  s_{d}'$ meets $Y'$  properly too. 
  By Theorem \ref{thm:1}\eqref{item:7}, 
  \begin{multline*}
      \h_{\ov
      L_{0},\dots,\ov L_{d-1},\ov L_{d}\otimes \ov L_{d}'}(Y'; s_{0},\dots, s_{d-1},
 s_{d}\otimes s_{d}')=          \h_{\ov
      L_{0},\dots,\ov L_{d}}(Y'; s_{0},\dots, 
 s_{d})\\+          \h_{\ov
      L_{0},\dots,\ov L_{d-1}, \ov L_{d}'}(Y'; s_{0},\dots, s_{d-1},
 s_{d}')   
  \end{multline*} 
and a similar formula holds for the canonical metric.
By the definition of the toric local height,  $$ \htor_{\ov
      L_{0},\dots,\ov L_{d-1},\ov L_{d}\otimes \ov L_{d}'}(Y)= \htor_{\ov
      L_{0},\dots,\ov L_{d}}(Y)+ \htor_{\ov
      L_{0},\dots,\ov L_{d-1}, \ov L_{d}'}(Y).$$
The inclusion-exclusion formula follows readily from the symmetry and
the multilinearity of the local toric height. 
\end{proof}

\begin{defn}
  \label{def:34} Let $(\ov L, s)$ be a metrized toric line bundle with a
  toric section. Then
the \emph{roof function}\index{roof function} associated to $(\ov L,s)$ is the
concave function $\vartheta_{\ov L,s}\colon \Delta_{\Psi }\to \R$
defined as  
\begin{displaymath}
  \vartheta_{\ov L,s}=\psiabs_{\ov L,s}^{\vee}= \lambda_{K}\phiK_{\ov L,s}^{\vee}.
\end{displaymath}
\nomenclature[g08]{$\vartheta_{\ov L,s}$}{roof function}%
The concave function  $\phiK_{\ov L,s}^{\vee}$ will be called the
\emph{rational roof function}.
\index{roof function!rational}%
When the toric section
$s$ is clear from the context, we will denote $\psiabs_{\ov L,s}$ and
$\vartheta_{\ov L,s}$ by $\psiabs_{\|\cdot\|}$ and $\vartheta
_{\|\cdot\|}$ respectively. 
\end{defn}

In the non-Archimedean case, recall that the function $\phiK _{\|\cdot\|}$ is not
invariant under field extensions (see Proposition \ref{prop:115}) but
it has the advantage that, if the
metric $\|\cdot\|$ 
is algebraic, then it is rational with respect to the lattice $N$.  
By contrast, the function $\psiabs_{\|\cdot\|}$ is invariant under field
extensions. It is not rational, but it takes values in $\lambda
_{K}\Q$ on $\lambda _{K}N_{\Q}$. This is the function that appears in
\cite{BPS09}. In particular, the roof function is also invariant under
field extension.

\begin{lem}\label{lemm:5}
  Let ${K'}/K$ be a finite extension of valued fields of the type considered
  at the beginning of this section.
  Let $\ov L,s$ be as before and $\ov L_{{K'}},s_{{K'}}$ the toric metrized line
  bundle  with toric section obtained by base change. Then
  \begin{displaymath}
    \vartheta _{\ov L_{{K'}},s_{{K'}}}=\vartheta _{\ov L,s}.
  \end{displaymath}
\end{lem}

In case $\phiK_{\|\cdot\|}$ is a piecewise affine concave function,
$\vartheta_{\|\cdot\|}$ and $\phiK_{\|\cdot\|}^{\vee}$ parameterize the
upper envelope of some extended polytope, as explained in
Lemma~\ref{lemm:17}, hence the terminology ``roof function''.  In case
$K$ is non-Archimedean and $\|\cdot\|$ is algebraic, the function
$\phiK_{\|\cdot\|}^{\vee}$ is a rational piecewise affine concave function.

\begin{thm}\label{thm:19} Let $\Sigma$ be a complete fan on $N_{\R}$.
  Let $\ov L=(L,\|\cdot\|)$ be a toric line bundle on $X_{\Sigma }$
  equipped  with a semipositive toric metric.  Choose any toric section 
  $s$ of $L$, let $\Psi $ be the associated support function on $\Sigma $ 
  and put $\Delta _{\Psi }=\Stab (\Psi )$ for the associated polytope. Then,
  the toric local height of $X_{\Sigma }$ with respect to $\ov L$ is
  given by
  \begin{equation}\label{eq:108}
    \htor_{\ov L}(X_{\Sigma })=(n+1)!\int _{\Delta _{\Psi
      }}\vartheta_{\ov L,s}\dd \Vol_{M}=
    (n+1)!\, \lambda _{K} \int _{\Delta _{\Psi
      }}\phiK _{\ov L,s}^{\vee}\dd \Vol_{M},
  \end{equation}
  where $\dd \Vol_{M}$ is the unique Haar measure of $M_{\R}$ such
  that the covolume of $M$ is one and $\phiK_{\ov L,s}^{\vee}$ is the
  Legendre-Fenchel dual to the function $\phiK_{\ov L,s}$ (Definition \ref{def:68}).
\end{thm}
\begin{proof}
  We note that, by Theorem \ref{thm:13}\eqref{item:38}, the function
  $\psiabs_{\ov L,s}$ is concave because the metric $\|\cdot\|$ on 
  $L^{\an}$ is semipositive. For short, we set $\Delta = \Delta _{\Psi }$, 
  $\psiabs =\psiabs_{\|\cdot\|}$ and $\vartheta =\psiabs^{\vee}$.

  We first reduce to the case of an ample line bundle.
  Let $\Sigma _{\Delta
  }$ be the fan associated to $\Delta $ as in Remark
  \ref{rem:11}. There is a toric morphism $\varphi\colon X_{\Sigma }\to
  X_{\Sigma _{\Delta }}$. By Theorem \ref{thm:13}, the
  function $\psiabs ^{\vee}$ defines a semipositive{} metric
  $\|\cdot\|_{0}$ on the line bundle 
  $\mathcal{O}(D_{\Psi _{\Delta }})^{\an}$ over $X_{\Sigma _{\Delta
    }}$. We denote $\ov L_{0}=(\mathcal{O}(D_{\Psi 
    _{\Delta }}),\|\cdot\|_{0})$. Then there is an isometry
  $\varphi^{\ast}(\ov L_{0})=\ov L$. By Corollary \ref{cor:20} there is
  an isometry $\varphi^{\ast}(\ov L_{0}^{\can})=\ov L^{\can}$.

  If the dimension of $\Delta $ is less than $n$, then the right-hand
  side of equation \eqref{eq:108} is zero. Moreover, 
  $n=\dim(X_{\Sigma })>\dim(X_{\Sigma _{\Delta }})$ and the metrized
  line bundles $\ov L$ and $\ov L ^{\can}$ come from a variety of
  smaller dimension. Therefore, by Theorem \ref{thm:1}\eqref{item:8}, 
  the left-hand side 
  of the equation \eqref{eq:108} is also zero, because $\varphi_{*}
  X_{\Sigma}=0$. If $\Delta $ has dimension $n$
  then $\varphi$ is a birational morphism, so, by Theorem
  \ref{thm:1}\eqref{item:8}, 
  \begin{displaymath}
    \htor_{\ov L}(X_{\Sigma })=\htor_{\ov L_{0}}(X_{\Sigma _{\Delta }}).
  \end{displaymath}
  Therefore it is enough to prove the theorem for $X_{\Sigma _{\Delta
    }}$. By construction, the fan $\Sigma _{\Delta }$ is
  regular; hence the variety $X_{\Sigma _{\Delta }}$ is
  projective and $L_{0}$ is ample. Thus we are reduced to prove the
  theorem in the case when $\Sigma $ is
  regular and $L$ is ample. 
  
  Now the proof is done by
  induction on $n$, the dimension of $X_{\Sigma 
  }$. 
  If $n=0$ then $X_{\Sigma }=\P^{0}$, $\Psi =0$, $\Delta =\{0\}$ and 
  $L=\mathcal{O}(D_{0})=\mathcal{O}_{\P^{0}}$.
  By the equation \eqref{eq:98},  $\log \|s\|=\psiabs
  (0)$ and $\log \|s\|_{\can}=\Psi 
  (0)=0$. The Legendre-Fenchel dual of $\psiabs$ satisfies
  $\vartheta (0)=- 
  \psiabs (0)$. By the equation \eqref{eq:96}, $\h_{\ov L}(X_{\Sigma
  };s)=-\psiabs (0)$ and $\h_{\ov L^{\can}}(X_{\Sigma
  };s)=0$. Therefore
  \begin{displaymath}
    \htor_{\ov L}(X_{\Sigma
    })=-\psiabs(0)=\vartheta (0)=1!\, \int _{\Delta 
      }\vartheta \dd \Vol_{M}.
  \end{displaymath}
Let $n\ge 1$ and let $s_{0},\dots, s_{n-1}$ be rational sections of
$\cO(D_{\Psi})$ such that $s_{0},\dots, s_{n-1},s$ intersect
$X_{\Sigma}$ properly. 
By the construction of local heights (Definition \ref{def:3}),
\begin{align} \label{eq:100}
\h_{\ov L}(X_{\Sigma}; s_{0},\dots,  s_{n-1},s)
=\h_{\ov L}(\div(s); & s_{0},\dots,
  s_{n-1})\\
&-  \int_{X_{\Sigma}^{\an}}\log\| s\| \chern_{1}(\ov L)^{\wedge n}\wedge
\delta_{X_{\Sigma}} \nonumber
\end{align}
and a similar formula holds for the canonical metric. 

For each facet $F$ of $\Delta $, let $v_{F}\in N$ be as in Notation
\ref{def:79}. Since  $L$ is ample, Proposition \ref{prop:65} implies 
\begin{equation}\label{eq:104}
\h_{\ov L}(\div(s);  s_{0},\dots,
  s_{n-1})= \sum_{F}-\langle F, v_{F}\rangle 
\h_{\ov L}(V(\tau_{F}); s_{0},\dots,
  s_{n-1}),
\end{equation}
where the sum is over the facets $F$ of $\Delta$.
Observe that the local height of $V(\tau_{F})$ with respect to the
metrized line bundle $\ov L$ coincides with the local height
associated to the restriction of $\ov L$ to this subvariety. Moreover
by Corollary \ref{cor:19}, 
the restriction of the canonical metric of $L^{\an}$ to this subvariety agrees
with the canonical metric of $L^{\an}|_{V(\tau_{F})}$.
Hence, by 
substracting from the equation \eqref{eq:104} the analogous formula for the
canonical metric, we obtain
\begin{align}
  \label{eq:105}
\sum_{F}-\langle F, v_{F}\rangle 
\htor_{\ov L|_{V(\tau_{F})}}(V(\tau_{F})) = \h_{\ov L}(\div(s);  &s_{0},\dots,
  s_{n-1})\\
  &\nonumber -\h_{\ov L^{\can}}(\div(s);  s_{0},\dots,
  s_{n-1}).
\end{align}

Theorem \ref{prop:4}\eqref{item:12} implies that the measure of
$X_{\Sigma}^{\an}\setminus X_{\Sigma,0}^{\an}$ with respect to $\chern_{1}(\ov L)^{\wedge n}\wedge
\delta_{X_{\Sigma}}$ is zero. Hence, 
\begin{displaymath}
 \int_{X_{\Sigma}^{\an}}\log\| s\| \chern_{1}(\ov L)^{\wedge n}\wedge
\delta_{X_{\Sigma}}=  \int_{X_{\Sigma,0}^{\an}}\log\| s\| \chern_{1}(\ov L)^{\wedge n}\wedge
\delta_{X_{\Sigma}}.  
\end{displaymath}
By the equation \eqref{eq:98}, $\log\| s\|
=\val^{*}(\psiabs)$, where $\val$ is the valuation map
introduced in the
diagram \eqref{eq:89}.  Moreover
\begin{align*}
\int_{X_{\Sigma,0}^{\an}}\val^{*}(\psiabs) \chern_{1}(\ov L)^{\wedge n}\wedge
\delta_{X_{\Sigma}}
=& \int_{N_{\R}} \psiabs 
\, \val_{*}(\chern_{1}(\ov L)^{\wedge n}\wedge
\delta_{X_{\Sigma}}  )  
\end{align*}
and by Theorem \ref{thm:16}, $\val_{*}(\chern_{1}(\ov L)^{\wedge n}\wedge
\delta_{X_{\Sigma}}) =n! \cM_{M}(\psiabs)$. Hence,
\begin{equation}\label{eq:99}
 \int_{X_{\Sigma}^{\an}}\log\| s\| \chern_{1}(\ov L)^{\wedge n}\wedge
\delta_{X_{\Sigma}}= n!  \, \int_{N_{\R}}\psiabs 
\dd\cM_{M}(\psiabs ) .  
\end{equation}
By Example \ref{exm:30},  $\cM_{M}(\Psi) =
\Vol_{M}(\Delta) \delta_{0}$. Therefore,  in the case of the canonical
metric, the equation \eqref{eq:99} reads as
\begin{equation}\label{eq:102}
   \int_{X_{\Sigma}^{\an}}\log\| s\|_{\can} \chern_{1}(\ov L^{\can})^{\wedge n}\wedge
\delta_{X_{\Sigma}}=  n!  \,\Vol_{M}(\Delta) \Psi(0) =0.  
\end{equation}
Thus, substracting from \eqref{eq:100} the analogous formula for the
canonical metric and using the equations \eqref{eq:105},   \eqref{eq:99}
and \eqref{eq:102}, we obtain
\begin{displaymath}
  \htor_{\ov L}(X_{\Sigma})= \sum_{F}-\langle F, v_{F}\rangle 
\htor_{\ov L|_{V(\tau_{F})}}(V(\tau_{F}))
-  n!  \, \int_{N_{\R}}\psiabs 
\dd\cM_{M}(\psiabs ).
\end{displaymath}

By the induction hypothesis and the equation \eqref{eq:92} 
\begin{displaymath}
\htor_{\ov L|_{V(\tau_{F})}}(V(\tau_{F}))=
n!\,\int_{F}\vartheta  \dd \Vol_{M({F})}.  
\end{displaymath}
Hence, by Corollary \ref{cor:7},
\begin{align*}
  \htor_{\ov L}(X_{\Sigma})&= -n!\,\sum_{F}\langle
  F, v_{F}\rangle \int_{F}\vartheta  \dd \Vol_{M({F})} - n!\,
  \int_{N_{\R}}\psiabs
  \dd\cM_{M}(\psiabs )\\
  & =(n+1)!  \,\int_{\Delta} \vartheta \dd\Vol_{M},
  \end{align*}
proving the theorem.
\end{proof}

\begin{rem}\label{rem:20} The left-hand side of the equation
  \eqref{eq:108} only depends on the structure of toric line bundle of
  $L$ and not on a particular choice of toric section, while the
  right-hand side seems to depend on the section $s$. We can see
  directly that the right hand side actually does not depend on the
  section. If we pick a different toric section, say $s'$, then the
  corresponding support function $\Psi '$ differs from $\Psi $ by a
  linear functional. The polytope $\Delta _{\Psi '}$ is the translated
  of $\Delta _{\Psi '}$ by the corresponding element of $M$. The
  function $\psiabs _{\ov L,s'}$ differs from $\psiabs _{\ov L,s}$ by the
  same linear functional and $\vartheta _{\ov L,s'}$ is the
  translated of $\vartheta _{\ov L,s}$ by the same element of
  $M$. Thus the integral on the right has the same value whether we
  use the section $s$ or the section $s'$.
\end{rem}

Theorem \ref{thm:19} can
be reformulated in terms of an integral over $N_{\R}$. 

\begin{cor} \label{cor:21}
Let notation be as in Theorem \ref{thm:19} and write $\psiabs= \psiabs_{\ov
  L,s}$ for short. 
Then 
  \begin{displaymath} 
    \htor_{\ov L}(X_{\Sigma })=(n+1)!\, \int_{N_{\R}}
     (\vartheta \circ \partial  \psiabs) \dd \cM_{M}( \psiabs),
  \end{displaymath}
where $\vartheta \circ \partial  \psiabs$ is the integrable function
defined by \eqref{eq:121}.
When $\psiabs\in \cC^{2}(N_{\R})$, 
\begin{displaymath}
    \htor_{\ov L}(X_{\Sigma })=(-1)^{n}(n+1)!\, \int_{N_{\R}} (\langle
    \nabla \psiabs(u),u\rangle -\psiabs(u)) \det(\Hess(\psiabs)) \, \dd \Vol_{N}.
\end{displaymath}
When $\psiabs$ is piecewise affine,
\begin{displaymath}
\htor_{\ov L}(X_{\Sigma })=(n+1)!\, \sum_{v\in \Pi(\phiK)^{0}} 
 \int_{v^{*}} (\langle x,v\rangle - \psiabs(v))\, \dd \Vol_{M}(x),
\end{displaymath}
where $v^{*}\in \Pi(\vartheta )$ is the polytope corresponding to the
vertex $v$ with respect to the dual pair of convex decompositions
induced by $\psiabs$ (definitions \ref{def:9} and \ref{def:19}).
\end{cor}

\begin{proof}
  The first statement follows readily from  Theorem \ref{thm:19}
  and the equations~ \eqref{eq:121} and \eqref{eq:634}. The second
  statement follows from 
  Proposition \ref{prop:5} and Example \ref{exm:29}\eqref{item:96},
  while the third one 
  follows from Proposition \ref{prop:32} and Example~\ref{exm:29}\eqref{item:97}. 
\end{proof}

Theorem \ref{thm:19} can be extended to compute the local toric height
associated to distinct line bundles in term of the mixed integral of
the associated roof functions.

\begin{cor}
\label{cor:13}
  Let $\Sigma $ be a complete fan on $N_{\R}$ and $\ov
  L_{i}=(L_{i},\|\cdot\|_{i})$, $i=0,\dots, n$, be toric line bundles on
  $X_{\Sigma }$ equipped with semipositive{} toric metrics. Choose toric 
  sections $s_{i}$ of $L_{i}$ and let $\Psi_{i} $ be the corresponding 
  support functions. Then the toric height of $X_{\Sigma }$ with respect 
  to $\ov L_{0},\dots, \ov L_{n}$ is given by
  \begin{displaymath} \htor_{\ov L_{0},\dots, \ov L_{n}}(X_{\Sigma })
= \MI_{M}(\vartheta_{\|\cdot\|_{0}}, \dots,
\vartheta_{\|\cdot\|_{n}})=
\lambda_{K}\MI_{M}(\phiK_{\|\cdot\|_{0}}^{\vee}, \dots,
\phiK_{\|\cdot\|_{n}}^{\vee}).
  \end{displaymath}
\end{cor}

\begin{proof}
Let $0\le i_{0}<\dots<i_{j}\le n$. 
By the propositions~\ref{prop:24}~\eqref{item:72} and \ref{prop:10}~\eqref{item:33} 
\begin{displaymath}
  (\psiabs_{\ov L_{i_{0}}\otimes\dots\otimes \ov L_{i_{j}},
    s_{i_{0}}\otimes\dots\otimes  s_{i_{j}}})^{\vee}= 
  \psiabs_{\ov L_{i_{0}}, s_{i_{0}}} ^{\vee}\boxplus\dots\boxplus
  \psiabs_{\ov L_{i_{j}}, s_{i_{j}}} ^{\vee}. 
\end{displaymath}
The result then follows from \eqref{eq:118}, the definition
of the mixed integral (Definition~\ref{def:49}) and Theorem \ref{thm:19}. 
\end{proof}

\begin{rem}\label{rem:7}
In the DSP case, the toric height can be expressed as an
alternating sum of mixed integrals as follows. Let
$\ov L_{i}=(L_{i},\|\cdot\|_{i})$, $i=0,\dots, n$, be toric line bundles on $X_{\Sigma }$
  equipped with DSP
  toric metrics and set $\ov L_i= \ov L_{i,+} \otimes
  \ov L_{i,-}^{\otimes -1}$ for some 
semipositive{} metrized toric line bundles $\ov
L_{i,+}$, $ \ov L_{i,-}$. Choose a toric section for each line bundle
and write  $\vartheta_{i,+}$ and
$\vartheta_{i,-}$ for the corresponding roof functions. 
Then
  \begin{displaymath}
    \htor_{\ov L_{0},\dots, \ov L_{n}}(X_{\Sigma
    })=\sum_{\epsilon_{0},\dots, \epsilon_{n}\in \{\pm1\}}
\epsilon_{0}\dots \epsilon_{n}
\MI_{M}(\vartheta_{{0},\epsilon_{0}}, \dots,
\vartheta_{n,\epsilon_{n}}).
  \end{displaymath}
\end{rem}

We have defined and computed the local height of a toric variety. We
now will compute the toric height of toric subvarieties. We start with
the case of orbits.

\begin{prop}\label{prop:84} Let $\Sigma $ be a complete fan on $N_{\R}$
  and 
  $\sigma \in \Sigma $ a cone of codimension $d$. Let $V(\sigma )$ be
  the closure of the orbit associated to $\sigma $ and
   $\iota_{\sigma
  }\colon X_{\Sigma (\sigma 
    )}\to X_{\Sigma }$ the closed immersion of Proposition
  \ref{prop:73}. Let $L$ be a toric line bundle on $X_{\Sigma }$, 
  $s$ a toric section, $\Psi $ the corresponding support function and 
  $\|\cdot\|$ a semipositive{} toric metric on $L^{\an}$. As usual write 
  $\ov L=(L,\|\cdot\|)$. Then
  \begin{displaymath}
    \htor_{\ov L}(V(\sigma ))
    =\htor_{\iota_{\sigma }^{\ast} \ov L}(X_{\Sigma (\sigma
      )})= (d+1)! 
    \int_{F_{\sigma }}\vartheta_{\ov L,s} \dd\Vol_{M(F_{\sigma })},
  \end{displaymath}
  where $F_{\sigma }$ is the face of $\Delta _{\Psi }$ corresponding
  to $\sigma $, $M(F_{\sigma })$ is the lattice induced by $M$ on
  the linear space associated to $F_{\sigma }$ and $\iota_{\sigma } ^{\ast}L$
  has the toric line bundle structure of Proposition \ref{prop:81}.
\end{prop}

\begin{proof} By Corollary  \ref{cor:19} the restriction of the
  canonical metric of $L^{\an}$ is the canonical metric of $\iota_{\sigma } ^{\ast}
  L^{\an}$. Therefore, the equality $\htor_{\ov L}(V(\sigma
  ))=\htor_{\iota_{\sigma }^{\ast} 
    \ov L}(X_{\Sigma (\sigma 
      )})$  follows from Theorem
  \ref{thm:1}\eqref{item:8}.

  To prove the second equality, choose $m_{\sigma }\in F_{\sigma }\cap
  M$. We use the notation of Proposition \ref{prop:62}. By Theorem
  \ref{thm:19},
  \begin{displaymath}
    \htor_{\iota_{\sigma }^{\ast}
    \ov L}(X_{\Sigma (\sigma 
      )})=(d+1)!\int_{\Delta _{(\Psi -m_{\sigma })(\sigma )}}
    \vartheta _{\|\cdot\|_{\sigma }}\dd\Vol_{M(\sigma )}.
  \end{displaymath}
  By Proposition \ref{prop:47}, 
  \begin{math}
    \Delta _{(\Psi -m_{\sigma })(\sigma )}=(\pi _{\sigma
    }^{\vee}+m_{\sigma })^{-1}F_{\sigma }. 
  \end{math} By Proposition \ref{prop:62}
  \begin{displaymath}
    \vartheta _{\|\cdot\|_{\sigma }}=\psiabs^{\vee}_{\|\cdot\|_{\sigma }} =
    (\pi _{\sigma }^{\vee}+m_{\sigma
    })^{\ast}\psiabs^{\vee}_{\|\cdot\|}=
    (\pi _{\sigma }^{\vee}+m_{\sigma
    })^{\ast}\vartheta _{\|\cdot\|}.   
  \end{displaymath}
  Since $M(F_{\sigma })=M(\sigma )$, we obtain
  \begin{displaymath}
    \int_{\Delta _{(\Psi -m_{\sigma })(\sigma )}}
    \vartheta _{\|\cdot\|_{\sigma }}\dd\Vol_{M(\sigma )}=
    \int_{F_{\sigma }}\vartheta_{\|\cdot\|} \dd\Vol_{M(F_{\sigma })},
  \end{displaymath}
  proving the result.
\end{proof}

We now study the behaviour of the toric local height with respect to
toric morphisms. 

\begin{notn} \label{def:89}
Let $N_{1}$ be a lattice of rank $d$ and $M_{1}$ the
dual lattice. Let $H\colon
N_{1}\to N$ be a linear map and $\Sigma _{1}$ a complete fan on $N_{1,\R}$
such that, for each cone $\sigma \in \Sigma _{1}$, $H(\sigma )$ is
contained in a cone of $\Sigma $. Let $\varphi\colon X_{\Sigma
  _{1}}\to X_{\Sigma }$ be the associated morphism of proper toric
varieties over $K$. Denote
$Q=H(N_{1})^{\sat}$ the saturated sublattice of $N$ and let
$Y_{Q}$ be the image of $X_{\Sigma _{1}}$ under $\varphi$. Then
$Y_{Q}$ is equal to the toric subvariety $Y_{\Sigma ,Q}=Y_{\Sigma
  ,Q,x_{0}}$ of Definition \ref{def:74}, where we recall that $x_0$
denote the distinguished point of the principal orbit of
$X_{\Sigma}$.  
\end{notn}

\begin{prop} \label{prop:82} With Notation \ref{def:89}, let $\ov L$
  be a toric line bundle on $X_{\Sigma }$ equipped with a
  semipositive{} toric metric. We put on $\varphi^{\ast}
  L$ the structure of toric line bundle of Remark \ref{rem:19}. Choose
  a toric section $s$ of $L$ and let $\Psi $ be the associated support
  function.   
  \begin{enumerate}
  \item \label{item:94} If $H$ is not injective, then
    $\htor_{\varphi^{\ast}\ov L}(X_{\Sigma _{1}})=0$.
  \item \label{item:95} If $H$ is injective, then
    $\htor_{\varphi^{\ast}\ov L}(X_{\Sigma _{1}})=[Q:H(N_{1})]\htor_{\ov
      L}(Y_{Q})$. Moreover
    \begin{displaymath}
      \htor_{\varphi^{\ast}\ov L}(X_{\Sigma _{1}})=
      (d+1)!\int_{H^{\vee}(\Delta _{\Psi 
        })}(H^{\vee})_{\ast}(\vartheta _{\|\cdot\|} )\dd\Vol_{M_{1}}. 
    \end{displaymath}
  \end{enumerate}
\end{prop}
\begin{proof}
  By Corollary \ref{cor:20}, the inverse image of the canonical metric
  by a toric morphism is the canonical metric. Thus \eqref{item:94}
  and the first statement of \eqref{item:95} follow from
  Theorem \ref{thm:1}~\eqref{item:8} and the equation (\ref{eq:54}).
 
  By the propositions \ref{prop:11} and \ref{prop:58} and Theorem
  \ref{thm:19} we deduce 
  \begin{multline*}
    \htor_{\varphi^{\ast}\ov L}(X_{\Sigma _{1}})=(d+1)!\lambda
    _{K}\int_{\Delta _{\Psi
        \circ H}} (H^{\ast}\phiK_{\|\cdot\|})^{\vee}\dd\Vol_{M_{1}}\\=
    (d+1)!\int_{H^{\vee}(\Delta _{\Psi 
        })}(H^{\vee})_{\ast}(\vartheta _{\|\cdot\|} )\dd\Vol_{M_{1}},
  \end{multline*}
proving the result.
\end{proof}

We now study the case of an equivariant morphism. Let $N$, $N_{1}$,
$d$, $H$, $\Sigma $ and $\Sigma _{1}$ as in Notation \ref{def:89}. For
simplicity, we assume that $H\colon N_{1}\to N$ is injective and that
$Q=H(N_{1})$ is a saturated sublattice, because the effect of a
non-injective map or a non-saturated sublattice can be deduced from
Proposition \ref{prop:82}. Let $p\in X_{\Sigma ,0}(K)$ be a point of
the principal open subset and $u=\val(p)\in N_{\R}$. Denote
$\varphi=\varphi_{p,H}$ the equivariant morphism determined by $H$ and
$p$ as in \eqref{eq:16}, also denote $Y=Y_{\Sigma ,Q,p}$ the image of
$X_{\Sigma _{1}}$ by $\varphi$ as in~\eqref{eq:31}. Finally write
$A=H+u$ for the associated affine map.

Let $\ov L$ be a toric line bundle equipped with a
  semipositive{} toric metric.  As explained in Remark \ref{rem:19},
there is no natural structure of toric line bundle on the
inverse image $\varphi^{\ast}L$. To obtain one, we choose a toric
section~$s$ of $L$ and we denote by 
$\ov L_{1}$ the line bundle $\varphi
^{\ast}L$ with the metric induced by $\|\cdot\|$ and the toric
structure induced by the chosen section $s$. We denote by $\Psi $ the support
function associated to $(L,s)$.

\begin{prop}\label{prop:83}
  With the previous hypothesis and notations, the equality
  \begin{multline}
    \label{eq:88}
    \htor_{\ov L_{1}}(X_{\Sigma _{1}})=
(d+1)!\,
    \int_{H^{\vee}(\Delta _{\Psi
      })}(A^{\ast}\psiabs _{\ov L,s})^{\vee}\dd\Vol_{M_{1}}\\
    =(d+1)!\, 
    \int_{H^{\vee}(\Delta _{\Psi
      })}(H^{\vee})_{\ast}(\psiabs ^{\vee}_{\ov L,s}-u)\dd\Vol_{M_{1}}
  \end{multline}
  holds. Moreover
  \begin{multline}
    \label{eq:117}
    \htor_{\ov L_{1}}(X_{\Sigma _{1}})-\htor_{\ov L}(Y)= (d+1)!\, 
    \int_{H^{\vee}(\Delta _{\Psi
      })}(A^{\ast}\Psi)^{\vee}\dd\Vol_{M_{1}}\\
    =(d+1)!\,
    \int_{H^{\vee}(\Delta _{\Psi
      })}(H^{\vee})_{\ast}(\iota_{\Delta _{\Psi }} -u)\dd\Vol_{M_{1}},
  \end{multline}
  where $\iota _{\Delta _{\Psi }}$ is the indicator function of
  $\Delta _{\Psi }$ (see Example \ref{exm:7}).
\end{prop}

\begin{proof}
  By Proposition \ref{prop:58}, $\psiabs _{\ov
    L_{1},\varphi^{\ast} s}=A^{\ast}\psiabs _{\ov L,s}$. By Proposition
  \ref{prop:11}\eqref{item:23} we obtain that $\Stab(A^{\ast}\psiabs _{\ov
    L,s})=H^{\vee}(\Delta _{\Psi })$ and that
  \begin{displaymath}
    (A^{\ast}\psiabs _{\ov L,s})^{\vee}=H^{\vee}_{\ast}(\psiabs _{\ov L,s}-u).
  \end{displaymath}
  Then \eqref{eq:88} follows from Theorem \ref{thm:19}. 
  
  To prove (\ref{eq:117}), possibly replacing $\Sigma_{1}$ by a
  refinement, we assume that $X_{\Sigma_{1}}$ is projective.  Since
  $H$ is injective and $Q$ is saturated, by (\ref{eq:54}), the map
  $X_{\Sigma _{1}}\to Y$ has degree one. Then, by Definition
  \ref{def:35},
  \begin{equation}\label{eq:36}
    \htor_{\ov L}(Y) = \h_{\ov L_1}(X_{\Sigma_1};s_{0},
    \dots, s_{d-1}) - 
    \h_{\varphi^\ast(\ov L^\can)}(X_{\Sigma_1};s_{0},
    \dots, s_{d-1}),
  \end{equation}
  where $s_{i}$, $i=0,\dots,d-1$, is a collection of rational sections
  of $\varphi^\ast L$ meeting $X_{\Sigma_{1}}$ properly, and
  $\varphi^{\ast}(\ov L^{\can})$ has the toric structure induced by
  $s$ and the metric induced by the canonical metric of $L$. We recall
  that this metric may differ from the canonical metric of
  $\varphi^{\ast}L$.  Anyway, subtracting $\h_{\ov{\varphi^\ast
      L}^\can}(X_{\Sigma_1};s_{0}, \dots, s_{d-1}) $ from both terms
  of the difference in the right hand side of (\ref{eq:36}) and
  rearranging the equation, we get
  \begin{displaymath}
    \htor_{\ov L_{1}}(X_{\Sigma _{1}})-\htor_{\ov
      L}(Y)=\htor_{\varphi^{\ast}(\ov L^{\can})}(X_{\Sigma _{1}}).
  \end{displaymath}
  Now \eqref{eq:117} follows from 
  \eqref{eq:88}, Example \ref{exm:7} and the definition of the
  canonical metric.  
\end{proof}

\begin{cor}\label{cor:25} With the previous hypothesis
  \begin{displaymath}
    \htor_{\varphi^{\ast}(\ov L^{\can})}(X_{\Sigma _{1}})= (d+1)!\, 
    \int_{H^{\vee}(\Delta _{\Psi
      })}(A^{\ast}\Psi)^{\vee}\dd\Vol_{M_{1}}.
  \end{displaymath}
\end{cor}

\begin{exmpl}\label{exm:36}
  We continue with Example \ref{exm:35}. Let $\Z^{r}$
 be the standard lattice of rank $r$, $\Delta ^{r}$ the standard
 simplex of dimension $r$ and $\Sigma _{\Delta ^{r}}$ the fan of
 $\R^{r}$ associated to $\Delta ^{r}$. The corresponding toric variety
 is $\P^{r}$. Let $H\colon N\to \Z^{r}$ be an injective linear
 morphism such that $H(N)$ is a saturated sublattice. Denote
 $m_{i}=e_{i}^{\vee}\circ H\in M$, $i=1,\dots,r$. Let $\Sigma 
 $ be the regular fan on $N$ defined by $H$ and $\Sigma _{\Delta
   ^{r}}$. Let $\Psi _{\Delta ^{r}}$ be the support function of
 $\Delta ^{r}$ and let $\Psi =\Psi _{\Delta ^{r}}\circ H$. Explicitly,
 \begin{displaymath}
   \Psi (v)=\min(0,m_{1}(v),\dots,m_{r}(v)).
 \end{displaymath}
Let $p\in \P^{r}_{0}(K)$ and $u=\val(p)\in \R^{r}$. Write 
$u=(u _{1},\dots,u _{r})$. If $p=(1:\alpha _{1}:\dots:\alpha _{r})$,
then $u_{i}=-\log|\alpha_{i}|$. 
There is an equivariant morphism
$\varphi:=\varphi_{p,H}\colon X_{\Sigma }\to \P^{r}$. Consider the
toric line bundle with toric section determined by $\Psi _{\Delta
  ^{r}}$ with the canonical metric and denote by $(\ov L,s)$ the
induced toric line bundle with toric section on $X_{\Sigma }$ equipped
with the induced metric. Then
\begin{displaymath}
  \psiabs _{\ov L,s}(v)=\min(0,m_{1}(v)+u_{1},\dots,m_{r}(v)+u
  _{r}). 
\end{displaymath}
Thus $\Delta =\Stab(\psiabs _{\ov
  L,s})=\Conv(0,m_{1},\dots,m_{r})=H^{\vee}(\Delta ^{r})$. 
 By Proposition \ref{prop:19} the Legendre-Fenchel dual $\vartheta  _{\ov
   L,s}\colon \Delta \to \R$ is given by
 \begin{displaymath}
   \vartheta  _{\ov L,s}(x) =
   \sup\bigg\{\sum_{j=1}^{r}-\lambda_{j}u_{j}\bigg|\  
\lambda_{j}\ge 0, \sum_{j=1}^{r}\lambda_{j}\le 1, \  
\sum_{j=1}^{r}\lambda_{j}m_{j}=x \bigg \}\    \text{ for }  x\in \Delta .
 \end{displaymath}
Thus the roof function $\vartheta _{\ov L,s}=\psiabs
_{\ov L,s}^{\vee}$ is the upper
envelope of the extended polytope 
\begin{multline*}
  \Conv\left((0,0),(m_{1},-u_{1}),\dots,(m_{r},-u_{r})\right)\\=
  \Conv\left((0,0),(m_{1},\log|\alpha _{1}|),\dots,(m_{r},\log|\alpha
    _{r}|)\right ).
\end{multline*}
\end{exmpl}

\section{Global heights of toric varieties}
\label{sec:global-height-toric}

In this section we prove the integral formula for the global height of
a toric variety.
Let $(\K, \mathfrak{M})$ be an adelic field as in Definition
\ref{def:6}. Let $\Sigma $ be a complete
fan on $N_{\R}$ and  $\Psi_{i}$, $i=0,\dots, d$, be virtual support
functions 
on $\Sigma$. For each $i$, let $L_{i}=L_{\Psi_{i}}$ and
$s_{\Psi_{i}}$ be the 
associated toric line bundle and toric section, and $\|\cdot\|_{i}=(\|\cdot
\|_{i,v})_{v\in \mathfrak{M}}$ 
a DSP adelic  toric metric on $L_{i}$. 
Write $\ov L_{i}=(L_{i}, \|\cdot\|_{i})$ and $\ov L_{i}^{\can}$
for the same line bundles equipped with   
the canonical metric at all the places.
By Example \ref{exm:39}, it
is also a DSP adelic toric metric.

From the local toric height we can define a toric (global) height for
adelic toric metrics as follows.

\begin{defn} \label{def:67} Let $Y$ be a $d$-dimensional cycle
  of $X_{\Sigma }$.
  The \emph{toric height} of $Y$ with respect to 
$\ov L_{0},\dots, \ov L_{d}$
\index{height of cycles!toric} is 
\begin{displaymath}
  \htor_{\ov L_{0},\dots, \ov L_{d}}(Y)= \sum_{v\in \mathfrak{M}}n_{v}
  \htor_{v,\ov L_{0},\dots, \ov L_{d}}(Y)\in \R,
\end{displaymath}
where $\htor_{v}$ denotes the local toric height of $Y_{v}$.
\end{defn}
 

\begin{rem}\label{rem:22} 
  Definition \ref{def:67} makes sense because the
  condition of the metrics being adelic imply that only a finite 
  number of terms in the sum are nonzero. Moreover, 
  the value of the toric height depends on the toric structure of
  the involved line bundle, but its class in $\R/\!\df(\K^{\times})$
  does not.
\end{rem}

\begin{rem}\label{rem:23}
  In general, the toric height is not a global height in the
  sense of Definition \ref{def:61}. When the $d$-dimensional cycle $Y$
  is integrable 
  with respect to $\ov L_{0}^{\can},\dots, \ov L_{d}^{\can}$ 
  (Definition \ref{def:47}), then it 
  is also integrable with respect to $\ov L_{0},\dots, \ov
  L_{d}$ and
  \begin{displaymath}
    \htor_{\ov L_{0},\dots, \ov L_{d}}(Y)=
    \h_{\ov L_{0},\dots, \ov L_{d}}(Y)-
    \h_{\ov L_{0}^{\can},\dots, \ov L_{d}^{\can}}(Y).
  \end{displaymath}
  Observe also that, by Proposition \ref{prop:29} and Theorem
  \ref{thm:23}, when $\K$ is a global field, all cycles are
  integrable with respect to line bundles with DSP adelic toric
  metrics. 
\end{rem}

The next result shows that the closure of an
orbit or a toric subvariety is always integrable, even if the adelic
field is
not a global field, and that its global height 
agrees with its toric height.  

\begin{prop} \label{prop:88}
With notations as above, let $Y$ be either the closure of an orbit or
a toric subvariety. Then 
$Y$ is integrable with respect to $\ov L_{0},\dots, \ov
L_{d}$. Moreover, its global height is given by
\begin{equation*}
  \h_{\ov L_{0},\dots, \ov L_{d}}(Y)=
  \left[ \htor_{\ov L_{0},\dots, \ov L_{d}}(Y)\right] \in \R/\!\df(\K^{\times}).
\end{equation*}
\end{prop}

\begin{proof} In view of the propositions \ref{prop:84} and \ref{prop:82} and
  the fact that the restriction of the canonical metric to closures of
  orbits and to toric subvarieties is the canonical metric
  (corollaries \ref{cor:19} and \ref{cor:20}), we are reduced to treat
  the case $Y=X_{\Sigma }$. By the toric Chow's lemma \cite[Proposition
  2.17]{Oda88}, Proposition \ref{prop:94}\eqref{item:11} and Theorem
  \ref{thm:14}\eqref{item:93}, we can reduce to the case when
  $X_{\Sigma }$ is projective. 
  
  Thus we assume that $X_{\Sigma }$ has dimension $d$. 
We next prove that $X_{\Sigma}$ is integrable with respect to  $\ov
L_{0}^{\can},\dots, \ov 
L_{d}^{\can}$ and that the corresponding global height is zero.
By a polarization argument as in \eqref{eq:118}, we can reduce to the case
$\Psi_{0}=\dots=\Psi_{d}=\Psi$. 
The proof is done by induction on $d$. 
For short, write $L=\cO(D_{\Psi})$ and  $s=s_{\Psi}$. 
 
Let $d=0$. Then $X_{\Sigma }$ reduces to the point $x_{0}$. By
the equation \eqref{eq:96}, 
for each $v\in \mathfrak{M}$, 
\begin{displaymath}
  \h_{v,\ov L^{\can}}(X_{\Sigma};s)= -\log
  \|s(x_{0})\|_{v,\Psi}= \Psi(0)=0.
\end{displaymath}
Furthermore,  $  \h_{\ov L^{\can}}(X_{\Sigma};s)= \sum_{v}n_{v }\h_{v,\ov
  L^{\can}}(X_{\Sigma};s)=0$. 

Now let $d\ge 1$. Choose sections (non-necessarily toric) $s_{0},
\dots,s_{d-1}$ such that 
$s_{0}, \dots,s_{d-1}, s$ meet $X_{\Sigma }$ properly.
By the construction of local heights, for each $v\in \mathfrak{M}$,
\begin{align}
\h_{v,\ov L^{\can}}(X_{\Sigma};
s_{0},\dots,s_{d-1},s) 
=& \h_{v,\ov L^{\can}}(\div(s);  s_{0},\dots,
  s_{d-1})\\
& -  \int_{X_{\Sigma,v}^{\an}}\log\| s\|_{v,\Psi} 
\chern_{1}(\ov L^{v,\can})^{\wedge d}\wedge
\delta_{X_{\Sigma}}. \nonumber
\end{align}
As shown in \eqref{eq:102}, the last term in the equality
above vanishes. Hence
\begin{displaymath}
  \h_{v,\ov L^{\can}}(X_{\Sigma};
s_{0},\dots,s_{d-1},s) 
= \h_{v,\ov L^{\can}}(\div(s);  s_{0},\dots,
  s_{d-1}).
\end{displaymath}
The divisor $\div(s)$ is a linear combination of subvarieties of the
form $V(\tau)$, $\tau\in \Sigma^{1}$, and the restriction of the
canonical metric to these varieties coincides with their canonical
metrics. With the inductive hypothesis, this shows that $X_{\Sigma}$
is integrable with respect to  $\ov L^{\can}$. 
Adding up the resulting equalities over all places,
\begin{displaymath}
  \h_{\ov L^{\can}}(X_{\Sigma};s_{0},\dots,s_{d-1},s) 
= \h_{\ov L^{\can}}(\div(s);s_{0},\dots,s_{d-1}).
\end{displaymath}
Using again the inductive hypothesis, $  \h_{\ov
  L^{\can}}(X_{\Sigma};s_{0},\dots,s_{d-1},s)\in \df(\K^{\times})$. 

We now prove the statements of the theorem.  Again by a polarization
argument, we can also reduce to the case when $\ov L_{0}=\dots=\ov
L_{d}=\ov L$.  By the definition of semipositive{} adelic toric
metrics, $X_{\Sigma}$ is also integrable with respect to $\ov
L$. Furthermore,
\begin{displaymath}
  \htor_{\ov L}(X_{\Sigma})=  \h_{\ov L}(X_{\Sigma};s_{0},\dots,s_{d})
-\h_{\ov L^{\can}}(X_{\Sigma};s_{0},\dots,s_{d})
\end{displaymath}
for any choice of sections $s_{i}$ intersecting $X_{\Sigma}$ properly.
Hence, the classes of $  \htor_{\ov L}(X_{\Sigma})$ and of $ \h_{\ov
  L}(X_{\Sigma};s_{0},\dots,s_{d})$ agree up to
$\df(\K^{\times})$. But the latter is the global height of
$X_{\Sigma}$ with respect to $\ov L$, hence the second statement.
\end{proof}

Summing up the preceding results we obtain a formula for the height of
a toric variety.

\begin{thm} \label{thm:22} Let $\Sigma$ be a complete fan on $N_{\R}$.
  Let $\ov L_{i}$, $i=0,\dots, n$, be toric line
  bundles on $X_{\Sigma }$ equipped with semipositive
  adelic toric metrics. 
  For each $i$, let $s_{i}$ be a toric section of $L_{i}$.
  Then the height of $X_{\Sigma }$ with respect to
  $\ov L_{0},\dots, \ov L_{n}$ is
  \begin{align*}
    \h_{\ov L_{0},\dots, \ov L_{n}}(X_{\Sigma })
    =& 
\left[    \sum_{v\in \mathfrak{M}}n_{v}\MI_{M}(\vartheta_{v,\ov L_{0}, s_{0}}, \dots,
    \vartheta_{v,\ov L_{n},s_{n}}) \right] \in \R/\!\df(\K^{\times}),
  \end{align*}
  where $\vartheta_{v,\ov L, s}$ denotes the local roof function.
  In particular, if $\ov L_{0}=\dots=\ov L_{n}=\ov L$, let $s$ be a
  toric section and put
  $\Delta =\Stab (\Psi _{L,s} )$. Then 
  \begin{displaymath}
    \h_{\ov L}(X_{\Sigma }) =\left[(n+1)!\sum_{v\in
        \mathfrak{M}}n_{v}\int_{\Delta}
    \vartheta_{v,\ov L,s} \dd\Vol_{M}\right].
  \end{displaymath}
\end{thm}

\begin{proof}
  This follows readily from Corollary \ref{cor:13} and Proposition
  \ref{prop:88}.
\end{proof}

\begin{cor}\label{cor:26}
Let $H\colon N\to \Z^{r}$ be an injective map such that $H(N)$ is a
saturated sublattice of $\Z^{r}$,  $p\in \P_{0}^{r}(\K)$ a point in
the principal open subset and $Y\subset
\P^{r}$ the closure of the
image of the map $\varphi_{p,H}\colon \T\to \P^{r}$.  
Let $m_{0}\in M$ and $m_{i}=e_{i}^{\vee}\circ H+m_{0}\in M$,
$i=1,\dots,r$, and write  
$p=(p_{0}:\dots:p_{r})$ with $p_{i}\in \K^{\times}$. Let
$\Delta=\Conv(m_{0},\dots, m_{r})\subset M_{\R}$ and
$\vartheta_{v}\colon\Delta\to \R$ the function
parameterizing the upper
envelope of the extended polytope 
\begin{displaymath}
  \Conv\left((m_{0},\log|p_{0}|_{v}),\dots,(m_{r},\log|p
    _{r}|_{v})\right ) \subset M_{\R}\times \R.
\end{displaymath}
Let $\ov{\cO(1)}^{\can}$ be the universal line bundle on $\P^{r}$ with
the canonical metric as in Example~\ref{exm:37}\eqref{item:107}.
Then $Y$ is integrable with respect to $\ov{\cO(1)}^{\can}$ and
\begin{displaymath}
  \h_{\ov{\cO(1)}^{\can}}(Y)= \left[ (n+1)!  \sum_{v\in \mathfrak{M}} n_{v}
   \int _{\Delta}\vartheta _{v}\dd \Vol_{M}\right] \in \R/\!\df(\K^{\times}).
  \end{displaymath}
\end{cor}

\begin{proof}
By the definition of adelic field,  $\val_{\K_{v}}(p)=0$ for almost all $v\in
\mathfrak{M}$.
Therefore, the integrability of $Y$ follows as in the proof of
Proposition \ref{prop:88}. 

Let $\Sigma$ be the complete regular fan of $N_{\R}$ induced by $H$
and $\Sigma_{\Delta^{r}}$, and let
$X_{\Sigma}$ be the associated toric variety. 
Write $\varphi=\varphi_{p,H}$ for short.
The fact that  $H(N)$ is saturated implies that $\varphi$ has
degree 1 and so  $Y=\varphi_{*}
X_{\Sigma}$. 
By the functoriality of the global height (Theorem
\ref{thm:14}\eqref{item:93}), 
\begin{displaymath}
\h_{\ov{\cO(1)}^{\can}}(Y) = \h_{\varphi^{*}(\ov{\cO(1)}^{\can})}(X_{\Sigma}).   
\end{displaymath}
Let $v\in \mathfrak{M}$. Using the results in Example \ref{exm:36},
it follows from Theorem \ref{thm:22} that
  \begin{displaymath}
    \h_{\varphi^{\ast}(\ov \cO(1)^{\can})}(X_{\Sigma })= \left[(n+1)!
 \sum_{v}   \int_{\ov \Delta} n_{v}\ov \vartheta_{v} \dd\Vol_{M}\right].
  \end{displaymath}
where $\ov \Delta= \Conv(0, m_{1}-m_{0}, \dots, m_{r}-m_{0})\subset M_{\R}$
and $\ov \vartheta_{v}$ is the function parameterizing the upper envelope
of the extended polytope
\begin{displaymath}
  \Conv\left((0,0), (m_{1}-m_{0},
    \log|p_{1}/p_{0}|_{v}),\dots,(m_{r}-m_{0},
    \log|p_{r}/p_{0}|_{v})\right ) \subset M_{\R}\times \R.
\end{displaymath}
We have that $\ov \Delta= \Delta- m_{0}$ and
 $\ov \vartheta_{v}= \tau_{-m_{0}}\vartheta_{v}- \log|p_{0}|_{v}$. 
Hence,
\begin{displaymath}
  \int_{\ov \Delta} \ov \vartheta_{v} \dd\Vol_{M}=
 \int_{\Delta} \vartheta_{v} \dd\Vol_{M} - \log|p_{0}|_{v}
 \Vol_{M}(\Delta).  
\end{displaymath}
Using that $\sum_{v} n_{v}
\log|p_{0}|_{v}\in \df(\K^{\times})$ and that, by 
Proposition \ref{prop:46},
$n!\Vol_{M}(\Delta )=\deg_{\cO(1)}(Y)\in \Z$,
we deduce the result.
\end{proof}

\begin{rem}\label{rem:9}
The above corollary  can be easily extended to the mixed case by using 
an argument similar to that in the proof of Corollary \ref{cor:13}. 
Applying the obtained result to the case when $\K$ is a number field
(respectively, the field of rational functions of a complete curve) we
 recover  \cite[Th\'eor\`eme~0.3]{PS08a} (respectively,
\cite[Proposition 4.1]{MR2419926}). 
\end{rem}
 

\chapter{Metrics from polytopes} \label{Metricfrompolytopes}
\section{Integration on polytopes} \label{Intepol}

In this chapter, we present a closed formula for the integral over a
polytope of a function of one variable composed with a linear form,
extending in this direction Brion's formula for the case of a simplex
\cite{Brion:PointsEntiers}, see Proposition \ref
{propdefcoeffdintegration} and Corollary~\ref{cor:16} below.  In the
next section, these formulae will allow us to compute the height of
toric varieties with respect to some interesting metrics arising from
polytopes.

We consider the vector space $\R^{n}$ with its usual scalar
product, that we denote $\langle \cdot,\cdot \rangle$, and its Lebesgue
measure, that we denote $\Vol_{n}$. We also consider a polytope 
$\Delta\subset \R^n$ of dimension $n$.

\begin{defn}\label{def:53}
Let $u\in
\R^n$ be a vector. For each $c\in \R$, an \index{aggregate of a polytope}
\emph{aggregate of~$\Delta$ in the direction $u$} is the
union of all
faces of $\Delta$ contained in the affine subspace
\begin{displaymath}
 \{x\in \R^{n}\mid \langle x,u\rangle=c\}. 
\end{displaymath}
We denote by $\dim(V)$ the maximal dimension of a face of $\Delta $
contained in $V$. In particular,  $\dim(\emptyset)=-1$.

We write $\Delta(u)$ for the set of non-empty aggregates of $\Delta$ in the
direction $u$. In particular, $\Delta (0)=\{\Delta\}$. 
Note that, if $V \in \Delta(u)$ and $x$ is a point
in the affine space spanned by $V$, then the value $\langle
x,u\rangle$ is independent of $x$. We denote this common value by
$\langle V,u\rangle$. 
\end{defn}

For any two aggregates $V_{1},V_{2}\in
\Delta(u)$, we have $V_{1}=V_{2}$ if and only if $\langle V_{1},u\rangle
=\langle V_{2},u\rangle $.

\begin{exmpl}\ 
  \begin{enumerate}
  \item Every facet of a polytope is an aggregate in the direction
    orthogonal to the facet.
  \item If $u$ is general enough, the set $\Delta (u)$
    agrees with the
    set of vertices of $\Delta $.
  \item Let $\Delta =\{(x,y)\in \R^{2}\mid 0\le x,y\le 1\}$ be the
    unit square and $u=(1,1)$. Then the set of aggregates $\Delta (u)$
    contains three elements: $\{(0,0)\}$,
    $\{(1,0),(0,1)\}$ and $\{(1,1)\}$.
  \end{enumerate}
\end{exmpl}

In each facet $F$ of $\Delta$ we choose a point $m_{F}$. Let $L_{F}$
be the linear hyperplane defined by $F$ and $\pi_F$
the orthogonal projection of $\R^n$ onto $L_{F}$.
Then, $F-m_{F}$ is a
polytope in $L_{F}$ of full dimension $n-1$. To ease the notation, we
 identify $F-m_{F}$ with $F$. Observe that, with this
identification, for $V\in \Delta(u)$, the  
intersection $V\cap F$ is an aggregate of $F$ in the
direction $\pi _{F}(u)$. 
We also denote by $u_{F}$  the inner normal vector to $F$ of norm 1.

\begin{defn} \label{def:63} Let $u\in\R^{n}$ be a vector. For each
  aggregate $V$ in the direction of $u$, we define the
  coefficients $C_{k}(\Delta,u,V)$, $k\in \N$, recursively.
If $u=0$, then  $V$ is either $\emptyset$ or $\Delta$. For both cases,
we set
\nomenclature[aC904uv]{$C_{k}(\Delta,u,V)$}{coefficient of $C(\Delta,u,V)$}%
\begin{displaymath}
  C_k(\Delta,0,V)=
  \begin{cases}
    \Vol_{n}(V)&\text{ if }k=n,\\
     0&\text{ otherwise.}
  \end{cases}
\end{displaymath}
If $u\ne 0$, we set  
\begin{equation*} 
C_k(\Delta,u,V) = -\sum_{F} 
\frac{\langle u_{F},u\rangle}{\|u\|^{2}} C_k(F,\pi_F(u), V\cap F),
\end{equation*}
where the sum is over the facets $F$ of $\Delta$.
This recursive formula implies that $C_k(\Delta,u,V)=0$ for all $k>\dim(V)$.

Finally, we define the polynomial associated to an aggregate 
\nomenclature[aC904uv]{$C(\Delta,u,V)$}{polynomial associated to an aggregate}%
by
$$
C(\Delta,u,V)(z) = \sum_{k=0}^{\dim(V)}
\frac{k!}{\dim(V)!}C_k(\Delta,u,V)  z^{\dim(V)-k}  \in \R[z].
$$
In particular, we have always $C(\Delta,u,\emptyset)=0$.
\end{defn}

As usual, we write $\mathscr{C}^{n}(\R)$ for the space of 
functions of one real variable which are 
$n$-times continuously differentiable.
For $f\in \mathscr{C}^{n}(\R)$  and $0\le k\le n$, we write
$f^{(k)}$ for the $k$-th derivative of $f$.

We want to give a formula that, for $f\in \mathscr{C}^{n}(\R)$, computes
$\int_\Delta f^{(n)}(\langle x,u\rangle)\dd \Vol_{n}(x)$ in terms of the values
of the function $x\mapsto f(\langle x,u\rangle)$ at the vertices of
$\Delta$. However, when $u$ is orthogonal
to some faces of $\Delta$ of positive dimension, such a formula
necessarily depends
on the values of the derivatives of $f$. 



\begin{prop}\label{propdefcoeffdintegration}
Let $\Delta\subset \R^n$ be a polytope of dimension $n$ and $u\in
\R^n$. Then, for any $f\in \mathscr{C}^{n}(\R)$, 
\begin{align}\label{formuledintegration}
\int_\Delta f^{(n)}(\langle x,u\rangle)\dd \Vol_{n}(x) &=
\sum_{V\in\Delta(u)}\sum_{k\ge0} C_k(\Delta,u,V) f^{(k)}(\langle
V,u\rangle)\\
&=
\sum_{V\in\Delta(u)}
\frac{\dd^{\dim(V)}}{\dd z^{\dim(V)}}
\big(C(\Delta,u,V)(z)\cdot f(z+\langle V,u\rangle)\big)\big|_{z=0}.\notag
\end{align}
The coefficients $C_k(\Delta,u,V) $ are uniquely determined by this
identity. 
\end{prop}

\begin{proof}
In view of Definition \ref{def:63}, both formulae in the above
statement are equivalent and so it is enough to prove the first one. 
In case $u=0$, we have $\Delta(u)=\{\Delta\}$ and
formula~(\ref{formuledintegration}) holds because
$$
\int_\Delta f^{(n)}(\langle x,0\rangle) \dd \Vol_{n}(x) = \Vol(\Delta)
f^{(n)}(0) =\sum_{k\geq 0} C_k(\Delta,0,\Delta) f^{(k)}(0),
$$ 
We prove~(\ref{formuledintegration}) by induction on the dimension
$n$. In case $n=0$, we have $u=0$ and so the verification reduces to the above one. 
Hence, we assume 
 $n\ge 1$ and $u\ne 0$. 
For short, we write $\dd x=\dd x_{1}\wedge\dots \wedge \dd
x_{n}$. 
Choose any vector $v\in \R^n$ of norm $1$ such that $\langle v,u\rangle\not=0$. Performing 
an orientation-preserving orthonormal change of variables, we may assume $v=(1,0,\dots,0)$. We have
$$ f^{(n)}(\langle x,u\rangle) \dd x= \frac{1}{\langle v,u\rangle}\dd
\big(f^{(n-1)}(\langle
x,u\rangle)\dd x_2 \wedge\dots\wedge\dd x_n \big).$$
With {Stokes}' theorem, we obtain
\begin{align} \label{eqpreuveC}
\int_\Delta f^{(n)}(\langle x,u\rangle)\dd \Vol_{n}(x) &= 
\int_\Delta f^{(n)}(\langle x,u\rangle)\dd x \\
&= \frac{1}{\langle v,u\rangle}\sum_F \int_F f^{(n-1)}(\langle x,u\rangle)\dd x_2\wedge\dots\wedge\dd x_n.
\nonumber \end{align}
where the sum is over the facets $F$ of $\Delta$, and we equip each
facet with the induced  orientation.

For each facet $F$ of $\Delta$, we let 
$\iota_{u_{F}}(\dd x)$ be the differential form of order $n-1$
obtained by contracting $\dd x$ with the vector $u_{F}$.
The form  $\dd x_2\wedge\dots\wedge\dd x_n$ is invariant under
translations and its restriction 
to the  linear hyperplane $L_{F}$ coincides with  $ \langle
v,u_F\rangle \iota_{u_{F}}(\dd x)$. 
Therefore,
\begin{equation*}
 \int_F f^{(n-1)}(\langle x,u\rangle)\dd x_2\wedge\dots\wedge\dd x_n = 
\langle v,u_F\rangle  \int_{F-m_{F}} f^{(n-1)}(\langle x+m_{F},u\rangle) \iota_{u_{F}}(\dd
 x).
\end{equation*}
Let  $\Vol_{n-1}$ denote the Lebesgue measure on $L_{F}$. 
We can verify that  $\Vol_{n-1}$ coincides with the measure
associated to the differential form  $-\iota_{u_{F}}(\dd x)\mid_{L_{F}}$ and
the orientation of $L_{F}$ induced by $u_{F}$.
Let  $g\colon \R\to \R$ be the function defined as
$g(z) =f(z+\langle m_{F},u\rangle)$. Then $ f^{(n-1)}(\langle
x+m_{F},u\rangle)= g^{(n-1)}(\langle x,\pi_F(u)\rangle)$ for all $x\in
L_{F}$. Hence, 
\begin{displaymath}
 \int_{F-m_{F}} f^{(n-1)}(\langle x+m_{F},u\rangle) \iota_{u_{F}}(\dd
 x)
= - \int_{F-m_{F}} g^{(n-1)}(\langle x,\pi_{F}(u)\rangle) \dd \Vol_{n-1}(x).
\end{displaymath}
Applying the inductive hypothesis to $F$ and the function 
$g$ we obtain
\begin{align*}
\int_{F} g^{(n-1)}(\langle x,\pi_F(u)\rangle)\dd \Vol_{n-1}(x) &= \sum_{V'\in F(\pi_F(u))} \sum_{k\geq 0}
C_k(F,\pi_F(u),V') g^{(k)}(\langle V',\pi_F(u)\rangle) \\
&= \sum_{V'\in F(\pi_F(u))} \sum_{k\ge 0}
C_k(F,\pi_F(u),V') f^{(k)}(\langle V',u\rangle).
\end{align*}
Each aggregate $V'\in F(\pi_F(u))$ is contained in a unique
$V\in\Delta(u)$ and it coincides with $V\cap F$. 
Therefore, we can transform the right-hand side of the last equality in
$$
\sum_{V\in \Delta(u)} \sum_{k\ge 0}
C_k(F,\pi_F(u),V\cap F) f^{(k)}(\langle V,u\rangle) ,
$$
where, for simplicity, we have set
$C_k(F,\pi_F(u),V\cap F)=0$ whenever $V\cap F=\emptyset$.  
Plugging the resulting expression into~(\ref{eqpreuveC}) and exchanging the
summations on  $V$ and $F$,
 we obtain that $\int_\Delta f^{(n)}(\langle x,u\rangle)\dd \Vol_{n}(x) $ is equal to
\begin{equation}\label{eq:110}
\sum_{V\in \Delta(u)} \sum_{k\ge 0} \bigg( -\sum_{F} \frac{\langle v,u_{F}\rangle}{\langle v,u\rangle}
C_k(F,\pi_F(u),V\cap F) f^{(k)}(\langle V,u\rangle)\bigg).
\end{equation}
Specializing this identity to $v=u$, we readily derive 
formula~(\ref{formuledintegration}) from Definition~\ref{def:63} of the
coefficients $C_{k}(\Delta,u,V)$. 

For the last statement, 
observe that the values $f^{(k)}(\langle V,u\rangle)$ can be
arbitrarily chosen.
Hence, the  coefficients
$C_k(\Delta,u,V)$ are uniquely determined from the linear system
obtained from the identity~(\ref{formuledintegration}) 
for enough functions $f$.
\end{proof}

\begin{cor}\label{cor:17}
Let $\Delta\subset \R^n$ be a polytope of dimension $n$ and $u\in
\R^n$. Then, 
\begin{equation*}
\sum_{V\in \Delta (u)}\sum_{k=0}^{\min\{i,\dim(V)\}} C_k(\Delta,u,V) \frac{\langle V,u\rangle^{i-k}}{(i-k)!} =
\begin{cases}
0 &\mbox{for } i=0,\dots,n-1,\\
\Vol_n(\Delta) &\mbox{for } i=n.
\end{cases}
\end{equation*} 
\end{cor}

\begin{proof}
This follows from formula~(\ref{formuledintegration}) 
applied to the function $f(z)= z^{i}/i!$.
\end{proof}

\begin{prop}\label{prop:68}
Let $\Delta\subset \R^n$ be a polytope of dimension $n$ and $u\in
\R^n$. Let $V\in \Delta(u)$ and $k\ge 0$.  
\begin{enumerate}
\item \label{item:82} 
The coefficient  $C_k(\Delta,u,V)$ is homogeneous of weight $k-n$ in the sense
that, for $\lambda\in\R^{\times}$,
\begin{displaymath}
C_k(\Delta,\lambda u,V)=\lambda^{k-n} C_k(\Delta,u,V).  
\end{displaymath}
\item \label{item:88} The coefficients $C_k(\Delta,u,V)$
  satisfy the vector relation
\begin{equation}\label{eqpourCi}
C_k(\Delta,u,V) \cdot u = -\sum_{F} 
C_k(F,\pi_F(u), V\cap F) \cdot u_F,
\end{equation}
where the sum is over the facets $F$ of $\Delta$.
\item \label{item:86}
Let $\Delta_1,\Delta_2\subset \R^n$ be two polytopes of dimension $n$
intersecting along a common facet  and  such that
$\Delta=\Delta_1\cup\Delta_2$. 
Then $V\cap \Delta_{i}=\emptyset$ or $V\cap \Delta_{i}\in \Delta_{i}(u)$ and 
\begin{equation*}
C_k(\Delta,u,V) = C_k(\Delta_1,u,V\cap \Delta_{1}) + C_k(\Delta_2,u,V\cap \Delta_{2}).
\end{equation*}  
\end{enumerate}
\end{prop}

\begin{proof}
Statement~\eqref{item:82} follows easily from the definition of $C_{k}(\Delta,u,V)$.
For statement \eqref{item:88}, we use that, from~\eqref{eq:110}, 
the integral formula in Proposition \ref{propdefcoeffdintegration}
also holds for the choice of coefficients
\begin{displaymath}
-\sum_{F} \frac{\langle v,u_{F}\rangle}{\langle v,u\rangle}
C_k(F,\pi_F(u),V\cap F)   
\end{displaymath}
for any vector  $v$ of norm 1 such that $\langle v,u\rangle \ne 0$.
But the coefficients satisfying that formula are unique. Hence, this
choice necessarily coincides with $C_k(\Delta,u,V)$ for all such $v$.  
Hence,
\begin{displaymath}
\langle v,u\rangle C_k(\Delta,u,V)=
-\sum_{F} \langle v,u_{F}\rangle
C_k(F,\pi_F(u),V\cap F)   
\end{displaymath}
and formula~(\ref{eqpourCi}) follows. 
Statement \eqref{item:86} follows from formula~(\ref{formuledintegration}) applied to
$\Delta$, $\Delta_1$ and $\Delta_2$ together with the additivity of
the integral and the fact that the coefficients $C_k(\Delta,u,V)$ are
uniquely determined.
\end{proof}

In case $\Delta$ is a simplex, the linear system given by Corollary
\ref{cor:17} has as many unknowns as equations.  In this case, the
coefficients corresponding to an aggregate in a given direction
are determined by this linear system. The following result gives a
closed formula for those coefficients.

\begin{prop}\label{prop:103}
Let $\Delta\subset \R^{n}$ be simplex and $u\in \R^{n}$. 
Write  $d_{W}=\dim(W)$ for $W\in \Delta(u)$. Then, for $V\in \Delta(u)$ and $0\le k\le \dim(V)$, 
\begin{equation*}
C_k(\Delta, u,V) = (-1)^{d_{V}-k} \frac{n!}{k!}
\Vol_n(\Delta) 
\sum_{\substack{\eta\in \N^{\Delta(u)\setminus \{V\}}\\ |\eta|=d_{V}-k}} \
\prod_{W\in \Delta(u)\setminus \{V\}}\frac{{d_{W}+\eta_{W}\choose
    d_{W}}}{\langle V-W ,u\rangle^{d_{W}+\eta_{W}+1}}\ .
\end{equation*}
\end{prop}

\begin{proof}
 Consider the Hermite interpolation polynomial
  $p_{V,k}\in \R[t]$ of degree $n$ characterized by the conditions
  that, for $W\in \Delta(u)$ and $l=0,\dots, d_{W}$,
\begin{displaymath}
  p_{V,k}^{(l)}(\langle u,W\rangle)=
  \begin{cases}
    l! & \text{ if } W=V \text{ and } l=k,\\
0 & \text{ otherwise.} 
  \end{cases}
\end{displaymath}
  
By Proposition \ref{propdefcoeffdintegration} and the choice of~$p_{V,k}$, 
\begin{displaymath}
\int_{\Delta}   p_{V,k}^{(n)}(\langle u,x\rangle)\dd \Vol_{n}(x)=  k!\,
C_k(\Delta, u,V). 
\end{displaymath}
Furthermore, $\int_{\Delta}   p_{V,k}^{(n)}(\langle u,x\rangle)\dd
\Vol_{n}(x)=  n!  \Vol_{n}(\Delta)
\coeff_{t^{n}}(p_{V,k})$, where
$\coeff_{t^{n}}(p_{V,k}) $ denotes the leading coefficient of
$p_{V,k}$. 

An explicit formula for $p_{V,k}$ can be found, for instance, in
\cite[Proposition~2.3]{DAndreaKrickSzanto:smr}. From that formula, we
deduce that
\begin{displaymath}
  \coeff_{t^{n}}(p_{V,k})= (-1)^{d_{V}-k} 
\sum_{\substack{\eta\in \N^{\Delta(u)\setminus \{V\}}\\ |\eta|=d_{V}-k}} \
\prod_{W\in \Delta(u)\setminus \{V\}}\frac{{d_{W}+\eta_{W}\choose
    d_{W}}}{\langle V-W ,u\rangle^{d_{W}+\eta_{W}+1}}\ , 
\end{displaymath}
which concludes the proof.
\end{proof}

\begin{rem}\label{rem:27} We can rewrite the formula in Proposition
  \ref{prop:103} in terms of vertices instead of aggregates as
  follows:
  \begin{equation}\label{eq:29}
    C_k(\Delta, u,V) = (-1)^{d_{V}-k} \frac{n!}{k!}
    {\Vol}_n(\Delta)
    \sum_{|\beta|=d_{V}-k}\prod_{\nu \notin V}\langle V-\nu ,u\rangle^{-\beta_\nu -1},
  \end{equation} 
  where the product is over the vertices  $\nu$ of  $\Delta$ not lying
  in $V$ and the sum is over the tuples 
  $\beta$ of non negative integers of length $d_{V}-k$, indexed by
  those same vertices of $\Delta$ that are not in $V$, that is,
  $\beta\in\N^{n-d_V}$ and $|\beta|=d_{V}-k$.
\end{rem}

\begin{exmpl}\label{integrationsimplexe}
  Let $\Delta\subset \R^{n}$ be a simplex and $u\in \R^{n}$. If 
  a vertex $\nu_{0}$ of $\Delta$  is an aggregate in the
  direction of $u$, then formula \eqref{eq:29}  reduces to
\begin{equation}\label{formulecassommetsimplex}
C_0(\Delta, u,\nu _0) = n!
\Vol_n(\Delta)\prod_{\nu \ne \nu _0}\langle \nu _0-\nu ,u\rangle^{-1},
\end{equation} 
where the product runs over all vertices of $\Delta $ different from
$\nu_{0}$. 
Suppose that the simplex is presented as the intersection of $n+1$
halfspaces as
\begin{displaymath}
  \Delta = \bigcap_{i=0}^n\{ x \in \R^n | \, \langle x,u_i\rangle
-\lambda_i\ge0\}
\end{displaymath}
with $u_i\in \R^n\setminus \{0\}$ and $\lambda_i\in\R$. Up to a
reordering, we can assume that $u_{0}$ is an inner normal vector to
the unique face of $\Delta$ not containing~$\nu_{0}$. We denote by
$\varepsilon$ the sign of $(-1)^{n}\det(u_{1},\dots,u_{n})$.  Then the
above coefficient can be alternatively written as
\begin{equation*}
C_0(\Delta, u,\nu _0) = \frac{\varepsilon
  \det(u_1,\dots,u_n)^{n-1}}{\prod_{i=1}^n \det(u_1,\dots,
  u_{i-1},u,u_{i+1},\dots, u_n)}.
\end{equation*}
\end{exmpl}

From the equation (\ref{formulecassommetsimplex}), we obtain the following extension of  Brion's ``short
formula'' 
\index{Brion's ``short formula''}%
for the case of a
simplex~\cite[Th\'eor\`eme~3.2]{Brion:PointsEntiers}, see
also~\cite{BBLKV08}.

\begin{cor} \label{cor:16}
Let $\Delta\subset \R^{n}$ be a simplex of dimension $n$ that is the
convex hull of points $\nu_{i}$,
$i=0,\dots, n$,  and let $u\in \R^n$ such that $\langle
\nu_{i},u\rangle \ne \langle 
\nu_{j},u\rangle$ for $i\ne j$.
Then, for any $f\in \mathscr{C}^{n}(\R)$, 
\begin{align*}
\int_\Delta f^{(n)}(\langle x,u\rangle)\dd \Vol_{n}(x) =  n!\Vol_n(\Delta)\sum_{i=0}^{n}
\frac{f(\langle \nu_i,u\rangle )}{\prod_{j\ne i}\langle \nu_i-\nu_j ,u\rangle}.
\end{align*}
\end{cor}

\begin{proof}
  This follows from Proposition \ref{propdefcoeffdintegration} and
the equation \eqref{formulecassommetsimplex}.
\end{proof}

In the next section, we will have to compute integrals over a polytope
of functions of 
the form $\ell(x)\log(\ell(x))$ where $\ell$ is an affine
function. The following result gives the value of such integral for
the case of a simplex. 

\begin{prop} \label{intsimplexepourpotgui}
Let $\Delta\subset \R^{n}$ be a simplex of dimension $n$ and
let $\ell\colon \R^{n}\to \R$ be an affine function which is non-negative on
$\Delta$. Write $\ell(x)=\langle x,u\rangle-\lambda $ for some vector
$u$ and constant $\lambda $.
Then  $\displaystyle \frac{1}{\Vol_n(\Delta)}\int_\Delta
\ell(x)\log(\ell(x))\dd \Vol_{n}(x)$ equals
\begin{equation}\label{integrationsimplexepotgui} 
  \sum_{V\in\Delta(u)}
  \sum_{\beta'} \binom{n}{n-|\beta'|} \frac{\ell(V)\left(\log(\ell(V)) -
      \sum_{j=2}^{|\beta'|+1}\frac{1}{j}\right)}{(|\beta'|+1)\prod_{\nu
      \notin V}\left(-\big(\frac{\ell(\nu
        )}{\ell(V)}-1\big)^{\beta'_\nu }
    \right)},
\end{equation}
where the second sum runs over $\beta '\in(\N^\times)^{n-\dim(V)}$ with
$ |\beta'|\leq n$ and the product is over the $n-\dim(V)$ vertices
$\nu $ of $\Delta$
not in $V$. 

If $\ell(x)$ is the defining equation of a hyperplane containing a
facet $F$ of $\Delta$, then
\begin{equation}\label{integrationsimplexepotguifac}
\frac{1}{\Vol_n(\Delta)}\int_\Delta \ell(x)\log(\ell(x))\dd x = \frac{\ell(\nu_{F} )}{n+1}\bigg(\log(\ell(\nu_{F} )) - \sum_{j=2}^{n+1}\frac{1}{j}\bigg),
\end{equation}
where $\nu_{F} $  denotes the unique vertex of $\Delta$ not contained
in $F$.
\end{prop}

\begin{proof}
This follows from the formulae~(\ref{formuledintegration})
and (\ref{eq:29}) with the function
$f^{(n)}(z)=(z-\lambda)\log(z-\lambda)$, a $(n-k)$-th primitive of which is
$$f^{(k)}(z) =
\frac{(z-\lambda)^{n-k+1}}{(n-k+1)!}\left(\log(z-\lambda) -
  \sum_{j=2}^{n-k+1}\frac{1}{j}\right).
$$  
\end{proof}

We end this section with a lemma specific to integration on the standard simplex.

\begin{lem}\label{lemmesemimonomail}
Let $\Delta^{r}$ be the standard simplex of $\R^{r}$ and  $\beta=(\beta_0,\dots,\beta_{r-1})\in\N^r$.
Let $f\in \mathscr{C}^{|\beta|+r}([0,1])$ where  $|\beta|
=\beta_0+\dots+\beta_{r-1}$.
For $(w_{1},\dots, w_{r})\in \Delta^{r}$ write $w_0 =1-w_1-\dots-w_r$. 
Then
$$\int_{\Delta^r}\bigg(
\prod_{i=0}^{r-1}\frac{w_i^{\beta_i}}{\beta_{i}!}\bigg)
f^{(|\beta|+r)}(w_r) \dd w_{1}\wedge \cdots\wedge \dd w_{r} = f(1)-\sum_{j=0}^{|\beta|+r-1}\frac{f^{(j)}(0)}{j!}.
$$
\end{lem}

\begin{proof}
We proceed by induction on $r$. Let $r=1$. Applying 
$\beta_{0}+1$ successive integrations by parts, the
integral computes as  
$$\sum_{j=0}^{\beta_{0}}\left[\frac{(1-w_1)^{j}}{j!}f^{(j)}(w_1)\right]_0^1
= f(1) - \sum_{j=0}^{\beta_{0}}\frac{f^{(j)}(0)}{j!},
$$
as stated. Let $r\ge 2$. Applying the case $r-1$ to the function $f(z)=\frac{z^{|\beta|+r-1}}{(|\beta|+r-1)!}$,
$$\frac{1}{\beta_0!\dots\beta_{r-1}!}\int_{\Delta^{r-1}}w_0^{\beta_0}w_1^{\beta_1}\dots
w_{r-1}^{\beta_{r-1}} \dd w_1\wedge\dots\wedge \dd w_{r-1} = \frac{1}{(|\beta|+r-1)!}
$$
and,  after rescaling, 
$$\frac{1}{\beta_0!\dots\beta_{r-1}!}\int_{(1-w_r)\Delta^{r-1}}w_0^{\beta_0}w_1^{\beta_1}\dots
w_{r-1}^{\beta_{r-1}} \dd w_1\wedge\dots\wedge \dd w_{r-1} = \frac{(1-w_r)^{|\beta|+r-1}}{(|\beta|+r-1)!}
 .$$
Therefore, the left-hand side of the equality to be proved reduces to
$$\frac{1}{(|\beta|+r-1)!}\int_0^1(1-w_r)^{|\beta|+r-1}f^{(|\beta|+r)}(w_r)\dd w_r.
$$
Applying the case $r=1$ and index $|\beta|+r-1\in
\N$, we find that this integral equals  $f(1) -
\sum_{j=0}^{|\beta|+r-1}f^{(j)}(0)/j!$, which concludes the proof.
\end{proof}

\begin{cor} \label{calculmonomeetmonomelog}
Let $\alpha\in\N^{r+1}$. For $(w_{1},\dots, w_{r})\in \Delta^{r}$, write $w_0 =1-w_1-\dots-w_r$. 
Then 
$$\int_{\Delta^r}w_0^{\alpha_0}w_1^{\alpha_1}\dots w_{r}^{\alpha_{r}}
\dd w_{1}\wedge \dots\wedge \dd w_{r} = \frac{\alpha_0!\dots\alpha_r!}{(|\alpha|+r)!}
$$
and, for $i=0,\dots,r$,
$$\int_{\Delta^r}w_0^{\alpha_0}w_1^{\alpha_1}\dots w_{r}^{\alpha_{r}}\log(w_i) \dd w_{1}\wedge \dots\wedge \dd w_{r} = -\frac{\alpha_0!\dots\alpha_r!}{(|\alpha|+r)!}\sum_{j=\alpha_i+1}^{|\alpha|+r}\frac{1}{j}.
$$
\end{cor}

\begin{proof}
The formula for the first integral follows from Lemma~\ref{lemmesemimonomail} applied
with 
 $\beta=(\alpha_0,\dots,\alpha_{r-1})$ and
 $f(z)=\frac{z^{|\alpha|+r}}{(|\alpha|+r)!}$.
The second one follows similarly,  applying
Lemma~\ref{lemmesemimonomail} to the  function
$f(z)=\frac{z^{|\alpha|+r}}{(|\alpha|+r)!}\left(\log(z)-\sum_{j=\alpha_i+1}^{|\alpha|+r}\frac{1}{j}\right)$,
after some possible permutation (for $i=1,\dots, r-1$)  or linear
change of variables (for $i=0$).
\end{proof}

\section{Metrics and heights from polytopes} \label{Metrics}

In this section we will consider some metrics 
arising from polytopes. 
We will use the notation of \S 4 and \S 5. In particular, we consider a 
split torus over the field of rational numbers $\T\simeq
\G_{m,\Q}^{n}$ and we denote 
by $N,M,N_{\R},M_{\R}$ the lattices and dual spaces corresponding to $\T$. 

Let $\Delta\subset M_\R$ be a lattice polytope of dimension $n$. Let
$\ell_i$, $i=1,\dots, r$, be affine
functions on $M_{\R}$ defined as $\ell_{i}(x) =\langle x,u_i\rangle-\lambda_i$ for some $u_{i}\in N_{\R}$ and $\lambda_{i}\in
\R$ such that $\ell_{i}\ge 0$ on $\Delta$ and let also $c_{i}>0$.
Write $\ell=(\ell_{1},\dots, \ell_{r})$ and
$c=(c_{1},\dots, c_{r})$.  
We consider the function $\vartheta_{\Delta,\ell,c}\colon\Delta\to \R $
defined, for $x\in \Delta$, by 
\begin{equation}\label{potentialfunction}
\vartheta_{\Delta,\ell,c}(x)= -\sum_{i=1}^rc_i\ell_i(x)\log(\ell_i(x)).
\end{equation}
When $\Delta,\ell,c$ are clear from the context, we write  for short
$\vartheta=\vartheta_{\Delta,\ell,c}$.

\begin{lem} \label{lemm:18} Let notation be as above.
  \begin{enumerate} 
  \item \label{item:90} The function $\vartheta_{\Delta,\ell,c}$ is concave. 
  \item \label{item:91} If the family  $\{u_{i}\}_{i}$ generates $N_{\R}$,
    then  $\vartheta_{\Delta,\ell,c}$ is strictly concave. 
  \item \label{item:92} If $\Delta=\bigcap_{i}\{ x\in M_{\R}| \ell_{i}(x)\geq0\}$, then the
    restriction of $\vartheta_{\Delta,\ell,c}$ to $\Delta^{\circ}$,
    the interior of the polytope, is of Legendre type
    (Definition~\ref{def:28}). 
  \end{enumerate}
\end{lem}

\begin{proof}
Let $1\le i\le r$ and consider the affine map $\ell_{i}\colon
\Delta\to \R_{\ge0}$. We have that  $-z\log(z)$ is a strictly concave
function on $\R_{\ge0}$ and  $-\ell_i\log(\ell_i)=
\ell_{i}^{*}(-z\log(z))$. Hence, each function $-c_i\ell_i(x)\log(\ell_i(x))$ is
concave and so is $\vartheta$, as stated in \eqref{item:90}.

For statement~\eqref{item:91}, let $x_{1},x_{2}$ be two different points
of $\Delta$. The assumption that~$\{u_{i}\}_{i}$ generates 
$N_{\R}$ implies that
$\ell_{i_{0}}(x_{1})\ne \ell_{i_{0}}(x_{2})$  for some $i_{0}$. Hence, 
the affine map $\ell_{i_{0}}$ gives an injection of the segment
$\ov{x_{1}x_{2}}$ into $\R_{\ge0}$. We deduce that 
$-c_{i_{0}}\ell_{i_{0}}\log(\ell_{i_{0}})$ is strictly concave on
$\ov{x_{1}x_{2}}$ and so is $\vartheta$. Varying $x_{1},x_{2}$, we
deduce that $\vartheta$ is strictly concave on $\Delta$.

For statement \eqref{item:92}, it is clear that $\vartheta|_{\Delta^{\circ}}$
is differentiable. Moreover, the assumption that $\Delta$ is the intersection of the halfspaces
defined by the $\ell_{i}$'s implies that the $u_{i}$'s generate
$N_{\R}$ and so $\vartheta$ is strictly concave.
The gradient of $\vartheta$ is given, for $x\in \Delta^{\circ}$, by
\begin{equation}\label{eq:113}
  \nabla \vartheta(x)= -\sum_{i=1}^{r}c_{i}u_{i}(\log (\ell_{i}(x))+1).
\end{equation}
Let $\|\cdot\|$ be a fixed norm on $M_{\R}$ and  $(x_{j})_{j\ge 0}$ a
sequence in $\Delta^{\circ}$ converging to a point in the border. Then there exists some $i_{1}$ such $
\ell_{i_{1}}(x_{j})\stackrel{j}{\to}0$. Thus, $\|  \nabla
\vartheta(x_{j})\|\stackrel{j}{\to} \infty$ and the  statement follows. 
\end{proof}

\begin{defn} \label{def:62}
Let $\Sigma_{\Delta}$ and $\Psi_{\Delta}$ be the fan and the support
function on $N_{\R}$ induced by $\Delta$. Let $(X_{\Sigma_{\Delta}},
D_{\Psi_{\Delta}})$ be the associated polarized toric variety over $\Q$ and
write $L=\cO(D_{\Psi_{\Delta}})$.
By Lemma \ref{lemm:18}\eqref{item:90},  $\vartheta_{\Delta,\ell,c} $
is a concave 
function on $\Delta$. By  Theorem~\ref{thm:13}, it corresponds
to some semipositive{} toric metric on $L(\C)$. We denote
this metric by $\|\cdot\|_{\Delta,\ell,c}$. 
We write $\ov L$ for the line bundle $L$ equipped with the 
metric $\|\cdot\|_{\Delta,\ell,c}$ at the Archimedean place of $\Q$
and with the canonical metric at the non-Archimedean places. This is
an example of an adelic toric metric.
\end{defn}

\begin{exmpl} \label{exm:31} Following the notation in Example
  \ref{exm:5}, consider the standard simplex $\Delta^{n}$ and the
  concave function $\vartheta=\frac12 \varepsilon_{n}$ on
  $\Delta^{n}$.  From examples \ref{exm:5}
  and~\ref{exm:37}\eqref{item:107}, we deduce that the corresponding 
  metric is the Fubini-Study metric on
  $\cO(1)^{\an}$ over $\C$. 
\end{exmpl}

In case $\Delta$ is the intersection of the halfspaces defined by the $\ell_{i}$'s,
Lemma \ref{lemm:18}\eqref{item:92} shows that $\vartheta|_{\Delta^{\circ}}$ of Legendre type
    (Definition~\ref{def:28}). 
By Theorem \ref{thm:2} and equation~\eqref{eq:113}, the gradient of
    $\vartheta$ gives a  homeomorphism between
    $\Delta^{\circ}$ and $N_{\R}$ and, for $x\in \Delta^{\circ}$,
    \begin{equation*}
\vartheta^{\vee}(\nabla \vartheta(x)) 
=  -\sum_{i=1}^r c_i\left(\lambda_i\log(\ell_i(x))+\langle x,u_i\rangle\right).
    \end{equation*}
This gives an explicit expression of the function
$\psiabs_{\|\cdot\|_{\Delta,\ell,c}}=\vartheta^{\vee}$, and {\it a fortiori} of the
metric ${\|\cdot\|_{\Delta,\ell,c}}$, in the coordinates of the
polytope. 
Up to our knowledge, there is  no simple expression for $\psiabs$ in
linear coordinates of $N_{\R}$, except for special cases like Fubini-Study.

\begin{rem} \label{rem:14}
This kind of metrics are interesting when studying the K\"ahler
geometry of toric varieties.
Given a Delzant polytope $\Delta\subset M_{\R}$ (Remark \ref{rem:12}),
Guillemin has
constructed a ``canonical'' K\"ahler structure on the associated
symplectic toric variety~\cite{Gui94}. 
The corresponding  symplectic potential is the function
$-\vartheta_{\Delta,\ell, c}$, for the case when  $r$ is the number of facets of $\Delta$,  $c_{i}=1/2$ for
all $i$, and  
$u_{i}$ is a primitive vector in $N$ and $\lambda_{i}$ is an integer such that 
 $\Delta=\{x\in
M_{\R}| \langle x,u_{i}\rangle \ge \lambda_{i}, i=1,\dots, r\}$, see~\cite[Appendix 2, (3.9)]{Gui94}.  

In this case, the metric $\|\cdot \|_{\Delta,\ell,c}$ on the line
bundle  $\cO(D_{\Psi})^{\an}$ is smooth
and positive and, as explained in Remark \ref{rem:12}, its 
Chern form gives this canonical K\"ahler form. 
\end{rem}

We obtain the following formula for the height of
$X_{\Sigma_{\Delta}}$ with respect to the line bundle with adelic
toric metric
$\overline{L}$, in terms of the coefficients $C_{k}(\Delta,u_{i},V)$. 

\begin{prop}\label{prop:69}
  Let notation be as in Definition \ref{def:62}. 
  Then $\h_{\overline{L}}(X_{\Sigma_{\Delta}})$ equals
  $$ {(n+1)!} \sum_{i=1}^{r}c_i\sum_{V\in\Delta(u_i)}
  \sum_{k=0}^{\dim(V)}C_k(\Delta,u_i,V)\frac{\ell_i(V)^{n-k+1}}{(n-k+1)!}
  \left(\sum_{j=2}^{n-k+1}\frac{1}{j}-\log(\ell_i(V))\right).
  $$
  Suppose furthermore that $\Delta\subset\R^n$ is a simplex, $r=n+1$ and 
  that  $\ell_i$, 
  $i=1, \dots, n+1$, are affine functions such that $\Delta=
  \bigcap_{i}\{x\in M_{\R}| \ell_{i}(x)\ge0\}$.
  Then
  \begin{equation}\label{eq:115}
    \h_{\overline{L}}(X_{\Sigma_{\Delta}})= n!\Vol_{M}(\Delta)
    \sum_{i=1}^{n+1}c_i\ell_i(\nu_i) \bigg(
    \sum_{j=2}^{n+1}\frac{1}{j}-\log(\ell_i(\nu_i))\bigg), 
  \end{equation}
  where $\nu_{i}$ is the unique vertex of $\Delta$ not contained in the
  facet defined by $\ell_{i}$.
\end{prop}

\begin{proof}
  The first statement follows readily from Theorem \ref{thm:22} and
  Proposition~\ref{propdefcoeffdintegration} applied to the functions
  $f_{i}(z)=\left(\log(z-\lambda_i)-\sum_{j=2}^{n+1}\frac{1}{j}\right)(z-\lambda_i)^{n+1}/(n+1)!$.
  The second statement follows similarly from Proposition
  \ref{intsimplexepourpotgui}.
\end{proof}

\begin{exmpl}\label{exm:32}
Let $\cO(1)$ be the universal line bundle of $\P^{n}$. As we have seen
in Example \ref{exm:31}, the Fubini-Study metric of $\cO(1)^{\an}$ 
corresponds to the case of the standard simplex, $\ell_{i}(x)=x_{i}$,
$i=1,\dots, n$, and $\ell_{n+1}(x)=1-\sum_{i=1}^{n}x_{i}$ and the
choice 
$c_i= 1/2$ for all $i$.
Hence we recover from \eqref{eq:115} the well known expression for the
height of $\P^{n}$ with respect to the Fubini-Study metric in
\cite[Lemma~3.3.1]{BostGilletSoule:HpvpGf}: 
\begin{displaymath}
  \h_{\ov{\cO(1)}}(\P^{n})= \frac{n+1}{2}\sum_{j=2}^{n+1}\frac1j=\sum_{h=1}^n\sum_{j=1}^h\frac{1}{2j}.
\end{displaymath}
\end{exmpl}

\begin{exmpl}\label{exm:33}
In dimension $1$, a polytope is an interval of the form
$\Delta=[m_{0},m_{1}]$ for some $m_{i}\in \Z$.
The corresponding roof function in \eqref{potentialfunction} writes down, 
for $x\in[m_0,m_1]$, as 
\begin{equation*}
\vartheta(x) = -\sum_{i=1}^rc_i\ell_i(x)\log(\ell_i(x))
\end{equation*}
for affine function $\ell_i(x)=u_ix-\lambda_i$ which take non negative
values on $\Delta$ and $c_{i}>0$

The polarized toric variety corresponding to $\Delta$ is $\P^1$ 
together with the ample divisor $m_1[(0:1)]-m_0[(1:0)]$. 
Write $L={\mathcal O}_{\P^{1}}(x_1-x_0)$ for the associate line bundle and $\ov
L$ for the line bundle with adelic toric metric corresponding to the
function~$\vartheta$.
The { Legendre-Fenchel} dual to
$-c_i\ell_i(x)\log(\ell_i(x))$ is the function $f_i\colon\R\rightarrow
\R$  defined, for $v\in \R$,  by
$$f_i(v) = \frac{\lambda_i}{u_i}v -c_i{\e}^{-1-\frac{v}{{c_iu_i}}}.
$$ 
Therefore, the function $\psiabs=\vartheta^{\vee}$
is the sup-convolution of these function, namely 
$\psiabs=f_1\boxplus\dots\boxplus f_m$
For the height, a simple computation shows that 
$$
{\h_{\overline{L}}(\P^1)} = 2\int_{m_{0}}^{m_{1}} \vartheta \dd x=
\sum_{i=1}^r
\frac{c_i}{2u_i}\Big[\ell_i(x)^2
\left(1-2\log(\ell_i(x))\right)\Big]^{m_{1}}_{m_{0}}.
$$
\end{exmpl}

\section{Heights and entropy}
\label{sec:height-entropy}

In some cases, the height of a toric variety with respect to the
metrics constructed in the previous section
has an interpretation in terms of the average entropy of a family of
random processes.

Let $\Gamma$ be an arbitrary polytope containing $\Delta$. 
For a point $x \in \ri(\Delta)$, we consider the partition $\Pi_x$ of
$\Gamma$ which consists of the cones $\eta_{x,F}$ of vertex $x$ and base the
relative interior of each proper face $F$ of $\Gamma$.
We consider $\Gamma$  as a probability space endowed with the 
uniform probability distribution. Let $\beta_x$ be the random variable 
that maps a
point $y \in\Gamma$ to the base $F$ of the unique cone $\eta_{x,F}$ 
that contains $y$. 
Clearly, the probability that a given face $F$ is
returned is the ratio of the volume of the 
cone based on $F$ to the volume of $\Gamma$. 
We have 
$\Vol_n(\eta_{x,F})={n}^{-1}{\dist}(x,F)\Vol_{n-1}(F)$ where, as
before, $\Vol_{n}$ and $\Vol_{n-1}$ denote the Lebesgue measure on
$\R^{n}$ and on $L_{F}$, respectively. 
Hence, 
\begin{equation}\label{eq:116}
  P(\beta_x=F)=
  \begin{cases}
\displaystyle \frac{{\dist}(x,F)\Vol_{n-1}(F)}{n\Vol_{n}(\Gamma)} & \text{ if }
\dim(F) =n-1,\\
    0 & \text{ if } \dim(F)\le n-2.
  \end{cases}
\end{equation}
\index{entropy}%
The {entropy} of the random variable $\beta_x$ is
\begin{displaymath}
\cE(x) = -\sum_{F} P(\beta_x=F)\log(P(\beta_x=F)),
\end{displaymath}
where the sum is over the facets $F$ of $\Gamma$.

For each facet $F$ of $\Gamma$ we let $u'_{F}\in \R^{n}$ be the inner normal
vector to $F$ of Euclidean norm $(n-1)!\Vol_{n-1}(F)$. Therefore
$u'_{F}=(n-1)!\Vol_{n-1}(F) u_{F}$, with $u_{F}$ as
in \S \ref{Intepol}. Set
\begin{displaymath}
  \lambda({F})=\Psi_{\Gamma}(u'_{F})=(n-1)!\Vol_{n-1}(F)\Psi_{\Gamma}(u_{F}).
\end{displaymath}
Consider the affine
form defined as $\ell_{F}(x)=\langle x,u'_{F}\rangle
-\lambda({F})$, so that
\begin{displaymath}
  \Gamma=\{x\in M_{\R}| \ell_{F}(x)\ge 0,\ \forall F\}.
\end{displaymath}
Set 
\begin{math}
\lambda(\Gamma)=\sum_{F}\lambda(F),  
\end{math}
where the sum is over the facets $F$ of $\Gamma$. 
Since, by \cite[Lemma~5.1.1]{Sch93}, the vectors $u'_{F}$ satisfy the
Minkowski condition 
$\sum_{F}u'_{F}=0$, we deduce that
\begin{displaymath}
 \sum_{F}\ell_{F}=-\sum_{F}\lambda({F})=-\lambda(\Gamma).   
\end{displaymath}


Let $c>0$ be a real number.
The concave function 
\begin{math}
  \vartheta(x)= -\sum_{F} {c\, \ell_{F}(x)}\log ({\ell_{F}(x)})
\end{math}
belongs to the class of functions considered in Definition
\ref{def:62}. Thus, we obtain a line bundle with an adelic
toric metric $\ov L$ on $X_{\Delta }$. For short, we write
$X=X_{\Delta }$.   

The following result shows that the average entropy of the random
variable $\beta_{x}$ with respect to the uniform distribution on
$\Delta$ can be expressed in terms of the height of the toric variety
$X$ with respect to $\ov L$. 

\begin{prop} \label{prop:71}
With the above notation, 
\begin{displaymath}
\frac{1}{\Vol_n(\Delta)}\int_{\Delta}\cE\dd\Vol_n = 
\frac{1}{n!\Vol_n(\Gamma)} \bigg( \frac{\h_{\ov L}(X)}{c(n+1) \deg_{L}(X)}-
\lambda(\Gamma){\log(n!\Vol_n(\Gamma))}
\bigg).
\end{displaymath}
In particular, if $\Gamma=\Delta$, 
\begin{displaymath}
\frac{1}{\Vol_n(\Delta)}\int_{\Delta}\cE\dd\Vol_n = 
\frac{\h_{\ov L}(X)}{c(n+1) \deg_{L}(X)^{2}}-
\lambda(\Gamma)\frac{\log(\deg_{L}(X))}{ \deg_{L}(X)}.
\end{displaymath}
\end{prop}

\begin{proof}
  For $x\in \ri(\Delta)$ and $F$ a facet of $\Gamma$, we deduce from
the equation  \eqref{eq:116} that 
$P(\beta_x=F)= \ell_{F}(x)/(n!\Vol_n(\Gamma))$. Hence, 
\begin{align*}
  \cE(x) &= -\sum_{F} \frac{\ell_{F}(x)}{n!\Vol_n(\Gamma)}
\log \Big(\frac{\ell_{F}(x)}{n!\Vol_n(\Gamma)}\Big)\\
&= \frac{1}{n!\Vol_n(\Gamma)}
\bigg( -\sum_{F} {\ell_{F}(x)}\log ({\ell_{F}(x)}) 
-\lambda(\Gamma)\log ({n!\Vol_n(\Gamma)})
\bigg)\\
&= \frac{1}{n!\Vol_n(\Gamma)}
\bigg( \frac{\vartheta(x)}{c}
-\lambda(\Gamma)\log ({n!\Vol_n(\Gamma)})
\bigg).
\end{align*}
The result then follows from Theorem \ref{thm:22} and (\ref{eq:3}).
\end{proof}

\begin{exmpl} \label{exm:34}
The Fubini-Study metric of $\cO(1)^{\an}$ corresponds to the case
when $\Gamma$ and $\Delta$ are the standard simplex $\Delta^{n}$ and
$c=1/2$. In that case, the average entropy of the random variable
$\beta_{x}$ is 
  \begin{displaymath}
 \frac{1}{n!}\int_{\Delta^{n}}\cE \dd\Vol_n =    \frac{2\h_{\ov{\cO(1)}}(\P^{n})}{(n+1)}=\sum_{j=2}^{n+1}\frac1j.
  \end{displaymath}
\end{exmpl}


\chapter{Variations on Fubini-Study metrics} \label{courbes}

\section{Height of toric projective curves}\label{Projective toric schemes}

In this chapter, we study the Arakelov invariants of curves which are
the image of an equivariant map into a projective space. In the
Archimedean case we equip the projective space with
the Fubini-Study metric, while in the non-Archimedean case we equip it
with the canonical metric. For each of these curves, the metric, measure
and toric local height can be computed in terms of the roots of a
univariate polynomial associated to the relevant equivariant map. 

\medskip Let $K$ be either $\R, \C$ or a complete field with respect
to an absolute value associated to a nontrivial discrete
valuation. On $\P^{r}$, we consider the universal line bundle $\cO(1)$
equipped with the Fubini-Study metric in the Archimedean case, and
with the canonical metric in the non-Archimedean case. We write
$\ov{\cO(1)}$ for the resulting metrized line bundle. We also consider
the toric section $s_{\infty}$ of $\cO(1)$ whose Weil divisor is the
hyperplane at infinity.  The next result gives the function
$\psiabs_{\|\cdot\|}$ associated to the induced metric
on a subvariety of~$\P^{r}$ which is the image of an equivariant map.

\begin{prop}\label{prop:86}
  Let $H\colon N\to \Z^{r}$ be an injective map such that $H(N)$ is a
  saturated sublattice of $\Z^{r}$, and $p\in \P^{r}_{0}(K)$.
  Consider the map $\varphi_{H,p}\colon\T\to \P^{r}$, set $\ov L=
  \varphi_{H,p}^{*}\ov{\cO(1)}$ and $s=\varphi_{H,p}^{*}s_{\infty}$,
  and let $\psiabs_{\ov L,s}\colon N_{\R}\to \R$ be the associated
  concave function. Let $e_{i}^{\vee}$ be the $i$th vector in the dual
  standard basis of $\Z^{r}$ and set $m_{i}=e_{i}^{\vee}\circ H\in M$,
  $i=1,\dots,r$, and $p=(1:p_{1}:\dots:p_{r})$ with $p_{i}\in
  K^{\times}$. Then, for $u\in N_{\R}$,
\begin{displaymath}
\psiabs_{\ov L,s}(u)=
\begin{cases}
  -\frac{1}{2} \log(1+\sum_{i=1}^{r} |p_{i}|^{2} \e^{-2\langle
    m_{i},u\rangle} ) & \text{in the Archimedean
    case},\\
  \min_{1\le i\le r}\{0, \langle m_{i},u\rangle +\val (p_{i})\} &
  \text{in the non-Archimedean case}.
\end{cases}
\end{displaymath}
\end{prop}

\begin{proof}
  In the Archimedean case, the expression for the concave function
  $\psiabs$ follows from that for $\P^{r}_{K}$ (Example
  \ref{exm:37}\eqref{item:108}) and Proposition \ref{prop:58}.  
  The non-Archimedean case follows from Example \ref{exm:35}.
\end{proof}

Let $Y\subset \P^{r}$ be the closure of the image of the map
$\varphi_{H,p}$. In the Archimedean case, the roof function seems difficult
to calculate. Hence it is difficult to use it directly to compute the
toric local height
(see Example \ref{exm:16}). A more promising approach is to apply the
formula of Corollary \ref{cor:21}. Writing $\psiabs =\psiabs _{\ov L,s}$
this formula reads
\begin{equation}\label{eq:134}
    \htor_{\ov L}(Y) = (n+1)! \, \int_{N_{\R}}
    \psiabs^{\vee}\circ \partial \psiabs 
    \dd \cM_{M}(\psiabs).
\end{equation}
To make this formula more explicit in the 
Archimedean case, we choose a basis of $N$, hence
   coordinate systems in $N_{\R}$ and $M_{\R}$ and we write
 \begin{displaymath}
   g=(g_{1},\dots,g_{n}):=\nabla \psiabs \colon N_{\R}\longrightarrow \Delta,
 \end{displaymath}
 where $\Delta =\Stab
 (\psiabs )$ is the associated polytope. Then, from 
 Proposition \ref{prop:5} and
Example \ref{exm:29}\eqref{item:96}, we derive
\begin{align*}
  \htor_{\ov L}(Y) &= (n+1)! \int_{N_{\R}}
  (\langle \nabla
\psiabs (u), u\rangle - \psiabs(u))\,(-1)^{n}\det(\Hess(\psiabs)) \,
\dd\Vol_{N}\notag \\
& = (n+1)! \int_{N_{\R}}
   (\left< g(u),u \right>-\psiabs (u))\,(-1)^{n}\dd g_{1}\land\dots \land \dd g_{n}.
\end{align*}

When $K$ is not Archimedean, we
have $ \cM_{M}(\psiabs)= \sum_{v\in
  \Pi(\psiabs)^{0}}\Vol_{M}(v^{*})\delta_{v}$ and, for 
$v\in \Pi(\psiabs)^{0}$,
$$ \psiabs^{\vee}\circ \partial
\psiabs(v)=\frac{1}{\Vol_{M}(v^{*})}\int_{v^{*}} \langle x,v\rangle \dd
\Vol_{M} - \psiabs(v),$$ see Proposition~\ref{prop:32} and Example
\ref{exm:29}\eqref{item:97}. Thus, if now we denote by $g\colon
N_{\R}\to M_{\R}$ the function that sends a point $u$ to the barycentre of
$\partial\psiabs (u)$, then   
\begin{equation*}
    \htor_{\ov L}(Y) = (n+1)! \, \sum_{v\in \Pi(\psiabs)^{0}} 
  (\left< g(v),v \right>-\psiabs (v)).
\end{equation*}

\index{toric curve}%
In the case of curves, the integral in \eqref{eq:134}, can be
transformed into another integral that will prove useful for explicit
computations. We introduce a notation for derivatives of concave
functions of one variable. Let
$f\colon \R\to \R$  be a concave function. For $u\in \R$, write
\begin{equation} \label{eq:133}
  f'(u)=\frac{1}{2}(D_{+}f(u)+D_{-}f(u)),
\end{equation} 
where $D_{+}f$ and $D_{-}f$ denote the right and left derivatives  of
$f$ respectively, that exist always \cite[Theorem 23.4]{Roc70}. Then $f'$ is monotone and is
continuous almost
everywhere (with respect to the Lebesgue measure).
The associated distribution agrees with the derivative of $f$ in
the sense of distributions. This implies that, if $(f_{n})_{n}$ is a sequence
of concave functions converging uniformly to $f$ on compacts, then $(f'_{n})_{n}$
converges to $f'$ almost everywhere.

\begin{lem}\label{lemm:7}
  Let $\psiabs\colon \R \to \R$ be a concave function whose stability set is
  an interval $[a,b]$ and $  \psiabs ^{\vee}\circ \partial \psiabs$ the
  $\cM_{\Z}(\psiabs )$-measurable function defined in (\ref{eq:121}). Then 
  \begin{displaymath}
    2\int_{\R}
  \psiabs ^{\vee}\circ \partial \psiabs \dd \cM_{\Z}(\psiabs )=  (b-a)(\psiabs
  ^{\vee}(a)+\psiabs ^{\vee}(b)) +\int_{\R}(\psiabs '(u)-a)(b-\psiabs '(u))\dd
  u. 
  \end{displaymath}
\end{lem}
\begin{proof} 
  We argue as in the proof of Theorem \ref{thm:20}. By the properties
  of the Monge-Amp\`ere measure (Proposition \ref{prop:87}) and of the
  Legendre-Fenchel dual (Proposition \ref{prop:3}), the left-hand side
  is continuous with respect to uniform convergence of
  functions. Again by Proposition \ref{prop:3} and the discussion
  before the lemma, the right-hand side is also continuous with
  respect to uniform convergence of functions. By the compacity of the
  stability set of $\psiabs$, Lemma \ref{lemm:20} implies that there is a
  sequence of strictly concave smooth 
  functions $(\psiabs_{n})_{n\ge 1}$ converging uniformly to $\psiabs$.
  Hence, it is enough to treat the case when $\psiabs $ is smooth and
  strictly concave.

Using Example \ref{exm:29}\eqref{item:96}, we obtain
  \begin{displaymath}
    \int_{\R}
    \psiabs ^{\vee}\circ \partial \psiabs \dd \cM_{\Z}(\psiabs )=\int_{\R}
    (\psiabs (u)-u\psiabs '(u))\psiabs ''(u)\dd u.
  \end{displaymath}
  Consider the function
  \begin{align*}
    \gamma (u)&=\Big(\psiabs '(u)-\frac{a+b}{2}\Big)\psiabs(u) -u
    \frac{(\psiabs 
      '(u))^{2}}{2}+u\frac{ab}{2}\\&=
    -\Big(\psiabs '(u)-\frac{a+b}{2}\Big)\psiabs ^{\vee}(\psiabs
    '(u))-\frac{u}{2}(\psiabs 
    '(u)-a)(b-\psiabs '(u)). 
  \end{align*}
  Then
  \begin{displaymath}
    \lim_{u\to \infty}\gamma (u)=\frac{b-a}{2}\psiabs ^{\vee}(a),\qquad
    \lim_{u\to -\infty}\gamma (u)=\frac{a-b}{2}\psiabs ^{\vee}(b),
  \end{displaymath}
  and
  \begin{displaymath}
    \dd \gamma =(\psiabs -u\psiabs ')\psiabs ''\dd
    u-\frac{1}{2}(\psiabs '-a)(b-\psiabs 
    ')\dd u,
  \end{displaymath}
  from which the result follows.
\end{proof}

\medskip
With the notation in Proposition \ref{prop:86}, assume that $N=\Z$. 
The elements $m_{j}\in N^{\vee}$ can be identified
with integer numbers and the hypothesis that the image of $H$ is a
saturated sublattice is equivalent to $\gcd(m_{1},\dots,m_{r})=1$. Moreover, by
reordering the variables of 
$\P^{r}$ and multiplying the expression of $\varphi_{H,p}$ by a
monomial (which does not change the equivariant map), we may assume
that $0\le m_{1}\le \dots \le m_{r}$. We make the further hypothesis
that $0< m_{1}< \dots <m_{r}$. With these conditions, we 
next obtain explicit expressions for the
concave function $\psiabs$ and the associated measure and toric local
height in terms of the roots of a univariate polynomial.  
We consider the absolute value $|\cdot|$
of the algebraic closure $\ov K$ extending the absolute value of $K$.

\begin{thm}\label{thm:26} Let $0< m_{1}<\dots <m_{r}$ be
  integers with $\gcd(m_{1},\dots ,m_{r})=1$, and
  $p_{1},\dots,p_{r}\in K^{\times}$. Let 
  $\varphi\colon \T\to \P^{r}$ be the map given by
  $\varphi(t)=(1:p_{1}t^{m_{1}}:\dots :p_{r}t^{m_{r}})$ and let $Y$ be
  the closure of the image of $\varphi$. 
  Consider the polynomial $q\in K[z]$
  defined as
  \begin{align*}
    q=
    \begin{cases}
      1+\sum_{j=1}^{r}|p_{j}|^{2}z^{m_{j}} & \text{ in the
        Archimedean case},\\
      1+\sum_{j=1}^{r}p_{j}z^{m_{j}} & \text{ in the non-Archimedean
        case}.
    \end{cases}
  \end{align*}
  Let $\{\xi_{i}\}_{i}\subset \ov K^{\times}$ be the set of roots of
  $q$ and, for each $i$, let $\ell_{i}\in \N$ be the multiplicity
  of $\xi_{i}$. Let $\ov L$ and $s$ be as in Proposition
  \ref{prop:86}. 
Then, in the Archimedean case,
  \begin{enumerate}
\item \label{item:100} $\displaystyle \psiabs_{\ov L,s}(u)=-\log|p_{r}|
  -\frac{1}{2}\sum_{i}\ell_{i}\log|\e^{-2u}-\xi_{i}|$ for $u\in \R$,
  \item \label{item:99} $\displaystyle \cM_{\Z}(\psiabs_{\ov L,s})=
    -2\sum_{i}\ell_{i}\frac{\xi_{i} \e^{2u}}{(1-\xi_{i} \e^{2u})^{2}}
    \, \dd u$,
  \item \label{item:109} $\displaystyle \htor_{\ov L}(Y)=
    m_{r}\log|p_{r}| + \frac{1}{2}\sum_{i}\ell_{i}^{2}+
    \frac{1}{2}\sum_{i<j}\ell_{i}\ell_{j}
    \frac{\xi_{i}+\xi_{j}}{\xi_{i}-\xi_{j}}(\log(-\xi_{i}) 
    - \log(-\xi_{j}))$,
    where $\log $ is the principal determination of the logarithm.
\end{enumerate}
While in the non-Archimedean case, 
\begin{enumerate}
\setcounter{enumi}{3}
\item \label{item:110} $\displaystyle \psiabs_{\ov L,s}(u)= \val(p_{r}) +
  \sum_{i}\ell_{i}\min\{u, \val(\xi_{i})\}$ for $u\in \R$,
\item \label{item:111} $\displaystyle \cM_{\Z}(\psiabs_{\ov L,s})=
  \sum_{i}\ell_{i}\delta_{\val(\xi_{i})}$,
\item \label{item:112} $\displaystyle \htor_{\ov L}(Y)= m_{r}\log|p_{r}|+
  \sum_{i<j}\ell_{i}\ell_{j} \big|\log|\xi_{i}|- \log|\xi_{j}| \big|$.
\end{enumerate}
\end{thm}

\begin{rem} The real roots of the polynomial $q$ are all negative,
  which allows the use of the principal determination of the logarithm
  in~\eqref{item:109}. Introducing the argument
  $\theta_i\in]-\pi,\pi[$ of $-\xi_i$, the last sum
  in~\eqref{item:109} can be rewritten
$$\frac{1}{2}\sum_{i<j}\ell_{i}\ell_{j}
\frac{(|\xi_i|^2-|\xi_j|^2)\log|\xi_i/\xi_j| +
  2|\xi_i||\xi_j|(\theta_i-\theta_j)
  \sin(\theta_i-\theta_j)}{|\xi_i|^2+
  |\xi_j|^2-2|\xi_i||\xi_j|\cos(\theta_i-\theta_j)}  
$$
showing that it is real.
\end{rem}

\begin{proof}[{Proof of Theorem \ref{thm:26}}]
  Write $\psiabs=\psiabs_{\ov L,s}$ for short.  First we
  consider the Archime\-dean case. We have that 
  $q=|p_{r}|^{2}\prod_{i}(z-\xi_{i})^{\ell_{i}}$. By Proposition
  \ref{prop:86},
\begin{displaymath}
  \psiabs(u)= -\frac{1}{2}\log (q(\e^{-2u})) = -\log|p_{r}|
  -\frac{1}{2}\sum_{i}\ell_{i}\log|\e^{-2u}-\xi_{i}|,
\end{displaymath}
which proves \eqref{item:100}. Hence,   
\begin{displaymath}
 \psiabs'(u)= \sum_{i}\ell_{i}\frac{1}{1-\xi_{i} \e^{2u}}\quad \text{and} \quad
\psiabs''(u)=  \sum_{i}2\ell_{i}\frac{\xi_{i} \e^{2u}}{(1-\xi_{i} \e^{2u})^{2}}.
\end{displaymath}
The  Monge-Amp\`ere measure of $\psiabs$ is given by
$-\psiabs''\dd u$, and so the above proves \eqref{item:99}.
To prove \eqref{item:109} we apply the equation (\ref{eq:134}) and
Lemma \ref{lemm:7}. We have that 
$\Stab(\psiabs )=[0,m_{r}]$, $\psiabs ^{\vee}(0)=0$, and $\psiabs
^{\vee}(m_{r})=\log|p_{r}|$. Thus,
\begin{equation}\label{eq:130}
  \htor_{\ov L}(Y)=m_{r}\log|p_{r}|+ \int_{-\infty}^{\infty}
  (m_{r}-\psiabs')\psiabs' \dd u.  
\end{equation}
We have $\displaystyle m_{r}-\psiabs'(u)= \sum_{i}\ell_{i}\bigg(1-
\frac{1}{1-\xi_{i} \e^{2u}}\bigg) =-
\sum_{i}\ell_{i}\frac{\xi_{i}\e^{2u}}{1-\xi_{i} \e^{2u}}$.  Hence,
\begin{multline*}
  (m_{r}-\psiabs'(u))\psiabs'(u)=
  -\bigg(\sum_{i}\ell_{i}\frac{\xi_{i}\e^{2u}}{1-\xi_{i}
    \e^{2u}}\bigg)\bigg( \sum_{j}\ell_{j}\frac{1}{1-\xi_{j}
    \e^{2u}}\bigg) \\
  = -\sum_{i}\ell_{i}^{2} \frac{\xi_{i}\e^{2u}}{(1-\xi_{i}
    \e^{2u})^{2}}- \sum_{i\ne
    j}\ell_{i}\ell_{j}\frac{\xi_{i}\e^{2u}}{(1-\xi_{i}
    \e^{2u})(1-\xi_{j} \e^{2u})}.
\end{multline*}
Moreover
\begin{math} \displaystyle \int_{-\infty}^{\infty}
  \frac{\xi_{i}\e^{2u}}{(1-\xi_{i} \e^{2u})^{2}}\dd u =
  \bigg[\frac{1}{2(1-\xi_{i} \e^{2u})}\bigg]^{\infty}_{-\infty} =
  -\frac{1}{2}
\end{math}
and  
\begin{multline*}
  \int_{-\infty}^{\infty} \frac{\xi_{i}\e^{2u}}{(1-\xi_{i}
    \e^{2u})(1-\xi_{j} \e^{2u})} \dd u=
  \bigg[\frac{\xi_{i}}{2(\xi_{i}-\xi_{j})}( \log(1-\xi_{j}
  \e^{2u})-\log(1-\xi_{i} \e^{2u})) \bigg]^{\infty}_{-\infty}\\
  =\frac{\xi_{i}}{2(\xi_{i}-\xi_{j})}( \log
  (-\xi_{j})-\log(-\xi_{i})),
\end{multline*}
for the principal determination of $\log$. These calculations together
with the equation 
\eqref{eq:130} imply that
\begin{multline*}
  \htor_{\ov L}(Y)= m_{r}\log|p_{r}|+ \frac{1}{2}\sum_{i}\ell_{i}^{2} +
  \frac{1}{2}\sum_{i\ne
    j}\ell_{i}\ell_{j}\frac{\xi_{i}}{\xi_{i}-\xi_{j}}( \log
  (-\xi_{i})-\log(-\xi_{j})) \\= m_{r}\log|p_{r}|+
  \frac{1}{2}\sum_{i}\ell_{i}^{2} +
  \frac{1}{2}\sum_{i<j}\ell_{i}\ell_{j}\frac{\xi_{i}+\xi_{j}}{\xi_{i}-\xi_{j}}( \log
  (-\xi_{i})-\log(-\xi_{j})) ,
\end{multline*}
which proves \eqref{item:109}.

Next we consider the non-Archimedean case. Let $\zeta\in \ov
K^{\times}$ and write $v_{i}=\val(\xi_{i})$ for short. Proposition
\ref{prop:86} and the condition $m_{i}\not 
= m_{j}$ for $i\not = j$, imply, after possibly multiplying $\zeta$ by
a sufficiently general root of unity, that
\begin{displaymath}
  \psiabs(\val(\zeta))= 
  \min_{i}\{0, m_{i}\val(\zeta) + \val(p_{i})\}= 
  \val(q(\zeta)).
\end{displaymath}
By the factorization of $q$, 
\begin{displaymath}
\val(q(\zeta))= \val(p_{r}) + \sum_{i} \ell_{i}\val(\zeta-\xi_{i})  
= \val(p_{r}) + \sum_{i} \ell_{i}\min\{\val(\zeta),v_{i}\}.
\end{displaymath}
The image of $\val\colon \ov K^{\times}\to \R$
is a dense subset. For $u\in \R$, we deduce that
\begin{displaymath}
  \psiabs(u)= \val(p_{r}) +
  \sum_{i}\ell_{i}\min\{u, v_{i}\},
\end{displaymath}
which proves \eqref{item:110}. The
sup-differential of this function is, for $u\in \R$, 
\begin{displaymath}
  \partial \psiabs(u) =
  \begin{cases}
    \Big[\sum_{j: v_{j}>v_{i}}\ell_{j} ,\sum_{j: v_{j}\ge
      v_{i}}\ell_{j}\Big] & \text{ if } 
    u=v_{i} \text{ for some } i,\\
    \sum_{j: v_{j}>u}\ell_{j} & \text{ otherwise.}
  \end{cases}
\end{displaymath}
Hence, the associated Monge-Amp\`ere measure is
\begin{math}
 \sum_{i}\ell_{i}\delta_{v_{i}},
\end{math} which proves \eqref{item:111}.
The derivative of $\psiabs $ in the sense of \eqref{eq:133} is, for $u\in \R$,
\begin{displaymath}
  \psiabs'(u) =
  \begin{cases}
    \sum_{j: v_{j}>v_{i}}\ell_{j}+\frac{1}{2}\sum_{j: v_{j}=v_{i}}\ell_{j}  & \text{ if } 
    u=v_{i} \text{ for some } i,\\
    \sum_{j: v_{j}>u}\ell_{j} & \text{ otherwise.}
  \end{cases}
\end{displaymath}
Moreover, $\Stab(\psiabs )=[0,m_{r}]$, $\psiabs ^{\vee}(0)=0$ and $\psiabs
^{\vee}(m_{r})=-\val(p_{r})$. By (\ref{eq:134}) and  Lemma \ref{lemm:7}
\begin{equation}\label{eq:137}
  \htor_{\ov L}(Y)=-m_{r}\val(p_{r})
  +\int_{-\infty}^{\infty}(m_{r}-\psiabs ')\psiabs '\dd u.
\end{equation}
If we write
\begin{displaymath}
  f_{i}(u)=
  \begin{cases}
    0& \text{ if }u\le v_{i}\\
    \ell_{i} &\text{ if } u> v_{i},
  \end{cases}
\end{displaymath}
then, we have that $\psiabs '(u)=\sum_{i}\ell_{i}-f_{i}(u)$
and $m_{r}-\psiabs '(u)=\sum_{i}f_{i}$ almost everywhere. Therefore
\begin{equation}\label{eq:136}
  \int_{-\infty}^{\infty}(m_{r}-\psiabs ')\psiabs '\dd u=\sum_{i,j}
  \int_{-\infty}^{\infty}f_{i}(\ell_{j}-f_{j})\dd u=
  \sum_{i,j}\ell_{i}\ell_{j}\max\{0,v_{j}-v_{i}\}.
\end{equation}
Thus, joining together \eqref{eq:137}, \eqref{eq:136} and the
relation $\log|\zeta|=-\val(\zeta )$ we deduce
\begin{displaymath}
  \htor_{\ov L}(Y)=m_{r}\log|p_{r}|+
  \sum_{i,j}\ell_{i}\ell_{j}\max\{0,\log(|\xi _{i}|/|\xi _{j}|)\}, 
\end{displaymath}
finishing the proof of the theorem since,  for $i<j$,
\begin{displaymath}
\max\{0,\log(|\xi _{i}|/|\xi _{j}|)\}
+ \max\{0,\log(|\xi _{j}|/|\xi _{i}|)\} = \big|\log|\xi_{i}|-
\log|\xi_{j}| \big|.  
\end{displaymath}
\end{proof}
We now treat the global case. \index{toric curve!height of}

\begin{cor}\label{cor:24}
  Let $\K$ be a global field. Let $0< m_{1}<\dots
  <m_{r}$ be integer numbers with $\gcd(m_{1},\dots ,m_{r})=1$, and
  $p_{1},\dots,p_{r}\in \K^{\times}$. Let 
  $\varphi\colon \T\to \P^{r}$ be the map given by
  $\varphi(t)=(1:p_{1}t^{m_{1}}:\dots :p_{r}t^{m_{r}})$, $Y$
  the closure of the image of $\varphi$, and
  $\ov L= \varphi^{*}\ov{\cO(1)}$, where $\ov{\cO(1)}$ is equipped
  with the Fubini-Study
  metric for the Archimedean places and with the
  canonical metric for the non-Archimedean places.
  For $v\in \mathfrak{M}_{\K}$, set 
  \begin{align*}
    q_{v}=
    \begin{cases}
      1+\sum_{j=1}^{r}|p_{j}|_{v}^{2}z^{m_{j}} & \text{ if } v \text{ is
        Archimedean},\\
      1+\sum_{j=1}^{r}p_{j}z^{m_{j}} & \text{ if } v \text{ is
        not Archimedean}. 
    \end{cases}
  \end{align*}
  Let $\{\xi_{v,i}\}\subset \ov \K_{v}^{\times}$ be the
  set of roots of $q_{v}$ and, for each $i$, let $\ell_{v,i}\in \N$ denote the
  multiplicity of $\xi_{v,i}$. Then
  \begin{multline*}
    \h_{\ov L}(Y)= \sum_{v|\infty} n_{v} \bigg(
    \frac{1}{2}\sum_{i}\ell_{v,i}^{2}+
    \frac{1}{2}\sum_{i<j}\ell_{v,i}\ell_{v,j}\frac{\xi_{v,i}+\xi_{v,j}}
    {\xi_{v,i}-\xi_{v,j}}(\log(-\xi_{v,i})-\log(-\xi_{v,j}))\bigg) 
    \\+ \sum_{v\nmid \infty} n_{v}
    \bigg(  \sum_{i<j}\ell_{v,i}\ell_{v,j} \big|\log|\xi_{v,i}|- \log|\xi_{v,j}| \big|\bigg).
  \end{multline*}    
\end{cor}

\begin{proof}
  This follows readily from Proposition \ref{prop:88},
  Theorem~\ref{thm:26}, and the product formula.
\end{proof}

\index{Veronese curve}%
\begin{cor}
  \label{cor:27} Let $C_{r}\subset \P^{r}_{\Q}$ be the Veronese
  curve of degree $r$ and $\ov{\cO(1)}$ the universal line bundle on $
  \P^{r}_{\Q}$ equipped with the Fubini-Study metric at the
  Archimedean place and with the canonical metric at the
  non-Archimedean ones. Then
\begin{equation*}
  \h_{\ov{\cO(1)}}(C_r) = \frac{r}{2}+\pi\sum_{j=1}^{\lfloor
    r/2\rfloor}\bigg(1-\frac{2\, j}{r+1} \bigg)\,
  \cot\bigg(\frac{\pi\, j}{r+1}\bigg) \in \frac{r}{2} 
  + \pi \, \ov \Q.
\end{equation*} 
\end{cor}

\begin{proof}
  The curve $C_{r}$ coincides
  with the closure of the image of the map $\varphi\colon \T\to
  \P^{r}$ given by $\varphi(t)=(1:t:t^{2}:\dots:t^{r})$.
  With the notation in Corollary~\ref{cor:24}, this map corresponds to $m_{i}=i$
  and $p_{i}=1$, for $i=1,\dots,r$. Then $q_{v}= \sum_{j=0}^{r}z^{j}$ for all $v\in
  \mathfrak{M}_{\Q}$.  Consider the primitive $(r+1)$-th root of unity
  $\omega=\e^{\frac{2\pi i}{r+1}}$. The polynomial $q_{v}$ is
  separable and its set of roots is $\{\omega^{l}\}_{l=1,\dots,
    r}$. Since $|\omega^{l}|_{v}=1$ for all~$v$, 
  Corollary~\ref{cor:24} implies that
\begin{multline}
  \label{eq:129} \h_{\ov L}(C_{r})= \frac{r}{2}+
  \frac{1}{2}\sum_{l<j}\frac{\omega^{l}+\omega^{j}}
  {\omega^{l}-\omega^{j}}(\log(-\omega^{l})-\log(-\omega^{j})) \\=
  \frac{r}{2}+ \frac{1}{2}\sum_{l\ne j}
  \frac{\omega^{l}+\omega^{j}}{\omega^{l}-\omega^{j}}\log(-\omega^{l}).
\end{multline}
We have that
\begin{displaymath}
\sum_{j=1}^{r}\frac{\omega^{j}+1}{\omega^{j}-1} = 
\sum_{j=1}^{r}\frac{\omega^{j}}{\omega^{j}-1} + \sum_{j=1}^{r}\frac{1}{\omega^{j}-1} = 
\sum_{j=1}^{r}\frac{1}{1- \omega^{-j}} + \sum_{j=1}^{r}\frac{1}{\omega^{j}-1} = 0.
\end{displaymath}
This implies, for $l=1,\dots, r$, 
\begin{displaymath}
 \sum_{1\le j\le r, j\ne
  l}\frac{\omega^{l}+\omega^{j}}{\omega^{l}-\omega^{j}} =
-\frac{\omega^{l}+1}{\omega^{l}-1} = i \cot\Big( \frac{\pi l}{r+1}\Big).
\end{displaymath}
Hence, 
\begin{multline*}
    \frac{1}{2}\sum_{l\ne
       j}\frac{\omega^{l}+\omega^{j}}{\omega^{l}-\omega^{j}}\log(-\omega^{l})
=   \frac{i}{2}\sum_{l=1}^{r} \cot\Big( \frac{\pi l}{r+1}\Big) \log(-\omega^{l})
\\=   \pi\sum_{l=1}^{\lfloor r/2\rfloor} \cot\Big( \frac{\pi
  l}{r+1}\Big) \Big( 1- \frac{2l}{r+1}\Big) ,
\end{multline*}
since $\cot( \frac{\pi (r+1-l)}{r+1}) \log(-\omega^{r+1-l}) =\cot(
\frac{\pi l}{r+1}) \log(-\omega^{l})$ for $l=1,\dots, \lfloor
r/2\rfloor$ and $\log(-\omega^{\frac{r+1}{2}})=0$ whenever $r$ is
odd. The statement follows from these calculations together with~\eqref{eq:129}.
\end{proof}

Here follow some special values:
\begin{center} 
{\begin{tabular*}{0.89\textwidth}{@{\extracolsep{1.8mm}}c||ccccc} 
$r$ & $1$& $2$& $3$& $5$& $7$ \\
\hline\\[-3mm] 
$\h_{\ov{\cO(1)}}(C_{r})$ & 
$\displaystyle\frac{1}{2}$ &
$\displaystyle 1 + \frac{1}{3\, \sqrt{3}} \, \pi $ &
$\displaystyle\frac{3}{2} + \frac{1}{2} \, \pi $ &
$\displaystyle\frac{5}{2} + \frac{7}{3\, \sqrt{3}}\, \pi $ &
$\displaystyle\frac{7}{2} + (1+\sqrt{2}) \, \pi$ 
\end{tabular*} }
\end{center}

\medskip

\begin{cor} 
\label{cor:28} With the notation of Corollary \ref{cor:27}, 
   $\h_{\ov{\cO(1)}}(C_{r})= r\log r + O(r)$ for
$r\to \infty$. 
\end{cor}

\begin{proof}
  We have that $\displaystyle \pi\cot(\pi x)= \frac{1}{x}+O(1)$ for $x\to
  0$. Hence, 
  \begin{displaymath}
      \h_{\ov{\cO(1)}}(C_r) = \sum_{j=1}^{\lfloor
    r/2\rfloor}\bigg(1-\frac{2\, j}{r+1} \bigg)\,
 \frac{r+1}{j} + O(r) = r \bigg(\sum_{j=1}^{\lfloor
    r/2\rfloor}\frac{1}{j}\bigg)+ O(r)= r\log r + O(r). 
  \end{displaymath}
\end{proof}

\index{successive minima}%
By the theorem of algebraic successive minima~\cite[Theorem 5.2]{Zhang:plb},
\begin{displaymath}
  \mu^{\ess}(C_{r})\le \frac{  \h_{\ov{\cO(1)}}(C_r)}{
    \deg_{{\cO(1)}}(C_r)} \le 2 \mu^{\ess}(C_{r})
\end{displaymath}
\index{essential minimum}%
The essential minimum of $C_{r}$ is $\mu^{\ess}(C_{r})=
\frac{1}{2}\log(r+1)$ \cite[Th\'eor\`eme 0.1]{Sombra:msvtp}. Hence, the quotient $\frac{
  \h_{\ov{\cO(1)}}(C_r)}{
    \deg_{{\cO(1)}}(C_r)}$ is asymptotically closer to the upper bound
  than to the lower bound. 

\section{Height of toric bundles} \label{Height of toric bundles}
\index{toric projective bundle}%

Let $n\ge 0$ and write $\P^n=\P^n_{\Q}$ for short.  Given integers $a_{r}\ge \dots \ge
a_{0}\ge 1$, consider the bundle $\P(E)\rightarrow\P^n$ of
hyperplanes of the vector bundle
$$
E=\cO(a_0)\oplus \cO(a_1)\oplus \dots\oplus
\cO(a_r) \longrightarrow \P^n,
$$
where $\cO(a_{j})$ denotes the $a_{j}$-th power of the universal line
bundle of $\P^{n}$.  Equivalently, $\P(E)$ can be defined as the bundle of lines of the
dual vector bundle $E^{\vee}$.  The fibre of the map
$\pi\colon\P(E)\to \P^n$ over each point $p\in \P^n(\ov \Q)$ is a projective
space of dimension $r$.  This bundle is a smooth toric variety over
$\Q$ of dimension $n+r$, see \cite[pages~58--59]{Oda88},
\cite[page~42]{Ful93}.  The particular case $n=r=1$ corresponds to
Hirzebruch surfaces\index{Hirzebruch surface}%
: for $b\ge 0$, we have
$\F_{b}=\P(\cO(0)\oplus \cO(b))\simeq
\P(\cO(a_{0})\oplus \cO(a_{0}+b)) $ for any $a_{0}\ge1$.

The \emph{tautological line bundle} of $\P(E)$, denoted $\cO_{\P(E)}(-1)$, is
defined as a subbundle of~$\pi^{*}E^{\vee}$. Its fibre over a point of
$\P(E)$ is the inverse image under $\pi$ of the line in $E^{\vee}$
which is dual to the hyperplane of $E$ defining the given point.  The
\emph{universal line bundle} $\cO_{\P(E)}(1)$ of $\P(E)$ is defined as
the dual of the tautological one.  Since $\cO(a_{j})$,
$j=0,\dots, r$, is ample, the universal line bundle is also
ample~\cite[propositions~2.2 and 3.2]{MR0193092}.
This is the line bundle corresponding to the Cartier divisor $a_0
D_0+D_1$, where $D_0$ denotes 
the inverse image in $\P(E)$ of the hyperplane at infinity of
$\P^n$ and $D_1=\P(0\oplus \cO(a_1)\oplus\dots\oplus
\cO(a_r))$.
Observe that, although  $\P(E)$ is isomorphic to the bundle
associated to the family of  
integers $a_{i}+c$ for any $c\in \N$, this is not the case for the
associated universal 
line bundle, that depends on the choice of $c$.

Following Example \ref{exm:10}, we regard $\P^{n}$ as a toric
variety over $\Q$ equipped with the action of the split torus
$\G_{m}^{n}$.  Let $s$ be the toric section
of~$\cO(1)$ which corresponds to the hyperplane at infinity
$H_0$ and let $s_{j}=s^{\otimes -a_{j}}$, which is a section of
$\cO(-a_{j})$.  Let $U= \P^{n}\setminus H_{0}$. The
restriction of $\P(E)$ to $U$ is isomorphic to $U\times\P^{r}$
through the map $\varphi$ defined, for $p\in U$ and $q\in \P^{r}$,
as
\begin{displaymath}
  (p,q) \longmapsto (p,q_{0}s_{0}(p)\oplus \dots \oplus
  q_{r}s_{r}(p)).
\end{displaymath}
The torus $\T:=\G_{m}^{n+r}$ can then be included as an open subvariety
of $\P(E)$ through the map $\varphi$ composed with the standard
inclusion of $\G_{m}^{n+r}$ into $U\times\P^{r}$.  The action
of $\T$ on itself by translation extends to an action
of the torus on the whole of $\P(E)$. Hence $\P(E)$
is a toric variety over $\Q$. With this action the divisor $a_0
D_0+D_1$ is a $\T$-Cartier divisor.
 
By abuse of notation, we also denote $E^{\vee}$ the total space
associated to the vector bundle $E^{\vee}$.
The map $\G_{m}^{n+r}\to E^{\vee}$ defined as
\begin{equation}\label{eq:143}
  (z,w)\longmapsto ((1:z),(s_{0}(1:z)\oplus w_{1}s_{1}(1:z)\oplus
  \dots\oplus w_{r}s_{r}(1:z))) 
\end{equation}
induces a nowhere vanishing section of the tautological line bundle of
$\P(E)$ over the open subset $\T$.  Its inverse defines a rational
section of $\cO_{\P(E)}(1)$, denoted $s_{\P(E)}$, that is regular and
nowhere vanishing on $\T$. In particular, this section induces a
structure of toric line bundle on $\cO_{\P(E)}(1)$. The divisor of
$s_{\P(E)}$ is precisely the $\T$-Cartier divisor $a_{0}D_{0}+D_{1}$
considered above.

We now introduce an adelic toric metric on $\cO_{\P(E)}(1)$. For
$v=\infty$, we consider the complex vector bundle $E(\C)$ that can be
naturally metrized by the direct sum of the Fubiny-Study metric on
each factor $\cO(a_j)(\C)$.  By duality, this gives a metric on
$E^{\vee}(\C)$, which induces by restriction a metric on the
tautological line bundle.  Applying duality one more time, we obtain
a smooth metric, denoted $\|\cdot \|_{\infty}$, on~$O_{\P(E)(\C)}(1)$.
Since the Fubini-Study metric on each $\cO(a_j)(\C)$ is toric, then
$\|\cdot \|_{\infty}$ is toric too.  For $v\in {\mathfrak M}_{\Q}\setminus
\{\infty\}$, we equip $O_{\P(E)}(1)$ with the canonical metric
(Proposition-Definition~\ref{def:57}).  We write
$\overline{\cO_{\P(E)}(1)}=(\cO_{\P(E)}(1),(\|\cdot\|_{v})_{v\in
  {\mathfrak M}_{\Q}})$ for the obtained adelic metrized toric line bundle.

We have made a choice of splitting of $\T$ and
therefore a choice of an identification $N=\Z^{n+r}$. Thus we obtain
a system of coordinates in the real 
vector space associated to the toric variety $\P(E)$, namely
$N_{\R}=\R^{n+r}=\R^{n}\times\R^{r}$.
Since the metric considered at each non-Archimedean place is the
canonical one, the only nontrivial contribution to the global height will
come from the Archimedean place. 
The restriction to the principal open
subset    $\P(E)_{0}(\C)\simeq (\C^{\times})^{n+r}
=(\C^{\times})^{n}\times(\C^{\times})^{r}$ 
of the valuation map is expressed, in
these coordinates, as the map
$\val\colon (\C^{\times})^{n+r} \to N_{\R}$ defined by 
$$\val (z,w)=
(-\log|z_{1}|, \dots,  -\log|z_{n}|, -\log|w_{1}|,\dots,
-\log|w_{r}|).
$$ 

Write $\psiabs_{\infty}\colon N_{\R}\to \R$ for the function
corresponding to the metric $\|\cdot\|_{\infty}$ and the toric section
$s_{\P(E)}$ defined above.  

\begin{lem} \label{lemm:16} 
The function $\psiabs_{\infty}$ is defined,
  for $u \in \R^{n}$ and  $v\in \R^{r}$, as
  \begin{displaymath}
\psiabs_{\infty}(u,v)=-\frac{1}{2}\log\left(
\sum_{j=0}^r{\e}^{-2v_j}\left(\sum_{i=0}^n
  {\e}^{-2u_i}\right)^{a_j}\right),
  \end{displaymath}
with the convention $u_{0}=v_{0}=0$.
It is a strictly concave function. 
\end{lem}

\begin{proof}
The metric on $E^{\vee}(\C)$ is given, for $p\in \P^{n}(\C)$ and
$q_{0},\dots,q_{r}\in \C$, by 
\begin{displaymath}
  ||q_{0}s_{0}(p)\oplus \dots \oplus
  q_{r}s_{r}(p)||_{\infty}^{2}=|q_{0}|^{2}||s_{0}(p)||^{2}+ \dots
  + |q_{r}|^{2}||s_{r}(p)||^{2},
\end{displaymath}
where $||s_{j}(p)||$ is the norm of $s_{j}(p)$ with respect to the
Fubini-Study metric on $\cO(-a_{j})^{\an}$. By Example \ref{exm:3}, 
\begin{displaymath}
  ||s_{j}(p)||^{2}= \bigg(
  \frac{|p_{0}|^{2}}{|p_{0}|^{2}+\dots+|p_{n}|^{2}}\bigg)^{-a_{j}}.
\end{displaymath}
Using Definition \ref{def:68} and Proposition
\ref{prop:24}\eqref{item:79}, we compute the function $\psiabs_\infty$ via
the Green function $-\log\Vert s_{\P(E)}^{\otimes -1} \Vert$ relative to the
toric section $s_{\P(E)}^{\otimes -1}$, defined in \eqref{eq:143}, of the
tautological bundle, dual ${\cO}_{\P(E)}(1)$.  The explicit
description (\ref{eq:143}) of the section $s_{\P(E)}^{\otimes -1}$
implies that, for $(z,w)=(z_{1},\dots,z_{n},w_{1},\dots,w_{r})\in
\G_{m}^{n+r}$, writing $z_{0}=w_{0}=1$, we have
\begin{equation}\label{eq:144}
\|s_{\P(E)}^{\otimes -1}(z,w) \|^{2}= 
\sum_{j=0}^r|w_{j}|^{2}\left(\sum_{i=0}^n |z_{i}|^{2}\right)^{a_j}.
\end{equation}
If $\val(z,w)=(u_{1},\dots,u_{n},v_{1},\dots,v_{r})$ and writing
$u_{0}=v_{0}=0$, equation (\ref{eq:144}) can be written as
\begin{equation}\label{eq:109}
\|s_{\P(E)}^{\otimes -1}(z,w) \|^{2}= 
\sum_{j=0}^r{\e}^{-2v_j}\left(\sum_{i=0}^n {\e}^{-2u_i}\right)^{a_j}.
\end{equation}
Now  $\psiabs_\infty(u,v)$ equals $-\log\Vert s_{\P(E)}^{\otimes -1}(z,w) \Vert$, that is $-1/2$ times the logarithm of the right hand side in \eqref{eq:109}. This proves the equality of the lemma.

For the last statement, observe that the functions 
${\e}^{-2v_j}(\sum_{i=0}^n {\e}^{-2u_i})^{a_j}$
  are log-strictly convex, because $-1/2$ times their logarithm is the function
  associated to the Fubini-Study metric on $\cO(a_{j})^{\an}$,
  which is a strictly concave function.
Their sum is also log-strictly convex~\cite[\S 3.5.2]{MR2061575}. Hence,
$\psiabs_{\infty}$ is strictly concave.
\end{proof}

\begin{cor} \label{cor:15}
    The metric $\|\cdot\|_{\infty}$ is a 
    semipositive smooth toric metric.
\end{cor}

\begin{proof}
    The facts that $\|\cdot\|_{\infty}$  is smooth and toric follow from
    its construction. The fact that it is semipositive follows from
    Lemma \ref{lemm:16} and Theorem \ref{thm:13}\eqref{item:37}.
\end{proof}

In the following result, we summarize the combinatorial data describing
the toric structure of $\P(E)$ and
$\ov {\cO_{\P(E)}(1)}$.

\begin{prop} \label{prop:67} 
\begin{enumerate}
  \item \label{item:83} Let $e_i$, $1\le i\le n$, and $f_j$,  $1\le j\le r$,
be the $i$-th and $(n+j)$-th vectors of the standard basis of
$N=\Z^{n+r}$. Set $f_0=-f_1-\cdots-f_r$ and
$e_0=a_0f_0+\cdots+a_rf_r-e_1-\cdots -e_n$. 
The fan $\Sigma$ corresponding to $\P(E)$ is the fan in $N_{\R}$ whose maximal cones
are the convex hull of the rays generated by the vectors
$$e_0,\cdots,e_{k-1},e_{k+1},\cdots,e_n,f_0,\cdots,f_{\ell-1},f_{\ell+1},\cdots,f_{r}$$
for $0\le k\le n, 0\le \ell\le r$.
This is a complete regular fan.

\item \label{item:84} The support function $\Psi\colon N_{\R}\to {\R}$
  corresponding to the universal line bundle $\cO_{\P(E)}(1)$ and the
  toric section $s_{\P(E)}$ is defined, for $u\in \R^{n}$ and
  $v\in\R^{r}$, as
    \begin{displaymath}
      \Psi(u,v)=\mathop{\min_{0\le k\le n}}_{0\le \ell\le r} (a_{\ell}u_{k}+v_{\ell}),
    \end{displaymath}
where, for short, we have set $u_{0}=v_{0}=0$. 
  \item \label{item:85} The polytope $\Delta$ in $M_{\R}=\R^{n}\times\R^{r}$ associated to $(\Sigma,\Psi)$ is 
$$
\Big\{(x,y)| y_1,\dots,y_r\geq0,\
\sum_{\ell=1}^ry_\ell\leq 1,\ x_1,\dots,x_n\geq 0,\ \sum_{k=1}^nx_k\leq
L(y)\Big\}
$$
with $L(y)=a_0+\sum_{\ell=1}^r(a_\ell-a_0) y_\ell$. Using the
convention $y_{0}=1-\sum_{\ell=1}^{r}y_{\ell}$ and
$x_{0}=L(y)-\sum_{k=1}^{n}x_{k}$, then $L(y)=\sum_{\ell=0}^ra_\ell
y_\ell$ and the polytope $\Delta $ can be written as
$$
\Big\{(x,y)| y_0,\dots,y_r\geq0,\
x_0,\dots,x_n\geq 0\Big\}.
$$\item \label{item:87}  The 
Legendre-Fenchel dual of $\psiabs_\infty$ is the
concave function $\vartheta _\infty \colon \Delta\to \R$ defined, for
$(x,y)\in \Delta$, as 
$$
\vartheta _\infty(x,y)=
\frac{1}{2}\left(\varepsilon_r(y_1,\dots,y_r) + L(y)\cdot 
  \varepsilon_n\left(\frac{x_1}{L(y)},\dots,
    \frac{x_n}{L(y)}\right)\right),
$$
where, for $k\ge 0$,  $\varepsilon_k$ is the
function defined in \eqref{eq:103}.
For $v\ne \infty$, the concave function  $\vartheta _{v}=\psiabs^\vee_v$ is the
indicator function of $\Delta$.
  \end{enumerate}
\end{prop}

\begin{proof}
By Corollary \ref{cor:8}, we have $\Psi=\rec(\psiabs_{\infty})$. By
the equation \eqref{eq:49}, we have 
$\rec(\psiabs_{\infty})(u,v)= \lim_{\lambda\to \infty}
\lambda^{-1}\psiabs_{\infty}(\lambda (u,v))$.  
Statement \eqref{item:84} follows readily from this and from the
expression for $\psiabs_{\infty}$ in Lemma \ref{lemm:16}. 

The function $\Psi$ is strictly concave on $\Sigma$, because
$\cO_{\P(E)}(1)$ is an ample line bundle. Hence $\Sigma= \Pi(\Psi)$
and this is the fan described in statement \eqref{item:83}.

Let
$(e_{1}^{\vee},\dots,e_{n}^{\vee},f_{1}^{\vee},\dots,f_{r}^{\vee})$ be
the dual basis of $M$ induced by the basis of $N$.
By Proposition \ref{prop:19} and statement \eqref{item:84}, we have 
\begin{displaymath}
\Delta= \Conv\Big(0 , (a_{0}e^{\vee}_{k})_{1\le k\le n}, (f^{\vee}_{\ell})_{1\le \ell\le
r}, (a_{\ell}e^{\vee}_{k}+f^{\vee}_{\ell})_{\stackrel{1\le k\le n}{\scriptscriptstyle
  1\le \ell\le r}}\Big).
\end{displaymath}
Statement \eqref{item:85} follows readily from this. 

For the first part of statement \eqref{item:87}, 
it suffices to compute the Legendre-Fenchel dual of $\psiabs_{\infty}$ at a
point $(x,y)$ in the interior of the polytope. 
Lemma \ref{lemm:16} shows that $\psiabs_{\infty}$ is strictly
concave. Hence, by Theorem \ref{thm:2}\eqref{item:16},
 $\nabla \psiabs_{\infty}$ is a homeomorphism between $N_{\R}$
and $\Delta^{\circ}$.
Thus, there exist a unique $(u,v)\in N_{\R}$ such that, for
$i=1,\dots, n$ and $j=1,\dots, r$,  
\begin{displaymath}
  x_{i}=\frac{\partial \psiabs_{\infty}}{\partial u_{i}}(u,v), 
\quad   y_{j}=\frac{\partial \psiabs_{\infty}}{\partial v_{j}}(u,v). 
\end{displaymath}
We use the conventions
$x_{0}=L(y)-\sum_{i=1}^{n}x_{i}$, $y_{0}=1-\sum_{j=1}^{r}y_{j}$, and 
$u_{0}=v_{0}=0$ as before, and also $\eta=\sum_{i=0}^{n}\e^{-2u_{i}}$
and $\psiabs=\psiabs_{\infty}$, so that 
$-2\psiabs=\log
\big(\sum_{j=0}^{r}\e^{-2v_{j}}\eta^{a_{j}}\big)$. 
Computing the gradient of $\psiabs$, we obtain, for
$i=1,\dots, n$ and $j=1,\dots, r$,  
\begin{displaymath}
  x_{i} \e^{-2\psiabs} =\Big(\sum_{j=0}^{r}a_{j}\eta^{a_{j}-1}\e^{-2v_{j}}\Big)
  \e^{-2u_{i}},\quad 
y_{j} \e^{-2\psiabs}=\eta^{a_{j}} \e^{-2v_{j}}.
\end{displaymath}
Combining these expressions, we obtain, for $i=0,\dots, n$ and $j=0,\dots, r$,  
\begin{displaymath}
  \frac{x_{i}}{L(y)}=\frac{\e^{-2u_{i}}}{\eta},\quad 
y_{j}= {\eta^{a_{j}}}{\e^{-2v_{j}+2\psiabs}}.
\end{displaymath}
From the case $i=0$ we deduce $\eta= L(y)/x_{0}$ and from the case
$j=0$ it results $2\psiabs = \log(y_{0})+a_{0}\log(x_{0}/L(y))$.
From this, one can verify 
\begin{displaymath}
  u_{i}=\frac12 \log\Big(\frac{x_{0}}{x_{i}}\Big) ,
\quad v_{j}= \frac12 \log\Big( \frac{y_{0}}{y_{j}}\Big) + 
 \frac{a_{0}-a_{j}}{2}\log\Big(   \frac{x_{0}}{L(y)}\Big).
\end{displaymath}
From Theorem \ref{thm:2}\eqref{item:51}, we have $\psiabs^{\vee}(x,y)=
\langle x,u\rangle+\langle y,v\rangle -\psiabs(u,v)$. Inserting the
expressions above for $\psiabs$, $u_{i}$ and $v_{j}$ in terms of $x,y$,
we obtain the stated formula.

For $v\ne \infty$, we
have $\psiabs_{v}=\Psi$. The last statement follows  from Example \ref{exm:7}.
\end{proof}

Proposition \ref{prop:46} and Theorem \ref{thm:22} imply
\begin{equation}\label{eq:toricbundle}
  \begin{aligned}
    \deg_{\cO_{\P(E)}(1)}(\P(E)) &= (n+r)! \Vol(\Delta) ,\\
    \h_{\overline{\cO_{\P(E)}(1)}}(\P(E)) &= (n+r+1)! \int_\Delta
    \psiabs^\vee_\infty \ \dd x \, \dd y,
  \end{aligned}
\end{equation}
where, for short, $\dd x$ and $\dd y$ stand for $\dd x_{1}\dots\dd
x_{n}$ and $\dd y_{1}\dots\dd y_{r}$, respectively.

We now compute these volume and integral giving the degree and the 
height of $\P(E)$. We show, in particular, that the
height is a rational number.
Recall that $\Delta^r$ and  $\Delta^n$ are the standard simplexes of
$\R^r$ and $\R^{n}$, respectively.

\begin{lem}\label{prop:20} 
With the above notation, we have
  \begin{align*}
&\deg_{\cO_{\P(E)}(1)}(\P(E)) =
\frac{(n+r)!}{n!}\int_{\Delta^r}L(y)^n\dd y,\\
&  \h_{\overline{\cO_{\P(E)}(1)}}(\P(E))=\frac{(n+r+1)!}{(n+1)!}
    \h_{\ov{\cO(1)}}(\P^n) \int_{\Delta^r} L(y)^{n+1} \dd y\\ 
&\hspace{195pt} + \frac{(n+r+1)!}{2\, n!} \int_{\Delta^r}L(y)^{n} \varepsilon_r(y)
    \dd y, \nonumber
  \end{align*}
  where $\h_{\ov{\cO(1)}}(\P^n)=\sum_{h=1}^n\sum_{j=1}^h\frac{1}{2j}$ is the
  height of the projective space relative to the Fubini-Study metric.
\end{lem}
\begin{proof}

The equation~(\ref{eq:toricbundle}) shows that the degree of $\P(E)$ is
equal to $(n+r)!\Vol(\Delta)$.
The same equation together with
  Proposition~\ref{prop:67}\eqref{item:87} gives that the height of 
$ \P(E)$ is equal to
  \begin{equation}\label{forhprop:20}
    \frac{(n+r+1)!}{2}\left(\int_{\Delta}\varepsilon_r(y)\dd x\dd y +
      \int_{\Delta}L(y)\cdot\varepsilon_n(L(y)^{-1}x)\dd x\dd y
    \right). 
  \end{equation}
 Let $I_1$ and $I_2$ be the two above integrals. Observe $\Delta =
  \bigcup_{y\in \Delta^r}(\{y\}\times (L(y)\cdot \Delta^n))$. Then
  \begin{align*}
\Vol(\Delta)&= \int_{\Delta^r}\left(\int_{L(y)\cdot\Delta^n}dx\right)\dd
  y=  \frac{1}{n!}\int_{\Delta^r}L(y)^n\dd y , \\
I_1 &= \int_{\Delta^r}\left(\int_{L(y)\cdot \Delta^n}\dd
  x\right)\varepsilon_r(y)\dd y =
\frac{1}{n!}\int_{\Delta^r}L(y)^n\varepsilon_r(y)\dd y,
  \end{align*}
since $\int_{L(y)\cdot \Delta^n}\dd x = L(y)^n/n!$. For the
second integral, we have that
\begin{align*}
  I_2 &= \int_{\Delta^r}L(y)\left(\int_{L(y)\cdot \Delta^n}\varepsilon_n(L(y)^{-1}x)\dd x\right)\dd y \\
  &= \left(\int_{\Delta^r}L(y)^{n+1}\dd
    y\right)\left(\int_{\Delta^n}\varepsilon_n(x)\dd x\right) =
  \frac{2\h_{\ov{\cO(1)}}(\P^n)}{(n+1)!}\int_{\Delta^r}L(y)^{n+1}\dd
  y ,
\end{align*}
since $\int_{L(y)\cdot \Delta^n}\varepsilon_n(L(y)^{-1}x)\dd x =
L(y)^n\int_{\Delta^n}\varepsilon_n(x)\dd x$ and, by Example \ref{exm:32},
$$\int_{\Delta^n}\varepsilon_n(x)\dd x =
\frac{1}{(n+1)!}\sum_{h=1}^n\sum_{j=1}^h\frac{1}{j} =
\frac{2\h_{\ov{\cO(1)}}(\P^n)}{(n+1)!}. 
$$ 
 The expression for
$\Vol(\Delta)$ gives the formula for the degree. Lemma~\ref{prop:20}
then follows by carrying up the expressions of $I_1$ and $I_2$ into
(\ref{forhprop:20}).
\end{proof}

\begin{thm}\label{fibrestoriques}
  With the above notation, we have
\index{toric projective bundle!degree of}%
\index{toric projective bundle!height of}%
  \begin{align*}
    \deg_{\cO_{\P(E)}(1)}(\P(E)) &= \sum_{{i_0,\dots,i_r\in{\N}}\atop{i_0+\dots+i_r=n}} a_0^{i_0}\dots a_r^{i_r}\\
    \h_{{\overline{\cO_{\P(E)}(1)}}}(\P(E)) &= \left(\sum_{{i_0,\dots,i_r\in{\N}}\atop{i_0+\dots+i_r=n+1}} a_0^{i_0}\dots a_r^{i_r}\right) \h_{\ov{\cO_{\P^n}(1)}}(\P^n)\\
    &\kern3.5cm+ \sum_{{i_0,\dots,i_r\in{\N}}\atop{i_0+\dots+i_r=n}} a_0^{i_0}\dots a_r^{i_r}
    A_{n,r}(i_0,\dots,i_r) ,
  \end{align*}
  where 
  $A_{n,r}(i_0,\dots,i_r)=\sum_{m=0}^r(i_m+1)\sum_{j=i_m+2}^{n+r+1}\frac{1}{2j}$.
In particular, the height of $\P(E)$ is a positive rational number.
\end{thm}

\begin{proof}
 To prove this result it suffices to
  compute the two integrals appearing in Lemma~\ref{prop:20}. However
$$L(y)=a_0+\sum_{\ell=1}^r(a_\ell-a_0) y_\ell=a_0y_0+\dots+a_ry_r
,$$ with $y_0=1-y_1-\dots-y_r$, and therefore
$$
L(y)^n
=\sum_{\stackrel{\alpha\in\N^{r+1}}{\scriptscriptstyle|\alpha|=n}}
\binom{n}{\alpha_0,\dots,\alpha_r} \prod_{\ell=0}^r(a_\ell
y_\ell)^{\alpha_\ell}
$$
and similarly for $L(y)^{n+1}$. Now, Corollary~\ref{calculmonomeetmonomelog} gives
\begin{align*}
  \int_{\Delta^r}y_0^{\alpha_0}y_1^{\alpha_1}\dots y_{r}^{\alpha_{r}} \dd y &= \frac{\alpha_0!\dots\alpha_r!}{(|\alpha|+r)!},\\
  \int_{\Delta^r}y_0^{\alpha_0}y_1^{\alpha_1}\dots
  y_{r}^{\alpha_{r}}\log(y_j) \dd y &=
  -\frac{\alpha_0!\dots\alpha_r!}{(|\alpha|+r)!}\sum_{\ell=\alpha_j+1}^{|\alpha|+r}\frac{1}{\ell},
\end{align*}
which, combined with the above expression for $L(y)^n$ and
$L(y)^{n+1}$, gives
\begin{align*}
  \int_{\Delta^r} L(y)^{n} dy &= \sum_{\stackrel{\alpha\in\N^{r+1}}{\scriptscriptstyle|\alpha|=n}} \frac{n!}{(n+r)!} \prod_{\ell=0}^r a_\ell^{\alpha_\ell} = \frac{n!}{(n+r)!}\sum_{\stackrel{i_0,\dots,i_r\in\N}{\scriptscriptstyle i_0+\dots+i_r=n}} \prod_{\ell=0}^r a_\ell^{i_\ell}\\
  \int_{\Delta^r} L(y)^{n+1} dy &= \sum_{\stackrel{\alpha\in\N^{r+1}}{\scriptscriptstyle|\alpha|=n+1}} \frac{(n+1)!}{(n+1+r)!} \prod_{\ell=0}^r a_\ell^{\alpha_\ell} = \frac{(n+1)!}{(n+1+r)!}\sum_{\stackrel{i_0,\dots,i_r\in\N}{\scriptscriptstyle i_0+\dots+i_r=n+1}} \prod_{\ell=0}^r a_\ell^{i_\ell}
\end{align*}
and
\begin{align*}
  \int_{\Delta^r}L(y)^{n} \varepsilon_r(y) dy &=
  \sum_{m=0}^{r}\sum_{\stackrel{\alpha\in\N^{r+1}}{\scriptscriptstyle|\alpha|=n}}
  \frac{n!(\alpha_m+1)}{(n+1+r)!} \bigg( \prod_{\ell=0}^r
  a_\ell^{\alpha_\ell} \bigg) \sum_{\ell=\alpha_m+2}^{n+1+r}\frac{1}{\ell}\\
  &=
  \frac{n!}{(n+1+r)!}\sum_{\stackrel{i_0,\dots,i_r\in\N}{\scriptscriptstyle
      i_0+\dots+i_r=n}} \bigg( \prod_{\ell=0}^r a_\ell^{i_\ell}\bigg) \sum_{m=0}^r(i_m+1)\sum_{\ell=i_m+2}^{n+1+r}\frac{1}{\ell}\\
  &=
  \frac{2\, n!}{(n+1+r)!}\sum_{\stackrel{i_0,\dots,i_r\in\N}{\scriptscriptstyle
      i_0+\dots+i_r=n}} \bigg( \prod_{\ell=0}^r
  a_\ell^{i_\ell} \bigg)A_{n,r}(i_0,\dots,i_r) .
\end{align*}
The statement follows from these expressions together with Lemma~\ref{prop:20}.
\end{proof}

\begin{rem} \label{rem:24}
We check $A_{1,1}(0,1)=A_{1,1}(1,0)={3}/{4}$.
Let $b\ge 0$ and let 
 $\ov{\cO_{\F_{b}}(1)}$ the adelic line bundle
on $\F_{b}$  associated to $a_{0}=1$ and $a_{1}=b+1$.
Putting $n=r=1$, $a_0=1$ and $a_1=b+1$ in Theorem \ref{fibrestoriques}, 
we recover the expression
  for the height of Hirzebruch surfaces established in~\cite{Mou06}:
$\h_{\ov{\cO_{\F_{b}}(1)}}(\F_b)=\frac{1}{2}b^2+\frac{9}{4}b+3$.
\end{rem}


\backmatter

\bibliographystyle{smfalpha}
\newcommand{\etalchar}[1]{$^{#1}$}
\newcommand{\noopsort}[1]{} \newcommand{\printfirst}[2]{#1}
  \newcommand{\singleletter}[1]{#1} \newcommand{\switchargs}[2]{#2#1}
  \def\cprime{$'$}
\providecommand{\bysame}{\leavevmode\hbox to3em{\hrulefill}\thinspace}
\providecommand{\MR}{\relax\ifhmode\unskip\space\fi MR }
\providecommand{\MRhref}[2]{%
  \href{http://www.ams.org/mathscinet-getitem?mr=#1}{#2}
}
\providecommand{\href}[2]{#2}

\printnomenclature[43mm]
\printindex

\end{document}